\pdfoutput=1 
\documentclass[11pt,centertags]{amsart}
\usepackage{amssymb}            
\usepackage{enumitem}           
\usepackage{url}                
\usepackage{graphicx}
\usepackage{mathtools}          
\usepackage[marginratio=1:1,height=8.5in,width=6.5in,tmargin=1.25in]{geometry}
\usepackage[normalem]{ulem}     
\usepackage[usenames,dvipsnames]{color}

\usepackage{todonotes}          
\usepackage{hyperref}
\hypersetup{
    colorlinks,
    citecolor=red,
    filecolor=orange,
    linkcolor=green,
    urlcolor=purple
  }

\usepackage[T1]{fontenc}         
\usepackage[utf8]{inputenc}      
\usepackage{textcomp}            

\usepackage{bbm}                
\usepackage{mathpazo}       
\allowdisplaybreaks

\newtheorem*{main*}{Main Theorem}

\newtheorem{theorem}{Theorem}[section]
\newtheorem*{theorem*}{Theorem}
\newtheorem{proposition}[theorem]{Proposition}
\newtheorem{lemma}[theorem]{Lemma}
\newtheorem{corollary}[theorem]{Corollary}

\newtheorem*{question*}{Question}

\newtheorem*{conjecture*}{Conjecture}
\newtheorem{notation}[theorem]{Notations}

\newtheorem*{claim}{Claim}

\newtheorem*{fact*}{Fact}

\theoremstyle{definition}

\newtheorem{definition}[theorem]{Definition}
\newtheorem*{definition*}{Definition}

\newtheorem{pexample}[theorem]{Possible Example}

\newtheorem{example}[theorem]{Example}

\theoremstyle{remark}

\numberwithin{equation}{section}

\newcommand{\isom}{\cong} 
\newcommand{\union}{\cup}
\newcommand{\ints}{\cap}
\newcommand{\bigunion}{\bigcup}

\renewcommand{\AA}{\mathbb{A}}

\newcommand{\RR}{\mathbb{R}}

\newcommand{\TT}{\mathbb{T}}

\newcommand{\ZZ}{\mathbb{Z}}

\newcommand{\bG}{\mathbf{G}}
\newcommand{\bH}{\mathbf{H}}

\newcommand{\bP}{\mathbf{P}}

\newcommand{\sP}{\mathsf{P}}

\newcommand{\cA}{\mathcal{A}}
\newcommand{\cB}{\mathcal{B}}
\newcommand{\cC}{\mathcal{C}}
\newcommand{\cD}{\mathcal{D}}
\newcommand{\cE}{\mathcal{E}}
\newcommand{\cF}{\mathcal{F}}

\newcommand{\cH}{\mathcal{H}}
\newcommand{\cI}{\mathcal{I}}

\newcommand{\cP}{\mathcal{P}}

\newcommand{\cR}{\mathcal{R}}
\newcommand{\cS}{\mathcal{S}}

\newcommand{\cU}{\mathcal{U}}
\newcommand{\cV}{\mathcal{V}}
\newcommand{\cW}{\mathcal{W}}

\newcommand{\del}{\partial}

\newcommand{\Db}{\Delta_{\mathrm{bar}}}
\newcommand{\DbF}[1]{\Delta_{\mathrm{bar},#1}}

\newcommand{\vol}{\mathrm{Vol}}

\newcommand{\id}{\mathrm{Id}}

\newcommand{\hess}{\mathrm{Hess}}

\newcommand{\Sep}{{\mathrm{Sep}}}
\newcommand{\PSep}{{\mathrm{PSep}}}
\newcommand{\stype}[2]{\{#1,#2\setminus #1\}}
\newcommand{\rstype}[2]{\{#1\ints #2,#2\setminus #1\}}
\newcommand{\typemark}[3]{(#1,\{#2,#3\setminus #2\})}
\newcommand{\rtypemark}[3]{(#1,\{#2\ints#3,#3\setminus #2\})}
\newcommand{\BFloor}{\mathrm{BFloor}}

\newcommand{\Floor}{\mathrm{Floor}}
\newcommand{\BChamber}{\mathrm{BChamber}}
\newcommand{\Chamber}{\mathrm{Chamber}}
\newcommand{\Cone}{\mathrm{Cone}}

\newcommand{\Proj}{\mathrm{Proj}}
\newcommand{\pr}{\mathrm{pr}}
\newcommand{\seq}[1]{\{#1\}_{j=1}^\infty}
\newcommand{\hF}{{\widehat F}}
\newcommand{\hI}{{\widehat I}}
\newcommand{\comb}[2]{\left(\begin{array}{c}#1\\#2
\end{array}\right)}
\newcommand{\AnSep}{\cA_{0,\Sep}}
\newcommand{\AnPSep}{\cA_{0,\PSep}}
\newcommand{\MCS}{\mathrm{MCS}}
\newcommand{\Path}{\mathrm{Path}}
\newcommand{\length}{\mathrm{length}}
\newcommand{\res}{\mathrm{Res}}
\newcommand{\GV}{G^{(V)}}

\newcommand{\GW}{G^{(W)}}
\newcommand{\GU}{G^{(U)}}
\newcommand{\GVW}{G^{(V;W)}}

\newcommand{\GWU}{G^{(W;U)}}
\newcommand{\WP}{P^{(W)}}

\newcommand{\VQ}{Q^{(V)}}

\newcommand{\UO}{O^{(U)}}
\newcommand{\subord}{\mathrm{Subord}}
\newcommand{\Enrich}{\mathrm{Enrich}}
\newcommand{\Fl}{\mathrm{Fl}}
\newcommand{\ovec}{\overrightarrow}
\newcommand{\supp}{\mathrm{supp}}
\newcommand{\abs}{\left|\right|}
\newcommand{\Sing}{\mathrm{Sing}}

\newcommand{\FBCone}{{\overline{\mathrm{Cone}}}}

\newcommand{\vertsupp}{\mathrm{vertsupp}}

\theoremstyle{definition}
\newtheorem{thm}{Theorem}[section]

\newtheorem{dfn}[thm]{Definition}
\theoremstyle{definition}

\newtheorem{conj}[thm]{Conjecture}
\newtheorem*{rmk}{Remark}
\newtheorem{cor}[thm]{Corollary}

\begin{document}

\author[Chris Connell]{Chris Connell}

\author[Yuping Ruan]{Yuping Ruan}

\author[Shi Wang]{Shi Wang}

\title{Simplicial volume and isolated, closed totally geodesic submanifolds of codimension one}

\address{Indiana University} 
\email{cconnell@indiana.edu} 

\address{ShanghaiTech University}
\email{shiwang.math@gmail.com}

\address{Northwestern University}
\email{ruanyp@northwestern.edu}
\subjclass[2020]{Primary 53C23; Secondary 57R19, 20F65}

\begin{abstract}
We show that for any closed Riemannian manifold with dimension at least two and with nonpositive curvature, if it admits an isolated, closed totally geodesic submanifold of codimension one, then its simplicial volume is positive. As a direct corollary of this, for any nonpositively curved analytic manifold with dimension at least three, if its universal cover admits a codimension one flat, then either it has non-trivial Euclidean de Rham factors, or it has positive simplicial volume. 
\end{abstract}
\maketitle
\tableofcontents


\thispagestyle{empty} 

\section{Introduction}
\subsection{Main results}
Let $M$ be a connected, oriented, closed topological manifold. The \emph{simplicial volume} of $M$, defined as
\[||M||:=\inf\{\sum_{i=1}^\ell |a_i|\;:\;\sum_{i=1}^\ell a_i \sigma_i \textrm{ is a singular real chain representing the fundamental class}\},\]
is a homotopy invariant introduced by Gromov \cite{Gromov82} and Thurston \cite{Thurston97}. One of the main themes in the study of simplicial volume is to investigate its interplay with the underlying Riemannian structures on $M$. For example, if $M$ admits a metric of nonnegative Ricci curvature, then the fundamental group is virtually abelian \cite{Wilking00}, so $||M||=0$ \cite{Gromov82}. {More generally, Gromov \cite{Gromov82} shows that this vanishing occurs whenever the fundamental group of $M$ is amenable.} On the other hand, if $M$ admits a nonpositively curved metric, then philosophically speaking, $||M||>0$ holds whenever $M$ has enough negative {sectional} curvature. Many results have been obtained towards this direction, including negatively curved manifolds \cite{Gromov82, Thurston97, InoueYano82}, locally symmetric spaces of noncompact type \cite{LafontSchmidt06, ConnellFarb03, Bucher07}, geometric rank one manifolds which has a point of certain rank/curvature condition \cite{ConnellWang19, ConnellWang20}, manifolds whose fundamental groups are hyperbolic \cite{Mineyev01} or relative hyperbolic with respect to lower dimensional peripheral subgroups \cite{MineyevYaman, Franceschini}.

Let $M$ be a connected, closed, oriented Riemannian manifold of nonpositive curvature with dimension at least two. We denote $X$ the Riemannian universal cover of $M$ with distance function $d$ and $\Gamma$ the fundamental group of $M$. We say a totally geodesic submanifold $N\subset M$ is \emph{isolated} if any lift $F$ of $N$ in $X$ satisfies the following two properties:
\begin{enumerate}
	\item (\textbf{$N$ is embedded.}) $\gamma F\ints F=\emptyset$, for any $\gamma\in\Gamma\setminus\mathrm{Stab}_F(\Gamma)$. 
	\item (\textbf{No separate parallel geodesics.}) For any bi-infinite geodesic $c(t)$ in $X$, if $d(c(t),F)$ is uniformly bounded, then $c(t)\in F$ for all $t\in \RR$. 
\end{enumerate}
Our main result is the following.
\begin{thm}\label{thm:main}
Let $M$ be a connected, closed, oriented Riemannian manifold of nonpositive curvature with dimension at least two. If $M$ admits an isolated, codimension one, closed totally geodesic submanifold $N$, then $||M||>0$.
\end{thm}
\begin{rmk}
The isolated condition implies for any lift $F$ of $N$, $\partial F\subset \partial X$ is disconnected to its complement under the Tits metric. Moreover, one can observe some hyperbolicity in the orthogonal direction of $F$.  (See (K1) and (K2) in \hyperlink{idea-and-plan}{the outline and plan of the proof}.) This notion of hyperbolicity is the key to prove Theorem \ref{thm:main}.
\end{rmk}
\begin{cor}\label{cor:main-1}
Let $M$ be an analytic Riemannian manifold of nonpositive curvature with dimension at least three. If the universal cover of $M$ admits a flat of codimension one, then either of the following holds:
\begin{enumerate}
\item $M$ has non-trivial Euclidean de Rham factors. In particular, $\|M\|=0$.
\item $\|M\|>0$.
\end{enumerate}
\end{cor}
Besides the family of examples for Theorem \ref{thm:main} mentioned in Corollary \ref{cor:main-1}, other examples which satisfy the condition of Theorem \ref{thm:main} include all nonpositively curved, analytic $4$-manifolds with a $3$-dimensional, isolated maximal higher rank submanifold (discussed in Subsection \ref{subsect:4 dim story} with more details), and those by Nguyen-Phan \cite{Phan13} using the smooth hyperbolization technique of Ontaneda \cite{Ontaneda20}. Concrete examples of nonpositively curved, analytic $4$-manifolds with a $3$-dimensional, isolated maximal higher rank submanifold are established in \cite{Schroeder91,AbreschSchroeder}. (In fact, \cite{AbreschSchroeder} constructs a more general family of analytic manifolds of nonpositive curvature, not necessarily restricting to dimension $4$. Some of these examples with dimension greater than $4$ also satisfy the condition of Theorem \ref{thm:main}.) We note that the simplicial volume of the examples constructed in \cite{Phan13,Schroeder91,AbreschSchroeder} are already known to be positive due to the work of \cite{ConnellWang20}.

On the other hand, generalized graph manifolds (in the sense of \cite{FrigerioEtAl15}) which do not have hyperbolic pieces have zero simplicial volume. One can verify that such manifolds do not satisfy the condition of Theorem \ref{thm:main}. (See Subsection \ref{subsect:graph mflds}.)

\subsection{Connections with relative hyperbolicity}\label{subsect:rel hyp}
In \cite[Theorem 1.2.1]{HruskaKleiner}, Hruska and Kleiner show that for a CAT(0) space $X$ with a geometric action of $\Gamma$, it is relatively hyperbolic to a $\Gamma$-invariant family of flats if and only if the following holds:
\begin{itemize}
\item There exists $D>0$ such that any flat is contained in the $D$-tubular neighborhood of a flat in this family. (See \cite[(IF2)-(1)]{HruskaKleiner}.)
\item These flats are ``isolated'' from each other. (See \cite[(IF2)-(2)]{HruskaKleiner}.)
\end{itemize} 
In the appendix of the same paper, Hindawi and the same authors obtain a generalized result for families of certain convex subspaces, replacing families of flats. In particular, if $X$ is the universal cover of a compact, analytic manifold $M$ of nonpositive curvature and one of the convex subspaces is the universal cover of a closed, embedded, totally geodesic submanifold $N$, then $N$ is isolated in $M$ in the sense of Theorem \ref{thm:main}. If in addition that $\dim(N)=\dim(M)-1$ and that all convex subspaces are lifts of closed, totally geodesic submanifolds of $M$, positivity of $\|M\|$ can be proved by either \cite{MineyevYaman,Franceschini} using relative hyperbolicity, or Theorem \ref{thm:main}. 

On the other hand, compared to the notion of being ``isolated'' in \cite{HruskaKleiner}, the isolated condition in Theorem \ref{thm:main} is more local. Roughly speaking, if $X$ has a family of convex subspaces which are ``isolated'' from each other in the sense of \cite{HruskaKleiner}, then all the ``flat behavior'' in $X$ are ``trapped'' by these convex subspaces. (See \cite[Theorem A.0.1, (A)]{HruskaKleiner}.) However, in the setting of Theorem \ref{thm:main}, if $M$ admits an isolated, totally geodesic submanifold $N$ of codimension one, ``flat behavior'' may still occur in two other subspaces which are away from any lift of $N$ and are not isolated from each other. A rich family of examples come from $4$-dimensional analytic manifolds of nonpositive curvature with rank one. See the next subsection or \cite{Schroeder89}.

\subsection{Connections with $4$-dimensional analytic manifolds of nonpositive curvature.}\label{subsect:4 dim story}
In \cite{Schroeder89}, Schroeder classifies all maximal higher rank submanifolds in the universal cover of a closed, rank one, $4$-dimensional analytic manifold of nonpositive curvature: All maximal higher rank submanifolds are closed. Moreover, for any maximal higher rank submanifold $F$, one of the following holds:
\begin{enumerate}
\item $F$ is a $2$-dimensional flat.
\item $F$ is a $3$-dimensional flat.
\item $F$ is isometric to $\Sigma\times \RR$, where $\Sigma$ is a non-flat $2$-dimensional Hadamard manifold. There are two cases:
\begin{enumerate}
\item[(3a)] $F$ does not intersect any other maximal higher rank submanifold of the same type.
\item[(3b)] $F$ intersects another maximal higher rank submanifold of the same type.
\end{enumerate}
\end{enumerate}
Concrete examples of rank one, $4$-dimensional analytic manifolds of nonpositive curvature containing any of the above maximal higher rank submanifolds can be found in \cite{Schroeder91,AbreschSchroeder}.

If the universal cover of $M$ does not admit maximal higher rank submanifolds of type (3b), then all maximal higher rank submanifolds are isolated from each other in the sense of \cite{HruskaKleiner}. Hence by \cite[Theorem A.0.1, (5)]{HruskaKleiner}, $\pi_1(M)$ is relatively hyperbolic to the collection of stabilizers of all maximal higher rank submanifolds. The simplicial volume $\|M\|$ is then positive due to \cite{MineyevYaman,Franceschini}.

One can check that any maximal higher rank submanifold of type (2) or type (3a) is the lift of an isolated, codimension one, closed totally geodesic submanifold in $M$ (in the sense of Theorem \ref{thm:main}). Therefore we have the following corollary of Theorem \ref{thm:main}:
\begin{cor}\label{cor:main-2}
If the universal cover of a closed, rank one, nonpositively curved $4$-dimensional analytic manifold $M$ admits maximal higher rank submanifolds of type (2) or type (3a), then $\|M\|>0$.
\end{cor} 
In this case, maximal higher rank submanifolds of type (3b) are allowed. If there exists a maximal higher rank submanifold of type (3b), by \cite[Theorem A.0.3]{HruskaKleiner} and \cite{Schroeder89}, $\pi_1(M)$ fails to be relatively hyperbolic to any $\pi_1(M)$-invariant sub-collection of the stabilizers of maximal higher rank submanifolds. (Indeed, if the sub-collection misses a stabilizer of a maximal higher rank submanifold, then $\pi_1(M)$ violates \cite[Theorem A.0.3, (1)]{HruskaKleiner}. Otherwise, $\pi_1(M)$ violates \cite[Theorem A.0.3, (2)]{HruskaKleiner} due to the existence of maximal higher rank submanifolds of type (3b).)

We recall the following two conjectures due to Gromov \cite{SavageRichardP82} and \cite[p. 232]{GromovMikhail96}.
\begin{conj}\label{conj:Ricci}
If $M$ admits a Riemannian metric with nonpositive sectional curvature and negative definite Ricci curvature, then $||M||>0$.
\end{conj}
\begin{conj}\label{conj:Euler}
If $M$ is aspherical and $||M||=0$, then the Euler characteristic of $M$ is zero.
\end{conj}
The local straightening technique in \cite{ConnellWang20} suggests that the Ricci condition in Conjecture \ref{conj:Ricci} need only be satisfied at one point, thus a stronger conjecture is proposed as follows \cite[Conjecture 4.1]{ConnellWang20}.
\begin{conj}\label{conj:Ricci-strong}
If $M$ admits a Riemannian metric with nonpositive sectional curvature everywhere and negative definite Ricci curvature at some point, then $||M||>0$.
\end{conj}
In particular, based on the above conjectures, we believe that for any closed, rank one, analytic manifold $M$ of nonpositive curvature with dimension at least two, the simplicial volume $\|M\|$ is positive. When $\dim(M)=4$, by the above discussion, the only remaining case is the following:
\begin{itemize}
\item The only maximal higher rank submanifolds in the universal cover of $M$ are of type (1) or type (3b).
\item There exists at least one maximal higher rank submanifold of type (3b) in the universal cover of $M$. 
\end{itemize}
\subsection{Comments on graph manifolds}\label{subsect:graph mflds}
According to \cite{FrigerioEtAl15}, a closed, smooth $n$-dimensional Riemannian manifold $M$ is a generalized graph manifold if it is constructed in the following way: (See \cite{ConnellSuarez} for more general notions of generalized graph manifolds.)
\begin{enumerate}
\item For every $i=1,...,r$, let $N_i$ be a complete, finite-volume, $n_i$-dimensional, non-compact hyperbolic manifold with toric cusps. Here, $0\leq n_i\leq n$. 
\item Remove the horospherical neighborhood of each cusp in $N_i$ to obtain $\overline{N}_i$.
\item Under a suitable pairing, glue $V_i:=\overline{N}_i\times\TT^{n-n_i}$ together to form $M$, where $\TT^{k}$ is the $k$-dimensional torus. The metric near the boundary of $\overline{N}_i$ is modified so that $\overline{N}_i$ is a nonpositively curved manifold with totally geodesic boundary.
\end{enumerate}
$V_1,...,V_r$ are called \emph{pieces} of $M$. The manifold $\overline{N}_i$ is the \emph{base} of $V_i$, while every subset of the form $\{*\}\times \TT^{n-n_i}$ is a \emph{fiber} of $V_i$. A piece of $M$ is called a \emph{pure} piece if its fibers are trivial. \emph{Surface} pieces of $M$ refer to those pieces with $2$-dimensional bases.
\begin{figure}[h]
	\centering
	\includegraphics[width=4in]{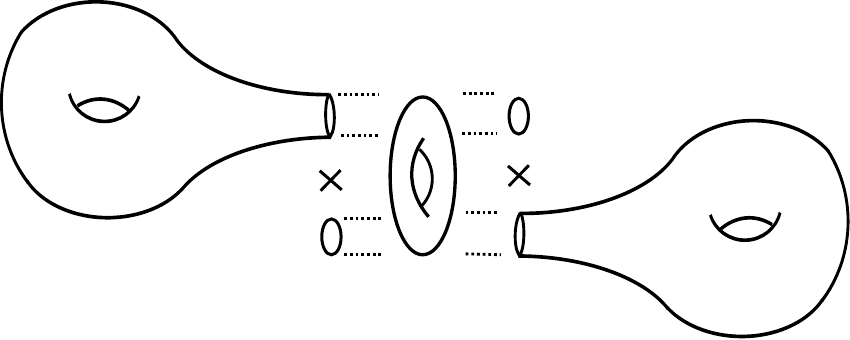}
	\caption{ \label{graphmfld}}
\end{figure}

When a generalized graph manifold $M$ has a pure piece, the simplicial volume of $M$ is positive due to \cite{ConnellSuarez,ConnellWang20}. On the other hand, as a corollary of \cite[Theorem 1]{ConnellSuarez}, $\|M\|=0$ when it has no pure pieces. One can also verify that in this case, $M$ does not have any closed, isolated totally geodesic submanifold of codimension one. If not, there exists a closed, isolated totally geodesic submanifold of codimension one $N$ in $M$. Then, one of the following must hold:
\begin{itemize}
\item $N$ is a gluing torus.
\item $N$ transversally intersects a gluing torus.
\item $N$ is trapped in a piece of $M$.
\end{itemize}
When $M$ has no pure pieces, any of the above contradicts the isolated assumption on $N$. 

One final comment is that Theorem \ref{thm:main} fails easily if we drop the codimension one assumption. Consider the classical $3$-dimensional graph manifold obtained by gluing two copies of the same surface piece. (See Figure \ref{graphmfld}.) With a careful choice of the surface piece, one can find a rank one periodic geodesic which visits both pieces and is not self-intersecting. In particular, this periodic geodesic is isolated. Yet the simplicial volume of this graph manifold is zero since it does not have any pure pieces.
\subsection{Outline of the proof and plan of the paper}\label{subsect:idea and plan}
\hypertarget{idea-and-plan}{We} first briefly review two approaches to prove positivity of simplicial volume:
\begin{enumerate}
\item[(A1)] One way to prove positivity of simplicial volume is to show that for any singular simplex $\sigma$ of top dimension, one can construct a ``straightened simplex'' $\mathrm{st}(\sigma)$ such that the following holds:
\begin{itemize}
\item The ``volume'' of $\mathrm{st}(\sigma)$ is uniformly bounded from above;
\item For any $\sum_{i=1}^la_l\sigma_l$ representing the fundamental class, $\sum_{i=1}^la_l\mathrm{st}(\sigma_l)$ also represents the same class.
\end{itemize}
This strategy shows up in many earlier works. For example, when proving positivity of simplicial volume for negatively curved manifolds, \cite{Gromov82, Thurston97} use geodesic simplices as ``straightened simplices'' and show that there is a uniform upper bound on their Riemannian volume. A similar argument is used for locally symmetric spaces of non-compact type in \cite{LafontSchmidt06} using a notion of barycentric simplices instead. Later, \cite{ConnellWang20} notices that for any nonpositively curved manifold which has a point of strict negative curvature, the above barycentric simplices have uniform volume control near this special point. Hence the simplicial volume is positive in this case. Sufficient negative curvature is crucial when verifying the uniform upper bound on the volumes of ``straightened simplices''.
\item[(A2)] There is a parallel approach without using the manifold structure. Let $Y$ be the geometric realization of the homogeneous bar-construction for $\pi_1(M)$ (See \cite[Example 1B.7]{Hatcher02}.) In particular, $Y$ is a $K(\pi_1(M),1)$-space. In the case when $\pi_1(M)$ is hyperbolic, \cite{Mineyev01} shows that for any $k$-dimensional simplex $\sigma$ in $Y$ with $k\geq 2$, one can construct a ``straightened simplex'' $\mathrm{st}_k(\sigma)$ with dimension $k$ such that the following holds:
\begin{itemize}
\item There are finitely many $k$-dimensional simplices $\sigma_1,...,\sigma_l$ such that $\mathrm{st}_k(\sigma)$ is a linear combination of elements in $\pi_1(M)\cdot\{\sigma_1,...,\sigma_l\}$;
\item The $l^1$-norm of $\mathrm{st}_k(\sigma)$ is uniformly bounded from above;
\item $\mathrm{st}_\bullet(\sigma)$ can be naturally extended to a $\pi_1(M)$-equivariant linear map which is chain homotopic to the identity.
\end{itemize}
Positivity of $\|M\|$ then follows immediately from the properties of $\mathrm{st}_{\mathrm{dim}(M)}$. The construction of $\mathrm{st}_\bullet$ in \cite{Mineyev01} relies on a notion of ``homological bicombing'' so that for any $2$-simplex in $Y$, the ``straightening'' of its boundary has a uniform $l^1$-norm upper bound. Hyperbolicity of the fundamental group is heavily used to construct ``homological bicombings''.
\end{enumerate}

In this paper and under the setting of Theorem \ref{thm:main}, we provide a different way to ``straighten'' singular simplices on $M$ using ideas from both (A1) (simplices constructed via barycenter methods and the use of top form supported in a subset from \cite{ConnellWang20}) and (A2) (homological bicombings from \cite{Mineyev01}): For any $k$-dimensional singular simplex $\sigma$ with $k\geq 0$, we construct a ``straightened simplex'' $\mathrm{st}_{k}(\sigma)$ such that the following holds:
\begin{enumerate}
\item[(S1)] When $k\geq 2$, any lift of $\mathrm{st}_{k}(\sigma)$ in the universal cover of $M$ is a linear combination of $k$-dimensional special barycentric simplices with $l^1$-norm uniformly controlled from above;
\item[(S2)] If a lift of a top dimensional simplex $\sigma$ is a special barycentric simplex, then its volume near an annulus region centered at the codimension one submanifold is uniformly bounded. This is due to the following properties of special barycentric simplices (defined in Definition \ref{special bar simplices}):
\begin{itemize}
\item There exists a uniform constant $\cC(k)>0$ such that any $k$-dimensional barycentric simplex is close to at most $\cC(k)$ lifts of the codimension one submanifold; (This is the key property which distinguishes special barycentric simplices from barycentric simplices. Special barycentric simplices are alway ``small''.)
\item For any lift $F$ of the codimension one submanifold and any barycentric simplex ${\sigma}$ with vertices away from $F$, the image of ${\sigma}$ intersecting an annulus region centered at $F$ is bounded by a ball of uniform radius;
\item Any barycentric simplex is ``almost'' injective. (It is not clear if geodesic simplices satisfy this property. This is the main reason we choose barycentric simplices instead of geodesic simplices.)
\end{itemize}
\item[(S3)] For any $\sum_{i=1}^la_l\sigma_l$ representing the fundamental class, $\sum_{i=1}^la_l\mathrm{st}_{\mathrm{dim}(M)}(\sigma_l)$ also represent the same class.
\end{enumerate}
Roughly speaking, the driving mechanism behind all these constructions and properties is the ``\emph{hyperbolicity away from the codimension one submanifold}''. To be specific, under the setting of Theorem \ref{thm:main}, for any lift $F$ of the codimension one submanifold, the following key properties hold:
\begin{enumerate}
\item[(K1)] For any geodesic right triangle in the universal cover of $M$, if one edge of the right angle lies in $F$, then it is $\delta$-thin for some uniform choice of $\delta$; (See Lemma \ref{hyperbolicity of right triangles}.)
\item[(K2)] For any geodesic segment $[p,q]$ connecting two points $p,q$ in the universal cover of $M$, if $p,q$ are away from $F$, then the distance between $[p,q]$ and $F$ is closely related to the distance between the orthogonally projected images of $p$ and $q$ onto $F$. (The codimension one assumption and the previous bullet point are heavily used here. See Lemma \ref{key hyp lem}.)
\end{enumerate}
For any barycentric simplex in the universal cover of $M$, it satisfies the last two bullet points in (S2) as long as its vertices are away from all lifts of the codimension one submanifold. (The second bullet point for $\sigma$ is due to the above discussion, proved in Lemma \ref{bar bded ints away from flats}; the third bullet point is essentially due to the definition of barycentric simplices. See Corollary \ref{bar almost injectivity}.) However, it may not satisfy the first bullet point in (S2) if it has a huge size. (For example, if one edge is very close to too many lifts of the codimension one submanifold, then the first bullet point in (S2) is not satisfied.) The rough idea behind constructing the ``straightened simplex'' $\mathrm{st}_k(\sigma)$ is the following:
\begin{itemize}
\item Let $\widetilde{\sigma}$ be a lift of $\sigma$ in the universal cover. We ``move'' the vertices of $\widetilde{\sigma}$ away from all lifts of the codimension one submanifold (in order to use (K2) and all lemmas which rely on (K2). This step is standard and is mostly done in the construction of $\psi_\bullet$ in Subsection \ref{subsect:phi and psi}.)
\item We ``cut'' the barycentric simplex with the same vertices as $\widetilde{\sigma}$ and its lower dimensional faces into ``small polygonal pieces''. Here, the ``cutting'' is done by lifts of the codimension one submanifold which ``almost separate'' the $1$-skeleton into two disjoint pieces. With the help of (K1) and (K2), the notion of ``almost separation'' satisfies some similar properties compared to ``actual separation''. (See Proposition \ref{almost ints positioning}.) Similar to ``almost separation'', for any two different lifts $F_1$ and $F_2$ of the codimension one submanifold, we define all lifts of the codimension one submanifold which are ``almost between'' $F_1$ and $F_2$. (See Subsection \ref{subsect:Omega and Theta}.)  The aforementioned notion of ``almost separation'' is subsequently improved in Section \ref{sec combinatorics} and exhibits more properties similar to ``actual separation''.
\item We modify $1$-dimensional pieces via homological bicombings inspired by \cite{Mineyev01}. This is crucial when verifying (S1). As is stated in (A2), the bicombing construction in \cite{Mineyev01} relies on hyperbolicity. Here, our bicombing construction relies on a notion of hyperbolicity on all lifts of the codimension one submanifold. (The origin of this notion of hyperbolicity is essentially (K1) and (K2).) The bicombing construction heavily relies on the language of the aforementioned ``almost between'' relation among lifts of the codimension one submanifold. (See Section \ref{sec bicombing}.)
\item We ``refill'' each ``polygonal piece'' by special barycentric simplices inductively from dimension $2$ to dimension $k$. The ``refilling'' process is essentially done by triangulating  each ``polygonal piece'' in a consistent way. (See Section \ref{sec hom arg}.)
\end{itemize}
The rest of this paper is organized as follows: The basic setting and notations can be found in Section \ref{sect:setting}. Section \ref{sect:hyperbolicity} mainly focuses on the hyperbolicity away from the codimension one submanifold. The concepts of barycentric simplices, the aforementioned ``almost separation'' and ``almost between'' relations among lifts of the codimension one submanifold are introduced in Section \ref{sec:bar simplex and almost sep} along with their properties. Section \ref{sec bicombing} discusses homological bicombings. Section \ref{sec combinatorics} is the most technical part of this paper, which explains the combinatorics behind the way we ``cut'' simplices into ``small polygonal pieces''. In particular, we construct a graph related to how those lifts of the codimension one submanifold ``almost separate''  the $1$-skeleton of a barycentric simplex and its lower dimensional faces. The vertices roughly correspond to ``how the simplex is being cut'' and the maximal complete subgraphs roughly correspond to the ``small polygonal pieces'' after the ``cutting''. The technicalities mostly come from certain bizzare phenomena in ``almost separation'' which do not show up in ``actual separation''. Section \ref{sec hom arg} mostly focuses on how ``polygonal pieces'' are refilled. Theorem \ref{thm:main} is finally proved in Section \ref{last sec} with all the previous preparations.
\medskip

\noindent\textbf{Acknowledgements}. We would like to thank Ralf Spatzier for helpful discussions. C. C. was supported in part by Simons Foundation collaboration grant \#965245. S. W. was supported in part by NSFC grant \#12301085. C. C and S. W. would like to thank the Max Planck Institute for Mathematics whose hospitality was enjoyed during some of this work.

\section{Setting and basic notations}\label{sect:setting}
Let $M$ be a compact, nonpositively curved manifold of dimension $n\geq 2$ and $\Gamma=\pi_1(M)$ be its fundamental group. Let $X$ be the universal cover of $M$ and hence $M=\Gamma\backslash X$. Denoted by $\del X$ the geodesic boundary of $X$. Let $d(\cdot,\cdot)$ be the distance function on $X$ and $d_T(\cdot,\cdot)$ be the Tits distance on $\del X$. Let $\pi:TX\to X$ be the natural projection which maps every vector to its base point. The covering map of $X$ over $M$ is denoted by $\sP:X\to M$.

Let $N$ be a closed, totally geodesic submanfold of $M$ with codimension one. In addition, we assume that $N$ is \emph{isolated} in the following sense: Let $F$ be a lift of $N$ in $X$. Then $F$ satisfies the following two properties:
\begin{enumerate}
\item[\hypertarget{I-1}{(1).}] $\gamma F\ints F=\emptyset$, for any $\gamma\in\Gamma\setminus\mathrm{Stab}_\Gamma(F)$;
\item[\hypertarget{I-2}{(2).}] For any bi-infinite geodesic $c:\RR\to X$ in $X$, if $d(c(t),F)$ is bounded by a uniform constant from above, then $c(\RR)\subset F$.
\end{enumerate}

Denoted by $[p,q]$ the geodesic segment connecting $p,q\in X$. For any geodesically convex subset $A\subset X$, we denote by $\Proj_A(\cdot)$ the closest point projection onto $A$.

Complicated summations are used in some parts of this paper. To avoid confusions, we denote by
$$\sum_{\substack{x_1,...,x_k:\cP_1(x_1,...,x_k)\\\cP_2(x_1,...,x_k)\\\cdots\cdots\\\cP_l(x_1,...,x_k)}}f(x_1,...,x_k)$$
the summation of $f(x_1,...,x_k)$ over all possible $(x_1,...,x_k)$ satisfying the properties $\cP_j(x_1,...,x_k)$, $j=1,...,l$. We also use similar notations for taking products and taking unions.

\section{Hyperbolicity away from submanifolds}\label{sect:hyperbolicity}
Since $F$ is a totally geodesic submanifold of $X$, the geodesic boundary $\del F$ of $F$ is naturally a subset of $\del X$. Then we have the following lemma.
\begin{lemma}\label{geo rays asymp to flats}
For any geodesic ray $c:[0,\infty)\to X$ with $c(\infty)\in\del F$, we have $\lim_{t\to\infty}d(c(t),F)=0$.
\end{lemma}
\begin{proof}
By the fact that $c(\infty)\in\del F$, $d(c(t),F)$ is uniformally bounded from above. Choose an increasing sequence $\{t_j\}_{j=1}^\infty\subset \RR$ with $t_j\to \infty$ as $j\to\infty$. By compactness of $N$, there exist a sequence $\seq{\gamma_j}\subset\pi_1(N)\isom\mathrm{Stab}_\Gamma(F)$ such that
$\seq{\pi(\gamma_j\dot{c}(t_j))}$ is a bounded sequence. Passing to a subsequence, we can assume that $v_j:=\gamma_{j} \dot{c}(t_j)\to v\in  TX$ as $j\to\infty$. Then the bi-infinite geodesic tangent to $v$ is uniformly close to $F$. By \hyperlink{I-2}{Property (2)} of being isolated, $v\in TF$. In particular
$$0\leq\lim_{j\to\infty}d(c(t_j),F)=\lim_{j\to\infty}d(\gamma_j(c(t_j)),F)=\lim_{j\to\infty}d(\pi(v_j),F)\leq\lim_{j\to\infty}d(\pi(v_j),\pi(v))=0,$$
which implies $\lim_{j\to\infty}d(c(t_j),F)=0$. Since $X$ is nonpositively curved, $d(c(t),F)$ is a convex function in $t$. Hence we have $\lim_{t\to\infty}d(c(t),F)=0$.
\end{proof}

We recall some classical results in the study of CAT(0) spaces, Hadamard manifolds and their Tits boundaries.

\begin{proposition}[{\cite[9.2 Proposition and 9.5. Proposition, (1)\&(2)]{BridsonHaefliger99}}]\label{BH9.2and9.5}
Let $X$ be a complete metric CAT(0) space. $\overline{X}=X\union\del X$.
\begin{enumerate}
\item[(i).] For fixed $p\in X$, the function $(x,x')\to \angle_p(x,x')$ is continuous on $\{(x,x')\in \overline{X}\times\overline{X}|x\neq p\mathrm{~and~}x'\neq p\}$.
\item[(ii).] The function $(p,x,x')\to \angle_p(x,x')$ is upper semicontinuous at points $(p,x,x')\in X\times\overline{X}\times\overline{X}$ with $x\neq p$ and $x'\neq p$.
\item[(iii).] For any $\xi,\xi'\in\del X$, we define $\angle(\xi,\xi')=\sup_{x\in X}\angle_x(\xi,\xi')$. This function defines a metric called the \emph{angular metric}. The extension to $\del X$ of any isometry of $X$ is an isometry of $\del X$ with respect to the angular metric.
\item[(iv).] The function $(\xi,\xi')\to\angle(\xi,\xi')$ is lower semicontinuous with respect to the cone topology: for every $\epsilon>0$, there exist neighborhoods $U$ of $\xi$ and $U'$ of $\xi'$ in the cone topology such that $\angle(\eta,\eta')>\angle(\xi,\xi')-\epsilon$ for all $\eta\in U$ and $\eta'\in U'$.
\end{enumerate}
\end{proposition}
\begin{proposition}[{\cite[9.8 Proposition, (2)]{BridsonHaefliger99}}]\label{BH9.8}
Let $X$ be a complete CAT(0) space with basepoint $x_0$. Let $\xi, \xi'\in\del X$ and let $c,c':[0,\infty)\to X$ be geodesic rays such that $c(0)=c'(0)=x_0$, $c(\infty)=\xi$ and $c'(\infty)=\xi'$. Then $\angle(\xi,\xi')=\lim_{t\to\infty}\angle_{c(t)}(\xi,\xi')$.
\end{proposition}
\begin{proposition}[{\cite[9.21 Proposition]{BridsonHaefliger99}}]\label{BH9.21}
Let $X$ be a proper CAT(0) space and let $\xi_0,\xi_1\in\del X$ be distinct points. Then the following holds:

\begin{enumerate}
\item[(i).] If $d_T(\xi_0,\xi_1)>\pi$, then there is a bi-infinite geodesic $c:\RR\to X$ such that $c(\infty)=\xi_0$ and $c(-\infty)=\xi_1$;
\item[(ii).] If $\xi_0$ and $\xi_1$ are not the infinity endpoints of a bi-infinite geodesic in $X$, then $d_T(\xi_0,\xi_1)=\angle(\xi_0,\xi_1)$ and there is a geodesic segment in $(\del X, d_T)$ joining $\xi_0$ to $\xi_1$. In particular, this applies to any $\xi_0$ and $\xi_1$ such that $\angle(\xi_0,\xi_1)<\pi$.
\end{enumerate}
\end{proposition}
\begin{proposition}[{\cite[Proposition 3.3]{Hindawi}}]\label{Hindawi}
Let $X$ be a Hadamard manifold. Let $p\in X$ and $[x_j,y_j]$ be a sequence of geodesic segments such that $x_j,y_j\in X$ converge to $\xi_x,\xi_y\in\del X$ respectively as $j\to\infty$. If $d(p,[x_j,y_j])\to\infty$ as $j\to\infty$, then $d_T(\xi_x,\xi_y)\leq \pi$. Moreover, $\xi_x$ and $\xi_y$ are connected in the Tits boundary $(\del X,d_T)$.
\end{proposition}
\begin{proof}
The proof is identical to the proof of \cite[Proposition 3.3]{Hindawi} except for the first and the last lines. If $\xi_x$ and $\xi_y$ are not connected in the Tits boundary $(\del X,d_T)$, by Proposition \ref{BH9.21}, $\xi_x$ and $\xi_y$ are connected by a bi-infinite geodesic in $X$. The proof of \cite[Proposition 3.3]{Hindawi} then shows that this bi-infinite geodesic bounds a half flat, which implies that $\xi_x$ and $\xi_y$ are actually connected in the Tits boundary $(\del X,d_T)$. This gives the contradiction.
\end{proof}

\begin{lemma}\label{flat sector lem}
Let $U:=\{(r\cos\theta,r\sin\theta)|r\in[0,\infty)\mathrm{~and~}\theta\in[0,\theta_0]\}\subset\RR^2$ for some $\theta_0\in [0,2\pi]$ and $c:U\to X$ be a metric space embedding, where $U$ is equipped with the standard Euclidean metric on $\RR^2$. (In other words, $c(U)$ is a closed flat sector in $X$). For any $\theta\in[0,\theta_0]$ and $t\geq 0$, we define geodesic rays $c_\theta(t):=c(t\cos\theta,t\sin\theta)$ in $X$. Suppose $c_0(\infty)\in\del F$, then $c_\theta(\infty)\in\del F$ for any $\theta\in[0,\theta_0]$.
\end{lemma}
\begin{proof}
Since we can chop $U$ into sectors with angles less than $\pi/2$, we assume WLOG that $\theta_0<\pi/2$. For simplicity, we let $p:=c(0,0)$. If there exists some $\theta\in (0,\theta_0]$ such that $c_\theta(\infty)\in\del F$, then by convexity of distance function in nonpositive curvature,
$$d(p,F)\geq \max\{d(c_\theta(t),F), d(c_0(t),F)\},~\forall t\in[0,\infty).$$
By convexity of distance function in nonpositive curvature again, this implies that $d(p,F)\geq d(c_{\theta'}(t),F)$ for any $\theta'\in[0,\theta]$. In particular, $c_{\theta'}(\infty)\in\del F$ for any $\theta'\in[0,\theta]$.

Assume the contrary, by the above discussions, there exists some $U$ with $\theta_0<\pi/2$ such that
\begin{align}\label{eqn:only one point in del F}
\theta_0>\sup\{\theta\in[0,\theta_0]|c_\theta(\infty)\in\del F\}=:\theta_1.
\end{align}
WLOG, we assume that $\theta_1=0$. For any $x\in\RR_{\geq 0}$, define
$$L(x):=\sup\{y\in[0,x\tan\theta_0]|d(c(x,y),F)\leq d(p,F)+1\}.$$
By convexity of distance function in nonpositive curvature and the fact that $c_0(\infty)\in\del F$, the following holds:
\begin{itemize}
\item For any $y'\in[0,L(x)]$, $d(c(x,y'),F)\leq d(p,F)+1$;
\item $L(x)$ is concave;
\item $L(x)$ is non-decreasing;
\item If $L(x)<x\tan\theta_0$, then $d(c(x,L(x)),F)=d(p,F)+1$.
\end{itemize}
In particular, the following subset
\begin{align}\label{eqn:good points in sector}
U':=\{(x,y)\in U|d(c(x,y),F)\leq d(p,F)+1\}=\{(x,y)\in\RR^2|x\geq 0\mathrm{~and~}0\leq y\leq L(x)\}
\end{align}
is convex. Define
$$s_L(x):=\sup_{z>y\geq x}\frac{L(z)-L(y)}{z-y},~\forall x\in\RR_{\geq 0}.$$
Then for any $z>y\geq x\geq 0$, we have
\begin{align}\label{eqn:convexity of L}
s_L(x)\geq s_L(y)\geq\frac{L(z)-L(y)}{z-y}\geq s_L(z)\geq s:=\lim_{t\to\infty}s_L(t)\geq 0.
\end{align}
In particular, by \eqref{eqn:good points in sector}, for any $\theta\in[0,\arctan(s)]$, $c_\theta(\RR_{\geq 0})\in U'$. By the assumption that $\theta_1=0$ (see \eqref{eqn:only one point in del F}), we have $s=0$.

For any $l\in\ZZ_+$, we choose $x_l\in\RR_+$ such that $s_L(x_l)-s\leq (1/2l)\cdot\min\{1,\tan\theta_0\}$ and $L(x_l)<x_l\tan\theta_0$. (Existence of $x_l$ is guaranteed by \eqref{eqn:convexity of L}.) Let $y_l=x_l+2l\cdot ((d(p,F)+1)/2)$, $p_l:=c(x_l,L(x_l))$, $q_l:=c(y_l,L(y_l))$ and $m_l:=c((x_l+y_l)/2,(L(x_l)+L(y_l))/2)$ for any $l\in\ZZ_+$. Then the following holds:
\begin{itemize}
\item $m_l$ is the midpoint of the geodesic segment $[p_l,q_l]$ connecting $p_l$ and $q_l$ for any $l\in\ZZ_+$;
\item Length of $[p_l,q_l]$ is greater than or equal to $2l\cdot ((d(p,F)+1)/2)$;
\item For any $x'\geq x_l$, $L(x')<x'\tan\theta_0$ due to \eqref{eqn:convexity of L};
\item For any point $q\in[p_l,q_l]$, there exists some $q'\in c(U')$ such that $d(q',F)=d(p,F)+1$ and $d(q,q')\leq (d(p,F)+1)/2$. In particular, $d(p,F)+1\geq d(q,F)\geq (d(p,F)+1)/2$. (See Figure \ref{lem3.6}.)
\end{itemize}

\begin{figure}[h]
	\centering
	\includegraphics[width=4in]{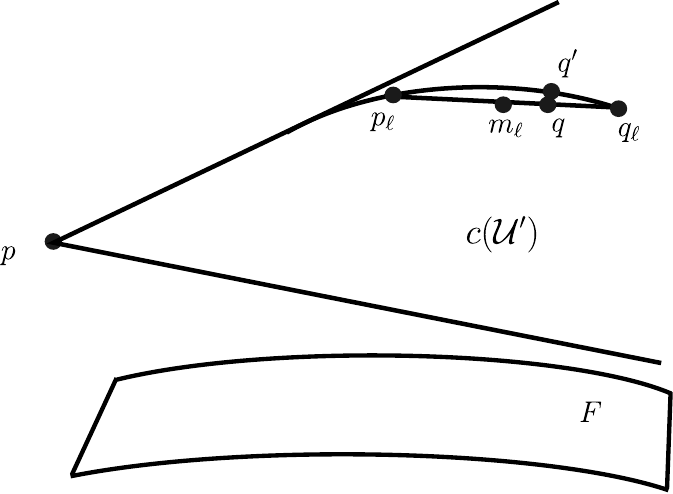}
	\caption{ \label{lem3.6}}
\end{figure}

Let $v_l\in T^1_{m_l}X$ be a unit vector tangent to $[p_l,q_l]$ for any $l\in\ZZ_+$. Then there exists a sequence $\{\gamma_l\}_{l\in\ZZ_+}\subset\mathrm{Stab}_\Gamma (F)$ such that $\{\gamma_lm_l\}_{l\in\ZZ_+}$ and $\{\gamma_lv_l\}$ are bounded sequences. Up to a subsequence, we can assume that $\gamma_lm_l\to m\in X$ and $\gamma_lv_l\to v\in T^1_mX$ as $l\to \infty$. Let $\gamma_{(l)}$ and $\gamma$ be geodesics with initial vectors $v_l$ and $v$ respectively, $l\in\ZZ_+$. Since $d(\gamma_{(l)}(t),F)\in[(d(p,F)+1)/2,d(p,F)+1]$ for any $l\in\ZZ_+$ and any $t$ such that $|t|\leq l\cdot ((d(p,F)+1)/2)$, we have
$$d(\gamma(t),F)\in[(d(p,F)+1)/2,d(p,F)+1],~\forall t\in\RR.$$
This is impossible due to the assumption that $N$ is isolated. (See Property \hyperlink{I-2}{(2)} of being isolated.)
\end{proof}
\begin{lemma}\label{res angle lem}
For any $R>0$, there exists some $\upsilon(R)>0$ such that for any $p\in X$ and $q_1,q_2\in F$ satisfying $d(p,F)=R$, the angle $\angle_p(q_1,q_2)\leq \pi-\upsilon(R)$.
\end{lemma}
\begin{proof}
If not, there exists some $R>0$ and sequences of points $\{p_l\}_{l\in\ZZ_+}\subset X$, $\{q_{1;l}\}_{l\in\ZZ_+}\subset F$ and $\{q_{2;l}\}_{l\in\ZZ_+}\subset F$ such that $d(p_l,F)=R$ for any $l\in\ZZ_+$ and $\angle_{p_l}(q_{1;l},q_{2;l})\to\pi$ as $l\to\infty$. Let $v_{1;l},v_{2;l}\in T_{p_l}^1X$ be initial vectors of the geodesic rays from $p_l$ to $q_{1;l}$ and $q_{2;l}$ respectively. Then there exist a sequence $\{\gamma_l\}_{l\in\ZZ_+}\subset\mathrm{Stab}_\Gamma (F)$ such that the sequences $\{\gamma_lp_l\}_{l\in\ZZ_+}$, $\{\gamma_lv_{1;l}\}_{l\in\ZZ_+}$ and $\{\gamma_lv_{2;l}\}_{l\in\ZZ_+}$ are bounded sequences. After passing to a subsequence, we can assume that $\gamma_lp_l\to p$, $\gamma_lv_{1;l}\to v_1\in T^1_pX$ and $\gamma_lv_{2;l}\to v_2\in T^1_pX$ as $l\to \infty$. Since $d(p_l,F)=R$ for any $l\in\ZZ_+$, and $\angle_{p_l}(q_{1;l},q_{2;l})\to\pi$ as $l\to\infty$, we have $d(p,F)=R$ and $v_1=-v_2=:v$. Let $\gamma$, $\gamma_{1;l}$ and $\gamma_{2;l}$ be geodesic with initial vectors $v$, $v_{1;l}$ and $v_{2;l}$ respectively for any $l\in\ZZ_+$. Since $\overline{F}=F\cup\del F$ is compact, up to a subsequence, we can assume that $q_{1;l}\to\xi_1\in\overline{F}$ and $q_{2;l}\to\xi_2\in\overline{F}$ as $l\to\infty$. Since $v_1=-v_2$, we first observe that $\xi_1,\xi_2\in\del F$. (Otherwise, there exists some segment of $\gamma$ containing $p$ such that either it is a finite segment connecting two points on $F$, or it is a geodesic ray starting at a point on $F$ with infinity at $\del F$. Since $d(p,F)=R>0$, both cases are impossible due to the convexity of distance function in nonpositive curvature.) Assume that $q_{1;l}=\gamma_{1;l}(t_{1;l})$ and $q_{2;l}=\gamma_{2;l}(t_{2;l})$ for some $t_{1;l},t_{2;l}>0$. In particular, $t_{1;l},t_{2;l}\to \infty$ as $l\to \infty$. Notice that $d(\gamma_{1;l}(t),F)$ is decreasing when $t\leq t_{1;l}$ and $d(\gamma_{2;l}(t),F)$ is decreasing when $t\leq t_{2;l}$,  Hence $d(\gamma(t),F)$ admits a global maximum on $t\in\RR$. By convexity of distance functions in nonpositive curvature, $d(\gamma(t),F)\equiv R>0$. This is impossible due to the assumption that $N$ is isolated. (See Property \hyperlink{I-2}{(2)} of being isolated.)
\end{proof}
\begin{lemma}\label{isolated Tits bdry}
$\del F$ is disconnected from $\del X\setminus \del F$ in the Tits boundary $(\del X, d_T)$.
\end{lemma}
\begin{proof}
If not, there exists some $\eta\in\del X\setminus\del F$ such that $d_T(\eta,\del F)<\infty$. Since the Tits distance is a length metric, we can assume WLOG that $d_T(\eta,\del F)<\pi/2$. Let $\seq{\xi_j}\subset \del F$ be a sequence of points such that $\seq{d_T(\eta,\xi_j)}$ is a decreasing sequence with limit equal to $d_T(\eta,\del F)$. WLOG we assume that $\angle(\eta,\xi_j)=d_T(\eta,\xi_j)<\pi/2$. By triangle inequality, $\angle(\xi_j,\xi_l)<\pi$ for any $j,l>0$. Therefore $\seq{\xi_j}$ is a bounded sequence in $\del F$ with respect to the cone topology. By compactness of $\del F$ in the cone topology, we can assume WLOG that $\xi_j$ converge to $\xi\in \del F$ as $j\to \infty$ with respect to the cone topology. By the fourth assertion in Proposition \ref{BH9.2and9.5}, $\angle(\eta,\xi)-\epsilon\leq\lim_{j\to\infty}\angle(\eta,\xi_j)=d_T(\eta,\del F)$ for any $\epsilon>0$. Hence $\angle(\eta,\xi)=d_T(\eta,\xi)\leq d_T(\eta,\del F)$. By the fact that $\xi\in\del F$, we have
\begin{align}\label{eqn:min dist exists}
d_T(\eta,\xi)=d_T(\eta,\del F).
\end{align}

Let $p\in F$ and let $c(t)$, $c'(t)$ be geodesic rays such that $c(0)=c'(0)=p$, $c(\infty)=\xi$ and $c'(\infty)=\eta$. Let $\seq{t_j}\subset \RR_{\geq 0}$ be an increasing sequence such that $t_j\to\infty$ as $j\to\infty$. Since $\eta\not\in\del F$, we can let $q_j\in[c(t_j),\eta]$ such that $d(q_j,F)=1$. (See Figure \ref{lem3.8}.) By the fact that $X$ is nonpositively curved, we have
$$\angle_{q_j}(\xi,\eta)=\pi-\angle_{q_j}(\xi,c(t_j))\geq \pi-(\pi-\angle_{c(t_j)}(\xi,\eta))=\angle_{c(t_j)}(\xi,\eta).$$
Therefore by Proposition \ref{BH9.8},
$$\angle(\xi,\eta)\geq \limsup_{j\to\infty}\angle_{q_j}(\xi,\eta)\geq\liminf_{j\to\infty}\angle_{q_j}(\xi,\eta)\geq \liminf_{j\to\infty}\angle_{c(t_j)}(\xi,\eta)=\angle(\xi,\eta),$$
which implies that
\begin{align}\label{alt computation of angle}
\angle(\xi,\eta)= \lim_{j\to\infty}\angle_{q_j}(\xi,\eta).
\end{align}
\begin{figure}[h]
	\centering
	\includegraphics[width=4in]{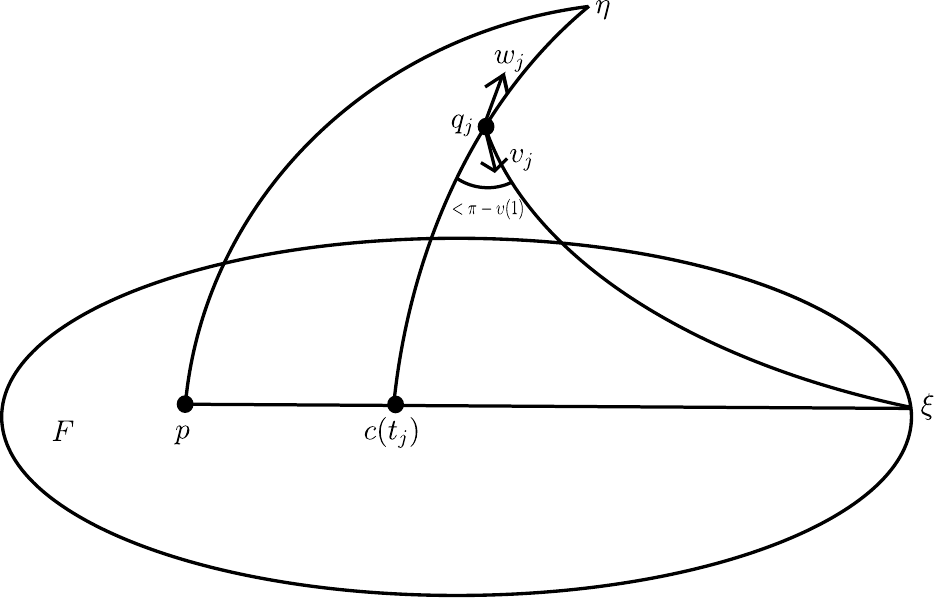}
	\caption{ \label{lem3.8}}
\end{figure}

We denote by $v_j,w_j\in T_{q_j}X$ the unit vector at $q_j$ pointing towards $\xi,\eta$ respectively. Then by compactness of $N,M$ and the assumption that $d(q_j,F)=1$ for any $j\geq 1$, up to a subsequence, there exist $\seq{\gamma_j}\subset \mathrm{Stab}_\Gamma(F)$ such that
$$\begin{cases}
\displaystyle \gamma_jq_j\to p_0\in X\\
\displaystyle \gamma_jv_j\to v\in T_{p_0}X\\
\displaystyle \gamma_jw_j\to w\in T_{p_0}X\\
\displaystyle \gamma_j\xi\to\xi_0,~\gamma_j\eta\to \eta_0\mathrm{~w.r.t.~the~cone~topology}
\end{cases}\quad\mathrm{~as~} j\to\infty.$$
In particular, by the fact that $\del F$ is closed with respect to the cone topology, $\xi_0\in\del F$. By \eqref{eqn:min dist exists}, the fourth assertion in Proposition \ref{BH9.2and9.5} and the second assertion in Proposition \ref{BH9.21}, we have
$$\angle(\xi_0,\eta_0)\leq \lim_{j\to\infty} \angle(\gamma_j\xi,\gamma_j\eta)=\angle(\xi,\eta)=d_T(\xi,\eta)=d_T(\eta,F)<\pi.$$
On the other hand, by the second assertion in Proposition \ref{BH9.2and9.5} and \eqref{alt computation of angle},
$$\angle(\xi,\eta)=\lim_{j\to\infty}\angle_{{q_j}}(\xi,\eta)=\lim_{j\to\infty}\angle_{\gamma_j{q_j}}(\gamma_j\xi,\gamma_j\eta)\leq \angle_{p_0}(\xi_0,\eta_0)\leq \angle(\xi_0,\eta_0).$$
Therefore $\angle(\xi_0,\eta_0)=\angle(\xi,\eta)=\angle_{p_0}(\xi_0,\eta_0)<\pi/2$. Hence by \cite[9.9 Corollary]{BridsonHaefliger99}, the geodesic rays $[p_0,\xi_0]$ and $[p_0,\eta_0]$ span a flat sector.
By Lemma \ref{flat sector lem} and the fact that $\xi_0\in\del F$, we have $\eta_0\in\del F$. However, for any $\zeta\in\del F$ and any $j\in\ZZ_+$, by Lemma \ref{res angle lem} and the first assertion in Proposition \ref{BH9.2and9.5},
$$\angle_{\gamma_jq_j}(\gamma_j\eta,\zeta)=\pi-\angle_{\gamma_jq_j}(\gamma_j c(t_j),\zeta)\geq \upsilon(d(\gamma_jq_j, F))=\upsilon(1)>0.$$
By the second assertion in Proposition \ref{BH9.2and9.5}, we have
$$\angle_{p_0}(\eta_0,\zeta)\geq \limsup_{j\to\infty}\angle_{\gamma_jq_j}(\gamma_j\eta,\zeta)\geq \upsilon(1)>0.$$
The arbitrariness of $\zeta\in\del F$ implies that $\eta_0\not\in\del F$. This contradicts $\eta_0\in\del F$.
\end{proof}

\begin{lemma}\label{hyperbolicity of right triangles}
There exists some $\delta>0$ such that for any $p\in F$, $q\in X\setminus F$ and $x=\Proj_F(q)$, $d(x,[p,q])<\delta$. As a corollary, the geodesic triangle with vertices $x,p,q$ is $\delta$-thin.
\end{lemma}
\begin{proof}
If not, there exist $\seq{p_j}\subset F$, $\seq{q_j}\subset X\setminus F$, $\seq{x_j}\subset F$ such that $x_j=\Proj_F(q_j)$ and $d(x_j,[p_j,q_j])\to\infty$ as $j\to\infty$. By compactness of $N$ and $M$, up to a subsequence, there exists $\seq{\gamma_j}\subset \mathrm{Stab}_F(\Gamma)$ such that $\gamma_jx_j\to x\in F$, $\gamma_jp_j\to\xi_p\in\del F$ and $\gamma_jq_j\to\xi_q\in \del X$ as $j\to\infty$. For simplicity, we identify $x_j,p_j,q_j$ with $\gamma_jx_j,\gamma_jp_j,\gamma_jq_j$ for any $j>0$. Since $[q_j,x_j]$ is perpendicular to $F$, by the second assertion of Proposition \ref{BH9.2and9.5}, for any $p'\neq x_j, x$ in $F$,
$$\angle_x(p',\xi_q)\geq\limsup_{j\to\infty}\angle_{x_j}(p',q_j)\geq \frac{\pi}{2}.$$
Hence $\xi_q\not\in\del F$. By Proposition \ref{Hindawi}, $\xi_p$ is connected to $\xi_q$ in the Tits boundary, which contradicts Lemma \ref{isolated Tits bdry}.
\end{proof}

For any $\gamma\in\Gamma$, it is easy to see that $\gamma F$ divides $X$ into 2 geodesically convex open subsets which are connected components of $X\setminus \gamma F$. We say $p,q\in X\setminus \gamma F$ are \emph{on the same side of $\gamma F$} if $[p,q]\ints \gamma F=\emptyset$. Otherwise, we say $p,q$ are \emph{on opposite sides of $\gamma F$}.

\begin{lemma}\label{key hyp lem}
Let $q_1,q_2\in X\setminus F$. Denoted by $r_j=\Proj_F(q_j)$, j=1,2.
\begin{enumerate}
\item[\hypertarget{hyp-1}{(1).}] For any $\epsilon>0$, there exist some $R_1(\epsilon)>0$ such that if $d(r_1,r_2)\geq R_1(\epsilon)$, then $d([q_1,q_2],F)\leq \epsilon$. In other words, if $d([q_1,q_2],F)> \epsilon$,  then $d(r_1,r_2)< R_1(\epsilon)$;
\item[\hypertarget{hyp-2}{(2).}] If we assume in addition that $q_1,q_2$ are on the same side of $F$, for any $r>0$, $R>0$, there exists some $c_1(r,R)>0$ such that if $d(q_j,F)\geq r$, $j=1,2$ and $d(r_1,r_2)\leq R$, then $d([q_1,q_2],F)\geq c_1(r,R)$.
\end{enumerate}
\end{lemma}

\begin{proof}
\begin{enumerate}
\item[(1).] This assertion is trivial if $q_1,q_2$ are on different sides of $F$. Therefore we assume that they are on the same side of $F$.

\begin{figure}[h]
	\centering
	\includegraphics[width=4in]{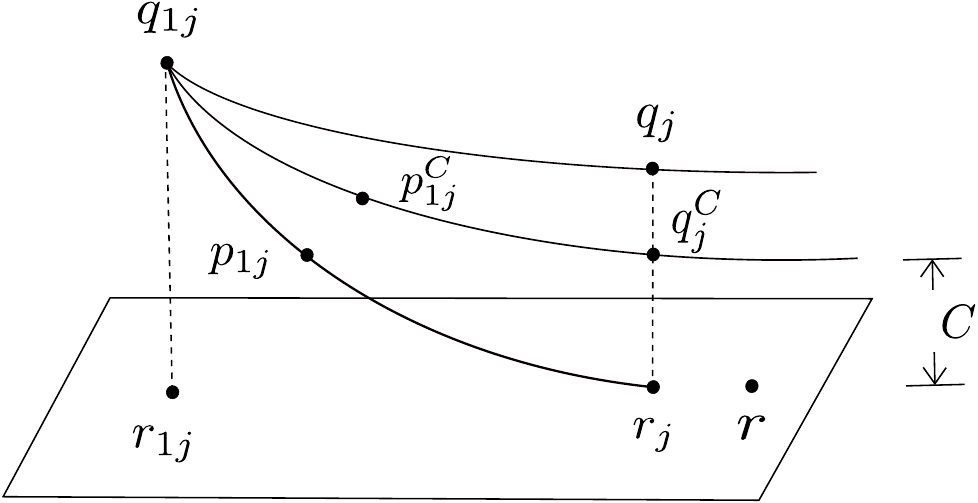}
	\caption{ \label{lem3.10-1}}
\end{figure}

If the statement is not true, there exist $\epsilon>0$ and $\seq{q_{lj}}$, $l=1,2$, with $r_{lj}=\Proj_F(q_{lj})$, such that $d([q_{1j},q_{2j}],F)>\epsilon$ and $d(r_{1j},r_{2j})\to\infty$ as $j\to\infty$. (See Figure \ref{lem3.10-1}.) Let $q_j\in[q_{1j},q_{2j}]$ be the closest point on this geodesic segment to $F$ and $r_j:=\Proj_F(q_j)$. By compactness of $N$ and $M$, we can assume WLOG that $r_j\to r\in F$, $d(r_{1j},r_j)\to \infty$ and $q_{1j}\to\xi_1\in\del X$ as $j\to \infty$. (To be specific, there exist $\seq{\gamma_j}\subset\mathrm{Stab}_\Gamma(F)$ such that, up to a subsequence, $\gamma_jr_j\to r\in F$. WLOG, we identify $\gamma_jr_j$ with $r_j$. Since $d(r_{1j},r_{2j})\to\infty$ as $j\to\infty$, up to a subsequence, we can assume WLOG that $ d(r_{1j},r_{j})\to\infty$ as $j\to\infty$. By compactness of $\overline{X}=X\union\del X$, up to a subsequence, $q_{1j}\to\xi_1\in\overline{X}$ as $j\to \infty$. Since $r_{1j}=\Proj_F(q_{1j})$  and $d(r_{1j},r)\geq d(r_{1j},r_{j})-d(r,r_{j})\to\infty$, we have $d(q_{1j},r)\geq d(r_{1j},r)\to\infty$ and hence $\xi_1\not\in X$. Therefore $\xi_1\in\del X$.) By Lemma \ref{hyperbolicity of right triangles}, there exists some $p_{1j}\in[q_{1j}, r_j]$ such that $d(r_{1j},p_{1j})<\delta$. For any $C>0$ fixed such that $C<\epsilon$, let $q_j^C\in[r_j,q_j]$ such that $d(r_j,q_j^C)=C$. Then by convexity of distance functions in nonpositive curvature, we can find $p_{1j}^C\in[q_{1j},q_j^C]$ such that $d(p_{1j},p_{1j}^C)\leq d(r_j,q_j^C)=C$.  Notice that
$$d(r_{1j},\Proj_F(p_{1j}^C))=d(\Proj_F(r_{1j}),\Proj_F(p_{1j}^C))\leq d(r_{1j},p_{1j}^C)\leq d(r_{1j},p_{1j})+d(p_{1j},p_{1j}^C)\leq \delta+C.$$
Therefore
$$d(p_{1j}^C,q_j^C)\geq d(\Proj_F(p_{1j}^C),\Proj_F(q_j^C))=d(\Proj_F(p_{1j}^C),r_j)\geq d(r_{1j},r_j)-d(r_{1j},\Proj_F(p_{1j}^C))\to\infty,$$
as $j\to\infty$. Since $d(q_j^C,F)=C$ and
$$d(p_{1j}^C,F)\leq d(p_{1j}^C,r_{1j})\leq d(p_{1j}^C,p_{1j})+d(p_{1j},r_{1j})\leq \delta+C,$$
up to a subsequence, $q_j^C$ converge to $q^C$ with $d(q^C,F)=\lim_{j\to\infty} d(q_j^C,F)=C$ (by the fact that $r_j\to r$ as $j\to\infty$ and $d(q_j^C,r_j)=C$), and the geodesic segment $[q_j^C,p_{1j}^C]$ converge to a geodesic ray $[q^C,\xi_1]$ which lies within the $(\delta+C)$-neighborhood of $F$, as $j\to\infty$. (This is because $p_{1j}^C\in[q_j^C,q_{1j}]$, $d(p_{1j}^C,q_j^C)\to\infty$ and $q_{1j}\to\xi_1$ as $j\to\infty$). In particular, we have $\xi_1\in \del F$. Hence by Lemma \ref{geo rays asymp to flats}, $d([q^C,\xi_1],F)=0$.

On the other hand, since $q_j$ is the closest point on $[q_{1j},q_{2j}]$ to $F$, $\angle_{q_j}(q_{1j},r_{j})\geq\pi/2$. Hence
$$\angle_{q_j^C}(q_{1j},r_j)\geq \angle_{q_j}(q_{1j},r_{j})+\angle_{q_{1j}}(q_j,q_j^C)\geq \pi/2.$$
Therefore $d([q_j^C,p_{1j}^C],F)=d(q_j^C,F)=C$ for any $j>0$. This implies that $d([q^C,\xi_1],F)=C>0$, which contradicts $d([q^C,\xi_1],F)=0$.
\item[(2).] If not, there exist $r,R>0$ and sequences $\seq{q_{lj}}$ with $r_{lj}=\Proj_F(q_{lj})$, $j=1,2$, such that
$$d(q_{lj},F)\geq r,~j=1,2,~ d(r_{1j} r_{2j})\leq R~\mathrm{and}~d([q_{1j},q_{2j}],F)\to 0.$$
Let $q_j\in [q_{1j},q_{2j}]$ be the closest point on this geodesic segment to $F$ and $v_j\in T_{q_j}X$ be a unit vector tangent to $[q_{1j},q_{2j}]$. (See Figure \ref{lem3.10-2}.) By the fact that $d(q_{lj},F)\geq r$ for any $l=1,2$, we assume WLOG that $q_j$ is in the interior of the geodesic segment $[q_{1j},q_{2j}]$. Denoted by $r_j:=\Proj_F(q_j)$ and $w_j\in T_{q_j}X$ a unit vector tangent to $[q_j,r_j]$. In particular $v_j\perp w_j$ for any $j>0$. By compactness of $N$ and $M$, we can assume WLOG that $q_j\to q\in F$, $v_j\to v\in T_qX$, $w_j\to w\in T_qX$ as $j\to \infty$. (To be specific, there exist $\seq{\gamma_j}\subset \mathrm{Stab}_F(\Gamma)$ such that $\gamma_jq_j\to q\in X$. The fact that $q\in F$ follows from $d([q_{1j},q_{2j}],F)=d(q_j,F)\to 0$ as $j\to \infty$. For simplicity, we identify $\gamma_jq_j$ with $q_j$. Then, up to a subsequence, $v_j\to v\in T_qX$, $w_j\to w\in T_qX$ as $j\to \infty$.) By definition of $w_j$, $w\perp T_qF$. Therefore by the fact that $F$ has codimension equal to one, $v\in T_qF$.

\begin{figure}[h]
	\centering
	\includegraphics[width=4in]{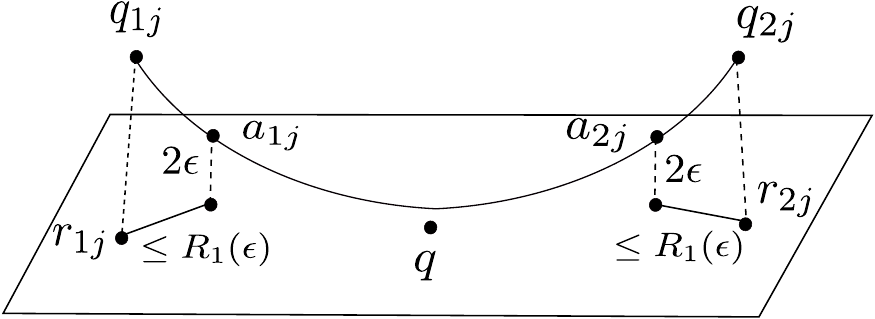}
	\caption{ \label{lem3.10-2}}
\end{figure}

Let $0<\epsilon\ll r/2$. For any $l=1,2$ and any $j$, we denote by $a_{lj}\in[q_{lj},q_j]$ the point such that $d(a_{1j},F)=d(a_{2j},F)=2\epsilon$. Since $v_j\to v\in TF$ as $j\to\infty$, by the fact that $M$ is compact, we can choose $K\gg 1$ such that
$$d(a_{1j},a_{2j})>2R_1(\epsilon)+R+4\epsilon+1,\quad\forall j>K.$$
Then for any $j>K$,
\begin{align*}
 &d(r_{1j},r_{2j})&\\
 \geq& d(\Proj_F(a_{1j}),\Proj_F(a_{2j}))-d(r_{1j},\Proj_F(a_{1j}))-d(r_{2j},\Proj_F(a_{2j}))&(\mathrm{Triangle~ineq.})\\
  \geq& d(\Proj_F(a_{1j}),\Proj_F(a_{2j}))-R_1(\epsilon)-R_1(\epsilon)&(\mathrm{Lemma~\ref{key hyp lem}, \hyperlink{hyp-1}{(1)}.})\\
\geq &d(a_{1j},a_{2j})-d(a_{1j},\Proj_F(a_{1j}))-d(a_{2j},\Proj_F(a_{2j}))-2R_1(\epsilon)&(\mathrm{Triangle~ineq.})\\
\geq &d(a_{1j},a_{2j})-4\epsilon-2R_1(\epsilon)>R+1.
\end{align*}
This contradicts the assumption that $d(r_{1j}, r_{2j})\leq R$ for all $j\geq 1$. \qedhere
\end{enumerate}
\end{proof}
\begin{rmk}
By applying isometries in $\Gamma$, all results in this section hold for any $\gamma F$, $\gamma\in \Gamma$.
\end{rmk}

\section{Barycentric simplices and almost separation}\label{sec:bar simplex and almost sep}
\subsection{Barycentric simplices and their properties}
In this subsection, we introduce barycentric simplices and some of its properties. In particular, we would like to verify the last two bullet points in (S2) in \hyperlink{idea-and-plan}{the outline and plan of the proof}.

We first recall an elementary lemma.
\begin{lemma}[{\cite[Lemma 20, p.165]{Peterson}}]\label{convexity of energy}
For any Hadamard manifold $(M',g)$ and any $x\in M'$, the function $E_{x;M'}:M'\to\RR$ defined by $E_{x;M'}(y)=(d_{M'}(x,y))^2/2$ is a smooth function. Moreover, it is strictly convex with $\mathrm{Hess}_g E_{x;M'}\geq g$.
\end{lemma}
\begin{rmk}
For simplicity, we write $E_p:=E_{p;X}$ for any $p\in X$. One can clearly see that for any $p,q\in X$ and any $\gamma\in \Gamma$, we have $E_{\gamma p}(\gamma q)=E_{p}(q)$.
\end{rmk}
\begin{dfn}[Singular simplices]
For any $k\in\ZZ_{\geq 0}$, a \emph{singular simplex in $X$} is a continuous map $\sigma:\Delta^k_{\RR^{k+1}}\to X$, where
$$\Delta^k_{\RR^{k+1}}=\left\{(t_0,...,t_k)\in\RR^{k+1}\left|0\leq t_0,...,t_k\leq 1, \sum_{j=0}^kt_j=1\right.\right\}.$$
Points of the form $\sigma(\delta_{0j},...,\delta_{kj})$ with $0\leq j\leq k$ are called \emph{vertices} of $\sigma$, where $\delta_{ij}=0$ when $i\neq j$ and $\delta_{ij}=1$ when $i=j$.
\end{dfn}
\begin{definition}[Barycentric simplex]\label{bar simplex}
For any $(k+1)$-tuple of points $(p_0,...,p_k)$, we defined the \emph{$k$-dimensional barycentric simplex} $\Db^k(p_0,...,p_k)$ as a singular simplex such that $\Db^k(p_0,...,p_k)(a_0,...,a_k)$ is the unique minimum of the function $E_{p_0,...,p_k}(x;a_0,...,a_k):=\sum_{j=0}^ka_jE_{p_j}(x)$ for any $(a_0,...,a_k)\in\Delta_{\RR^{k+1}}^k$.
\end{definition}
\begin{rmk}
Barycentric simplices naturally satisfy the following properties:
\begin{enumerate}
\item[(1).] $\Db^k(p_0,...,p_k)(a_0,...,a_k)$ is the unique point in $X$ such that $DE_{p_0,...,p_k}(x;a_0,...,a_k)$ vanishes for any $p_0,...,p_k$ and $(a_0,...,a_k)$.
\item[(2).] $\Db^k(\cdot,...,\cdot)(\cdot,...,\cdot)$ is a smooth function on $X^{k+1}\times \mathrm{int}(\Delta_{\RR^{k+1}}^k)$. Also, $\Db(\cdot,...,\cdot)(a_0,...,a_k)$ is a smooth function on $X^{k+1}$ for any $(a_0,...,a_k)\in\Delta_{\RR^{k+1}}^k$ fixed.
\item[(3).] ($\Gamma$-equivariance) For any $\gamma\in \Gamma$, any $p_0,...,p_k\in X$ and any $(a_0,...,a_k)\in \Delta_{\RR^{k+1}}^k$, we have $\gamma\Db^k(p_0,...,p_k)(a_0,...,a_k)=\Db^k(\gamma p_0,...,\gamma p_k)(a_0,...,a_k)$.
\item[(4).] For any $k\geq 1$, $\del_k^X\Db^k(p_0,...,p_k)$ is a linear combination of $(k-1)$-dimensional barycentric simplices, where $\del_k^X$ is introduced in Definition \ref{singular chains in X}.
\end{enumerate}
When $k=1$, one can easily verify that the image of $\Db^1(p_0,p_1)$ is the geodesic segment $[p_0,p_1]$. Moreover, $\Db^1(p_0,p_1)$ is injective whenever $p_0\neq p_1$.
\end{rmk}
For simplicity, let $\AA_k:=\{(a_0,...,a_k)\in\RR^{k+1}|a_0+...+a_k=1\}$ for any $k\geq 0$. In particular, $\Delta^k_{\RR^{k+1}}\subset\AA_k$. We first prove that barycentric simplices are smooth simplices in the sense of \cite[p.416]{Lee03}.
\begin{lemma}[Smoothness of barycentric simplices]\label{smoothness of bar simplex}
For any $k\geq 0$ and any $(k+1)$-(ordered) tuple of points $(p_0,...,p_k)$ in $X$, there exists some neighborhood $\cU$ of $\Delta^k_{\RR^{k+1}}\subset\AA_k$ such that the barycentric simplex map $\Db^k(p_0,...,p_k)$ can be smoothly extended onto $\cU$.
\end{lemma}
\begin{proof}
Since $X$ is nonpositively curved, for any $\theta\in (0,\pi/4)$, there exists some $R>(n+1000)(1+\max_{i,j\in\{0,...,k\}}\{d(p_i,p_j)\})$ depending on $\theta$ and $(p_0,...,p_k)$, such that for any $q\in X\setminus B_R(p_0)$, $q\not\in\{p_0,...,p_k\}$ and
\begin{align}\label{no soln outside}
\angle(\mathrm{grad} E_{p_i}(q), \mathrm{grad} E_{p_j}(q))<\theta,~\forall i,j\in\{0,...,k\}.
\end{align}
In particular, there exists some $\theta'\in(0,\min\{0.1,\theta\})$ depending on $\theta$, $R$ and $(p_0,...,p_k)$, such that
\begin{align}\label{convex inside}
\hess E_{p_i}(q)(v,v)\leq (2\theta')^{-1}\|v\|^2,~\forall q\in B_{2R}(p_0),~\forall v\in T_qX~\mathrm{and}~\forall i\in\{0,...,k\}.
\end{align}
Let $\cU:=\{(a_0,...,a_k)\in\AA_k|a_j>- \theta'/(n+1),~\forall j\in\{0,...,k\}\}$. By Lemma \ref{convexity of energy}, \eqref{no soln outside}, \eqref{convex inside} and the lower bound on $R$, the function $E_{p_0,...,p_k}(x;a_0,...,a_k):=\sum_{j=0}^ka_jE_{p_j}(x)$ defined on $\cU$ satisfies the following properties:
\begin{itemize}
\item $E_{p_0,...,p_k}(\cdot;a_0,...,a_k)$ does not admit critical points on $X\setminus B_R(p_0)$ for any $(a_0,...,a_k)\in\cU$ (guaranteed by \eqref{no soln outside}). 
\item $E_{p_0,...,p_k}(\cdot;a_0,...,a_k)$ is strictly convex on $B_R(p_0)$ for any $(a_0,...,a_k)\in\cU$ (guaranteed by \eqref{convex inside} and Lemma \ref{convexity of energy}).
\item $E_{p_0,...,p_k}(\cdot;a_0,...,a_k)$ achieves a unique global minimum in the interior of $B_R(p_0)$ for any $(a_0,...,a_k)\in\cU$ (guaranteed by the fact that $R>(n+1000)(1+\max_{i,j\in\{0,...,k\}}\{d(p_i,p_j)\})$).
\end{itemize}
Hence, by the smoothness of $E_{p_0,...,p_k}$ and the implicit function theorem, the barycentric simplex map $\Db^k(p_0,...,p_k)$ introduced in Definition \ref{bar simplex} can be smoothly extended onto $\cU$.
\end{proof}
\begin{lemma}\label{Lip bar simplex}
For any open balls $B_{R_j}(p_j)\subset X$ centered at $p_j\in X$, $j=0,...,k$, $\Db^k(\cdot,...,\cdot)(\cdot,...,\cdot)$ is Lipschitz on $B_{R_0}(p_0)\times...\times B_{R_k}(p_k)\times \Delta_{\RR^{k+1}}^k$.
\end{lemma}
\begin{proof}
We first prove that $\Db^k(\cdot,...,\cdot)(\cdot,...,\cdot)$ is continuous on $B_{R_0}(p_0)\times...\times B_{R_k}(p_k)\times \Delta_{\RR^{k+1}}^k$. Let $\{(q_0^{(m)},...,q_k^{(m)};a_{0}^{(m)},...,a_k^{(m)})\}_{m\in\ZZ_+}$ be a sequence of points in $B_{R_0}(p_0)\times...\times B_{R_k}(p_k)\times \Delta_{\RR^{k+1}}^k$ which converges to $(q_0,...,q_k;a_0,...,a_k)\in B_{R_0}(p_0)\times...\times B_{R_k}(p_k)\times \Delta_{\RR^{k+1}}^k$ as $m\to\infty$. For simplicity, let $x^{(m)}:=\Db^k(q_0^{(m)},...,q_k^{(m)})(a_{0}^{(m)},...,a_k^{(m)})$ and $x:=\Db^k(q_0,...,q_k)(a_{0},...,a_k)$. Let
\begin{align}\label{eqn:very large R}
R=(n+1000)\left(1+2\sum_{j=0}^kR_k+\sum_{0\leq i,j\leq k}d(p_i,p_j)\right).
\end{align}
By the first remark after Definition \ref{bar simplex}, $\{x_m\}_{m\in\ZZ_+}$ and $x$ are contained in $B_R(p_0)$. In particular, $\{x_m\}_{m\in\ZZ_+}$ has convergent subsequences. For any subsequence $\{x_{m_l}\}_{l\in\ZZ_+}$ of $\{x_m\}_{m\in\ZZ_+}$ which converges to some point $x'\in X$, we have
$$
0=\lim_{l\to\infty} DE_{q_0^{(m_l)},...,q_k^{(m_l)}}(x_{m_l};a_{0}^{(m_l)},...,a_k^{(m_l)})=DE_{q_0,...,q_k}(x';a_{0},...,a_k).
$$
Hence by the first remark after Definition \ref{bar simplex}, we have $x'=x$. This completes the proof that $\Db^k(\cdot,...,\cdot)(\cdot,...,\cdot)$ is continuous on $B_{R_0}(p_0)\times...\times B_{R_k}(p_k)\times \Delta_{\RR^{k+1}}^k$.

It remains to prove that $\Db^k(\cdot,...,\cdot)(\cdot,...,\cdot)$ is Lipschitz on $B_{R_0}(p_0)\times...\times B_{R_k}(p_k)\times \Delta_{\RR^{k+1}}^k$. By the second remark after Definition \ref{bar simplex}, it suffices to find an upper bound for $\|(D\Db^k)(\cdot,...,\cdot)(\cdot,...,\cdot)\|$ on $B_{R_0}(p_0)\times...\times B_{R_k}(p_k)\times \mathrm{int}(\Delta_{\RR^{k+1}}^k)$. Let $(q_0,...,q_k;a_0,...,a_k)$ be an arbitrary point in $B_{R_0}(p_0)\times...\times B_{R_k}(p_k)\times \mathrm{int}(\Delta_{\RR^{k+1}}^k)$ and $x:=\Db^k(q_0,...,q_k)(a_{0},...,a_k)$. By the first remark after Definition \ref{bar simplex} and the implicit function theorem applied to
$$0=DE_{q_0,...,q_k}(x;a_0,...,a_k)=\sum_{j=0}^ka_jDE_{q_j}(x),$$
for any $\xi=(v_0,...,v_k;\delta_0,...,\delta_k)\in T_{(q_0,...,q_k;a_0,...,a_k)}\left(B_{R_0}(p_0)\times...\times B_{R_k}(p_k)\times \Delta_{\RR^{k+1}}^k\right)$, we have
\begin{align}\label{bar jacobian computations}
\begin{split}
&\mathrm{Hess}( E_{q_0,...,q_k}(x;a_{0},...,a_k))((D\Db^k)(q_0,...,q_k)(a_{0},...,a_k)(\xi),\cdot)\\
=&-\left(\sum_{j=0}^k(a_j(D_{q_j}(DE_{q_j}(x)(\cdot)))(v_j))+\sum_{j=0}^k\delta_jDE_{q_j}(x)(\cdot)\right).
\end{split}
\end{align}
By the first remark after Definition \ref{bar simplex}, $B_{R_0}(p_0)\times...\times B_{R_k}(p_k)\times \Delta_{\RR^{k+1}}^k$ and its image under $\Db^k$ is contained in $B_R(p_0)$. (See \eqref{eqn:very large R}.) Therefore the norm of the RHS of \eqref{bar jacobian computations} is uniformly bounded from above by some positive constant. Then by Lemma \ref{convexity of energy} and \eqref{bar jacobian computations}, $\|(D\Db^k)(\cdot,...,\cdot)(\cdot,...,\cdot)\|$ has a uniform upper bound on $B_{R_0}(p_0)\times...\times B_{R_k}(p_k)\times \mathrm{int}(\Delta_{\RR^{k+1}}^k)$. This completes the proof of Lemma \ref{Lip bar simplex}.
\end{proof}

For any positive integer $k$, we denote by $d\vol_{X^k}$ the volume form of the product Riemannian manifold $X^k$. Then we have the following lemma.
\begin{lemma}\label{almost injectivity lem}
For any $q\in X$ and $d\vol_{X^{n+1}}$-a.e. $(p_0,...,p_{n})\in X^{n+1}$, $\mathrm{span}\{\mathrm{grad}E_{p_i}(q)\}_{i=0}^{n}= T_qX$. Moreover, all points $(p_0,...,p_{n})\in X^{n+1}$ satisfying $\mathrm{span}\{\mathrm{grad}E_{p_i}(q)\}_{i=0}^{n}= T_qX$ form an open and dense subset of $X^{n+1}$.
\end{lemma}
\begin{proof}
Let $d\vol_{T_qX}$ be the Euclidean volume on $T_qX$ induced by the Riemannian metric of $X$ at $q$. Consider the map
$$\exp^{(n+1)}_q:(T_qX)^{n+1}\to X^{n+1},~\exp^{(n+1)}_q(v_0,...,v_n)=(\exp_q(v_0),...,\exp_q(v_n)).$$
Then $\exp^{(n+1)}_q$ is a diffeomorphism. In particular, the pullback measure $(\exp^{(n+1)}_q)^*(d\vol_{X^{n+1}})$ is equivalent to $d\vol_{(T_qX)^{n+1}}$, the volume form of the product Riemannian manifold $(T_qX)^{n+1}$. Therefore it suffices to show that
$$
\{(v_0,...,v_n)\in (T_qX)^{n+1}|\mathrm{span}\{\mathrm{grad}E_{\exp_q(v_i)}(q)\}_{i=0}^{n}= T_qX\}
$$
is an open and dense subset of $(T_qX)^{n+1}$ whose complement has $0$ measure with respect to $d\vol_{(T_qX)^{n+1}}$. This easily follows from linear algebra since $\mathrm{grad}E_{\exp_q(v)}(q)=v$ for any $v\in T_qM$.
\end{proof}
\begin{corollary}\label{bar almost injectivity}
For $d\vol_{X^{n+1}}$-a.e. $(p_0,...,p_{n})\in X^{n+1}$, there exists an open and dense subset $\cU_{(p_0,...,p_{n})}\subset X$ whose complement has $0$ measure with respect to $d\vol_X$, such that for any $q\in \Db^n(p_0,...,p_n)(\Delta^n_{\RR^{n+1}})\cap(\cU_{(p_0,...,p_{n})})$, there exists a unique $(a_0,...,a_n)\in \Delta^n_{\RR^{n+1}}$ satisfying
$$\Db^n(p_0,...,p_n)(a_0,...,a_n)=q.$$
\end{corollary}
\begin{proof}
Let
\begin{align}\label{good points for top simplex}
\cU_{(p_0,...,p_{n})}=\{q\in X|\mathrm{span}\{\mathrm{grad}E_{p_i}(q)\}_{i=0}^{n}= T_qX\}.
\end{align}
By Lemma \ref{almost injectivity lem} and Fubini's theorem, for $d\vol_{X^{n+1}}$-a.e. $(p_0,...,p_{n})\in X^{n+1}$, we have $X\setminus \cU_{(p_0,...,p_{n})}$ has $0$ measure with respect to $d\vol_X$. In particular, $\cU_{(p_0,...,p_{n})}$ is dense. The fact that $\cU_{(p_0,...,p_{n})}$ is open follows directly from the fact that $(p,q)\to\mathrm{grad} E_q(p)$ is smooth on $X\times X$.

Suppose there exist $q\in \cU_{(p_0,...,p_n)}$ and $(a_0,...,a_n), (a_0',...,a_n')\in\Delta^n_{\RR^{n+1}}$ such that
$$
\Db^n(p_0,...,p_n)(a_0,...,a_n)=\Db^n(p_0,...,p_n)(a_0',...,a_n')=q.
$$
Then by the first remark after Definition \ref{bar simplex},
$$
\sum_{i=0}^n a_i\mathrm{grad}E_{p_i}(q)=\sum_{i=0}^n a_i'\mathrm{grad}E_{p_i}(q)=0.
$$
By \eqref{good points for top simplex} and the fact that $q\in\cU_{(p_0,...,p_n)}$, we have $\RR\cdot(a_0,...,a_n)=\RR\cdot(a_0',...,a_n')$. (Here, we are making use of the fact that if the span of $(n+1)$ vectors $v_0,...,v_n$ has dimension equal to $n$, then the space of solutions $(a_0,...,a_n)$ to the equation $\sum_{i=0}^na_iv_i=0$ has dimension equal to $1$.) Since $a_0+...+a_n=a_0'+...+a_n'=1$, we have $(a_0,...,a_n)=(a_0',...,a_n')$. This completes the proof of Corollary \ref{bar almost injectivity}.
\end{proof}
\begin{lemma}\label{bar bded ints away from flats}
For any $\epsilon>0$, $\widehat{F}\in\Gamma F=\{\gamma F|\gamma\in\Gamma\}$ and $k\in\ZZ_{\geq 0}$, let
$$\DbF{(\epsilon;\widehat{F})}^k(p_0,...,p_k):=\{p\in\mathrm{Im}(\Db^k(p_0,...,p_k))|d(p,\widehat{F})\geq \epsilon\}.$$
Then
$$\Proj_{\widehat{F}}(\DbF{(\epsilon;\widehat{F})}^k(p_0,...,p_k))\subset\bigcup_{j=0}^kB_{R_1(\epsilon/2)}(\Proj_{\widehat F}(p_j)),$$
where $B_R(p)$ is the open geodesic ball of radius $R$ centered at $b\in X$.
\end{lemma}
\begin{proof}
For any $q\in \DbF{(\epsilon;\widehat{F})}^k(p_0,...,p_k)$, let $v_q\in T_q^1X$ such that $v_q$ is the initial vector of the geodesic ray starting at $q$ passing through $\mathrm{Proj}_{\widehat{F}}(q)$. Then by the first remark after Definition \ref{bar simplex}, there exist some $j\in\{0,...,k\}$ such that $DE_{p_j}(q)(v_q)\leq 0$. Therefore $p_j$ and $q$ are on the same side of $\hF$ and $d([p_j,q],\widehat{F})=d(q,\widehat{F})$ (due to the fact that $X$ is nonpositively curved). The rest of the proof follows directly from the first assertion in Lemma \ref{key hyp lem}.
\end{proof}

\subsection{Almost separation by elements in $\Gamma F$}
In this subsection, we introduce the aforementioned notion of ``almost separation'' of a barycentric simplex in \hyperlink{idea-and-plan}{the outline and plan of the proof}. We also show some similarity between  ``almost separation'' and ``actual separation''. A rough summary is the following:
\begin{enumerate}
\item If $\hF\in\Gamma F$ (almost) separates an edge of the barycentric simplex into two pieces, it will also (almost) separate the $1$-skeleton of this simplex into two pieces; (See Proposition \ref{almost sep edge, almost sep simplex}.)
\item If $\hF\in\Gamma F$ fails to (almost) separate all edges of a barycentric simplex, then it must be far away from this simplex; (See Lemma \ref{bar: away from edge, away from simplex}.)
\item Let $x,y,z\in X$ be points which are away from all elements in $\Gamma F$. Assume that $F_1\in\Gamma F$ (almost) separate $[x,y]$ and $[x,z]$ and that $F_2\in\Gamma F\setminus\{F_1\}$ (almost) separate $[x,y]$ at somewhere closer to $x$ compared to $F_1$. Then the following holds: (See Proposition \ref{almost ints positioning}.)
\begin{itemize}
\item $F_2$ cannot (almost) separate $[y,z]$;
\item If $F_2$ (almost) separates $[x,z]$, it must (almost) separate $[x,z]$ at somewhere closer to $x$ compare to $F_1$.
\end{itemize}
\end{enumerate}
This subsection is ended by introducing two lemmas (Lemma \ref{prep1} and Lemma \ref{prep2}) which are used in the next subsection.
\begin{proposition}\label{almost sep edge, almost sep simplex}
For any $k\in \ZZ_+$, $r>0$ and $0\leq\epsilon\leq r$, there exist $0< c_2(k,r,\epsilon)<\epsilon$ such that for any $(k+1)$-tuple of points $(p_0,...,p_k)$, if $d(p_j,F)\geq r$ for any $j=0,...,k$, and there exist $i_0\neq j_0$ such that $d([p_{i_0},p_{j_0}],F)<c_2(k,r,\epsilon)$, then there exist $\emptyset\neq I\subsetneq \{0,...,k\}$ such that $i_0\in I$, $j_0\not\in I$ and
$$d([p_i,p_j],F)\leq \epsilon,\quad \forall i\in I, j\in\{0,...,k\}\setminus I.$$
\end{proposition}
\begin{proof}
Denote by $r_j=\Proj_F(p_j)$. Let $c_2(k,r,\epsilon)=c_1(r,2kR_1(\epsilon))$. (See Lemma \ref{key hyp lem} for the definition of $c_1$ and $R_1$.) If $p_{i_0}$ and $p_{j_0}$ are on opposite sides of $F$, then choose $I=\{0\leq i\leq k|p_i,p_{i_0}\mathrm{~on~the~same~side~of~}F\}$ and the result follows. We now assume that $p_{i_0}$ and $p_{j_0}$ are on the same side of $F$.

For any $J\subset\{0,...,k\}$, we define
$$\cD(J):=\left\{0\leq j\leq k\left| \overline{B_{R_1(\epsilon)}(r_j)}\ints\overline{B_{R_1(\epsilon)}(r_l)}\neq\emptyset\mathrm{~and~}[p_l,p_{j}]\cap F=\emptyset\mathrm{~for~some~}l\in J\right.\right\},$$
and $\cD^\infty(J):=\union_{j=1}^\infty\cD^j(J)$, where $\cD^0(J)=J$ and $\cD^j(J):=\cD(\cD^{j-1}(J))$. Clearly, $J\subset \cD(J)$ and $\cD^\infty(J)=\cD^{k}(J)$. By the definition of $\cD$, for any $t\geq 1$ and $j\in\cD^t(J)$, there exist some $i\in\cD^{t-1}(J)$ such that $j\in\cD(\{i\})$.

Let $I:=\cD^\infty(\{i_0\})\neq \emptyset$. In particular, for any $i,j\in I$, $p_i$ and $p_j$ are on the same side of $F$. For any $j\in I$, there exist a unique $0\leq t\leq k$ such that $j\in\cD^t(\{i_0\})\setminus\cD^{t-1}(\{i_0\})$. For simplicity, we set $i_t=j$. Inductively (in the reverse order), we choose $i_l\in\cD^{l}(\{i_0\})$ such that $i_{l+1}\in\cD(\{i_l\})$, $0\leq l\leq t-1$. If $i_l=i_s$ for some $0\leq l<s\leq t$, then we have $j=i_t\in\cD^{t-(s-l)}(\{i_0\})\subset \cD^{t-1}(\{i_0\})$, contradictory to the assumption on $j$ and $t$. Therefore $i_0,...,i_t=j$ are distinct numbers in $\{0,...,k\}$. By the definition of $\cD$ and the triangle inequality, $d(r_{i_0},r_j)\leq 2tR_1(\epsilon)\leq 2kR_1(\epsilon)$. By the second assertion in Lemma \ref{key hyp lem}, $d([p_{i_0},p_j],F)\geq c_1(r,2kR_1(\epsilon) )=c_2(k,r,\epsilon)$. Therefore by the assumption of $j_0$, $j_0\not\in I$.

On the other hand, for any $i\in I$ and $j\in\{0,...,k\}\setminus I$, either $p_i, p_j$ are on opposite sides of $F$, or $\overline{B_{R_1(\epsilon)}(r_i)}\ints\overline{B_{R_1(\epsilon)}(r_j)}=\emptyset$. In the first case, $[p_i,p_j]$ intersects $F$ and hence $d([p_i,p_j],F)=0$. In the second case, $d(r_i,r_j)\geq 2R_1(\epsilon)$. By the first assertion in Lemma \ref{key hyp lem}, $d([p_i,p_j],F)\leq \epsilon$.
\end{proof}
\begin{rmk}
 The above proposition holds for any $\widehat F\in \Gamma F$, not just $F$.
\end{rmk}
\begin{lemma}\label{bar: away from edge, away from simplex}
For any $k\in\ZZ_+$, $\widehat{F}\in\Gamma F$ and any $\epsilon>0$, there exist some $c_3(\epsilon)\in (0,\epsilon)$, such that for any $(k+1)$-tuple of points $p_0,...,p_k$ satisfying $d([p_i,p_j],\widehat{F})>\epsilon$ for any $i\neq j$, $i,j\in\{0,...,k\}$, we have
$$d(q,\widehat{F})\geq c_3(\epsilon)$$
for any $q\in\Db^k(p_0,...,p_k)(\Delta^k_{\RR^{k+1}})$.
\end{lemma}
\begin{proof}
If $p_0=...=p_k$, then the lemma holds trivially. From now on, we assume that $p_0,...,p_k$ are not equal to the same point. Also, we assume WLOG that $\epsilon\leq1 $. (For $\epsilon> 1$, we simply choose $c_3(\epsilon)=c_3(1)$.)

Let $q\in\Db^k(p_0,...,p_k)(\Delta^k_{\RR^{k+1}})$ and $r:=\Proj_{\hF}(q)$. (See Figure \ref{lem4.9-1}.) Notice that the lemma holds trivially if $q\in\{p_0,...,p_k\}$. Therefore we assume WLOG that $q\not\in\{p_0,...,p_k\}$. For simplicity, we let $r_j:=\mathrm{Proj}_{\widehat{F}}(p_j)$, $j=0,...,k$. Then by Lemma \ref{hyperbolicity of right triangles}, one can pick a point $p_j'\in[r,p_j]$ such that $d(p_j',r_j)<\delta$ for any $j=0,...,k$. (Here, $\delta$ is from Lemma \ref{hyperbolicity of right triangles}.) Since $X$ is nonpositively curved, there exist a $p_j''\in[q,p_j]$ such that $d(p_j'',p_j')\leq d(q,r)$. In particular
\begin{align}\label{bar: away from edge, away from simplex - 1}
d(p_j'',r_j)\leq d(p_j'',p_j')+d(p_j',r_j)\leq \delta+d(q,r),~\forall j\in\{0,...,k\}.
\end{align}

\begin{figure}[h]
	\centering
	\includegraphics[width=4in]{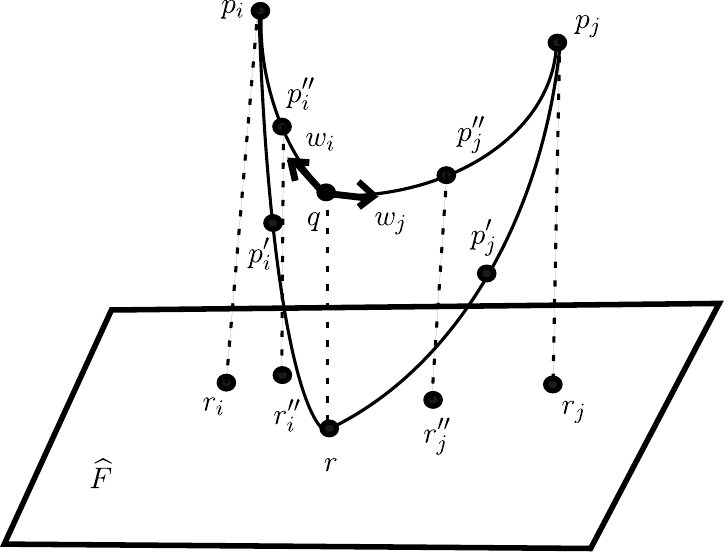}
	\caption{ \label{lem4.9-1}}
\end{figure}

By the first assertion in Lemma \ref{key hyp lem}, $d(r_i,r_j)<R_1(\epsilon)$ for any $i\neq j$, $i,j\in\{0,...,k\}$. Therefore by triangle inequality and \eqref{bar: away from edge, away from simplex - 1}, we have
$$d(p_i'',p_j'')\leq d(p_i'',r_i)+d(r_i,r_j)+d(r_j,p_j'')\leq 2\delta+2d(q,r)+R_1(\epsilon),~\forall i\neq j,~i,j\in\{0,...,k\}.$$
Let $r_j''=\mathrm{Proj}_{\widehat{F}}(p_j'')$ for any $j=0,...,k$. (See Figure \ref{lem4.9-1}.) Then the above implies that $d(r_i'',r_j'')\leq 2\delta+2d(q,r)+R_1(\epsilon)$ for any $i\neq j$, $i,j\in\{0,...,k\}$.

We first claim that
\begin{align}\label{bar: away from edge, away from simplex - 2}
d(r,r_j'')\leq d(q,p_j'')\leq 3(2\delta+2d(q,r)+R_1(\epsilon)),~\forall j\in\{0,...,k\}.
\end{align}
Indeed, if $ d(q,p_j'')>3(2\delta+2d(q,r)+R_1(\epsilon))$ for some $j\in\{0,...,k\}$, then by the triangle inequality, $d(q,p_i'')\geq d(q,p_j'')-d(p_j'',p_i'')>2(2\delta+2d(q,r)+R_1(\epsilon))$ for any $i\in\{0,...,k\}$. For any $i\in\{0,...,k\}$, let $w_i\in T_q^1X$ be the unit vector at $q$ pointing towards $p_i$. (This is well-defined since we assumed that $q\not\in\{p_0,...,p_k\}$.) Then by the fact that $X$ is nonpositively curved and the fact that $2d(p_i'',p_j'')\leq \min\{d(q,p_i''),d(q,p_j'')\}$, the angle $\angle(w_j,w_i)\leq \theta_0$ for some $\theta_0\in [0,\pi/2)$ and any $i\in\{0,...,k\}$. This contradicts the first remark after Definition \ref{bar simplex} and the fact that $q\in\Db^k(p_0,...,p_k)(\Delta^k_{\RR^{k+1}})$. Hence $\eqref{bar: away from edge, away from simplex - 2}$ holds. In particular, if $d(q,r)<1$, by the triangle inequality and \eqref{bar: away from edge, away from simplex - 1}, for any $j\in\{0,...,k\}$,
\begin{align}\label{bar: away from edge, away from simplex - 3}
d(r,r_j)\leq d(r,r_j'')+d(r_j'',r_j)\leq d(r,r_j'')+d(p_j'',r_j)\leq 3(2\delta+2+R_1(\epsilon))+\delta+1.
\end{align}
Let $v_q\in T^1_qX$ be the unit vector at $q$ pointing towards $r$. (See Figure \ref{lem4.9-2}.) Recall that $w_i\in T_q^1X$ is the unit vector at $q$ pointing towards $p_i$ for any $i\in\{0,...,k\}$. By the first remark after Definition \ref{bar simplex}, there exist distinct $l,l'\in\{0,...,k\}$ such that $\langle v_q,w_l\rangle\geq 0$ and $\langle v_q,w_{l'}\rangle\leq 0$. By the fact that $\langle v_q,w_{l'}\rangle\leq 0$ and the fact that $d([p_i,p_j],\hF)>\epsilon$ for any $i,j\in\{0,...,k\}$, $q$ and $p_0,...,p_k$ are on the same side of $\widehat{F}$. Let $\gamma_{w_l}(t)$ be the bi-infinite geodesic with $\dot\gamma_{w_l}(t_0)=w_l$ for some $t_0\in\RR$ to be determined. If we assume in addition that $d(q,r)\leq\min\{1,\epsilon\}$, then there exists a unique $q'\in[q,p_l]$ such that $d(q',\widehat{F})=\min_{t\in\RR}d(\gamma_{w_l}(t),\widehat{F})>0$. For simplicity, we let $\gamma_{w_l}(0)=q'$, $\gamma_{w_l}(-t_1)=q$ and $\gamma_{w_l}(t_2)=p_l$ for some $t_1\geq 0$ and $t_2>0$. (In particular, $t_0=-t_1$.)

\begin{figure}[h]
	\centering
	\includegraphics[width=4in]{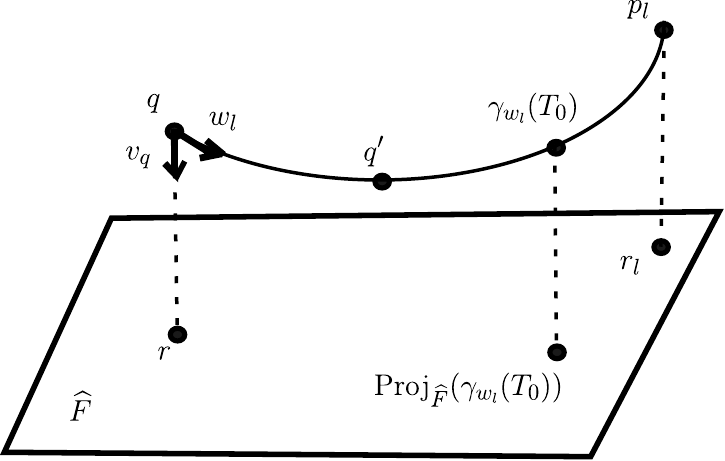}
	\caption{ \label{lem4.9-2}}
\end{figure}

Since $M=\Gamma\backslash X$ is nonpositively curved and $N=\mathrm{Stab}_{\Gamma}(F)\backslash F$ is compact, for any $R>0$, there exist some $\kappa(\epsilon, R)>0$ such that for any $p\in X$ and any $w_p\in T^1_pX$ satisfying
\begin{itemize}
\item $d(p,r)<\kappa(\epsilon, R)$;
\item $w_p\perp v_p$. Here $v_p\in T^1_pX$ is the unit vector at $p$ pointing towards $\mathrm{Proj}_{\widehat{F}}(p)$ (when $p\not\in\widehat{F}$), or a normal vector of $\widehat {F}$ at $p$ (when $p\in\widehat{F}$),
\end{itemize}
there exists some $T\geq R$ such that $d(\exp_p(tw_p),\widehat{F})\leq \epsilon$ for any $t\in[0,T]$. Now, if we further assume that
\begin{align}\label{choice of c3'}
d(q,\widehat{F})\leq c_3(\epsilon):=\min\{1,\epsilon/2,\kappa(\epsilon,3(2\delta+2+R_1(\epsilon))+\delta+3+R_1(\epsilon/2))\},
\end{align}
then we have
$$t_2\geq T_0:=\min_{t\geq 0}\{t|d(\gamma_{w_l}(t),\widehat{F})>\epsilon\}\geq 3(2\delta+2+R_1(\epsilon))+\delta+3+R_1(\epsilon/2)).$$
(See Figure \ref{lem4.9-2}.) In particular, by triangular inequality, we have
\begin{align}\label{bar: away from edge, away from simplex - 4}
\begin{split}
d\left(r,\mathrm{Proj}_{\widehat{F}}(\gamma_{w_l}(T_0))\right)\geq& d(q,\gamma_{w_l}(T_0))-d(q,r)-d\left(\gamma_{w_l}(T_0),\mathrm{Proj}_{\widehat{F}}(\gamma_{w_l}(T_0))\right) \\
> &d(q,\gamma_{w_l}(T_0))-2\epsilon\\
=&T_0+t_1-2\epsilon>T_0-2\geq 3(2\delta+2+R_1(\epsilon))+\delta+1+R_1(\epsilon/2).
\end{split}
\end{align}
Since $d([\gamma_{w_l}(T_0),p_l],\widehat{F})=d(\gamma_{w_l}(T_0),\widehat{F})=\epsilon$, by the first assertion in Lemma \ref{key hyp lem}, we have $d\left(r_l, \mathrm{Proj}_{\widehat{F}}(\gamma_{w_l}(T_0))\right)\leq R_1(\epsilon/2)$. Therefore, \eqref{bar: away from edge, away from simplex - 4} and the triangle inequality imply that
$$d(r_l,r)\geq d\left(r,\mathrm{Proj}_{\widehat{F}}(\gamma_{w_l}(T_0))\right)-d\left(r_l, \mathrm{Proj}_{\widehat{F}}(\gamma_{w_l}(T_0))\right)> 3(2\delta+2+R_1(\epsilon))+\delta+1.$$
This contradicts \eqref{bar: away from edge, away from simplex - 3}. Therefore the choice of $c_3(\epsilon)$ in \eqref{choice of c3'} completes the proof for Lemma \ref{bar: away from edge, away from simplex}.
\end{proof}
\begin{definition}[Almost separation, almost separation types]\label{almost sep}
For any $\epsilon\geq 0$, $\widehat F\in \Gamma F$ and any barycentric $k$-simplex $\Db^k(p_0,...,p_k)$, we say $\widehat F$ is \emph{$\epsilon$-almost separating} $\Db^k(p_0,...,p_k)$ if there exist $\emptyset\neq I\subsetneq \{0,...,k\}$ such that
$$d([p_i,p_j],\widehat F)\leq \epsilon,\quad \forall i\in I, j\in\{0,...,k\}\setminus I.$$
In this case, $\stype{\{p_i|i\in I\}}{\{p_0,...,p_k\}}$ is called an \emph{$\epsilon$-almost separation type} for $\widehat F$ with respect to $\Db^k(p_0,...,p_k)$. The collection of $\epsilon$-almost separation types for $\widehat F$ with respect to $\Db^k(p_0,...,p_k)$ is a subset of $\{\stype{I}{\{p_0,...,p_k\}}|\emptyset\neq I\subsetneq \{p_0,...,p_k\}\}\subset 2^{\{p_0,...,p_k\}}\times 2^{\{p_0,...,p_k\}}$, where $2^{\{p_0,...,p_k\}}$ is the set of all subsets of $\{p_0,...,p_k\}$.
\end{definition}
\begin{rmk}
If we set $\epsilon=0$, then $\widehat F$ $\epsilon$-almost separates $\Db^k(p_0,...,p_k)$ really means that $\widehat F$ intersects $\Db^k(p_0,...,p_k)$. In this simple case, if $p_j\not\in\widehat F$ for all $j=0,...,k$, $\widehat F$ has a unique $\epsilon$-almost separation type, describing how it cuts the 1-skeleton of the simplex into 2 pieces. However, if we allow $\epsilon>0$, $\epsilon$-almost separation types are not unique for a single $\widehat F$. For example, we can let all $p_j$ lie on the same side of $\widehat F$, far away from $\widehat F$, and set $\Proj_{\widehat F}(p_j)$ extremely far away from each other. By the first assertion in Lemma \ref{key hyp lem}, all edges of $\Db^k(p_0,...,p_k)$ can be $\epsilon$-close to $\widehat F$. Then any set of the form $\stype{I}{\{p_0,...,p_k\}}$ for some $\emptyset\neq I\subsetneq \{p_0,...,p_k\}$ is an $\epsilon$-almost separation type for $\widehat F$ with respect to $\Db^k(p_0,...,p_k)$.
\end{rmk}

Proposition \ref{almost sep edge, almost sep simplex} shows that for submanifolds $\widehat F\in \Gamma F$, almost separating an edge implies almost separating the barycentric simplex. Proposition \ref{bar: away from edge, away from simplex} shows that failure to almost separate a barycentric simplex means being not very close to this simplex.

Let $\rho_0>0$ be such that $d(F_1,F_2)>\rho_0$ for any $F_1\neq F_2\in\Gamma F$.
\begin{proposition}\label{almost ints positioning}
There exist some $\epsilon_0\in(0,\rho_0/3)$ sufficiently small such that for any triple of distinct points $x,y,z\in X$ and any $F_1\neq F_2\in \Gamma F$, the following holds: If $a_1, a_2\in [x,y]$, $b_1\in[x,z]$ such that
\begin{align*}
\begin{cases}
\displaystyle a_2\in[x,a_1],~d(a_2, F_2)\leq \epsilon_0,\\
\displaystyle d(a_1, F_1)\leq \epsilon_0,~d(b_1, F_1)\leq \epsilon_0,
\end{cases}
\end{align*}
Then
\begin{enumerate}
\item[(1)] $d([y,z], F_2)>\epsilon_0$.
\item[(2)] If there exist some $b_2\in[x,z]$ such that $d(b_2, F_2)\leq \epsilon_0$, then $b_2\in[x, b_1]$.
\end{enumerate}
\end{proposition}
\begin{proof}
\begin{enumerate}
\item[(1)] If not, for any $0<\epsilon<\rho_0/3$, we can find distinct points $x,y,z\in X$, distinct $F_1, F_2\in \Gamma F$ and $a_1,a_2\in[x,y]$, $b_1\in[x,z]$, $c\in[y,z]$ such that $d(c, F_2)\leq \epsilon$ and
\begin{align*}
\begin{cases}
\displaystyle a_2\in[x,a_1],~d(a_2, F_2)\leq \epsilon,\\
\displaystyle d(a_1, F_1)\leq \epsilon,~d(b_1, F_1)\leq \epsilon,
\end{cases}
\end{align*}
\begin{figure}[h]
	\centering
	\includegraphics[scale=1]{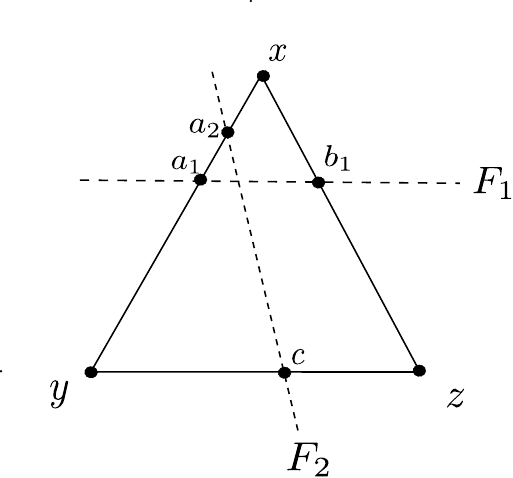}
	\caption{ \label{prop4.11-1}}
\end{figure}
(See Figure \ref{prop4.11-1}.) We first claim that $d([y,c], F_1)\neq d(c, F_1)$:

If not,
$$d([y,c], F_1)=d(c, F_1)\geq d(F_2, F_1)-d(c, F_2)>2\rho_0/3>0,$$ which implies that $y,c$ are on the same side of $F_1$. Notice that $d(F_1, F_2)>\rho_0$ and that $d([a_2,c],F_2)\leq\max\{d(a_2,F_2),d(c,F_2)\}\leq\epsilon$, we have for any $q\in[a_2,c]$,
$$d(q, F_1)\geq d(F_2, F_1)-d(q, F_2)>2\rho_0/3,$$
and hence $d([a_2,c],F_1)>2\rho_0/3$. In particular, $a_2$ and $c$ are on the same side of $F_1$.
By the first assertion in Lemma \ref{key hyp lem},
$$d(\Proj_{F_1}(y),\Proj_{F_1}(a_2))\leq d(\Proj_{F_1}(y),\Proj_{F_1}(c))+d(\Proj_{F_1}(c),\Proj_{F_1}(a_2))< 2R_1(2\rho_0/3).$$
By the second assertion in Lemma \ref{key hyp lem}, we have
$$d([y,a_2],F_1)\geq c_1(2\rho_0/3,2R_1(2\rho_0/3)).$$
This contradicts $\epsilon\geq d(a_1,F_1)\geq d([y,a_2],F_1)$ and the arbitrariness of $\epsilon>0$.

We then claim that $a_2,c,z,x$ are on the same side of $F_1$: (See Figure \ref{prop4.11-1}.)

If $[x,y]$ intersects $F_1$, by the fact that $d(F_1, F_2)>\rho_0$ and the convexity of distance functions, $[x,a_2]$ does not intersect $F_1$. Hence $y$ and the $2\rho_0/3$-neighborhood of $F_2$ are on opposite sides of $F_1$. This implies that $[y,c]$ intersects $F_1$ and $[c,z]$ does not intersect $F_1$. Therefore $a_2,c,z,x$ are on the same side of $F_1$.

If $[x,y]$ does not intersect $F_1$, then $x,a_2,y,c$ are on the same side of $F_1$. If $[y,z]$ does not intersect $F_1$, then $[x,y],[y,z],[z,x]$ are all on the same side of $F_1$. In particular $a_2,c,z,x$ are on the same side of $F_1$. If $y$ and $z$ are on opposite sides of $F_1$, then $F_1$ intersects $[c,z]$. In particular, by convexity of distance functions in nonpositive curvature, $d([y,c], F_1)=d(c, F_1)>2\rho_0/3,$ which is impossible as discussed previously.

We have proved that $a_2,c,z,x$ are on the same side of $F_1$ and $d([y,c], F_1)\neq d(c, F_1)$. By convexity of distance functions in $X$ and the fact that $d([y,c], F_1)\neq d(c, F_1)$, we have
$$d([c,z], F_1)=d(c, F_1)\geq d(F_2, F_1)-d(c, F_2)>2\rho_0/3.$$
Similarly,
$$d([a_2,x],F_1)=d(a_2,F_1)\geq d(F_2, F_1)-d(a_2, F_2)>2\rho_0/3.$$
Notice that
$$d([c,a_2], F_1)\geq \min_{q\in[c ,a_2]}\{d(F_2, F_1)-d(q, F_2)\}=d(F_2,F_1)-\max_{q\in[c,a_2]}(d(q,F_2))>2\rho_0/3.$$
Therefore by the first assertion in Lemma \ref{key hyp lem},
\begin{align*}
&~d(\Proj_{F_1}(x),\Proj_{F_1}(z))\\
\leq &~d(\Proj_{F_1}(x),\Proj_{F_1}(a_2))+d(\Proj_{F_1}(a_2),\Proj_{F_1}(c))+d(\Proj_{F_1}(c),\Proj_{F_1}(z))\leq 3R_1(2\rho_0/3).
\end{align*}
By the second assertion Lemma \ref{key hyp lem},
$$d([x,z],F_1)\geq c_1(2\rho_0/3,3R_1(2\rho_0/3)).$$
This contradicts $\epsilon\geq d(b_1,F_1)\geq d([x,z],F_1)$ and the arbitrariness of $\epsilon>0$.
\item[(2)] If not, for any $0<\epsilon<\rho_0/3$, we can find distinct points $x,y,z\in X$, distinct $F_1, F_2\in \Gamma F$ and $a_1,a_2\in[x,y]$, $b_1,b_2\in[x,z]$ such that $d(b_2, F_2)\leq \epsilon$, $b_1\in[x,b_2]$ and
\begin{align*}
\begin{cases}
\displaystyle a_2\in[x,a_1],~d(a_2, F_2)\leq \epsilon,\\
\displaystyle d(a_1, F_1)\leq \epsilon,~d(b_1, F_1)\leq \epsilon,
\end{cases}
\end{align*}
\begin{figure}[h]
	\centering
	\includegraphics[scale=1]{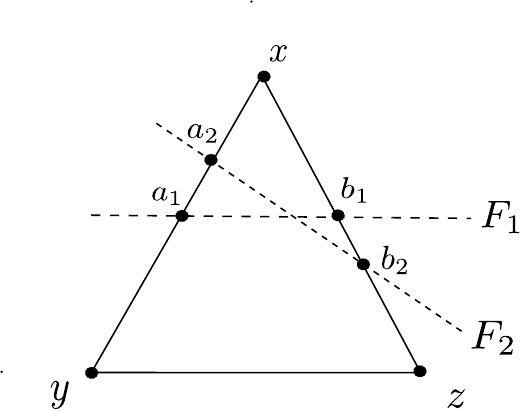}
	\caption{ \label{prop4.11-2}}
\end{figure}
(See Figure \ref{prop4.11-2}.) Since $d(a_1,F_1)\leq \epsilon< 2\rho_0/3< d(F_1,F_2)-d(a_2, F_2)\leq d(a_2, F_1)$, we have $d([x,a_2],F_1)=d(a_2, F_1)>2\rho_0/3$. This implies that $x$ and $a_2$ are on the same side of $F_1$. Notice that $d(F_1, F_2)>\rho_0$ and that $\max_{q\in[a_2,b_2]}d(q,F_2)=\max\{d(a_2,F_2),d(b_2,F_2)\}\leq\epsilon$ (due to nonpositive curvature), we have
$$d([a_2,b_2], F_1)\geq \min_{q\in[a_2,b_2]}\{d(F_2, F_1)-d(q, F_2)\}=d(F_2, F_1)-\max_{q\in[a_2,b_2]}d(q,F_2)>2\rho_0/3,$$
and hence $a_2,b_2$ are on the same side of $F_1$. By the first assertion in Lemma \ref{key hyp lem},
$$d(\Proj_{F_1}(x),\Proj_{F_1}(b_2))\leq d(\Proj_{F_1}(x),\Proj_{F_1}(a_2))+d(\Proj_{F_1}(a_2),\Proj_{F_1}(b_2))< 2R_1(2\rho_0/3).$$
By the second assertion in Lemma \ref{key hyp lem}, we have
$$d([x,b_2],F_1)\geq c_1(2\rho_0/3,2R_1(2\rho_0/3)).$$
This contradicts $\epsilon\geq d(b_1,F_1)\geq d([x,b_2],F_1)$ and the arbitariness of $\epsilon>0$.\qedhere
\end{enumerate}
\end{proof}
\begin{rmk}
From the proof, we can choose
$$\epsilon_0=\min\{c_1(2\rho_0/3,2R_1(2\rho_0/3)),c_1(2\rho_0/3,3R_1(2\rho_0/3)),\rho_0/3\}/2,$$
which only depends on the Riemannian manifold $M$.
\end{rmk}

\begin{corollary}\label{almost ints type property}
Let $F_1\neq F_2\in\Gamma F$ $\epsilon$-almost separate $\Db^k(p_0,...,p_k)$ for some $0\leq \epsilon<\epsilon_0$. For simplicity, we denote $V=\{p_0,...,p_k\}$. For any $\epsilon$-almost separation types $\stype{I_j}{V}$ of $F_j$ with respect to $\Delta^k(p_0,...,p_k)$, there exist $I_j'\in\stype{I_j}{V}$ such that $I_1'\ints I_2'=\emptyset$.
\end{corollary}
\begin{proof}
If not, we can choose $x_{11}\in I_1\ints I_2$, $x_{12}\in I_1\ints(V\setminus I_2)$, $x_{21}\in I_2\ints(V\setminus I_1)$, $x_{22}\in(V\setminus I_1)\ints(V\setminus I_2)$. Then $d([x_{1j},x_{2l}], F_1)\leq \epsilon$ and $d([x_{j1},x_{l2}], F_2)\leq \epsilon$, $1\leq j,l\leq 2$. Let $a_j$ be the closest point on $[x_{11},x_{22}]$ to $F_j$, $j=1,2$ respectively. Then by Proposition \ref{almost ints positioning},
\begin{equation*}
\left.
\begin{aligned}
F_2~\epsilon\textrm{-}\mathrm{close~to}~[x_{11},x_{12}]~\mathrm{and}~[x_{11},x_{22}] \\
F_1~\epsilon\textrm{-}\mathrm{close~to}~[x_{12},x_{22}]~\mathrm{and}~[x_{11},x_{22}]
\end{aligned}
\right\}\implies a_2 \in [a_1,x_{11}]
\end{equation*}
and
\begin{equation*}
\left.
\begin{aligned}
F_2~\epsilon\textrm{-}\mathrm{close~to}~[x_{11},x_{22}]~\mathrm{and}~[x_{21},x_{22}] \\
F_1~\epsilon\textrm{-}\mathrm{close~to}~[x_{11},x_{21}]~\mathrm{and}~[x_{11},x_{22}]
\end{aligned}
\right\}\implies a_2 \in [a_1,x_{22}].
\end{equation*}
Hence $a_1=a_2$ and $d(F_1, F_2)\leq 2\epsilon<\rho_0$, which is impossible.
\end{proof}

\begin{notation}\label{Fx def}
For any $x\in X$ such that there exist some $\hF\in \Gamma F$ satisfying $d(x,\hF)\leq\epsilon_0/2$, we define $F_x:=\hF$. Since $d(F_1, F_2)>\rho_0\geq 3\epsilon_0$ for any $F_1\neq F_2\in\Gamma F$, the choice of $\hF$ is unique. This implies that $F_x$ is well-defined.

Fix some $x_0\in X$ such that $d(x_0,F)=\epsilon_0/2$, where $\epsilon_0$ is introduced in Proposition \ref{almost ints positioning}. In particular, for any $x\in \Gamma x_0$, we have $d(x,F_x)=\epsilon_0/2$.
\end{notation}
\begin{lemma}\label{prep1}
For any $F_1\in\Gamma F$, $m\in\ZZ_{\geq 0}$, $p_1,p_2,a_1,a_2\in X$ such that
\begin{enumerate}
\item[(1).] $d([p_j,a_j],F_1)=d(a_j,F_1)=\epsilon_0/2$, $j=1,2$.
\item[(2).] There exist $p_1=q_0,q_1,...,q_m,q_{m+1}=p_2\in X$ such that $d([q_j,q_{j+1}],F_1)>2\rho_0/3$.
\end{enumerate}
Then there exists $\cC_2(\rho_0,\epsilon_0,m)>0$ such that
$$d(\Proj_{F_1}(a_1),\Proj_{F_1}(a_2))\leq \cC_2(\rho_0,\epsilon_0,m).$$
\end{lemma}
\begin{proof}
By the assumptions of this lemma, $a_1,a_2,q_0,...,q_{m+1}$ are on the same side of $F_1$. Let $y_j:=\Proj_{F_1}(q_j)$, $j=0,...,m+1$. By the first assertion in Lemma \ref{key hyp lem}, $d(y_j,y_{j+1})\leq R_1(2\rho_0/3)$ for any $0\leq j\leq m$ and $d(\Proj_{F_1}(a_l),\Proj_{F_1}(p_l))\leq R_1(\epsilon_0/4)$, $ l=1,2$. (See Figure \ref{lem4.14}.) Hence
\begin{figure}[h]
	\centering
	\includegraphics[width=4in]{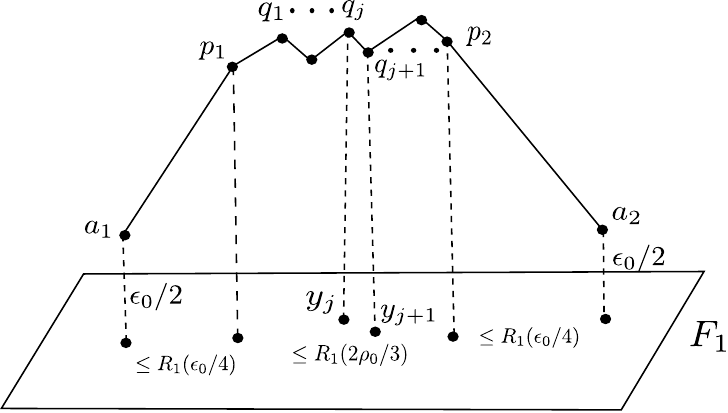}
	\caption{ \label{lem4.14}}
\end{figure}
\begin{align*}
&d(\Proj_{F_1}(a_1),\Proj_{F_1}(a_2))\\
\leq&d(\Proj_{F_1}(a_1),\Proj_{F_1}(p_1))+d(\Proj_{F_1}(a_2),\Proj_{F_1}(p_2))+d(\Proj_{F_1}(p_1),\Proj_{F_1}(p_2))\\
\leq&d(\Proj_{F_1}(a_1),\Proj_{F_1}(p_1))+d(\Proj_{F_1}(a_2),\Proj_{F_1}(p_2))+\sum_{j=0}^m d(y_j,y_{j+1}) \\
\leq& 2R_1(\epsilon_0/4)+(m+1)R_1(2\rho_0/3)=:\cC_2(\rho_0,\epsilon_0,m).
\end{align*}
\end{proof}
\begin{lemma}\label{prep2}
Let $F_1\in\Gamma F$ and $0<\epsilon\leq \epsilon_0$. Suppose $p_1,p_2,q_1,q_2\in X$ satisfies
\begin{enumerate}
\item[(1).] $d([p_1,p_2],F_1)>2\rho_0/3$ and $d([q_1,q_2], F_1)>2\rho_0/3$.
\item[(2).] $d([p_1,q_1],F_1)<c_4(\epsilon,\rho_0)$, where
$$c_4(\epsilon,\rho_0):=\min\left\{\frac{\epsilon}{4},c_1\left(\frac{2\rho_0}{3}, R_1\left(\frac{\epsilon}{4}\right)+1+2R_1\left(\frac{2\rho_0}{3}\right)\right)\right\}.$$
\end{enumerate}
Then $d([p_2,q_2],F_1)\leq \epsilon/4$. (See Figure \ref{lem4.15}.)
\end{lemma}

\begin{figure}[h]
	\centering
	\includegraphics[width=4in]{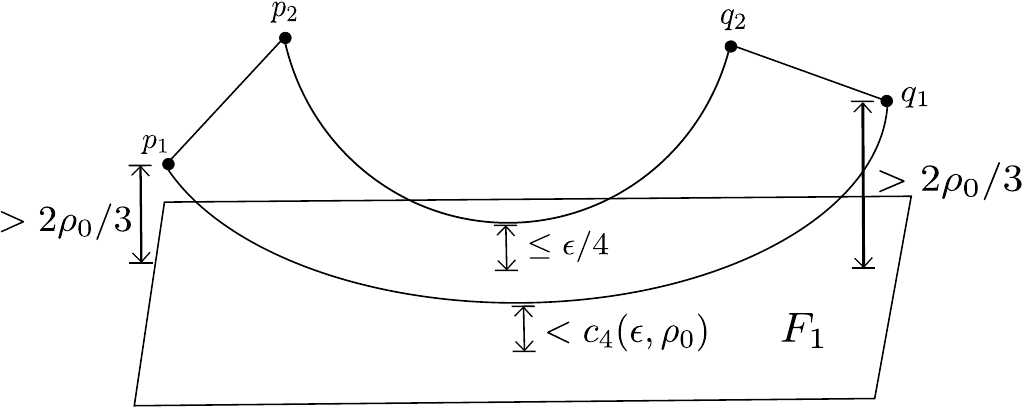}
	\caption{ \label{lem4.15}}
\end{figure}

\begin{proof}
If $p_1,q_1$ are on opposite sides of $F_1$, notice that $p_1,p_2$ are on the same side of $F_1$ and $q_1,q_2$ are on the same side of $F_1$, we have $p_2,q_2$ are on opposite sides of $F_1$. Hence $d([p_2,q_2],F_1)=0$ and the lemma follows.

Now we assume that $p_1,q_1$ are on the same side of $F_1$. Then $p_1,p_2,q_1,q_2$ are on the same side of $F_1$. By the second assertion in Lemma \ref{key hyp lem} (with $r=2\rho_0/3$ and $R=R_1(\epsilon/4)+1+2R_1(2\rho_0/3)$),
\begin{equation*}
\left.
\begin{aligned}
&d(p_1,F_1)>\frac{2\rho_0}{3},~d(q_1, F_1)>\frac{2\rho_0}{3},\\
&d([p_1,q_1],F_1)<c_1\left(\frac{2\rho_0}{3}, R_1\left(\frac{\epsilon}{4}\right)+1+2R_1\left(\frac{2\rho_0}{3}\right)\right)
\end{aligned}
\right\}\implies
\begin{aligned}
&d(\Proj_{F_1}(p_1),\Proj_{F_1}(q_1)) \\
>~&R_1\left(\frac{\epsilon}{4}\right)+1+2R_1\left(\frac{2\rho_0}{3}\right).
\end{aligned}
\end{equation*}

By the first assertion in Lemma \ref{key hyp lem},
\begin{align*}
d([p_1,p_2],F_1)>\frac{2\rho_0}{3}\implies d(\Proj_{F_1}(p_1),\Proj_{F_1}(p_2))<R_1\left(\frac{2\rho_0}{3}\right),\\
d([q_1,q_2],F_1)>\frac{2\rho_0}{3}\implies d(\Proj_{F_1}(q_1),\Proj_{F_1}(q_2))<R_1\left(\frac{2\rho_0}{3}\right).
\end{align*}
Hence by triangle inequality, we have
\begin{align*}
&d(\Proj_{F_1}(p_2),\Proj_{F_1}(q_2))\\
\geq& d(\Proj_{F_1}(p_1),\Proj_{F_1}(q_1))-d(\Proj_{F_1}(p_1),\Proj_{F_1}(p_2))-d(\Proj_{F_1}(q_1),\Proj_{F_1}(q_2))\\
>&R_1\left(\frac{\epsilon}{4}\right)+1+2R_1\left(\frac{2\rho_0}{3}\right)-2R_1\left(\frac{2\rho_0}{3}\right)=R_1\left(\frac{\epsilon}{4}\right)+1.
\end{align*}
By the first assertion in Lemma \ref{key hyp lem}, $d([p_2,q_2],F_1)\leq \epsilon/4$.
\end{proof}

\subsection{$\Omega(\cdot,\cdot)$, $\Theta(\cdot,\cdot)$ and their properties}\label{subsect:Omega and Theta} In this subsection, we introduce the aforementioned notion of ``almost between'' relations among elements in $\Gamma F$ mentioned in \hyperlink{idea-and-plan}{the outline and plan of the proof}.

Let
\begin{align}\label{eqn:epsilon 0-2}
\epsilon_1=c_4(\epsilon_0,\rho_0)\mathrm{~and~}\epsilon_2=c_4(\epsilon_1,\rho_0),
\end{align}
where $c_4(\cdot,\cdot)$ is introduced in Lemma \ref{prep2}. In particular, $\epsilon_2\leq \epsilon_1/4\leq\epsilon_0/16\leq \rho_0/48$.

For any $F_1, F_2\in\Gamma F$, we define $\Omega_0(F_1, F_2)\subset \Gamma F$ such that
$$\Omega_0(F_1, F_2)=\left\{\widehat F\in\Gamma F\left| \exists p_j\in X~\mathrm{s.t.~} d(p_j,F_j)\leq \epsilon_0/2, j=1,2,~\mathrm{and}~d([p_1,p_2],\widehat F)<\epsilon_2/2.\right.\right\}.$$
Clearly, $F_j\in\Omega_0(F_1, F_2)$, $j=1,2$ and $\Omega_0(F_1,F_1)=\{F_1\}$. Inductively we define
$$\Omega_k(F_1,F_2):=\bigcup_{F',F''\in\Omega_{k-1}(F_1,F_2)}\Omega_0(F',F'').$$
Then $\Omega_0(F_1,F_2),...,\Omega_k(F_1,F_2),...$ is an increasing sequence of subsets of $\Gamma F$. Finally, we define
$$\Omega(F_1,F_2):=\bigcup_{j=0}^\infty\Omega_{k}(F_1,F_2).$$
In particular, one can inductively show that $\Omega(F_1,F_1)=\{F_1\}$. One should think about $\Omega(F_1,F_2)$ as a notion of ``candidates'' in $\Gamma F$ which are almost between $F_1,F_2$. We have the following properties for $\Omega$.
\begin{lemma}\label{properties of Omega}
$\Omega(\cdot,\cdot)$ satisfies the following properties:
\begin{enumerate}
\item[\hypertarget{Omega-1}{($\Omega$1).}] For any $F'\in\Omega(F_1,F_2)\setminus\{F_1,F_2\}$ and any $p_j\in X$ such that $d(p_j, F_j)\leq \epsilon_0/2$, j=1,2, we have $d([p_1,p_2],F')\leq \epsilon_1/4$. (See \eqref{eqn:epsilon 0-2} for the definition of $\epsilon_0,\epsilon_1$ and $\epsilon_2$.) As a corollary, $|\Omega(F_1,F_2)|<\infty$;
\item[\hypertarget{Omega-2}{($\Omega$2).}] Let $F_1',...,F_k'$ be distinct elements in $\Gamma F\setminus \{F_1,F_2\}$ such that $\Omega(F_1,F_2)=\{F_1,F_1',...,F_k',F_2\}$. For any $p_j\in X$ such that $d(p_j,F_j)\leq \epsilon_0/2$ and for any $r_j\in[p_1,p_2]$ such that $d(r_j,F_j')\leq \epsilon_0/2$, $1\leq j\leq k$, if $r_i\in[p_1,r_j]$ whenever $i<j$, then for any $j\in\{1,...,k\}$, we have $\Omega(F_1,F_j')=\{F_1,F_1',...,F_j'\}$ and $\Omega(F_j',F_2)=\{F_j',...,F_k',F_2\}$. (Here, the existence of $r_j$ is guaranteed by \hyperlink{Omega-1}{property ($\Omega$1) of $\Omega(\cdot,\cdot)$} and \eqref{eqn:epsilon 0-2} for any $j\in\{1,...,k\}$.)

As a direct corollary, for any $F'\in\Omega(F_1,F_2)$, $\Omega(F_1,F_2)=\Omega(F_1, F')\union\Omega(F',F_2)$ and $\Omega(F_1,F')\ints\Omega(F', F_2)=\{F'\}$.
\end{enumerate}
\end{lemma}
\begin{proof}
Since $\Omega(F_1,F_1)=\{F_1\}$, the lemma holds trivially if $F_1=F_2$. Hence we assume that $F_1\neq F_2$.
\begin{enumerate}
\item[($\Omega$1).] We prove inductively on $k$ for $F'\in\Omega_k(F_1,F_2)\setminus\{F_1,F_2\}$. When $k=0$, by definition of $\Omega_0(\cdot,\cdot)$, there exists  $q_j\in X$ with $d(q_j, F_j)\leq \epsilon_0/2$ such that $d([q_1,q_2],F')\leq \epsilon_2/2<c_4(\epsilon_1,\rho_0)$. By convexity of distance functions in $X$, $\max_{q\in[p_j,q_j]}d(q,F_j)=\max\{d(p_j,F_j),d(q_j,F_j)\}\leq \epsilon_0/2$. Since $F'\neq F_j$, $j=1,2$, $\max_{q\in[p_j,q_j]}d(q,F_j)\leq\epsilon_0/2$ implies that $d([p_j,q_j],F')\geq d(F',F_j)-(\max_{q\in[p_j,q_j]}d(q,F_j))>\rho_0-\epsilon_0/2>2\rho_0/3$, $j=1,2$. By Lemma \ref{prep2}, $d([p_1,p_2],F')\leq \epsilon_1/4$.

Suppose \hyperlink{Omega-1}{property ($\Omega$1) of $\Omega(\cdot,\cdot)$} is proved for any $F'\in \Omega_{k-1}(F_1,F_2)$, then for any $F'\in\Omega_k(F_1,F_2)\setminus\Omega_{k-1}(F_1,F_2)$, there exist some $F_1'\neq F_2'\in\Omega_{k-1}(F_1,F_2)$ such that $F'\in \Omega_0(F_1',F_2')\setminus\{F_1',F_2'\}$. ($F_1'\neq F_2'$ is because $\Omega(F_1',F_1')=\{F_1'\}$ and $F'\not\in\Omega_{k-1}(F_1,F_2)$.) By the case $k-1$ applied to $F_j'\in\Omega_{k-1}(F_1,F_2)$, there exists $p_j'\in[p_1,p_2]$ such that either $p_j'=p_l$ and $F_j'=F_l$ for some $l\in\{1,2\}$, or $d(p_j', F_j')\leq \epsilon_1/4$. In both cases, we have $d(p_j', F_j')\leq \epsilon_0/2$. By the case $0$ applied to $F'\in\Omega_0(F_1',F_2')\setminus\{F_1',F_2'\}$, $d([p_1',p_2'],F')\leq \epsilon_1/4$. Hence $d([p_1,p_2],F')\leq d([p_1',p_2'],F')\leq \epsilon_1/4$.

As a corollary, let $p_j\in X$ such that $d(p_j,F_j)\leq \epsilon_0/2$, $j=1,2$. For any $\widehat F\in\Omega(F_1,F_2)\setminus\{F_1,F_2\}$, let $r_{\widehat F}\in[p_1,p_2]$ such that $d(r_{\widehat F}, \widehat F)\leq \epsilon_1/4$. For simplicity, we set $r_{F_j}=p_j$, $j=1,2$. Then for any $F'\neq F''\in\Omega(F_1,F_2)$, $d(r_{F'}, r_{F''})>\rho_0-\epsilon_0\geq2\rho_0/3$. Hence $|\Omega(F_1,F_2)|\leq 1+((d(p_1,p_2)+\epsilon_0)/(2\rho_0/3))<\infty$.
\item[($\Omega$2).] The proof of $\Omega(F_j',F_2)=\{F_j',...,F_k',F_2\}$ is the same as $\Omega(F_1,F_j')=\{F_1,F_1',...,F_j'\}$ by switching the order of $F_1$ and $F_2$. Therefore we only need to prove $\Omega(F_1,F_j')=\{F_1,F_1',...,F_j'\}$.

\textbf{Proof of $\Omega(F_1,F_j')\subset\{F_1,F_1',...,F_j'\}$}: For simplicity, we write $F_{k+1}':=F_2$. If $F_i'\in \Omega(F_1,F_j')$ for some $i>j$, then by \hyperlink{Omega-1}{property ($\Omega$1) of $\Omega(\cdot,\cdot)$}, there exist $r_i'\in[p_1,r_j]$ such that $d(r_i',F_i')\leq \epsilon_1/4$. (Since $F_1\neq F_2$ and $F_2\neq F_j'$ by the assumption that $1\leq j\leq k$.) By convexity of distance functions in $X$, $d(q_i,F_i')\leq \max\{d(r_i,F_i'),d(r_i',F_i')\}\leq \epsilon_0/2$ for any $q_i\in[r_i',r_i]$. In particular, since $r_j\in[p_1,r_i]$ due to the assumptions on the ordering of $r_1,...,r_k$ and the assumption that $i>j$, $r_j\in[r_i',r_i]$ and hence $d(r_j,F_i')\leq \epsilon_0/2$. This implies $d(F_j',F_i')\leq d(F_j',r_j)+d(F_i',r_j)\leq \epsilon_0$. Since $F_i'\neq F_j'$ and hence $d(F_j',F_i')>\rho_0>\epsilon_0$. This leads to a contradiction. Similarly, $\Omega(F_j',F_2)\subset\{F_j',...F_k',F_2\}$.

\textbf{Proof of $\{F_1,F_1',...,F_j'\}\subset\Omega(F_1,F_j')$}: Since $\Omega(F_1,F_2)$ is the union of $\Omega_m(F_1,F_2)$, $m\geq0$, it suffices for us to prove the following \emph{claim}:
\begin{center}
If $i<j$ and $F_i'\in \Omega_m(F_1,F_2)$ for some $m\geq 0$, then $F_i'\in \Omega(F_1,F_j')$.
\end{center}
We prove this using induction on $m$. If $m=0$, then by definition of $\Omega_0$, there exist $q_l\in X$ such that $d(q_l,F_l)\leq \epsilon_0/2$, $l=1,2,$ and $d([q_1,q_2],F_i')\leq\epsilon_2/2$. (See Figure \ref{lem4.16-1}.) Let $x_i\in [q_1,q_2]$ such that $d(x_i,F_i')\leq\epsilon_2/2$. By \hyperlink{Omega-1}{property ($\Omega$1) of $\Omega(\cdot,\cdot)$}, there exists $x_j\in[q_1,q_2]$ such that $d(x_j,F_j')\leq\epsilon_1/4<\epsilon_0/2$. If $x_i\in [x_j,q_2]$, then $F_i'\in\Omega_0(F_j', F_2)$. This contradicts with the fact that $\Omega(F_j',F_2)\subset\{F_j',...F_k',F_2\}$. (See the previous paragraph.) Therefore $x_i\in[q_1,x_j]$ and hence $F_i'\in\Omega(F_1,F_j')$.

\begin{figure}[h]
	\centering
	\includegraphics[width=4in]{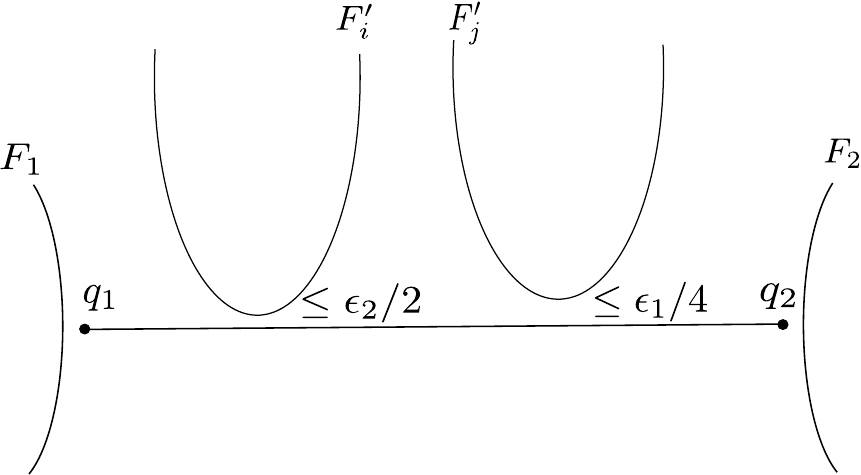}
	\caption{ \label{lem4.16-1}}
\end{figure}

Suppose the above claim is true for $\leq m-1$, for any $F_i'\in \Omega_m(F_1,F_2)\setminus \Omega_{m-1}(F_1,F_2)$, there exist some $F',F''\in \Omega_{m-1}(F_1,F_2)$ such that $F_i'\in\Omega_0(F',F'')\setminus\{F',F''\}$. By the definition of $\Omega_0(\cdot,\cdot)$, there exist $p',p''\in X$ and $y_i\in[p',p'']$ such that $d(p',F')\leq\epsilon_0/2$, $d(p'',F'')\leq\epsilon_0/2$ and $d(y_i, F_i')\leq \epsilon_2/2$. There are 3 possible cases for $F', F''$:

\textbf{Case 1:} $F', F''\in \Omega_{m-1}(F_1,F_2)\ints \{F_1,F_1',...,F_j'\}$. By the claim in the case of $m-1$, $F',F''\in\Omega(F_1,F_j')$. Hence $F_i'\in\Omega(F',F'')\subset\Omega(F_1,F_j')$.

\textbf{Case 2:} $F', F''\in \Omega_{m-1}(F_1,F_2)\ints \{F_j',...F_k',F_2\}$.
By the claim in the case of $m-1$, $F',F''\in\Omega(F_j',F_2)$. Hence $F_i'\in\Omega(F',F'')\subset\Omega(F_j',F_2)\subset \{F_j',...F_k',F_2\}$. (See the second paragraph in the proof of \hyperlink{Omega-2}{property ($\Omega$2) of $\Omega(\cdot,\cdot)$}.) This is impossible since we assumed that $i<j$ in the claim.

\textbf{Case 3:} One of $F', F''$ is in $\{F_1,F_1',...,F_{j-1}'\}$ and the other one is in $\{F_{j+1}',...F_k',F_2\}$. WLOG, we assume that $F'\in\{F_1,F_1',...,F_{j-1}'\}$ and $F''\in\{F_{j+1}',...F_k',F_2\}$. By the claim in the case $m-1$, $F'\in\Omega(F_1,F_j')$ and $F''\in\Omega(F_j',F_2)$.

\begin{figure}[h]
	\centering
	\includegraphics[width=4in]{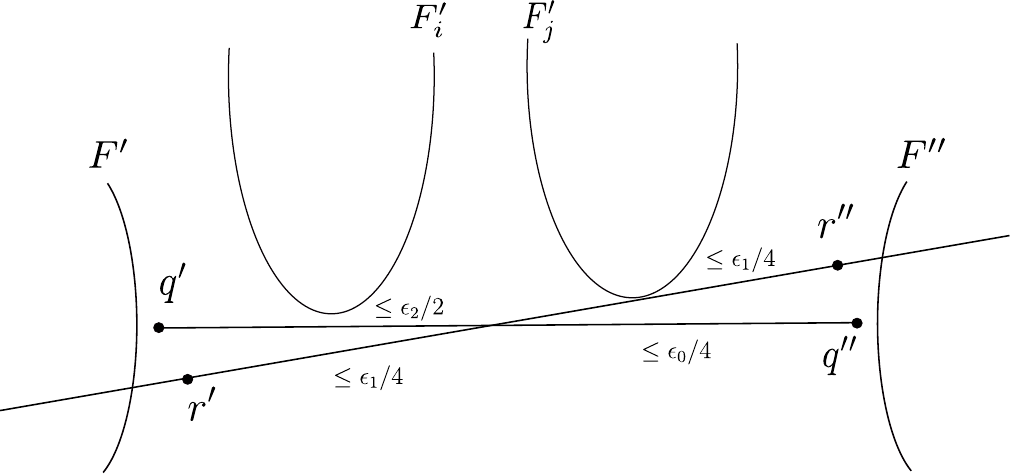}
	\caption{ \label{lem4.16-2}}
\end{figure}

By previous assumptions on $r_l$, there exist $r'\in \{p_1,r_1,... ,r_{j-1}\}$ and $r''\in\{r_{j+1},...,r_k,p_2\}$ such that $d(r',F')\leq \epsilon_0/2$ and $d(r'',F'')\leq \epsilon_0/2$. Since $F_i'\in \Omega(F',F'')\setminus\{F',F''\}$, there exist $q',q''\in X$ such that $d(q',F')\leq \epsilon_0/2$, $d(q'',F'')\leq \epsilon_0/2$ and $d([q',q''],F_i')\leq \epsilon_2/2$. (See Figure \ref{lem4.16-2}.) By \hyperlink{Omega-1}{property ($\Omega$1) of $\Omega(\cdot,\cdot)$}, $d([r',r''],F_j')\leq \epsilon_1/4<\epsilon_0/2$. Also, $F_j'\not\in\{ F', F''\}$ (from the assumptions of \textbf{Case 3}) implies that $d([r',q'], F_j'),d([r'',q''], F_j')>\rho_0-(\epsilon_0/2)\geq 2\rho_0/3$. (This is because $[r',q']$ lies in the $\epsilon_0/2$-neighborhood of $F'$ and $[r'',q'']$ lies in the $\epsilon_0/2$-neighborhood of $F''$.) By Lemma \ref{prep2} and \eqref{eqn:epsilon 0-2}, $d([q',q''],F_j')\leq \epsilon_0/4$. Let $y_i,y_j\in[q', q'']$ such that $d(y_i,F_i')\leq\epsilon_2/2$ (from the definition of $q', q''$) and $d(y_j,F_j')\leq\epsilon_0/2$. If $y_i\in[y_j,q'']$, then $F_i'\in\Omega(F_j', F'')$. This contradicts $\Omega(F_j, F'')\subset\Omega(F_j',F_2)\subset\{F_j',...F_k',F_2\}$. (See the second paragraph in the proof of \hyperlink{Omega-2}{property ($\Omega$2) of $\Omega(\cdot,\cdot)$}.) Therefore $y_i\in[q',y_j]$ and hence $F_i'\in\Omega_0(F',F_j')\subset\Omega(F',F_j')\subset \Omega(F_1, F_j')$.\qedhere
\end{enumerate}
\end{proof}

For technical reasons in the bicombing construction, we want to introduce $\Theta(F_1,F_2)$ as a refined notion of elements in $\Gamma F$ which are between $F_1,F_2$ compared to $\Omega(F_1,F_2)$. In particular, $\Theta(F_1,F_2)\subset\Omega(F_1,F_2)$ and elements in $\Theta(F_1,F_2)$ can be thought of as ``elites'' in $\Omega(F_1,F_2)$ since they satisfy additional properties. (Compare Lemma \ref{properties of Omega} and Lemma \ref{properties of Theta}.)

Let
\begin{align}\label{eqn:epsilon 3-4}
\epsilon_3=c_2(2,\epsilon_0/2,\epsilon_2/4)\leq \epsilon_2/4\mathrm{~and~}
\epsilon_4=c_4(\epsilon_3,\rho_0)\leq \epsilon_3/4.
\end{align}
(See Proposition \ref{almost sep edge, almost sep simplex} for $c_2(\cdot,\cdot,\cdot)$ and Proposition \ref{almost ints positioning} for $c_4(\cdot,\cdot)$.) For any $F_1, F_2\in\Gamma F$, we define
$$\Theta_0(F_1,F_2)=\{F_1,F_2\}\union\left\{\widehat F\in\Gamma F\left| \begin{aligned}
&\exists p_j'\in X~\mathrm{and}~F_j'\in\Omega(F_1,F_2)~\mathrm{s.t.}~\widehat F\neq F_j',\\
& d(p_j',F_j')\leq \epsilon_0/2, j=1,2,~\mathrm{and}~d([p_1',p_2'],\widehat F)<\epsilon_4/2.
\end{aligned}\right.\right\}.$$
Clearly $F_j\in\Theta_0(F_1,F_2)$, $j=1,2$ and $\Theta_0(F_1,F_1)=\{F_1\}$. Similar to $\Omega(\cdot,\cdot)$, we define
$$\Theta_k(F_1,F_2):=\bigcup_{F',F''\in\Theta_{k-1}(F_1,F_2)}\Theta_0(F',F'').$$
Then $\Theta_0(F_1,F_2),...,\Theta_k(F_1,F_2),...$ is an increasing sequence of subsets of $\Gamma F$. Finally, we define
$$\Theta(F_1,F_2):=\bigcup_{j=0}^\infty\Theta_{k}(F_1,F_2).$$
We have the following properties for $\Theta$.

\begin{lemma}\label{properties of Theta}
$\Theta(\cdot,\cdot)$ satisfies the following properties:
\begin{enumerate}
\item[\hypertarget{Theta-0}{($\Theta$0).}] $\Theta_0(F_1,F_2)\subset \Omega(F_1,F_2)$;
\item[\hypertarget{Theta-1}{($\Theta$1).}] For any $F'\in\Theta_0(F_1,F_2)\setminus\{F_1,F_2\}$ and any $p_j\in X$ such that $d(p_j, F_j)\leq \epsilon_0/2$, j=1,2, we have $d([p_1,p_2],F')\leq \epsilon_3/4$. As a corollary, $|\Theta_0(F_1,F_2)|<\infty$;
\item[\hypertarget{Theta-2}{($\Theta$2).}] $\Theta(\cdot,\cdot)=\Theta_0(\cdot,\cdot)$. In particular, combined with property $(\Theta0)$ of $\Theta(\cdot,\cdot)$, we have $\Theta(\cdot,\cdot)\subset \Omega(\cdot,\cdot)$;
\item[\hypertarget{Theta-3}{($\Theta$3).}] Let $F_1',...,F_k'$ be distinct elements in $\Gamma F\setminus\{F_1,F_2\}$ such that $\Theta(F_1,F_2)=\{F_1,F_1',...,F_k',F_2\}$ and let $F'\in\Omega(F_1,F_2)$. For any $p_1,p_2\in X$ such that $d(p_1,F_1)\leq \epsilon_0/2$ and $d(p_2,F_2)\leq\epsilon_0/2$, if $r,r_1,...,r_k$ are points in $[p_1,p_2]$ such that the following holds: (The existence of $r,r_1,...,r_k$ are guaranteed by either Lemma \ref{properties of Omega}, \hyperlink{Omega-1}{property ($\Omega$1) of $\Omega(\cdot,\cdot)$}, or Lemma \ref{properties of Theta},  \hyperlink{Theta-1}{property ($\Theta$1) of $\Theta(\cdot,\cdot)$}.)
\begin{itemize}
\item $d(r, F')\leq \epsilon_0/2$. Also, $d(r_i,F_i')\leq \epsilon_0/2$ for any $i\in\{1,...,k\}$.
\item $r_i\in[p_1,r_l]$ whenever $i<l$.
\item For simplicity, we let $r_0:=p_1$, $r_{k+1}:=p_2$, $F_0':=F_1$ and $F_{k+1}':=F_2$. By the previous bullet point and the above notations, we have $r\in[r_j,r_{j+1}]$ for some $j\in\{0,...,k\}$. We assume in addition that $r\in\{r_j,r_{j+1}\}$ if $F'\in\{F_j',F_{j+1}'\}$.
\end{itemize}
Then we have
$$\Theta(F_1,F')=\{F_1,F_1',...,F_j'\}\union\{F'\}=(\Omega(F_1,F')\ints\Theta(F_1,F_2))\union\{F'\}$$
and
$$\Theta(F',F_2)=\{F'\}\union\{ F_{j+1}',...,F_k',F_2\}=\{F'\}\union(\Omega(F',F_2)\ints\Theta(F_1,F_2)).$$
As a direct corollary, for any $F'\in\Omega(F_1,F_2)$, $\Theta(F_1,F_2)\union\{F'\}=\Theta(F_1, F')\union\Theta(F',F_2)$ and $\Theta(F_1,F')\ints\Theta(F', F_2)=\{F'\}$.
\item[\hypertarget{Theta-4}{($\Theta$4).}] For any distinct points $x,y,z\in \Gamma x_0$, we define $\cF(x,y,z)=\Theta(F_x,F_y)\union\Theta(F_y,F_z)\union\Theta(F_z,F_x)$ and
$$\cA(x,y,z)=(\Theta(F_x,F_y)\ints\Theta(F_y,F_z))\union(\Theta(F_y,F_z)\ints\Theta(F_z,F_x))\union(\Theta(F_z,F_x)\ints\Theta(F_x,F_y))$$
Then $|\cF(x,y,z)\setminus\cA(x,y,z)|\leq 3$.
\end{enumerate}
\end{lemma}

\begin{proof}
Since $\Theta(F_1,F_1)=\Theta_0(F_1,F_1)=\{F_1\}$, properties ($\Theta0$)-($\Theta3$) of $\Theta(\cdot,\cdot)$ hold trivially if $F_1=F_2$. Hence we assume that $F_1\neq F_2$.
\begin{enumerate}
\item[($\Theta$0).] Let $F'\in\Theta_0(F_1,F_2)\setminus\{F_1,F_2\}$. Then there exist $q_j'\in X$ and $F_j'\in\Omega(F_1,F_2)$ such that $F'\neq F_1',F_2'$, $d(q_j',F_j')\leq\epsilon_0/2$, $j=1,2,$ and $d([q_1',q_2'], F')\leq \epsilon_4/2$. By \eqref{eqn:epsilon 3-4}, $\epsilon_4\leq\epsilon_3/4\leq\epsilon_2/4$. In particular, we have $F'\in\Omega_0(F_1', F_2')\subset \Omega(F_1,F_2)$.
\item[($\Theta$1).] Let $q_j'\in X$ and $F_j'\in\Omega(F_1,F_2)$ such that $F'\neq F_1',F_2'$, $d(q_j',F_j')\leq\epsilon_0/2$, $j=1,2,$ and $d([q_1',q_2'], F')\leq \epsilon_4/2$. By Lemma \ref{properties of Omega}, \hyperlink{Omega-1}{property ($\Omega$1) of $\Omega(\cdot,\cdot)$}, there exists $p_j'\in[p_1,p_2]$ such that $d(p_j', F_j')\leq \epsilon_0/2$, $j=1,2$. (If in addition that $F_j'\neq F_1, F_2$, then $d(p_j', F_j')\leq \epsilon_1/4$.) Since $F'\neq F_1', F_2' $, for any $q'\in [p_j',q_j']$, $d(q',F')\geq d(F_j',F)-d(q',F_j')\geq d(F_j',F')-\max\{d(p_j',F_j'),d(q_j',F_j')\}>\rho_0-\epsilon_0\geq 2\rho_0/3$, $j=1,2$. Therefore by Lemma \ref{prep2} (applied to $F'$ and $q_1',p_1',q_2',p_2'$) and \eqref{eqn:epsilon 3-4}, $d([p_1,p_2],F')\leq d([p_1',p_2'],F')\leq \epsilon_3/4$ .
\item[($\Theta$2).] By definition of $\Theta_k(\cdot,\cdot)$, it suffices to show that $\Theta_1(F_1,F_2)=\Theta_0(F_1,F_2)$. Suppose there exist some $F'\in \Theta_1(F_1,F_2)\setminus\Theta_0(F_1,F_2)$, (in particular $F'\neq F_1, F_2$,) there exist some $F_1',F_2'\in\Theta_0(F_1,F_2)$ such that $F'\in\Theta_0(F_1',F_2')\setminus\{F_1',F_2'\}$. By definition of $\Theta_0(\cdot,\cdot)$, there exist $F_1'', F_2''\in\Omega(F_1',F_2')$ and $p_j''\in X$ such that $F'\neq F_1'', F_2''$, $d(p_j'', F_j'')\leq\epsilon_0/2$ for $j=1,2$, and $d([p_1'',p_2''], F')\leq \epsilon_4/2$. By \hyperlink{Theta-0}{property ($\Theta$0) of $\Theta(\cdot,\cdot)$}, $F_1', F_2'\in\Theta_0(F_1, F_2)\subset\Omega(F_1, F_2)$. Hence $F_1'', F_2''\in\Omega(F_1', F_2')\subset\Omega(F_1, F_2)$. This implies that $F'\in\Theta_0(F_1,F_2)$. This contradicts the assumption that $F'\in \Theta_1(F_1,F_2)\setminus\Theta_0(F_1,F_2)$. Hence $\Theta_0(F_1,F_2)=\Theta_1(F_1,F_2)$.

\item[($\Theta$3).] The result is clearly true when $F'=F_1$ or $F_2$. Therefore we assume that $F'\neq F_1,F_2$. The proof of
$\Theta(F_1,F')=\{F_1,F_1',...,F_j'\}\union\{F'\}$ is the same as the proof of  $\Theta(F_j',F_2)=\{F'\}\union\{ F_{j+1}',...,F_k',F_2\}$ by switching the order of $F_1, F_2$. Therefore we only need to prove $\Theta(F_1,F')=\{F_1,F_1',...,F_j'\}\union\{F'\}$.

We first show that $\Theta(F_1,F')\subset \Theta(F_1,F_2)\cup\{F'\}$. Clearly $F_1,F'\in\Theta(F_1,F_2)\cup\{F'\}$. For any $F''\in\Theta(F_1,F')\setminus\{F_1,F'\}$, by \hyperlink{Theta-2}{property ($\Theta$2) of $\Theta(\cdot,\cdot)$}, Lemma \ref{properties of Omega}, \hyperlink{Omega-2}{property ($\Omega$2) of $\Omega(\cdot,\cdot)$} and the definition of $\Theta_0(\cdot,\cdot)$, there exist $F_1'', F_2''\in\Omega(F_1,F')\subset\Omega(F_1,F_2)$ and $p_j''\in X$ such that $F'\neq F_1'', F_2''$, $d(p_j'', F_j'')\leq\epsilon_0/2$ for any $j=1,2$, and $d([p_1'',p_2''], F'')\leq \epsilon_4/2$. This implies that $F''\in\Theta(F_1,F_2)$. Hence we conclude that $\Theta(F_1,F')\subset \Theta(F_1,F_2)\cup\{F'\}$.

\textbf{Proof of $\Theta(F_1,F_j')\subset\{F_1,F_1',...,F_j'\}\union\{F'\}$}: By \hyperlink{Theta-2}{property ($\Theta$2) of $\Theta(\cdot,\cdot)$} and Lemma \ref{properties of Omega}, \hyperlink{Omega-2}{property ($\Omega$2) of $\Omega(\cdot,\cdot)$}, for any $F_i'\in\Theta(F_1,F_2)$ with $i>j$ and $F_i'\neq F'$, we have $F_i'\not\in\Omega(F_1,F')$. Therefore by \hyperlink{Theta-2}{property ($\Theta$2) of $\Theta(\cdot,\cdot)$}, $F_i'\not\in \Theta(F_1,F')\subset\Omega(F_1,F')$. Hence by the fact that $\Theta(F_1,F')\subset \Theta(F_1,F_2)\cup\{F'\}$, we have $\Theta(F_1,F_j')\subset\{F_1,F_1',...,F_j'\}\cup\{ F'\}$. (Similarly, $\Theta(F',F_2)\subset\{F'\}\cup\{ F_{j+1}',...F_k',F_2\}$.)

\textbf{Proof of $\{F_1,F_1',...,F_j'\}\union\{F'\}\subset\Theta(F_1,F')$}: Since $\Theta_0(\cdot,\cdot)=\Theta(\cdot,\cdot)$ (by \hyperlink{Theta-2}{property ($\Theta$2) of $\Theta(\cdot,\cdot)$}), for any $F_i'\in\Theta(F_1,F_2)$ such that $1\leq i\leq j$ and $F_i'\neq F'$, there exist $F_1'', F_2''\in\Omega(F_1,F_2)$ and $q_1'',q_2''\in X$ such that $F_i'\neq F_1'', F_2''$, $d(q_j'', F_j'')\leq\epsilon_0/2$, $j=1,2,$ and $d([q_1'',q_2''],F_i')\leq\epsilon_4/2$. By Lemma \ref{properties of Omega}, \hyperlink{Omega-2}{property ($\Omega$2) of $\Omega(\cdot,\cdot)$}, we have the following $3$ cases:

\textbf{Case 1}: If $F_1'',F_2''\in\Omega(F_1, F')$, then by the definition of $\Theta_0(\cdot,\cdot)$ and \hyperlink{Theta-2}{property ($\Theta$2) of $\Theta(\cdot,\cdot)$}, $F_i'\in\Theta_0(F_1,F')=\Theta(F_1,F')$.

\textbf{Case 2}:  If $F_1'',F_2''\in\Omega(F', F_2)$, then by the definition of $\Theta_0(\cdot,\cdot)$ and \hyperlink{Theta-2}{property ($\Theta$2) of $\Theta(\cdot,\cdot)$}, $F_i'\in\Theta_0(F', F_2)=\Theta(F', F_2)\subset\Omega(F', F_2)$. On the other hand, by Lemma \ref{properties of Omega}, \hyperlink{Omega-2}{property ($\Omega$2) of $\Omega(\cdot,\cdot)$} and the assumptions on $r,r_1,...,r_k$, $F_i'\in\Omega(F_1,F')\setminus\{F'\}$. This contradicts with $\Omega(F_1,F')\ints\Omega(F',F_2)=\{F'\}$ and $F_i'\in\Omega(F', F_2)$.

\textbf{Case 3}: If one of $F_1'', F_2''$ is in $\Omega(F_1, F')\setminus\{F'\}$ and the other one is in $\Omega(F', F_2)\setminus\{F'\}$, WLOG, we assume that $F_1''\in\Omega(F_1, F')\setminus\{F'\}$ and $F_2''\in\Omega(F', F_2)\setminus\{F'\}$. By both properties of Lemma \ref{properties of Omega}, $F'\in\Omega(F_1'',F_2'')\setminus\{F_1'',F_2''\}$ and hence there exists $q'\in[q_1'',q_2'']$ such that $d(q',F')= d([q_1'',q_2''],F')\leq \epsilon_1/4$. If $d([q',q_2''], F_i')\leq \epsilon_4/2$, then $F_i'\in\Omega(F',F_2'')\subset \Omega(F',F_2)$. On the other hand, by Lemma \ref{properties of Omega}, \hyperlink{Omega-2}{property ($\Omega$2) of $\Omega(\cdot,\cdot)$} and the assumptions on $r,r_1,...,r_k$, $F_i'\in\Omega(F_1,F')\setminus\{F'\}$,  which contradicts $\Omega(F_1,F')\ints\Omega(F',F_2)=\{F'\}$. Therefore $d([q_1'',q'],F_i')\leq \epsilon_4/2$ and hence $F_i'\in\Theta_0(F_1,F')=\Theta(F_1, F')$.
\item[($\Theta$4).] By \hyperlink{Theta-1}{property ($\Theta$1) of $\Theta(\cdot,\cdot)$}, for any $F'\in\cF(x,y,z)\setminus\cA(x,y,z)$, $F'$ is $\epsilon_3/4$-close to one of $[x,y]$, $[y,z]$ and $[z,x]$. Since $\epsilon_3=c_2(2,\epsilon_0/2,\epsilon_2/4)$ (due to \eqref{eqn:epsilon 3-4}) and that $d(p,\widehat F)\geq\epsilon_0/2$ for any $\widehat F\in\Gamma F$ and any $p\in\{x,y,z\}$, by Proposition \ref{almost sep edge, almost sep simplex}, $F'$ is $\epsilon_2/4$-almost separating $\Db^2(x,y,z)$. Notice that we have at most $3$ different $\epsilon_2/4$-almost separation types for $\Db^2(x,y,z)$. If \hyperlink{Theta-4}{property ($\Theta$4) of $\Theta(\cdot,\cdot)$} does not hold, by Proposition \ref{almost sep edge, almost sep simplex} again, we can assume WLOG that there exist distinct $F', F''\in\cF(x,y,z)\setminus\cA(x,y,z)$ satisfying the following properties:
\begin{enumerate}
\item[(i).] $F', F''\in \Theta(F_x,F_y)\union\Theta(F_x,F_z)$.
\item[(ii).] $F', F''$ are $\epsilon_2/4$-close to $[x,y],[x,z]$.
\end{enumerate}
By the definition of $\Omega(\cdot,\cdot)$ and the second property above, $F' , F''\in\Omega(F_x,F_y)\ints\Omega(F_x,F_z)\setminus\{F_x,F_y,F_z\}.$ By Lemma \ref{properties of Omega}, \hyperlink{Omega-2}{property ($\Omega$2) of $\Omega(\cdot,\cdot)$}, we assume WLOG that $F'\in\Omega(F_x,F'')\setminus\{F_x,F''\}$. (See Figure \ref{lem4.17}.) Then by \hyperlink{Theta-3}{property ($\Theta$3) of $\Theta(\cdot,\cdot)$}, we have
\begin{figure}[h]
	\centering
	\includegraphics[width=4in]{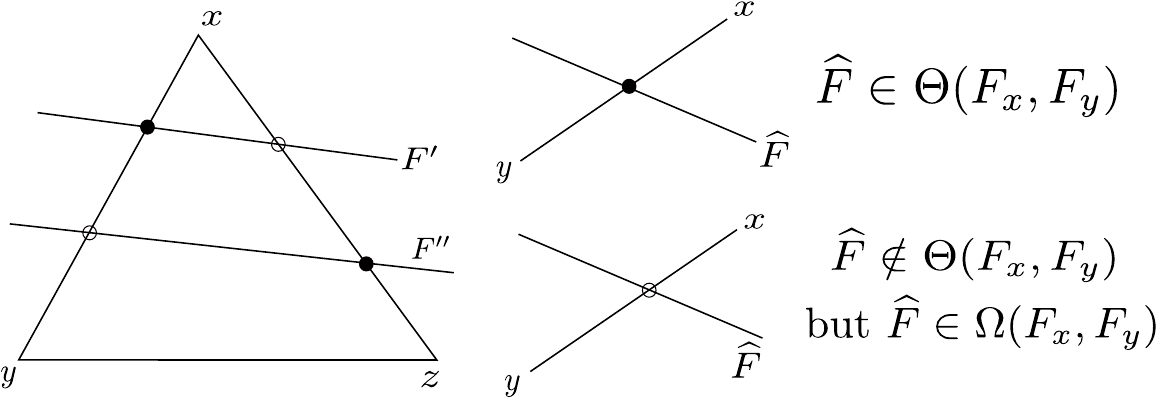}
	\caption{ \label{lem4.17}}
\end{figure}
\begin{equation*}
\left.
\begin{aligned}
&F'\in \Theta(F_x,F_y)\union\Theta(F_x,F_z)\\
&F''\in\Omega(F_x,F_y)\ints\Omega(F_x,F_z)\\
&F'\in\Omega(F_x,F'')\setminus\{F_x,F''\}\subset\Omega(F_x,F'')
\end{aligned}
\right\}\implies F'\in\Theta(F_x,F'')\setminus\{F''\}\subset\Theta(F_x,F_y)\ints\Theta(F_x,F_z).
\end{equation*}
(To be specific, the LHS implies that
\begin{align*}
F'\in& ((\Theta(F_x,F_y)\cup\Theta(F_x,F_z))\cap\Omega(F_x,F''))\setminus\{F''\}\\
=&((\Theta(F_x,F_y)\cap\Omega(F_x,F''))\cup(\Theta(F_x,F_z)\cap\Omega(F_x,F'')))\setminus\{F''\}=\Theta(F_x,F'')\setminus\{F''\},
\end{align*}
where the last equality uses \hyperlink{Theta-3}{property ($\Theta$3) of $\Theta(\cdot,\cdot)$}.) Therefore $F'\in\cA(x,y,z)$. This contradicts the original assumption that $F'\in\cF(x,y,z)\setminus\cA(x,y,z)$.
\qedhere
\end{enumerate}
\end{proof}

\section{Bicombing construction inspired by Mineyev}\label{sec bicombing}
We denote by $\cS_k(X)$ the set of all singular simplices in $X$ and $\cS_{k,x_0}(X)$ the set of all singular simplices in $X$ with vertices in $\Gamma x_0$.
\begin{definition}[Singular chains in $X$ and boundary map $\del^X$ for singular chains]\label{singular chains in X}
Let $C_k(X;\RR)$ and $C_{k,x_0}(X;\RR)$ be $\RR$-vector spaces with basis $\cS_k(X)$ and $\cS_{k,x_0}(X)$ respectively. Elements in $C_k(X;\RR)$ are called \emph{$k$-dimensional singular chains} in $X$. $C_k(X;\RR)$ is called the \emph{space of $k$-dimensional singular chains} in $X$ and $C_{k,x_0}(X;\RR)$ is called the space of $k$-dimensional singular chains in $X$ with vertices in $\Gamma x_0$. For any $k\geq 1$, we define the \emph{boundary map} $\del^X_k:C_k(X;\RR)\to C_{k-1}(X;\RR)$ as an $\RR$-linear map such that for any $\sigma\in \cS_k(X)$, $\del^X_k\sigma=\sum_{j=0}^k(-1)^j\sigma\circ\iota_j$, where
$$\iota_j:\Delta^{k-1}_{\RR^{k}}\to\Delta^k_{\RR^{k+1}},~\iota_j(t_0,...,\widehat{t_j},...,t_k)=(t_0,...,t_{j-1},0,t_{j+1},...,t_k),~\forall(t_0,...,\widehat{t_j},...,t_k)\in \Delta^{k-1}_{\RR^{k}}.$$
It is obvious to see that $\del^X_k(C_{k,x_0}(X;\RR))\subset C_{k-1,x_0}(X;\RR)$.
\end{definition}
\begin{rmk}
For simplicity, we omit $k$ and write $\del^X$ instead of $\del^X_k$ throughout this section only.
\end{rmk}
\begin{definition}[$l^1$-(semi)norm]\label{l1-seminorm}
For any $k\geq0$ and any $a\in C_k(X;\RR)$, we define $|a|_{l_1}$ to be its \emph{$l^1$-norm} with respect to the basis $\cS_k(X)$. To be specific, if $a=\sum_{\sigma\in\cS_k(X)}\alpha_\sigma\sigma$ for some $\alpha_\sigma\in\RR$, then $|a|_{l^1}=\sum_{\sigma\in\RR}|\alpha_\sigma|$. One can easily check that if $a\in C_{k,x_0}(X;\RR)$, $|a|_{l^1}$ is also the $l^1$-norm in $C_{k,x_0}(X;\RR)$ with respect to the basis $\cS_{k,x_0}(X)$.

For any $\cF\subset\Gamma F$ and any $a\in C_{k,x_0}(X;\RR)$ with $a=\sum_{\sigma\in\cS_{k,x_0}(X)}\alpha_\sigma\sigma$, we define the \emph{$l^1$-seminorm} near $\cF$ as
$$|a|_{\cF,l^1}:=\sum_{\substack{\sigma:\sigma\in\cS_{k,x_0}(X)\\ \forall x\in\Gamma x_0, x\mathrm{~is~a~vertex~of~}\sigma\implies F_x\in\cF}}|\alpha_\sigma|.$$
Clearly in $C_{k,x_0}(X;\RR)$,
$$|\cdot|_{\cF,l^1}\leq|\cdot|_{\cF',l^1}\leq|\cdot|_{l^1}=|\cdot|_{\Gamma F, l^1}$$
for any $\cF\subset\cF'\subset\Gamma F$ and any $k\geq 0$.
\end{definition}
\begin{definition}[Geodesic bicombing]\label{dfn:geo.bicomb}
For any $p,q\in X$, we denote by $\ovec{[p,q]}:=(\Db^1(p,q)-\Db^1(q,p))/2\in  C_1(X;\RR)$ \emph{the geodesic bicombing from $p$ to $q$}.
\end{definition}
\begin{rmk}
\begin{enumerate}
\item For any $m\in\ZZ_+$, any real numbers $\alpha_1,...,\alpha_m\in\RR$ and any pairwise distinct subsets $\{p_1,q_1\},...,\{p_m,q_m\}\subset \{V\subset X: |V|=2\}$, one can easily verify that
$$\left|\sum_{i=1}^m\alpha_i\cdot\ovec{[p_I,q_i]}\right|_{l^1}=\sum_{i=1}^m|\alpha_i|.$$
\item By the remark after Definition \ref{bar simplex}, the image of $\Db^1(p,q)$ is exactly the geodesic segment connecting $p$ and $q$. Hence, one can view $\ovec{[p,q]}$ as the oriented geodesic segment from $p$ to $q$. Moreover, $\ovec{[p,q]}+\ovec{[q,p]}=0$. $p$ and $q$ are called the \emph{endpoints} of $\ovec{[p,q]}$.
\end{enumerate}
\end{rmk}

For the rest of this section, we only consider barycentric simplices with vertices in $\Gamma x_0$. The goal of this section is to construct a \emph{homological bicombing} inspired by Mineyev \cite{Mineyev01}. First we introduce the definition of a homological bicombing.
\begin{definition}\label{Homological bicombing}
A \emph{homological bicombing} is a map $\beta:\Gamma x_0\times\Gamma x_0\to C_{1,x_0}(X;\RR)$ such that ${\del^X}\beta[p,q]=q-p$ for any $p,q\in \Gamma x_0$. Moreover, $\beta$ satisfies the following properties
\begin{enumerate}
\item[(1).] ($\Gamma$-equivariance) $\gamma\cdot\beta[p,q]=\beta[\gamma\cdot p,\gamma\cdot q]$ for any $\gamma\in\Gamma$ and $p,q\in\Gamma x_0$;
\item[(2).] (Anti-symmetry) $\beta[p,q]+\beta[q,p]=0$ for any $p,q\in\Gamma x_0$.
\end{enumerate}
\end{definition}
One can naturally see that the map $(p,q)\to\ovec{[p,q]}$ is a homological bicombing in the above sense. Hence any homological bicombing can be thought of as a generalization of oriented geodesic segments. Also, any homological bicombing can be naturally extended to a $\RR$-bilinear map $\beta:C_{0,x_0}(X;\RR)\times C_{0,x_0}(X;\RR)\to C_{1,x_0}(X;\RR)$. (Here, the extended map is also denoted by $\beta$.)

The rest of this section is devoted to contruct a special homological bicombing satisfying similar properties as in Mineyev \cite{Mineyev01}.
\begin{notation}\label{1-lvl lower flat}
For any $p,q\in \{x\in X|\exists \hF\in\Gamma F\mathrm{~s.t.~}d(x,\hF)\leq \epsilon_0/2\}\supset \Gamma x_0$ (see Notation \ref{Fx def}), by Lemma \ref{properties of Theta}, \hyperlink{Theta-3}{property ($\Theta$3) of $\Theta(\cdot,\cdot)$}, we can define $F_{p,q}\in\Theta(F_p,F_q)$ as follows: $F_{p,q}=F_p$ if $\Theta(F_p,F_q)=\{F_p,F_q\}$. Otherwise $F_{p,q}$ is the unique element in $\Theta(F_p,F_q)$ such that $\Theta(F_{p,q},F_q)=\Theta(F_p,F_q)\setminus\{F_p\}$. In other words, if $\Theta(F_p,F_q)=\{F_p,F_1,...,F_k,F_q\}$ for some $k\geq 0$ and distinct $F_1,...,F_k\in\Gamma F\setminus\{F_p,F_q\}$ such that $
\Theta(F_p,F_j)=\{F_p,F_1,...,F_j\}$, then
$$F_{p,q}=
\begin{cases}
\displaystyle  F_p,\quad&\mathrm{if}~k=0,\\
\displaystyle  F_1,\quad&\mathrm{if}~k\geq 1,
\end{cases}
\mathrm{~and~}
\Theta(F_{p,q},F_q)=
\begin{cases}
\displaystyle  \Theta(F_p,F_q),\quad&\mathrm{if}~k=0,\\
\displaystyle  \Theta(F_p,F_q)\setminus\{F_p\},\quad&\mathrm{if}~k\geq 1.
\end{cases}$$
\end{notation}
\begin{definition}[Flower]\label{flower}
For any $q\in X$ which is $\epsilon_0/2$-close to some $\hF$, we define the \emph{flower} at $q$ as
$$\Fl(q)=\left\{
 x\in\Gamma x_0\left|
\begin{aligned}
&d(x,\hF)=\epsilon_0/2 \\
&d(\Proj_\hF(x),\Proj_\hF(q))<\cC_2(\rho_0,\epsilon_0,0)+\mathrm{diam}(N)+1
\end{aligned}
\right.\right\},$$
where $\cC_2$ is introduced in Lemma \ref{prep1} and $N$ is introduced in Section \ref{sect:setting}. In particular, $\Fl(q)$ is never empty. (See Figure \ref{def5.6}.)
\end{definition}
\begin{figure}[h]
	\centering
	\includegraphics[width=4in]{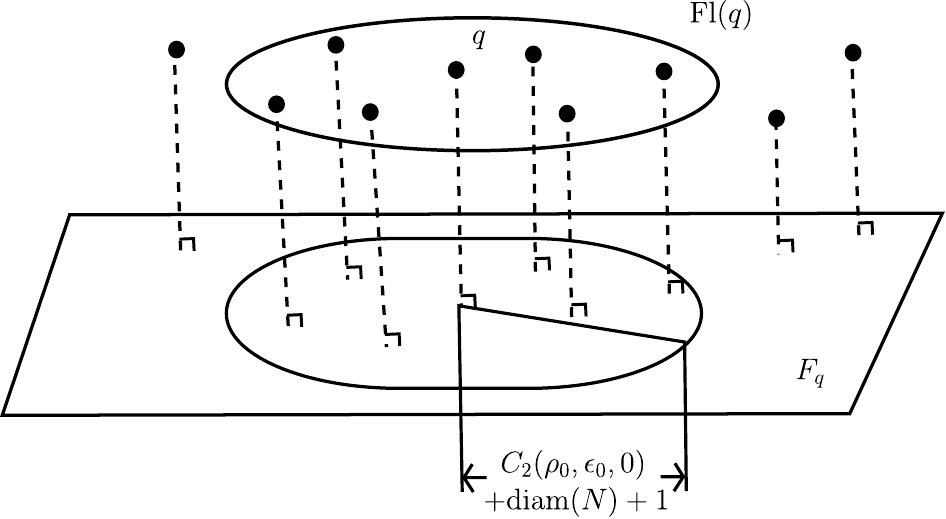}
	\caption{ \label{def5.6}}
\end{figure}
\begin{definition}[One-level lower projection]\label{1-lvl lower proj}
For any $p,q\in\{x\in X|\exists \hF\in\Gamma F\mathrm{~s.t.~}d(x,\hF)\leq \epsilon_0/2\}\supset \Gamma x_0$, we define the \emph{one-level lower projection} $\pr_p(q)$ from $q$ to $p$ satisfying the following two conditions:
\begin{itemize}
\item $\pr_p(q)\in\left\{x\in[p,q]| d(x,F_{q,p})=d([x,q],F_{q,p})\leq\epsilon_0/2\right\}=:L_{p,q,1}$;
\item Notice that $L_{p,q,1}$ is a closed subset of $[p,q]$, we further require that $d(\pr_p(q),q)=\min\{d(x,q)|x\in L_{p,q,1}\}$.
\end{itemize}
In other words, if we travel from $q$ to $p$ along the geodesic segment $[p,q]$, $\pr_p(q)$ is the first point which is $\epsilon_0/2$-close to $F_{q,p}$. (See Figure \ref{def5.7}.)
\end{definition}
\begin{figure}[h]
	\centering
	\includegraphics[width=4in]{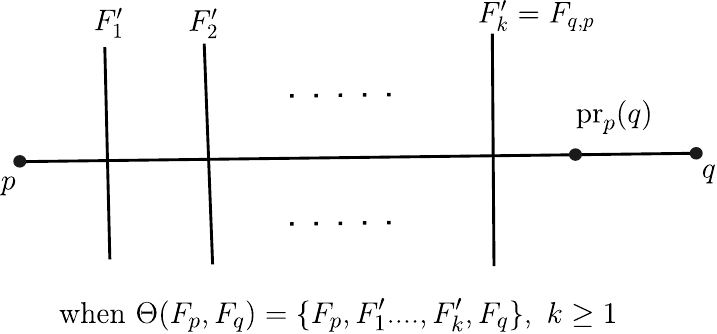}
	\caption{ \label{def5.7}}
\end{figure}
\begin{lemma}\label{easiest}
For any $x,y\in\{x\in X|\exists \hF\in\Gamma F\mathrm{~s.t.~}d(x,\hF)\leq \epsilon_0/2\}\supset \Gamma x_0$, $\pr_x(\cdot)$ satisfies the following properties:
\begin{enumerate}
\item[(1).] For any $m\geq 1$, $\pr_x^m(y)$ satisfies the following properties:
\begin{enumerate}
\item[(i).] $F_{\pr_x^{m}(y)}=F_{\pr_x^{m-1}(y),x}\in\Theta(F_x,F_y)$. Moreover, $\pr_x^m(y)\neq \pr_x^{m-1}(y)$ if and only if $|\Theta(F_x,F_{\pr_x^{m-1}(y)})|\geq 3$;
\item[(ii).] $\pr^m_x(y)\in\left\{w\in[x,y]\left|d(w, F_{\pr_x^m(y)})=d([w,y],F_{\pr_x^m(y)})\leq\epsilon_0/2\right.\right\}=:L_{x,y,m}$;
\item[(iii).] $d(\pr^m_x(y),y)=\min\{d(x,y)|x\in L_{x,y,m}\}$.
\end{enumerate}
As a straightforward corollary of the above properties and the convexity of distance functions in nonpositive curvature, $d(\pr^m_x(y),F_{\pr^m_x(y)})=\epsilon_0/2$ if $y\neq \pr^m_x(y)$.
\item[(2).] For any $m\geq\max\{0,|\Theta(F_x,F_y)|-2\}$, we have
$\pr_x^{m+1}(y)=\pr_x^m(y)=\lim_{j\to\infty}\pr_x^j(y)=:\pr_x^\infty(y)$. Moreover, when $|\Theta(F_x,F_y)|\geq 3$, we have $F_{\pr_x^\infty(y)}=F_{x,y}$.
\end{enumerate}
\end{lemma}
\begin{proof}
\begin{enumerate}
\item[(1).] We prove by induction on $m$. When $m=1$, $F_{\pr_x(y)}=F_{y,x}\in\Theta(F_x,F_y)$, (ii) and (iii) follow immediately from Definition \ref{1-lvl lower proj}. If $|\Theta(F_x,F_y)|\geq 3$, then $F_{y,x}\neq F_x,F_y$. Hence $y\not\in L_{x,y,1}$ and therefore $\pr_x(y)\neq y$. On the other hand, if $F_y=F_{\pr_x(y)}=F_{y,x}$, by Notation \ref{1-lvl lower flat}, $\Theta(F_x,F_y)=\{F_x,F_y\}$ and therefore $|\Theta(F_x,F_y)|< 3$. This finishes the proof when $m=1$.

Now we assume that the assertion is proved whenever $m\leq k-1$ for some $k\geq 2$. By (i) in the case $m=k-1$ and Lemma \ref{properties of Theta}, \hyperlink{Theta-3}{property ($\Theta$3) of $\Theta(\cdot,\cdot)$}, we have $\Theta(F_x,F_{\pr_x^{m-1}(y)})\subset\Theta(F_x,F_y)$. Write $z=F_{\pr_x^{m-1}(y)}$ for simplicity and (i) in the case $m=k$ follows verbatim from (i) in the case $m=1$ applied to $x,z$.

It remains for us to prove (ii) and (iii) in the case $m=k$. We divide the rest of the proof into two cases:

\textbf{Case 1}: $F_{z,x}=F_z$, then by Notation \ref{1-lvl lower flat} and Definition \ref{1-lvl lower proj}, $\Theta(F_x,F_z)=\{F_x,F_z\}$ and $\pr_x^m(y)=\pr_x(z)=z=\pr_x^{m-1}(y)$. Hence (ii) and (iii) in the case $m=k$ is the same as (ii) and (iii) in the case $m=k-1$.

\textbf{Case 2}: $F_{z,x}\neq F_z$, then $d([z,y],F_{z,x})>\epsilon_0/2$. (To be specific, notice that $\pr_x^m(y)=\pr_x(z)\in[x,z]$ and $d(\pr_x(z),F_{z,x})\leq \epsilon_0/2$, if $d([z,y],F_{z,x})\leq \epsilon_0/2$, by convexity of distance functions in nonpositive curvature, $d(z,F_{z,x})\leq \epsilon_0/2$. This implies that $F_{z,x}=F_z$, which contradicts the assumptions of this case.) Recall that $\pr_x^m(y)=\pr_x(z)\in[x,z]$ and $d(\pr_x(z),F_{z,x})\leq \epsilon_0/2$. For any $w\in[x,z]$, $d([w,y],F_{z,x})=d([w,z],F_{z,x})$ due to convexity of distance functions in nonpositive curvature. Hence $L_{x,y,m}=L_{x,z,1}$ and (ii) in the case $m=k$ follows directly from the definition of $\pr_x(z)=\pr_x^m(y)$. Since $d(w,y)=d(w,z)+d(z,y)$ for any $w\in[x,z]$, by the fact that $L_{x,y,m}=L_{x,z,1}\subset[x,z]$.
\begin{align*}
d(\pr_x^m(y),y)=d(\pr_x(z),y)=&d(\pr_x(z),z)+d(z,y)\\
=&\min\{d(w,z)|w\in L_{x,z,1}\}+d(z,y)=\min\{d(w,y)|w\in L_{x,y,m}\},
\end{align*}
which proves (iii) in the case $m=k$. This completes the induction.
\item[(2).] By Notation \ref{1-lvl lower flat} and Definition \ref{1-lvl lower proj}, we have
$$|\Theta(F_x,F_{\pr_x^m(y)})|=
\begin{cases}
\displaystyle |\Theta(F_x,F_{\pr_x^{m-1}(y)})|,\quad& \mathrm{~if~}|\Theta(F_x,F_{\pr_x^{m-1}(y)})|\leq 2,\\
\displaystyle |\Theta(F_x,F_{\pr_x^{m-1}(y)})|-1,\quad &\mathrm{~if~}|\Theta(F_x,F_{\pr_x^{m-1}(y)})|\geq 3.
\end{cases}$$
Hence when $m\geq\max\{0,|\Theta(F_x,F_y)|-2\}$, $|\Theta(F_x,F_{\pr_x^m(y)})|\leq 2$. Moreover, for such $m$, $|\Theta(F_x,F_{\pr_x^m(y)})|= 2$ when $|\Theta(F_x,F_y)|\geq 2\Leftrightarrow F_x\neq F_y$. By the first assertion in Lemma \ref{easiest}, $\pr_x^{m+1}(y)=\pr_x^m(y)=\lim_{j\to\infty}\pr_x^j(y)=:\pr_x^\infty(y)$ for any $m\geq\max\{0,|\Theta(F_x,F_y)|-2\}$.

In the case when $|\Theta(F_x,F_y)|\geq 3$, $|\Theta(F_x,F_{\pr_x^\infty(y)})|=2$. Since $\Theta(F_x,F_{\pr_x^\infty(y)})\subset\Theta(F_x,F_y)$ and $F_{x,y}\neq F_x$, by Notation \ref{1-lvl lower flat} and Lemma \ref{properties of Theta}, \hyperlink{Theta-3}{property ($\Theta$3) of $\Theta(\cdot,\cdot)$}, we have $F_{\pr_x^\infty(y)}=F_{x,y}$.\qedhere
\end{enumerate}
\end{proof}

The construction of the homological bicombing is a weighted sum of {geodesic bicombings}. The endpoints of all these {geodesic bicombings} are closely related to the following construction: For any $p,q\in\Gamma x_0$, we defined a singular $0$-chain $f(p,q)\in C_{0,x_0}(X;\RR)$ by the following induction.
\begin{align}\label{f}
f(x,y)=
\begin{cases}
\displaystyle y,\quad&\mathrm{if~}\Theta(F_x,F_y)=\{F_x,F_y\},     \\
\displaystyle \frac{1}{|\Fl(\pr_x(y))|}\sum_{z\in\Fl(\pr_x(y))}f(x,z), \quad& \mathrm{else}.
\end{cases}
\end{align}
From Definition \ref{flower}, Definition \ref{1-lvl lower proj} and \eqref{f}, we notice that for any $z\in\Fl(\pr_x(y))$, $F_z=F_{\pr_x(y)}=F_{y,x}$. By Lemma \ref{properties of Theta}, \hyperlink{Theta-1}{property ($\Theta$1) of $\Theta(\cdot,\cdot)$}, the inductive definition in \eqref{f} using the second case will end in finite steps. Hence $f(\cdot,\cdot)$ is well-defined. Similar to homological bicombings, we can naturally extend $f$ to an $\RR$-bilinear map on $C_{0,x_0}(X;\RR)\times C_{0,x_0}(X;\RR)$ with images in $C_{0,x_0}(X;\RR)$. To be specific, for any $a_j=\sum_{x\in\Gamma x_0}{\alpha^{(j)}_x}\cdot x\in C_{0,x_0}(X;\RR)$ with $j\in\{1,2\}$ and $\alpha^{(j)}_x\in\RR$, the extended map (also denoted by $f$) satisfies
$$f(a_1,a_2)=\sum_{x,y\in\Gamma x_0}\alpha^{(1)}_x\alpha^{(2)}_yf(x,y).$$

\begin{definition}[Support of a singular $0$-chain with vertices in $\Gamma x_0$]\label{supp 0-chain}
For any singular $0$-chain $a\in C_{0,x_0}(X;\RR)$, there exists a unique map $x\to\alpha_x$ from $\Gamma x_0$ to $\RR$ such that $a=\sum_{x\in\Gamma x_0}\alpha_x\cdot x$. We define its \emph{support} as
$$\supp(a)=\{x\in\Gamma x_0|\alpha_x\neq 0\}.$$
For any $\hF\in\Gamma x_0$, we say that $a$ is \emph{supported near $\hF$} if $F_x=\hF$ for any $x\in\supp(a)$.
\end{definition}

\begin{definition}[Convex combination]\label{convex comb}
A singular $0$-chain $a\in C_{0,x_0}(X;\RR)$ is called a \emph{convex combination} if there exists a unique map $x\to\alpha_x$ from $\Gamma x_0$ to $\RR_{\geq 0}$ such that $a=\sum_{x\in\Gamma x_0}\alpha_x\cdot x$ with $|a|_{l^1}=\sum_{x\in\Gamma x_0}\alpha_x=1$.
\end{definition}

The following lemma is an analogue of \cite[Proposition 7]{Mineyev01}.

\begin{proposition}\label{prop of f}
The map $f:\Gamma x_0\times\Gamma x_0\to C_{0,x_0}(X;\RR)$ satisfies the following properties:
\begin{enumerate}
\item[(1).] For any $x,y\in \Gamma x_0$, $f(x,y)$ is a convex combination supported near $F_{\pr^\infty_x(y)}$. Moreover, for any $\hF\in\Theta(F_x,F_y)\setminus\{F_x\}$, there exist a convex combination $a\in C_{0,x_0}(X;\RR)$ supported near $\hF$ such that $f(x,y)=f(x,a)$;
\item[(2).] $f(x,y)$ is $\Gamma$-equivariant. Namely, $f(\gamma x,\gamma  y)=\gamma\cdot f(x,y)$ for any $x,y\in\Gamma x_0$ and $\gamma\in\Gamma$;
\item[(3).] For any $x,y\in \Gamma x_0$ and $m\geq\max\{0,|\Theta(F_x,F_y)|-2\}$,
$$\supp(f(x,y))\subset\left\{z\in\Gamma x_0\left|\begin{aligned}
&F_z=F_{\pr^\infty_x(y)}, \\
&d(\Proj_{F_{\pr^\infty_x(y)}}(z),\Proj_{F_{\pr^\infty_x(y)}}(\pr^\infty_x(y)))<2\cC_2(\rho_0,\epsilon_0,0)+\mathrm{diam}(N)+2
\end{aligned}\right.\right\}.$$
\item[(4).] There exist some $\cC_3=\cC_3(\rho_0,\epsilon_0)>0$ and $0\leq \lambda=\lambda(\rho_0,\epsilon_0)<1$ such that
$$|f(x,y)-f(x,z)|_{l^1}\leq \cC_3\lambda^{(F_y|F_z)_{F_x}},$$
where $(F_1|F_2)_{F_3}:=|\Theta(F_1,F_3)\ints\Theta(F_2,F_3)|$ for any $F_1,F_2,F_3\in\Gamma F$;
\item[(5).] For the same $\lambda$ as in the above and any $y,z\in\Gamma x_0$ with $F_y=F_z$, we have
$$|f(y,x)-f(z,x)|_{l^1}\leq 2\lambda,\quad\forall~x\in \Gamma x_0\mathrm{~with~}F_x\neq F_y.$$
\end{enumerate}
\end{proposition}
\begin{proof}
\begin{enumerate}
\item[(1).] The assertion clearly holds when $F_x=F_y$. Therefore we assume that $F_x\neq F_y$. Let $\Theta(F_x,F_y)=\{F_x,F_1,...,F_k,F_y\}$ for some $k\geq 0$ and $F_1,...,F_k\not\in\{F_x,F_y\}$ distinct such that $
\Theta(F_p,F_j)=\{F_p,F_1,...,F_j\}$. Then for any $1\leq j\leq k$, we have
\begin{align*}
f(x,y)
=&\frac{1}{|\Fl(\pr_x(y))|}\sum_{z_k:z_k\in\Fl(\pr_x(y))}f(x,z_k)\\
=&\frac{1}{|\Fl(\pr_x(y))|}\sum_{z_k:z_k\in\Fl(\pr_x(y))}\frac{1}{|\Fl(\pr_x(z_k))|}\sum_{z_{k-1}:z_{k-1}\in\Fl(\pr_x(z_{k}))}f(x,z_{k-1})\\
=&\frac{1}{|\Fl(\pr_x(y))|}\sum_{\substack{z_{k-2},z_{k-1},z_k:z_k\in\Fl(\pr_x(y))\\z_{k-1}\in\Fl(\pr_x(z_{k}))\\ z_{k-2}\in\Fl(\pr_x(z_{k-1}))}}\frac{1}{|\Fl(\pr_x(z_k))|\cdot|\Fl(\pr_x(z_{k-1}))|}f(x,z_{k-2}) \\
=&\cdots=\frac{1}{|\Fl(\pr_x(y))|}\sum_{\substack{z_j,z_{j+1},...,z_{k-1},z_k:z_k\in\Fl(\pr_x(y))\\z_{k-1}\in\Fl(\pr_x(z_{k}))\\z_{k-2}\in\Fl(\pr_x(z_{k-1}))\\\cdots\cdots\\z_{j}\in\Fl(\pr_x(z_{j+1}))}}\frac{1}{\prod_{l=j}^{k-1}|\Fl(\pr_x(z_{l+1}))|}f(x,z_j).
\end{align*}
Denoted by
\begin{align}\label{a_j}
a_j=\frac{1}{|\Fl(\pr_x(y))|}\sum_{\substack{z_j,z_{j+1},...,z_{k-1},z_k:z_k\in\Fl(\pr_x(y))\\z_{k-1}\in\Fl(\pr_x(z_{k}))\\z_{k-2}\in\Fl(\pr_x(z_{k-1}))\\\cdots\cdots\\z_{j}\in\Fl(\pr_x(z_{j+1}))}}\frac{1}{\prod_{l=j}^{k-1}|\Fl(\pr_x(z_{l+1}))|}z_j.
\end{align}

Then $f(x,y)=f(x,a_j)$. Let $x\to\alpha^{(j)}_x$ be the unique map from $\Gamma x_0$ to $\RR$ such that $a_j=\sum_{x\in\Gamma x_0}\alpha^{(j)}_x\cdot x$. Then by the definition of $a$, $\alpha^{(j)}_x\geq0$ for any $x\in \Gamma x_0$. Moreover,
\begin{align}\label{supp a_j}
\begin{split}
&\supp(a_j)\\
\subset&\{z_j\in\Gamma x_0|\exists z_k,...,z_{j+1}\in\Gamma x_0\mathrm{~s.t.~} z_k\in\Fl(\pr_x(y))\mathrm{~and~}z_l\in\Fl(\pr_x(z_{l+1})),\forall~j\leq l\leq k-1\}.
\end{split}
\end{align}
By Definition \ref{flower} and Definition \ref{1-lvl lower proj}, for any $z_j,...,z_k\in\Gamma x_0$ such that $z_k\in\Fl(\pr_x(y))$ and $z_l\in\Fl(\pr_x(z_{l+1}))$ for any $j\leq l\leq k-1$, we have $F_{z_l}=F_l$ for any $j\leq l\leq k$. Therefore $\supp(a_j)\subset\{z\in\Gamma x_0|F_z=F_j\}$ and hence $a_j$ is supported near $F_j$.

Notice that
\begin{align*}
|a_j|_{l^1}
=&\frac{1}{|\Fl(\pr_x(y))|}\sum_{\substack{z_j,z_{j+1},...,z_{k-1},z_k:z_k\in\Fl(\pr_x(y))\\z_{k-1}\in\Fl(\pr_x(z_{k}))\\z_{k-2}\in\Fl(\pr_x(z_{k-1}))\\\cdots\cdots\\z_{j}\in\Fl(\pr_x(z_{j+1}))}}\frac{1}{\prod_{l=j}^{k-1}|\Fl(\pr_x(z_{l+1}))|} \\
=&\frac{1}{|\Fl(\pr_x(y))|}\sum_{\substack{z_j,z_{j+1},...,z_{k-1},z_k:z_k\in\Fl(\pr_x(y))\\z_{k-1}\in\Fl(\pr_x(z_{k}))\\z_{k-2}\in\Fl(\pr_x(z_{k-1}))\\\cdots\cdots\\z_{j+1}\in\Fl(\pr_x(z_{j+2}))}}\frac{1}{\prod_{l=j}^{k-1}|\Fl(\pr_x(z_{l+1}))|}\left(\sum_{z_j\in\Fl(\pr_x(z_{j+1}))}1\right) \\
=&\frac{1}{|\Fl(\pr_x(y))|}\sum_{\substack{z_j,z_{j+1},...,z_{k-1},z_k:z_k\in\Fl(\pr_x(y))\\z_{k-1}\in\Fl(\pr_x(z_{k}))\\z_{k-2}\in\Fl(\pr_x(z_{k-1}))\\\cdots\cdots\\z_{j+1}\in\Fl(\pr_x(z_{j+2}))}}\frac{1}{\prod_{l=j}^{k-1}|\Fl(\pr_x(z_{l+1}))|} \\
=&\cdots=\frac{1}{|\Fl(\pr_x(y))|}\sum_{\substack{z_{k-1},z_k:z_k\in\Fl(\pr_x(y))\\z_{k-1}\in \Fl(\pr_x(z_{k}))}}\frac{1}{|\Fl(\pr_x(z_{k}))|} \\
=&\frac{1}{|\Fl(\pr_x(y))|}\left(\sum_{z_k:z_k\in\Fl(\pr_x(y))}1\right)=1.
\end{align*}
This finishes the proof of $a_j$ being a convex combination supported near $F_j$ and $f(x,y)=f(x,a_j)$. In particular, if $j=1$,  by the fact that $\Theta(F_x,F_1)=\{F_x,F_1\}$ and the first case in \eqref{f}, we have $f(x,y)=f(x,a_1)=a_1$, which is a convex combination supported near $F_1=F_{pr^\infty_x(y)}$ due to the second assertion of Lemma \ref{easiest}.
\item[(2).] Notice that for any $F_1,F_2\in\Gamma F$ and any $\gamma\in\Gamma$, by the definition of $\Omega(\cdot,\cdot)$ (before Lemma \ref{properties of Omega}, which only uses distance functions on $X$) and the definition of $\Theta(\cdot,\cdot)$ (before Lemma \ref{properties of Theta}, which only uses distance functions on $X$), we have $\gamma(\Omega(F_1,F_2))=\Omega(\gamma F_1,\gamma F_2)$ and $\gamma(\Theta(F_1,F_2))=\Theta(\gamma F_1,\gamma F_2)$. Notice that for any $z\in\{q\in X|\exists \hF\in\Gamma F\mathrm{~s.t.~}d(q,\hF)\leq \epsilon_0/2\}$ and any $\gamma\in \Gamma$, we have $\gamma F_z=F_{\gamma z}$. Therefore by Notation \ref{1-lvl lower flat} and Lemma \ref{properties of Theta}, \hyperlink{Theta-3}{property ($\Theta$3) of $\Theta(\cdot,\cdot)$}, for any $x,y\in \{q\in X|\exists \hF\in\Gamma F\mathrm{~s.t.~}d(q,\hF)\leq \epsilon_0/2\}$ and any $\gamma\in \Gamma$, $\gamma F_{x,y}=F_{\gamma x,\gamma y}$. Since in Definition \ref{flower} and Definition \ref{1-lvl lower proj}, only distance functions in $X$ are involved, for any $p,x,y\in\{q\in X|\exists \hF\in\Gamma F\mathrm{~s.t.~}d(q,\hF)\leq \epsilon_0/2\}$ and any $\gamma\in\Gamma$, we have $\gamma\Fl(p)=\Fl(\gamma p)$ and $\gamma\pr_x(y)=\pr_{\gamma x}(\gamma y)$. In particular, the map $w\to\gamma w$ gives a bijection between $\Fl(\pr_x(y))$ and $\Fl(\pr_{\gamma x}(\gamma y))$. According to \eqref{f}, finally we have $\gamma f(x,y)=f(\gamma x,\gamma y)$ for any $x,y\in\Gamma x_0$ and any $\gamma \in\Gamma$.

\item[(3).] If $|\Theta(F_x,F_y)|=2$, this is equivalent to $\Theta(F_x,F_y)=\{F_x,F_y\}$. By the first assertion in Lemma \ref{easiest} and \eqref{f}, $f(x,y)=y=\pr_x^\infty(y)$ and the assertion clearly follows.

If $|\Theta(F_x,F_y)|=3$, by the second assertion in Lemma \ref{easiest} and \eqref{f}, we have $\pr_x(y)=\pr_x^\infty(y)$ and
$$f(x,y)= \frac{1}{|\Fl(\pr_x(y))|}\sum_{z\in\Fl(\pr_x(y))}f(x,z)= \frac{1}{|\Fl(\pr_x(y))|}\sum_{z\in\Fl(\pr_x(y))}z.$$
Therefore $\supp(f(x,y))= \Fl(\pr_x(y))=\Fl(\pr_x^\infty(y))$. The rest of this case follows directly from Definition \ref{flower}.

From now on, we assume that $|\Theta(F_x,F_y)|\geq 4$.

Similar to the proof of the first assertion in Proposition \ref{prop of f}, we let $\Theta(F_x,F_y)=\{F_x,F_1,...,F_k,F_y\}$ for some $k\geq 2$ and distinct $F_1,...,F_k \in\Gamma F\setminus\{F_x,F_y\}$ such that $
\Theta(F_x,F_j)=\{F_x,F_1,...,F_j\}$. By \eqref{a_j} and \eqref{supp a_j}, we have $f(x,y)=f(x,a_j)$ with $a_j$ supported near $F_j$ for any $1\leq j\leq k$. Moreover, by \eqref{supp a_j} we have
\begin{align}\label{supp induction}
\supp(a_j)\subset\bigcup_{z\in\supp(a_{j+1})}\Fl(\pr_x(z)),\quad\forall~1\leq j\leq k-1.
\end{align}
By Definition \ref{1-lvl lower proj}, we have $F_{\pr_x^{k+1-j}(y)}=F_j$ for any $1\leq j\leq k$. In particular, $F_1\neq F_2$ implies that $\pr_x^{k-1}(y)\neq \pr_x^k(y)$. By the first assertion in Lemma \ref{easiest}, $d(\pr_x^k(y),F_1)=\epsilon_0/2$. The same argument shows that $d(\pr_x(z),F_1)=\epsilon_0/2$ for any $z\in\supp(a_2)$. Therefore for any $z\in\supp(a_{2})$, the first assertion in Lemma \ref{easiest} and the convexity of distance function in nonpositive curvature imply (See Figure \ref{prop5.11}.)
$$\left\{\begin{aligned}
&d(\pr_x^{k-1}(y),F_{2}),d(z, F_2)\leq \epsilon_0/2\implies d([\pr_x^{k-1}(y),z], F_1)\geq \rho_0-\epsilon_0/2>2\rho_0/3,\\
&d([\pr_x^{k-1}(y),\pr_x^k(y)],F_1)=d(\pr_x^k(y),F_1)=d([\pr_x(z),z],F_1)=d(\pr_x(z),F_1)= \epsilon_0/2.
\end{aligned}\right.$$
\begin{figure}[h]
	\centering
	\includegraphics[width=4in]{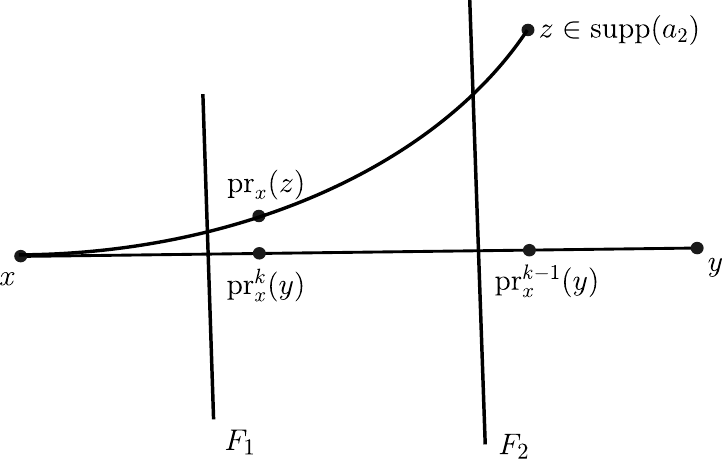}
	\caption{ \label{prop5.11}}
\end{figure}
Applying Lemma \ref{prep1} to the above yields
$$d(\Proj_{F_1}(\pr_x(z)),\Proj_{F_1}(\pr_x^k(y)))\leq \cC_2(\rho_0,\epsilon_0,0).$$
Hence by Definition \ref{flower} and the second assertion in Lemma \ref{easiest}, we have $\pr_x^k(y)=\pr_x^\infty(y)$ and
$$\Fl(z)\subset\left\{w\in\Gamma x_0\left|\begin{aligned}
&F_w=F_{\pr^\infty_x(y)}, \\
&d(\Proj_{F_{\pr^\infty_x(y)}}(w),\Proj_{F_{\pr^\infty_x(y)}}(\pr^\infty_x(y)))<2\cC_2(\rho_0,\epsilon_0,0)+\mathrm{diam}(N)+2
\end{aligned}\right.\right\}.$$
By arbitrariness of $z\in\supp(a_2)$ and \eqref{supp induction}, we have
$$\supp(a_1)\subset\left\{w\in\Gamma x_0\left|\begin{aligned}
&F_w=F_{\pr^\infty_x(y)}, \\
&d(\Proj_{F_{\pr^\infty_x(y)}}(w),\Proj_{F_{\pr^\infty_x(y)}}(\pr^\infty_x(y)))<2\cC_2(\rho_0,\epsilon_0,0)+\mathrm{diam}(N)+2
\end{aligned}\right.\right\}.$$
Notice that $a_1$ is supported near $F_1$, by the first case in \eqref{f}, we have $f(x,y)=f(x,a_1)=a_1$ and $\supp(f(x,y))=\supp(a_1)$. This finishes the proof.

\item[(4).] We choose $\cC_3= 2\lambda^{-2}$. Then for any $x,y,z\in\Gamma x_0$
$$|f(x,y)-f(x,z)|_{l^1}\leq |f(x,y)|_{l^1}+|f(x,z)|_{l^1}\leq 1+1=2.$$
In particular, the assertion holds when $(F_y|F_z)_{F_x}\leq 2$. From now on, we assume that $(F_y|F_z)_{F_x}\geq 3$. We further choose
$$\lambda=\lambda(\rho_0,\epsilon_0)=1-\left(\max_{p\in X}\left|\Gamma x_0\ints B_{\cC_2(\rho_0,\epsilon_0,1)+2+\epsilon_0+\mathrm{diam}(N)}(p)\right|\right)^{-2}.$$
Then for any $p\in X$ such that $d(p,\hF)\leq\epsilon_0/2$ for some $\hF\in\Gamma F$, we have
\begin{align}\label{exp decay}
|\Fl(p)|\leq \max_{p\in X}\left|\Gamma x_0\ints B_{\cC_2(\rho_0,\epsilon_0,1)+2+\epsilon_0+\mathrm{diam}(N)}(p)\right|=\left(\frac{1}{1-\lambda}\right)^{1/2}.
\end{align}

We first prove this assertion with the above choices of $\cC_2$ and $\lambda$ for a special case:

\textbf{Special case}: Assume $F_y=F_z$. We prove the assertion via induction on $(F_y|F_z)_{F_x}=|\Theta(F_x,F_y)|$. Since the assertion is true whenever $|\Theta(F_x,F_y)|\leq 2$, suppose that the assertion holds when $|\Theta(F_x,F_y)|=k+1$ for some $k\geq 1$, we consider the case when $|\Theta(F_x,F_y)|=k+2$. Similar to the proof of the first assertion in Proposition \ref{prop of f}, we let $\Theta(F_x,F_y)=\{F_x,F_1,...,F_k,F_y\}$ with $k\geq 1$ and distinct $F_1,...,F_k\in\Gamma F\setminus \{F_x,F_y\}$ such that $
\Theta(F_x,F_j)=\{F_x,F_1,...,F_j\}$. Notice that $F_y\neq F_k$ implies that $\pr_x(y)\neq y$ and $\pr_x(z)\neq z$. By Definition \ref{1-lvl lower proj} and the first assertion in Lemma \ref{easiest}, we have $d(\pr_x(y),F_k)=d(\pr_x(z),F_k)=\epsilon_0/2$. Hence
$$\left\{
\begin{aligned}
& d(y,F_y),d(z,F_y)\leq\epsilon_0/2\implies d([y,z], F_k)\geq \rho_0-\epsilon_0/2>2\rho_0/3\\
& d(\pr_x(y),F_k)=d([\pr_x(y),y],F_k)=d(\pr_x(z),F_k)=d([\pr_x(z),z],F_k)
=\epsilon_0/2
\end{aligned}
\right.$$
By Lemma \ref{prep1}, we have
$$d(\Proj_{F_k}(\pr_x(y)),\Proj_{F_k}(\pr_x(z)))\leq \cC_2(\rho_0,\epsilon_0,0).$$
Together with Definition \ref{flower}, we have
\begin{align}\label{local cancel}
\Fl(\pr_x(y))\ints\Fl(\pr_x(z))\neq\emptyset
\end{align}
For any $\cP\subset\Gamma x_0$, we define the \emph{characteristic function of $\cP$} by
\begin{align}\label{char fcn}
\chi_\cP(x)=\begin{cases}
\displaystyle  1,\quad&\mathrm{if~}x\in\cP,\\
\displaystyle  0,\quad&\mathrm{if~}x\not\in\cP.
\end{cases}
\end{align}
Recall that in the second case of \eqref{f} we have $f(x,y)=f(x,a_k)$ and $f(x,z)=f(x,b_k)$, where
$$a_k=\frac{1}{|\Fl(\pr_x(y))|}\sum_{w\in\Fl(\pr_x(y))}w\mathrm{~and~}b_k=\frac{1}{|\Fl(\pr_x(z))|}\sum_{w\in\Fl(\pr_x(z))}w.$$
Therefore we can write $a_k-b_k=\sum_{p\in\Gamma x_0}\alpha^{(k)}_{x;y,z}(p)p$, where
\begin{align}\label{coeff of difference-1}
\alpha_{x;y,z}^{(k)}(p)=\frac{\chi_{\Fl(\pr_x(y))}(p)}{|\Fl(\pr_x(y))|}-\frac{\chi_{\Fl(\pr_x(z))}(p)}{|\Fl(\pr_x(z))|}.
\end{align}
In particular,
\begin{align}\label{ch4 cycle-1}
\sum_{p\in\Gamma x_0}\alpha_{x;y,z}^{(k)}(p)=0.
\end{align}
Moreover, by \eqref{exp decay} and \eqref{local cancel}, we have
\begin{align}\label{shrinking norm}
\begin{split}
&|a_k-b_k|_{l^1}
=\sum_{p\in\Gamma x_0}|\alpha_{x;y,z}^{(k)}(p)|\\
=&\sum_{\substack{p:p\in\Gamma x_0\\ \chi_{\Fl(\pr_x(y))}(p)\cdot\chi_{\Fl(\pr_x(z))}(p)=0}}|\alpha_{x;y,z}^{(k)}(p)|+\sum_{\substack{p\in\Gamma x_0\\ \chi_{\Fl(\pr_x(y))}(p)\cdot\chi_{\Fl(\pr_x(z))}(p)=1}}|\alpha_{x;y,z}^{(k)}(p)|\\
=&\sum_{\substack{p:p\in\Gamma x_0\\ \chi_{\Fl(\pr_x(y))}(p)\cdot\chi_{\Fl(\pr_x(z))}(p)=0}}\frac{\chi_{\Fl(\pr_x(y))}(p)}{|\Fl(\pr_x(y))|}+\frac{\chi_{\Fl(\pr_x(z))}(p)}{|\Fl(\pr_x(z))|} \\
&+\sum_{\substack{p:p\in\Gamma x_0\\ \chi_{\Fl(\pr_x(y))}(p)\cdot\chi_{\Fl(\pr_x(z))}(p)=1}}\left|\frac{\chi_{\Fl(\pr_x(y))}(p)}{|\Fl(\pr_x(y))|}-\frac{\chi_{\Fl(\pr_x(z))}(p)}{|\Fl(\pr_x(z))|}\right|\\
\leq&\sum_{\substack{p:p\in\Gamma x_0\\ \chi_{\Fl(\pr_x(y))}(p)\cdot\chi_{\Fl(\pr_x(z))}(p)=0}}\frac{\chi_{\Fl(\pr_x(y))}(p)}{|\Fl(\pr_x(y))|}+\frac{\chi_{\Fl(\pr_x(z))}(p)}{|\Fl(\pr_x(z))|} \\
&+\sum_{\substack{p:p\in\Gamma x_0\\ \chi_{\Fl(\pr_x(y))}(p)\cdot\chi_{\Fl(\pr_x(z))}(p)=1}}\left(\frac{\chi_{\Fl(\pr_x(y))}(p)}{|\Fl(\pr_x(y))|}+\frac{\chi_{\Fl(\pr_x(z))}(p)}{|\Fl(\pr_x(z))|}-\frac{2}{|\Fl(\pr_x(y))|\cdot|\Fl(\pr_x(z))|}\right)\\
\leq &2-\frac{2}{|\Fl(\pr_x(y))|\cdot|\Fl(\pr_x(z))|}\leq 2-\frac{2}{1/(1-\lambda)}=2\lambda.
\end{split}
\end{align}
Let $\alpha^{(k),+}_{x;y,z}(p)=\max\{\alpha^{(k)}_{x;y,z}(p),0\}\geq 0$ and $\alpha^{(k),-}_{x;y,z}(p)=\max\{-\alpha^{(k)}_{x;y,z}(p),0\}\geq 0$ for any $p\in\Gamma x_0$. Then we have
$$\alpha^{(k),+}_{x;y,z}(p)\cdot\alpha^{(k),-}_{x;y,z}(p)=0\mathrm{~and~}\alpha^{(k),+}_{x;y,z}(p)-\alpha^{(k),-}_{x;y,z}(p)=\alpha^{(k)}_{x;y,z}(p),~\forall~p\in\Gamma x_0.$$
Define
$$a^{(k),+}=\sum_{p\in\Gamma x_0}\alpha^{(k),+}_{x;y,z}(p)p\mathrm{~and~}a^{(k),-}=\sum_{p\in\Gamma x_0}\alpha^{(k),-}_{x;y,z}(p)p.$$
Then $a_k-b_k=a^{(k),+}-a^{(k),-}$. Moreover, by \eqref{ch4 cycle-1} and \eqref{shrinking norm}, we have
\begin{align}\label{shrinking half norm}
\alpha^{(k)}:=\frac{1}{2}|a_k-b_k|_{l^1}=\sum_{p\in\Gamma x_0}\alpha^{(k),+}_{x;y,z}(p)=|a^{(k),+}|_{l^1}=\sum_{p\in\Gamma x_0}\alpha^{(k),-}_{x;y,z}(p)=|a^{(k),-}|_{l^1}\leq \lambda.
\end{align}
Notice that by \eqref{coeff of difference-1}, $\alpha^{(k)}_{x;y,z}(p)=0$ whenever $F_p\neq F_{\pr_x(y)}=F_{\pr_x(z)}=F_k$. Therefore $\alpha^{(k),\pm}_{x;y,z}(p)=0$ whenever $F_p\neq F_{\pr_x(y)}=F_{\pr_x(z)}=F_k$. By \eqref{shrinking half norm} and the inductive step (i.e. the case when when $|\Theta(F_x,F_p)|=|\Theta(F_x,F_q)|=k+1$), we have
\begin{align*}
|f(x,y)-f(x,z)|_{l^1}=&|f(x,a_k-b_k)|_{l^1}
=|f(x,a^{(k),+}-a^{(k),-})|_{l^1}\\
=&\left|\sum_{p\in\Gamma x_0}\alpha^{(k),+}_{x;y,z}(p)f(x,p)-\sum_{q\in\Gamma x_0}\alpha^{(k),-}_{x;y,z}(q)f(x,q)\right|_{l^1} \\
=&\left|\sum_{p,q:p,q\in\Gamma x_0}\frac{\alpha^{(k),+}_{x;y,z}(p)\alpha^{(k),-}_{x;y,z}(q)}{\alpha^{(k)}}(f(x,p)-f(x,q))\right|_{l^1}\\
=&\left|\sum_{\substack{p,q:p,q\in\Gamma x_0\\F_p=F_q=F_k}}\frac{\alpha^{(k),+}_{x;y,z}(p)\alpha^{(k),-}_{x;y,z}(q)}{\alpha^{(k)}}(f(x,p)-f(x,q))\right|_{l^1} \\
\leq&\sum_{\substack{p,q:p,q\in\Gamma x_0\\F_p=F_q=F_k}}\frac{\alpha^{(k),+}_{x;y,z}(p)\alpha^{(k),-}_{x;y,z}(q)}{\alpha^{(k)}}\left|f(x,p)-f(x,q)\right|_{l^1} \\
\leq&\sum_{\substack{p,q:p,q\in\Gamma x_0\\F_p=F_q=F_k}}\frac{\alpha^{(k),+}_{x;y,z}(p)\alpha^{(k),-}_{x;y,z}(q)}{\alpha^{(k)}}\cC_3\lambda^{k+1}=\alpha^{(k)}\cC_3\lambda^{k+1}\leq\cC_3\lambda^{k+2}.
\end{align*}
This finishes the proof of this assertion in the special case of $F_y=F_z$.

\textbf{General case}: Notice that for any $F'\in\Theta(F_x,F_y)\ints\Theta(F_x,F_z)$, by Lemma \ref{properties of Theta}, \hyperlink{Theta-3}{property ($\Theta$3) of $\Theta(\cdot,\cdot)$}, $\Theta(F_x,F')\subset\Theta(F_x,F_y)\ints\Theta(F_x,F_z)$. Therefore, by Lemma \ref{properties of Theta}, \hyperlink{Theta-3}{property ($\Theta$3) of $\Theta(\cdot,\cdot)$} again, there exist some $\hF\in\Theta(F_x,F_y)\ints\Theta(F_x,F_z)$ such that $\Theta(F_x,\hF)=\Theta(F_x,F_y)\ints\Theta(F_x,F_z)$. By the first assertion in Proposition \ref{prop of f}, there exist convex conbinations $a,b$ supported near $\hF$ such that $f(x,y)=f(x,a)$ and $f(x,z)=f(x,b)$. Write $a=\sum_{p:p\in\Gamma x_0, F_p=\hF}\alpha_a(p)p$ and $b=\sum_{q:q\in\Gamma x_0, F_q=\hF}\alpha_b(q)q$. Then $\alpha_a(\cdot),\alpha_b(\cdot)\geq 0$ and
$$1=|a|_{l^1}=\sum_{p:p\in\Gamma x_0, F_p=\hF}\alpha_a(p)=|b|_{l^1}=\sum_{q:q\in\Gamma x_0, F_q=\hF}\alpha_b(q).$$
By the \textbf{Special case}, we have
\begin{align*}
|f(x,y)-f(x,z)|_{l^1}=&|f(x,a)-f(x,b)|_{l^1}\\
=&\left|\sum_{\substack{p,q:p,q\in\Gamma x_0\\F_p=F_q=\hF}}\alpha_a(p)\alpha_b(q)(f(x,p)-f(x,q))\right|_{l^1} \\
\leq&\sum_{\substack{p,q:p,q\in\Gamma x_0\\F_p=F_q=\hF}}\alpha_a(p)\alpha_b(q)\left|f(x,p)-f(x,q)\right|_{l^1} \\
\leq&\sum_{\substack{p,q:p,q\in\Gamma x_0\\F_p=F_q=\hF}}\alpha_a(p)\alpha_b(q)\cC_3\lambda^{(F_p|F_q)_{F_x}}=\cC_3\lambda^{(\hF|\hF)_{F_x}}=\cC_3\lambda^{(F_y|F_z)_{F_x}}.
\end{align*}
This completes the proof of this assertion.
\item[(5).] If $|\Theta(F_x,F_y)|\leq 2$, then $f(y,x)=f(z,x)=x$ and the assertion follows trivially. From now on, we assume that $|\Theta(F_x,F_y)|\geq 3$. By Lemma \ref{properties of Theta}, \hyperlink{Theta-3}{property ($\Theta$3) of $\Theta(\cdot,\cdot)$}, we can choose $F_1\neq F_2\in\Theta(F_x,F_y)\setminus{F_y}$ such that $\Theta(F_y,F_1)=\{F_y,F_1\}$ and $\Theta(F_y,F_2)=\{F_y,F_1,F_2\}$. By the first assertion of Proposition \ref{prop of f}, there exist convex combinations $a,b$ supported near $F_2$ such that $f(y,x)=f(y,a)$ and $f(z,x)=f(z,b)$.

Notice that for any $p,q\in X$ such that $F_p=F_q=F_2$, we have $F_{\pr_y(p)}=F_{\pr_z(q)}=F_1\neq F_2$. In particular $p\neq\pr_y(p)$ and $q\neq\pr_z(q)$. By the first assertion of Lemma \ref{easiest},
$$\left\{
\begin{aligned}
&d(p,F_2),d(q,F_2)\leq\epsilon_0/2\implies d([p,q],F_1)\geq \rho_0-\epsilon_0/2>2\rho_0/3, \\
&d(\pr_y(p),F_1)=d([\pr_y(p),p],F_1)=d(\pr_z(q),F_1)=d([\pr_z(q),q],F_1)
=\epsilon_0/2.
\end{aligned}
\right.
$$
By Lemma \ref{prep1}, we have $d(\Proj_{F_1}(\pr_y(p)),\Proj_{F_1}(\pr_z(q)))\leq \cC_2(\rho_0,\epsilon_0,0).$ As a direct corollary of Definition \ref{flower}, we have $\Fl(\pr_y(p))\ints\Fl(\pr_z(q))\neq\emptyset.$ Recall that in \eqref{f} we have $f(y,p)=f(y,a_p)=a_p$ and $f(z,q)=f(z,b_q)=b_q$, where $a_p=\frac{1}{|\Fl(\pr_y(p))|}\sum_{w\in\Fl(\pr_y(p))}w$ and $b_q=\frac{1}{|\Fl(\pr_z(q))|}\sum_{w\in\Fl(\pr_z(q))}w.$
Therefore we can write $a_p-b_q=\sum_{w\in\Gamma x_0}\alpha_{y,z;p,q}(w)w$, where
\begin{align}\label{coeff of difference-2}
\alpha_{y,z;p,q}(w)=\frac{\chi_{\Fl(\pr_y(p))}(w)}{|\Fl(\pr_y(p))|}-\frac{\chi_{\Fl(\pr_z(q))}(w)}{|\Fl(\pr_z(q))|},\quad\forall w\in\Gamma x_0.
\end{align}
In particular, $\sum_{w\in\Gamma x_0}\alpha_{y,z;p,q}(w)=0$. Similar to \eqref{shrinking norm}, \eqref{exp decay} and the fact that $\Fl(\pr_y(p))\ints\Fl(\pr_z(q))\neq\emptyset$ imply
\begin{align}\label{shrinking norm-2}
&|f(y,p)-f(z,q)|_{l^1}=|a_p-b_q|_{l^1}
=\sum_{w\in\Gamma x_0}|\alpha_{y,z;p,q}(w)|\nonumber\\
=&\sum_{\substack{w:w\in\Gamma x_0\\ \chi_{\Fl(\pr_y(p))}(w)\cdot\chi_{\Fl(\pr_z(q))}(w)=0}}|\alpha_{y,z;p,q}(w)|+\sum_{\substack{w:w\in\Gamma x_0\\ \chi_{\Fl(\pr_y(q))}(w)\cdot\chi_{\Fl(\pr_z(q))}(w)=1}}|\alpha_{y,z;p,q}(w)|\nonumber\\
=&\sum_{\substack{w:w\in\Gamma x_0\\ \chi_{\Fl(\pr_y(p))}(w)\cdot\chi_{\Fl(\pr_z(q))}(w)=0}}\frac{\chi_{\Fl(\pr_y(p))}(w)}{|\Fl(\pr_y(p))|}+\frac{\chi_{\Fl(\pr_z(q))}(w)}{|\Fl(\pr_z(q))|} \nonumber\\
&+\sum_{\substack{w:w\in\Gamma x_0\\ \chi_{\Fl(\pr_y(p))}(w)\cdot\chi_{\Fl(\pr_z(q))}(w)=1}}\left|\frac{\chi_{\Fl(\pr_y(p))}(w)}{|\Fl(\pr_y(p))|}-\frac{\chi_{\Fl(\pr_z(q))}(w)}{|\Fl(\pr_z(q))|}\right|\nonumber\\
\leq&\sum_{\substack{w:w\in\Gamma x_0\\ \chi_{\Fl(\pr_y(p))}(w)\cdot\chi_{\Fl(\pr_z(q))}(w)=0}}\frac{\chi_{\Fl(\pr_y(p))}(w)}{|\Fl(\pr_y(p))|}+\frac{\chi_{\Fl(\pr_z(q))}(w)}{|\Fl(\pr_z(q))|} \nonumber\\
&+\sum_{\substack{w:w\in\Gamma x_0\\ \chi_{\Fl(\pr_y(p))}(w)\cdot\chi_{\Fl(\pr_z(q))}(w)=1}}\left(\frac{\chi_{\Fl(\pr_y(p))}(w)}{|\Fl(\pr_y(p))|}+\frac{\chi_{\Fl(\pr_z(q))}(w)}{|\Fl(\pr_z(q))|}-\frac{2}{|\Fl(\pr_y(p))|\cdot|\Fl(\pr_z(q))|}\right)\nonumber\\
\leq &2-\frac{2}{|\Fl(\pr_y(p))|\cdot|\Fl(\pr_z(q))|}\leq 2-\frac{2}{1/(1-\lambda)}=2\lambda.
\end{align}
Back to the discussion on $f(y,x)=f(y,a)$ and $f(z,x)=f(z,b)$, write
$$a=\sum_{w:w\in\Gamma x_0, F_w=F_2}\alpha_a(w)w\mathrm{~and~}b=\sum_{v:v\in\Gamma x_0, F_v=F_2}\alpha_b(v)v,$$
where $\alpha_a,\alpha_b\geq0$ and
$$\sum_{w:w\in\Gamma x_0, F_w=F_2}\alpha_a(w)=\sum_{v:v\in\Gamma x_0, F_v=F_2}\alpha_b(v)=1.$$
Then by \eqref{shrinking norm-2}, we have
\begin{align*}
|f(y,x)-f(z,x)|_{l^1}=&|f(y,a)-f(z,b)|_{l^1}\\
=&\left|\sum_{\substack{v,w:v,w\in\Gamma x_0\\F_v=F_w=F_2}}\alpha_a(w)\alpha_b(v)(f(y,w)-f(z,v))\right|_{l^1} \\
\leq&\sum_{\substack{v,w:v,w\in\Gamma x_0\\F_v=F_w=F_2}}\alpha_a(w)\alpha_b(v)\left|f(y,w)-f(z,v)\right|_{l^1} \\
\leq&2\lambda\sum_{\substack{v,w:v,w\in\Gamma x_0\\F_v=F_w=F_2}}\alpha_a(w)\alpha_b(v)=2\lambda.
\end{align*}
This finishes the proof.\qedhere
\end{enumerate}
\end{proof}

Now we construct the desired ``homological bicombing'' without the anti-symmetry assumptions. Define the $\RR$-bilinear map $\beta':C_{0,x_0}(X;\RR)\times C_{0,x_0}(X;\RR)\to C_{1,x_0}(X;\RR)$ inductively as follows: (See Definition \ref{dfn:geo.bicomb} for $\ovec{[\cdot,\cdot]}$.)
\begin{align}\label{beta'}
\beta'[x,y]=
\begin{cases}
\displaystyle \ovec{[x,y]},\quad&\mathrm{if~}|\Theta(F_x,F_y)|\leq 2,\\
\displaystyle \beta'[x,f(y,x)]+\ovec{[f(y,x),y]}, \quad&\mathrm{else},
\end{cases}
\quad\forall x,y\in\Gamma x_0=\cS_{0,x_0}(X),
\end{align}
where $\ovec{[\sum_{p\in\Gamma}\alpha_pp,x]}:=\sum_{p\in\Gamma}\alpha_p\ovec{[p,x]}$ for any $\sum_{p\in\Gamma x_0}\alpha_pp\in C_{0,x_0}(X;\RR)$.

When $|\Theta(F_x,F_y)|\leq 2$, we define $\cI'_{F_x, x}[x,y]=x$, $\cI'_{F_y, y}[x,y]=y$ and $\beta'_{F_xF_y}[x,y]=\beta'[x,y]$. (These definitions can be viewed as a trivial case for later definitions when $|\Theta(F_x,F_y)|\geq 3$.) If $F_x\neq F_y$, we write $\cI'_{F_x}[x,y]=\cI'_{F_x, x}[x,y]$ and $\cI'_{F_y}[x,y]=\cI'_{F_y, y}[x,y]$ for simplicity. In particular, by the first case in \eqref{beta'}, we have
\begin{align}\label{easy del beta' seg}
\del^X\beta'_{F_xF_y}[x,y]=\del^X \beta'[x,y]=y-x=\cI'_{F_y, y}[x,y]-\cI'_{F_x, x}[x,y].
\end{align}

When $|\Theta(F_x,F_y)|\geq 3$, we let $\Theta(F_x,F_y)=\{F_x,F_1,...,F_k,F_y\}$ for some $k\geq 1$ and distinct $F_1,...,F_k\in\Gamma F\setminus \{F_x,F_y\}$ such that $
\Theta(F_{x},F_j)=\{F_{x},F_1,...,F_j\}$. Define $\cI'_{F_x,x}[x,y]:=\cI'_{F_x}[x,y]=x$, $\cI'_{F_y,y}[x,y]:=\cI'_{F_y}[x,y]=y$ and
\begin{align}\label{beta' seg endpts}
\cI'_{F_k}[x,y]=f(y,x)=f(\cI'_{F_{y}}[x,y],x),~\cI'_{F_j}[x,y]=f(\cI'_{F_{j+1}}[x,y],x),~\forall~1\leq j\leq k-1.
\end{align}
By the first assertion in Proposition \ref{prop of f} and the second assertion in Lemma \ref{easiest}, one can inductively prove that $\cI'_{\hF}[x,y]\in C_{0,x_0}(X;\RR)$ are convex combinations supported near $\hF$. In particular,
\begin{align}\label{beta' seg endpts norm}
|\cI'_{\hF}[x,y]|_{l^1}=1, ~\forall \hF\in\Theta(F_x,F_y).
\end{align}
Let $\cI'_{F_j}[x,y]=\sum_{z:z\in\Gamma x_0, F_z=F_j}\alpha_z^{(j)}z$ for some $\alpha_z^{(j)}\geq 0$ such that $\sum_{z:z\in\Gamma x_0, F_z=F_j}\alpha_z^{(j)}=1$. Then we define
\begin{align}\label{beta' seg}
\begin{split}
&\beta'_{F_jF_{j+1}}[x,y]=\sum_{z:z\in \Gamma x_0,F_z=F_{j+1}}\alpha_z^{(j+1)}\ovec{[f(z,x),z]},~\forall~1\leq j\leq k-1, \\
&\beta'_{F_kF_y}[x,y]=\ovec{[f(y,x),y]},~\beta'_{F_xF_1}[x,y]=\ovec{[x, \cI'_{F_1}[x,y]]}.
\end{split}
\end{align}
Then by \eqref{beta' seg endpts} and \eqref{beta' seg}, we have
\begin{align}\label{del beta' seg}
\begin{split}
\del^X\beta'_{F_jF_{j+1}}[x,y]=&\sum_{z:z\in \Gamma x_0,F_z=F_{j+1}}\alpha_z^{(j+1)}\del^X\ovec{[f(z,x),z]}\\
=&\sum_{z:z\in \Gamma x_0,F_z=F_{j+1}}\alpha_z^{(j+1)}(z-f(z,x)) \\
=&\cI'_{F_{j+1}}[x,y]-f(\cI'_{F_{j+1}}[x,y],x)=\cI'_{F_{j+1}}[x,y]-\cI'_{F_{j}}[x,y],~\forall 1\leq j\leq k-1,\\
\del^X\beta'_{F_kF_y}[x,y]=&\del^X\ovec{[f(y,x),y]}=y-f(y,x)=\cI'_{F_{y}}[x,y]-\cI'_{F_{k}}[x,y], \\
\del^X\beta'_{F_xF_1}[x,y]=&\del^X\ovec{[x, \cI'_{F_1}[x,y]]}=\cI'_{F_1}[x,y]-x=\cI'_{F_{1}}[x,y]-\cI'_{F_{x}}[x,y].
\end{split}
\end{align}
Since for any $1\leq j\leq k-1$ and any $z\in\Gamma x_0$ such that $F_z=F_{j+1}$,  using notations from Definition \ref{l1-seminorm} and the remark after Definition \ref{dfn:geo.bicomb}, the first assertion in Proposition \ref{prop of f} implies that
\begin{align*}
|f(z,x)|_{l^1}=|\ovec{[f(z,x),z]}|_{l^1}=|\ovec{[f(z,x),z]}|_{\{F_j,F_{j+1}\},l^1}=1.
\end{align*}
Therefore, {by \eqref{beta' seg}, the notations from Definition \ref{l1-seminorm} and the remark after Definition \ref{dfn:geo.bicomb}, we have}
\begin{align}\label{l1 norms of beta' seg}
\begin{split}
|\beta'_{F_jF_{j+1}}[x,y]|_{l^1}=&|\beta'_{F_jF_{j+1}}[x,y]|_{\{F_j,F_{j+1}\},l^1} \\
=&\left|\sum_{z:z\in \Gamma x_0,F_z=F_{j+1}}\alpha_z^{(j+1)}\ovec{[f(z,x),z]}\right|_{\{F_j,F_{j+1}\},l^1}\\
=&\sum_{z:z\in \Gamma x_0,F_z=F_{j+1}}\alpha_z^{(j+1)}\left|\ovec{[f(z,x),z]}\right|_{\{F_j,F_{j+1}\},l^1}=\sum_{z:z\in \Gamma x_0,F_z=F_{j+1}}\alpha_z^{(j+1)}=1,\\
|\beta'_{F_kF_y}[x,y]|_{l^1}=&|\beta'_{F_kF_y}[x,y]|_{\{F_k,F_y\}, l^1}=|\ovec{[f(y,x),y]}|_{\{F_k,F_y\}, l^1}=1,\\
|\beta'_{F_xF_1}[x,y]|_{l^1}=&|\beta'_{F_xF_1}[x,y]|_{\{F_x,F_1\}, l^1}=|\ovec{[x,\cI'_{F_1}[x,y]]}|_{\{F_x,F_1\}, l^1}=1.
\end{split}
\end{align}

Following the definitions in \eqref{beta'}, \eqref{beta' seg endpts} and \eqref{beta' seg}, for any $1\leq j\leq k$, we obtain
\begin{align}
\beta'[x,y]=&\beta'[x,f(y,x)]+\ovec{[f(y,x),y]} \nonumber\\
=&\beta'[x,\cI'_{F_{k}}[x,y]]+\beta'_{F_kF_y}[x,y] \nonumber\\
=&\beta'\left[x,\sum_{z:z\in\Gamma x_0, F_z=F_k}\alpha_z^{(k)}z\right]+\beta'_{F_kF_y}[x,y] \nonumber\\
=& \beta'\left[x,f\left(\sum_{z:z\in\Gamma x_0, F_z=F_k}\alpha_z^{(k)}z,x\right)\right]+\left(\sum_{z:z\in\Gamma x_0, F_z=F_k}\alpha_z^{(k)}\ovec{[f(z,x),z]}\right)+\beta'_{F_kF_y}[x,y]    \nonumber\\
=&\beta'[x,f(\cI'_{F_{k}}[x,y],x)]+\left(\sum_{z:z\in\Gamma x_0, F_z=F_k}\alpha_z^{(k)}\ovec{[f(z,x),z]}\right)+\beta'_{F_kF_y}[x,y] \nonumber\\
=&\beta'[x,\cI'_{F_{k-1}}[x,y]]+\beta'_{F_{k-1}F_{k}}[x,y]+\beta'_{F_kF_y}[x,y] \nonumber\\
=&\cdots=\beta'[x,\cI'_{F_{j}}[x,y]]+\beta'_{F_{j}F_{j+1}}[x,y]+\cdots+\beta'_{F_{k-1}F_{k}}[x,y]+\beta'_{F_kF_y}[x,y] \label{beta' detailed partial}\\
=&\cdots=\beta'[x,\cI'_{F_{1}}[x,y]]+\beta'_{F_{1}F_{2}}[x,y]+\cdots+\beta'_{F_{k-1}F_{k}}[x,y]+\beta'_{F_kF_y}[x,y] \nonumber\\
=&\ovec{[x,\cI'_{F_{1}}[x,y]]}+\beta'_{F_{1}F_{2}}[x,y]+\cdots+\beta'_{F_{k-1}F_{k}}[x,y]+\beta'_{F_kF_y}[x,y] \nonumber\\
=&\beta'_{F_{x}F_{1}}[x,y]+\beta'_{F_{1}F_{2}}[x,y]+\cdots+\beta'_{F_{k-1}F_{k}}[x,y]+\beta'_{F_kF_y}[x,y]. \label{beta' detailed full}
\end{align}
(See Figure \ref{eqn5.18} for a rough graph of $\beta'$ and its decompostion in \eqref{beta' detailed full}.) In particular, by \eqref{del beta' seg} and \eqref{beta' detailed full},
\begin{align}\label{is bicombing}
\begin{split}
&{\del^X}\beta'[x,y]\\
=&{\del^X}\beta'_{F_{x}F_{1}}[x,y]+{\del^X}\beta'_{F_{1}F_{2}}[x,y]+\cdots+{\del^X}\beta'_{F_{k-1}F_{k}}[x,y]+{\del^X}\beta'_{F_kF_y}[x,y] \\
=&\cI'_{F_{1}}[x,y]-\cI'_{F_{x}}[x,y]+\cI'_{F_{2}}[x,y]-\cI'_{F_{1}}[x,y]+\cdots+\cI'_{F_{k}}[x,y]-\cI'_{F_{k-1}}[x,y]+\cI'_{F_{y}}[x,y]-\cI'_{F_{k}}[x,y]\\
=&\cI'_{F_{y}}[x,y]-\cI'_{F_{x}}[x,y]=y-x.
\end{split}
\end{align}
\begin{figure}[h]
	\centering
	\includegraphics[width=4in]{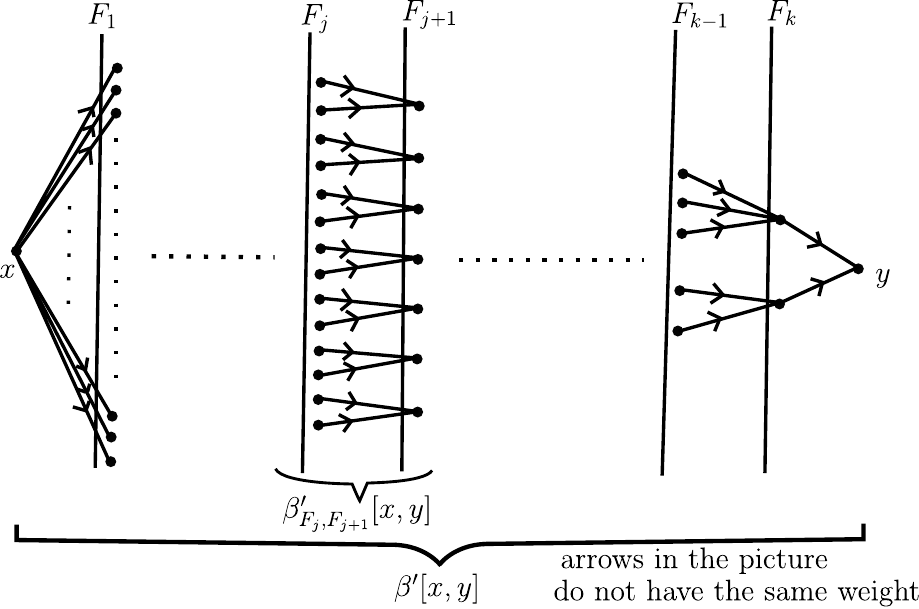}
	\caption{ \label{eqn5.18}}
\end{figure}
For any $x,y,z\in\Gamma x_0$, for simplicity we write
\begin{align}\label{I am lazy}
\cF_x(y,z):=\Theta(F_x,F_y)\ints\Theta(F_x,F_z).
\end{align}
Then we have
\begin{proposition}\label{prop of beta'}
For any $x,y,z\in\Gamma x_0$, there exists some $\cC_4(\rho_0,\epsilon_0)>0$ such that
\begin{enumerate}
\item[(1).] $|\beta'[x,y]-\beta'[x,z]|_{\cF_x(y,z),l^1}\leq \cC_4(\rho_0,\epsilon_0)$ and
$$\sum_{\hF\in\cF_x(y,z)\setminus F_x}|\cI'_\hF[x,y]-\cI'_\hF[x,z]|_{l^1}\leq \cC_4(\rho_0,\epsilon_0);$$
\item[(2).] $|\beta'[y,x]-\beta'[z,x]|_{\cF_x(y,z),l^1}\leq \cC_4(\rho_0,\epsilon_0)$ and
$$\sum_{\hF\in\cF_x(y,z)\setminus F_x}|\cI'_\hF[y,x]-\cI'_\hF[z,x]|_{l^1}\leq \cC_4(\rho_0,\epsilon_0).$$
\end{enumerate}
\end{proposition}
\begin{proof}

By Lemma \ref{properties of Theta}, \hyperlink{Theta-3}{property ($\Theta$3) of $\Theta(\cdot,\cdot)$}, we can assume WLOG that $\cF_x(y,z)=\{F_x,F_1,...,F_k\}$ for some $k\geq 0$ and distinct $F_x,F_1,...,F_k\in\Gamma F$ such that $\Theta(F_x,F_j)=\{F_x,F_1,...,F_j\}$ for any $1\leq j\leq k$. In particular, $|\cF_x(y,z)|=k+1$. By Definition \ref{l1-seminorm}, \eqref{beta'}, \eqref{l1 norms of beta' seg} and \eqref{beta' detailed full}, we have
\begin{align*}&|\beta'[x,y]|_{\cF_x(y,z),l^1} \\
=&
\begin{cases}
\displaystyle |\ovec{[x,y]}|=1\leq |\cF_x(y,z)|,&\mathrm{if~}|\Theta(F_x,F_y)|= 1,\\
\displaystyle |\beta'_{F_xF_1}[x,y]|_{l^1}+|\beta'_{F_1F_2}[x,y]|_{l^1}+\cdots+|\beta'_{F_{k-1}F_k}[x,y]|_{l^1}=k<|\cF_x(y,z)|,~&\mathrm{otherwise}
\end{cases}
\end{align*}
and
\begin{align*}&|\beta'[y,x]|_{\cF_x(y,z),l^1} \\
=&
\begin{cases}
\displaystyle |\ovec{[y,x]}|=1\leq |\cF_x(y,z)|,&\mathrm{if~}|\Theta(F_x,F_y)|= 1,\\
\displaystyle |\beta'_{F_kF_{k-1}}[x,y]|_{l^1}+\cdots+|\beta'_{F_2F_1}[x,y]|_{l^1}+|\beta'_{F_{1}F_x}[x,y]|_{l^1}=k<|\cF_x(y,z)|,~&\mathrm{otherwise.}
\end{cases}
\end{align*}
Similarly, $|\beta'[x,z]|_{\cF_x(y,z),l^1}\leq |\cF_x(y,z)|$ and $|\beta'[z,x]|_{\cF_x(y,z),l^1}\leq |\cF_x(y,z)|$. Therefore by triangular inequality of $|\cdot|_{\cF_x(y,z),l^1}$, we have
$$|\beta'[x,y]-\beta'[x,z]|_{\cF_x(y,z),l^1}\leq 2|\cF_x(y,z)|\text{ and }|\beta'[y,x]-\beta'[z,x]|_{\cF_x(y,z),l^1}\leq 2|\cF_x(y,z)|.$$
WLOG we choose $\cC_4\geq 4$. Then by \eqref{beta' seg endpts norm} and the above, the proposition holds trivially when $|\cF_x(y,z)|=2$. It remains to consider the case when $|\cF_x(y,z)|=k+1\geq 3$. We prove the two assertions separately.
\begin{enumerate}
\item[(1).]
By \eqref{beta' seg}, \eqref{l1 norms of beta' seg} and \eqref{beta' detailed full},
\begin{align}\label{cut tail}
\begin{split}
&|\beta'[x,y]-\beta'[x,z]|_{\cF_x(y,z),l^1} \\
=&\left|\beta'_{F_xF_1}[x,y]-\beta'_{F_xF_1}[x,z]+\sum_{j=1}^{k-1}(\beta'_{F_jF_{j+1}}[x,y]-\beta'_{F_jF_{j+1}}[x,z])\right|_{l^1} \\
\leq&\left|\beta'_{F_xF_1}[x,y]\right|_{l^1}+\left|\beta'_{F_xF_1}[x,z]\right|_{l^1}+\sum_{j=1}^{k-1}\left|\beta'_{F_jF_{j+1}}[x,y]-\beta'_{F_jF_{j+1}}[x,z]\right|_{l^1} \\
\leq&2+\sum_{j=1}^{k-1}\left|\beta'_{F_jF_{j+1}}[x,y]-\beta'_{F_jF_{j+1}}[x,z]\right|_{l^1}.
\end{split}
\end{align}
Therefore, it suffices to find an upper bound for $\left|\beta'_{F_jF_{j+1}}[x,y]-\beta'_{F_jF_{j+1}}[x,z]\right|_{l^1}$ for any $1\leq j\leq k-1$. For any $1\leq j\leq k$, by the discussion near \eqref{beta' seg endpts norm}, we let
\begin{align*}
\cI'_{F_j}[x,y]=\sum_{\substack{w:w\in\Gamma x_0\\F_w=F_j}}\alpha_{y,w}^{(j)}w\mathrm{~and~}\cI'_{F_j}[x,z]=\sum_{\substack{w:w\in\Gamma x_0\\F_w=F_j}}\alpha_{z,w}^{(j)}w,
\end{align*}
where
\begin{align*}
\alpha_{y,w}^{(j)},\alpha_{z,w}^{(j)}\geq 0,~\forall w\in\Gamma x_0\mathrm{~s.t.~}F_w=F_j,\mathrm{~and~}\sum_{\substack{w:w\in\Gamma x_0\\F_w=F_j}}\alpha_{y,w}^{(j)}=\sum_{\substack{w:w\in\Gamma x_0\\F_w=F_j}}\alpha_{z,w}^{(j)}=1.
\end{align*}
Let
\begin{align*}
\cI'_{F_j,+}(x;y,z)=\sum_{\substack{w:w\in\Gamma x_0\\F_w=F_j}}\alpha_{w,+}^{(j)}w\mathrm{~and~}\cI'_{F_j,-}(x;y,z)=\sum_{\substack{w:w\in\Gamma x_0\\F_w=F_j}}\alpha_{w,-}^{(j)}w,
\end{align*}
where
\begin{align*}
\alpha^{(j)}_{w,+}=\max\{0,\alpha_{y,w}^{(j)}-\alpha_{z,w}^{(j)}\}\mathrm{~and~}\alpha^{(j)}_{w,-}=\max\{0,\alpha_{z,w}^{(j)}-\alpha_{y,w}^{(j)}\}.
\end{align*}
Hence
\begin{align}\label{rearange-1}
\alpha_{w,+}^{(j)}-\alpha_{w,-}^{(j)}=\alpha_{y,w}^{(j)}-\alpha_{z,w}^{(j)}\mathrm{~and~}\cI'_{F_{j}}[x,y]-\cI'_{F_{j}}[x,z]=\cI'_{F_j,+}(x;y,z)-\cI'_{F_j,-}(x;y,z).
\end{align}
Moreover,
\begin{align}\label{norm of difference at F_j-1}
\alpha^{(j)}:=\sum_{\substack{w:w\in\Gamma x_0\\ F_w=F_j}}\alpha_{w,+}^{(j)}=\sum_{\substack{w:w\in\Gamma x_0\\ F_w=F_j}}\alpha_{w,-}^{(j)}=\frac{1}{2}\sum_{\substack{w:w\in\Gamma x_0\\ F_w=F_j}}|\alpha_{w,+}^{(j)}-\alpha_{w,-}^{(j)}|=\frac{1}{2}\left|\cI'_{F_{j}}[x,y]-\cI'_{F_{j}}[x,z]\right|_{l^1}.
\end{align}
By \eqref{beta' seg endpts}, \eqref{rearange-1} and \eqref{norm of difference at F_j-1}, we notice that
\begin{align*}
\cI'_{F_{j}}[x,y]-\cI'_{F_{j}}[x,z]=&f(\cI'_{F_{j+1}}[x,y]-\cI'_{F_{j+1}}[x,z],x) \\
=&f(\cI'_{F_{j+1},+}(x;y,z)-\cI'_{F_{j+1},-}(x;y,z),x) \\
=&\sum_{\substack{w:w\in\Gamma x_0\\F_w=F_{j+1}}}\alpha_{w,+}^{(j+1)}f(w,x)-\sum_{\substack{v:v\in\Gamma x_0\\F_v=F_{j+1}}}\alpha_{v,-}^{(j+1)}f(v,x) \\
=&\sum_{\substack{w,v:w,v\in\Gamma x_0\\F_w=F_v=F_{j+1}}}\frac{\alpha_{w,+}^{(j+1)}\alpha_{v,-}^{(j+1)}}{\alpha^{(j+1)}}(f(w,x)-f(v,x)).
\end{align*}
Apply the fifth assertion in Proposition \ref{prop of f} and \eqref{norm of difference at F_j-1} to the above, we obtain
\begin{align*}
2\alpha^{(j)}=|\cI'_{F_{j}}[x,y]-\cI'_{F_{j}}[x,z]|_{l^1}\leq\sum_{\substack{w,v\in\Gamma x_0\\F_w=F_v=F_{j+1}}}\frac{\alpha_{w,+}^{(j+1)}\alpha_{v,-}^{(j+1)}}{\alpha^{(j+1)}}|f(w,x)-f(v,x)|_{l^1}\leq 2\lambda\alpha^{(j+1)}.
\end{align*}
Hence for any $1\leq j\leq k$,
\begin{align}\label{segment norm induction-1}
\begin{split}
2\alpha^{(j)}=&|\cI'_{F_{j}}[x,y]-\cI'_{F_{j}}[x,z]|_{l^1}\\\leq&2\lambda\alpha^{(j+1)}\leq\cdots\leq 2\lambda^{k-j}\alpha^{(k)}\leq\lambda^{k-j}(|\cI'_{F_{k}}[x,y]|_{l^1}+|\cI'_{F_{k}}[x,z]|_{l^1})\leq 2\lambda^{k-j}.
\end{split}
\end{align}
As a consequence of \eqref{beta' seg} (used in the second (in)equality of \eqref{C4 est-2}), \eqref{cut tail} (used in the first inequality of \eqref{C4 est-2}), \eqref{rearange-1} (used in the third (in)equality of \eqref{C4 est-2}), \eqref{norm of difference at F_j-1} (used in the fifth (in)equality of \eqref{C4 est-2}), \eqref{segment norm induction-1} (used in \eqref{C4 est-1} and the sixth (in)equality of \eqref{C4 est-2}) and the remark after Definition \ref{dfn:geo.bicomb} (used in the fourth (in)equality of \eqref{C4 est-2}), we have
\begin{align}\label{C4 est-1}
\sum_{\hF\in\cF_x(y,z)\setminus F_x}|\cI'_\hF[x,y]-\cI'_\hF[x,z]|_{l^1}=\sum_{j=1}^k|\cI'_{F_j}[x,y]-\cI'_{F_j}[x,z]|_{l^1}\leq\sum_{j=1}^k2\lambda^{k-j}\leq\frac{2}{1-\lambda}
\end{align}
and
\begin{align}\label{C4 est-2}
\begin{split}
&|\beta'[x,y]-\beta'[x,z]|_{\cF_x(y,z),l^1} \\
\leq&2+\sum_{j=1}^{k-1}\left|\beta'_{F_jF_{j+1}}[x,y]-\beta'_{F_jF_{j+1}}[x,z]\right|_{l^1} \\
=&2+\sum_{j=1}^{k-1}\left|\sum_{\substack{w:w\in\Gamma x_0\\F_w=F_{j+1}}}\alpha_{y,w}^{(j+1)}\ovec{[f(w,x),w]}-\sum_{\substack{w:w\in\Gamma x_0\\F_w=F_{j+1}}}\alpha_{z,w}^{(j+1)}\ovec{[f(w,x),w]}\right|_{l^1} \\
=&2+\sum_{j=1}^{k-1}\left|\sum_{\substack{w:w\in\Gamma x_0\\F_w=F_{j+1}}}\left(\alpha_{w,+}^{(j+1)}-\alpha_{w,-}^{(j+1)}\right)\ovec{[f(w,x),w]}\right|_{l^1} \\
=&2+\sum_{j=1}^{k-1}\sum_{\substack{w:w\in\Gamma x_0\\F_w=F_{j+1}}}\left|\alpha_{w,+}^{(j+1)}-\alpha_{w,-}^{(j+1)}\right|= 2+\sum_{j=1}^{k-1}2\alpha^{(j+1)}\leq2+\sum_{j=1}^{k-1}2\lambda^{k-j-1}\leq\frac{2}{1-\lambda}+2.
\end{split}
\end{align}
\item[(2).] Similar to the beginning of the discussions regarding the first requirement, by \eqref{beta' detailed partial} and \eqref{beta' detailed full}, we have
\begin{align}\label{cut head}
\begin{split}
&|\beta'[y,x]-\beta'[z,x]|_{\cF_x(y,z),l^1} \\
=&\left|\left(\sum_{j=1}^{k-1}\beta'_{F_{j+1}F_{j}}[y,x]-\beta'_{F_{j+1}F_j}[z,x]\right)+\beta'_{F_1F_x}[y,x]-\beta'_{F_1F_x}[z,x]\right|_{l^1} \\
\leq&\left(\sum_{j=1}^{k-1}\left|\beta'_{F_{j+1}F_j}[y,x]-\beta'_{F_{j+1}F_j}[z,x]\right|_{l^1}\right)+\left|\beta'_{F_1F_x}[y,x]\right|_{l^1}+\left|\beta'_{F_1F_x}[z,x]\right|_{l^1} \\
\leq&2+\sum_{j=1}^{k-1}\left|\beta'_{F_{j+1}F_j}[y,x]-\beta'_{F_{j+1}F_j}[z,x]\right|_{l^1}.
\end{split}
\end{align}
Therefore, it suffices to find a suitable upper bound for $\left|\beta'_{F_{j+1}F_j}[y,x]-\beta'_{F_{j+1}F_j}[z,x]\right|_{l^1}$ for any $1\leq j\leq k-1$. For any $1\leq j\leq k$, by the discussion near \eqref{beta' seg endpts norm}, we let
\begin{align*}
\cI'_{F_j}[y,x]=\sum_{\substack{w\in\Gamma x_0\\F_w=F_j}}\alpha_{y,w}^{(j)}w\mathrm{~and~}\cI'_{F_j}[z,x]=\sum_{\substack{w\in\Gamma x_0\\F_w=F_j}}\alpha_{z,w}^{(j)}w,
\end{align*}
where
\begin{align*}
\alpha_{y,w}^{(j)},\alpha_{z,w}^{(j)}\geq 0,~\forall w\in\Gamma x_0\mathrm{~s.t.~}F_w=F_j,\mathrm{~and~}\sum_{\substack{w:w\in\Gamma x_0\\F_w=F_j}}\alpha_{y,w}^{(j)}=\sum_{\substack{w:w\in\Gamma x_0\\F_w=F_j}}\alpha_{z,w}^{(j)}=1.
\end{align*}
Let
\begin{align}\label{+- decomp-2}
\cI'_{F_j,+}(x;y,z)=\sum_{\substack{w:w\in\Gamma x_0\\F_w=F_j}}\alpha_{w,+}^{(j)}w\mathrm{~and~}\cI'_{F_j,-}(x;y,z)=\sum_{\substack{w:w\in\Gamma x_0\\F_w=F_j}}\alpha_{w,-}^{(j)}w,
\end{align}
where
\begin{align}\label{+- decomp coeff-2}
\alpha^{(j)}_{w,+}=\max\{0,\alpha_{y,w}^{(j)}-\alpha_{z,w}^{(j)}\}\leq \alpha^{(j)}_{y,w}\mathrm{~and~}\alpha^{(j)}_{w,-}=\max\{0,\alpha_{z,w}^{(j)}-\alpha_{y,w}^{(j)}\}\leq \alpha^{(j)}_{z,w}.
\end{align}
Hence
\begin{align}\label{rearange-2}
\alpha_{w,+}^{(j)}-\alpha_{w,-}^{(j)}=\alpha_{y,w}^{(j)}-\alpha_{z,w}^{(j)}\mathrm{~and~}\cI'_{F_{j}}[y,x]-\cI'_{F_{j}}[z,x]=\cI'_{F_j,+}(x;y,z)-\cI'_{F_j,-}(x;y,z).
\end{align}
and
\begin{align}\label{norm of difference at F_j-2}
\alpha^{(j)}:=\sum_{\substack{w:w\in\Gamma x_0\\ F_w=F_j}}\alpha_{w,+}^{(j)}=\sum_{\substack{w:w\in\Gamma x_0\\ F_w=F_j}}\alpha_{w,+}^{(j)}=\frac{1}{2}\sum_{\substack{w:w\in\Gamma x_0\\ F_w=F_j}}|\alpha_{w,+}^{(j)}-\alpha_{w,-}^{(j)}|=\frac{1}{2}\left|\cI'_{F_{j}}[y,x]-\cI'_{F_{j}}[z,x]\right|_{l^1}.
\end{align}
For any $w\in\Gamma x_0$ with $F_w=F_j$, by \eqref{+- decomp coeff-2} and \eqref{rearange-2} we have $\alpha^{(j)}_{w,0}=\alpha^{(j)}_{y,w}-\alpha^{(j)}_{w,+}=\alpha^{(j)}_{z,w}-\alpha^{(j)}_{w,-}\geq 0$. Hence from \eqref{+- decomp-2} and \eqref{norm of difference at F_j-2}, we can define
\begin{align}\label{overlap terms}
\cI'_{F_j,0}(x;y,z):=\cI'_{F_j}[y,x]-\cI'_{F_j,+}(x;y,z)=\cI'_{F_j}[z,x]-\cI'_{F_j,-}(x;y,z)=\sum_{\substack{w:w\in\Gamma x_0\\F_w=F_j}}\alpha^{(j)}_{w,0}w
\end{align}
and
\begin{align}\label{overlap norm}
|\cI'_{F_j,0}(x;y,z)|_{l^1}=\sum_{\substack{w:w\in\Gamma x_0\\F_w=F_j}}|\alpha^{(j)}_{w,0}|=\sum_{\substack{w:w\in\Gamma x_0\\F_w=F_j}}\alpha^{(j)}_{w,0}=\sum_{\substack{w:w\in\Gamma x_0\\F_w=F_j}}\alpha^{(j)}_{y,w}-\alpha^{(j)}_{w,+}=1-\alpha^{(j)}.
\end{align}
We first observe that for any $1\leq j\leq k-1$, by \eqref{beta' seg endpts} (used in the second equality), \eqref{+- decomp-2} (used in the sixth (in)equality), \eqref{norm of difference at F_j-2} (used in the first and the sixth (in)equalities), \eqref{overlap terms} (used in the third and the fifth (in)equalities), \eqref{overlap norm} (used in the eighth (in)equality) and the fourth assertion in Proposition \ref{prop of f} (used in the seventh (in)equality), we have
\begin{align*}
&2\alpha^{(j+1)}\\
=&\left|\cI'_{F_{j+1}}[y,x]-\cI'_{F_{j+1}}[z,x]\right|_{l^1}\\
=&\left|f(\cI'_{F_{j}}[y,x],y)-f(\cI'_{F_{j}}[z,x],z)\right|_{l^1} \\
=& \left|f(\cI'_{F_{j},+}(x;y,z),y)+f(\cI'_{F_{j},0}(x;y,z),y)-f(\cI'_{F_{j},-}(x;y,z),z)-f(\cI'_{F_{j},0}(x;y,z),z)\right|_{l^1}\\
\leq&\left|f(\cI'_{F_{j},+}(x;y,z),y)\right|_{l^1}+\left|f(\cI'_{F_{j},-}(x;y,z),z)\right|_{l^1}+\left|f(\cI'_{F_{j},0}(x;y,z),y)-f(\cI'_{F_{j},0}(x;y,z),z)\right|_{l^1} \\
=&\left|\cI'_{F_{j},+}(x;y,z)\right|_{l^1}+\left|\cI'_{F_{j},-}(x;y,z)\right|_{l^1}+\left|\sum_{\substack{w:w\in\Gamma x_0\\F_w=F_j}}\alpha^{(j)}_{w,0}(f(w,y)-f(w,z))\right|_{l^1} \\
\leq&2\alpha^{(j)}+\sum_{\substack{w:w\in\Gamma x_0\\F_w=F_j}}\alpha^{(j)}_{w,0}\left|f(w,y)-f(w,z)\right|_{l^1}\\
\leq&2\alpha^{(j)}+\cC_3\lambda^{(F_y|F_z)_{F_j}}\sum_{\substack{w:w\in\Gamma x_0\\F_w=F_j}}\alpha^{(j)}_{w,0}=2\alpha^{(j)}+\cC_3\lambda^{k-j+1}(1-\alpha^{(j)})\leq 2\alpha^{(j)}+\cC_3\lambda^{k-j+1}.
\end{align*}
Hence by \eqref{beta' seg endpts}, \eqref{norm of difference at F_j-2}, the fourth assertion in Proposition \ref{prop of f} and the above,
\begin{align}\label{seg norm induction-2}
\begin{split}
&\left|\cI'_{F_{j}}[y,x]-\cI'_{F_{j}}[z,x]\right|_{l^1}=2\alpha^{(j)}\\
\leq &2\alpha^{(j-1)}+\cC_3\lambda^{k-j+2} \\
\leq& 2\alpha^{(j-2)}+\cC_3\lambda^{k-j+2}+\cC_3\lambda^{k-j+3}\\
\leq&\cdots\leq 2\alpha^{(1)}+\cC_3\lambda^{k-j+2}+\cdots+\cC_3\lambda^{k} \\
\leq&\left|\cI'_{F_{1}}[y,x]-\cI'_{F_{1}}[z,x]\right|_{l^1}+\cC_3\lambda^{k-j+2}+\cdots+\cC_3\lambda^{k} \\
=&\left|f(x,y)-f(x,z)\right|_{l^1}+\cC_3\lambda^{k-j+2}+\cdots+\cC_3\lambda^{k}=\cC_3\lambda^{k-j+2}+\cdots+\cC_3\lambda^{k+1}\leq\frac{\cC_3\lambda^{k-j+2}}{1-\lambda}.
\end{split}
\end{align}
As a direct corollary of \eqref{seg norm induction-2}, we have
\begin{align}\label{C4 est-3}
\begin{split}
\sum_{\hF\in\cF_x(y,z)\setminus\{F_x\}}\left|\cI'_{\hF}[y,x]-\cI'_{\hF}[z,x]\right|_{l^1}=\sum_{j=1}^{k}\left|\cI'_{F_{j}}[y,x]-\cI'_{F_{j}}[z,x]\right|_{l^1}\leq \sum_{j=1}^{k}\frac{\cC_3\lambda^{k-j+2}}{1-\lambda}\leq\frac{\cC_3\lambda^2}{(1-\lambda)^2}.
\end{split}
\end{align}
In order to bound $\left|\beta'_{F_{j+1}F_j}[y,x]-\beta'_{F_{j+1}F_j}[z,x]\right|_{l^1}$ from above for any $1\leq j\leq k-1$, we notice that by the remark after Definition \ref{dfn:geo.bicomb} (used in the fifth (in)equality), \eqref{beta' seg} (used in the first equality), \eqref{norm of difference at F_j-2} (used in the seventh (in)equality), \eqref{overlap terms} (used in the second equality), \eqref{overlap norm} (used in the seventh (in)equality), \eqref{seg norm induction-2} (used in the ninth (in)equality), the first and the fourth assertions in Proposition \ref{prop of f} (used in the sixth (in)equality), we have
\begin{align*}
&\left|\beta'_{F_{j+1}F_j}[y,x]-\beta'_{F_{j+1}F_j}[z,x]\right|_{l^1} \\
=&\left|\sum_{\substack{w:w\in\Gamma x_0\\F_w=F_j}}\alpha^{(j)}_{y,w}\ovec{[f(w,y),w]}-\sum_{\substack{w:w\in\Gamma x_0\\F_w=F_j}}\alpha^{(j)}_{z,w}\ovec{[f(w,z),w]}\right|_{l^1} \\
=&\left|\sum_{\substack{w:w\in\Gamma x_0\\F_w=F_j}}(\alpha^{(j)}_{w,+}+\alpha^{(j)}_{w,0})\ovec{[f(w,y),w]}-\sum_{\substack{w:w\in\Gamma x_0\\F_w=F_j}}(\alpha^{(j)}_{w,-}+\alpha^{(j)}_{w,0})\ovec{[f(w,z),w]}\right|_{l^1} \\
\leq&\left|\sum_{\substack{w:w\in\Gamma x_0\\F_w=F_j}}\alpha^{(j)}_{w,+}\ovec{[f(w,y),w]}\right|_{l^1}+\left|\sum_{\substack{w:w\in\Gamma x_0\\F_w=F_j}}\alpha^{(j)}_{w,-}\ovec{[f(w,z),w]}\right|_{l^1}+\left|\sum_{\substack{w:w\in\Gamma x_0\\F_w=F_j}}\alpha^{(j)}_{w,0}\ovec{[f(w,y)-f(w,z),w]}\right|_{l^1} \\
\leq&\sum_{\substack{w:w\in\Gamma x_0\\F_w=F_j}}\alpha^{(j)}_{w,+}\left|\ovec{[f(w,y),w]}\right|_{l^1}+\sum_{\substack{w:w\in\Gamma x_0\\F_w=F_j}}\alpha^{(j)}_{w,-}\left|\ovec{[f(w,z),w]}\right|_{l^1}+\sum_{\substack{w:w\in\Gamma x_0\\F_w=F_j}}\alpha^{(j)}_{w,0}\left|\ovec{[f(w,y)-f(w,z),w]}\right|_{l^1}\\
=&\sum_{\substack{w:w\in\Gamma x_0\\F_w=F_j}}\alpha^{(j)}_{w,+}\left|f(w,y)\right|_{l^1}+\sum_{\substack{w:w\in\Gamma x_0\\F_w=F_j}}\alpha^{(j)}_{w,-}\left| f(w,z)\right|_{l^1}+\sum_{\substack{w:w\in\Gamma x_0\\F_w=F_j}}\alpha^{(j)}_{w,0}\left| f(w,y)-f(w,z)\right|_{l^1}\\
\leq&\sum_{\substack{w:w\in\Gamma x_0\\F_w=F_j}}\alpha^{(j)}_{w,+}+\sum_{\substack{w:w\in\Gamma x_0\\F_w=F_j}}\alpha^{(j)}_{w,-}+\sum_{\substack{w:w\in\Gamma x_0\\F_w=F_j}}\alpha^{(j)}_{w,0}\cC_3\lambda^{(F_y|F_z)_{F_j}} \\
=&2\alpha^{(j)}+\cC_3\lambda^{k-j+1}(1-\alpha^{(j)})<2\alpha^{(j)}+\cC_3\lambda^{k-j+1}\leq\cC_3\frac{\lambda^{k-j+2}}{1-\lambda}+\cC_3\lambda^{k-j+1}=\cC_3\frac{\lambda^{k-j+1}}{1-\lambda}.
\end{align*}
Applying the above to \eqref{cut head}, we have
\begin{align}\label{C4 est-4}
\begin{split}
&|\beta'[y,x]-\beta'[z,x]|_{\cF_x(y,z),l^1} \\
\leq&2+\sum_{j=1}^{k-1}\left|\beta'_{F_{j+1}F_j}[y,x]-\beta'_{F_{j+1}F_j}[z,x]\right|_{l^1}\leq 2+\cC_3\sum_{j=1}^{k-1}\frac{\lambda^{k-j+1}}{1-\lambda}\leq 2+\frac{\cC_3}{(1-\lambda)^2}.
\end{split}
\end{align}
\end{enumerate}
As a summary, by \eqref{C4 est-1}, \eqref{C4 est-2}, \eqref{C4 est-3} and \eqref{C4 est-4}, we choose
$$\cC_4=\max\left\{4,\frac{2}{1-\lambda},\frac{2}{1-\lambda}+2,\frac{\cC_3\lambda^2}{(1-\lambda)^2},2+\frac{\cC_3}{(1-\lambda)^2}\right\}.$$
and the proposition follows. (Dependence on $\rho_0$ and $\epsilon_0$ follows directly from the fourth and the fifth assertions in Proposition \ref{prop of f}.)
\end{proof}
Now we are ready to construct our desired homological bicombing. For any $x,y\in\Gamma x_0$ and $\hF\in\Theta(F_x,F_y)$, we define
\begin{align}\label{beta}
\beta[x,y]=\frac{1}{2}\left(\beta'[x,y]-\beta'[y,x]\right)~\mathrm{and~if~}F_x\neq F_y,~\mathrm{then}~\cI_\hF[x,y]:=\frac{1}{2}\left(\cI'_\hF[x,y]-\cI'_\hF[y,x]\right).
\end{align}
Then we have the following corollary of Proposition \ref{prop of beta'}.

\begin{corollary}\label{bicombing finished}
$\beta[\cdot,\cdot]$ is a homological bicombing in the sense of Definition \ref{Homological bicombing}. Moreover, there exists $\cC_5=\cC_5(\rho_0,\epsilon_0)>0$ such that for any $x,y,z\in\Gamma x_0$,
$$\left|\beta[x,y]+\beta[y,z]+\beta[z,x]\right|_{l^1}\leq \cC_5.$$
\end{corollary}
\begin{proof}
By the second assertion in Proposition \ref{prop of f} and \eqref{beta'}, we have $\gamma\beta'[x,y]=\beta'[\gamma x,\gamma y]$ for any $x,y\in\Gamma x_0$ and any $\gamma\in\Gamma$. Therefore $\beta$ is a homological bicombing in the sense of Definition \ref{Homological bicombing} following \eqref{beta}.

It remains for us to find an upper bound for $\left|\beta[x,y]+\beta[y,z]+\beta[z,x]\right|_{l^1}$. If two of $F_x,F_y,F_z$ are the same, WLOG we assume that $F_x=F_y$. Then $\beta[x,y]=\ovec{[x,y]}$, $\cF_x(y,z)=\{F_x\}=\cF_y(z,x)$ and $\cF_{z}(x,y)=\Theta(F_x,F_z)$. In this case, by \eqref{l1 norms of beta' seg}, \eqref{beta' detailed full} and Proposition \ref{prop of beta'},
\begin{align}\label{C5 est-2}
\begin{split}
&\left|\beta[x,y]+\beta[y,z]+\beta[z,x]\right|_{l^1} \leq\left|\ovec{[x,y]}\right|_{l^1}+\left|\beta[y,z]+\beta[z,x]\right|_{l^1}\\
=&1+\frac{1}{2}\left|\beta'[y,z]-\beta'[z,y]+\beta'[z,x]-\beta'[x,z]\right|_{l^1}\\
\leq&1+\frac{1}{2}\left|\beta'[y,z]-\beta'[x,z]\right|_{l^1}+\frac{1}{2}\left|\beta'[z,x]-\beta'[z,y]\right|_{l^1} \\
=&1+\frac{1}{2}\left|\beta'[y,z]-\beta'[x,z]\right|_{\cF_z(x,y), l^1}+\frac{1}{2}\left|\beta'[z,x]-\beta'[z,y]\right|_{\cF_z(x,y), l^1} \leq 1+\cC_4.
\end{split}
\end{align}

From now on, we assume that $F_x,F_y,F_z$ are distinct elements in $\Gamma F$. By Lemma \ref{properties of Theta}, \hyperlink{Theta-3}{property ($\Theta$3) of $\Theta(\cdot,\cdot)$}, we can find $F_{x;y,z}\in\cF_{x}(y,z)$, $F_{y;z,x}\in\cF_{y}(z,x)$ and $F_{z;x,y}\in\cF_{z}(x,y)$ such that $\Theta(F_x,F_{x;y,z})=\cF_{x}(y,z)$, $\Theta(F_y,F_{y;z,x})=\cF_{y}(z,x)$ and $\Theta(F_z,F_{z;x,y})=\cF_{z}(x,y)$. (See Figure \ref{eqn5.45}.) We also notice that
\begin{figure}[h]
	\centering
	\includegraphics[width=4in]{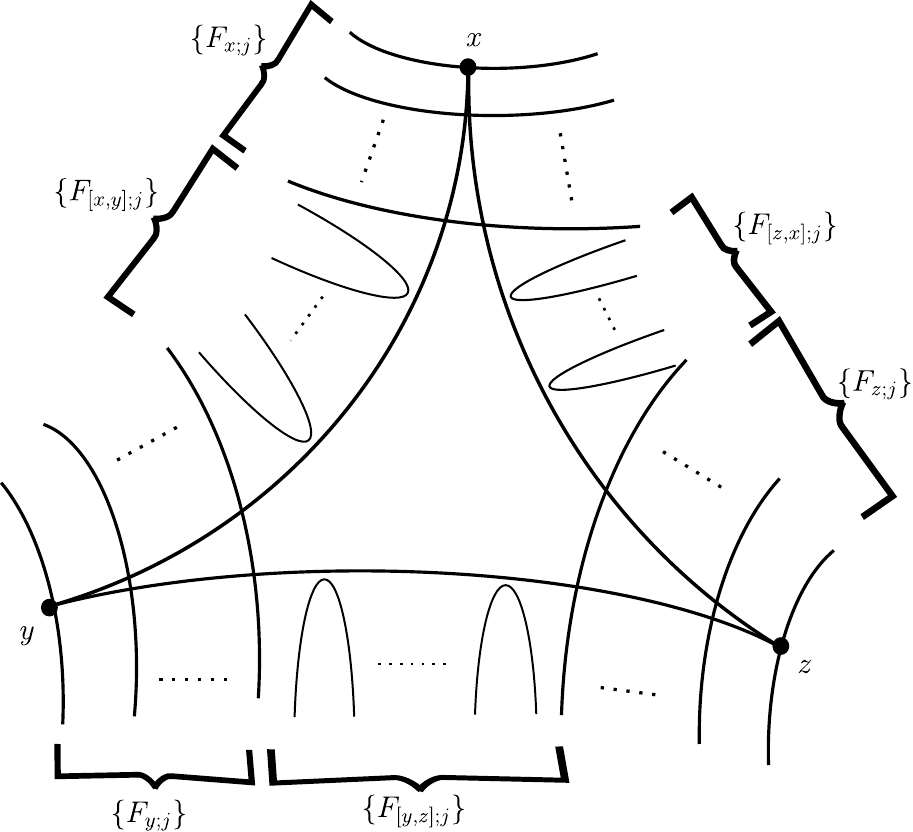}
	\caption{ \label{eqn5.45}}
\end{figure}
$$\cF_x(y,z)\cap\cF_{y}(x,z)=\cF_{x}(y,z)\cap\cF_{z}(x,y)=\cF_y(x,z)\cap\cF_{z}(x,y)=\Theta(F_x,F_y)\cap\Theta(F_y,F_z)\cap\Theta(F_z,F_x).$$
We claim that $|\Theta(F_x,F_y)\cap\Theta(F_y,F_z)\cap\Theta(F_z,F_x)|\leq 1$. In fact, if there exist distinct $F',F''\in\Theta(F_x,F_y)\cap\Theta(F_y,F_z)\cap\Theta(F_z,F_x)$, then by Lemma \ref{properties of Theta}, \hyperlink{Theta-1}{property ($\Theta$1) of $\Theta(\cdot,\cdot)$}, $F',F''$ are $\epsilon_0/2$ close to $[x,y],[y,z]$ and $[z,x]$. By Proposition \ref{almost ints positioning} and the fact that $F'\neq F''$, this is impossible. Hence we conclude that $|\Theta(F_x,F_y)\cap\Theta(F_y,F_z)\cap\Theta(F_z,F_x)|\leq 1$. As a corollary of Lemma \ref{properties of Theta}, \hyperlink{Theta-3}{property ($\Theta$3) of $\Theta(\cdot,\cdot)$}, we have
\begin{align}\label{eqn:no triple ints}
\begin{split}
\Theta(F_x,F_y)=(\cF_x(y,z)\setminus\{F_{x;y,z}\})\sqcup\Theta(F_{x;y,z},F_{y;x,z})\sqcup(\cF_y(x,z)\setminus\{F_{y;x,z}\}),\\
\Theta(F_y,F_z)=(\cF_y(x,z)\setminus\{F_{y;x,z}\})\sqcup\Theta(F_{y;x,z},F_{z;x,y})\sqcup(\cF_z(x,y)\setminus\{F_{z;x,y}\}),\\
\Theta(F_x,F_z)=(\cF_x(y,z)\setminus\{F_{x;y,z}\})\sqcup\Theta(F_{x;y,z},F_{z;x,y})\sqcup(\cF_z(x,y)\setminus\{F_{z;x,y}\}).
\end{split}
\end{align}
By Lemma \ref{properties of Theta}, \hyperlink{Theta-4}{property ($\Theta$4) of $\Theta(\cdot,\cdot)$} and \eqref{eqn:no triple ints}, we have the following cardinality estimate.
\begin{align}\label{bad cardinality dim 1}
\begin{split}
&\left|\Theta(F_{x;y,z},F_{y;z,x})\union\Theta(F_{y;z,x},F_{z;x,y})\union\Theta(F_{z;x,y},F_{x;y,z})\right| \\
=&\left|(\cF(x,y,z)\setminus\cA(x,y,z))\sqcup\{F_{x;y,z},F_{y;z,x},F_{z;x,y}\}\right|\leq 6.
\end{split}
\end{align}
By Lemma \ref{properties of Theta}, \hyperlink{Theta-3}{property ($\Theta$3) of $\Theta(\cdot,\cdot)$}, we can assume that $\Theta(F_x,F_{x;y,z})=\{F_{x;0}:=F_x,F_{x;1},...,F_{x:k_x}:=F_{x;y,z}\}$, $\Theta(F_y,F_{y;z,x})=\{F_{y;0}:=F_y,F_{y;1},...,F_{y:k_y}:=F_{y;z,x}\}$ and $\Theta(F_z,F_{z;x,y})=\{F_{z;0}:=F_z,F_{z;1},...,F_{z:k_z}:=F_{z;x,y}\}$ for some $k_x,k_y,k_z\geq 0$ such that
\begin{itemize}
\item $F_{p,0},...,F_{p,k_p}$ are distinct elements in $\Gamma x_0$ for any $p\in\{x,y,z\}$;
\item $\Theta(F_{p,i},F_{p,j})=\{F_{p,i},F_{p,i+1},...,F_{p,j}\}$ for any $0\leq i\leq j\leq k_p$ and $p\in\{x,y,z\}$.
\end{itemize}
(See Figure \ref{eqn5.45}.) We further assume that $\Theta(F_{x;y,z},F_{y;z,x})=\{F_{[x,y],0}:=F_{x;y,z},F_{[x,y],1},...,F_{[x,y],m_{[x,y]}}:=F_{y;z,x}\}$, $\Theta(F_{y;z,x},F_{z;x,y})=\{F_{[y,z],0}:=F_{y;z,x},F_{[y,z],1},...,F_{[y,z],m_{[y,z]}}:=F_{z;x,y}\}$ and $\Theta(F_{z;x,y},F_{x;y,z})=\{F_{[z,x],0}:=F_{z;x,y},F_{[z,x],1},...,F_{[z,x],m_{[z,x]}}:=F_{x;y,z}\}$ for any $m_{[x,y]},m_{[y,z]},m_{[z,x]}\geq 0$ such that
\begin{itemize}
\item $F_{[p,q],0},...,F_{[p,q],m_{[p,q]}}$ are distinct elements in $\Gamma x_0$ for any distinct $p, q\in\{x,y,z\}$;
\item $\Theta(F_{[p,q],i},F_{[p,q],j})=\{F_{[p,q],i},F_{[p,q],i+1},...,F_{[p,q],j}\}$ for any $0\leq i\leq j\leq m_{[p,q]}$ and $p\neq q\in\{x,y,z\}$.
\end{itemize}
By \eqref{bad cardinality dim 1}, we have
\begin{align}\label{bad cardinality dim 1 rewrite}
(m_{[x,y]}+1)+(m_{[y,z]}+1)+(m_{z,x}+1)\leq 18.
\end{align}
With the above definitions, it follows from \eqref{l1 norms of beta' seg} (used in the fifth (in)equality), \eqref{beta' detailed full} (used in the second and the fifth (in)equalities), Proposition \ref{prop of beta'} (used in the last inequality), \eqref{eqn:no triple ints} (used in the second equality) and \eqref{bad cardinality dim 1 rewrite} (used in the last inequality) that
\begin{align}\label{C5 est-3}
&\left|\beta[x,y]+\beta[y,z]+\beta[z,x]\right|_{l^1} \nonumber\\
=&\frac{1}{2}\left|\beta'[x,y]-\beta'[y,x]+\beta'[y,z]-\beta'[z,y]+\beta'[z,x]-\beta'[x,z]\right|_{l^1} \nonumber\\
=&\frac{1}{2}\left|\left[\left(\sum_{j=0}^{k_x-1}\beta'_{F_{x;j}F_{x;j+1}}[x,y]\right)+\left(\sum_{j=0}^{m_{[x,y]}-1}\beta'_{F_{[x,y],j}F_{[x,y],j+1}}[x,y]\right)+\left(\sum_{j=0}^{k_y-1}\beta'_{F_{y;j+1}F_{y;j}}[x,y]\right)\right]\right. \nonumber\\
&\quad-\left[\left(\sum_{j=0}^{k_x-1}\beta'_{F_{x;j+1}F_{x;j}}[y,x]\right)+\left(\sum_{j=0}^{m_{[x,y]}-1}\beta'_{F_{[x,y],j+1}F_{[x,y],j}}[y,x]\right)+\left(\sum_{j=0}^{k_y-1}\beta'_{F_{y;j}F_{y;j+1}}[y,x]\right)\right] \nonumber\\
&\quad+\left[\left(\sum_{j=0}^{k_y-1}\beta'_{F_{y;j}F_{y;j+1}}[y,z]\right)+\left(\sum_{j=0}^{m_{[y,z]}-1}\beta'_{F_{[y,z],j}F_{[y,z],j+1}}[y,z]\right)+\left(\sum_{j=0}^{k_z-1}\beta'_{F_{z;j+1}F_{z;j}}[y,z]\right)\right]\nonumber \\
&\quad-\left[\left(\sum_{j=0}^{k_y-1}\beta'_{F_{y;j+1}F_{y;j}}[z,y]\right)+\left(\sum_{j=0}^{m_{[y,z]}-1}\beta'_{F_{[y,z],j+1}F_{[y,z],j}}[z,y]\right)+\left(\sum_{j=0}^{k_z-1}\beta'_{F_{z;j}F_{z;j+1}}[z,y]\right)\right] \nonumber\\
&\quad+\left[\left(\sum_{j=0}^{k_z-1}\beta'_{F_{z;j}F_{z;j+1}}[z,x]\right)+\left(\sum_{j=0}^{m_{[z,x]}-1}\beta'_{F_{[z,x],j}F_{[z,x],j+1}}[z,x]\right)+\left(\sum_{j=0}^{k_x-1}\beta'_{F_{x;j+1}F_{x;j}}[z,x]\right)\right] \nonumber\\
&\quad-\left.\left[\left(\sum_{j=0}^{k_z-1}\beta'_{F_{z;j+1}F_{z;j}}[x,z]\right)+\left(\sum_{j=0}^{m_{[z,x]}-1}\beta'_{F_{[z,x],j+1}F_{[z,x],j}}[x,z]\right)+\left(\sum_{j=0}^{k_x-1}\beta'_{F_{x;j}F_{x;j+1}}[x,z]\right)\right]\right|_{l^1}\nonumber \\
=& \frac{1}{2}\left|\left(\sum_{j=0}^{k_x-1}\beta'_{F_{x;j}F_{x;j+1}}[x,y]-\beta'_{F_{x;j}F_{x;j+1}}[x,z]\right)-\left(\sum_{j=0}^{k_x-1}\beta'_{F_{x;j+1}F_{x;j}}[y,x]-\beta'_{F_{x;j+1}F_{x;j}}[z,x]\right)\right.\nonumber\\
&\quad+\left(\sum_{j=0}^{k_y-1}\beta'_{F_{y;j+1}F_{y;j}}[x,y]-\beta'_{F_{y;j+1}F_{y;j}}[z,y]\right)-\left(\sum_{j=0}^{k_y-1}\beta'_{F_{y;j}F_{y;j+1}}[y,x]-\beta'_{F_{y;j}F_{y;j+1}}[y,z]\right)\nonumber\\
&\quad+\left(\sum_{j=0}^{k_z-1}\beta'_{F_{z;j+1}F_{z;j}}[y,z]-\beta'_{F_{z;j+1}F_{z;j}}[x,z]\right)-\left(\sum_{j=0}^{k_z-1}\beta'_{F_{z;j}F_{z;j+1}}[z,y]-\beta'_{F_{z;j}F_{z;j+1}}[z,x]\right) \nonumber\\
&\quad+\left(\sum_{j=0}^{m_{[x,y]}-1}\beta'_{F_{[x,y],j}F_{[x,y],j+1}}[x,y]\right)-\left(\sum_{j=0}^{m_{[x,y]}-1}\beta'_{F_{[x,y],j+1}F_{[x,y],j}}[y,x]\right) \nonumber\\
&\quad +\left(\sum_{j=0}^{m_{[y,z]}-1}\beta'_{F_{[y,z],j}F_{[y,z],j+1}}[y,z]\right)- \left(\sum_{j=0}^{m_{[y,z]}-1}\beta'_{F_{[y,z],j+1}F_{[y,z],j}}[z,y]\right)\nonumber\\
&\quad\left.+\left(\sum_{j=0}^{m_{[z,x]}-1}\beta'_{F_{[z,x],j}F_{[z,x],j+1}}[z,x]\right)-\left(\sum_{j=0}^{m_{[z,x]}-1}\beta'_{F_{[z,x],j+1}F_{[z,x],j}}[x,z]\right) \right|_{l^1}\nonumber\\
\leq& \frac{1}{2}\left|\sum_{j=0}^{k_x-1}\beta'_{F_{x;j}F_{x;j+1}}[x,y]-\beta'_{F_{x;j}F_{x;j+1}}[x,z]\right|_{l^1}+\frac{1}{2}\left|\sum_{j=0}^{k_x-1}\beta'_{F_{x;j+1}F_{x;j}}[y,x]-\beta'_{F_{x;j+1}F_{x;j}}[z,x]\right|_{l^1}\nonumber\\
&+\frac{1}{2}\left|\sum_{j=0}^{k_y-1}\beta'_{F_{y;j+1}F_{y;j}}[x,y]-\beta'_{F_{y;j+1}F_{y;j}}[z,y]\right|_{l^1}+\frac{1}{2}\left|\sum_{j=0}^{k_y-1}\beta'_{F_{y;j}F_{y;j+1}}[y,x]-\beta'_{F_{y;j}F_{y;j+1}}[y,z]\right|_{l^1}\nonumber\\
&+\frac{1}{2}\left|\sum_{j=0}^{k_z-1}\beta'_{F_{z;j+1}F_{z;j}}[y,z]-\beta'_{F_{z;j+1}F_{z;j}}[x,z]\right|_{l^1}+\frac{1}{2}\left|\sum_{j=0}^{k_z-1}\beta'_{F_{z;j}F_{z;j+1}}[z,y]-\beta'_{F_{z;j}F_{z;j+1}}[z,x]\right|_{l^1} \nonumber\\
&+\frac{1}{2}\sum_{j=0}^{m_{[x,y]}-1}\left|\beta'_{F_{[x,y],j}F_{[x,y],j+1}}[x,y]\right|_{l^1}+\frac{1}{2}\sum_{j=0}^{m_{[x,y]}-1}\left|\beta'_{F_{[x,y],j+1}F_{[x,y],j}}[y,x]\right|_{l^1} \nonumber\\
&+\frac{1}{2}\sum_{j=0}^{m_{[y,z]}-1}\left|\beta'_{F_{[y,z],j}F_{[y,z],j+1}}[y,z]\right|_{l^1}+\frac{1}{2}\sum_{j=0}^{m_{[y,z]}-1}\left|\beta'_{F_{[y,z],j+1}F_{[y,z],j}}[z,y]\right|_{l^1}\nonumber\\
&+\frac{1}{2}\sum_{j=0}^{m_{[z,x]}-1}\left|\beta'_{F_{[z,x],j}F_{[z,x],j+1}}[z,x]\right|_{l^1}+\frac{1}{2}\sum_{j=0}^{m_{[z,x]}-1}\left|\beta'_{F_{[z,x],j+1}F_{[z,x],j}}[x,z] \right|_{l^1}\nonumber\\
=&\frac{1}{2}\left|\beta'[x,y]-\beta'[x,z]\right|_{\cF_x(y,z),l^1}+\frac{1}{2}\left|\beta'[y,x]-\beta'[z,x]\right|_{\cF_x(y,z),l^1}\nonumber\\
&+\frac{1}{2}\left|\beta'[x,y]-\beta'[z,y]\right|_{\cF_y(x,z),l^1}+\frac{1}{2}\left|\beta'[y,x]-\beta'[y,z]\right|_{\cF_y(x,z),l^1}\nonumber\\
&+\frac{1}{2}\left|\beta'[y,z]-\beta'[x,z]\right|_{\cF_z(x,y),l^1}+\frac{1}{2}\left|\beta'[z,y]-\beta'[z,x]\right|_{\cF_z(x,y),l^1} \nonumber\\
&+m_{[x,y]}+m_{[y,z]}+m_{[z,x]} \nonumber\\
\leq&3\cC_4+18.
\end{align}
By \eqref{C5 est-2} and \eqref{C5 est-3}, we can choose
$$\cC_5=\max\{1+\cC_4,3\cC_4+18\}=3\cC_4+18$$
and the corollary follows. (Dependence on $\rho_0$ and $\epsilon_0$ follows from Proposition \ref{prop of beta'}.)
\end{proof}

\section{Combinatorics of simplices and $\Theta$-separations}\label{sec combinatorics}
Despite the technicality of this section, the underlying idea is much simpler. Using the notion of $\Theta(\cdot,\cdot)$, the combinatorics of ``almost separation'' of a barycentric simplex by elements in $\Gamma F$ is overall very similar to the combinatorics of ``actual separation''. There are a few difference which requires extra care, but eventually they do not cause any essential trouble for the final proof of Theorem \ref{thm:main}.

\subsection{Motivations: actual separation of simplices by elements in $\Gamma F$}\label{subsec actual sep}
In this subsection, we briefly explain the combinatorics of actual separation of a barycentric simplex by elements in $\Gamma F$. This should be viewed as a toy model for the discussion in subsequent subsections. The goal of this subsection is to improve readability and does not include any serious mathematical statements that contribute to the results of this paper.

Also, in this subsection only, we will use the same notations for the combinatorics of ``actual separation'' as their counterparts in the combinatorics of ``almost separation'' introduced in the rest of this section.

\textbf{Combinatorics of ``actual separation'':} Fix some $y_0\in F$. Let $V=\{p_0,...,p_k\}$ be a subset of $\Gamma y_0$ with cardinality $k+1$ and $\sigma$ be a $k$-dimensional barycentric simplex with vertex set $V$. We assume that
\begin{enumerate}
\item[(1)] $F_{p_0},...,F_{p_k}$ are pairwise disjoint. Moreover, $F_{p_j}$ intersects the $1$-skeleton of $\sigma$ only at $p_j$ for any $j\in\{0,...,k\}$.
\end{enumerate}
Let $\cF(V)$ be all elements in $\Gamma F$ which intersect the $1$-skeleton of $\sigma$ and $\cA(V)$ be all elements in $\Gamma F$ which either separate the $1$-skeleton of $\sigma$ into two disjoint connected pieces or visit a vertex of $\sigma$. We observe the following properties:
\begin{enumerate}
\item[(2)] $\cF(V)=\cA(V)$.
\item[(3)] For any $\hF\in\cA(V)$, there exists a unique non-empty subset $I\subset V$ such that $\hF$ intersects all edges of $\sigma$ connecting a point in $I$ and a point in $V\setminus I$. We call $\stype{I}{V}$ the (unique) separation type of $\hF$ with respect to $V$.
\item[(4)] For any $\hF\in\cA(V)$ and any edge $[x,y]$ of $\sigma$ with intersects $\hF$, if $\stype{I}{V}$ is the (unique) separation type of $\hF$, then $|\{x,y\}\cap I|=1$.
\item[(5)] For any $W\subset V$ and any $\hF\in\cA(W)$, we have $\hF\in\cA(V)$. Moreover, if $\stype{I}{V}$ is the (unique) separation type of $\hF$ with respect to $V$, then $\{I\cap W,W\setminus I\}$ is the (unique) separation type of $\hF$ with respect to $W$.
\end{enumerate}
Elements in $\cA(V)$ divide $X$ into several connected components. Let $\bG_V$ be a graph defined as follows:
\begin{itemize}
\item Vertex set of $\bG_V$ is $\cA(V)$;
\item For any distinct $F_1,F_2\in\cA(V)$, $F_1$ and $F_2$ are connected by an edge if there are no elements in $\cA(V)$ which are strictly between $F_1$ and $F_2$.
\end{itemize}
Then after dividing $X$ using elements in $\cA(V)$, connected components of $X\setminus \left(\bigcup_{\hF\in\cA(V)}\hF\right)$ which intersects $\sigma$ corresponds to maximal complete subgraphs of $\bG_V$. One can also easily verify the following:
\begin{enumerate}
\item[(\hypertarget{MCS--1}{MCS-1})] Every vertex is shared by at most two maximal complete subgraphs of $\bG_V$. Moreover, $\{F_{p}\}_{p\in V}$ is all those vertices which are contained in only one maximal complete subgraph of $\bG_V$;
\item[(\hypertarget{MCS--2}{MCS-2})] Every edge is contained in a unique maximal complete subgraph of $\bG_V$.
\end{enumerate}
One can image that $\cA(V)$ cut the simplex $\sigma$ into ``polygonal pieces'', similar to the way a number of hyperplanes in $\RR^n$ cut a generic $k$-dimensional geodesic simplex in $\RR^n$ into polygonal pieces, with the additional assumption that these hyperplanes do not intersect each other in the interior of the simplex. The ``polygonal pieces'' then corresponds to maximal complete subgraphs of $\bG_V$.

For any $W\subset V$, we let $\cA(V;W)=\cA(W)\cap\cA(V)$ and let $\bG_{V;W}$ be the restriction of $\bG_V$ onto $\cA(V;W)$. Then one can easily verify the following:
\begin{enumerate}
\item[(6)] $\cA(V;W)=\cA(W)$. Moreover, $\bG_{W}=\bG_{V;W}$;
\item[(7)] For any maximal complete subgraph $\GV$ in $\bG_V$ such that $\GV\cap\bG_{V;W}\neq\emptyset$, $\GV\cap\bG_W$ is the unique maximal complete subgraph in $\bG_W$ contained in $\GV$.
\end{enumerate}

\textbf{Switching to ``almost separation'':} Let $V=\{p_0,...,p_k\}$ be a subsect of $\Gamma x_0$ with cardinality $k+1$ and $\sigma$ be a $k$-dimensional barycentric simplex with vertex set $V$.

The goal of this section is to provide the ``almost separation'' version of the above story with the help of $\Theta(\cdot,\cdot)$. For simplicity, we use the terminology $\Theta$-separation for ``almost separation'' involving $\Theta(\cdot,\cdot)$. With a $\Theta$-separation version of $\cF(V)$ and $\cA(V)$, the overall structure of the story for ``almost separation'' is very similar to the above discussion. However, comparing to (1)-(7) in the above, we have the following concerns:
\begin{enumerate}
\item We do not assume that $F_{p_0},...,F_{p_k}$ are pairwise disjoint any more.
\item There might be cases where $\cA(V)\subsetneq \cF(V)$.
\item There can be more than one $\Theta$-separation type of $\hF\in\cA(V)$ with respect to $V$.
\item Let $\hF\in\cA(V)$. It might be possible that there exists an edge $[x,y]$ of $\sigma$ which almost intersects $\hF$ such that for any $\Theta$-separation type $\stype{I}{V}$ of $\hF$, we have $|\{x,y\}\cap I|\neq 1$.
\item Let $W\subset V$ and $\hF\in\cA(W)$. Then there might be cases where $\hF\not\in\cA(V)$. Since $\Theta$-separation types are not necessarily unique, later in this section, we introduce a notion of ``good'' separation types. Unfortunately, if $\stype{I}{V}$ is a ``good'' separation type of $\hF$ with respect to $V$ such that $I\cap W\neq \emptyset$ and $W\setminus I\neq\emptyset$, $\{I\cap W,W\setminus I\}$ may not be a ``good'' separation type for $\hF$ with respect to $W$.
\item Let $W\subset V$. It might be possible that $\cA(V;W)\subsetneq\cA(W)$. Hence, even with the modified definitions. $\bG_W$ is not necessarily a subgraph of $\bG_V$.
\item With the modified definitions, let $\GV$ be a maximal complete subgraph in $\bG_V$ such that $\GV\cap\bG_{V;W}\neq\emptyset$. Then we have the following:
\begin{enumerate}
\item[(7a)] The number of maximal complete subgraph in $\bG_W$ which ``subordinates'' to $\GV$ can be $0$;
\item[(7b)] The number of maximal complete subgraph in $\bG_W$ which ``subordinates'' to $\GV$ can be more than $1$.
\end{enumerate}
\end{enumerate}

\textbf{Outline of the rest of this section}: Preparations for defining $\bG_V$ are in Subsection \ref{subsect:graph prep}. Subsection \ref{subsect:GV and its properties} discusses the definition of  $\bG_V$ along with its properties. In particular, \hyperlink{MCS--1}{(MCS-1)} and \hyperlink{MCS--2}{(MCS-2)} are proved in Proposition \ref{key prop of ASep graph} and Corollary \ref{remaining properties of sep graph}. We study the relations between $\bG_V$ and $\bG_W$ when $W\subset V$ in Subsection \ref{subsec most annoying}.

In all cases except one special case ($|V|=2$ and $|\cF(V)|=1$), the main changes in the vertex set of $\bG_V$ compared to the version for ``actual separation'' are the following:
\begin{itemize}
\item We throw away certain ``bad'' elements in $\cA(V)$ due to the concerns in (4). (See Notation \ref{sep type marking}.)
\item We pair the remaining ``good'' elements in $\cA(V)$ with ``good'' separation types due to the concerns in (1) and (3). (See Lemma \ref{prim decomp}, Notation \ref{sep type marking} and Definition \ref{bdry}.)
\end{itemize}
One advantage of the above is that for any $W\subset V$ and any ``good'' $\hF$ in $\cA(V)$, if $\stype{I}{V}$ is a ``good'' separation type of $\hF$ for $V$ such that $I\cap W\neq \emptyset$ and $W\setminus I\neq\emptyset$, we have $\hF$ is also a``good'' element in $\cA(W)$ and $\{I\cap W,W\setminus I\}$ is also a ``good'' separation type of $\hF$ for $W$. This addresses the concern in (5). (See Lemma \ref{res of vertices}.)

Thanks to Lemma \ref{almost all good}, there are at most finitely many elements in $\cF(V)$ which are not in $\cA(V)$ or are  ``bad'' elements in $\cA(V)$. Hence, the concerns in (2) and throwing away ``bad'' elements in $\cA(V)$ do not cause any harm in our construction of ``straightened simplex'' satisfying the first bullet point in (S2). (See \hyperlink{idea-and-plan}{the outline and plan of the proof}.)

Due to the changes in the vertex set of $\bG_V$, the definition of edges becomes much more technical. This is because not only we need to consider whether two elements in $\cA(V)$ are ``adjacent'', but also we need to check whether their separation types are ``adjacent'' as well. This is explained in Definition \ref{graph of sep} and Lemma \ref{reinterpretation of edges}. Preparations for these discussions can be found in Definition \ref{relations between sep types}, Definition \ref{in between for ASep} and Lemma \ref{key lem for ASep}.

For any $W\subset V$, the relations between $\bG_W$ and $\bG_V$ also become much more sophisicated due to the changes in the vertex set of $\bG_V$. (Roughly speaking, $\bG_W$ is similar to $\bG_{V;W}$, a subgraph of $\bG_V$, with possibly more ``details''.) In view of (6), we introduce the restriction map (in Definition \ref{res of types}), the notion of ``subordination'' (in Definition \ref{subordinate}) and the notion of ``enrichment'' (in Definition \ref{enrich}) to study the relations between $\bG_W$ and $\bG_V$. The last two notions rely on Lemma \ref{properties of W-face} which is highly technical. The concerns in (7a) are addressed in Definition \ref{singular} and Lemma \ref{singular case}. Lemma \ref{easy properties of enrich} and Lemma \ref{combinatorics behind d2=0} discuss the properties of ``subordination'' and ``enrichment''. In particular, they address the concerns in (7b) and play a crucial role in Section \ref{sec hom arg}.

\subsection{Preparations}\label{subsect:graph prep}
In Lemma \ref{properties of Theta}, \hyperlink{Theta-4}{property ($\Theta$4) of $\Theta(\cdot,\cdot)$}, we introduced subsets $\cF(\cdot)$,$\cA(\cdot)$ of $\Gamma F$ as certain elements in $\Gamma F$ related to the vertex set of a geodesic $2$-simplex with vertices in $\Gamma x_0$. We would like to first generalize the story to geodesic $k$-simplices for arbitrary $k\geq 2$.
\begin{definition}[$\Theta$-separation]\label{Theta sep}
For any finite subset $V$ of $\Gamma x_0$ with $|V|\geq 2$ (see Notation \ref{Fx def}), we define $\cF(V)=\union_{x,y\in V}\Theta(F_x,F_y)$. An element $\widehat F\in\cF(V)$ is \emph{$\Theta$-separating} $V$ if there exist $\emptyset\neq I\subsetneq V$ such that
$$\widehat F\in\bigcap_{x\in I,y\in V\setminus I}\Theta(F_x,F_y).$$
In this case, $\stype{I}{V}$ is called a \emph{$\Theta$-separation type of $\widehat F$ with respect to $V$}. Denoted by $\Sep_V(\widehat F)\subset\{\stype{I}{V}|\emptyset\neq I\subsetneq V\}$ the collection of all $\Theta$-separation type of $\widehat F$ with respect to $V$ and $\cA(V)$ the collection of all $\widehat F\in\cF(V)$ which are $\Theta$-separating $V$.
\end{definition}
\begin{rmk}
One should think about $\cF(V)$ as the collection of elements in $\Gamma F$ which can be arbitarily close to any geodesic simplex with vertex set equal to $V$. $\Theta$-separation is a subtle refinement of almost separation introduced in Definition \ref{almost sep} for technical purposes. Details are explained in the subsequent Lemma \ref{properties of Theta sep}.
\end{rmk}
\begin{notation}[Elements in $\cF(V)$ with edge marking]\label{edge marking}
For any finite subset $V$ of $\Gamma x_0$ with $|V|\geq 2$, we define
$$\widetilde \cF(V)=\{(\widehat F,\{x,y\})|x\neq y\in V, \widehat F\in\Theta(F_x, F_y)\}$$
and
$$\widetilde \cA(V)=\{(\widehat F,\{x,y\})|\widehat F\in\cA(V), \exists \stype{I}{V}\in\Sep_V(\widehat F) \mathrm{~s.t.~}x\in I, y\in V\setminus I\}.$$
\end{notation}
\begin{pexample}[Possible $(\widehat F,\{x,y\})\in\widetilde\cF(V)\setminus\widetilde\cA(V)$ and $\widehat F\in\cA(V)$]
Let $V=\{x_1,...,x_4\}$ and $\widehat F\in \Gamma F$ such that
$$\widehat F\in\Theta(F_{x_1},F_{x_4})\ints\Theta(F_{x_2},F_{x_4})\ints\Theta(F_{x_3},F_{x_4})\ints\Theta(F_{x_2},F_{x_3}),~\mathrm{and}~\widehat F\not\in\Theta(F_{x_1},F_{x_2})\union\Theta(F_{x_1},F_{x_3}).$$
Then $\widehat F\in\cA(V)$ and $\Sep_V(\widehat F)=\left\{\{\{x_4\},\{x_1,x_2,x_3\}\}\right\}$. However, $(\widehat F,\{x_2,x_3\})\in\widetilde\cF(V)\setminus\widetilde\cA(V)$. In Figure \ref{ex6.3}, one should expect that $[x_2,x_3]$ is not extremely close to $\hF$, but it is close enough to $\hF$ such that $\hF\in\Theta(F_{x_2},F_{x_3})$; $[x_2,x_1]$ and $[x_3,x_1]$ are not very far away from $\hF$, but they are far enough away from $\hF$ such that $\hF\not\in\Theta(F_{x_1},F_{x_3})\cup\Theta(F_{x_2},F_{x_1})$.
\end{pexample}

\begin{figure}[h]
	\centering
	\includegraphics[width=4in]{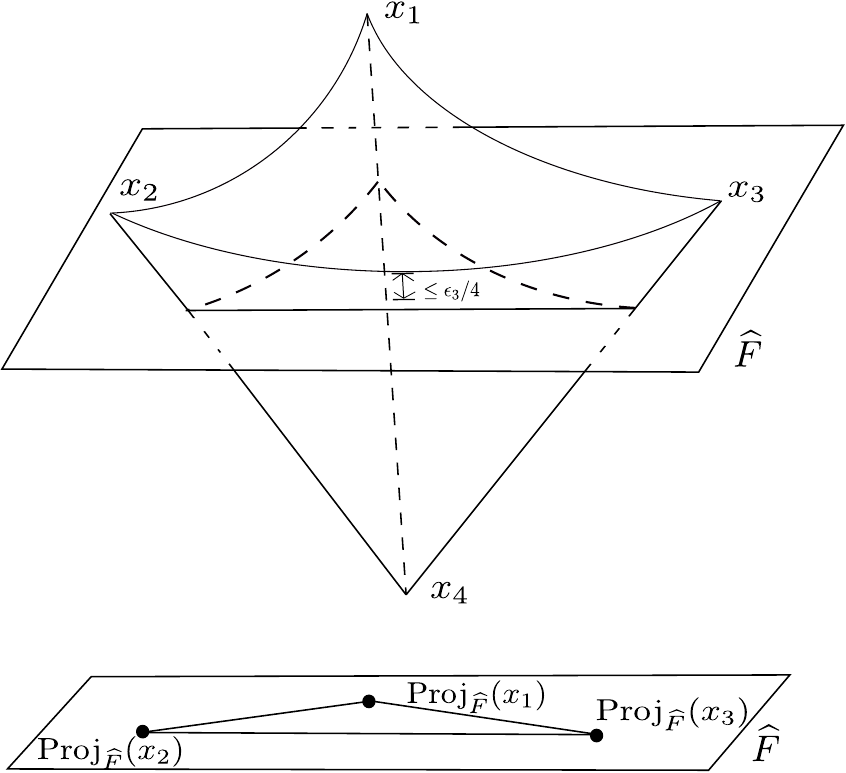}
	\caption{ \label{ex6.3}}
\end{figure}

\begin{lemma}\label{properties of Theta sep}
For any $k\geq 2$ and $V\subset \Gamma x_0$ with $V=\{x_0,...,x_k\}$, the following holds:
\begin{enumerate}
\item[\hypertarget{Sep-1}{(1).}] For any $\widehat F\not\in\cF(V)$, $d(\hF,\Db^k(x_0,...,x_k))\geq c_3(\epsilon_4/4)$;
\item[\hypertarget{Sep-2}{(2).}] For any $\hF\in\cA(V)\setminus\{F_x|x\in V\}$, $\hF$ is $\epsilon_3/4$-almost separating $\Db^k(x_0,...,x_k)$. Moreover, for any $\stype{I}{V}\in\Sep_V(\hF)$, $\stype{I}{V}$ is a $\epsilon_3/4$-almost separation type of $\hF$ with respect to $\Db^k(x_0,...,x_k)$;
\item[\hypertarget{Sep-3}{(3).}] For any distinct $F_j\in\cA(V)$ and $\stype{I_j}{V}\in\Sep_V(F_j)$, $j=1,2$, there exist $I_j'\in\stype{I_j}{V}$, $j=1,2,$ such that $I_1'\ints I_2'=\emptyset$.
\item[\hypertarget{Sep-4}{(4).}] Let $\widetilde\pi:\widetilde\cF(V)\to\cF(V)$ be the natural projection onto the first component. Then $\widetilde\pi(\widetilde\cF(V))=\cF(V)$ and $\widetilde\pi(\widetilde\cA(V))=\cA(V)$. Moreover
$$|\cF(V)\setminus\cA(V)|\leq|\widetilde\cF(V)\setminus\widetilde\cA(V)|\leq 3\cdot\comb{k+1}{3}\cdot\comb{k+1}{2},$$
where $\comb{m}{k}:=\frac{m!}{k!(m-k)!}$.
\end{enumerate}
\end{lemma}
\begin{proof}
\begin{enumerate}
\item[(1).] {This follows directly from Lemma \ref{bar: away from edge, away from simplex} and the definition of $\Theta(\cdot,\cdot)$ before Lemma \ref{properties of Theta}.}
\item[(2).] This is a direct corollary of Lemma \ref{properties of Theta}, \hyperlink{Theta-1}{property ($\Theta$1) of $\Theta(\cdot,\cdot)$} applied to every pair $x\neq y\in V$ such that $\hF\in\Theta(F_x, F_y)\setminus\{F_x, F_y\}$.
\item[(3).] This is a direct corollary of the above assertion (the second assertion in Lemma \ref{properties of Theta sep}) and Corollary \ref{almost ints type property}.
\item[(4).] $\widetilde\pi(\widetilde\cF(V))=\cF(V)$ and $\widetilde\pi(\widetilde\cA(V))=\cA(V)$ follow directly from definitions. The rest of this assertion largely relies on the following claim:

\emph{For any $(\hF,\{x,y\})\in\widetilde\cF(V)\setminus\widetilde\cA(V)$, there exist distinct points $x',y',z'\in V$ such that $\hF\in\cF(x',y',z')\setminus\cA(x',y',z')$.}

If the claim is false, let $(\hF,\{x,y\})\in\widetilde\cF(V)\setminus\widetilde\cA(V)$ such that for any distinct points $x',y',z'\in V$,$\hF\in\cA(x',y',z')$. We choose $I:=\{p\in V|\hF\not\in\Theta(F_x,F_p)\}\union\{x\}$. Clearly, $x\in I$, $y\in V\setminus I$ and $\hF\in\Theta(F_q, F_x)$ for any $q\in V\setminus I$. For any $p\in I\setminus\{x\}$ and $q\in V\setminus I$, $\hF\not\in\Theta(F_x, F_p)$ and $\hF\in\cA(x,p,q)$ imply that $\hF\in \Theta(F_q, F_p)$. Hence $\stype{I}{V}\in\Sep_V(\hF)$ with $x\in I$ and $y\in V\setminus I$. This implies that $(\hF,\{x,y\})\in\widetilde\cA(V)$, contradictory to the assumption that $(\hF,\{x,y\})\in\widetilde\cF(V)\setminus\widetilde\cA(V)$. Therefore the above claim holds.

By the above claim, there exists a map
$$\Phi:\widetilde\cF(V)\setminus\widetilde\cA(V)\to\bigsqcup_{J\subset V,|J|=3}(\cF(J)\setminus\cA(J))\times\{J\}$$
such that for any $(\hF,\{x,y\})\in \widetilde\cF(V)\setminus\widetilde\cA(V)$, $\Phi(\hF,\{x,y\})=(\hF,J)$ with $\hF\in\cF(J)\setminus\cA(J)$ for some appropriate choice of $J$. (The claim guarantees that there is always at least one choice of $J$. Such a map is not necessarily unique.) By the definition of $\Phi$, any point in $\bigsqcup_{J\subset V,|J=3|}(\cF(J)\setminus\cA(J))\times\{J\}$ has at most $\comb{k+1}{2}$ pre-images. By Lemma \ref{properties of Theta}, \hyperlink{Theta-4}{property ($\Theta$4) of $\Theta(\cdot,\cdot)$}, for any $J\subset V$ with $|J|=3$, $|\cF(J)\setminus\cA(J)|\leq 3$. Therefore
$$|\widehat\cF(V)\setminus\widehat\cA(V)|\leq\comb{k+1}{2}\cdot\left|\bigsqcup_{J\subset V,|J=3|}(\cF(J)\setminus\cA(J))\times\{J\}\right|=3\cdot\comb{k+1}{3}\cdot\comb{k+1}{2}.$$
\end{enumerate}
\end{proof}
\begin{rmk}
As a direct corollary of the third assertion, for any distinct $F_j\in\cA(V)$ and $\stype{I_j}{V}\in\Sep_V(F_j)$, $j=1,2$, the following holds:
\begin{itemize}
\item There exist $I_j'\in\stype{I_j}{V}$, $j=1,2$, such that $I_1'\cap I_2'\neq \emptyset$ and $(V\setminus I_1')\cap (V\setminus I_2')\neq\emptyset$.
\item If $I_1\cap I_2\neq \emptyset$ and $(V\setminus I_1)\cap (V\setminus I_2)\neq\emptyset$, then either $I_1\subset I_2$ or $I_2\subsetneq I_1$.
\end{itemize}
\end{rmk}

\begin{lemma}[Decomposition of $\Theta$-separation types]\label{prim decomp}
Let $V\subset \Gamma x_0$ with $|V|\geq 2$. For any $\hF\in\cA(V)$ and $p\in V$, we define
$$H_\hF^V(p)=\bigcap_{\substack{I:\stype{I}{V}\in\Sep_V(\hF)\\I\ni p}}I.$$
Then we have the following properties:
\begin{enumerate}
\item[(1).] For any $p,q\in V$, either $H_\hF^V(p)\ints H_\hF^V(q)=\emptyset$, or $H_\hF^V(p)= H_\hF^V(q)$;
\item[(2).] There exist disjoint subsets $H_1,...,H_m$ of $V$ such that
\begin{enumerate}
\item[(i).] $H_j=H_\hF^V(p_j)$ for some $p_j\in V$;
\item[(ii).] $V=\sqcup_{j=1}^mH_j$;
\item[(iii).] $\Sep_V(\hF)$ satisfies
$$\Sep_V(\hF)=\left\{\left.\left\{\bigsqcup_{j\in J}H_j,\bigsqcup_{j\in\{1,...,m\}\setminus J}H_j\right\}\right|\emptyset\neq J\subsetneq\{1,...,m\}\right\}.$$
\end{enumerate}
\end{enumerate}
\end{lemma}
\begin{proof}
\begin{enumerate}
\item[(1).] If there exists $x\in H_\hF^V(p)\ints H_\hF^V(q)$, then for any $\Theta$-separation type $\stype{I}{V}$ of $\hF$ such that $q\in I$, $x\in H_\hF^V(q)\subset I$. Since $x\in H_\hF^V(p)$, $p\in I$. (Otherwise, $p\in V\setminus I$ implies that $x\in H_\hF^V(p)\subset V\setminus I$. This contradicts $x\in I\ints H_\hF^V(p)$.) Therefore, $H_\hF^V(p)\subset I$. By arbitariness of $I$ which contains $q$, we have
$$H_\hF^V(p)\subset \left(\bigcap_{\substack{I:\stype{I}{V}\in\Sep_V(\hF)\\I\ni q}}I\right)=H_\hF^V(q).$$
Similarly $H_\hF^V(q)\subset H_\hF^V(p)$. Therefore $H_\hF^V(q)= H_\hF^V(p)$.
\item[(2).] Let $\cH:=\{H_\hF^V(p)|p\in V\}=\{H_1,...,H_m\}$. Then by the first assertion, $H_j=H_\hF^V(p_j)$ for some $p_j\in V$ and $V=\sqcup_{j=1}^mH_j$. For any $p\in V$, any $x\in H_\hF^V(p)$ and any $y\in V\setminus  H_\hF^V(p)$, there exists some $\stype{I}{V}\in\Sep_V(\hF)$ such that $H_\hF^V(p)\subset I$ and $y\not\in I$. Therefore $\hF\in\Theta(F_y, F_x)$. As a corollary, $\stype{H_\hF^V(p)}{V}\in \Sep_V(\hF)$, for any $p\in V$. In particular, this proves that
$$\Sep_V(\hF)\supset\left\{\left.\left\{\bigsqcup_{j\in J}H_j,\bigsqcup_{j\in\{1,...,m\}\setminus J}H_j\right\}\right|\emptyset\neq J\subsetneq\{1,...,m\}\right\}.$$
On the other hand, for any $\stype{I}{V}\in \Sep_V(\hF)$, $I=\cup_{p\in I}H_\hF^V(p)$. By the first assertion, there exists $\emptyset\neq J\subsetneq\{1,...,m\}$ such that $I=\sqcup_{j\in J}H_j$. This proves
$$\Sep_V(\hF)\subset\left\{\left.\left\{\bigsqcup_{j\in J}H_j,\bigsqcup_{j\in\{1,...,m\}\setminus J}H_j\right\}\right|\emptyset\neq J\subsetneq\{1,...,m\}\right\}.\qedhere$$
\end{enumerate}
\end{proof}

\begin{definition}[Primitive $\Theta$-separation type]\label{prim sep type}
For any finite subset $V\subset \Gamma x_0$ with $|V|\geq 2$ and any $\hF\in\cA(V)$, a $\Theta$-separation type is \emph{primitive} if $I=H_\hF^V(p)$ for some $p\in I$ or $V\setminus I= H_\hF^V(q)$ for some $q\in V\setminus I$.

We denote by $\PSep_V(\hF)$ the collection of primitive $\Theta$-separation types of $\hF$ with respect to $V$.
\end{definition}

The above constructions are ``$\Gamma$-equivariant'' in the following sense.
\begin{lemma}\label{gamma shift-1}
For any $\gamma\in\Gamma$ and any finite subset $V$ of $\Gamma x_0$ with $|V|\geq 2$, the map $\hF\to\gamma\hF$ satisfies the following.
\begin{itemize}
\item It gives a bijection from $\cF(V)$ to $\cF(\gamma V)$;
\item It gives a bijection from $\cA(V)$ to $\cA(\gamma V)$.
\end{itemize}
Hence the map $(\hF,\{p,q\})\to(\gamma\hF,\{\gamma p,\gamma q\})$ satisfies the following.
\begin{itemize}
\item It gives a bijection from $\widetilde\cF(V)$ to $\widetilde\cF(\gamma V)$;
\item It gives a bijection from $\widetilde\cA(V)$ to $\widetilde\cA(\gamma V)$.
\end{itemize}
Moreover, for any $\hF\in\cA(V)$ and any $\stype{I}{V}\in\Sep_V(\hF)$, the map $\stype{I}{V}\to\stype{\gamma I}{\gamma V}$ satisfies the following.
\begin{itemize}
\item It gives a bijection from $\Sep_V(\hF)$ to $\Sep_{\gamma V}(\gamma \hF)$;
\item It gives a bijection from $\PSep_V(\hF)$ to $\PSep_{\gamma V}(\gamma\hF)$.
\end{itemize}
\end{lemma}
\begin{proof}
By the definition of $\Omega(\cdot,\cdot)$, $\gamma\Omega(F',F'')=\Omega(\gamma F',\gamma F'')$ for any $F',F''\in\Gamma F$. Hence by the definition of $\Theta(\cdot,\cdot)$, $\gamma\Theta(F',F'')=\Theta(\gamma F',\gamma F'')$ for any $F',F''\in\Gamma F$. The rest of the proof follows directly from the corresponding definitions in Definition \ref{Theta sep}, Notations \ref{edge marking}, Lemma \ref{prim decomp} and Definition \ref{prim sep type}.
\end{proof}
\begin{definition}[Relations between $\Theta$-separation types]\label{relations between sep types}
Let $\stype{I_j}{V}\in\{\stype{I}{V}|\emptyset\neq I\subsetneq V\}$, $j=1,2,3$. We say that $\stype{I_3}{V}$ is \emph{between} $\stype{I_1}{V}$ and $\stype{I_2}{V}$ if there exist $ I_1'\in\stype{I_1}{V}$ and $ I_2'\in\stype{I_2}{V}$ such that
$$I_1'\subset I_3\subset I_2' ~\mathrm{or}~I_2'\subset I_3\subset I_1'.$$
We say $\stype{I_1}{V}$, $\stype{I_2}{V}$ and $\stype{I_3}{V}$ are in \emph{triangle relation} if there exist $I_j'\in\stype{I_j}{V}$, $j=1,2,3$, such that $I_1',I_2',I_3'$ are disjoint from each other.
\end{definition}
\begin{rmk}
Under the above notations, it is easy to see that for any $\gamma\in\Gamma$, $\stype{I_3}{V}$ is between $\stype{I_1}{V}$ and $\stype{I_2}{V}$ if and only if $\stype{\gamma I_3}{\gamma V}$ is between $\stype{\gamma I_1}{\gamma V}$ and $\stype{\gamma I_2}{\gamma V}$; $\stype{I_1}{V}$, $\stype{I_2}{V}$ and $\stype{I_3}{V}$ are in triangle relation if and only if $\stype{\gamma I_1}{\gamma V}$, $\stype{\gamma I_2}{\gamma V}$ and $\stype{\gamma I_3}{\gamma V}$ are in triangle relation.

{Also, by Lemma \ref{prim decomp}, for any $\hF\in\Gamma F$ and any distinct $\stype{I_1}{V},\stype{I_2}{V},\stype{I_3}{V}\in\PSep_V(F)$, we have $\stype{I_1}{V}$, $\stype{I_2}{V}$ and $\stype{I_3}{V}$ are in triangle relation. }
\end{rmk}

\begin{example}[Example of an in-between relation]
Let $V=\{p_0,p_1,p_2,p_3\}\in\Gamma x_0$ and pairwise distinct $F_1,F_2,F_3\in\cF(V)$ such that the following holds:
\begin{itemize}
\item $F_1$ intersects the geodesic segments $[p_0,p_i]$ for any $i\in\{1,2,3\}$;
\item $F_2$ intersects the geodesic segments $[p_i,p_j]$ for any $i\in\{0,1\}$ and any $j\in\{2,3\}$;
\item $F_3$ intersects the geodesic segments $[p_i,p_3]$ for any $i\in\{0,1,2\}$.
\end{itemize}
Then $\{\{p_0\},\{p_1,p_2,p_3\}\}\in\Sep_V(F_1)$, $\{\{p_0,p_1\},\{p_2,p_3\}\}\in\Sep_V(F_2)$ and $\{\{p_0,p_1,p_2\},\{p_3\}\}\in\Sep_V(F_3)$. Moreover, $\{\{p_0,p_1\},\{p_2,p_3\}\}$ is between $\{\{p_0\},\{p_1,p_2,p_3\}\}$ and $\{\{p_0,p_1,p_2\},\{p_3\}\}$ in the sense of Definition \ref{relations between sep types}. One can easily check that in this case, $F_1$ and $F_3$ are on opposite sides of $F_2$. (See Figure \ref{pic:in between}.)
\begin{figure}[h]
		\centering
		\includegraphics[scale=0.5]{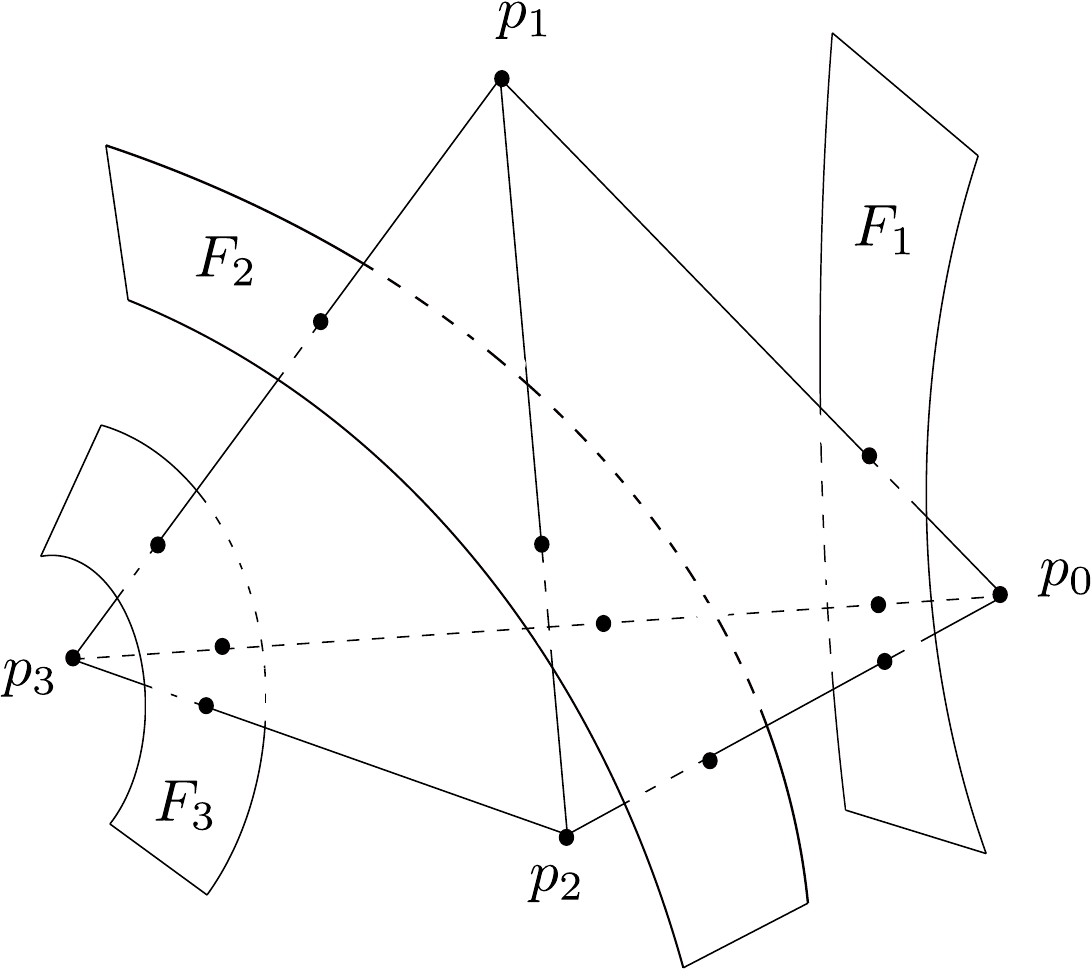}
		\caption{ \label{pic:in between}}
\end{figure}
\end{example}
\begin{example}[Example of triangle relation]
Let $V=\{p_0,p_1,p_2\}\in\Gamma x_0$ and pairwise distinct $F_0,F_1,F_2\in\cF(V)$ such that for any $i\in\{0,1,2\}$ and any $j\in V\setminus \{i\}$, $F_i$ intersects the geodesic segment $[p_i,p_j]$. Then $\{\{p_0\},\{p_1,p_2,\}\}\in\Sep_V(F_0)$, $\{\{p_1\},\{p_2,p_0\}\}\in\Sep_V(F_1)$ and $\{\{p_2\},\{p_0,p_1\}\}\in\Sep_V(F_2)$. Moreover, $\{\{p_0\},\{p_1,p_2,\}\}$, $\{\{p_1\},\{p_2,p_0\}\}$ and $\{\{p_2\},\{p_0,p_1\}\}$ are in triangle relation in the sense of Definition \ref{relations between sep types}. Assuming $F_3:=F_0$ and $F_4:=F_1$, one can easily check that for any $j\in\{0,1,2\}$, $F_{j+1}$ and $F_{j+2}$ are on the same side of $F_j$. (See Figure \ref{pic:trangle relation}.)
\begin{figure}[h]
		\centering
		\includegraphics[scale=0.5]{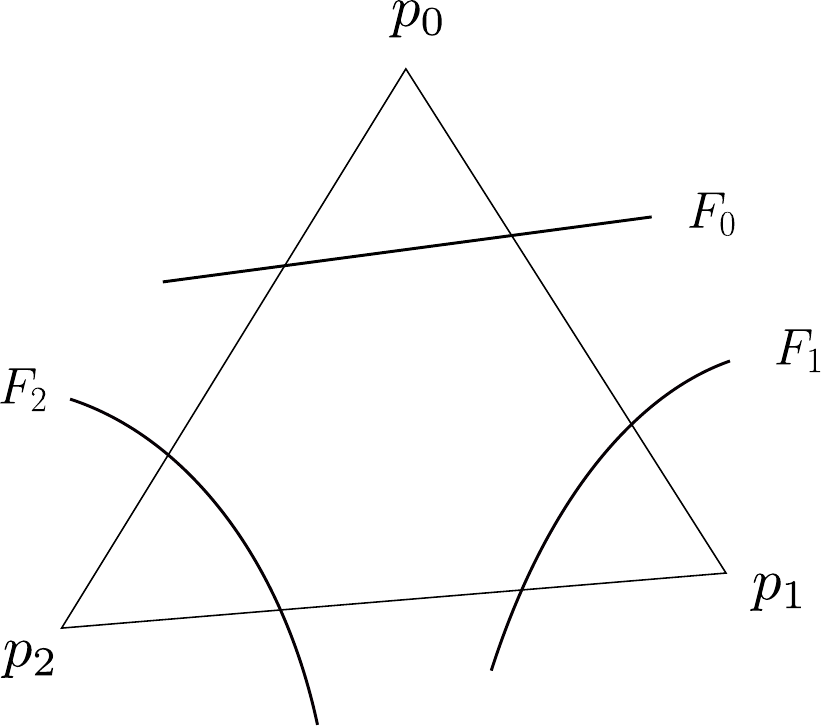}
		\caption{ \label{pic:trangle relation}}
\end{figure}
\end{example}
\begin{notation}[``Good'' elements in $\cF(V)$, $\Theta$-separation types as markings]\label{sep type marking}
Let $V\subset \Gamma F$ be a finite subset such that $|V|\geq 2$. Let $\cA_0(V)$ be the collection of ``good'' elements in $\cF(V)$ in the sense that
$$\cA_0(V)=\{\widehat F\in\cF(V)|(\widehat F,\{x,y\})\in\widetilde \cA(V),\forall (\widehat F,\{x,y\})\in\widetilde \cF(V) \}=\cF(V)\setminus\widetilde\pi\left(\widetilde\cF(V)\setminus\widetilde\cA(V)\right).$$

We define $$\cA_\Sep(V)=\{\typemark{\hF}{I}{V}|\hF\in\cA(V), \stype{I}{V}\in\Sep_V(\hF)\}$$
and
$$\cA_\PSep(V)=\{\typemark{\hF}{I}{V}|\hF\in\cA(V), \stype{I}{V}\in\PSep_V(\hF)\}.$$
Let
$$\AnSep(V)=\{\typemark{\hF}{I}{V}\in\cA_\Sep(V)|\hF\in\cA_0(V)\}$$
and
$$\AnPSep(V)=\{\typemark{\hF}{I}{V}\in\cA_\PSep(V)|\hF\in\cA_0(V)\}.$$
\end{notation}
\begin{rmk}
\begin{enumerate}
\item[(1).] If we think of elements in $(\widehat F,\{x,y\})\in\widetilde\cF(V)$ as ``the intersection point between $\widehat F$ and $[x,y]$'' and those in $\widetilde\cA(V)$ are considered ``good intersection point'', then ``good'' elements in $\cF(V)$ (i.e. elements in $\cA_0(V)$) do not have ``bad intersection points'' with edges of the geodesic simplex with vertex set $V$. Here,  the word ``good'', vaguely speaking, is used to extract those aspects in the story of $\Theta$-separation of simplices which are similar to the story of actual separation of simplices by elements in $\Gamma F$.
\item[(2).] We first consider the toy example of cutting a geodesic simplex in $\RR^n$ by a collection of finitely many hyperplanes in $\RR^n$. In addition, we assume the following:
\begin{itemize}
\item The hyperplanes do not intersect each other in the interior of the geodesic simplex;
\item For any vertex of the simplex, there is a unique hyperplane in this collection which passes through this vertex. Moreover, this hyperplane only intersects the simplex at this vertex. 
\end{itemize}
Then, the simplex is cut into several polygonal pieces. Each piece is determined by at least two hyperplanes. We will informally call these polygonal pieces ``chambers''. The faces of the polygonal pieces determined by hyperplanes are informally called ``floors''. Each hyperplane determines at most one face of any ``chamber''.

Here, we would like to imagine that elements in $\cA_0(V)$ also cut a barycentric simplex with vertex set $V$ in a similar way to the above toy example. Unfortunately, since $\Theta$-separation is a subtle version of almost separation, some bizzare phenomena can occur. For example, sometimes one $\Theta$-separating submanifold is enough to ``bound a chamber'' and this submanifold will show up in multiple faces of this ``chamber''. (See \eqref{1D chamber special}, \eqref{1D chamber} and \eqref{Chamber}.) Most ``floors'' are now described using the concept of $\AnPSep$. (See \eqref{0D floor special-1}, \eqref{0D floor special-2}, \eqref{0D floor-1}, \eqref{0D floor-2}, \eqref{floor}.) As is shown in the definition, one submanifold may contribute to many ``faces of chambers''. For example, in Figure \ref{fig:one flat many faces}, we consider some $V=\{x,y,z\}$ and some $\widehat F\in\cA_0(V)$ such that
$$\Sep_{V}(\widehat F)=\{(\{x\},\{y,z\}),(\{y\},\{x,z\}),(\{z\},\{x,y\})\}.$$
In other words, $\widehat F$ almost separates all edges of the triangle with vertices $\{x,y,z\}$. If we imagine that these almost separations are actual intersections (which is actually impossible), then $\widehat F$ alone ``bounds'' the ``shaded region''. The ``shaded region'' is a ``chamber'' bounded by three ``faces''. Since these ``faces'' all come from $\widehat F$ almost separating the triangle with vertices $\{x,y,z\}$, in order to distinguish these faces, we need to also mention the different $\Theta$-separation types of $\widehat F$.
\item[(3).] By Lemma \ref{gamma shift-1}, for any $\gamma\in\Gamma$, the map $\hF\to\gamma\hF$ gives a bijection from $\cA_0(V)$ to $\cA_0(\gamma V)$. Moreover, for any $\typemark{\hF}{I}{V}\in\cA_\Sep(V)$ and any $\gamma\in\Gamma$, we can naturally define $\gamma\typemark{\hF}{I}{V}=\typemark{\gamma \hF}{\gamma I}{\gamma V}$. Then by Lemma \ref{gamma shift-1} and the remark after Definition \ref{relations between sep types}, the map $\typemark{\hF}{I}{V}\to\typemark{\gamma \hF}{\gamma I}{\gamma V}$ satisfies the following
\begin{itemize}
\item It gives a bijection from $\cA_\Sep(V)$ to $\cA_\Sep(\gamma V)$.
\item It gives a bijection from $\cA_\PSep(V)$ to $\cA_\PSep(\gamma V)$.
\item It gives a bijection from $\AnSep(V)$ to $\AnSep(\gamma V)$.
\item It gives a bijection from $\AnPSep(V)$ to $\AnPSep(\gamma V)$.
\end{itemize}
\begin{figure}[h]
	\centering
	\includegraphics[width=4in]{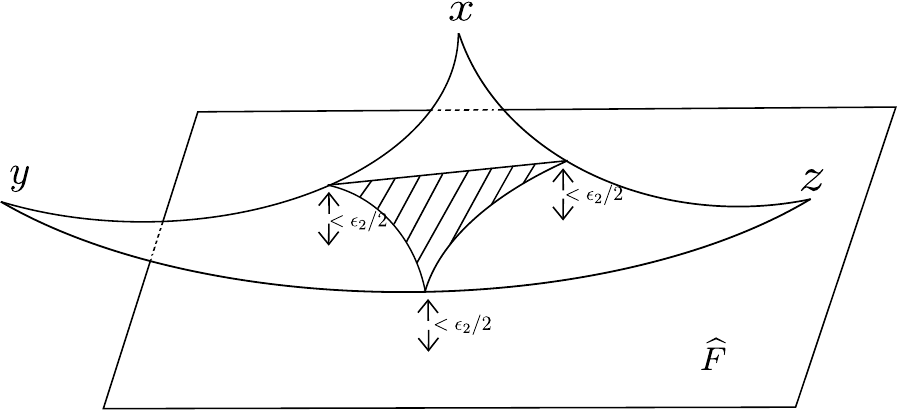}
	\caption{\label{fig:one flat many faces}}
\end{figure}

\end{enumerate}
\end{rmk}

\begin{lemma}\label{almost all good}
Let $V\subset \Gamma F$ be a finite subset such that $|V|=k+1\geq 2$. There are at most finitely many elements in $\cF(V)$ which are not ``good''. To be precise
$$|\cF(V)\setminus\cA_0(V)|\leq |\widetilde\cF(V)\setminus\widetilde\cA(V)|\leq 3\cdot\comb{k+1}{3}\cdot\comb{k+1}{2}.$$
Moreover, if we define $\widetilde\cA_0(V)=\{(\hF,\{x,y\})\in\widetilde\cF(V)|\hF\in\cA_0(V)\}$, then
$$|\widetilde\cF(V)\setminus\widetilde\cA_0(V)|\leq\comb{k+1}{2}\cdot|\cF(V)\setminus\cA_0(V)|\leq 3\cdot\comb{k+1}{3}\cdot\comb{k+1}{2}^2.$$
\end{lemma}
\begin{proof}
By the definition of $\cA_0(V)$ in Notation \ref{sep type marking}, for any $\hF\in\cF(V)\setminus\cA_0(V)$, one can choose some $\Phi_1(\hF)\in\widetilde\cF(V)\setminus\widetilde\cA(V)$ such that $\widetilde\pi(\Phi_1(\hF))=\hF$. Moreover, $\widetilde\pi\circ\Phi_1=\id_{\cF(V)\setminus\cA_0(V)}$ implies that $\Phi_1:\cF(V)\setminus\cA_0(V)\to\widetilde\cF(V)\setminus\widetilde\cA(V)$ is injective. Hence by the fourth assertion in Lemma \ref{properties of Theta sep}, we have
$$|\cF(V)\setminus\cA_0(V)|\leq |\widetilde\cF(V)\setminus\widetilde\cA(V)|\leq 3\cdot\comb{k+1}{3}\cdot\comb{k+1}{2}.$$
Notice that $\widetilde\pi(\widetilde\cF(V)\setminus\widetilde\cA_0(V))\subset\cF(V)\setminus\cA_0(V)$. Since any point in the image of $\widetilde\pi$ has at most $\comb{k+1}{2}$ preimages, hence
$$|\widetilde\cF(V)\setminus\widetilde\cA_0(V)|\leq\comb{k+1}{2}\cdot|\cF(V)\setminus\cA_0(V)|\leq 3\cdot\comb{k+1}{3}\cdot\comb{k+1}{2}^2.\qedhere$$
\end{proof}
\begin{definition}[``In between'' relations]\label{in between for ASep}
Let $V\subset \Gamma F$ be a finite subset such that $|V|\geq 2$. For any $\typemark{F_j}{I_j}{V}\in\cA_\Sep(V)$, $j=1,2,3$, we say that $\typemark{F_2}{I_2}{V}$ is \emph{between} $\typemark{F_1}{I_1}{V}$ and $\typemark{F_3}{I_3}{V}$ if $F_2\in\Theta(F_1,F_3)$ and $\stype{I_2}{V}$ is between $\stype{I_1}{V}$ and $\stype{I_3}{V}$. (See Definition \ref{relations between sep types}.)

When $|V|=2$ and $|\cF(V)|=1$, we assume WLOG that $V=\{p,q\}\subset\Gamma x_0$ for some distinct $p,q\in \Gamma x_0$ such that $F_p=F_q$. In this case, we define the ``in between'' relations on $V$ as follows:
\begin{itemize}
\item For any $x\in \{p,q\}$, $x$ is between $p$ and $q$.
\item For any $x\in\{p,q\}$, $x$ is between $x$ and $x$.
\item $p$ is not between $q$ and $q$; $q$ is not between $p$ and $p$.
\end{itemize}
\end{definition}
\begin{rmk}
The case when $|V|=2$ and $|\cF(V)|=1$ is a special case for Definition \ref{graph of sep}. In this case, one cannot find a injective map from $V$ to $\cA_\Sep(V)$ because $\cA_\Sep(V)$ contains only one element. If we do not treat this case in a special way and do the regular constructions involving $\cA_\Sep(V)$, many later arguments will become messier. This is why we define the ``in between'' relations when $|V|=2$ and $|\cF(V)|=1$ in a separate way.

By the remark after Definition \ref{relations between sep types} and the third remark after Notations \ref{sep type marking}, for any $\gamma\in\Gamma$, $\typemark{F_2}{I_2}{V}$ is between $\typemark{F_1}{I_1}{V}$ and $\typemark{F_3}{I_3}{V}$ if and only if $\gamma\typemark{F_2}{I_2}{V}$ is between $\gamma\typemark{F_1}{I_1}{V}$ and $\gamma\typemark{F_3}{I_3}{V}$.
\end{rmk}

\begin{lemma}\label{key lem for ASep}
Let $V\subset \Gamma F$ be a finite subset such that $|V|\geq 2$. The following holds:
\begin{enumerate}
\item[\hypertarget{ASep-1}{(1).}] For any $\typemark{F_j}{I_j}{V}\in\cA_\Sep(V)$, $j=1,2,3$, with $F_3\in\Theta(F_1, F_2)\setminus\{F_1,F_2\}$, either $\stype{I_1}{V}$, $\stype{I_2}{V}$ and $\stype{I_3}{V}$ are in triangle relation in the sense of Definition \ref{relations between sep types}, or $\typemark{F_3}{I_3}{V}$ is between $\typemark{F_1}{I_1}{V}$ and $\typemark{F_2}{I_2}{V}$ in the sense of Definition \ref{in between for ASep}.

Moreover, for any fixed $\typemark{F_j}{I_j}{V}\in\cA_\Sep(V)$, $j=1,2$, with some $F_3\in\cA_0(V)\ints(\Theta(F_1, F_2)\setminus\{F_1,F_2\})$, there always exists some $\stype{I_3}{V}\in\Sep_V(F_3)$ such that $\typemark{F_3}{I_3}{V}$ is between $\typemark{F_1}{I_1}{V}$ and $\typemark{F_2}{I_2}{V}$ in the sense of Definition \ref{in between for ASep}.
\item[\hypertarget{ASep-2}{(2).}] For any $\typemark{F_j}{I_j}{V}\in\cA_\Sep(V)$, $j=1,2,3$, with $F_1,F_2,F_3$ distinct, either $\stype{I_1}{V}$, $\stype{I_2}{V}$ and $\stype{I_3}{V}$ are in triangle relation in the sense of Definition \ref{relations between sep types}, or there exist some $l\in\{1,2,3\}$ such that $\typemark{F_l}{I_l}{V}$ is between the other two in the sense of Definition \ref{in between for ASep};
\item[\hypertarget{ASep-3}{(3).}] For any distinct $F_1, F_2\in \cA(V)$, there exist a unique $\stype{I_j}{V}\in \PSep_V(F_j)$, $j=1,2$, such that for any $\stype{I_j'}{V}\in \Sep_V(F_j)$, $j=1,2$, $\typemark{F_j}{I_j}{V}$ is between $\typemark{F_1}{I_1'}{V}$ and $\typemark{F_2}{I_2'}{V}$, $j=1,2$.
\end{enumerate}
\end{lemma}
\begin{proof}
\begin{enumerate}
\item[(1).]
By the third assertion in Lemma \ref{properties of Theta sep}, we can assume WLOG that $I_1\subset I_2$, $I_1\ints I_3\neq\emptyset$ and $(V\setminus I_1)\ints(V\setminus I_3)\neq\emptyset$. (This is because if $I_j'\in\stype{I_j}{V}$, $j=1,2$, satisfies $I_1'\ints I_2'=\emptyset$, then we can set $ I_1'=I_1$ and $V\setminus I_2'=I_2$. If $I_j''\in\stype{I_j}{V}$, $j=1,3$, satisfies $I_1''\ints I_3''=\emptyset$, then we set $I_3=V\setminus I_3''$ when $I_1''=I_1$ and $I_3=I_3''$ when $I_1''=V\setminus I_1$.) Hence either $I_1\subset I_3$ or $I_3\subsetneq I_1$.

\textbf{Case 1}: $I_3\subsetneq I_1\subset I_2$. Choose $x\in I_3$, $y\in I_1\setminus I_3$ and $z\in V\setminus I_2$. Then
$$F_3\in\Theta(F_x,F_y)\ints \Theta(F_x, F_z)~\mathrm{and}~ F_1,F_2\in\Theta(F_z,F_x)\ints\Theta(F_z,F_y).$$
\begin{figure}[h]
	\centering
	\includegraphics[scale=0.5]{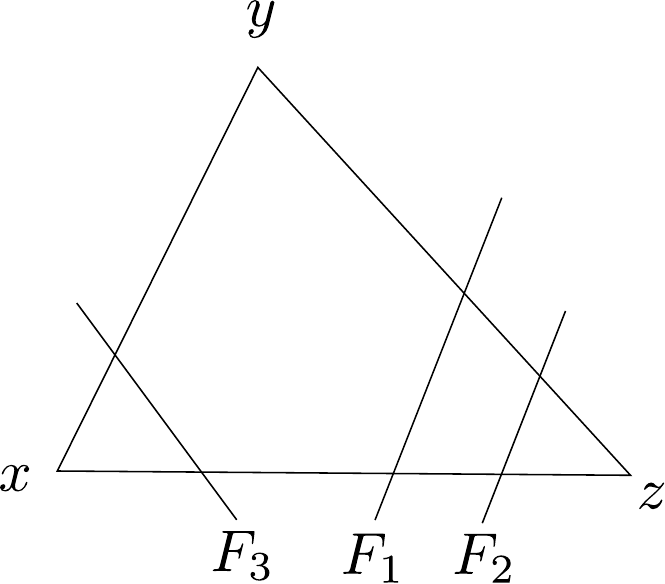}
	\caption{ \label{pic:key lem for ASep-1}}
\end{figure}
(See Figure \ref{pic:key lem for ASep-1}.) By the second assertion in Lemma \ref{properties of Theta sep},
$$d(F_3,[x,y]), d(F_3,[x,z]), d(F_1,[x,z]), d(F_2,[x,z]), d(F_1,[y,z]), d(F_2,[y,z])\leq \frac{\epsilon_0}{2}.$$
Let $p_j\in[x,z]$ such that $d(p_j, F_j)\leq \epsilon_0/2$, $j=1,2,3$. By Proposition \ref{almost ints positioning} applied to the above inequalities, we have $p_3\in[x,p_1]\ints[x,p_2]$. On the other hand, by the fact that $F_3\in\Theta(F_1,F_2)\setminus\{F_1,F_2\}$ and Lemma \ref{properties of Theta}, \hyperlink{Theta-3}{property ($\Theta$3) of $\Theta(\cdot,\cdot)$}, we have $p_1\in[x,p_3]$ or $p_2\in[x,p_3]$. Hence $p_3\in\{p_1,p_2\}$. This contradicts the fact that $F_3\not\in\{F_1,F_2\}$ and subsequently $d(F_3,F_1), d(F_3,F_2)>\rho_0\geq 3\epsilon_0.$ (See the definition of $\rho_0$ right before Proposition \ref{almost ints positioning}.) Therefore this case is impossible.

\textbf{Case 2}: $I_1\subset I_3$ and $I_1\subset I_2$. Then $I_1\ints(V\setminus I_3)=I_1\ints(V\setminus I_2)=\emptyset$. If $\stype{I_1}{V}$, $\stype{I_2}{V}$ and $\stype{I_3}{V}$ are \textbf{NOT} in triangle relation in the sense of Definition \ref{relations between sep types}, then $(V\setminus I_3)\ints(V\setminus I_2)\neq\emptyset$. Since $\emptyset\neq I_1\subset I_2\ints I_3$, by the remark after Lemma \ref{properties of Theta sep}, we have either $I_2\subsetneq I_3$ or $I_3\subset I_2$. In the second case, $I_1\subset I_3\subset I_2$ implies that $\stype{I_3}{V}$ is between $\stype{I_1}{V}$ and $\stype{I_2}{V}$ in the sense of Definition \ref{relations between sep types}. Therefore, by the assumption that $F_3\in\Theta(F_1,F_2)$, $\typemark{F_3}{I_3}{V}$ is between $\typemark{F_1}{I_1}{V}$ and $\typemark{F_2}{I_2}{V}$ in the sense of Definition \ref{in between for ASep}.

Now we assume that $I_1\subset I_2\subsetneq I_3$. Choose $I_j'=V\setminus I_j$, $j=1,2,3$. Then $I_3'\subsetneq I_2'\subset I_1'$. This reduces to \textbf{Case 1} after swtiching the order of $\typemark{F_1}{I_1}{V}$ and $\typemark{F_2}{I_2}{V}$, which is impossible. This completes the proof for the first half of the first assertion in this lemma.

If $F_3\in\cA_0(V)\ints(\Theta(F_1,F_2)\setminus\{F_1,F_2\})$, let $p\in I_1$ and $q\in V\setminus I_2$. Then $F_1,F_2\in\Theta(F_p,F_q)$ implies that $F_3\in\cA_0(V)\ints\Theta(F_p,F_q)$. By the definition of $\cA_0(V)$, $(F_3,\{p,q\})\in\widetilde\cA(V)$. Hence there exist $\stype{I_3}{V}\in\Sep_V(F_3)$ such that $p\in I_3$ and $q\in V\setminus I_3$. Therefore $\stype{I_1}{V}$, $\stype{I_2}{V}$ and $\stype{I_3}{V}$ are \textbf{NOT} in triangle relation in the sense of Definition \ref{relations between sep types}. By the first half of the first assertion in this lemma, $\typemark{F_3}{I_3}{V}$ is between $\typemark{F_1}{I_1}{V}$ and $\typemark{F_2}{I_2}{V}$ in the sense of Definition \ref{in between for ASep}. This finishes the second half of the first assertion in this lemma.

\item[(2).] Similar to the proof of the first assertion in this lemma, we can assume WLOG that $I_1\subset I_2$, $I_1\ints I_3\neq\emptyset$ and $(V\setminus I_1)\ints(V\setminus I_3)\neq\emptyset$ due to the third assertion in Lemma \ref{properties of Theta sep}. Hence either $I_1\subset I_3$ or $I_3\subset I_1$.

\textbf{Case 1}: $I_3\subset I_1\subset I_2$. Choose $p\in I_3$ and $q\in V\setminus I_2$, we have $F_1,F_2,F_3\in\Theta(F_p,F_q)$. By Lemma \ref{properties of Theta}, \hyperlink{Theta-3}{property ($\Theta$3) of $\Theta(\cdot,\cdot)$}, there exist some $l\in\{1,2,3\}$ such that $F_l\in\Theta(F_{l+1},F_{l+2})\setminus\{F_{l+1},F_{l+2}\}$, where $F_4:=F_1$ and $F_5:=F_2$. Since $\stype{I_1}{V}$, $\stype{I_2}{V}$ and $\stype{I_3}{V}$ are \textbf{NOT} in triangle relation in the sense of Definition \ref{relations between sep types}, by the first assertion of this lemma, $\typemark{F_l}{I_l}{V}$ is between $\typemark{F_{l+1}}{I_{l+1}}{V}$ and $\typemark{F_{l+2}}{I_{l+2}}{V}$ in the sense of Definition \ref{in between for ASep}, where $I_4:=I_1$ and $I_5:=I_2$. (See Figure \ref{pic:key lem for ASep-2}. Warning: In this case $l$ does not necessarily equal to $1$!)

\begin{figure}[h]
	\centering
	\includegraphics[scale=0.5]{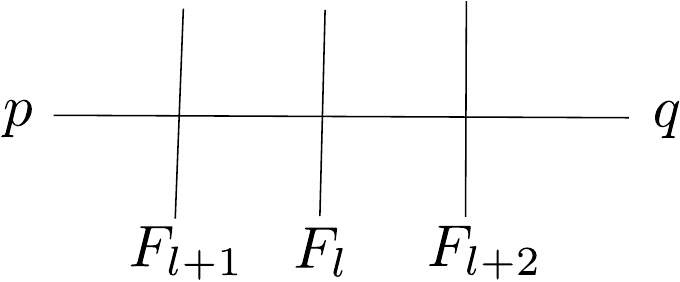}
	\caption{ \label{pic:key lem for ASep-2}}
\end{figure}

\textbf{Case 2}: $I_1\subset I_3$ and $I_1\subset I_2$. Then $I_1\ints(V\setminus I_3)=I_1\ints(V\setminus I_2)=\emptyset$. If $\stype{I_1}{V}$, $\stype{I_2}{V}$ and $\stype{I_3}{V}$ are \textbf{NOT} in triangle relation in the sense of Definition \ref{relations between sep types}, then $(V\setminus I_3)\ints(V\setminus I_2)\neq\emptyset$. Since $\emptyset\neq I_1\subset I_2\ints I_3$, by the remark after Lemma \ref{properties of Theta sep}, either $I_2\subset I_3$ or $I_3\subset I_2$.

\textbf{Case 2a}: $I_1\subset I_2\subset I_3$. This case is essentially the same as \textbf{Case 1}.

\textbf{Case 2b}: $I_1 \subset I_3\subset I_2$. This case is essentially the same as \textbf{Case 1}.

\item[(3).] This assertion is trivial if both $F_1$ and $F_2$ has a unique $\Theta$-separation type. Otherwise, there exists some $\stype{\hI_j}{V}\in\Sep_V(F_j)$, $j=1,2$, such that $\stype{\hI_1}{V}\neq\stype{\hI_2}{V}$. By the third assertion in Lemma \ref{properties of Theta sep}, we can assume WLOG that $\hI_1\subsetneq\hI_2$.

\textbf{Step 1}: For any $\stype{I_1'}{V}\in\Sep_V(F_1)$, we would like to show that either $I_1'\cap(V\setminus\hI_2)=\emptyset$ or $(V\setminus I_1')\cap(V\setminus\hI_2)=\emptyset$. As a corollary of this, we can assume WLOG that $I_1'\cap(V\setminus\hI_2)=\emptyset$.

By the third assertion in Lemma \ref{properties of Theta sep}, we assume WLOG that either $I_1'\cap(V\setminus\hI_2)=\emptyset$ or $I_1'\cap\hI_2=\emptyset$. It suffices to show that when $I_1'\cap\hI_2=\emptyset$, we have $(V\setminus I_1')\cap(V\setminus\hI_2)=\emptyset$.

If $I_1'\cap \hI_2=\emptyset$, then $\{I_1'\sqcup \hI_1,V\setminus(I_1'\sqcup \hI_1)\}\in\Sep_V(F_1)$. Notice that $(I_1'\sqcup \hI_1)\cap \hI_2= \hI_1\neq \emptyset$, that $(I_1'\sqcup \hI_1)\cap (V\setminus \hI_2)= I_1'\cap(V\setminus \hI_2)= I_1'\neq \emptyset$ and that $(V\setminus(I_1'\sqcup \hI_1))\cap\hI_2=(V\setminus \hI_1)\cap\hI_2\neq\emptyset$. By the fact that $F_1\neq F_2$, the fact that $\hI_1\subsetneq\hI_2$ and the third assertion in Lemma \ref{properties of Theta sep}, we have $(V\setminus(I_1'\sqcup \hI_1))\cap(V\setminus \hI_2)=(V\setminus I_1')\cap(V\setminus \hI_2')=\emptyset$. This finishes \textbf{Step 1}.

\textbf{Step 2}: For any $\stype{I_2'}{V}\in\Sep_V(F_2)$, since $\hI_1\subsetneq\hI_2\implies(V\setminus\hI_2)\subsetneq(V\setminus \hI_1)$, repeating the same arguments in \textbf{Step 1}, either $I_2'\cap\hI_1=\emptyset$ or $(V\setminus I_2')\cap\hI_1=\emptyset$. Therefore we can assume WLOG that $(V\setminus I_2')\cap \hI_1=\emptyset$ and hence $\hI_1\subset I_2'$.

\textbf{Step 3:} Construction of $I_1, I_2$.

Choose $I_2=H_{F_2}^V(p_2)$ for some $p_2\in \hI_1$ and $I_1=V\setminus(H_{F_1}^V(p_1))$ for some $p_1\in V\setminus\hI_2$. (See Lemma \ref{prim decomp} for the definition of $H_\hF^V(\cdot)$.) Notice that $I_2\ints \hI_1\supset\{p_2\}\neq\emptyset$, we have $\hI_1\subset I_2$ by \textbf{Step 2}. Similarly $(V\setminus I_1)\ints (V\setminus \hI_2)\supset \{p_1\}\neq\emptyset$ and hence $I_1\subset \hI_2$ by \textbf{Step 1}. Then, by Lemma \ref{prim decomp}, $I_1,I_2$ is independent of the choice of $p_1,p_2$. (To be specific, this is because $H_{F_1}^V(p_1)\supset (V\setminus \hI_2)\neq\emptyset$ and $H_{F_2}^V(p_2)\supset \hI_1\neq\emptyset$. The first assertion in Lemma \ref{prim decomp} implies that different choices of $p_1\in V\setminus \hI_2$ and $p_2\in\hI_1$ will lead to the same $H_{F_1}^V(p_1)$ and $H_{F_2}^V(p_2)$.) Moreover, for any $\stype{I_j'}{V}\in\Sep_V(F_j)$, $j=1,2$, by \textbf{Step 1} and \textbf{Step 2}, we can assume that $I_1'\subset\hI_2$ and $\hI_1\subset I_2'$. In particular, $(V\setminus I_1')\ints (V\setminus I_1)\supset (V\setminus\hI_2)\neq\emptyset$ and $I_2'\ints I_2\supset \hI_1\neq\emptyset$. By the second assertion of Lemma \ref{prim decomp}, $I_2'\supset I_2$ and $(V\setminus I_1')\supset (V\setminus I_1)\Leftrightarrow I_1'\subset I_1$. In particular, $\hI_2\supset I_2$ and $I_1\supset \hI_1$.

To show that $I_1,I_2$ constructed above satisfy all the requirements in this assertion, it remains for us to prove that $I_1\subset I_2$.

\textbf{Step 4:} Proving $I_1\subset I_2$.

Clearly $I_1\ints I_2\supset \hI_1\neq\emptyset$ and $(V\setminus I_1)\ints(V\setminus I_2)\supset(V\setminus \hI_2)\neq\emptyset$. Therefore by  the remark after Lemma \ref{properties of Theta sep}, either $I_1\subset I_2$ or $I_2\subsetneq I_1$. If $I_2\subsetneq I_1$, then we have $\hI_1\subset I_2\subsetneq I_1\subset \hI_2$. Choose $q_0\in \hI_1$, $q_1\in I_1\setminus I_2$ and $q_2\in V\setminus \hI_2$. Then
$$F_1, F_2\in\Theta(F_{q_0},F_{q_1})\ints\Theta(F_{q_0},F_{q_2})\ints\Theta(F_{q_1},F_{q_2}).$$
By the second assertion in Lemma \ref{properties of Theta sep},
$d(F_j, [q_s,q_t])\leq\epsilon_0/2$ for any $1\leq j\leq 2$ and $0\leq s,t\leq 2$. Since $F_1\neq F_2$, this is impossible by Proposition \ref{almost ints positioning}. Therefore $I_1\subset I_2$.

\textbf{Step 5}: Uniqueness of $I_1,I_2$ satisfying the requirements.

If $\stype{J_j}{V}\in \Sep_V(F_j)$, $j=1,2$ such that for any $\stype{I_j'}{V}\in \Sep_V(F_j)$, $j=1,2$, $\stype{J_j}{V}$ is between $\stype{I_1'}{V}$ and $\stype{I_2'}{V}$ in the sense of Definition \ref{relations between sep types}, $j=1,2$. In particular, $\stype{J_j}{V}$ is between $\stype{I_1}{V}$ and $\stype{I_2}{V}$ in the sense of Definition \ref{relations between sep types}, and $\stype{I_j}{V}$ is between $\stype{J_1}{V}$ and $\stype{J_2}{V}$ in the sense of Definition \ref{relations between sep types}, $j=1,2$.

If $I_1=I_2$, then $\stype{J_j}{V}=\stype{I_1}{V}=\stype{I_2}{V}$ for any $j=1,2$.

If $I_1\subsetneq I_2$, by the fact that $\stype{J_j}{V}$ is between $\stype{I_1}{V}$ and $\stype{I_2}{V}$ in the sense of Definition \ref{relations between sep types}, we can assume WLOG that $I_1\subset J_1,J_2\subset I_2$. Notice that $I_2=H_{F_2}^V(p_2)$ for some $p_2\in \hI_1$ and $I_1=V\setminus(H_{F_1}^V(p_1))$ for some $p_1\in V\setminus\hI_2$ from \textbf{Step 3}, $V\setminus J_1\subset V\setminus I_1=H_{F_1}^V(p_1)$ and $J_2\subset I_2=H_{F_2}^V(p_2)$. By Lemma \ref{prim decomp}, $I_1=J_1$ and $I_2=J_2$. This finishes the uniqueness of $I_1,I_2$ satisfying the requirements in this assertion.\qedhere
\end{enumerate}
\end{proof}

\begin{definition}[Boundary elements]\label{bdry}
For any finite subset $V\subset \Gamma x_0$ with $|V|\geq 2$ and any subset $\cB\subset\cA_\Sep(V)$, an element $\typemark{\hF}{I}{V}\in\cB$ is called a \emph{boundary element} of $\cB$ if for any pair of elements $\typemark{F_j}{I_j}{V}\in\cB$, $j=1,2$, $\typemark{\hF}{I}{V}$ is between $\typemark{F_1}{I_1}{V}$ and $\typemark{F_2}{I_2}{V}$ in the sense of Definition \ref{in between for ASep} if and only if $\typemark{\hF}{I}{V}=\typemark{F_l}{I_l}{V}$ for some $l=1,2$.

Denoted by $\del\cB$ the collection of boundary elements of $\cB$.
\end{definition}

\begin{rmk}
By the third remark after Notations \ref{sep type marking}, for any $\gamma\in\Gamma$, an element $\typemark{\hF}{I}{V}\in\cB\subset\cA_\Sep(V)$ is a boundary element of $\cB$ if and only if $\gamma\typemark{\hF}{I}{V}\in\gamma\cB\subset\cA_\Sep(\gamma V)$ is a boundary element of $\gamma\cB$.

When $|V|=2$ and $|\cF(V)|=1$, for any $\cB\subset V$, by the second half of Definition \ref{in between for ASep}, one can also define $\del\cB$ in a similar way. In particular, $\del\cB=\cB$ whenever $\cB\neq\emptyset$.
\end{rmk}

\begin{lemma}[Boundary elements of $\cA_\Sep(V)$ and $\cA_\PSep(V)$]\label{bdry=vertex}
For any finite subset $V\subset \Gamma x_0$ with $|V|\geq 2$,
$$\del\cA_\Sep(V)=\del\cA_\PSep(V)=\{\typemark{F_p}{\{p\}}{V}|p\in V\}.$$
\end{lemma}
\begin{proof}
By Lemma \ref{prim decomp}, $\del\cA_\Sep(V)\subset\cA_\PSep(V)$. If $\typemark{\hF}{I}{V}\in\cA_\PSep(V)$ such that $\hF\not\in\{F_p|p\in V\}$, we choose $p\in I$ and $q\in V\setminus I$ and then $\typemark{\hF}{I}{V}$ is between $\typemark{F_p}{\{p\}}{V}$ and $\typemark{F_q}{\{q\}}{V}$ in the sense of Definition \ref{in between for ASep}. Hence $\typemark{\hF}{I}{V}\not\in\del\cA_\Sep(V)$ or $\del\cA_\PSep(V)$.

If $\typemark{F_p}{I}{V}\in \cA_\PSep(V)$ for some $p\in I$ such that for any $p'\in V$ with $F_p=F_{p'}$, $\typemark{F_p}{I}{V}\neq \typemark{F_p}{\{p'\}}{V}$, then $\{p\}\subsetneq I$. Choose $q\in V\setminus I$, then $\typemark{F_p}{I}{V}$ is between $\typemark{F_p}{\{p\}}{V}$ and $\typemark{F_q}{\{q\}}{V}$ in the sense of Definition \ref{in between for ASep}. Hence by the assumptions on $\typemark{F_p}{I}{V}$,  we have $\typemark{F_p}{I}{V}\not\in\del\cA_\Sep(V)$ or $\del\cA_\PSep(V)$. As a short summary, we have
$$\del\cA_\Sep(V)=\del\cA_\PSep(V)\subset\{\typemark{F_p}{\{p\}}{V}|p\in V\}.$$

For any $p\in V$ and any $\typemark{F_j}{I_j}{V}\in\cA_\Sep(V)$, $j=1,2$, such that $\typemark{F_p}{\{p\}}{V}$ is between $\typemark{F_1}{I_1}{V}$ and $\typemark{F_2}{I_2}{V}$ in the sense of Definition \ref{in between for ASep}. WLOG, we assume that $I_1\subset\{p\}\subset I_2$. Then $I_1=\{p\}$. Choose $q\in V\setminus I_2$. Then $F_1, F_p, F_2\in\Theta(F_p,F_q)$ and $F_p\in\Theta(F_1, F_2)$. By Lemma \ref{properties of Theta}, \hyperlink{Theta-3}{property ($\Theta$3) of $\Theta(\cdot,\cdot)$}, $F_1=F_p$ or $F_2=F_p$. If $F_1=F_p$, then $\typemark{F_1}{I_1}{V}=\typemark{F_p}{\{p\}}{V}$. If $F_2=F_p$ and $I_2=\{p\}$, then $\typemark{F_2}{I_2}{V}=\typemark{F_p}{\{p\}}{V}$. Therefore it remains to consider the case when $F_2=F_p$, $\{p\}\subsetneq I_2$ and $F_1\neq F_p$. Choose $x\in I_2\setminus\{p\}$ and $y\in V\setminus I_2$. Then
$$F_1\in\Theta(F_p,F_x)\ints\Theta(F_p,F_y)~\mathrm{and}~F_p\in\Theta(F_p,F_x)\ints\Theta(F_p,F_y)\ints\Theta(F_x,F_y).$$
By the second assertion in Lemma \ref{properties of Theta sep},
$$d(F_1,[p,x]), d(F_1,[p,y]),d(F_p,[p,x]), d(F_p,[p,y]), d(F_p,[x,y])\leq\frac{\epsilon_0}{2}.$$
By Proposition \ref{almost ints positioning} and Lemma \ref{properties of Theta}, \hyperlink{Theta-3}{property ($\Theta$3) of $\Theta(\cdot,\cdot)$}, $F_1\in\Theta(F_p,F_p)=\{F_p\}$, which implies $F_1=F_p$, contradictory to the assumption that $F_1\neq F_p$. As a summary, this proves that
$$\del\cA_\Sep(V)=\del\cA_\PSep(V)\supset\{\typemark{F_p}{\{p\}}{V}|p\in V\}.\qedhere$$
\end{proof}

\subsection{Graph of $\Theta$-separation $\bG_V$ and its properties}\label{subsect:GV and its properties}
\begin{notation}
Throughout this paper, a \emph{graph} $\bG$ is always assumed to be an undirected (i.e. edges do not have directions), simple (i.e. every pair of vertices only admits at most 1 edge) graph without loops (i.e. endpoints of any edge are not the same). Its set of vertices is denoted as $\cV(\bG)$ and its set of edges is denoted as $\cE(\bG)$. An edge with endpoints $Q_0\neq Q_1$ is denoted as $Q_0\mbox{---}Q_1$.

For any graph $\bG$ and any subgraphs $G_1,G_2\subset \bG$, we denote as $G_1\union G_2$ the subgraph of $\bG$ with vertices $\cV(G_1)\union\cV(G_2)$ and edges $\cE(G_1)\union\cE(G_2)$. Similarly, we denote as $G_1\ints G_2$ the subgraph of $\bG$ with vertices $\cV(G_1)\ints\cV(G_2)$ and edges $\cE(G_1)\ints\cE(G_2)$.
\end{notation}

\begin{definition}[Graph of $\Theta$-separation]\label{graph of sep}
For any finite subset $V\subset \Gamma x_0$ with $|V|\geq 2$, we denote by $\bG_V$ as the \emph{graph of $\Theta$-separation with respect to $V$} which is defined as follows: (See Figure \ref{fig:graph examples} for examples.)
\begin{enumerate}
\item[{(1).}] {When $|\cF(V)|=1$ and $|V|=2$ with $V=\{p,q\}$, $\bG_V$ is a complete graph with vertices in $V$. In other words, $\cV(\bG_V)=V$ and $p,q$ are connected by a single edge in $\bG_V$;}
\item[{(2).}] {When $|\cF(V)|>1$ or $|V|>2$}, we have
\begin{itemize}
\item The set of vertices $\cV(\bG_V):=\AnPSep(V)\union\del\cA_\PSep(V)=\AnPSep(V)\union\del\cA_\Sep(V)$;
\item For any $Q_1\neq Q_2\in \cV(\bG_V)$, $Q_1,Q_2$ are connected by a single edge in $\bG_V$ if and only if for any $Q'\in\cV(\bG_V)$, $Q'$ is between $Q_1$ and $Q_2$ in the sense of Definition \ref{in between for ASep} if and only if $Q'\in\{Q_1,Q_2\}$.
\end{itemize}
\end{enumerate}
\begin{figure}[h]
	\centering
	\includegraphics[width=5in]{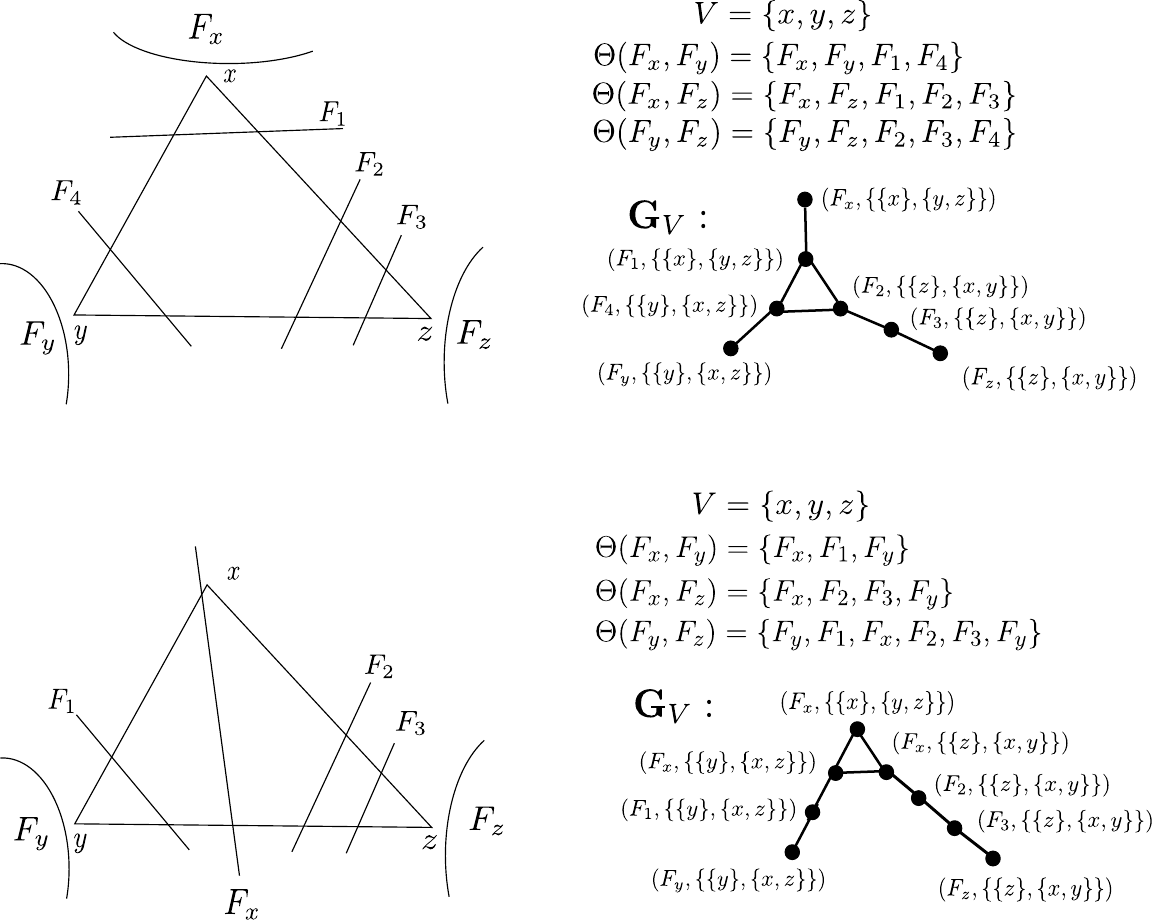}
	\caption{\label{fig:graph examples}}
\end{figure}
\end{definition}
\begin{rmk}
By the third remark after Notations \ref{sep type marking}, the remark after Definition \ref{in between for ASep} and the remark after Definition \ref{bdry}, for any $Q\in\cV(\bG_V)$, any $Q_1\mbox{---}Q_2\in\cE(\bG_V)$ and any $\gamma\in\Gamma$, the map $Q\to\gamma Q$ gives graph isomorphism from $\bG_V$ to $\bG_{\gamma V}$ in the sense that
\begin{itemize}
\item $Q\to\gamma Q$ gives a bijection from $\cV(\bG_V)$ to $\cV(\bG_{\gamma V})$.
\item $\gamma Q_1\mbox{---}\gamma Q_2\in\cE(\bG_{\gamma V})$ and hence $(Q_1\mbox{---} Q_2)\to(\gamma Q_1\mbox{---}\gamma Q_2)$ gives a bijection from $\cE(\bG_V)$ to $\cE(\bG_{\gamma V})$.
\end{itemize}
Therefore, for any subgraph $G\subset \bG_V$, we can define $\gamma G$ as the subgraph of $\bG_{\gamma V}$ with vertices in $\gamma \cV(G)$ and edges in $\gamma \cE(G)$.
\end{rmk}

\begin{lemma}\label{reinterpretation of edges}
Let $V\subset \Gamma x_0$ be a finite subset with $|V|\geq 2$. {Assume that $|\cF(V)|> 1$ or $|V|>2$. Then the following holds.}
\begin{enumerate}
\item[(1).] For any distinct elements $Q_j:=\typemark{F_j}{I_j}{V}\in\cV(\bG_V)$, $j=1,2$, $Q_1$ and $Q_2$ are connected by a single edge in $\bG_V$ if and only if one of the following holds:
\begin{enumerate}
\item[(i).] $F_1=F_2$;
\item[(ii).] $\Theta(F_1,F_2)\ints\cA_0(V)=\emptyset\subset\{F_1,F_2\}$ when $F_1\neq F_2$ and $F_1, F_2\not\in\cA_0(V)$;
\item[(iii).] When exactly one of $F_1, F_2$ is \textbf{NOT} in $\cA_0(V)$, we automatically have $F_1\neq F_2$. WLOG we assume that $F_1\not\in\cA_0(V)$ and $F_2\in\cA_0(V)$. Then $Q_1=\typemark{F_p}{\{p\}}{V}$ for some $p\in V$. In this case, we require that $\Theta(F_1,F_2)\ints\cA_0(V)=\{F_2\}\subset\{F_1, F_2\}$ and $\stype{I_2}{V}$ is the unique separation type for $F_2$ mentioned in the third assertion of Lemma \ref{key lem for ASep} with respect to the pair $F_1, F_2$.
\item[(iv).] When $F_1, F_2\in\cA_0(V)$ and $F_1\neq F_2$, we require that $\Theta(F_1,F_2)\ints\cA_0(V)=\{F_1, F_2\}$ and $\stype{I_j}{V}$ is the unique separation type mentioned in the third assertion of Lemma \ref{key lem for ASep} for $F_j$ with respect to the pair $F_1, F_2$, $j=1,2$.
\end{enumerate}
\item[(2).] For any distinct elements $Q_j:=\typemark{F_j}{I_j}{V}\in\cV(\bG_V)$, $j=1,2,3$, if $Q_2$ is connected to $Q_1$ and $Q_3$ by a single edge in $\bG_V$, then the following are equivalent:
\begin{enumerate}
\item[(i).] $Q_1$ and $Q_3$ are connected by a single edge in $\bG_V$;
\item[(ii).] $\stype{I_1}{V}$, $\stype{I_2}{V}$ and $\stype{I_3}{V}$ are in triangle relation in the sense of Definition \ref{relations between sep types};
\item[(iii).] $Q_2$ is \textbf{NOT} between $Q_1$ and $Q_3$ in the sense of Definition \ref{in between for ASep}.
\end{enumerate}
\end{enumerate}
\end{lemma}
\begin{proof}
\begin{enumerate}
\item[(1).] For any distinct elements $Q_j:=\typemark{F_j}{I_j}{V}\in\cV(\bG_V)$, $j=1,2$.

\textbf{Case 1}: If $F_1=F_2$, for any $\typemark{\hF}{I}{V}\in\cA_\PSep(V)$, $\hF\in\Theta(F_1,F_2)$ implies that $\hF=F_1=F_2$. Since $\cV(\bG_V)\subset \cA_\PSep(V)$, for any $\stype{I}{V}\in\PSep_{F_1}(V)$, by Lemma \ref{prim decomp}, either $\stype{I}{V}\in\{\stype{I_1}{V},\stype{I_2}{V}\}$, or $\stype{I}{V}$, $\stype{I_1}{V}$ and $\stype{I_2}{V}$ are in triangle relation in the sense of Definition \ref{relations between sep types}. Therefore $Q_1$ and $Q_2$ are always connected by a single edge in $\bG_V$ according to Definition \ref{graph of sep}.

\textbf{Case 2}: If $F_1\neq F_2$ and $F_1, F_2\not\in\cA_0(V)$, by Lemma \ref{bdry=vertex}, $Q_1=\typemark{F_p}{\{p\}}{V}$ and $Q_2=\typemark{F_q}{\{q\}}{V}$ for some $p,q\in V$, $p\neq q$ and $F_p, F_q\not\in\cA_0(V)$. For any $\typemark{\hF}{I}{V}\neq Q_1, Q_2$ which is between $Q_1$ and $Q_2$ in the sense of Definition \ref{in between for ASep}, it follows from Definition \ref{bdry} that $\typemark{\hF}{I}{V}\not\in\del\cA_\Sep(V)$. Since $\AnPSep(V)\union\del\cA_\Sep(V)=\cV(\bG_V)$ (due to Definition \ref{graph of sep}), if we assume in addition that $(\hF,\stype{I}{V})\in\cV(\bG_V)$, then $\hF\in\Theta(F_p, F_q)\ints \cA_0(V)$. This shows that if $\Theta(F_1, F_2)\ints\cA_0(V)=\emptyset,$ then $Q_1$ and $Q_2$ are connected by a single edge in $\bG_V$.

On the other hand, if there exists some $\hF\in\Theta(F_1, F_2)\ints\cA_0(V)$, then $\hF\neq F_1$ and $\hF\neq F_2$. By the third assertion in Lemma \ref{key lem for ASep} (applied to the pairs $F_1,\hF$ and $F_2,\hF$), there exists $\stype{I_1'}{V},\stype{I_2'}{V}\in\PSep_V(\hF)$ such that the following holds:
\begin{itemize}
\item $\stype{I_1'}{V}$ is between $\stype{I'}{V}$ and $\stype{I_1}{V}$ in the sense of Definition \ref{relations between sep types} for any $\stype{I'}{V}\in \Sep_V(\hF)$;
\item $\stype{I_2'}{V}$ is between $\stype{I'}{V}$ and $\stype{I_2}{V}$ in the sense of Definition \ref{relations between sep types} for any $\stype{I'}{V}\in \Sep_V(\hF)$.
\end{itemize}
Notice that $\hF\in\cA_0(V)$, $\typemark{\hF}{I_1'}{V},\typemark{\hF}{I_2'}{V}\in\cV(\bG_V)$. If one of $\typemark{\hF}{I_1'}{V}$ and $\typemark{\hF}{I_2'}{V}$ is between $Q_1$ and $Q_2$ in the sense of Definition \ref{in between for ASep}, then $Q_1$ and $Q_2$ are not connected by a single edge in $\bG_V$. Otherwise, by the first assertion in Lemma \ref{key lem for ASep}, $\stype{I_j'}{V}$, $\stype{I_1}{V}$ and $\stype{I_2}{V}$ are in triangle relation in the sense of Definition \ref{relations between sep types}, $j=1,2$. We show that this is impossible. WLOG, we assume that $I_1\cap I_2=\emptyset$ and $I_1\sqcup I_2\subset I_1'\cap I_2'$. Since $I_1\subsetneq I_1'$, by the properties of $I_1'$ (the first bullet point), either $I_1'\subset I_2'$ or $I_1'\subset V\setminus I_2'$. The latter case implies that $I_1\subset I_1'\cap I_2'\subset I_1'\subset V\setminus I_2'\subset V\setminus I_1$, which is impossible. Hence $I_1'\subset I_2'$. Similarly, using properties of $I_2'$ (the second bullet point), $I_2'\subset I_1'$. Therefore $I_1'=I_2'$. By the first assertion in Lemma \ref{key lem for ASep} applied to $F_1, F_2$ and $\hF$, there exists some $\stype{I}{V}\in\Sep_V(\hF)$ such that $\stype{I}{V}$ is between $\stype{I_1}{V}$ and $\stype{I_2}{V}$ in the sense of Definition \ref{relations between sep types}. Since $I_1\cap I_2=\emptyset$, WLOG, we assume that $I_1\subset I\subset V\setminus I_2$. By the properties of $I_1'$ (the first bullet point) and the fact that $I_1\sqcup I_2\subset I_1'$, either $I_1'\subset I$ or $I_1'\subset V\setminus I$. The latter case is impossible since it implies that $I_1\subsetneq I_1'\subset V\setminus I\subset V\setminus I_1$. Hence $I_1'\subset I$. On the other hand, by the properties of $I_2'$ (the second bullet point) and the fact that $I_1\sqcup I_2\subset I_2'$, either $I_2'\subset I$ or $I_2'\subset V\setminus I$. The former case is impossible since it implies that $I_2\subsetneq I_2'\subset I\subset V\setminus I_2$. Hence $I_2'\subset V\setminus I$. Notice that $I_1'=I_2'$, the above discussions show that $I_1'\subset I\subset V\setminus I_2'=V\setminus I_1'$, which is impossible.
Hence we conclude that if $\Theta(F_1, F_2)\ints\cA_0(V)\neq\emptyset$, then $Q_1$ and $Q_2$ are not connected by a single edge in $\bG_V$. This finishes the proof of the fact that in this case, $Q_1$ and $Q_2$ are connected by a single edge in $\bG_V$ if and only if $\Theta(F_1, F_2)\ints\cA_0(V)=\emptyset$.

\textbf{Case 3}: If $F_1\not\in\cA_0(V)$ and $F_2\in\cA_0(V)$,  by Lemma \ref{bdry=vertex}, we have $Q_1=\typemark{F_p}{\{p\}}{V}$ for some $p\in V$. For any $\typemark{\hF}{I}{V}\not\in\{ Q_1, Q_2\}$ which is between $Q_1$ and $Q_2$ in the sense of Definition \ref{in between for ASep}, it follows from Definition \ref{bdry} that $\typemark{\hF}{I}{V}\not\in\del\cA_\Sep(V)$. Since $\AnPSep(V)\union\del\cA_\Sep(V)=\cV(\bG_V)$ (due to Definition \ref{graph of sep}), if we assume in addition that $(\hF,\stype{I}{V})\in\cV(\bG_V)$, then either $\hF\in(\Theta(F_1, F_2)\ints\cA_0(V))\setminus\{F_1,F_2\}$ or $\hF=F_2$ and $\stype{I}{V}$ is between $\stype{I_1}{V}$ and $\stype{I_2}{V}$ in the sense of Definition \ref{relations between sep types}.  By the third assertion in Lemma \ref{key lem for ASep} and the fact that
$$\{\typemark{F_2}{I}{V}|\stype{I}{V}\in\PSep_V(F_2)\}\subset\AnPSep(V)\subset\cV(\bG_V),$$
if $\Theta(F_1, F_2)\ints\cA_0(V)=\{ F_2\}$ and $\stype{I_2}{V}$ is the unique separation type for $F_2$ mentioned in the third assertion of Lemma \ref{key lem for ASep} with respect to the pair $F_1, F_2$, then $Q_1$ and $Q_2$ are connected by a single edge in $\bG_V$.

On the other hand, if $\stype{I_2}{V}$ is not the unique separation type for $F_2$ satisfying the assumptions in the third assertion of Lemma \ref{key lem for ASep} with respect to the pair $F_1, F_2$, there exists some $\stype{I}{V}\in\PSep_V(F_2)$ such that $\stype{I}{V}\neq \stype{I_2}{V}$ and $\stype{I}{V}$ is between $\stype{I_2}{V}$ and $\stype{I_1}{V}$ in the sense of Definition \ref{relations between sep types}. Choose $Q_3=\typemark{F_2}{I}{V}\in\AnPSep(V)\subset\cV(\bG_V)$. Then $Q_3\not\in\{Q_1,Q_2\}$ and $Q_3$ is between $Q_1$ and $Q_2$ in the sense of Definition \ref{in between for ASep}. Hence $Q_1$ and $Q_2$ are not connected by a single edge in $\bG_V$.

If there exists some $\hF\in(\Theta(F_1, F_2)\ints\cA_0(V))\setminus \{ F_2\}$, following the same arguments in the second paragraph of \textbf{Case 2}, one can also show that there exists some $Q_3\in\cV(\bG_V)\setminus\{Q_1,Q_2\}$ such that $Q_3$ is between $Q_1$ and $Q_2$ in the sense of Definition \ref{in between for ASep}. Hence $Q_1$ and $Q_2$ are not connected by a single edge in $\bG_V$. This finishes the proof of the fact that $Q_1$ and $Q_2$ are connected by a single edge in $\bG_V$ if and only if $\Theta(F_1, F_2)\ints\cA_0(V)=\{ F_2\}$ and $\stype{I_2}{V}$ is the unique separation type for $F_2$ mentioned in the third assertion of Lemma \ref{key lem for ASep} with respect to the pair $F_1, F_2$.

\textbf{Case 4}: If $F_1\neq F_2$ and $F_1,F_2\in\cA_0(V)$, for any $\typemark{\hF}{I}{V}\not\in\{Q_1, Q_2\}$ which is between $Q_1$ and $Q_2$ in the sense of Definition \ref{in between for ASep}, it follows from Definition \ref{bdry} that $\typemark{\hF}{I}{V}\not\in\del\cA_\Sep(V)$. Since $\AnPSep(V)\union\del\cA_\Sep(V)=\cV(\bG_V)$ (due to Definition \ref{graph of sep}), if we assume in addition that $(\hF,\stype{I}{V})\in\cV(\bG_V)$, then either $\hF\in(\Theta(F_1, F_2)\ints\cA_0(V))\setminus\{F_1,F_2\}$ or $\hF\in\{F_1,F_2\}$ and $\stype{I}{V}$ is between $\stype{I_1}{V}$ and $\stype{I_2}{V}$ in the sense of Definition \ref{relations between sep types}. By the third assertion in Lemma \ref{key lem for ASep} and the fact that
$$\{\typemark{\hF}{I}{V}|\hF\in\{F_1,F_2\},\stype{I}{V}\in\PSep_V(\hF)\}\subset\AnPSep(V)\subset\cV(\bG_V),$$
if $\Theta(F_1,F_2)\ints\cA_0(V)=\{F_1, F_2\}$ and $\stype{I_j}{V}$ is the unique separation type satisfying the assumptions in the third assertion of Lemma \ref{key lem for ASep} for $F_j$ with respect to the pair $F_1, F_2$, $j=1,2$, then $Q_1$ and $Q_2$ are connected by a single edge in $\bG_V$.

On the other hand, if $\Theta(F_1,F_2)\ints\cA_0(V)\neq\{F_1, F_2\}$, or one of $\stype{I_1}{V}$ and $\stype{I_2}{V}$ is not the unique separation type satisfying the assumptions in the third assertion of Lemma \ref{key lem for ASep} with respect to the pair $F_1, F_2$, by the same arguments in the second and the third paragraph of \textbf{Case 3}, $Q_1$ and $Q_2$ are not connected by a single edge in $\bG_V$. Therefore we conclude that $Q_1$ and $Q_2$ are connected by a single edge in $\bG_V$ if and only if $\Theta(F_1,F_2)\ints\cA_0(V)=\{F_1, F_2\}$ and $\stype{I_j}{V}$ is the unique separation type mentioned in the third assertion of Lemma \ref{key lem for ASep} for $F_j$ with respect to the pair $F_1, F_2$, $j=1,2$.
\item[(2).] We prove in the order of (i)$\implies$(ii)$\implies$(iii)$\implies$(ii)$\implies$(i).

\textbf{Proof of} (i)$\implies$(ii): We split the proof into two cases.

\textbf{Case 1}: Assume that there exist distinct $i,j\in\{1,2,3\}$ such that $F_i=F_j$.

\textbf{Case 1a}: If $F_i\in\cA_0(V)$, then $Q_i,Q_j\in\AnPSep(V)$. Since any pair of vertices in $\{Q_1,Q_2,Q_3\}$ are connected by an edge due to (i), by (iii) and (iv) in the first assertion of Lemma \ref{reinterpretation of edges} and the third assertion in Lemma \ref{key lem for ASep}, we have $F_1=F_2=F_3\in\cA_0(V)$. In particular, $Q_1,Q_2,Q_3\in\AnPSep(V)$. By the remark after Lemma \ref{relations between sep types}, (ii) holds.

\textbf{Case 1b}: If $F_i\not\in\cA_0(V)$, then $Q_i,Q_j\not\in\AnPSep(V)$. Recall that $Q_1,Q_2,Q_3$ are distinct. By Definition \ref{graph of sep} and Lemma \ref{bdry=vertex}, there exist distinct $p,q\in V$ such that $F_p=F_q=F_i=F_j$, $Q_i=\typemark{F_p}{\{p\}}{V}$ and $Q_j=\typemark{F_q}{\{q\}}{V}$. Let $k$ be the unique element in $\{1,2,3\}\setminus\{i,j\}$. If $\{p,q\}\subset I_k$ or $\{p,q\}\subset V\setminus I_k$, then (ii) holds. Therefore it remains to consider the case when $\{p,q\}\cap I_k\neq\emptyset$ and $\{p,q\}\cap (V\setminus I_k)\neq\emptyset$. This implies that $F_k\in\Theta(F_p,F_q)=\{F_p\}$. Hence we have $F_1=F_2=F_3$ and therefore $Q_1,Q_2,Q_3\not\in\AnPSep(V)$. By Definition \ref{graph of sep}, Lemma \ref{bdry=vertex} and the assumption that $Q_1,Q_2,Q_3$ are distinct, (ii) holds.

\textbf{Case 2}: Assume that $F_1,F_2,F_3$ are pairwise distinct. If (ii) does not hold, by the second assertion in Lemma \ref{key lem for ASep}, there exists $l\in\{1,2,3\}$ such that $Q_l$ is between $Q_{l+1}$ and $Q_{l+2}$. ($Q_4:=Q_1$ and $Q_5:=Q_2$) Hence $Q_{l+1}$ and $Q_{l+2}$ are not connected by a single edge in $\bG_V$, contradictory to the assumption that any pair of vertices in $\{Q_1, Q_2, Q_3\}$ are connected by a single edge in $\bG_V$.

\textbf{Proof of} (ii)$\implies$(iii): If not, by Definition \ref{in between for ASep}, $\stype{I_2}{V}$ is between $\stype{I_1}{V}$ and $\stype{I_3}{V}$ in the sense of Definition \ref{relations between sep types}. WLOG we assume that $I_1\subset I_2\subset I_3$. Choose $p\in I_1$ and $q\in V\setminus I_3$. Then for any $I_j'\in\stype{I_j}{V}$, $j=1,2,3$, $|\{p,q\}\ints I_j'|=1$. In particular, there exist at least two elements $s\neq t\in \{1,2,3\}$ such that either $p\in I_s'\ints I_t'$ or $q\in I_s'\ints I_t'$. Hence $\stype{I_1}{V}$, $\stype{I_2}{V}$ and $\stype{I_3}{V}$ are \textbf{NOT} in triangle relation in the sense of Definition \ref{relations between sep types}. This contradicts (ii).

\textbf{Proof of} (iii)$\implies$(ii): We split the proof into two cases.

\textbf{Case 1}: Assume that $F_1,F_2,F_3$ are not pairwise distinct.

\textbf{Case 1a}: If $F_1=F_3\in\cA_0(V)$, then $Q_1,Q_3\in\AnPSep(V)$. By (iii), (iv) in the first assertion of Lemma \ref{reinterpretation of edges}, the assumption that $Q_1,Q_2,Q_3$ are distinct and the third assertion in Lemma \ref{key lem for ASep}, we have $F_1=F_2=F_3$. In this case, both (iii) and (ii) holds due to the remark after Definition \ref{relations between sep types}.

\textbf{Case 1b}: If $F_1=F_3\not\in\cA_0(V)$, then $Q_1,Q_3\not\in\AnPSep(V)$. Recall that $Q_1,Q_2,Q_3$ are distinct. By Definition \ref{graph of sep} and Lemma \ref{bdry=vertex}, there exist distinct $p,q\in V$ such that $F_p=F_q=F_1=F_3$, $Q_1=\typemark{F_p}{\{p\}}{V}$ and $Q_3=\typemark{F_q}{\{q\}}{V}$. Since we assume that (iii) holds, either $\{p,q\}\subset I_2$ or $\{p,q\}\subset V\setminus I_2$. In particular, (ii) holds.

\textbf{Case 1c}: If $F_1\neq F_3$, $F_2\in\{F_1,F_3\}$ and $F_2\in\cA_0(V)$, we assume WLOG that $F_1=F_2\neq F_3$. In this case, we have $Q_1,Q_2\in\AnPSep(V)$. By (iii) and (iv) in the first assertion of Lemma \ref{reinterpretation of edges} and the third assertion in Lemma \ref{key lem for ASep}, (iii) does not hold. Since we assume that (iii) holds, this leads to a contradiction.

\textbf{Case 1d}: If $F_1\neq F_3$, $F_2\in\{F_1,F_3\}$ and $F_2\not\in\cA_0(V)$, we assume WLOG that $F_1=F_2\neq F_3$. Recall that $Q_1,Q_2,Q_3$ are distinct. By Definition \ref{graph of sep} and Lemma \ref{bdry=vertex}, there exist distinct $p,q\in V$ such that $F_p=F_q=F_1=F_2$, $Q_1=\typemark{F_p}{\{p\}}{V}$ and $Q_2=\typemark{F_q}{\{q\}}{V}$. If $I_3\cap\{p,q\}\neq\emptyset$ and $\{p,q\}\cap(V\setminus I_3)\neq \emptyset$, then $F_3\in\Theta(F_p,F_q)=\{F_p\}$, which implies that $F_1=F_2=F_3$. This contradicts with the assumption of this case. Therefore, either $\{p,q\}\subset I_3$ or $\{p,q\}\subset (V\setminus I_3)$. In particular, (ii) holds.

\textbf{Case 2} Assume that $F_1,F_2,F_3$ are pairwise distinct. If (ii) does not hold, by the second assertion in Lemma \ref{key lem for ASep}, there are 2 cases for $Q_1,Q_2,Q_3$:

\textbf{Case 2a:} $Q_1$ is between $Q_2$ and $Q_3$ in the sense of Definition \ref{in between for ASep}. This contradicts $Q_2$ and $Q_3$ being connected by a single edge in $\bG_V$.

\textbf{Case 2b:} $Q_3$ is between $Q_2$ and $Q_1$ in the sense of Definition \ref{in between for ASep}. This contradicts $Q_2$ and $Q_1$ being connected by a single edge in $\bG_V$.

\textbf{Proof of} (ii)$\implies$(i): Since $Q_1$ and $Q_3$ are connected by a single edge in $\bG_V$ when $F_1=F_3$ (due to the first assertion of Lemma \ref{reinterpretation of edges}), WLOG, we assume that $I_1,I_2,I_3$ are pairwise disjoint and $F_1\neq F_3$.

\textbf{Case 1}: If $F_2\in\{F_1,F_3\}$, WLOG we assume that $F_1=F_2$.

\textbf{Case 1a}: If $F_2\in\cA_0(V)$, then $Q_1,Q_2\in\AnPSep(V)$. Since $Q_1\neq Q_2$, by (iii) and (iv) in the first assertion of Lemma \ref{reinterpretation of edges} and the third assertion Lemma \ref{key lem for ASep} applied to $Q_2$ and $Q_3$, $Q_2$ is between $Q_1$ and $Q_3$. This contradicts the fact that (ii)$\implies$(iii).

\textbf{Case 1b}: If $F_2\not\in\cA_0(V)$, then $Q_1,Q_2\not\in\AnPSep(V)$. By Lemma \ref{bdry=vertex} and Definition \ref{graph of sep}, there exist distinct $p,q\in V$ such that $F_p=F_q=F_1=F_2$, $Q_1=\typemark{F_p}{\{p\}}{V}$ and $Q_2=\typemark{F_q}{\{q\}}{V}$. Since we assume that $I_1, I_2, I_3$ are pairwise disjoint, we have $I_1=\{p\}$, $I_2=\{q\}$ and $V\setminus I_3\supset\{p,q\}$. Choose $x\in I_3$. Since $Q_2$ and $Q_3$ are connected by a single edge in $\bG_V$, by the first assertion of Lemma \ref{reinterpretation of edges}, Lemma \ref{prim decomp}, Lemma \ref{bdry=vertex} and the third assertion in Lemma \ref{key lem for ASep}, we claim that
\begin{align}\label{F3}
I_3=
\begin{cases}
\displaystyle \{x\},~&\mathrm{if}~F_3\not\in\cA_0(V),\\
\displaystyle V\setminus H_{F_3}^V(p)=V\setminus H_{F_3}^V(q),~&\mathrm{if}~F_3\in\cA_0(V).
\end{cases}
\end{align}
To be specific, if $F_3\not\in\cA_0(V)$, by Lemma \ref{bdry=vertex} and the fact that $Q_3\not\in\{ Q_1,Q_2\}$, $Q_3=\typemark{F_y}{\{y\}}{V}$ for some $y\in V\setminus\{ p,q\}$. Since $I_1, I_2, I_3$ are pairwise disjoint and $x\in I_3$, we have $I_3=\{y\}$ and hence $y=x$.

If $F_3\in\cA_0(V)$, by (iii) in the first assertion of Lemma \ref{reinterpretation of edges}, $\stype{I_3}{V}$ is the unique separation type for $F_3$ mentioned in the third assertion of Lemma \ref{key lem for ASep} with respect to the pair $F_3, F_2$. Since $I_2\subsetneq V\setminus I_3$, for any $\stype{I_3'}{V}\in\Sep_V(F_3)$, either $(V\setminus I_3)\subset I_3'$ or $(V\setminus I_3)\subset (V\setminus  I_3')$. By Lemma \ref{prim decomp}, $V\setminus I_3=H_{F_3}^V(z)$ for any $z\in V\setminus I_3$. In particular, one can choose $z=p$ or $q$. This proves the claim in \eqref{F3}.

By the assumption that $F_1=F_2$, the fact that $Q_2$ and $Q_3$ are connected by a single edge in $\bG_V$, and (ii), (iii) in the first assertion of Lemma \ref{reinterpretation of edges}, we have
$$\Theta(F_2,F_3)\ints\cA_0(V)\setminus\{F_2,F_3\}=\Theta(F_1,F_3)\ints\cA_0(V)\setminus\{F_1,F_3\}\subset\{F_1,F_3\}.$$

If $F_3\in\cA_0(V)$ and$\stype{I_3'}{V}$ is the unique separation type for $F_3$ mentioned in the third assertion of Lemma \ref{key lem for ASep} with respect to the pair $F_3, F_1$, by the fact that $I_1,I_2,I_3$ are pairwise disjoint (and hence $I_3\subsetneq V\setminus I_1$), we can assume WLOG that $I_3\subset I_3'\subset V\setminus I_1$. Since we showed in the previous paragraph that for any $\stype{I_3''}{V}\in\Sep_V(F_3)$, either $I_3''\subset I_3$ or $V\setminus I_3''\subset I_3$, we have $I_3'=I_3$ and hence $\stype{I_3}{V}$ is the unique separation type for $F_3$ mentioned in the third assertion of Lemma \ref{key lem for ASep} with respect to the pair $F_3, F_1$. By (ii) and (iii) in the first assertion of Lemma \ref{reinterpretation of edges}, $Q_1$ and $Q_3$ are connected by a single edge.

\textbf{Case 2}: $F_1,F_2,F_3$ are pairwise distinct. Choose $p_1\in I_1\subset V\setminus (I_2\cup I_3)$ and $p_2\in I_2\subset V\setminus (I_1\cup I_3)$. By (iii) and (iv) in the first assertion of Lemma \ref{reinterpretation of edges}, Lemma \ref{prim decomp}, Lemma \ref{bdry=vertex} and the third assertion in Lemma \ref{key lem for ASep}, the fact that $Q_1$ and $Q_2$ are connected by a single edge in $\bG_V$ implies that
\begin{align*}
I_1=
\begin{cases}
\displaystyle \{p_1\},~&\mathrm{if}~F_1\not\in\cA_0(V),\\
\displaystyle V\setminus H_{F_1}^V(p_2),~&\mathrm{if}~F_1\in\cA_0(V),
\end{cases}
~\mathrm{and}~
I_2=
\begin{cases}
\displaystyle \{p_2\},~&\mathrm{if}~F_2\not\in\cA_0(V),\\
\displaystyle V\setminus H_{F_2}^V(p_1),~&\mathrm{if}~F_2\in\cA_0(V).
\end{cases}
\end{align*}
(The details are the same as proof of \eqref{F3} in \textbf{Case 1b}. To be specific, if $F_1\not \in \cA_0(V)$, then by Definition \ref{graph of sep}, $\typemark{F_1}{I_1}{V}\in\del \cA_\Sep(V)$. By Lemma \ref{bdry=vertex}, $\typemark{F_1}{I_1}{V}=\typemark{F_p}{\{p\}}{V}$ for some $p\in V$. Since $I_1\subsetneq V\setminus I_2\subsetneq V$, we have $I_1=\{p\}$. Similar arguments can apply to the case when $F_2\not \in \cA_0(V)$.

If $F_1\in\cA_0(V)$, by (iii) and (iv) in the first assertion of Lemma \ref{reinterpretation of edges}, $\stype{I_1}{V}$ is the unique separation type for $F_1$ mentioned in the third assertion Lemma \ref{key lem for ASep} with respect to the pair $F_1, F_2$. Since $I_1\subsetneq V\setminus I_2$, for any $\stype{I_1'}{V}\in\Sep_V(F_1)$, either $I_1'\subset I_1$ or $V\setminus I_1'\subset I_1$. By Lemma \ref{prim decomp}, $V\setminus I_1=H_{F_1}^V(q)$ for any $q\in V\setminus I_1$. In particular, we can choose $q=p_2\in V\setminus I_1$. Similar arguments can apply to the case when $F_2\in\cA_0(V)$.)

Similarly
$$ I_3=
\begin{cases}
\displaystyle \{p_3\},~&\mathrm{if}~F_3\not\in\cA_0(V),\\
\displaystyle V\setminus H_{F_3}^V(p_2),~&\mathrm{if}~F_3\in\cA_0(V),
\end{cases}$$
where we choose $p_3\in I_3\subset V\setminus (I_1\cup I_2)$. Since $p_2,p_3\in V\setminus I_1$ and $p_2,p_1\in V\setminus I_3$, by Lemma \ref{prim decomp},
$$H_{F_1}^V(p_2)=H_{F_1}^V(p_3)~\mathrm{if}~F_1\in\cA_0(V),~\mathrm{and}~H_{F_3}^V(p_2)=H_{F_3}^V(p_1)~\mathrm{if}~F_3\in\cA_0(V).$$
Therefore
\begin{align}\label{F1F3}
I_1=
\begin{cases}
\displaystyle \{p_1\},~&\mathrm{if}~F_1\not\in\cA_0(V),\\
\displaystyle V\setminus H_{F_1}^V(p_3),~&\mathrm{if}~F_1\in\cA_0(V),
\end{cases}
~\mathrm{and}~
I_3=
\begin{cases}
\displaystyle \{p_3\},~&\mathrm{if}~F_3\not\in\cA_0(V),\\
\displaystyle V\setminus H_{F_3}^V(p_1),~&\mathrm{if}~F_3\in\cA_0(V).
\end{cases}
\end{align}

By the second assertion in Lemma \ref{properties of Theta sep}, (See Figure \ref{lem6.19(2)}.)
\begin{equation}\label{F1F2F3 Theta}
\left.
\begin{aligned}
F_1\in\Theta(F_{p_1},F_{p_2})\ints\Theta(F_{p_1},F_{p_3}) ,\\
F_2\in\Theta(F_{p_2},F_{p_1})\ints\Theta(F_{p_2},F_{p_3}) ,\\
F_3\in\Theta(F_{p_3},F_{p_1})\ints\Theta(F_{p_3},F_{p_2}) ,
\end{aligned}
\right\}\implies
\left\{
\begin{aligned}
d(F_1,[p_1,p_2]),d(F_1,[p_1,p_3])\leq\epsilon_0/2 ,\\
d(F_2,[p_2,p_1]),d(F_1,[p_2,p_3])\leq\epsilon_0/2, \\
d(F_3,[p_3,p_1]),d(F_3,[p_3,p_1])\leq\epsilon_0/2.
\end{aligned}
\right.
\end{equation}

\begin{figure}[h]
	\centering
	\includegraphics[scale=0.5]{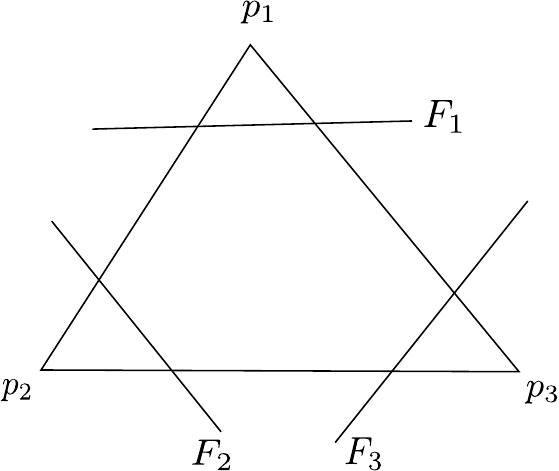}
	\caption{ \label{lem6.19(2)}}
\end{figure}

If $Q_1$ and $Q_3$ are not connected by a single edge in $\bG_V$, assume that $Q:=\typemark{\hF}{I}{V}\in\cV(\bG_V)$ is between $Q_1, Q_3$ and not equal to $Q_1$ or $Q_3$. Definition \ref{bdry} implies that $Q\not\in\del \cA_\Sep(V)$. Hence $\hF\in\cA_0(V)$ due to Definition \ref{graph of sep}. By \eqref{F1F3}, Lemma \ref{prim decomp} and the fact that $I_1, I_2, I_3$ are pairwise disjoint, $\hF\not\in\{F_1,F_3\}$ and therefore $\hF\in(\Theta(F_1,F_3)\ints\cA_0(V))\setminus\{F_1,F_3\}$. (To be specific, if $\hF=F_1$, then by \eqref{F1F3}, Lemma \ref{prim decomp} and the fact that $I_1, I_2, I_3$ are pairwise disjoint, $Q_1$ is between $Q$ and $Q_3$. Therefore $Q=Q_1$. This contradicts the assumption that $Q\neq Q_1$ or $Q_3$. $\hF\neq F_3$ for the same reason.)

\textbf{Case 2a:} If $ \hF\neq F_2$, then either $\hF\in\Theta(F_{p_1},F_{p_3})\ints\Theta(F_{p_1},F_{p_2})$, or $\hF\in\Theta(F_{p_1},F_{p_3})\ints\Theta(F_{p_2},F_{p_3})$ (due to the fact that $\hF\in\cA_0(V)$). By the second assertion in Lemma \ref{properties of Theta sep}, either $\max\{d(\hF,[p_1,p_3]),d(\hF,[p_1,p_2])\}\leq\epsilon_0/2$, or $\max\{d(\hF,[p_1,p_3]),d(\hF,[p_2,p_3])\}\leq\epsilon_0/2$. \eqref{F1F2F3 Theta} and Proposition \ref{almost ints positioning} applied to $\hF, F_2$ imply that either $\hF\in(\Theta(F_1,F_2)\ints\cA_0(V))\setminus\{F_1,F_2\}$ or $\hF\in(\Theta(F_2,F_3)\ints\cA_0(V))\setminus\{F_2,F_3\}$. This leads to a contradiction with either the fact that $Q_1, Q_2$ are connected by a single edge in $\bG_V$, or the fact that $Q_3, Q_2$ are connected by a single edge in $\bG_V$. (Here we used (iii) and (iv) in the first assertion of Lemma \ref{reinterpretation of edges}.)

\textbf{Case 2b:} Assume that $\hF=F_2$. Since $Q_1,Q_2$ and $Q_3,Q_2$ are connected by single edges in $\bG_V$ respectively, by (iii) and (iv) in the first assertion of Lemma \ref{reinterpretation of edges}, the third assertion in Lemma \ref{key lem for ASep} and the fact that $F_2=\hF\in\cA_0(V)$ (mentioned right before \textbf{Case 2a}), $Q_2$ is between $Q_j$ and $Q$ in the sense of Definition \ref{in between for ASep}, $j=1,3$. In particular, $\stype{I_2}{V}$ is between $\stype{I_j}{V}$ and $\stype{I}{V}$ in the sense of Definition \ref{relations between sep types}. Notice that $I_2\subsetneq V\setminus I_j$, $j=1,3$, we can assume WLOG that $I\subset I_2$. Then $I, I_1, I_3$ are pairwise disjoint due to $I_2, I_1,I_3$ being pairwise disjoint.

On the other hand, $Q$ is between $Q_1$ and $Q_3$ in the sense of Definition \ref{in between for ASep}. In particular, $\stype{I}{V}$ is between $\stype{I_1}{V}$ and $\stype{I_3}{V}$ in the sense of Definition \ref{relations between sep types}. Since $I_1\subsetneq V\setminus I_3$ and  $I_3\subsetneq V\setminus I_1$, either $I_1\subset I$ or $I_3\subset I$. This contradicts the fact that $I, I_1, I_3$ are pairwise disjoint.\qedhere
\end{enumerate}
\end{proof}

\begin{definition}[Maximal complete subgraph (with respect to inclusion)]\label{MCS def}
For any graph $\bG$, a subgraph $\bG'\subset\bG$ (i.e. the vertices and edges of $\bG'$ are subsets of those of $\bG$ respectively) is a \emph{maximal complete subgraph} if $\bG'$ is a complete graph (i.e. $\bG'$ is nonempty and every pair of its vertices admits an edge) and for any $\bG'\subset\bG''\subset\bG$, $\bG''$ is complete if and only if $\bG''=\bG'$.

We will use the abbreviation MCS for maximal complete subgraphs. Also, we denote by $\MCS(\bG)$ the set of all MCS of $\bG$.
\end{definition}

\begin{rmk}
By the remark after Definition \ref{graph of sep}, one can easily check that $G$ is a MCS of $\bG_V$ if and only if $\gamma G$ is a MCS of $\bG_{\gamma V}$. Moreover, the map $G\to\gamma G$ gives a bijection from $\MCS(\bG_V)$ to $\MCS(\bG_{\gamma V})$.
\end{rmk}

\begin{proposition}\label{key prop of ASep graph}
Let $V\subset \Gamma x_0$ be a finite subset such that $|V|\geq 2$. Then,
\begin{enumerate}
\item[(1).] Every edge in $\bG_V$ is contained in a unique MCS of $\bG_V$.
\item[(2).] Every vertex in $\bG_V$ is contained in at most 2 MCS of $\bG_V$.
\end{enumerate}
\end{proposition}
\begin{proof}
The proposition is trivial if $|\cF(V)|=1$ and $|V|=2$. Therefore we assume that $|\cF(V)|> 1$ or $|V|>2$.
\begin{enumerate}
	\begin{figure}[h]
		\centering
		\includegraphics[scale=1]{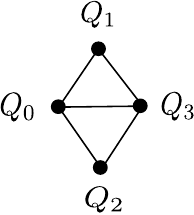}
		\caption{ \label{pic:key prop of ASep graph-1}}
	\end{figure}
\item[(1).] If not, there exist an edge with endpoints $Q_0\neq Q_3\in\cV(\bG_V)$ and $G_1\neq G_2\in\MCS(\bG_V)$ such that the edge $Q_0\mbox{---}Q_3\in\cE(G_1)\ints\cE(G_2)$. By maximality in the sense of Definition \ref{MCS def}, there exist $Q_1\in\cV(G_1)$ and  $Q_2\in\cV(G_2)$ such that $Q_1$ and $Q_2$ are not equal nor connected by a single edge in $\bG_V$. (See Figure \ref{pic:key prop of ASep graph-1}.) Otherwise, for any $Q_j'\in\cV(G_j)$, $j=1,2$, $Q_1'$ and $Q_2'$ are connected by a single edge. Then there exists a complete subgraph $G_3$ with vertex set equal to $\cV(G_1)\union \cV(G_2)$. Since $G_1,G_2$ are MCS, $G_1=G_2=G_3$, contradicting $G_1\neq G_2$. For such a choice of $Q_1, Q_2$, we have $Q_1\not\in\cV(G_2)$ and $Q_2\not\in\cV(G_1)$. Hence $Q_0,...,Q_3$ are distinct vertices of $\bG_V$. Moreover, for any $0\leq s\neq t\leq 3$, $Q_s$ and $Q_t$ are not connected by a single edge in $\bG_V$ if and only if $\{s,t\}=\{1,2\}$.

Let $Q_j=\typemark{F_j}{I_j}{V}$, $j=0,1,2,3$. We first prove that $F_0,F_1,F_2,F_3$ are distinct. By the second assertion of Lemma \ref{reinterpretation of edges}, $Q_0,Q_3$ are between $Q_1$ and $Q_2$ in the sense of Definition \ref{in between for ASep}. It follows from Definition \ref{bdry} that $Q_0,Q_3\not\in\del \cA_\Sep(V)$. Hence by Definition \ref{graph of sep}, $F_0,F_3\in \cA_0(V)$. By the first assertion of Lemma \ref{reinterpretation of edges}, the uniqueness in the third assertion of Lemma \ref{key lem for ASep} and the fact that $F_0,F_3\in \cA_0(V)$, $F_0,F_1,F_3$ are either the same, or pairwise distinct. Similarly, $F_0,F_2,F_3$ are either the same, or pairwise distinct. Notice that $F_1\neq F_2$ due to the first assertion of Lemma \ref{reinterpretation of edges} and the fact that $Q_1,Q_2$ are not connected by an edge in $\bG_V$. Therefore, $F_0,F_1,F_2,F_3$ are pairwise distinct.

Since $Q_0, Q_3$ are between $Q_1$ and $Q_2$ in the sense of Definition \ref{in between for ASep}, by the first assertion of Lemma \ref{reinterpretation of edges}, $F_0, F_3\in\Theta(F_1,F_2)\ints\cA_0(V)\setminus\{F_1,F_2\}.$ Since $F_0,F_1,F_2, F_3$ are distinct, by Lemma \ref{properties of Theta}, \hyperlink{Theta-3}{property ($\Theta$3) of $\Theta(\cdot,\cdot)$}, either $F_3\in\Theta(F_1,F_0)\ints\cA_0(V)\setminus\{F_1,F_0\}$, or $F_0\in\Theta(F_1,F_3)\ints\cA_0(V)\setminus\{F_1,F_3\}$. This contradicts either $Q_1\mbox{---}Q_0\in\cE(\bG_V)$ or $Q_1\mbox{---}Q_3\in\cE(\bG_V)$ due to the first assertion in Lemma \ref{reinterpretation of edges}. Hence every edge in $\bG_V$ must be contained in a unique MCS of $\bG_V$.
\begin{figure}[h]
	\centering
	\includegraphics[scale=1]{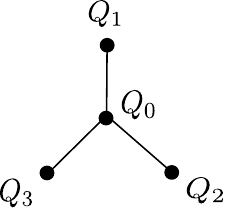}
	\caption{ \label{pic:key prop of ASep graph-2}}
\end{figure}
\item[(2).] If not, there exist some $Q_0\in\cV(\bG_V)$ and distinct MCS $G_1,G_2,G_3\in\MCS(\bG_V)$ such that $Q_0\in\cV(G_1)\ints\cV(G_2)\ints\cV(G_3)$. By the maximality of $G_j$ and the fact that they are distinct, $j=1,2,3$, $|\cV(G_j)|\geq 2$ for any $1\leq j\leq 3$. In particular, one can choose $Q_j\in\cV(G_j)\setminus\{Q_0\}$, $j=1,2,3$. If $Q_s\mbox{---}Q_t\in\cE(\bG_V)$ for some $1\leq s\neq t\leq 3$, then $Q_0,Q_t,Q_s$ and the edges connecting any pair of them form a complete subgraph of $\bG_V$. Let $G\in\MCS(\bG_V)$ be a MCS containing $Q_0,Q_t,Q_s$. Since $Q_s\mbox{---}Q_0\in\cE(G)\ints\cE(G_s)$ and $Q_t\mbox{---}Q_0\in\cE(G)\ints\cE(G_t)$, by the first assertion of Proposition \ref{key prop of ASep graph}, $G=G_s=G_t$. This contradicts the assumption that $G_s\neq G_t$. Therefore we can conclude that $Q_s, Q_t$ are \textbf{NOT} connected by a single edge in $\bG_V$ for any $1\leq s\neq t\leq 3$ and $Q_0,Q_j$ are connected by a single edge in $\bG_V$ for any $1\leq j\leq 3$. (See Figure \ref{pic:key prop of ASep graph-2}.)

Let $Q_j=\typemark{F_j}{I_j}{V}$, $j=0,1,2,3$. By the first assertion of Lemma \ref{reinterpretation of edges}, $F_1,F_2,F_3$ are pairwise distinct. By the second assertion of Lemma \ref{reinterpretation of edges}, for any $1\leq s\neq t\leq 3$, $Q_0$ is between $Q_s$ and $Q_t$ in the sense of Definition \ref{in between for ASep}. In particular, for any $1\leq s\neq t\leq 3$, $\stype{I_0}{V}$ is between $\stype{I_s}{V}$ and $\stype{I_t}{V}$ in the sense of Definition \ref{relations between sep types}. Therefore we can assume WLOG that either $I_0\subset I_j$ or $I_j\subset I_0$ for any $1\leq j\leq 3$.

We first claim that there exist some $l\in\{1,2,3\}$ such that $I_0=I_l$. If not, for any $l\in\{1,2,3\}$, either $I_l\subsetneq I_0$ or $I_0\subsetneq I_l$. WLOG, assume that $I_1\subsetneq I_0$. By the fact that $\stype{I_0}{V}$ is between $\stype{I_1}{V}$ and $\stype{I_t}{V}$ in the sense of Definition \ref{relations between sep types}, $t=2,3$, we have $I_0\subsetneq I_2$ and $I_0\subsetneq I_3$. Since $\stype{I_0}{V}$ is between $\stype{I_2}{V}$ and $\stype{I_3}{V}$ in the sense of Definition \ref{relations between sep types}, either $I_3\subsetneq I_0$ or $V\setminus I_3\subsetneq I_0$. Both options contradicts with $I_0\subsetneq I_3$. Hence there exist some $l\in\{1,2,3\}$ such that $I_0=I_l$. WLOG, we assume that $I_0=I_1$ and $I_2\subset I_0\subset I_3$ (by the fact that $\stype{I_0}{V}$ is between $\stype{I_2}{V}$ and $\stype{I_3}{V}$ in the sense of Definition \ref{relations between sep types}).

Choose $p\in I_2$ and $q\in V\setminus I_3$, we have $p\in I_1\ints I_2\ints I_3$ and $q\in (V\setminus I_1)\ints(V\setminus I_2)\ints(V\setminus I_3)$. Therefore $\stype{I_1}{V}$, $\stype{I_2}{V}$ and $\stype{I_3}{V}$ are \textbf{NOT} in triangle relation in the sense of Definition \ref{relations between sep types}. By the second assertion Lemma \ref{key lem for ASep}, there exist some $l\in\{1,2,3\}$ such that $Q_l$ is between $Q_{l+1}$ and $Q_{l+2}$ in the sense of Definition \ref{relations between sep types}, where $Q_4:=Q_1$ and $Q_5:=Q_2$. Since both $Q_0$ and $Q_l$ are between $Q_{l+1}$ and $Q_{l+2}$ in the sense of Definition \ref{relations between sep types}, by Definition \ref{bdry} and Definition \ref{graph of sep}, $F_l, F_0\in\cA_0(V)$. We discuss the rest of the proof (which is basically finding contradictions) in the following 3 cases.

\textbf{Case 1}: If $F_0=F_{l+1}$ or $F_{l+2}$, then by the first assertion in Lemma \ref{reinterpretation of edges} and the fact that  $F_l\in\Theta(F_{l+1},F_{l+2})\ints \cA_0(V)\setminus\{F_{l+1},F_{l+2}\}$, $Q_0$ is \textbf{NOT} connected to either $Q_{l+2}$ or $Q_{l+1}$ by a single edge in $\bG_V$. This contradicts the fact that $Q_0\mbox{---} Q_{l+1},Q_0\mbox{---} Q_{l+2}\in\cE(\bG_V)$.

\textbf{Case 2}: If $F_0\not\in\{ F_1,F_2,F_3\}$, then $F_0, F_l\in \Theta(F_{l+1},F_{l+2})\ints \cA_0(V)\setminus\{F_{l+1},F_{l+2}\}$ are distinct. By Lemma \ref{properties of Theta}, \hyperlink{Theta-3}{property ($\Theta$3)}, either $F_l\in \Theta(F_{l+1},F_{0})\ints \cA_0(V)\setminus\{F_{l+1},F_{0}\}$ or $F_l\in \Theta(F_{0},F_{l+2})\ints \cA_0(V)\setminus\{F_{0},F_{l+2}\}$. By the first assertion in Lemma \ref{reinterpretation of edges}, this contradicts the fact that $Q_0\mbox{---} Q_{l+1},Q_0\mbox{---} Q_{l+2}\in\cE(\bG_V)$.

\textbf{Case 3}: If $F_0=F_l$, then either $I_0\subsetneq I_l$ or $I_l\subsetneq I_0$. WLOG we assume that $I_{l+1}\subset I_0\subset I_{l+2}$. (This is because $Q_0$ is between $Q_{l+1}$ and $Q_{l+2}$ in the sense of Definition \ref{relations between sep types}.) Since $Q_l$ is between $Q_{l+1}$ and $Q_{l+2}$ in the sense of Definition \ref{relations between sep types}, in particular, $\stype{I_{l}}{V}$ is between $\stype{I_{l+1}}{V}$ and $\stype{I_{l+2}}{V}$ in the sense of Definition \ref{relations between sep types}. If $I_{l+1}=I_{l+2}$, then $I_l=I_0=I_{l+1}$, contradictory to the fact that $I_0\neq I_l$. Therefore $I_{l+1}\subsetneq I_{l+2}$ and hence either $I_{l+1}\subset I_l\subset I_{l+2}$ or $I_{l+1}\subset V\setminus I_l\subset I_{l+2}$.

\textbf{Case 3a}: If $I_{l+1}\subset I_l\subset I_{l+2}$, by the fact that $I_{l+1}\subset I_0\subset I_{l+2}$, either $I_{l+1}\subset I_0\subsetneq I_l\subset I_{l+2}$ or $I_{l+1}\subset I_l\subsetneq I_0\subset I_{l+2}$. Hence $\stype{I_0}{V}$ is \textbf{NOT} the unique $\Theta$-separation type for $F_0$ satisfying the assumptions in the third assertion of Lemma \ref{properties of Theta sep} with respect to either the pair $F_0, F_{l+2}$ or the pair $F_0, F_{l+1}$. By the first assertion of Lemma \ref{reinterpretation of edges} and the fact that $F_0=F_l\in\cA_0(V)$, this contradicts the fact that $Q_0\mbox{---} Q_{l+1},Q_0\mbox{---} Q_{l+2}\in\cE(\bG_V)$.

\textbf{Case 3b}: If $I_{l+1}\subset V\setminus I_l\subset I_{l+2}$, then $V\setminus I_{l+2}\subset  I_l\subset V\setminus I_{l+1}$. When  $I_0\subsetneq I_l$, $I_{l+1}\subset I_0\subset I_{l+2}$ and $I_0\subsetneq I_l\subset V\setminus I_{l+1}$ imply that $I_{l+1}\subset V\setminus I_{l+1}$, which is impossible. When $I_l\subsetneq I_0$, $I_{l+1}\subset I_0\subset I_{l+2}$ and $V\setminus I_{l+2}\subset I_l\subsetneq I_0$ imply that $V\setminus I_{l+2}\subset I_{l+2}$, which is impossible.\qedhere
\end{enumerate}
\end{proof}
\begin{lemma}\label{sep ordering}
Let $V\subset \Gamma x_0$ be a finite subset such that $|V|\geq 2$. Assume that $|\cF(V)|>1$ or $|V|>2$. Let $Q_j=\typemark{F_j}{I_j}{V}\in\cV(\bG_V)$ be distinct vertices, $j=0,1,...,m$ such that
\begin{itemize}
\item For any $0\leq j\leq m-1$, $Q_j\mbox{---}Q_{j+1}\in\cE(\bG_V)$. In addition, by the third assertion in Lemma \ref{properties of Theta sep} (when $F_j\neq F_{j+1}$) and Lemma \ref{prim decomp} (when $F_j=F_{j+1}$), we can assume WLOG that either $I_j\subset I_{j+1}$ or $I_{j+1}\subset I_j$;
\item For any $0\leq j\leq m-2$, $Q_j$ and $Q_{j+2}$ are not connected by a single edge in $\bG_V$;
\item $m\geq 1$ and $I_0\supsetneq I_1$.
\end{itemize}
Then $I_0\supsetneq I_1\supset I_2\supset...\supset I_m$.
\end{lemma}
\begin{proof}
We prove by induction on $m$. When $m=1$, nothing needs to be proved. When $m=2$, this follows from the second assertion of Lemma \ref{reinterpretation of edges}. We now assume that $m\geq 3$.

Suppose the result holds for $\leq m-1$, then $I_0\supsetneq I_1\supset I_2\supset...\supset I_{m-1}$. It suffices for us to prove that $I_{m-1}\supset I_m$.

If not, we have $I_{m-1}\subsetneq I_{m}$. Let $I_j'=I_{m-j}$ and $Q_j'=Q_{m-j}=\typemark{F_{m-j}}{I_j'}{V}\in\cV(\bG_V)$, $j=1,2,...,m$. $Q_0',...,Q_{m-1}'$ satisfy the assumptions in this lemma. By the case of $m-1$, we have $I_0'\supsetneq I_1'\supset...\supset I_{m-1}'$, which is equivalent to $I_m\supsetneq I_{m-1}\supset...\supset I_1$. Hence $I_0\supsetneq I_1=I_2=...=I_{m-1}\subsetneq I_m$. Since $Q_j$ are distinct vertices, $j=0,...,m$, $F_j$ are distinct elements in $\cF(V)$, $j=1,...,m-1$. By the third assertion Lemma \ref{properties of Theta sep} (when $F_0\neq F_m$) and Lemma \ref{prim decomp} (when $F_0=F_m$), we have 3 remaining cases.
\begin{figure}[h]
	\centering
	\includegraphics[scale=0.7]{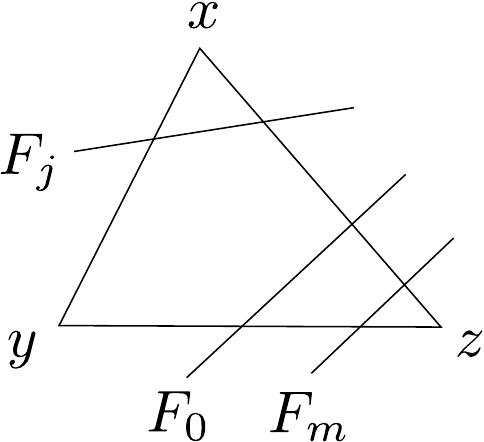}
	\caption{ \label{pic:sep ordering-1}}
\end{figure}

\textbf{Case 1}: $I_m\ints (V\setminus I_0)=\emptyset\Leftrightarrow I_m\subset I_0$. Choose $x\in I_1$, $y\in I_m\setminus I_1$ and $z\in V\setminus I_0$. Then
\begin{align*}
F_0,F_m\in \Theta(F_x,F_z)\ints\Theta(F_y,F_z),
~F_j\in\Theta(F_x,F_y)\ints\Theta(F_x,F_z),~1\leq j\leq m-1.
\end{align*}
(See Figure \ref{pic:sep ordering-1}.) By the second assertion in Lemma \ref{properties of Theta sep}, we have
$$d(F_0,[x,z]),d(F_m,[x,z]),d(F_0,[y,z]),d(F_m,[y,z]),d(F_j,[x,y]),d(F_j,[x,z])\leq \frac{\epsilon_0}{2},~~1\leq j\leq m-1.$$
By Proposition \ref{almost ints positioning}, Lemma \ref{properties of Theta}, \hyperlink{Theta-3}{property ($\Theta$3)} and the fact that $F_j\in\Theta(F_x,F_z)$ for any $0\leq j\leq m$, we have $F_j\in\Theta(F_0,F_x)\ints\Theta(F_m,F_x)$ for any $1\leq j\leq m-1$. By the second assertion of Lemma \ref{reinterpretation of edges}, $Q_{j+1}$ is between $Q_j$ and $Q_{j+2}$ in the sense of Definition \ref{in between for ASep} and therefore $F_{j+1}\in\Theta(F_j, F_{j+2})$ for any $0\leq j\leq m-2$. By Lemma \ref{properties of Theta}, \hyperlink{Theta-3}{property ($\Theta$3)}, we have
\begin{align}\label{sep ordering-1}
F_j\in\Theta(F_0,F_x)\ints\Theta(F_m,F_x)\implies \Theta(F_j,F_x)\subset\Theta(F_0,F_x)\ints\Theta(F_m,F_x), ~~1\leq j\leq m-1.
\end{align}
Recall that $F_j$ are distinct elements in $\cF(V)$, $j=1,...,m-1$, by \eqref{sep ordering-1} and Lemma \ref{properties of Theta}, \hyperlink{Theta-3}{property ($\Theta$3)} again,
\begin{align*}
\left.
\begin{aligned}
&F_{j+1}\in\Theta(F_j, F_{j+2}),~\forall 0\leq j\leq m-2,\\
&F_s\neq F_t,~\forall 1\leq s\neq t\leq m-1,
\end{aligned}
\right\}
\implies
\left\{
\begin{aligned}
\Theta(F_0,F_x)\supset\Theta(F_1,F_x)\supsetneq...\supsetneq\Theta(F_{m-1},F_x),\\
\Theta(F_m,F_x)\supset\Theta(F_{m-1},F_x)\supsetneq...\supsetneq\Theta(F_1,F_x),
\end{aligned}
\right.
\end{align*}
which is impossible.

\textbf{Case 2}: $I_0\ints (V\setminus I_m)=\emptyset\Leftrightarrow I_0\subset I_m$. This is the same as \textbf{Case 2} after reversing the order of subscripts of $Q_j, F_j, I_j$, $0\leq j\leq m$, which is impossible.

\begin{figure}[h]
	\centering
	\includegraphics[scale=0.7]{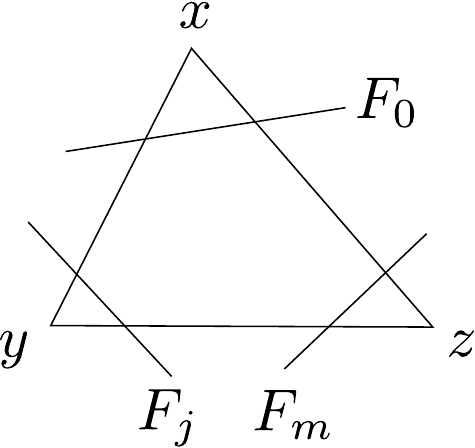}
	\caption{ \label{pic:sep ordering-2}}
\end{figure}

\textbf{Case 3}: $(V\setminus I_0)\ints (V\setminus I_m)=\emptyset$, then $V\setminus I_0,V\setminus I_m, I_1$ are pairwise disjoint. Choose $x\in V\setminus I_0$, $y\in I_1$ and $z\in V\setminus I_m$. Then
$$F_0\in\Theta(F_x,F_y)\ints \Theta(F_x,F_z), ~F_m\in\Theta(F_x,F_z)\ints\Theta(F_y,F_z),~F_j\in\Theta(F_x,F_y)\ints\Theta(F_z,F_y),~1\leq j\leq m-1.$$
(See Figure \ref{pic:sep ordering-2}.) By the second assertion in Lemma \ref{properties of Theta sep}, we have
$$d(F_0,[x,y]),d(F_0,[x,z]),d(F_m,[x,z]),d(F_m,[y,z]),d(F_j,[x,y]),d(F_j,[z,y])\leq \frac{\epsilon_0}{2},~~1\leq j\leq m-1.$$
By Proposition \ref{almost ints positioning}, Lemma \ref{properties of Theta}, \hyperlink{Theta-3}{property ($\Theta$3)} and the fact that $F_j\in\Theta(F_x,F_y)\ints\Theta(F_z,F_y)$ for any $1\leq j\leq m-1$, we have $F_j\in\Theta(F_0,F_y)\ints\Theta(F_m,F_y)$.
By the second assertion of Lemma \ref{reinterpretation of edges}, $Q_{j+1}$ is between $Q_j$ and $Q_{j+2}$ in the sense of Definition \ref{in between for ASep} and therefore $F_{j+1}\in\Theta(F_j, F_{j+2})$ for any $0\leq j\leq m-2$. By Lemma \ref{properties of Theta}, \hyperlink{Theta-3}{property ($\Theta$3)}, we have
\begin{align}\label{sep ordering-3}
F_j\in\Theta(F_0,F_y)\ints\Theta(F_m,F_y)\implies \Theta(F_j,F_y)\subset\Theta(F_0,F_y)\ints\Theta(F_m,F_y), ~~1\leq j\leq m-1.
\end{align}
Recall that $F_j$ are distinct elements in $\cF(V)$, $j=1,...,m-1$, by \eqref{sep ordering-3} and Lemma \ref{properties of Theta}, \hyperlink{Theta-3}{property ($\Theta$3)} again,
\begin{align}\label{sep ordering-4}
\left.
\begin{aligned}
&F_{j+1}\in\Theta(F_j, F_{j+2}),~\forall 0\leq j\leq m-2,\\
&F_s\neq F_t,~\forall 1\leq s\neq t\leq m-1,
\end{aligned}
\right\}
\implies \left\{
\begin{aligned}
\Theta(F_0,F_y)\supset\Theta(F_1,F_y)\supsetneq...\supsetneq\Theta(F_{m-1},F_y),\\
\Theta(F_m,F_y)\supset\Theta(F_{m-1},F_y)\supsetneq...\supsetneq\Theta(F_1,F_y),
\end{aligned}
\right.
\end{align}
which is impossible.
\end{proof}
\begin{definition}[Restriction of graphs]\label{res of graphs}
Let $\bG$ be a graph. For any $\cV\subset\cV(\bG)$, we define \emph{the restriction of $\bG$ onto $\cV$} as the subgraph $\bH$ such that
\begin{itemize}
\item $\cV(\bH):=\cV$;
\item For any $Q_1\neq Q_2\in \cV$, $Q_1\mbox{---}Q_2\in\cE(\bH)$ if and only if $Q_1\mbox{---}Q_2\in\cE(\bG)$;
\end{itemize}
\end{definition}
\begin{definition}[Walks and paths in a graph]\label{walks and paths}
Let $\bG$ be a graph. For any $Q_1\neq Q_2\in\cV(\bG)$, a \emph{walk $\cW$ from $Q_1$ to $Q_2$ with $k$ steps} is an ordered $(k+1)$-tuple $\cW:=(P_0,...,P_k)$ of vertices in $\cV(\bG)$ such that
\begin{itemize}
\item $P_0=Q_1$ and $P_k=Q_2$;
\item $P_j\neq P_{j+1}$ and $P_j\mbox{---}P_{j+1}\in\cE(\bG)$ for any $0\leq j\leq k-1$.
\end{itemize}
In particular, for any vertex $Q\in\cV(\bG)$, $Q$ may have multiple occurrences in this walk. The \emph{length} of the walk $\cW$ is defined to be $\length(\cW)=k$.

A subgraph $\bP\subset \bG$ is called a \emph{path connecting $Q_1$ and $Q_2$ in $\bG$} if $\cV(\bP)=\{P_0:=Q_1,P_1,...,P_m,P_{m+1}:=Q_2\}$ with distinct $P_0,...,P_{m+1}$ and $\cE(\bP)=\{P_j\mbox{---}P_{j+1}|0\leq j\leq m\}$. In this case, we use the notation
$$\bP=Q_1\mbox{---}P_1\mbox{---}\cdots\mbox{---}P_m\mbox{---}Q_2.$$
The \emph{length} of a path $\bP$ connecting $Q_1$ and $Q_2$ is defined to be $\length(\bP):=|\cE(\bP)|$. A \emph{shortest path connecting $Q_1$ and $Q_2$ in $\bG$} is a path connecting $Q_1$ and $Q_2$ in $\bG$ with the shortest possible length. The length of a shortest path connecting $Q_1$ and $Q_2$ in $\bG$ is called the \emph{distance between $Q_1$ and $Q_2$ in $\bG$}, denoted as $d_{\bG}(Q_1,Q_2)$.
\end{definition}
\begin{rmk}
One can easily show that any walk $\cW=(P_0:=Q_1,P_1,...,P_k,P_{k+1}=Q_2)$ from $Q_1$ to $Q_2$ in $\bG$ satisfies $\length(\cW)\geq d_{\bG}(Q_1,Q_2)$. Equality holds if and only if $P_j$ are distinct vertices of $\cV(\bG)$ and $P_0\mbox{---}P_1\mbox{---}\cdots\mbox{---}P_{k+1}$ is a shortest path connecting $Q_1$ and $Q_2$ in $\bG$.
\end{rmk}

\begin{notation}\label{PATH}
Let $V\subset \Gamma x_0$ be a finite subset with $|V|\geq 2$. For any distinct $p,q\in V$, we define
$$\cV_{p,q,V}=
\begin{cases}
\displaystyle \cV(\bG_V)=V,&\mathrm{if}~|\cF(V)|=1 \mathrm~{and}~ |V|=2, \\
\displaystyle \typemark{\hF}{I}{V}\in\{\cV(\bG_V)|\{p,q\}\ints I|=1\},~&\mathrm{if}~|\cF(V)|>1 \mathrm~{or}~ |V|>2.
\end{cases}$$
and $\bP_{p,q,V}$ to be the restriction of $\bG_V$ onto $\cV_{p,q,V}$.

It is easy to see that for any $\gamma\in\Gamma$ and any $Q\in\cV_{p,q,V}$, $Q\to\gamma Q$ gives a bijection from $\cV_{p,q,V}$ to $\cV_{\gamma p,\gamma q,\gamma V}$. Therefore by the remark after Definition \ref{graph of sep}, $\gamma\bP_{p,q,V}=\bP_{\gamma p,\gamma q,\gamma V}$.
\end{notation}
\begin{lemma}\label{unique shortest paths}
Let $V\subset \Gamma x_0$ be a finite subset with $|V|\geq 2$. For any distinct $p, q\in V$, let
$$Q_p:=\begin{cases}
\displaystyle p,&\mathrm{if~}|V|=2\mathrm{~and~}|\cF(V)|=1,\\
\displaystyle \typemark{F_p}{\{p\}}{V},~&\mathrm{if~}|V|>2\mathrm{~or~}|\cF(V)|>1
\end{cases}
$$
and
$$
Q_q:=\begin{cases}
\displaystyle q,&\mathrm{if~}|V|=2\mathrm{~and~}|\cF(V)|=1, \\
\displaystyle \typemark{F_q}{\{q\}}{V},~&\mathrm{if~}|V|>2\mathrm{~or~}|\cF(V)|>1.
\end{cases}$$
Then $\bP_{p,q,V}$ is the unique shortest path connecting $Q_p$ and $Q_q$ in $\bG_V$.

Moreover, for any vertices $Q_1,Q_2,Q_3\in\cV(\bP_{p,q,V})$, $Q_2$ is between $Q_1$ and $Q_3$ in the sense of Definition \ref{in between for ASep} if and only if the shortest path connecting $Q_1$ and $Q_3$ in $\bP_{p,q,V}$ contains $Q_2$.
\end{lemma}
\begin{proof}
The lemma is obviously true when $|V|=2$ and $|\cF(V)|=1$ due to Definition \ref{in between for ASep} and Definition \ref{graph of sep}. The lemma is also obviously true when $\Theta(F_p,F_q)\cap\cA_0(V)=\emptyset$ due to the first assertion in Lemma \ref{reinterpretation of edges}. Therefore we assume that $|V|>2$ or $|\cF(V)|>1$, and $\Theta(F_p,F_q)\cap\cA_0(V)\neq\emptyset$.

First we prove that $\bP_{p,q,V}$ is a path connecting $Q_p$ and $Q_q$. Since $\cV(\bG_V)=\AnPSep(V)\union\del\cA_\PSep(V)\subset\cA_\PSep(V)$ due to Definition \ref{graph of sep}, for any $Q=\typemark{\hF}{I}{V}\in\cV_{p,q,V}\subset\cV(\bG_V)$, by Lemma \ref{prim decomp}, $\typemark{\hF}{I}{V}=\typemark{\hF}{H_\hF^V(p)}{V}$ or $\typemark{\hF}{H_\hF^V(q)}{V}$. By Lemma \ref{properties of Theta}, \hyperlink{Theta-3}{property ($\Theta$3) of $\Theta(\cdot,\cdot)$}, we can assume that $\Theta(F_p,F_q)=\{F_0:=F_p,F_1,...,F_k,F_{k+1}:=F_q\}$ such that the following holds:
\begin{itemize}
\item $\Theta(F_i,F_j)=\{F_i,...,F_j\}$ for any $0\leq i\leq j\leq k+1$.
\item $F_1,...,F_k\in\Gamma F\setminus\{F_p,F_q\}$ are pairwise distinct.
\end{itemize}
If $F_p=F_q$, then by the above discussion, $\bP_{p,q,V}=Q_p\mbox{---}Q_q$ and it is a path. We now assume that $F_p\neq F_q$.

We make the following \hypertarget{claim}{claim}:
\begin{claim}
Under the above assumptions, for any $0\leq i<j\leq k+1$, $\stype{H^V_{F_i}(q)}{V}\in\PSep_V(F_i)$ and $\stype{H^V_{F_j}(p)}{V}\in\PSep_V(F_j)$ are the unique separation types mentioned in the third assertion of Lemma \ref{key lem for ASep} with respect to the pair $F_i, F_j$. In particular, $p\in (V\setminus H^V_{F_i}(q))\subset  H^V_{F_j}(p)\subset V\setminus H^V_{F_j}(q)\subset V\setminus\{q\}$.
\end{claim}
\begin{proof}[Proof of Claim]
Since for any $F_l$, $0\leq l\leq k+1$, either $F_l\in\{F_p,F_q\}$, or $F_l\in\cA_0(V)$. By Lemma \ref{prim decomp}, we have $p\not\in H^V_{F_l}(q)$ and $q\not\in H^V_{F_l}(p)$. In particular, $p\in(V\setminus H^V_{F_i}(q))\ints H^V_{F_j}(p)$ and $q\in(V\setminus\{p\})\ints H^V_{F_i}(q)\ints (V\setminus H^V_{F_j}(p))$. Therefore $\stype{\{p\}}{V}$, $\stype{H^V_{F_i}(q)}{V}$ and $\stype{H^V_{F_j}(p)}{V}$ are not in triangle relation in the sense of Definition \ref{relations between sep types}.

We first prove the following \hypertarget{lem6.26-subclaim}{sub-claim}:
\begin{center}
\emph{$\typemark{F_i}{H^V_{F_i}(q)}{V}$ is between $\typemark{F_0}{\{p\}}{V}$ and $\typemark{F_j}{H^V_{F_j}(p)}{V}$ in the sense of Definition \ref{in between for ASep}.}
\end{center}
Indeed, when $i\neq 0$, this follows directly from the first assertion in Lemma \ref{key lem for ASep}, the fact that $F_i\in\Theta(F_0,F_j)\setminus\{F_0,F_j\}$ and the above discussions. When $i=0$, let $\stype{I_0}{V}\in\PSep_V(F_0)$ be the unique separation type for $F_0$ mentioned in the third assertion of Lemma \ref{key lem for ASep} with respect to the pair $F_0,F_j$. By the third assertion of Lemma \ref{key lem for ASep} again, we assume WLOG that $\{p\}\subset I_0\subset H^V_{F_j}(p)$. In particular, $q\not\in I_0$. If $I_0\neq \{p\}$, then by Lemma \ref{prim decomp} and Definition \ref{graph of sep}, $V\setminus I_0=H^V_{F_0}(q)$. If $I_0=\{p\}$ and $H^V_{F_0}(q)\neq V\setminus \{p\}$, then by Lemma \ref{prim decomp}, $\{p\}=I_0\subsetneq V\setminus H^V_{F_0}(q)$. By the third assertion of Lemma \ref{key lem for ASep} applied to $F_0$ and $F_j$, $\stype{\{p\}}{V}$ is the only separation type of $F_j$. Let $p'\in (V\setminus H^V_{F_0}(q))\setminus \{p\}$. Then we have
$$F_0\in\Theta(F_p,F_q)\cap \Theta(F_{p'},F_q)~\mathrm{and}~F_j\in\Theta(F_p,F_q)\cap \Theta(F_p, F_{p'}). \text{ (See Figure \ref{lem6.26}).}$$
\begin{figure}[h]
	\centering
	\includegraphics[scale=0.5]{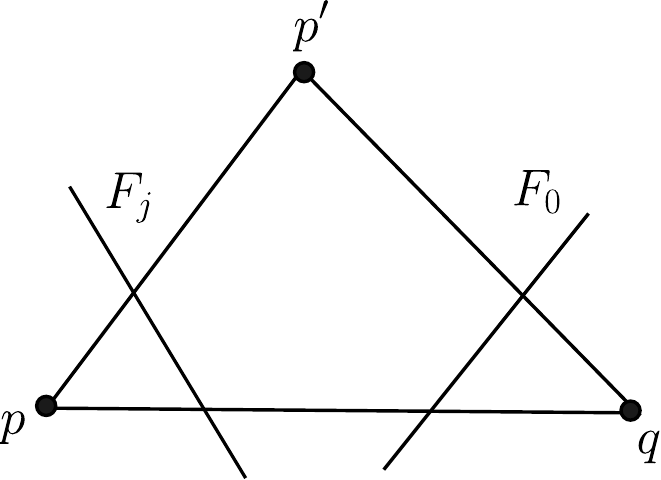}
	\caption{ \label{lem6.26}}
\end{figure}
By the second assertion in Lemma \ref{properties of Theta sep}, we have
$$d(F_0,[p,q]), d(F_0,[p',q]), d(F_j,[p,q]), d(F_j,[p,p'])\leq \epsilon_0/2.$$
Therefore by Lemma \ref{properties of Theta}, \hyperlink{Theta-3}{property ($\Theta$3) of $\Theta(\cdot,\cdot)$} and Proposition \ref{almost ints positioning}, $F_j\in\Theta(F_p,F_0)=\Theta(F_0,F_0)=\{F_0\}$. This contradicts the fact that $F_j\neq F_i=F_0$. Hence when $I_0=\{p\}$, $H^V_{F_0}(q)=V\setminus \{p\}$. This completes the proof of the \hyperlink{lem6.26-subclaim}{sub-claim} that $\typemark{F_i}{H^V_{F_i}(q)}{V}$ is between $\typemark{F_0}{\{p\}}{V}$ and $\typemark{F_j}{H^V_{F_j}(p)}{V}$ in the sense of Definition \ref{in between for ASep}.

It follows from the above that $p\in(V\setminus H^V_{F_i}(q))\subset H^V_{F_j}(p)$ and $q\in (V\setminus H^V_{F_j}(p))\subset H^V_{F_i}(q)\subset V\setminus\{p\}$. By Lemma \ref{prim decomp}, $\stype{H^V_{F_i}(q)}{V}\in\PSep_V(F_i)$ and $\stype{H^V_{F_j}(p)}{V}\in\PSep_V(F_j)$ are the unique separation types mentioned in the third assertion of Lemma \ref{key lem for ASep} with respect to the pair $F_i, F_j$

The above arguments proved the first inclusion in the \hyperlink{claim}{claim} that $p\in (V\setminus H^V_{F_i}(q))\subset  H^V_{F_j}(p)\subset V\setminus H^V_{F_j}(q)\subset V\setminus\{q\}$. The second and the third inclusion follow  immediately from Lemma \ref{prim decomp}, $p\not\in H^V_{F_l}(q)$ and $q\not\in H^V_{F_l}(p)$.
\end{proof}

Back to the proof of Lemma \ref{unique shortest paths}. 
We let $P_j^p=\typemark{F_j}{H_{F_j}^V(p)}{V}$ and $P_j^q=\typemark{F_j}{H_{F_j}^V(q)}{V}$ for any $0\leq j\leq k+1$. For any $Q', Q''\in\cV(\bG_V)$, we use the notation $Q'\sim Q''$ when either $Q'=Q''$ or $Q'\mbox{---}Q''\in\cE(\bG_V)$. Then when $F_p\neq F_q$, the first assertion in Lemma \ref{reinterpretation of edges} and the above \hyperlink{claim}{claim} imply that
$$\bP_{p,q,V}=
\begin{cases}
\displaystyle P_0^p\mbox{---}P_1^p\sim P_1^q\mbox{---}P_2^p\sim P_2^q\mbox{---}\cdots\mbox{---}P_k^p\sim P_k^q\mbox{---}P_{k+1}^q,~&\mathrm{when~}P_0^q, P_{k+1}^p\not\in\cV(\bG_V), \\
\displaystyle P_0^p\sim P_0^q\mbox{---}P_1^p\sim P_1^q\mbox{---}\cdots\mbox{---}P_k^p\sim P_k^q\mbox{---}P_{k+1}^q,~&\mathrm{when~}P_0^q\in \cV(\bG_V), P_{k+1}^p\not\in\cV(\bG_V),\\
\displaystyle P_0^p\mbox{---}P_1^p\sim P_1^q\mbox{---}P_2^p\sim P_2^q\mbox{---}\cdots\mbox{---}P_{k+1}^p\sim P_{k+1}^q,~&\mathrm{when~}P_0^q\not\in\cV(\bG_V), P_{k+1}^p\in\cV(\bG_V),\\
\displaystyle P_0^p\sim P_0^q\mbox{---}P_1^p\sim P_1^q\mbox{---}\cdots\mbox{---}P_{k+1}^p\sim P_{k+1}^q,~&\mathrm{when~}P_0^q, P_{k+1}^p\in\cV(\bG_V).
\end{cases}
$$
In particular, $\bP_{p,q,V}$ is a path connecting $Q_p=\typemark{F_p}{\{p\}}{V}=P^p_0$ and $Q_q=\typemark{F_q}{\{q\}}{V}=P^q_{k+1}$. Directly from the above analysis on $\bP_{p,q,V}$ and the above \hyperlink{claim}{claim}, for any vertices $Q_1,Q_2,Q_3\in\cV(\bP_{p,q,V})$, $Q_2$ is between $Q_1$ and $Q_3$ in the sense of Definition \ref{in between for ASep} if and only if the shortest path connecting $Q_1$ and $Q_3$ in $\bP_{p,q,V}$ contains $Q_2$.

It remains for us to prove that $\bP_{p,q,V}$ is the unique shortest path connecting $Q_p$ and $Q_q$. Let $\bP=Q_0:=Q_p\mbox{---}Q_1\mbox{---}\cdots\mbox{---}Q_m\mbox{---}Q_q=:Q_{m+1}$ be any other shortest path connecting $Q_p$ and $Q_q$ with $Q_j=\typemark{F_j'}{I_j'}{V}$. Then $Q_0,...,Q_{m+1}$ are distinct vertices. Moreover, $Q_j$ and $Q_{j+2}$ are not connected by a single edge in $\bG_V$. By Lemma \ref{sep ordering}, we can assume WLOG that $\{p\}=I_0'\subset I_1'\subset...\subset I_m'\subset I_{m+1}'=V \setminus \{q\}.$ (To be specific, by the remark after Lemma \ref{properties of Theta sep} and Lemma \ref{prim decomp}, we first assume WLOG that $I_0'=\{p\}$ and for any $0\leq j\leq m+1$, either $I_j'\subset I_{j+1}'$ or $I_{j+1}'\subset I_j'$.  Let $l$ be the smallest positive integer such that $I_j'\neq I_0'$. Since $I_j'\neq \emptyset$, $I_0'\subsetneq I_j'$. We then apply Lemma \ref{sep ordering} to $Q_{j-1},Q_j,...,Q_{m+1}$ and obtain $I_{j-1}'\subsetneq I_j'\subset...\subset I_{m+1}'$. In particular, $I_{m+1}'=V\setminus \{q\}$.) Hence $Q_j\in\cV_{p,q,V}=\cV(\bP_{p,q,V})$ for any $0\leq j\leq m+1$ and the path $Q_p\mbox{---}Q_1\mbox{---}\cdots\mbox{---}Q_m\mbox{---}Q_q$ is contained in $\bP_{p,q,V}$. Since $\bP_{p,q,V}$ is a path and therefore a tree, the path in $\bP_{p,q,V}$ connecting any $2$ vertices is unique. Therefore $Q_p\mbox{---}Q_1\mbox{---}\cdots\mbox{---}Q_m\mbox{---}Q_q=\bP_{p,q,V}$. This finishes the proof.
\end{proof}

Since $\bP_{p,q,V}$ is a tree, for any $Q_1\neq Q_2\in\cV(\bP_{p,q,V})$, there exist a unique path connecting $Q_1$ and $Q_2$ in $\bP_{p,q,V}$. We denote by $\Path_{p,q,V}(Q_1,Q_2)$ this unique path connecting $Q_1$ and $Q_2$ in $\bP_{p,q,V}$. In particular,
$$\bP_{p,q,V}=\begin{cases}
\Path_{p,q,V}(p,q),&~\text{if } |V|=2\text{ and }|\cF(V)|=1,\\
\Path_{p,q,V}(\typemark{F_p}{\{p\}}{V},\typemark{F_q}{\{q\}}{V}),&~\text{otherwise}.
\end{cases}$$
We have the following immediate corollary.
\begin{corollary}\label{remaining properties of sep graph}
Let $V\subset \Gamma x_0$ be a finite subset with $|V|\geq 2$.
\begin{enumerate}
\item[(1).] {Assume that $|V|>2$ or $|\cF(V)|>1$.} For any distinct $Q_1, Q_2\in\cV(\bG_V)$ with $Q_j=\typemark{F_j}{I_j}{V}$, $j=1,2$, WLOG we assume that $I_1\subset I_2$ (guaranteed by the third assertion in Lemma \ref{properties of Theta sep}). Then, $Q_1$ and $Q_2$ are connected by a unique shortest path. Moreover, for any $p\in I_1$ and $q\in V\setminus I_2$, this shortest path equals $\Path_{p,q,V}(Q_1, Q_2)$.
\item[(2).] All MCS of $\bG_V$ have at least $2$ vertices. Also, for any $Q\in\cV(\bG_V)$, $Q$ is contained in exactly 1 MCS of $\bG_V$ if and only if {$$Q\in\begin{cases}
\displaystyle \cV(\bG_V)=V,&\mathrm{if}~|V|=2\mathrm{~and~}|\cF(V)|=1, \\
\displaystyle \del\cA_\Sep(V)=\{\typemark{F_p}{\{p\}}{V}|p\in V\},~&\mathrm{if}~|V|>2\mathrm{~or~}|\cF(V)|>1.
\end{cases}$$}
\item[(3).] $\bG_V=\displaystyle\bigcup_{\substack{p,q:p,q\in V\\ p\neq q}}\bP_{p,q,V}$. In particular, $\bG_V$ is a connected graph (i.e. every two distinct vertices are joined by a path in the graph);
\item[(4).] For any $G\in\MCS(\bG_V)$ and any distinct $p, q\in V$, either $G\ints \bP_{p,q,V}=\emptyset$, or $G\ints\bP_{p,q,V}$ contains exactly 2 vertices (and hence exactly one edge).
\item[(5).] {Assume that $|V|>2$ or $|\cF(V)|>1$.} For any $G\in\MCS(\bG_V)$ whose vertices are given by $Q_j=\typemark{F_j}{I_j}{V}$, $1\leq j\leq k$, (assuming all $Q_j$'s are distinct), there exist $I_j'\in\stype{I_j}{V}$, $1\leq j\leq k$ such that $I_1',...,I_k'$ are pairwise disjoint and $\union_{j=1}^kI_j'=V$. {In particular, $|\cV(G)|\leq |V|$.}
\end{enumerate}
\end{corollary}
\begin{proof}
\begin{enumerate}
\item[(1).] Let $\bP$ be a shortest path connecting $Q_1$ and $Q_2$ in $\bG_V$, since $Q_1,Q_2\in\cV(\bP_{p,q,V})$ and $\bP_{p,q,V}$ is a path connecting $Q_p,Q_q$ , WLOG we assume that $\Path_{p,q,V}(Q_p, Q_1)$ and $\Path_{p,q,V}(Q_2, Q_q)$ are disjoint. Then by Lemma \ref{unique shortest paths} (used in the first equality), we have
\begin{align*}
&\length(\Path_{p,q,V}(Q_p, Q_1))+\length(\Path_{p,q,V}(Q_1,Q_2))+\length(\Path_{p,q,V}(Q_2, Q_q))\\
=&d(Q_p,Q_q)
\leq \length(\Path_{p,q,V}(Q_p, Q_1))+\length(\bP)+\length(\Path_{p,q,V}(Q_p, Q_1)).
\end{align*}
Since $\bP$ be a shortest path connecting $Q_1$ and $Q_2$ in $\bG_V$, the above inequality sign is actually equal. By Lemma \ref{unique shortest paths} again, we have $\bP=\Path_{p,q,V}(Q_1,Q_2)$.
\item[(2).] {The assertion is obviously true if $|V|=2$ and $|\cF(V)|=1$. Therefore we assume that $|V|>2$ or $|\cF(V)|>1$.} For any $Q=\typemark{\hF}{I}{V}\in\cV(\bG_V)$, let $p\in I$ and $q\in V\setminus I$. Then $Q\in\cV(\bP_{p,q,V})$. In particular, $Q$ is not an isolated vertex (i.e. a vertex which is not an endpoint of any edge). Therefore any MCS of $\bG_V$ has at least $2$ vertices.

If $Q\not\in\del\cA_\Sep(V)=\{\typemark{F_p}{\{p\}}{V}|p\in V\}$, there exist $Q_1\neq Q_2\in\cV(\bP_{p,q,V})$ such that $Q\mbox{---}Q_1,Q\mbox{---}Q_2\in\cE(\bP_{p,q,V})$. Let $G_j\in\MCS(\bG_V)$ such that $Q\mbox{---}Q_j\in\cE(G_j)$, $j=1,2$. Notice that by Lemma \ref{unique shortest paths}, $Q_1$ and $Q_2$ are not connected by a single edge in $\bG_V$, we have $G_1\neq G_2$. Hence by the second assertion in Proposition \ref{key prop of ASep graph}, $Q$ is contained in exactly $2$ MCS of $\bG_V$.

If $Q\in\del\cA_\Sep(V)=\{\typemark{F_p}{\{p\}}{V}|p\in V\}$ and there exist $2$ different MCS $G_1\neq G_2$ of $\bG_V$ containing $Q$, then we can find $Q_j\in\cV(G_j)\setminus\{Q\}$, $j=1,2$ such that $Q_1$ and $Q_2$ are not connected by a single edge in $\bG_V$. (Otherwise, the restriction of $\bG_V$ onto $\cV(G_1)\union \cV(G_2)$ is a complete subgraph of $\bG_V$. By maximality of MCS, $\cV(G_1)=\cV(G_2)$ and hence $G_1=G_2$, which contradicts $G_1\neq G_2$.) By the second assertion of Lemma \ref{reinterpretation of edges}, $Q$ is between $Q_1$ and $Q_2$ in the sense of Definition \ref{in between for ASep}. By Definition \ref{bdry} and the fact that $Q\in\del\cA_\Sep(V)$, $Q=Q_1$ or $Q_2$, which contradicts $Q_j\in\cV(G_j)\setminus\{Q\}$, $j=1,2$. Therefore $Q$ is contained in exactly $1$ MCS of $\bG_V$.
\item[(3).] {The assertion is obviously true if $|V|=2$ and $|\cF(V)|=1$. Therefore we assume that $|V|>2$ or $|\cF(V)|>1$.} For any $Q=\typemark{\hF}{I}{V}\in\cV(\bG_V)$, let $p\in I$ and $q\in V\setminus I$. Then $Q\in\cV(\bP_{p,q,V})$. Hence $\cV(\bG_V)=\displaystyle\bigcup_{\substack{p,q:p,q\in V\\ p\neq q}}\bP_{p,q,V}$.

For any $Q_1\mbox{---}Q_2\in\cE(\bG_V)$ with $Q_j=\typemark{F_j}{I_j}{V}$, $j=1,2$, by the third assertion in Lemma \ref{properties of Theta sep} and Lemma \ref{prim decomp}, we can assume WLOG that $I_1\subset I_2$. Choose $p\in I_1$ and $q\in V\setminus I_2$, then $Q_1\mbox{---}Q_2\in\cE(\bP_{p,q,V})$. This proves $\cE(\bG_V)=\bigcup_{\substack{p,q:p,q\in V\\ p\neq q}}\cE(\bP_{p,q,V})$. Hence $\bG_V=\displaystyle\bigcup_{\substack{p,q:p,q\in V\\ p\neq q}}\bP_{p,q,V}=\displaystyle\bigcup_{\substack{p,q:p,q\in V\\ p\neq q}} \Path_{p,q,V}(Q_p,Q_q)$.

\item[(4).] {The assertion is obviously true if $|V|=2$ and $|\cF(V)|=1$. Therefore we assume that $|V|>2$ or $|\cF(V)|>1$.} Let $G\in\MCS(\bG_V)$ and $p\neq q\in V$ such that $\cV(G\ints\bP_{p,q,V})\neq \emptyset$. Then by the fact that $\bP_{p,q,V}$ is a shortest path in $\bG_V$ (due to Lemma \ref{unique shortest paths}), $G\ints\bP_{p,q,V}$ contains at most 2 vertices. Moreover, when $G\ints\bP_{p,q,V}$ contains exactly 2 vertices, it also contains exactly 1 edge. Let $Q\in\cV(G\cap\bP_{p,q,V})$.

If in addtion, $Q\in\del\cA_{\Sep}(V)=\{\typemark{F_x}{\{x\}}{V}|x\in V\}$ (due to Lemma \ref{bdry=vertex}). Then by Lemma \ref{unique shortest paths}, $Q=\typemark{F_p}{\{p\}}{V}$ or $\typemark{F_q}{\{q\}}{V}$. Let $Q_1\in\cV(\bP_{p,q,V})\setminus\{Q\}$ be the unique vertex such that $Q_1\mbox{---} Q\in\cE(\bP_{p,q,V})$. By the first assertion in Proposition \ref{key prop of ASep graph}, there exist a unique $G'\in \MCS(\bG_V)$ containing $Q_1\mbox{---} Q$. By the second assertion in Corollary \ref{remaining properties of sep graph}, $G'=G$. In particular, $|\cV(G\ints\bP_{p,q,V})|= 2$ and $\cE(G\ints\bP_{p,q,V})|= 1$.

If in addition, $Q\not\in\del\cA_{\Sep}(V)=\{\typemark{F_x}{\{x\}}{V}|x\in V\}$, let $Q_j\in\cV(\bP_{p,q,V})\setminus\{Q\}$, $j=1,2$, be the only 2 vertices such that $Q_j\mbox{---} Q\in\cE(\bP_{p,q,V})$, $j=1,2$. Notice that by the first assertion in Proposition \ref{key prop of ASep graph}, for any $1\leq j\leq 2$,  there exist a unique $G_j'\in \MCS(\bG_V)$ containing $Q_j\mbox{---} Q$. By Lemma \ref{unique shortest paths}, $Q_1$ and $Q_2$ are not connected by a single edge in $\bG_V$. Therefore $G_1'\neq G_2'$. By the second assertion in Corollary \ref{remaining properties of sep graph} and the second assertion in Proposition \ref{key prop of ASep graph}, $G\in\{G_1',G_2'\}$. In particular, $|\cV(G\ints\bP_{p,q,V})|= 2$ and $|\cE(G\ints\bP_{p,q,V})|= 1$. This completes the proof.

\item[(5).] We first prove that there exist $I_j'\in\stype{I_j}{V}$, $1\leq j\leq k$, such that they are pairwise disjoint. When $k=2$, this is guaranteed by the third assertion in Lemma \ref{properties of Theta sep}. When $k\geq 3$, by the second assertion of Lemma \ref{reinterpretation of edges}, for any $3\leq j\leq k$, there exist $I_{1,j}'\in\stype{I_1}{V}$, $I_{2,j}'\in\stype{I_2}{V}$ and $I_j'\in\stype{I_j}{V}$ such that $I_{1,j}',I_{2,j}'$ and $I_j'$ are pairwise disjoint. In particular, $I_{1,j}'\subsetneq V\setminus I_{2,j}'$ and $I_{2,j}'\subsetneq V\setminus I_{1,j}'$. This implies that $I_{1,3}'=I_{1,4}'=...=I_{1,k}'=:I_1'$ and $I_{2,3}'=I_{2,4}'=...=I_{2,k}'=:I_2'$ (due to the uniqueness of the choices of $I_{1}'\in\stype{I_1}{V}$ and $I_{2}'\in\stype{I_2}{V}$ such that $I_1'\ints I_2'=\emptyset$ and $I_1'\subsetneq V\setminus I_2'$). Then for any $3\leq j\neq l\leq k$, by the second assertion of Lemma \ref{reinterpretation of edges}, there exist $I_{1,jk}'\in\stype{I_1}{V}$, $I_{j,jk}'\in\stype{I_j}{V}$ and $I_{k,jk}'\in\stype{I_k}{V}$ such that $I_{1,jk}'$, $I_{j,jk}'$ and $I_{k,jk}'$ are pairwise disjoint. By similar arguments, we have $I_1'=I_{1,jk}'$, $I_j'=I_{j,jk}'$ and $I_k'=I_{k,jk}'$. Therefore there exist $I_j'\in\stype{I_j}{V}$, $1\leq j\leq k$, such that $I_1',...,I_k'$ are pairwise disjoint. WLOG, we assume that $I_j=I_j'$ for any $1\leq j\leq k$.

It remains for us to prove that $V=\union_{j=1}^k I_j$. If not, then there exist $x$ such that $x\in V\setminus(\union_{j=1}^k I_j)=\ints_{j=1}^k (V\setminus I_j)$. Choose $p\in I_1$. Then $p\neq x$ and $\{p,x\}\in \ints_{j=2}^k (V\setminus I_j)$. This implies that $Q_2,...,Q_k\not\in\cV(\bP_{p,x,V})$. Notice that $Q_1\in\cV(\bP_{p,x,V})$, we have $\cV(G\ints\bP_{p,x,V})=\{Q_1\}$, contradicting the fourth assertion of Corollary \ref{remaining properties of sep graph}. Hence $V=\union_{j=1}^k I_j$. \qedhere
\end{enumerate}
\end{proof}
{
\begin{rmk}
By the remark after Definition \ref{graph of sep} and the first assertion in the above corollary, for any $\gamma\in\Gamma$ and any $Q_1,Q_2\in\cV(\bG_V)$, $\gamma \Path_{p,q,V}(Q_1,Q_2)=\Path_{\gamma p,\gamma q,\gamma V}(\gamma Q_1,\gamma Q_2)$.
\end{rmk}
}

\subsection{Relations between $\bG_V$ and $\bG_W$ when $W\subset V$}\label{subsec most annoying}
\begin{definition}[Restriction of $\Theta$-separation types]\label{res of types}
Let $W\subset V$ be finite subsets of $\Gamma x_0$ such that $|W|\geq 2$ {and $|V|\geq 3$.} We define
$$\cA_\Sep(V;W)=\{\typemark{\hF}{I}{V}\in\cA_\Sep(V)|I\ints W\neq \emptyset\mathrm{~and~}W\setminus I\neq \emptyset\},$$
$$\cA_\PSep(V;W)=\{\typemark{\hF}{I}{V}\in\cA_\PSep(V)|I\ints W\neq \emptyset\mathrm{~and~}W\setminus I\neq \emptyset\},$$
$$\AnSep(V;W)=\{\typemark{\hF}{I}{V}\in\AnSep(V)|I\ints W\neq \emptyset\mathrm{~and~}W\setminus I\neq \emptyset\},$$
and
$$\AnPSep(V;W)=\{\typemark{\hF}{I}{V}\in\AnPSep(V)|I\ints W\neq \emptyset\mathrm{~and~}W\setminus I\neq \emptyset\}.$$
See Notation \ref{sep type marking} for the relevant notations. Clearly $\AnPSep(V;W)\subset \AnSep(V;W), \cA_\PSep(V;W)\subset \cA_\Sep(V;W)$.

{When $|W|>2$ or $|\cF(W)|>1$}, we define the restriction map $\res_W^V:\cA_\Sep(V;W)\to\cA_\Sep(W)$ as
$$\res_W^V(\typemark{\hF}{I}{V})=\rtypemark{\hF}{I}{W}.$$

{When $|W|=2$ and $|\cF(W)|=1$, we assume that $W=\{p,q\}$, where $p\neq q\in\Gamma x_0$. For any $\typemark{\hF}{I}{V}\in\cA_\PSep(V;W)$, we have $\hF=F_p=F_q$. Since $\stype{\{p\}}{V},\stype{\{q\}}{V}\in\PSep(F_p)$, by Lemma \ref{prim decomp}, $\cA_\PSep(V;W)=\{\typemark{F_p}{\{p\}}{V},\typemark{F_p}{\{q\}}{V}\}$. In this case, we define $\res_W^V:\cA_\PSep(V;W)\to W$ such that
$$\res_W^V(\typemark{F_p}{\{p\}}{V})=p\mathrm{~and~}\res_W^V(\typemark{F_p}{\{q\}}{V})=q.$$}
\end{definition}
{
\begin{rmk}
\begin{enumerate}
\item Eventually, we would like to show that for any vertex $Q\in\bG_V$ such that $\res_W^V(Q)$ is well defined, $\res_W^V(Q)$ is a vertex in $\bG_W$. (See Lemma \ref{res of vertices}.) This explains why we define $\res_W^V$ in a very different way when $|W|=2$ and $|\cF(W)|=1$.
\item Under the above assumptions, by the third remark after Notation \ref{sep type marking}, one can verify that for any $\gamma\in\Gamma$ and any $Q\in\cA_\Sep(V;W)$, the map $Q\to\gamma Q$ satisfies the following.
\begin{itemize}
\item It gives a bijection from $\cA_\Sep(V;W)$ to $\cA_\Sep(\gamma V;\gamma W)$.
\item It gives a bijection from $\cA_\PSep(V;W)$ to $\cA_\PSep(\gamma V;\gamma W)$.
\item It gives a bijection from $\AnSep(V;W)$ to $\AnSep(\gamma V;\gamma W)$.
\item It gives a bijection from $\AnPSep(V;W)$ to $\AnPSep(\gamma V;\gamma W)$.
\item $\gamma\res_W^V(Q)=\res_{\gamma W}^{\gamma V}(\gamma Q)$, whenever $\res_W^V(Q)$ is well-defined.
\end{itemize}
\end{enumerate}
\end{rmk}
}
\begin{definition}[$W$-face of $\bG_V$]\label{W-face of graph}
Let $W\subset V$ be finite subsets of $\Gamma x_0$ such that $|W|\geq 2$ {and $|V|\geq 3$.} We define the \emph{$W$-face of $\bG_V$}, denoted by $\bG_{V;W}$, as the restriction of $\bG_V$ onto $\cV(\bG_V)\ints\cA_\PSep(V;W)$.

{In particular, for any $\gamma\in \Gamma$, by the remark after Definition \ref{res of types}, $\gamma\bG_{V;W}=\bG_{\gamma V;\gamma W}$.}
\end{definition}
\begin{lemma}\label{res of vertices}
Let $W\subset V$ be finite subsets of $\Gamma x_0$ such that $|W|\geq 2$ {and $|V|\geq 3$}. Then the following holds.
\begin{enumerate}
\item[(1).] {For any $\hF\in\cA_0(V)\ints\cF(W)$ and any $p\in W$, we have $\hF\in\cA_0(W)$ and $H_\hF^W(p)=H_\hF^V(p)\ints W$}. As a direct corollary of this, when $|W|>2$ or $|\cF(W)|>1$, $\res_W^V(\AnPSep(V;W))\subset \AnPSep(W)$.
\item[(2).] $\del \cA_\PSep(V;W)=\del \cA_\Sep(V;W)=\del\cA_\Sep(V)\ints\cA_\Sep(V;W)=\{\typemark{F_p}{\{p\}}{V}|p\in W\}$. Therefore, when $|W|>2$ or $|\cF(W)|>1$, we have
$$\res_W^V(\del \cA_\PSep(V;W))=\res_W^V(\del \cA_\Sep(V;W))=\del \cA_\PSep(W)=\del\cA_\Sep(W).$$
\end{enumerate}
As an immediate corollary of the first two assertions, {Definition \ref{res of types} in the case when $|W|=2$ and $|\cF(W)|=1$} and Definition \ref{graph of sep}, $\res_W^V(\cV(\bG_{V;W}))\subset\cV(\bG_W)$. Moreover, the following also holds.
\begin{enumerate}
\item[(3).] For any $Q_j=\typemark{F_j}{I_j}{V}\in\cV(\bG_{V;W})$, $j=1,2,3$, $\res_W^V(Q_2)$ is between $\res_W^V(Q_1)$ and $\res_W^V(Q_3)$ in the sense of Definition \ref{in between for ASep} if and only if one of the following holds:
\begin{itemize}
\item $Q_2$ is between $Q_1$ and $Q_3$ in the sense of Definition \ref{in between for ASep}.
\item $\res_W^V(Q_2)\in\{\res_W^V(Q_1),\res_W^V(Q_3)\}$.
\end{itemize}
\end{enumerate}
\end{lemma}
\begin{proof}
\begin{enumerate}
\item[(1).] {It suffices to verify the following two assertions.
\begin{itemize}
\item For any $\hF\in\cA_0(V)\ints\cF(W)$, we have $\hF\in \cA_0(W)$; (See Notation \ref{sep type marking}.)
\item Under the above assumptions, for any $p\in W$, we have $H_\hF^W(p)=H_\hF^V(p)\ints W$.
\end{itemize}
}

We first verify that $\hF\in \cA_0(W)$. $\hF\in\cF(W)$ implies that there exist distinct $p,q\in W$ such that $\hF\in\Theta(F_p,F_q)$. Recall Notation \ref{edge marking} and Notation \ref{sep type marking}, for any distinct $p, q\in W$ such that $\hF\in\Theta(F_p,F_q)\Leftrightarrow (\hF,\{p,q\})\in\widetilde\cF(W)$, by the fact that $\hF\in\cA_0(V)$, there exist $\stype{I'}{V}\in\Sep_V(\hF)$ such that $p\in I'$ and $q\in V\setminus I'$. Hence $p\in I'\ints W$ and $q\in W\setminus I'$. Moreover, $\stype{I'}{V}\in\Sep_V(\hF)$ implies that $\rstype{I'}{W}\in\Sep _W(\hF)$ and therefore $(\hF,\{p,q\})\in\widetilde\cA(W)$. (See Notation \ref{edge marking}.) This proves that $\hF\in \cA_0(W)$.

{By Lemma \ref{prim decomp}, we always have $H^W_\hF(p)\subset H^V_\hF(p)\ints W$. If $H^W_\hF(p)\subsetneq H^V_\hF(p)\ints W$, we choose $q\in (H^V_\hF(p)\ints W)\setminus H^W_\hF(p)$. Then there exist $\stype{J}{W}\in\Sep_W(\hF)$ such that $p\in J$ and $q\in W\setminus J$. In particular, $\hF\in\Theta(F_p,F_q)$. By the fact that $\hF\in\cA_0(V)$, there exists $\stype{I'}{V}$ such that $p\in I'$ and $q\in V\setminus I'$. By the definition of $H^V_\hF(\cdot)$ in Lemma \ref{prim decomp}, we have $q\not\in H^V_\hF(p)$, contradictory to $q\in (H^V_\hF(p)\ints W)\setminus H^W_\hF(p)$. Therefore $H^W_\hF(p)= H^V_\hF(p)\ints W$.}
\item[(2).] By Definition \ref{bdry} and Lemma \ref{bdry=vertex}, we have
$$\del \cA_\PSep(V;W),\del \cA_\Sep(V;W)\supset\del\cA_\Sep(V)\ints\cA_\Sep(V;W)=\{\typemark{F_p}{\{p\}}{V}|p\in W\}.$$
On the other hand, for any element $\typemark{\hF}{I}{V}$ in $\del\cA_\Sep(V;W)$ or $\del\cA_\PSep(V;W)$, we choose $x\in I\ints W$ and $q\in W\setminus I$. Then $\typemark{\hF}{I}{V}$ is between $\typemark{F_x}{\{x\}}{V}$ and $\typemark{F_q}{\{q\}}{V}$ in the sense of Definition \ref{in between for ASep}. Since both $\typemark{F_x}{\{x\}}{V}$ and $\typemark{F_q}{\{q\}}{V}$ are in $\cA_\PSep(V;W)\ints \cA_\Sep(V;W)$, by Definition \ref{bdry}, $\typemark{\hF}{I}{V}\in \{\typemark{F_p}{\{p\}}{V}|p\in W\}$. This proves that
$$\del \cA_\PSep(V;W)=\del \cA_\Sep(V;W)=\del\cA_\Sep(V)\ints\cA_\Sep(V;W)=\{\typemark{F_p}{\{p\}}{V}|p\in W\}.$$
\item[(3).] When $|W|=2$ and $|\cF(W)|=1$, we have $\cV(\bG_W)=W$. One can easily check that the second bullet point holds if and only if $\res_W^V(Q_2)$ is between $\res_W^V(Q_1)$ and $\res_W^V(Q_3)$ in the sense of Definition \ref{in between for ASep}. Therefore, we only need to consider the case when $|W|>2$ or $|\cF(W)|>1$.

If $Q_2$ is between $Q_1$ and $Q_3$ in the sense of Definition \ref{in between for ASep}, then $F_2\in\Theta(F_1,F_3)$ and we can assume WLOG that $I_1\subset I_2\subset I_3$. Hence $(I_1\ints W)\subset (I_2\ints W)\subset (I_3\ints W)$. This shows that $\res_W^V(Q_2)$ is between $\res_W^V(Q_1)$ and $\res_W^V(Q_3)$ in the sense of Definition \ref{in between for ASep}. Clearly, $\res_W^V(Q_2)$ is between $\res_W^V(Q_1)$ and $\res_W^V(Q_3)$ in the sense of Definition \ref{in between for ASep} when $\res_W^V(Q_2)\in\{\res_W^V(Q_1),\res_W^V(Q_3)\}$.

On the other hand, if $\res_W^V(Q_2)$ is between $\res_W^V(Q_1)$ and $\res_W^V(Q_3)$ in the sense of Definition \ref{in between for ASep} and $\res_W^V(Q_2)\not\in\{\res_W^V(Q_1),\res_W^V(Q_3)\}$, we have $Q_1$, $Q_2$ and $Q_3$ are distinct, $F_2\in\Theta(F_1,F_3)$ and we can assume WLOG that $(I_1\ints W)\subset (I_2\ints W)\subset (I_3\ints W)$. Choose $p\in I_1\ints W$ and $q\in W\setminus I_3$, we have
\begin{align}\label{eqn:lem6.30-1}
p\in I_1\ints I_2\ints I_3\ints W \text{ and } q\in (V\setminus I_1)\ints(V\setminus I_2)\ints(V\setminus I_3)\ints W.
\end{align}
This proves that $\stype{I_1}{V}$, $\stype{I_2}{V}$ and $\stype{I_3}{V}$ are \textbf{NOT} in triangle relation in the sense of Definition \ref{relations between sep types}. We split the rest into 3 cases:

\textbf{Case 1}: If $F_2\not\in\{F_1,F_3\}$, then $F_1,F_2,F_3$ are pairwise distinct and $F_2\in\cA_0(V)$ (due to Definition \ref{graph of sep}). By the first assertion in Lemma \ref{key lem for ASep}, $Q_2$ is between $Q_1$ and $Q_3$ in the sense of Definition \ref{in between for ASep}.

\textbf{Case 2}: If $F_2=F_1=F_3$, by Lemma \ref{prim decomp} and \eqref{eqn:lem6.30-1}, we have $ Q_1,Q_2,Q_3\in\{\typemark{F_1}{H_{F_1}^V(p)}{V},\typemark{F_1}{H_{F_1}^V(q)}{V}\}\subset\cV(\bG_{V;W})\subset\cA_\PSep(V;W)$. This contradicts the assumption that $Q_1$, $Q_2$ and $Q_3$ are distinct.

\textbf{Case 3}: If $F_2\in\{F_1,F_3\}$ and $F_1\neq F_3$, WLOG we assume that $F_1=F_2\neq F_3$. Since $\res_W^V(Q_2)\not\in\{\res_W^V(Q_1),\res_W^V(Q_3)\}$, and $\res_W^V(Q_2)$ is between $\res_W^V(Q_1)$ and $\res_W^V(Q_3)$ in the sense of Definition \ref{in between for ASep}, by Lemma \ref{properties of Theta}, \hyperlink{Theta-3}{property ($\Theta$3) of $\Theta(\cdot,\cdot)$}, one of the following 2 senarios hold:

\textbf{Case 3a}: If $F_3\in\Theta(F_p,F_1)\setminus\{F_1\}$, by the assumptions on $p,q$, Lemma \ref{prim decomp} (used in the first and last inclusion of \eqref{inclusion-1W} and \eqref{inclusion-1V}) and the \hyperlink{claim}{claim} in the proof of Lemma \ref{unique shortest paths} (used in the middle inclusion of \eqref{inclusion-1W} and \eqref{inclusion-1V}),
\begin{align}\label{inclusion-1W}
H_{F_3}^W(p)\subset W\setminus H_{F_3}^W(q)\subset H_{F_1}^W(p)\subset W\setminus H_{F_1}^W(q)
\end{align}
and
\begin{align}\label{inclusion-1V}
H_{F_3}^V(p)\subset V\setminus H_{F_3}^V(q)\subset H_{F_1}^V(p)\subset V\setminus H_{F_1}^V(q).
\end{align}
Notice that $Q_1,Q_2,Q_3\in\cV(\bG_V)\subset\cA_\PSep(V;W)$ (due to Definition \ref{graph of sep}). By Lemma \ref{prim decomp}, the first two assertions in Lemma \ref{res of vertices} and \eqref{eqn:lem6.30-1}, we have
\begin{align}\label{res Q1Q2W}
\res_W^V(Q_1),\res_W^V(Q_2)\in\{\typemark{F_1}{H_{F_1}^W(p)}{W},\typemark{F_1}{H_{F_1}^W(q)}{W}\}
\end{align}
and
$$\res_W^V(Q_3)\in\{\typemark{F_3}{H_{F_3}^W(p)}{W},\typemark{F_3}{H_{F_3}^W(q)}{W}\}$$
Since $\res_W^V(Q_2)$ is between $\res_W^V(Q_1)$ and $\res_W^V(Q_3)$ in the sense of Definition \ref{in between for ASep} and $\res_W^V(Q_2)\not\in\{\res_W^V(Q_1),\res_W^V(Q_3)\}$, by \eqref{inclusion-1W}, we have
\begin{align}\label{eqn:locate res(Q)}
\res_W^V(Q_1)=\typemark{F_1}{H_{F_1}^W(q)}{W}\neq \res_W^V(Q_2)=\typemark{F_1}{H_{F_1}^W(p)}{W}.
\end{align}
Moreover, we have $H_{F_1}^W(p)\subsetneq W\setminus H_{F_1}^W(q)$. Recall that $F_3\in\Theta(F_p,F_1)\setminus\{F_1\}$, we have $F_1\neq F_p$. In particular, by \eqref{eqn:lem6.30-1}, the fact that $Q_1\neq Q_2$ and $F_1=F_2$, we have $F_1\not\in\del\cA_\Sep(V)$. It follows from Definition \ref{graph of sep} and \eqref{eqn:lem6.30-1} that $F_1\in\cA_0(V)\cap\cF(W)$. By the first assertion in Lemma \ref{res of vertices}, we have $H_{F_1}^W(p)=H_{F_1}^V(p)\ints W$ and $H_{F_1}^W(q)=H_{F_1}^V(q)\ints W$. Hence
$$H_{F_1}^V(q)\ints W=H_{F_1}^W(q)\subsetneq W\setminus H_{F_1}^W(p)\mathrm{~and~}W\setminus H_{F_1}^V(q)=W\setminus H_{F_1}^W(q)\supsetneq H_{F_1}^W(p).$$
This and \eqref{eqn:locate res(Q)} imply that $\res_W^V\typemark{F_1}{H^V_{F_1}(q)}{V}\neq \res_W^V(Q_2)$. Similar to \eqref{res Q1Q2W}, \eqref{eqn:lem6.30-1} implies that
\begin{align}\label{Q1Q2V}
Q_1,Q_2\in\{\typemark{F_1}{H_{F_1}^V(p)}{V},\typemark{F_1}{H_{F_1}^V(q)}{V}\}
\end{align}
and
$$Q_3 \in\{\typemark{F_3}{H_{F_3}^V(p)}{V},\typemark{F_3}{H_{F_3}^V(q)}{V}\}.$$
It follows from \eqref{eqn:locate res(Q)} and the above discussion that
$$Q_1=\typemark{F_1}{H_{F_1}^V(q)}{V}\neq Q_2=\typemark{F_1}{H_{F_1}^V(p)}{V}.$$
By \eqref{inclusion-1V}, $Q_2$ is between $Q_1$ and $Q_3$ in the sense of Definition \ref{in between for ASep}.

\textbf{Case 3b}: If $F_1\in\Theta(F_p,F_3)\setminus\{F_3\}$, this is essentially the same as \textbf{Case 3a}. In this case, the same arguments in  \textbf{Case 3a} can show that
$$H_{F_1}^V(p)\subset V\setminus H_{F_1}^V(q)\subset H_{F_3}^V(p)\subset V\setminus H_{F_3}^V(q) .$$
$$Q_1=\typemark{F_1}{H_{F_1}^V(p)}{V}\neq Q_2=\typemark{F_1}{H_{F_1}^V(q)}{V}$$
and
$$Q_3 \in\{\typemark{F_3}{H_{F_3}^V(p)}{V},\typemark{F_3}{H_{F_3}^V(q)}{V}\}.$$
Hence $Q_2$ is between $Q_1$ and $Q_3$ in the sense of Definition \ref{in between for ASep}.\qedhere

\end{enumerate}
\end{proof}
\begin{figure}[h]
	\centering
	\includegraphics[width=4in]{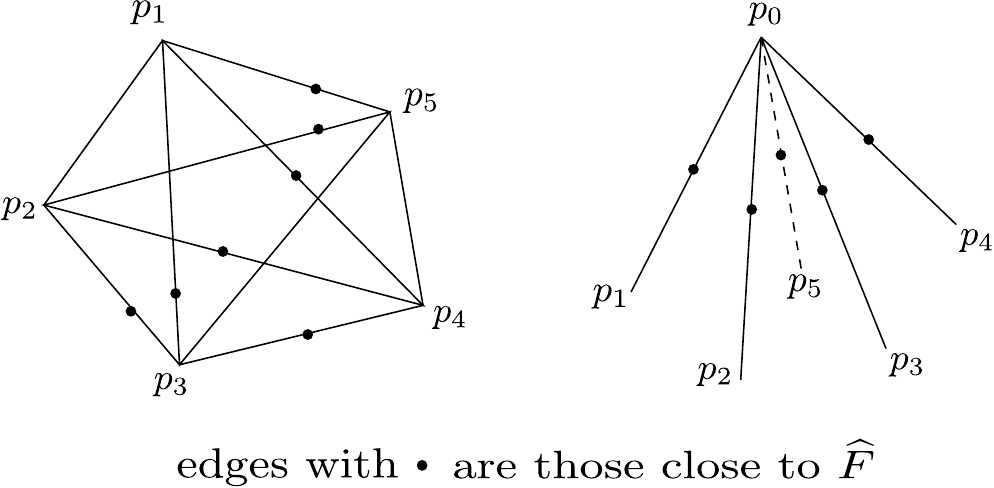}
	\caption{ \label{pic:res not PSep}}
\end{figure}
\begin{rmk}
The reason we introduce the notion of $\cA_0,\AnSep$ and $\AnPSep$ is that $\res_W^V(\cA_\PSep(V;W))$ may not be contained in $\cA_\PSep(W)$. See the following possible example.
\begin{pexample}
{Let $V=\{p_0,p_1,p_2,p_3,p_4,p_5\}\subset \Gamma x_0$, where $p_0,...,p_5$ are pairwise distinct. If there exists some $\hF\in\cF(V)$ such that $\hF\in\Theta(F_{p_i},F_{p_j})$ if and only if one of the following holds:
\begin{itemize}
\item $\{i,j\}=\{3,4\}$.
\item $\{i,j\}=\{0,l\}$ for some $l\in\{1,2,3,4,5\}$.
\item $|\{i,j\}\cap\{1,2\}|=|\{i,j\}\cap\{3,4,5\}|=1$.
\end{itemize}
Let $W=\{p_0,...,p_4\}$. Then one can check that $(\hF, \{\{p_0,p_1,p_2\},\{p_3,p_4,p_5\}\})\in \cA_\PSep(V;W)$, but
$$\res_W^V(\hF, \{\{p_0,p_1,p_2\},\{p_3,p_4,p_5\}\})=(\hF, \{\{p_0,p_1,p_2\},\{p_3,p_4\}\})\in\cA_\Sep(W)\setminus\cA_\PSep(W).$$ See Figure \ref{pic:res not PSep}.}
\end{pexample}
\end{rmk}

\begin{lemma}[Properties of $\bG_{V;W}$]\label{properties of W-face}
Let $W\subset V$ be finite subsets of $\Gamma x_0$ such that $|W|\geq 2$ and $|V|\geq 3$. Then $\bG_{V;W}$ satisfies the following properties:
\begin{enumerate}
\item[(1).] $\bG_{V;W}=\displaystyle\bigcup_{\substack{p,q:p,q\in W\\p\neq q}}\bP_{p,q,V}$. In particular, $\bG_{V;W}$ is a connected graph.
\item[(2).] For any $2$ different vertices $Q_1,Q_2$ of $\bG_{V;W}$, $Q_1$ and $Q_2$ are connected by a unique shortest path in $\bG_{V;W}$. Moreover, this unique shortest path in $\bG_{V;W}$ is the same as the unique shortest path connecting $Q_1$ and $Q_2$ in $\bG_V$.
\item[(3).] For any $\GVW\in\MCS(\bG_{V;W})$, there exists a unique $\GV\in\MCS(\bG_V)$ such that $\GVW\subset\GV$. Moreover, $\GV$ is the unique MCS in $\bG_V$ such that $\GVW=\GV\ints\bG_{V;W}$.
\item[(4).] Every edge in $\bG_{V;W}$ is contained in a unique $MCS$ of $\bG_{V;W}$. Also, for any $\GV\in\MCS(\bG_V)$, $\GV\ints\bG_{V;W}\in\MCS(\bG_{V;W})$ if and only if {$\GV\ints\bG_{V;W}\neq \emptyset$.}
\item[(5).] Every vertex in $\bG_{V;W}$ is contained in at most 2 MCS of $\bG_{V;W}$. Moreover, a vertex $\VQ\in\cV(\bG_{V;W})$ is contained in a unique MCS of $\bG_{V;W}$ if and only if $\VQ\in\del\cA_\Sep(V;W)\subset \cV(\bG_{V;W})$.
\item[(6).] For any distinct $p, q\in W$, we have $\res_W^V(\cV(\bP_{p,q,V}))\subset\cV(\bP_{p,q,W})$. Moreover, if
$\bP_{p,q,V}=\VQ_0:=\VQ_p\mbox{---}\VQ_1\mbox{---}\cdots\mbox{---}\VQ_m\mbox{---}\VQ_q=:\VQ_{m+1}$,
where $\VQ_p=\typemark{F_p}{\{p\}}{V}$ and $\VQ_q=\typemark{F_q}{\{q\}}{V}$, then
$$\res_W^V(\VQ_l)\in\cV(\Path_{p,q,W}(\res_W^V(\VQ_s),\res_W^V(\VQ_t)))$$
if and only if one of the following holds:
\begin{itemize}
\item $0\leq s\leq l\leq t\leq m+1$.
\item $0\leq t\leq l\leq s\leq m+1$.
\item $\res_W^V(\VQ_l)\in\{\res_W^V(\VQ_s),\res_W^V(\VQ_t)\}$.
\end{itemize}
\item[(7).] For any $\GVW\in\MCS(\bG_{V;W})$ and any distinct $p, q\in W$, either $\GVW\ints \bP_{p,q,V}=\emptyset$, or $\GVW\ints\bP_{p,q,V}$ contains exactly 2 vertices and exactly one edge. Moreover, if $\GVW$ has vertices $\VQ_j=\typemark{F_j}{I_j}{V}$, $1\leq j\leq k$, (assuming all $\VQ_j$'s are distinct), there exist $I_j'\in\stype{I_j}{V}$, $1\leq j\leq k$ such that $I_1',...,I_k'$ are pairwise disjoint and $W\subset  \union_{j=1}^kI_j'$.
\item[(8).] For any edge $\WP_1\mbox{---}\WP_2\in\cE(\bG_W)$, there exist a unique $\GVW\in MCS(\bG_{V;W})$ such that there exist distinct $\VQ_1,\VQ_2\in\cV(\GVW)$ satisfying
$$\WP_1\mbox{---}\WP_2\in\cE(\Path_{p,q,W}(\res_W^V(\VQ_1),\res_W^V(\VQ_2)))$$
for any suitable choice of $p,q\in W$ which make the right hand side well-defined. (By the first assertion of Corollary \ref{remaining properties of sep graph}, $\Path_{p,q,W}(\res_W^V(\VQ_1),\res_W^V(\VQ_2))$ is the unique shortest path connecting $\res_W^V(\VQ_1)$ and $\res_W^V(\VQ_2)$ in $\bG_W$, which is independent of the choice of $p,q$ as long as the expression is well-defined.)

Furthermore, the choice of $\GVW$ only depends on the unique MCS in $\bG_W$ containing the edge $\WP_1\mbox{---}\WP_2$.
\end{enumerate}
\end{lemma}
\begin{proof}
\begin{enumerate}
\item[(1).] It is straightforward from Notation \ref{PATH} and Definition \ref{W-face of graph} that $\bG_{V;W}\supset\displaystyle\bigcup_{\substack{p,q:p,q\in W\\p\neq q}}\bP_{p,q,V}$. We only need to prove that $\bG_{V;W}\subset\displaystyle\bigcup_{\substack{p,q:p,q\in W\\p\neq q}}\bP_{p,q,V}$.

For any $\typemark{\hF}{I}{V}\in\cV(\bG_{V;W})$, choose $p\in I\ints W$ and $q\in W\setminus I$ and we have $\typemark{\hF}{I}{V}\in\cV(\bP_{p,q,V})$. This proves that $\cV(\bG_{V;W})\subset\displaystyle\bigcup_{\substack{p,q:p,q\in W\\p\neq q}}\cV(\bP_{p,q,V})$.

For any $\VQ_1\mbox{---}\VQ_2\in\cE(\bG_{V;W})$, by Lemma \ref{prim decomp} (used when $F_1=F_2$) and the third assertion in Lemma \ref{properties of Theta sep} (used when $F_1\neq F_2$), we can assume WLOG that $\VQ_j=\typemark{F_j}{I_j}{V}$, $j=1,2,$ with $I_1\subset I_2$. Choose $p\in I_1\ints W$ and $q\in W\setminus I_2$. Then $\VQ_1\mbox{---}\VQ_2\in\cE(\bP_{p,q,V})$. This proves that $\cE(\bG_{V;W})\subset\displaystyle\bigcup_{\substack{p,q:p,q\in W\\p\neq q}}\cE(\bP_{p,q,V})$.

\item[(2).] For any $\VQ_1,\VQ_2\in\cV(\bG_{V;W})$, by Lemma \ref{prim decomp} (used when $F_1=F_2$) and the third assertion in Lemma \ref{properties of Theta sep} (used when $F_1\neq F_2$), we can assume WLOG that $\VQ_j=\typemark{F_j}{I_j}{V}$, $j=1,2,$ with $I_1\subset I_2$. Choose $p\in I_1\ints W$ and $q\in W\setminus I_2$. Then $\VQ_1,\VQ_2\in\cV(\bP_{p,q,V})$. Hence by the first assertion of Corollary \ref{remaining properties of sep graph}, $\Path_{p,q,V}(\VQ_1,\VQ_2)\subset \bG_{V;W}$ is simultaneously the unique shortest path connecting $\VQ_1,\VQ_2$  in $\bG_{V;W}$ and the unique shortest path connecting $\VQ_1,\VQ_2$  in $\bG_{V}$.

\item[(3).] By the first assertion in Lemma \ref{properties of W-face}, every MCS of $\bG_{V;W}$ contains at least 2 vertices and 1 edge. For any $\GVW\in\MCS(\bG_{V;W})$, let $\VQ_1\mbox{---}\VQ_2\in\cE(\GVW)$. By the first assertion of Proposition \ref{key prop of ASep graph}, there exists a unique $\GV\in \MCS(\bG_V)$ such that $\VQ_1\mbox{---}\VQ_2\in\cE(\GV)$. Since $\GVW$ is a complete subgraph of $\bG_{V;W}\subset\bG_V$, $\GVW$ is contained in a MCS of $\bG_V$. Let $\widetilde{\GV}$ be an arbitrary MCS of $\bG_V$ containing $\GVW$. In particular, $\VQ_1\mbox{---}\VQ_2\in\cE(\widetilde{\GV})$. By the first assertion of Proposition \ref{key prop of ASep graph}, we have $\GV=\widetilde{\GV}$. In particular, this shows that $\GV$ is the unique MCS in $\bG_V$ containing $\GVW$.

On the other hand, $\GV\ints\bG_{V;W}$ is a complete subgraph of $\bG_{V;W}\subset\bG_V$, it is contained in some $\widetilde{\GVW}\in\MCS(\bG_{V;W})$. Recall that $\GVW\in\MCS(\bG_{V;W})$, we have
$$\GVW\subset\GV\ints\bG_{V;W}\subset \widetilde{\GVW}\implies\GVW=\widetilde{\GVW}.$$
This proves $\GV\ints\bG_{V;W}=\GVW$.

\item[(4).] If there exist some edge of $\bG_{V;W}$ which is contained in $\GVW_1\neq\GVW_2\in\MCS(\bG_{V;W})$ with $\GVW_1\neq\GVW_2$. by the third assertion of Lemma \ref{properties of W-face}, there exist unique $\GV_1,\GV_2\in\MCS(\bG_V)$ such that $\GV_j\ints\bG_{V;W}=\GVW_j$, $j=1,2$. By the first assertion of Proposition \ref{key prop of ASep graph}, $\GV_1=\GV_2$, which implies that $\GVW_1=\GVW_2$, contradictory to the assumption that $\GVW_1\neq\GVW_2$. Hence every edge in $\bG_{V;W}$ is contained in a unique MCS of $\bG_{V;W}$.

For any $\GV\in\MCS(\bG_V)$, if $\GV\ints\bG_{V;W}\in\MCS(\bG_{V;W})$, by the first assertion in Lemma \ref{properties of W-face}, $\GV\ints\bG_{V;W}$ contains at least 1 edge. {Hence $\GV\ints\bG_{V;W}\neq \emptyset$}. On the other hand, {if $\GV\ints\bG_{V;W}\neq \emptyset$, by the first assertion in Lemma \ref{properties of W-face}, there exist $p\neq q\in W$ such that $\GV\ints\bG_{V;W}\supset\GV\ints\bP_{p,q,V}\neq \emptyset$. By the fourth assertion in Corollary \ref{remaining properties of sep graph}, $\GV\ints\bG_{V;W}$ contains at least 1 edge}. Since $G^{(V)}$ is complete, by the definition of $\bG_{V;W}$, we know that $\GV\ints\bG_{V;W}$ is a complete subgraph of $\bG_{V;W}$. By the previous paragraph, there exists a unique $\GVW\in\MCS(\bG_{V;W})$ such that $\GV\ints\bG_{V;W}\subset \GVW$. It follows from the third assertion of Lemma \ref{properties of W-face} that $\GVW=\widetilde\GV\ints\bG_{V;W}$ for some $\widetilde\GV\in\MCS(\bG_V)$. Since $\GV$ and $\widetilde\GV$ share an edge, by the first assertion of Proposition \ref{key prop of ASep graph}, $\GV=\widetilde\GV$. Hence $\GVW=\GV\ints\bG_{V;W}\in\MCS(\bG_{V;W})$. Thus we conclude that for any $\GV\in\MCS(\bG_V)$, $\GV\ints\bG_{V;W}\in\MCS(\bG_{V;W})$ if and only if $\GV\ints\bG_{V;W}\neq \emptyset$.

\item[(5).] If some vertex of $\bG_{V;W}$ is contained in 3 different MCS $\GVW_j\in \MCS(\bG_{V;W})$, $j=1,2,3$, by the third assertion in Lemma \ref{properties of W-face}, there exist unique MCS $\GV_j\in\MCS(\bG_V)$ such that $\GVW_j=\GV_j\ints\bG_{V;W}$, $j=1,2,3$. By the second assertion of Proposition \ref{key prop of ASep graph}, $\GV_s=\GV_t$ for some $1\leq s\neq t\leq 3$, which implies that $\GVW_s=\GVW_t$, contradictory to the assumption that $\GVW_1,\GVW_2,\GVW_3$ are distinct MCS.

For any $\VQ\in\cV(\bG_{V;W})$ and any pair $\GVW_1,\GVW_2\in \MCS(\bG_{V;W})$ containing $\VQ$, by the third assertion of Lemma \ref{properties of W-face}, there exist unique $\GV_j\in\MCS(\bG_V)$ such that $\GVW_j=\GV_j\ints\bG_{V;W}$, $j=1,2$. If $\GVW_1\neq\GVW_2$, then $\GV_1\neq\GV_2$. In particular, by the second assertion of Corollary \ref{remaining properties of sep graph} and the second assertion of Lemma \ref{res of vertices}, $\VQ\not\in\del\cA_\Sep(V;W)$. On the other hand, if  $\VQ\not\in\del\cA_\Sep(V;W)$, by the second the assertion of Corollary \ref{remaining properties of sep graph}, there exist exactly two MCS $\GV_1,\GV_2\in\MCS(\bG_{V})$ containing $\VQ$. Since $\GV_j\cap\bG_{V;W}\neq \emptyset$, $j=1,2$, by the fourth assertion in Lemma \ref{properties of W-face}, $\GV_j\cap\bG_{V;W}\in\MCS(\bG_{V;W})$, $j=1,2$. By the first assertion in Proposition \ref{key prop of ASep graph}, $\GV_1\cap\GV_2=\{\VQ\}$. Therefore, by the first assertion in Lemma \ref{properties of W-face}, $\GV_1\cap\bG_{V;W}\neq\GV_2\cap\bG_{V;W}$. Hence we conclude that $\VQ$ is contained in a unique MCS of $\bG_{V;W}$ if and only if $\VQ\in\del\cA_\Sep(V;W)\subset \cV(\bG_{V;W})$.
\item[(6).] For any distinct $p, q\in W$, $\res_W^V(\cV(\bP_{p,q,V}))\subset\cV(\bP_{p,q,W})$ follows directly from definitions of $\bP_{p,q,V}$ and $\res_W^V$ {(i.e. Notations \ref{PATH} and Definition \ref{res of types})}. The rest of the assertion follows directly from the third assertion of Lemma \ref{res of vertices} and Lemma \ref{unique shortest paths}.
\item[(7).] By the third assertion in Lemma \ref{properties of W-face}, there exist a unique $\GV\in\MCS(\bG_{V})$, such that $\GVW=\GV\ints \bG_{V;W}$. Notice that $\bP_{p,q,V}\subset \bG_{V;W}$ due to the first assertion in Lemma \ref{properties of W-face}. Therefore $\bP_{p,q,V}\ints\GVW=\bP_{p,q,V}\ints\GV$. By the fourth assertion in Corollary \ref{remaining properties of sep graph}, $\bP_{p,q,V}\ints\GVW$ is either an empty graph, or contains exactly 2 vertices and exactly 1 edge.

Let $\VQ_{k+1},...,\VQ_{k+l}$ be all the vertices in $\cV(\GV)\setminus\cV(\GVW)$. In particular, $\VQ_{k+1},...,\VQ_{k+l}\not\in\cV(\bG_{V;W})$. Assume that $\VQ_{k+1},...,\VQ_{k+l}$ are distinct and that $\VQ_{j}=\typemark{F_j}{I_j}{V}$ for any $k+1\leq j\leq k+l$. Then by the fifth assertion in Corollary \ref{remaining properties of sep graph}, there exist $I_j'\in\stype{I_j}{V}$, $1\leq j\leq k+l$, such that $I_j'$ are pairwise disjoint and $\union_{j=1}^{k+l}I_j'=V$. Since $\GVW=\GV\ints \bG_{V;W}$ and that $\emptyset\neq I_1\ints W\subset W\setminus I_j'$ for any $k+1\leq j\leq l+1$, the fact that $\VQ_{k+1},...,\VQ_{k+l}\not\in\cV(\bG_{V;W})$ implies that $I_{k+1}',...,I_{k+l}'\subset V\setminus W$. Therefore $W\subset \union_{j=1}^k I_j'$.

\item[(8).] This assertion has a long proof. We first prove the case when $|W|=2$ and $|\cF(W)|=1$. Let $W=\{p,q\}$ with $p\neq q$. Then $\bG_W$ is a complete graph and hence $\MCS(\bG_W)=\{\bG_W\}$. Notice that for any $\typemark{\hF}{I}{V}\in\bG_{V;W}$, $\hF=F_p=F_q$. By Lemma \ref{unique shortest paths} and the first assertion in Lemma \ref{properties of W-face}, $\cV(\bG_{V;W})=\{\typemark{F_p}{\{p\}}{V},\typemark{F_q}{\{q\}}{V}\}$ and $\bG_{V;W}$ is a complete graph. Hence $\MCS(\bG_{V;W})=\{\bG_{V;W}\}$. Recall that by Definition \ref{res of types}, $\res_W^V(\typemark{F_p}{\{p\}}{V})=p$ and $\res_W^V(\typemark{F_q}{\{q\}}{V})=q$, the assertion clearly holds in this case.

From now on, we assume that $|W|> 2$ or $|\cF(W)|> 1$, we split the rest of the proof into 3 steps. Here is an outline:
\begin{itemize}
\item \textbf{Step 1:} Existence of $\GVW$.
\item \textbf{Step 2:} Uniqueness of $\GVW$.
\item \textbf{Step 3:} For any 2 different edges in $\bG_W$ in the same MCS of $\bG_W$, the corresponding MCS in $\bG_{V;W}$ constructed above are identical.
\end{itemize}

\textbf{Step 1}: By the third assertion in Lemma \ref{properties of Theta sep} and Lemma \ref{prim decomp}, we can assume WLOG that $\WP_l=\typemark{\hF_l}{J_l}{W}$, $l=1,2$, with $J_1\subset J_2$. Choose $p\in J_1$ and $q\in W\setminus J_2$, then $\WP_1\mbox{---}\WP_2$ is an edge of $\bP_{p,q,W}$. Denoted by $\WP_p=\typemark{F_p}{\{p\}}{W}$ and $\WP_q=\typemark{F_q}{\{q\}}{W}$. WLOG we can assume that
$$\Path_{p,q,W}(\WP_p,\WP_1)\ints\Path_{p,q,W}(\WP_2,\WP_q)=\emptyset.$$
Assume that
$$\bP_{p,q,V}=\VQ_0:=\VQ_p\mbox{---}\VQ_1\mbox{---}\cdots\mbox{---}\VQ_m\mbox{---}\VQ_q=:\VQ_{m+1},$$
where $\VQ_p=\typemark{F_p}{\{p\}}{V}$ and $\VQ_q=\typemark{F_q}{\{q\}}{V}$. By the sixth assertion of Lemma \ref{properties of W-face}, $\res_W^V(\cV(\bP_{p,q,V}))\subset\cV(\bP_{p,q,W})$. Moreover, for any $0\leq  j \leq m+1$, the following holds:
\begin{itemize}
\item Either $\res_W^V(\VQ_j)\in\cV(\Path_{p,q,W}(\WP_p,\WP_1))$ or $\res_W^V(\VQ_j)\in\cV(\Path_{p,q,W}(\WP_2,\WP_q))$. This is due to the fact that $\bP_{p,q,W}$ is a path and $\WP_1\mbox{---}\WP_2$ is an edge of $\bP_{p,q,W}$.
\item If $\res_W^V(\VQ_j)\in\cV(\Path_{p,q,W}(\WP_p,\WP_1))$, for any $0\leq l\leq j$, we have $\res_W^V(\VQ_l)\in\cV(\Path_{p,q,W}(\WP_p,\WP_1))$. This follows from the sixth assertion of Lemma \ref{properties of W-face}.
\item If $\res_W^V(\VQ_j)\in\cV(\Path_{p,q,W}(\WP_2,\WP_q))$, for any $j\leq l\leq m+1$, we have $\res_W^V(\VQ_l)\in\cV(\Path_{p,q,W}(\WP_2,\WP_q))$. This follows from the sixth assertion of Lemma \ref{properties of W-face}.
\end{itemize}
Choose $j_0=\max\{j|0\leq j\leq m+1\mathrm{~and~}\res_W^V(\VQ_j)\in\cV(\Path_{p,q,W}(\WP_p,\WP_1))\}$ and $l_0=\min\{l|0\leq l\leq m+1\mathrm{~and~}\res_W^V(\VQ_l)\in\cV(\Path_{p,q,W}(\WP_2,\WP_q))\}$. Then we have $j_0=l_0-1$. Hence by the first assertion of Proposition \ref{key prop of ASep graph}, we can choose $\GVW\in\MCS(\bG_{V;W})$ to be the MCS containing the edge $\VQ_{j_0}\mbox{---}\VQ_{l_0}=\VQ_{j_0}\mbox{---}\VQ_{j_0+1}$. This shows the existence of $\GVW$ satisfying the requirements in this assertion.

\textbf{Step 2}: Now we prove that for a fixed $\WP_1\mbox{---}\WP_2$, the $\GVW$ contructed in \textbf{Step 1} is unique. By the third assertion in Lemma \ref{properties of Theta sep} and Lemma \ref{prim decomp}, we assume WLOG that $\WP_l=\typemark{\hF_l}{J_l}{W}$, $l=1,2$, with $J_1\subset J_2$. Suppose there exist $\VQ_{l1}\neq\VQ_{l2}\in\cV(\GVW_l)$ for some $\GVW_l\in\MCS(\bG_{V;W})$, $l=1,2$, such that
$$\WP_1\mbox{---}\WP_2\in\cE(\Path_{p_l,q_l,W}(\res_W^V(\VQ_{l1}),\res_W^V(\VQ_{l2}))),~l=1,2,$$
for some suitable choice of $p_1,p_2,q_1,q_2\in W$. (A direct consequence of this is that $\res_W^V(\VQ_{l1})\neq \res_W^V(\VQ_{l2})$ and hence $\VQ_{l1}\neq\VQ_{l2}$, $l=1,2$.) Then $|\{p_l,q_l\}\ints J_t|=||\{p_l,q_l\}\ints(W\setminus J_t)|=1$ for any $1\leq l,t\leq 2$. Therefore we can assume WLOG that $p_1,p_2\in J_1$ and $q_1,q_2\in W\setminus J_2$. In particular, $\{p_1,p_2\}\cap\{q_1,q_2\}=\emptyset$.

Let $\WP_{x}=\typemark{F_x}{\{x\}}{W}$ for any $x\in W$. By the first assertion in Corollary \ref{remaining properties of sep graph} and the path structure of $\bP_{p_l,q_l,W}$ for any $l=1,2$, we can also assume WLOG that the following properties hold: (See Figure \ref{lem6.32(8)-1}.)
\begin{align}\label{no overlap-0}
\begin{cases}
\displaystyle\res_W^V(\VQ_{l1})\in\Path_{p_l,q_l,W}(\WP_{p_l},\WP_1)=\Path_{p_l,q_{l+1},W}(\WP_{p_l},\WP_1)~\forall l=1,2.\\
\displaystyle\res_W^V(\VQ_{l2})\in\Path_{p_l,q_l,W}(\WP_{q_l},\WP_2)=\Path_{p_{l+1},q_l,W}(\WP_{q_l},\WP_2)~\forall l=1,2.\\
\displaystyle \WP_2\not\in\cV(\Path_{p_1,q_1,W}(\WP_{p_1},\WP_1))=\cV(\Path_{p_1,q_2,W}(\WP_{p_1},\WP_1)).\\
\displaystyle \WP_1\not\in\cV(\Path_{p_1,q_1,W}(\WP_{2},\WP_{q_1}))=\cV(\Path_{p_2,q_1,W}(\WP_{2},\WP_{q_1})).
\end{cases}
\end{align}
where $p_3:=p_1$ and $q_3:=q_1$.
Since $\Path_{p_1,q_2,W}(\WP_{p_1},\WP_{1}),\Path_{p_1,q_2,W}(\WP_2,\WP_{q_2})\subset \bP_{p_1,q_2,W}$, by \eqref{no overlap-0} and the first assertion in Corollary \ref{remaining properties of sep graph}, $\WP_2\in\cV(\Path_{p_1,q_2,W}(\WP_1,\WP_{q_2}))$ and hence
$$\WP_1\not\in\cV(\Path_{p_1,q_2,W}(\WP_2,\WP_{q_2}))=\cV(\Path_{p_2,q_2,W}(\WP_2,\WP_{q_2})).$$
A similar argument shows that $\WP_1\in\cV(\Path_{p_2,q_1,W}(\WP_{p_2},\WP_{2}))$ and hence
$$\WP_2\not\in\cV(\Path_{p_2,q_1,W}(\WP_{p_2},\WP_1))=\cV(\Path_{p_2,q_2,W}(\WP_{p_2},\WP_1)).$$ A short summary of \eqref{no overlap-0} and the above discussion is the following: (See Figure \ref{lem6.32(8)-1}.)
\begin{align}\label{no overlap-2}
\begin{split}
&\cV(\Path_{p_t,q_t,W}(\res_W^V(\VQ_{t1}),\WP_1))\ints\cV(\Path_{p_s,q_s,W}(\WP_2,\res_W^V(\VQ_{s2}))) \\
\subset&\cV(\Path_{p_t,q_t,W}(\WP_{p_t},\WP_1))\ints\cV(\Path_{p_s,q_s,W}(\WP_2,\WP_{q_s})\\
=&\cV(\Path_{p_t,q_s,W}(\WP_{p_t},\WP_1))\ints\cV(\Path_{p_t,q_s,W}(\WP_2,\WP_{q_s})=\emptyset,~1\leq s,t\leq 2.
\end{split}
\end{align}
\begin{figure}[h]
	\centering
	\includegraphics[width=4in]{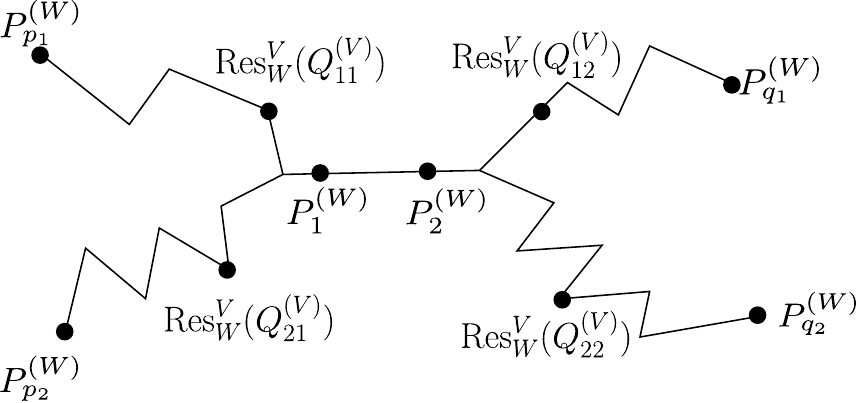}
	\caption{ \label{lem6.32(8)-1}}
\end{figure}
Moreover, for any $1\leq t,s\leq 2$, we have
\begin{align}
\begin{split}
\widetilde\bP_{ts}:=\bP_{p_t,q_s,W}
=&\Path_{p_t,q_t,W}(\WP_{p_t},\WP_2))\union\Path_{p_s,q_s,W}(\WP_2,\WP_{q_s}) \\
=&\Path_{p_t,q_t,W}(\WP_{p_t},\WP_1))\union\Path_{p_s,q_s,W}(\WP_1,\WP_{q_s}))\\
\end{split}
\end{align}
and
\begin{align}\label{diagonal also work}
\begin{split}
\bP_{ts}:=&\Path_{p_t,q_s,W}(\res_W^V(\VQ_{t1}),\res_W^V(\VQ_{s2}))\\
=&\Path_{p_t,q_t,W}(\res_W^V(\VQ_{t1}),\WP_2))\union\Path_{p_s,q_s,W}(\WP_2,\res_W^V(\VQ_{s2})) \\
=&\Path_{p_t,q_t,W}(\res_W^V(\VQ_{t1}),\WP_1))\union\Path_{p_s,q_s,W}(\WP_1,\res_W^V(\VQ_{s2}))\subset\widetilde\bP_{ts}.
\end{split}
\end{align}

%
%
%

Since $\VQ_{l1}\mbox{---}\VQ_{l2}\in\cE(\bG_{V;W})$ for any $l=1,2$, we claim that
\begin{align}\label{nothing in between-1}
\cV(\Path_{p_l,q_l,W}(\res_W^V(\VQ_{l1}),\WP_1)))\ints\res_W^V(\cV(\bG_{V;W}))=\{\res_W^V(\VQ_{l1})\}, ~1\leq l\leq 2,
\end{align}
and
\begin{align}\label{nothing in between-2}
\cV(\Path_{p_l,q_l,W}(\WP_2,\res_W^V(\VQ_{l2})))\ints\res_W^V(\cV(\bG_{V;W}))=\{\res_W^V(\VQ_{l2})\}, ~1\leq l\leq 2.
\end{align}
To see why \eqref{nothing in between-1} and \eqref{nothing in between-2} are true, for any $\widetilde{\VQ_{l1}},\widetilde{\VQ_{l2}}\in\cV(\bG_{V;W})$ such that
$$\res_W^V(\widetilde{\VQ_{l1}})\in\cV(\Path_{p_l,q_l,W}(\res_W^V(\VQ_{l1}),\WP_1)))\ints\res_W^V(\cV(\bG_{V;W}))$$
and
$$\res_W^V(\widetilde{\VQ_{l2}})\in\cV(\Path_{p_l,q_l,W}(\WP_2,\res_W^V(\VQ_{l2})))\ints\res_W^V(\cV(\bG_{V;W})),$$
we have $\res_W^V(\widetilde{\VQ_{l1}})$ and $\res_W^V(\widetilde{\VQ_{l2}})$ are between $\res_W^V(\VQ_{l1})$ and $\res_W^V(\VQ_{l2})$ in the sense of Definition \ref{in between for ASep} (due to Lemma \ref{unique shortest paths}). By \eqref{no overlap-2}, the third assertion of Lemma \ref{res of vertices} and the above, for any $l,t\in\{1,2\}$, one of the following holds:
\begin{itemize}
\item $\VQ_{lt}$ is between $\VQ_{l1}$ and $\VQ_{l2}$ in the sense of Definition \ref{in between for ASep}.
\item $\res_W^V(\widetilde{\VQ_{lt}})=\res_W^V(\VQ_{lt})$.
\end{itemize}
Since $\VQ_{l1}\mbox{---}\VQ_{l2}$ is an edge in $\bG_{V;W}$ for any $l\in\{1,2\}$, by Definition \ref{graph of sep}, the second bullet point in the above must holds.
This verifies \eqref{nothing in between-1} and \eqref{nothing in between-2}.

Recall that $\WP_1$ and $\WP_2$ are connected by an edge in $\bG_W$. As a corollary of \eqref{diagonal also work}, \eqref{nothing in between-1} and \eqref{nothing in between-2}, we have
\begin{align}\label{nothing in between-3}
\cV(\bP_{ts})\ints\res_W^V(\cV(\bG_{V;W}))=\{\res_W^V(\VQ_{t1}),\res_W^V(\VQ_{s2})\}, ~1\leq s,t\leq 2.
\end{align}
By the third assertion in Lemma \ref{res of vertices}, for any $\VQ\in\cV(\bG_{V;W})$ which is between $\VQ_{t1}$ and $\VQ_{s2}$ in the sense of Definition \ref{in between for ASep}, $\res_W^V(\VQ)\in\{\res_W^V(\VQ_{t1}),\res_W^V(\VQ_{s2})\}$. Hence we discuss the following 2 cases:

\textbf{Case 1}: If $\res_W^V(\VQ)=\res_W^V(\VQ_{t1})$, then $\VQ\in \cV(\bP_{p_t,q_t,V})$. By \eqref{no overlap-2} and the third assertion in Lemma \ref{res of vertices}, $\VQ,\VQ_{t1}$ is between $\VQ_{p_t}$ and $\VQ_{t2}$ in the sense of Definition \ref{in between for ASep}, where $\VQ_{p_t}=\typemark{F_{p_t}}{\{p_t\}}{V}$. By Definition \ref{graph of sep}, the third assertion in Lemma \ref{res of vertices} and the fact that $\VQ_{t1}\mbox{---}\VQ_{t2}\in\cE(\bG_{V;W})$, we have
\begin{align}\label{VQ position-1}
\VQ\in \cV(\Path_{p_t,q_t,V}(\VQ_{p_t}, \VQ_{t1}))=\cV(\Path_{p_t,q_s,V}(\VQ_{p_t}, \VQ_{t1})).
\end{align}
By \eqref{no overlap-0} and \eqref{no overlap-2}, we have
$$\Path_{p_t,q_s,W}(\WP_{p_t},\res_W^V(\VQ_{t1}))\ints\Path_{p_t,q_s,W}(\res_W^V(\VQ_{s2}),\WP_{q_s})=\emptyset.$$
Therefore, by the sixth assertion in Lemma \ref{properties of W-face},
$$\Path_{p_t,q_s,V}(\VQ_{p_t},\VQ_{t1})\ints\Path_{p_t,q_s,V}(\VQ_{s2},\VQ_{q_s})=\emptyset.$$
Hence
\begin{align}\label{VQ position-1'}
\cV(\Path_{p_t,q_s,V}(\VQ_{p_t},\VQ_{t1}))\ints\cV(\Path_{p_t,q_s,V}(\VQ_{t1},\VQ_{s2}))=\{\VQ_{t1}\}.
\end{align}
Since we assumed that $\VQ$ is between $\VQ_{t1}$ and $\VQ_{s2}$ in the sense of Definition \ref{in between for ASep}, by Lemma \ref{unique shortest paths}, $\VQ\in\Path_{p_t,q_s,V}(\VQ_{t1},\VQ_{s2})$. The fact that $\VQ=\VQ_{t1}$ then follows from \eqref{VQ position-1} and \eqref{VQ position-1'}.

\textbf{Case 2}: If $\res_W^V(\VQ)=\res_W^V(\VQ_{s2})$, by the same arguments as in \textbf{Case 1}, we have $\VQ=\VQ_{s2}$.

Hence by Definition \ref{graph of sep}, $\VQ_{t1}\mbox{---}\VQ_{s2}\in\cE(\bG_V)$ and therefore $\VQ_{t1}\mbox{---}\VQ_{s2}\in\cE(\bG_{V;W})$ (since the endpoints are in $\bG_{V;W}$).

Given the above preparations, we now prove that $\GVW_1=\GVW_2$:

If $\VQ_{11}=\VQ_{21}$ and $\VQ_{12}=\VQ_{22}$, then by the fourth assertion in Lemma \ref{properties of W-face}, $\GVW_1=\GVW_2$.

If at least one of $\VQ_{11}\neq\VQ_{21}$ or $\VQ_{12}\neq\VQ_{22}$ holds, we assume WLOG that $\VQ_{11}\neq \VQ_{21}$. We show that $\VQ_{11}$ and $\VQ_{21}$ are connected by a single edge in $\bG_{V;W}$. If not, by Definition \ref{W-face of graph}, we have $\VQ_{11}$ and $\VQ_{21}$ are not connected by a single edge in $\bG_{V}$. Recall that $\VQ_{11}$ and $\VQ_{21}$ are connected to $\VQ_{22}$ by a single edge in $\bG_{V;W}\subset\bG_V$. By the second assertion of Lemma \ref{reinterpretation of edges}, $\VQ_{22}$ is between $\VQ_{11}$ and $\VQ_{21}$. Hence by Lemma \ref{unique shortest paths} and the fact that $\VQ_{tl}\in\cV(\bG_{V;W})$ for any $1\leq t,l\leq 2$, there exist some $p,q\in W$ such that
\begin{align*}
\VQ_{22}\in\cV(\Path_{p,q,V}(\VQ_{11},\VQ_{21}))\setminus\{\VQ_{11},\VQ_{21}\}.
\end{align*}
By \eqref{diagonal also work}, we have
\begin{align}\label{use1}
\WP_1\mbox{---}\WP_2\in\cE(\Path_{p,q,W}(\res_W^V(\VQ_{11}),\res_W^V(\VQ_{22}))).
\end{align}
By the assumptions on $\VQ_{21},\VQ_{22}$, we have
\begin{align}\label{use2}
\WP_1\mbox{---}\WP_2\in\cE(\Path_{p,q,W}(\res_W^V(\VQ_{21}),\res_W^V(\VQ_{22}))).
\end{align}
Therefore $\res_W^V(\VQ_{11})\neq\res_W^V(\VQ_{22})$ and $\res_W^V(\VQ_{21})\neq\res_W^V(\VQ_{22})$. On the other hand, recall that $\VQ_{22}$ is between $\VQ_{11}$ and $\VQ_{21}$. By the third assertion of Lemma \ref{res of vertices}, we have $\res_W^V(\VQ_{22})$ is between $\res_W^V(\VQ_{11})$ and $\res_W^V(\VQ_{21})$. In particular, $\res_W^V(\VQ_{21})\neq\res_W^V(\VQ_{11})$. Hence, by Lemma \ref{unique shortest paths},
$$\cE(\Path_{p,q,W}(\res_W^V(\VQ_{11}),\res_W^V(\VQ_{22})))\ints\cE(\Path_{p,q,W}(\res_W^V(\VQ_{21}),\res_W^V(\VQ_{22})))=\emptyset.$$
This contradicts \eqref{use1} and \eqref{use2}. Therefore $\VQ_{11}\mbox{---}\VQ_{21}\in\cE(\bG_{V})$, which implies that $\VQ_{11}\mbox{---}\VQ_{21}\in\cE(\bG_{V;W})$.

Now, $\VQ_{11}$, $\VQ_{21}$ and $\VQ_{12}$ are vertices of a complete subgraph of $\bG_{V;W}$. By the fourth assertion of Lemma \ref{properties of W-face}, $\VQ_{11}\mbox{---}\VQ_{21}\in\cE(\GVW_1)$. Similarly, $\VQ_{11}$, $\VQ_{21}$ and $\VQ_{22}$ are vertices of a compete subgraph of $\bG_{V;W}$. By the fourth assertion of Lemma \ref{properties of W-face}, $\VQ_{11}\mbox{---}\VQ_{21}\in\cE(\GVW_2)$. Finally, by the fourth assertion of Lemma \ref{properties of W-face}, $\GVW_1=\GVW_2$. This finishes the uniqueness of $\GVW$ mentioned in the statement of this assertion.

\textbf{Step 3}: For any $\GW\in\MCS(\bG_W)$ and any $\WP_{11}\mbox{---}\WP_{12}, \WP_{21}\mbox{---}\WP_{22}\in\cE(\GW)$, as is already proved in \textbf{Step 1} and \textbf{Step 2}, there exist unique MCS $\GVW_1,\GVW_2\in\MCS(\bG_{V;W})$ such that there exist distinct $\VQ_{l1},\VQ_{l2}\in\cV(\GVW)$ satisfying
$$\WP_{l1}\mbox{---}\WP_{l2}\in\cE(\Path_{p_l,q_l,W}(\res_W^V(\VQ_{l1}),\res_W^V(\VQ_{l2}))), ~l=1,2.$$
In this step, we prove that $\GVW_1=\GVW_2$. In fact, it is sufficient to prove the special case when $\WP_{11}=\WP_{21}$. The general case follows from applying the special case to $\WP_{11}\mbox{---}\WP_{12}, \WP_{11}\mbox{---}\WP_{22}$ and to $\WP_{11}\mbox{---}\WP_{22},\WP_{21}\mbox{---}\WP_{22}$.

If in addition $\WP_{12}=\WP_{22}$, then $\WP_{11}\mbox{---}\WP_{12}= \WP_{21}\mbox{---}\WP_{22}$. By uniqueness of $\GVW_1$ proved in \textbf{Step 2}, $\GVW_2=\GVW_1$. Therefore we assume that $\WP_{12}\neq\WP_{22}$

By the path structure of $\bP_{p_l,q_l,W}$, similar to \textbf{Step 2}, we assume WLOG that
\begin{align}\label{no overlap-3}
\res_W^V(\VQ_{l1})\in\Path_{p_l,q_l,W}(\WP_{p_l},\WP_{l1})~\mathrm{and}~ \res_W^V(\VQ_{l2})\in\Path_{p_l,q_l,W}(\WP_{q_l},\WP_{l2}),~l=1,2,
\end{align}
as well as
\begin{align}\label{no overlap-4}
\begin{split}
&\Path_{p_l,q_l,W}(\res_W^V(\VQ_{l1}),\WP_{l1})\ints\Path_{p_l,q_l,W}(\WP_{l2},\res_W^V(\VQ_{l2})) \\
\subset&\Path_{p_l,q_l,W}(\WP_{p_l},\WP_{l1})\ints\Path_{p_l,q_l,W}(\WP_{l2},\WP_{q_l})=\emptyset,~l=1,2,
\end{split}
\end{align}
where $\WP_{x}=\typemark{F_x}{\{x\}}{W}$ for any $x\in W$. (See Figure \ref{lem6.32(8)-2}.) Here, by \eqref{no overlap-3}, for any $1\leq l\leq 2$, we have
$$\Path_{p_l,q_l,W}(\res_W^V(\VQ_{l1}),\WP_{l1})\subset\Path_{p_l,q_l,W}(\WP_{p_l},\WP_{l1})$$
and
$$\Path_{p_l,q_l,W}(\WP_{l2},\res_W^V(\VQ_{l2})) \subset \Path_{p_l,q_l,W}(\WP_{l2},\WP_{q_l}).$$
This explains the inclusion sign in \eqref{no overlap-4}. As a straightforward consequence of \eqref{no overlap-3} and \eqref{no overlap-4}, we have $\res_W^V(\VQ_{12})\neq\res_W^V(\VQ_{11})$ and $\res_W^V(\VQ_{21})\neq \res_W^V(\VQ_{22})$.
\begin{figure}[h]
	\centering
	\includegraphics[width=5in]{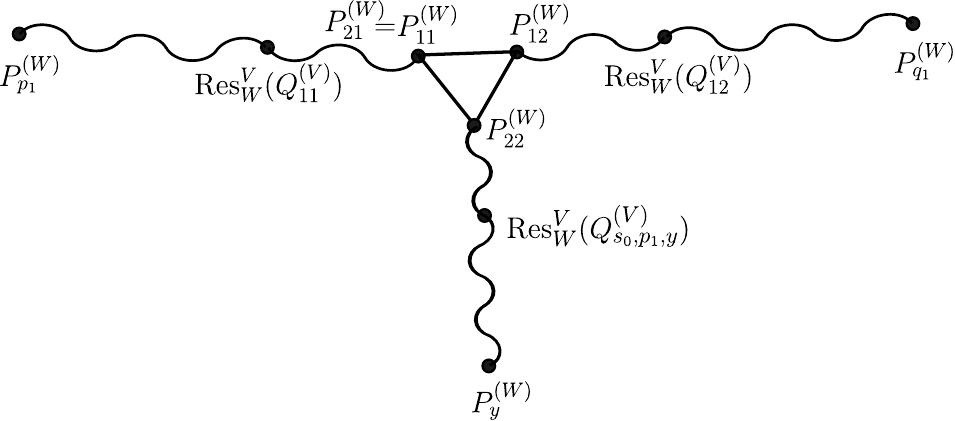}
	\caption{ \label{lem6.32(8)-2}}
\end{figure}

We first claim that $\WP_{22}\not\in\cV(\Path_{p_1,q_1,W}(\WP_{p_1},\WP_{11}))$. If not, since $\Path_{p_1,q_1,W}(\WP_{p_1},\WP_{11})$ is the shortest path connecting $\WP_{p_1}$ and $\WP_{11}$ in $\bG_W$ (by the first assertion in Corollary \ref{remaining properties of sep graph}), and that $\WP_{22}\mbox{---}\WP_{11}\in\cE(\bG_W)$, we have $d_{\bG_W}(\WP_{p_1},\WP_{22})= d_{\bG_W}(\WP_{p_1},\WP_{11})-1$. Hence by the assumption that $\WP_{22}$ and $\WP_{12}$ are in the same MCS of $\bG_W$, we have
\begin{align}\label{ridiculous dist est-1}
\begin{split}
d_{\bG_W}(\WP_{p_1},\WP_{12})\leq& d_{\bG_W}(\WP_{p_1},\WP_{22})+d_{\bG_W}(\WP_{22},\WP_{12})\\
=& d_{\bG_W}(\WP_{p_1},\WP_{11})-1+1=d_{\bG_W}(\WP_{p_1},\WP_{11}).
\end{split}
\end{align}
On the other hand, by the first assertion in Corollary \ref{remaining properties of sep graph}, $\Path_{p_1,q_1,W}(\WP_{p_1},\WP_{12})$ is the shortest path connecting $\WP_{p_1}$ and $\WP_{12}$ in $\bG_W$. Moreover, by \eqref{no overlap-4}, we have $\WP_{11}\in\cV(\Path_{p_1,q_1,W}(\WP_{p_1},\WP_{12}))$. Therefore
$$d_{\bG_W}(\WP_{p_1},\WP_{12})=d_{\bG_W}(\WP_{p_1},\WP_{11})+d_{\bG_W}(\WP_{11},\WP_{12})=d_{\bG_W}(\WP_{p_1},\WP_{11})+1.$$
This contradicts \eqref{ridiculous dist est-1}. Hence $\WP_{22}\not\in\cV(\Path_{p_1,q_1,W}(\WP_{p_1},\WP_{11}))$. In particular,
\begin{align}\label{a normal path-1}
\bP:=\Path_{p_1,q_1,W}(\WP_{p_1},\WP_{11})\union\left(\WP_{11}\mbox{---}\WP_{22}\right)
\end{align}
is a path connecting $\WP_{p_1}$ and $\WP_{22}$. Since we assumed that $\WP_{22}\neq\WP_{12}$ and that
$$\cV(\Path_{p_1,q_1,W}(\WP_{p_1},\WP_{12}))=\cV(\Path_{p_1,q_1,W}(\WP_{p_1},\WP_{11}))\union\{\WP_{12}\},~\textrm{following \eqref{no overlap-4}}, $$
we have $\WP_{22}\not\in\cV(\Path_{p_1,q_1,W}(\WP_{p_1},\WP_{12}))$. Hence, by \eqref{a normal path-1} and the uniqueness of shortest paths connecting $\WP_{p_1}$ and $\WP_{12}$ (i.e. the first assertion in Corollary \ref{remaining properties of sep graph} applied to $\WP_{p_1}$ and $\WP_{12}$), we have
\begin{align*}
d_{\bG_W}(\WP_{p_1},\WP_{12})\leq& d_{\bG_W}(\WP_{p_1},\WP_{22})+d_{\bG_W}(\WP_{22},\WP_{12})-1\\
=&d_{\bG_W}(\WP_{p_1},\WP_{22})\leq \mathrm{length}(\bP)=d_{\bG_W}(\WP_{p_1},\WP_{12}).
\end{align*}
This implies that $d_{\bG_W}(\WP_{p_1},\WP_{22})= \mathrm{length}(\bP)$ and therefore the path $\bP$ defined in \eqref{a normal path-1} is a shortest path. By the first assertion in Corollary \ref{remaining properties of sep graph}, we have
$$\bP=\Path_{p_1,y,W}(\WP_{p_1}, \WP_{22})~\textrm{for some}~y\in W\setminus\{p_1\}.$$
In particular, $\WP_{11}\mbox{---}\WP_{22}$ is an edge of $\bP_{p_1,y,W}$ with
\begin{align}\label{no overlap-5}
\Path_{p_1,y,W}(\WP_{p_1}, \WP_{11})\ints\Path_{p_1,y,W}(\WP_{y}, \WP_{22})=\emptyset.\text{ (See Figure \ref{lem6.32(8)-2}.)}
\end{align}
We can then apply the same arguments in \textbf{Step 1} to $\WP_{11}\mbox{---}\WP_{22}$ and $\bP_{p_1,y,W}$. To be specific, we let $\VQ_{0,p_1,y},...,\VQ_{m+1,p_1,y}$ be all the vertices in $\bP_{p_1,y,W}$ such that
$$\bP_{p_1,y,V}=\VQ_{0,p_1,y}:=\VQ_{p_1}\mbox{---}\VQ_{1,p_1,y}\mbox{---}\cdots\mbox{---}\VQ_{m,p_1,y}\mbox{---}\VQ_y=:\VQ_{m+1,p_1,y},$$
where $\VQ_x=\typemark{F_x}{\{x\}}{V}$ for any $x\in V$. Choose
\begin{align}\label{eqn:lem6.32(8)-t0}
t_0=\max\{t|0\leq t\leq m+1\mathrm{~and~}\res_W^V(\VQ_{t,p_1,y})\in\cV(\Path_{p_1,y,W}(\WP_{p_1},\WP_{11}))\}
\end{align}
and
$$s_0=\min\{s|0\leq s\leq m+1\mathrm{~and~}\res_W^V(\VQ_{s,p_1,y})\in\cV(\Path_{p_1,y,W}(\WP_{22},\WP_{y}))\}.$$
Then for the same reasons as in \textbf{Step 1}, $t_0=s_0-1$ and
\begin{align}\label{renewed corresp}
\WP_{11}\mbox{---}\WP_{22}\in\cE(\Path_{p_1,y,W}(\res_W^V(\VQ_{t_0,p_1,y}),\res_W^V(\VQ_{s_0,p_1,y}))).
\end{align}
Now we claim that $\VQ_{t_0,p_1,y}=\VQ_{11}$. If not, by \eqref{no overlap-3} and the definition of $\VQ_{t_0,p_1,y}$, either
\begin{align}\label{same VQs-1}
\res_W^V(\VQ_{11})\in\cV(\Path_{p_1,y,W}(\WP_{p_1},\res_W^V(\VQ_{t_0,p_1,y})))
\end{align}
or
$$\res_W^V(\VQ_{11})\in\cV(\Path_{p_1,y,W}(\res_W^V(\VQ_{t_0,p_1,y}),\WP_{11}))\setminus\{\res_W^V(\VQ_{t_0,p_1,y})\}.$$
In the latter case, we have
$$\res_W^V(\VQ_{11})\in\cV(\Path_{p_1,y,W}(\res_W^V(\VQ_{t_0,p_1,y}),\res_W^V(\VQ_{s_0,p_1,y})))\setminus\{\res_W^V(\VQ_{t_0,p_1,y}),\res_W^V(\VQ_{s_0,p_1,y})\}.$$
By the sixth assertion in Lemma \ref{properties of W-face} and the fact that $t_0=s_0-1$, we have
$$\VQ_{11}\in\cV(\Path_{p_1,y,V}(\VQ_{t_0,p_1,y},\VQ_{s_0,p_1,y}))\setminus\{\VQ_{t_0,p_1,y},\VQ_{s_0,p_1,y}\}=\emptyset.$$
This is impossible. Therefore it remains to consider the case when \eqref{same VQs-1} holds. By \eqref{no overlap-3}, \eqref{no overlap-4}, the first assertion in Corollary \ref{remaining properties of sep graph} and the definition of $\VQ_{t_0,p_1,y}$, we have
\begin{align}\label{same VQs-2}
\begin{split}
&\Path_{p_1,y,W}(\WP_{p_1},\res_W^V(\VQ_{t_0,p_1,y})) \\
\subset &\Path_{p_1,y,W}(\WP_{p_1},\WP_{11}) \\
=&\Path_{p_1,q_1,W}(\WP_{p_1},\WP_{11})
\subset\Path_{p_1,q_1,W}(\WP_{p_1},\WP_{12})\subset\Path_{p_1,q_1,W}(\WP_{p_1},\res_W^V(\VQ_{12})).
\end{split}
\end{align}
Therefore, by \eqref{same VQs-1} and \eqref{same VQs-2}, $\res_W^V(\VQ_{t_0,p_1,y})$ is between $\res_W^V(\VQ_{11})$ and $\res_W^V(\VQ_{12})$ in the sense of Definition \ref{in between for ASep}. By the third assertion of Lemma \ref{res of vertices}, either $\VQ_{t_0,p_1,y}$ is between $\VQ_{11}$ and $\VQ_{12}$ in the sense of Definition \ref{in between for ASep}, or $\res_W^V(\VQ_{t_0,p_1,y})\in\{\res_W^V(\VQ_{11}),\res_W^V(\VQ_{12})\}$. We discuss these 2 cases separately to prove that $\VQ_{t_0,p_1,y}= \VQ_{11}$.

\textbf{Case 1}: If $\VQ_{t_0,p_1,y}$ is between $\VQ_{11}$ and $\VQ_{12}$ in the sense of Definition \ref{in between for ASep}, by the fact that $\VQ_{11}\mbox{---}\VQ_{12}\in\cE(\bG_{V;W})\subset\cE(\bG_V)$, we have $\VQ_{t_0,p_1,y}\in\{\VQ_{11},\VQ_{12}\}$. By \eqref{no overlap-4}, $\res_W^V(\VQ_{12})\not\in\Path_{p_1,q_1,W}(\WP_{p_1},\WP_{11})$. Therefore by \eqref{same VQs-2}, $\res_W^V(\VQ_{t_0,p_1,y})\neq\res_W^V(\VQ_{12})$, which implies that $\VQ_{t_0,p_1,y}\neq \VQ_{12}$. As a consequnce, $\VQ_{t_0,p_1,y}= \VQ_{11}$.

\textbf{Case 2}: If $\res_W^V(\VQ_{t_0,p_1,y})\in\{\res_W^V(\VQ_{11}),\res_W^V(\VQ_{12})\}$, as is already proved in the above \textbf{Case 1}, $\res_W^V(\VQ_{t_0,p_1,y})\neq\res_W^V(\VQ_{12})$. Hence
\begin{align}\label{same VQs-3}\res_W^V(\VQ_{t_0,p_1,y})=\res_W^V(\VQ_{11})\neq\res_W^V(\VQ_{12}).
\end{align} Since $\res_W^V(\VQ_{11}),\res_W^V(\VQ_{12})\in\cV(\bP_{p_1,q_1,W})$, we have
\begin{align}\label{eqn:lem6.32-same line}
\VQ_{t_0,p_1,y},\VQ_{11},\VQ_{12}\in\cV(\bP_{p_1,q_1,V}).
\end{align}
By Lemma \ref{unique shortest paths}, one of $\VQ_{t_0,p_1,y},\VQ_{11},\VQ_{12}$ is between the other two in the sense of Definition \ref{in between for ASep}. By \eqref{same VQs-3} and the third assertion in Lemma \ref{res of vertices}, we have
$\VQ_{12}$ cannot be between $\VQ_{t_0,p_1,y}$ and $\VQ_{11}$ in the sense of Definition \ref{in between for ASep}. Therefore, there are two remaining subcases:

\textbf{Case 2a}: If $\VQ_{t_0,p_1,y}$ is between $\VQ_{11}$ and $\VQ_{12}$ in the sense of Definition \ref{in between for ASep}, this is the same as \textbf{Case 1} and we must have $\VQ_{t_0,p_1,y}= \VQ_{11}$.

\textbf{Case 2b}: Assume that $\VQ_{11}$ is between $\VQ_{12}$ and $\VQ_{t_0,p_1,y}$ in the sense of Definition \ref{in between for ASep}. By \eqref{no overlap-3}, \eqref{eqn:lem6.32(8)-t0} and \eqref{same VQs-3}, we have
\begin{align}\label{eqn:lem6.32-in between-1}
\VQ_{11}\text{ is between }\VQ_{p_1}\text{ and }\VQ_{t_0,p_1,y}\text{ in the sense of Definition \ref{in between for ASep}}.
\end{align}
Also, by \eqref{no overlap-3} and \eqref{no overlap-4}, $\res_W^V(\VQ_{11})$ is between $\res_W^V(\VQ_{p_1})=\WP_{p_1}$ and $\res_W^V(\VQ_{12})$ in the sense of Definition \ref{in between for ASep}. By the third assertion in Lemma \ref{res of vertices}, Definition \ref{bdry} and \eqref{same VQs-3}, we have
\begin{align}\label{eqn:lem6.32-in between-2}
\VQ_{11}\text{ is between }\VQ_{p_1}\text{ and }\VQ_{12}\text{ in the sense of Definition \ref{in between for ASep}.}
\end{align}
It follows from \eqref{same VQs-3}, the standing assumptions of \textbf{Case 2b}, \eqref{eqn:lem6.32-same line}, \eqref{eqn:lem6.32-in between-1}, \eqref{eqn:lem6.32-in between-2}, and Lemma \ref{unique shortest paths} that $\VQ_{11}=\VQ_{t_0,p_1,y}$ in \textbf{Case 2b}. This completes the proof of $\VQ_{11}=\VQ_{t_0,p_1,y}$.

Back to the remaining parts in the proof of \textbf{Step 3}. Now that we have $\VQ_{11}=\VQ_{t_0,p_1,y}$, we summarize what we have known at this point: (See Figure \ref{lem6.32(8)-2}.)

By \eqref{renewed corresp}, we have $\VQ_{11}\mbox{---}\VQ_{12},\VQ_{11}\mbox{---}\VQ_{s_0,p_1,y}\in\cE(\bG_{V;W})$ such that
\begin{align}\label{at last!-1}
\WP_{11}\mbox{---}\WP_{12}\in\cE(\Path_{p_1,q_1,W}(\res_W^V(\VQ_{11}),\res_W^V(\VQ_{12}))).
\end{align}
and
\begin{align}\label{at last!-2}
\WP_{21}\mbox{---}\WP_{22}=\WP_{11}\mbox{---}\WP_{22}\in\cE(\Path_{p_1,y,W}(\res_W^V(\VQ_{11}),\res_W^V(\VQ_{s_0,p_1,y}))).\end{align}
By the uniqueness of $\GVW_2$ proved in \textbf{Step 2}, $\VQ_{11}\mbox{---}\VQ_{s_0,p_1,y}\in\cE(\GVW_2)$. Recall that $\VQ_{11}\mbox{---}\VQ_{12}\in\cE(\GVW_1)$. If $\VQ_{s_0,p_1,y}\mbox{---}\VQ_{12}\in\cE(\bG_V)$, then $\VQ_{s_0,p_1,y}\mbox{---}\VQ_{12}\in\cE(\bG_{V;W})$. By the fourth assertion in Lemma \ref{properties of W-face}, $\GVW_1$ and $\GVW_2$ both equal to the unique MCS in $\bG_{V;W}$ which contains $\VQ_{11},\VQ_{12}, \VQ_{s_0,p_1,y}$.
Therefore, it remains for us to prove that $\VQ_{s_0,p_1,y}$ and $\VQ_{12}$ are connected by a single edge in $\bG_V$.

Suppose $\VQ_{s_0,p_1,y}$ and $\VQ_{12}$ are \textbf{NOT} connected by a single edge in $\bG_V$, by the second assertion of Lemma \ref{reinterpretation of edges}, $\VQ_{11}$ is between $\VQ_{s_0,p_1,y}$ and $\VQ_{12}$ in the sense of Definition \ref{in between for ASep}. By the third assertion of Lemma \ref{res of vertices}, $\res_W^V(\VQ_{11})$ is between $\res_W^V(\VQ_{s_0,p_1,y})$ and $\res_W^V(\VQ_{12})$ in the sense of Definition \ref{in between for ASep}. Therefore there exist some $p_0\neq q_0\in W$ such that,
$$\res_W^V(\VQ_{11}),\res_W^V(\VQ_{s_0,p_1,y}),\res_W^V(\VQ_{12})\in\cV(\bP_{p_0,q_0,W}).$$
By \eqref{at last!-1} and \eqref{at last!-2},
$$\WP_{11},\WP_{12},\WP_{22}\in\cV(\bP_{p_0,q_0,W})\ints\cV(\GW).$$
By Lemma \ref{unique shortest paths}, this is impossible because a shortest path connecting any pair of points in $\bG_{V}$ cannot contain 3 vertices of a complete subgraph of $\bG_{V}$. Hence $\VQ_{s_0,p_1,y}$ and $\VQ_{12}$ must be connected by a single edge in $\bG_V$. By the discussion in the previous paragraph, $\GVW_1=\GVW_2$ is the MCS containing $\VQ_{11}$, $\VQ_{12}$ and $\VQ_{s_0,p_1,y}$. This completes \textbf{Step 3} and hence the whole proof of the eighth assertion in Lemma \ref{properties of W-face}. \qedhere
\end{enumerate}
\end{proof}
\begin{rmk}
By the remark after Definition \ref{MCS def}, Definition \ref{W-face of graph} and the fourth assertion above, for any $\gamma\in\Gamma$ and $G\in\MCS(\bG_{V;W})$, the map $G\to\gamma G$ gives a bijection from $\MCS(\bG_{V;W})$ to $\MCS(\bG_{\gamma V;\gamma W})$.
\end{rmk}

Regarding the seventh assertion in Lemma \ref{properties of W-face}, although for any $\GVW\in\MCS(\bG_{V;W})$ and any $p\neq q$, $|\cV(\GVW)\ints\cV(\bP_{p,q,V})|\in\{0,2\}$, $\res_W^V(\cV(\GVW))\ints\cV(\bP_{p,q,W})$ may have cardinality equal to $1$, which does not happen in the world of actual separation described in Subsection \ref{subsec actual sep}. See the following example.
\begin{example}
Let $V=\{x,y,z\}\in\Gamma x_0$, where $x,y,z$ are pairwise distinct. Assume that $\hF\in\cF(V)\setminus\{F_x,F_y,F_z\}$ such that $\hF\in\Theta(F_p,F_q)$ for any distinct $p,q\in V$. In particular, $F_x,F_y,F_z$ are pairwise distinct. (See Figure \ref{example6.33}.)

Let $Q_{\hF,p}:=\typemark{\hF}{\{p\}}{V}$ for any $p\in V$. Then by the first assertion in Lemma \ref{reinterpretation of edges}, one can easily verify that there exists a $\GV\in\MCS(\bG_V)$ such that $\cV(\GV)=\{Q_{\hF,x},Q_{\hF,y},Q_{\hF,z}\}$. By the fourth assertion in Lemma \ref{properties of W-face}, for any $W\subset V$ with $|W|=2$, $\GVW:=\GV\cap\bG_{V;W}\in\MCS(\bG_{V;W})$. It follows from Definition \ref{res of types} that $|\res_W^V(\cV(\GVW))\cap\cV(\bG_W)|=1$. (Here, if $W=\{p,q\}$, then $\bG_W=\bP_{p,q,W}$.)
\begin{figure}[h]
	\centering
	\includegraphics[width=4in]{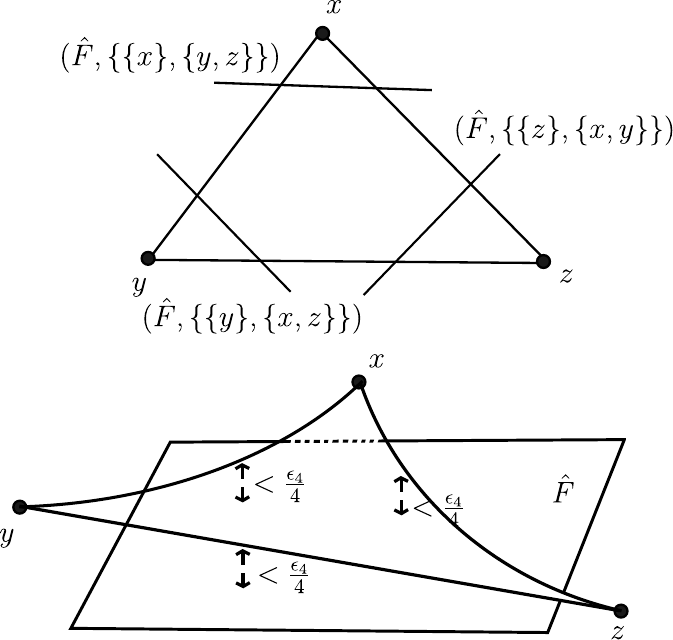}
	\caption{ \label{example6.33}}
\end{figure}
\end{example}
In Section \ref{sec hom arg}, we are mostly interested in the case when $|W|\geq 2$ and $|V\setminus W|=1$. In this case, we want to show that the above phenomenon only happens under very restrictive conditions.

\begin{definition}[Singularity of MCS in $\bG_V$]\label{singular}
Let $V\subset \Gamma x_0$ such that $|V|\geq 3$. For any $\GV\in\MCS(\bG_V)$ and any $p\in V$, we say that $\GV$ is \emph{singular} at $p$ if there exist distinct  $q_1, q_2\in W_p$ such that
$$\left|\res_{W_p}^V(\cV(\GV\ints\bG_{V;W_p}))\ints\cV(\bP_{q_1,q_2,W_p})\right|=1,$$
where $W_p:=V\setminus\{p\}$. Denoted by $\Sing(\GV)=\{p\in V|\GV\mathrm{~is~singular~at~}p\}$. If $\Sing(\GV)=\emptyset$, then we say that $\GV$ is \emph{regular}.
\end{definition}
\begin{rmk}
Under the above settings, for any $\gamma\in\Gamma$, by the remark after Definition \ref{MCS def}, Notations \ref{PATH}, the remark after Definition \ref{res of types} and the remark after Lemma \ref{properties of W-face}, we have $\gamma\Sing(\GV)=\Sing(\gamma\GV)$.
\end{rmk}
\begin{lemma}\label{singular case}
Let $V\subset \Gamma x_0$ such that $|V|\geq 3$. For any $\GV\in\MCS(\bG_V)$, if $p\in\Sing(\GV)$, then $|W_p|>2$ or $|\cF(W_p)|>1$, where $W_p:=V\setminus\{p\}$.
Moreover, there exist $\hF\in\cF(V)$ and $I_1,I_2\subset V$ such that
\begin{itemize}
\item $I_1\ints I_2=\emptyset$ and  $I_1\union I_2=V\setminus \{p\}$;
\item $\PSep_V(\hF)=\left\{\stype{I_1}{V},\stype{I_2}{V},\stype{\{p\}}{V}\right\}$;
\item $\cV(\GV)=\left\{\typemark{\hF}{I_1}{V},\typemark{\hF}{I_2}{V},\typemark{\hF}{\{p\}}{V}\right\}$.
\end{itemize}
As a direct consequence of this,
$$\left|\res_{W_p}^V(\cV(\GV\ints\bG_{V;W_p}))\right|=1.$$
\end{lemma}
\begin{proof}

For simplicity we write $W:=W_p=V\setminus\{p\}$. If $|W|=2$ and $|\cF(W)|=1$, let $W=\{q_1,q_2\}$ for some $q_1\neq q_2$. By the corresponding discussions in Definition \ref{res of types}, $\bP_{q_1,q_2,W}=\bG_W$ and the restriction map $\res_W^V:\cA_\PSep(V;W)\to W$ is bijective. Since $\cV(\GV\ints\bP_{q_1,q_2,V})\neq \emptyset$, by Definition \ref{graph of sep} and the seventh assertion in Lemma \ref{properties of W-face}, $\cV(\GV\ints\bP_{q_1,q_2,V})$ is a subset of $\cA_\PSep(V;W)$ with cardinality 2, which implies that $\cV(\GV\ints\bP_{q_1,q_2,V})=\cA_\PSep(V;W)$. Hence $\res_W^V(\cV(\GV\ints\bG_{V;W}))\ints\cV(\bG_W)=\res_W^V(\cV(\GV\ints\bG_{V;W}))\ints\cV(\bP_{q_1,q_2,W})=W$. This contradicts the assumption that $p\in\Sing(\GV)$. Therefore $|W|>2$ or $|\cF(W)|>1$.

It remains for us to verify the three bullet points. Since $p\in\Sing(\GV)$, there exist distinct $q_1, q_2\in W$ and nonempty subsets $I_1, I_2:=W\setminus I_1$ of $ W$ such that
$$\res_W^V(\cV(\GV\ints\bG_{V;W})\ints\cV(\bP_{q_1,q_2,W}))=\{\WP:=(\hF,\{I_1,I_2\})\}.$$
By the seventh assertion in Lemma \ref{properties of W-face}, $\GV\ints\cV(\bP_{q_1,q_2,W})\neq\emptyset$ and hence there exist distinct $\VQ_1,\VQ_2\in\cV(\GV)$ such that $\res_W^V(\VQ_j)=\WP$ for any $j=1,2$. Therefore we can assume WLOG that $\VQ_j=\typemark{\hF}{I_j}{V}$. By the fact that $\VQ_1,\VQ_2\in\cV(\GV)$ and Definition \ref{graph of sep}, we have $\stype{I_1}{V},\stype{I_2}{V}\in\PSep_V(\hF)$. Therefore by Lemma \ref{prim decomp},
$$\PSep_V(\hF)=\left\{\stype{I_1}{V},\stype{I_2}{V},\stype{\{p\}}{V}\right\}.$$
If $\VQ_3:=\typemark{\hF}{\{p\}}{V}\not\in\cV(\bG_V)$, then by Definition \ref{graph of sep}, $\VQ_1,\VQ_2\in\del\cA_\PSep(V)$. This implies that $I_1=\{q_1\}$, $I_2=\{q_2\}$ and $F_{q_1}=F_{q_2}=\hF$ for some $q_1\neq q_2$, which contradicts the fact that $|W|>2$ or $|\cF(W)|>1$. Hence $\VQ_3\in\cV(\bG_V)$. By the first assertion in Lemma \ref{reinterpretation of edges} and the first assertion in Proposition \ref{key prop of ASep graph}, $\{\VQ_1,\VQ_2\}\subset\cV(\GV)\implies \VQ_3\in\cV(\GV)$. On the other hand, by the fifth assertion in Corollary \ref{remaining properties of sep graph},
$\cV(\GV)\setminus\{\VQ_1,\VQ_2,\VQ_3\}=\emptyset$. Therefore $\cV(\GV)=\{\VQ_1,\VQ_2,\VQ_3\}$.
\end{proof}

Thanks to the eighth assertion in Lemma \ref{properties of W-face}, we introduce the following notions.
\begin{definition}[Subordinate]\label{subordinate}
Let $W\subset V\subset \Gamma x_0$ be finite subsets such that $|W|\geq 2$ and $|V|\geq 3$. We say $\GW\in\MCS(\bG_W)$ is \emph{subordinate} to $\GVW\in\MCS(\bG_{V;W})$ if $\GW$ and $\GVW$ satisfies the relation described in the eighth assertion in Lemma \ref{properties of W-face}. To be specifc, $\GW\in\MCS(\bG_W)$ is \emph{subordinate} to $\GVW\in\MCS(\bG_{V;W})$ if and only if one of the following holds:
\begin{enumerate}
\item[(i).] There exist an edge $\WP_1\mbox{---}\WP_2\in\cE(\GW)$, an edge $\VQ_1\mbox{---}\VQ_2\in\cE(\GVW)$ and distinct $p,q\in W$ such that
$$\WP_1\mbox{---}\WP_2\in\cE(\Path_{p,q,W}(\res_W^V(\VQ_1),\res_W^V(\VQ_2))).$$
\item[(ii).] For any edge $\WP_1\mbox{---}\WP_2\in\cE(\GW)$, there exist an edge $\VQ_1\mbox{---}\VQ_2\in\cE(\GVW)$ and distinct $p,q\in W$ satisfying
$$\WP_1\mbox{---}\WP_2\in\cE(\Path_{p,q,W}(\res_W^V(\VQ_1),\res_W^V(\VQ_2))).$$
\end{enumerate}
The eighth assertion in Lemma \ref{properties of W-face} guarantees that (i) and (ii) are equivalent. Moreover, for any $\GW\in\MCS(\bG_W)$, it is subordinate to a unique $\GVW\in\MCS(\bG_{V;W})$.

For any $\GVW\in\MCS(\bG_V)$, we denote by $\subord_W^V(\GVW)$ the collection of all $\GW\in\MCS(\bG_W)$ which are subordinate to $\GVW$. Namely
$$\subord_W^V(\GVW):=\{\GW\in\MCS(\bG_W)|\GW~\mathrm{is ~subordinate~to}~\GVW\}.$$
\end{definition}
\begin{rmk}
Under the above settings, for any $\gamma\in\Gamma$ and any $\GW\in\MCS(\bG_W)$, by the remark after Definition \ref{graph of sep}, the remark after Corollary \ref{remaining properties of sep graph} and the remark after Definition \ref{res of types}, the map $\GW\to\gamma\GW$ gives a bijection from $\subord_W^V(\GVW)$ to $\subord_{\gamma W}^{\gamma V}(\gamma\GVW)$.
\end{rmk}

\begin{definition}[$(W;V)$-enrichment]\label{enrich}
Let $W\subset V\subset \Gamma x_0$ be finite subsets such that $|W|\geq 2$ and $|V|\geq 3$. For any $\GVW\in\MCS(\bG_{V;W})$, we define the \emph{$(W;V)$-enriched $\GVW$} as a subgraph of $\bG_W$ such that
$$\Enrich_W^V(\GVW)=\bigcup_{\GW\in\subord_W^V(\GVW)}\GW.$$
In particular, by the remark after Definition \ref{subordinate}, for any $\gamma\in\Gamma$, we have $\gamma\Enrich_W^V(\GVW)=\Enrich_{\gamma W}^{\gamma V}(\gamma\GVW)$.
\end{definition}

\begin{example}\label{W trivial case}
When $|W|=2$ and $|\cF(W)|=1$, we let $W=\{p,q\}$ with $p\neq q$. By the third assertion in Corollary \ref{remaining properties of sep graph}, the first assertion in Lemma \ref{properties of W-face} and the discussion at the beginning of the proof for the eighth assertion in Lemma \ref{properties of W-face}, we have $|\bG_W|=|\bG_{V;W}|=2$, $\MCS(\bG_W)=\{\bG_W\}$ and $\MCS(\bG_{V;W})=\{\bG_{V;W}\}$. Moreover, $\bG_{V;W}=\{\VQ_p:=\typemark{F_p}{\{p\}}{V},\VQ_q:=\typemark{F_q}{\{q\}}{V}\}$ with $F_p=F_q$. By Definition \ref{res of types} and Definition \ref{subordinate}, $\subord_W^V(\bG_{V;W})=\{\bG_W\}$ and hence
$$\Enrich_W^V(\bG_{V;W})=\bG_W=\Path_{p,q,W}(\res_W^V(\VQ_p),\res_W^V(\VQ_q))=\bP_{p,q,W}.$$
\end{example}
{\begin{example}[The case when $|V|=3$]\label{ex:triangle}
Let $V=\{p_0,p_1,p_2\}\subset \Gamma x_0$ with pairwise distinct $p_0,p_1,p_2$, and let $V_j=V\setminus\{p_j\}$ for any $j\in\{0,1,2\}$. For simplicity, we write $p_3:=p_0$ and $p_4:=p_1$. In this case, we have $\cA(V)=\cA_0(V)$ and hence $\cA_\Sep(V)=\cA_\PSep(V)=\AnSep(V)=\AnPSep(V)$. (See Definition \ref{Theta sep}, Notation \ref{edge marking} and Notation \ref{sep type marking}.) For simplicity, we write $p_3:=p_0$ and $p_4:=p_1$. Define
\begin{align}\label{angle portion}
\cA_{p_j}(V):=\{\typemark{\hF}{\{p_j\}}{V}|\hF\in\cF_{p_j}(p_{j+1},p_{j+2})\}\subset\cA_\Sep(V),
\end{align}
where $\cF_x(y,z)=\Theta(F_x,F_y)\ints\Theta(F_x,F_z)$ is defined in \eqref{I am lazy}. Hence by Definition \ref{graph of sep}, we have
\begin{align}\label{vertices when k=2}
\cV(\bG_V)=\cA_\Sep(V)=\cA_{p_0}(V)\sqcup\cA_{p_1}(V)\sqcup\cA_{p_2}(V).
\end{align}

By Lemma \ref{properties of Theta}, \hyperlink{Theta-3}{property ($\Theta$3) of $\Theta(\cdot,\cdot)$}, we can assume WLOG that for any $0\leq j\leq 2$,
\begin{align}\label{eqn:Fj(j+1,j+2)}
\cF_{p_j}(p_{j+1},p_{j+2})=\{F_{j;0}:=F_{p_j},F_{j;1},...,F_{j;m_j}\}
\end{align}
for some $m_j\geq 0$ and distinct $F_{j;0},...,F_{j;m_j}\in\cF(V)$ such that $\Theta(F_{j;0},F_{j;t})=\{F_{j;0},...,F_{j;t}\}$ for any $0\leq t\leq m_j$. Let $Q_{j;t}:=\typemark{F_{j;t}}{\{p_j\}}{V}\in\cA_{p_j}(V)$ for any $0\leq j\leq 2$ and $0\leq t\leq m_j$. Moreover, by Proposition \ref{almost ints positioning}, Lemma \ref{properties of Theta}, \hyperlink{Theta-1}{properties ($\Theta$1) and ($\Theta$3) of $\Theta(\cdot,\cdot)$}, $F_{j,t}\not\in\Theta(F_{p_{j+1}},F_{p_{j+2}})$ for any $0\leq j\leq 2$ and $0\leq t\leq m_j-1$. Then the following holds: (See Figure \ref{example6.39}.)

\textbf{Properties of $\bG_V$ and $\bG_{V;V_j}$}:
\begin{itemize}
\item $\cV(\bG_V)=\{Q_{j,t}|0\leq j\leq 2, 0\leq t\leq m_j\}$. Moreover, for any $0\leq j_1,j_2\leq 2$ and $0\leq t_l\leq m_{j_l}$ with $l\in\{1,2\}$, $Q_{j_1;t_1}=Q_{j_2;t_2}$ if and only if $j_1=j_2$ and $t_1=t_2$. This is just rewriting \eqref{vertices when k=2};
\item For any $0\leq j\leq 2$ and $0\leq t\leq m_j-1$, $Q_{j;t}$ is only connected to $Q_{j;t-1}$ (when $t\geq 1$) and $Q_{j;t+1}$ by a single edge in $\bG_V$.
\item For any $0\leq j\leq 2$, $Q_{j;m_j}$ is connected to $Q_{j;m_j-1}$ (when $m_j\geq 1$) and $Q_{i;m_i}$ by a single edge in $\bG_V$ for any $i\in\{0,1,2\}\setminus\{j\}$. These are the only edges with one endpoint being $Q_{j;m_j}$.
\item For any $0\leq j\leq 2$, $\bG_{V;V_j}$ is the restriction of $\bG_V$ onto $\cA_{p_{j+1}}(V)\sqcup\cA_{p_{j+2}}(V)$ in the sense of Definition \ref{res of graphs}. Moreover, $F_{0;0},...,F_{0;m_0-1},F_{1;0},...,F_{1;m_1-1},F_{2;0},...,F_{2;m_2-1}$ are distinct elements in $\cF(V)\setminus\{F_{0,m_0},F_{1,m_1},F_{2,m_2}\}$ and hence the restriction map $\res_{V_j}^V$ is injective on $\cA_{p_{j+1}}(V)\sqcup\cA_{p_{j+2}}(V)\setminus\{Q_{j+1;m_{j+1}},Q_{j+2;m_{j+2}}\}$. Here, for simplicity we write $F_{3;t}:=F_{0;t}$ and $F_{4;s}:=F_{1;s}$ for any $0\leq t\leq m_3:=m_0$ and $0\leq s\leq m_4:=m_1$.
\end{itemize}

\begin{figure}[h]
	\centering
	\includegraphics[width=6in]{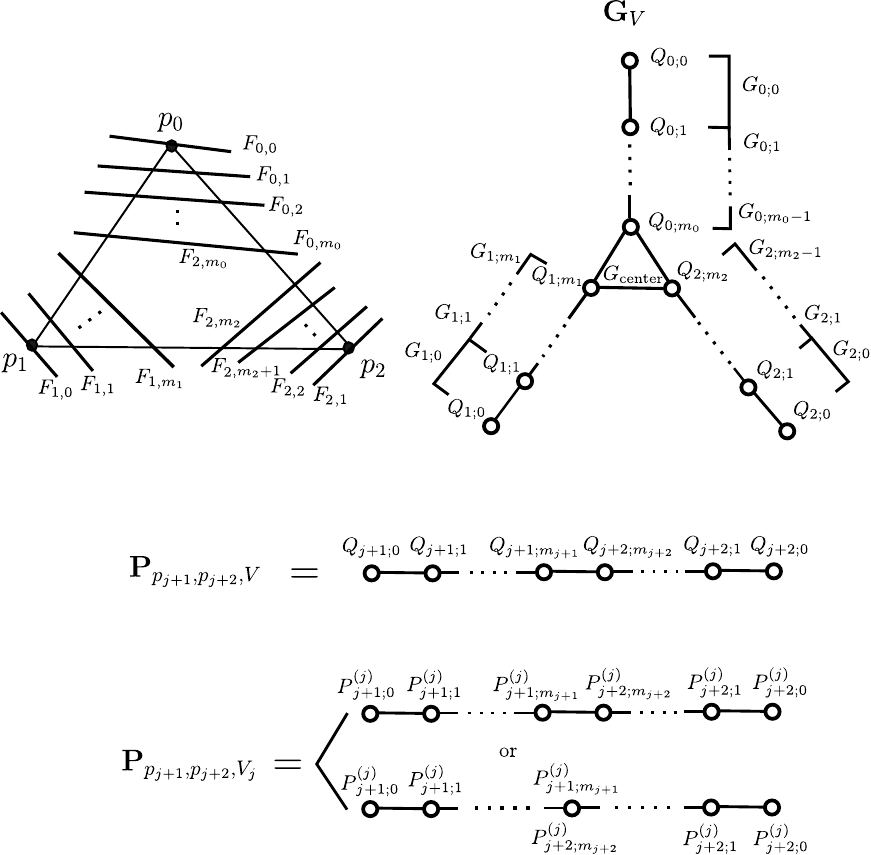}
	\caption{ \label{example6.39}}
\end{figure}

\textbf{MCS of $\bG_V$}: (See Figure \ref{example6.39}.)
\begin{itemize}
\item Let $G_{j;t}=Q_{j;t}\mbox{---} Q_{j;t+1}\in\MCS(\bG_V)$ and $G_{\mathrm{center}}$ be the complete subgraph of $\bG_V$ with $\cV(G_{\mathrm{center}})=\{Q_{0;m_0},Q_{1;m_1},Q_{2;m_2}\}$. Then we have
\begin{align}\label{MCS when k=2}
\MCS(\bG_V)=\{G_{j;t}|0\leq j\leq 2,0\leq t\leq m_j-1\}\sqcup\{G_{\mathrm{center}}\}.
\end{align}
\end{itemize}
For simplicity, $Q_{3;t}:=Q_{0;t}$ and $Q_{4;s}:=Q_{1;s}$ for any $0\leq t\leq m_0=m_3$ and any $0\leq s\leq m_1=m_4$. Similarly, we write $G_{3;t}:=G_{0;t}$ and $G_{4;s}:=G_{1;s}$ for any $0\leq t\leq m_0-1=m_3-1$ and any $0\leq s\leq m_1-1=m_4-1$. For any $0\leq j\leq 2$ fixed, any $0\leq t\leq m_{j+1}$ and any $0\leq s\leq m_{j+2}$, we define $P^{(j)}_{j+1;t}:=\res_{V_j}^V(Q_{j+1;t})$ and $P^{(j)}_{j+2;s}:=\res_{V_j}^V(Q_{j+2;s})$. We also write $P^{(1)}_{3;t}=:P^{(1)}_{0;t}$, $P^{(2)}_{3;s}=:P^{(2)}_{0;s}$ and $P^{(2)}_{4;s}=:P^{(2)}_{1;s}$ for any $0\leq t\leq m_0=m_3$ and any $0\leq s\leq m_1=m_4$. Then we obtain the following properties of $\bG_{V_j}$.

\textbf{Properties of $\bG_{V_j}$ and $\mathrm{Sing}(\cdot)$}: (See Definition \ref{singular} for $\mathrm{Sing}(\cdot)$. See Figure \ref{example6.39} for the corresponding picture.)
\begin{itemize}
\item $\bP_{p_{j+1},p_{j+2},V_j}=\Path_{p_{j+1},p_{j+2},V_j}(P^{(j)}_{j+1;0},P^{(j)}_{j+2;0})=\bG_{V_j}$.
\item For any $0\leq t\leq m_{j+1}-1$ and any $0\leq s\leq m_{j+2}-1$, we have $P^{(j)}_{j+1;t}=(F_{j+1;t},\{\{p_{j+1}\},\{p_{j+2}\}\})$ and $P^{(j)}_{j+2;s}=(F_{j+2;s},\{\{p_{j+1}\},\{p_{j+2}\}\})$. Moreover, $P^{(j)}_{j+1;0},...,P^{(j)}_{j+1;m_{j+1}-1},P^{(j)}_{j+2;0},...,P^{(j)}_{j+2;m_{j+2}-1}$ are distinct elements in $\cV(\bG_{V_j})\setminus\{P^{(j)}_{j+1;m_{j+1}},P^{(j)}_{j+2;m_{j+2}}\}$.

When $|\cF(V_j)|>1$, we also have $P^{(j)}_{j+1;m_{j+1}}=(F_{j+1;m_{j+1}},\{\{p_{j+1}\},\{p_{j+2}\}\})$ and $P^{(j)}_{j+2;m_{j+2}}=(F_{j+2;m_{j+2}},\{\{p_{j+1}\},\{p_{j+2}\}\})$. (This follows immediately from the aforementioned fact that $F_{0;0},...,F_{0;m_0-1},F_{1;0},...,F_{1;m_1-1},F_{2;0},...,F_{2;m_2-1}$ are distinct elements in $\cF(V)\setminus\{F_{0,m_0},F_{1,m_1},F_{2,m_2}\}$.)
\item For any $0\leq t\leq m_{j+1}$ and any $0\leq s\leq m_{j+2}$, $P^{(j)}_{j+1;t}\in\cV(\Path_{p_{j+1},p_{j+2},V_j}(P^{(j)}_{j+1;0},P^{(j)}_{j+2;m_{j+2}}))$ and $P^{(j)}_{j+2;s}\in\cV(\Path_{p_{j+1},p_{j+2},V_j}(P^{(j)}_{j+2;0},P^{(j)}_{j+1;m_{j+1}}))$. Moreover, we have
$$\Path_{p_{j+1},p_{j+2},V_j}(P^{(j)}_{j+1;0},P^{(j)}_{j+1;m_{j+1}})=P^{(j)}_{j+1;0}\mbox{---}\cdots\mbox{---}P^{(j)}_{j+1;m_{j+1}}$$
and
$$\Path_{p_{j+1},p_{j+2},V_j}(P^{(j)}_{j+2;0},P^{(j)}_{j+2;m_{j+2}})=P^{(j)}_{j+2;0}\mbox{---}\cdots\mbox{---}P^{(j)}_{j+2;m_{j+2}}.$$
(This follows directly from the assumptions on $F_{l;t}$ for any $0\leq l\leq 2$ and $0\leq t\leq m_l$.)
\item $\Sing(G_{j;t})=\emptyset$ for any $0\leq j\leq 2$ and $0\leq t\leq m_j-1$. Regarding $\Sing(G_{\mathrm{center}})$, we need to split the discussion into three cases:

\textbf{Case 1}: If $|\{F_{p_0},F_{p_1},F_{p_2}\}|=1$, then $m_0=m_1=m_2=0$. In this case $G_{\mathrm{center}}=\bG_V$. Since $|\cF(V_j)|=1$ for any $j\in\{0,1,2\}$, $\res^V_{V_j}$ is a bijection for any $j\in\{0,1,2\}$. (See Definition \ref{res of types}.) Hence $\mathrm{Sing}(G_{\mathrm{center}})=\emptyset$.

\textbf{Case 2}: If $|\{F_{p_0},F_{p_1},F_{p_2}\}|=2$, let $j\in\{0,1,2\}$ be the unique element such that $F_{p_j}\neq F_{p_i}$ for any $i\in\{0,1,2\}\setminus \{j\}$. In this case, $m_j\geq 1$ and $m_i=0$ for any $i\in\{0,1,2\}\setminus\{j\}$. Moreover, $F_{0;m_0}=F_{1;m_1}=F_{2;m_2}=F_i$ for any $i\in\{0,1,2\}\setminus\{j\}$. One can easily verify that $\mathrm{Sing}(G_{\mathrm{center}})=V_j=V\setminus\{p_j\}$.

\textbf{Case 3}: If $|\{F_{p_0},F_{p_1},F_{p_2}\}|=3$, then we have
$$\mathrm{Sing}(G_{\mathrm{center}})=\begin{cases}
\emptyset,~&\mathrm{if}~F_{0;m_0},F_{1;m_1},F_{2;m_2}~\text{are pairwise distinct},\\
V,~&\text{othewise}.
\end{cases}$$
\item Let $G^{(j)}_{j+1;t}=P^{(j)}_{j+1;t}\mbox{---}P^{(j)}_{j+1;t+1}$ and $G^{(j)}_{j+2;s}=P^{(j)}_{j+2;s}\mbox{---}P^{(j)}_{j+2;s+1}$ for any $0\leq t\leq m_{j+1}-1$ and any $0\leq s\leq m_{j+2}-1$. For simplicity, we let $G^{(1)}_{3;t}=:G^{(1)}_{0;t}$, $G^{(2)}_{3;t}=:G^{(2)}_{0;t}$ and $G^{(2)}_{4;t}=:G^{(2)}_{1;t}$. Then $G^{(j)}_{j+1;0},...,G^{(j)}_{j+1;m_{j+1}-1},G^{(j)}_{j+2;0},...,G^{(j)}_{j+2;m_{j+2}-1}$ are distinct MCS of $\bG_{V_j}$. Moreover, under the same assumptions on $t,s$, $\subord_{V_j}^V(G_{j+1;t})=\{G^{(j)}_{j+1;t}\}$ and $\subord_{V_j}^V(G_{j+2;s})=\{G^{(j)}_{j+2;s}\}$.
\end{itemize}
In Subsection \ref{subsec 2D case}, properties of this example will be heavily used. We conclude the discussion by introducing a few claims which are used in Subsection \ref{subsec 2D case}. The proof of these claims are straighforward from the above properties.

\textbf{Useful claims}:
\begin{enumerate}
\item (Used in the proof of the second assertion in Proposition \ref{full constructions} when $k=2$.) For any $j\in \{0,1,2\}$, any $G\in\MCS(\bG_V)$ and any $Q=\typemark{\hF}{\{p_j\}}{V}\in\cV(G)\setminus\del\cA_\Sep(V)$, the following holds:
\begin{itemize}
\item There exists a unique $G'\in\MCS(\bG_V)\setminus\{G\}$ such that $Q\in\cV(G')$.
\item Let $G'$ be the same as the above. If there exists $\typemark{F'}{I'}{V}\in\cV(G)\setminus \{Q\}$ such that $F'\in\Theta(F_{p_j},\hF)\setminus\{\hF\}$, then for any $\typemark{F''}{I''}{V}\in\cV(G')$, we have $\hF\in\Theta(F_{p_j},F'')$.
\item Let $G'$ be the same as the above. If there exists $\typemark{F'}{I'}{V}\in\cV(G)\setminus \{Q\}$ such that $\hF\in\Theta(F_{p_j},F')$, then for any $\typemark{F''}{I''}{V}\in\cV(G)$, we have $\hF\in\Theta(F_{p_j},F'')$. Moreover, for any $\typemark{F''}{I''}{V}\in\cV(G')\setminus\{Q\}$, we have $F''\in\Theta(F_{p_j},\hF)\setminus\{\hF\}$.
\end{itemize}
\item (Used in the proof of the third and the sixth assertions in Proposition \ref{full constructions} when $k=2$.) Let $G\in\MCS(\bG_V)$. If $\Sing(G)\neq \emptyset$ and $\cV(G)\cap\del\cA_\Sep(V)\neq \emptyset$, then $G=G_{\mathrm{center}}$ and there exists $\hF\in\{F_{p_0},F_{p_1},F_{p_2}\}$ such that
$$\cV(G)=\{\typemark{\hF}{\{p_0\}}{V},\typemark{\hF}{\{p_1\}}{V} ,\typemark{\hF}{\{p_2\}}{V}\}.$$
Moreover, $|\{p\in V|F_p=\hF\}|=|\cV(G)\cap\del\cA_\Sep(V)|\in\{1,2\}$.
\begin{itemize}
\item If $|\{p\in V|F_p=\hF\}|=|\cV(G)\cap\del\cA_\Sep(V)|=1$, then $\Sing(G)=V$.
\item If $|\{p\in V|F_p=\hF\}|=|\cV(G)\cap\del\cA_\Sep(V)|=2$, then $\Sing(G)=V\setminus\{j_0\}$, where $p_{j_0}$ is the unique element in $V$ which is not contained in $\{p\in V|F_p=\hF\}$. (Equivalently, $\typemark{F_{p_{j_0}}}{\{p_{j_0}\}}{V}$ is the unique element in $\del\cA_\Sep(V)\setminus\cV(G)$.) In this case, $|\cF(V_{j_0})|=1$ and $|V_{j_0}|=2$.
\end{itemize}
\item (Used in the proof of the sixth assertion in Proposition \ref{full constructions} when $k=2$.) For any $G\in\MCS(\bG_V)$ and any $j\in\{0,1,2\}$ such that $p_j\not\in\Sing(G)$, the map $\res_{V_j}^V|_{\cV(G\cap\bG_{V;V_j})}$ is injective.
\end{enumerate}
\end{example}}

\begin{lemma}\label{easy properties of enrich}
Let $W\subset V\subset \Gamma x_0$ be finite subsets such that $|W| \geq 2$ and $|V|\geq 3$. Then the following holds:
\begin{enumerate}
\item[(1).] For any $\GVW\in\MCS(\bG_{V;W})$, let $\VQ_0,...,\VQ_k$ be all the vertices of $\GVW$. Assume that $\VQ_0,...,\VQ_k$ are distinct. If $\subord_W^V(\GVW)\neq \emptyset$, then
$$\Enrich_W^V(\GVW)=\bigcup_{j,l:0\leq j< l\leq k}\Path_{p_{jl},q_{jl},W}(\res_W^V(\VQ_j),\res_W^V(\VQ_l))$$
for any suitable choices of distinct $p_{jl}, q_{jl}\in W$ such that the right hand side is well-defined.

In addition, for any distinct $p, q\in W$, either
$$\Enrich_W^V(\GVW)\ints\bP_{p,q,W}=\emptyset,$$
or there exist a unique pair of distinct vertices $\VQ_j,\VQ_l$ of $\GVW\cap\bP_{p,q,V}$ for some $j,l\in\{0,...,k\}$ such that $\res_W^V(\VQ_j)\neq \res_W^V(\VQ_l)$ and
$$\Enrich_W^V(\GVW)\ints\bP_{p,q,W}=\Path_{p,q,W}(\res_W^V(\VQ_j),\res_W^V(\VQ_l)).$$

\item[(2).] Let $\GVW\in\MCS(\bG_{V;W})$ be an arbitrary MCS of $\bG_{V;W}$ such that $\subord_W^V(\GVW)\neq\emptyset$. Then any MCS of $\Enrich_W^V(\GVW)$ is a MCS of $\bG_W$. Hence, for any $\WP\in\cV(\Enrich_W^V(\GVW))$, $\WP$ is contained in at most 2 MCS of $\Enrich_W^V(\GVW)$.

Moreover, when $|W|>2$ or $|\cF(W)|>1$, $ \del\Enrich_W^V(\GVW)=\res_W^V(\cV(\GVW))$, where $\del\Enrich_W^V(\GVW)$ is the collection of boundary elements of $\Enrich_W^V(\GVW)\subset\cA_\Sep(W)$ introduced in Definition \ref{bdry}.

In addition, $\WP$ is contained in a unique MCS of $\Enrich_W^V(\GVW)$ if and only if $\WP\in\res_W^V(\cV(\GVW))$. (This paragraph is true without the assumption that $|W|>2$ or $|\cF(W)|>1$.)
\item[(3).] Every vertex in $\bG_W$ is contained in at most two $(W;V)$-enriched MCS of $\bG_{V;W}$ and every edge in $\bG_{W}$ is contained in a unique $(W;V)$-enriched MCS of $\bG_{V;W}$. Moreover, for any vertex $\WP\in\bG_W$, the following holds:
\begin{itemize}
\item If $\WP\not\in\res_W^V(\cV(\bG_{V;W}))$, then $\WP$ is contained in a unique $(W;V)$-enriched MCS of $\bG_{V;W}$.
\item When $|W|=2$ and $|\cF(W)|=1$, for any $\WP\in\bG_W$, $\WP$ is contained in a unique $(W;V)$-enriched MCS of $\bG_{V;W}$.
\item Assume that $|W|>2$ or $|\cF(W)|>1$. If $\WP\in\res_W^V(\cV(\bG_{V;W}))\setminus\del\cA_\Sep(W)=\res_W^V(\cV(\bG_{V;W}))\setminus\{\typemark{F_x}{\{x\}}{W}|x\in W\}$, then $\WP$ is contained in exactly 2 $(W;V)$-enriched MCS of $\bG_{V;W}$.
\item Assume that $|W|>2$ or $|\cF(W)|>1$. If $\WP\in\del\cA_\Sep(W)=\{\typemark{F_x}{\{x\}}{W}|x\in W\}$, then $\WP$ is contained in a unique $(W;V)$-enriched MCS of $\bG_{V;W}$.
\end{itemize}
\end{enumerate}
\end{lemma}
\begin{proof}
\begin{enumerate}
\item[(1).] By Example \ref{W trivial case}, this assertion holds when $|W|=2$ and $|\cF(W)|=1$. Therefore we assume that $|W|>2$ or $|\cF(W)|>1$.

Notice that for any $j,l\in\{0,...,k\}$ with $j<l$ and any suitable choice of distinct $p_{jl},q_{jl}\in W$, by the first assertion in Corollary \ref{remaining properties of sep graph}, $\Path_{p_{jl},q_{jl},W}(\res_W^V(\VQ_j),\res_W^V(\VQ_l))$ is the unique shortest path connecting $\res_W^V(\VQ_j)$ and $\res_W^V(\VQ_l)$ in $\bG_W$. Moreover, for any edge $\WP_1\mbox{---}\WP_2\in\cE(\Path_{p_{jl},q_{jl},W}(\res_W^V(\VQ_j),\res_W^V(\VQ_l)))$, let $\GW\in\MCS(\bG_W)$ be the unique MCS in $\bG_W$ containing $\WP_1\mbox{---}\WP_2$. (Uniqueness is guaranteed by the first assertion in Proposition \ref{key prop of ASep graph}.) Then $\GW$ is subordinate to $\GVW$ in the sense of Definition \ref{subordinate}. Since we assumed that $\subord_W^V(\GVW)\neq \emptyset$, by Definition \ref{subordinate}, we have $|\res_W^V(\cV(\GVW))|\geq 2$. Hence the above discussions imply that
$$\Enrich_W^V(\GVW)=\bigcup_{\GW\in\subord_W^V(\GVW)}\GW\supset\bigcup_{j,l:0\leq j< l\leq k}\Path_{p_{jl},q_{jl},W}(\res_W^V(\VQ_j),\res_W^V(\VQ_l)).$$
On the other hand, for any $\GW\in\subord_W^V(\GVW)$ and any edge $\WP_1\mbox{---}\WP_2\in\cE(\GW)$, by Definition \ref{subordinate}, there exist some $j,l$ such that $0\leq j<l\leq k$ and
$$\WP_1\mbox{---}\WP_2\in\cE(\Path_{p_{jl},q_{jl},W}(\res_W^V(\VQ_j),\res_W^V(\VQ_l)))$$
for any suitable choice of distinct $p_{jl},q_{jl}\in W$ such that the right hand side is well-defined. This proves
$$\Enrich_W^V(\GVW)=\bigcup_{\GW\in\subord_W^V(\GVW)}\GW\subset\bigcup_{j,l:0\leq j< l\leq k}\Path_{p_{jl},q_{jl},W}(\res_W^V(\VQ_j),\res_W^V(\VQ_l)).$$
Hence
\begin{align}\label{enrich simplicial structure}
\Enrich_W^V(\GVW)=\bigcup_{j,l:0\leq j< l\leq k}\Path_{p_{jl},q_{jl},W}(\res_W^V(\VQ_j),\res_W^V(\VQ_l))
\end{align}
for any suitable choice of distinct $p_{jl}, q_{jl}\in W$ such that the right hand side is well-defined. This proves the first half of the first assertion.

Assume that $\Enrich_W^V(\GVW)\ints\bP_{p,q,W}\neq\emptyset$. Let $\WP_x:=\typemark{F_x}{\{x\}}{W}$ for any $x\in W$. We choose $\WP_{p;\GVW},\WP_{q;\GVW}\in\cV(\Enrich_W^V(\GVW)\ints\bP_{p,q,W})$ such that for any $\WP\in\cV(\Enrich_W^V(\GVW)\ints\bP_{p,q,W})$, we have
\begin{align}\label{eqn:lem6.40(1)-2prep0}
\WP\in\cV(\Path_{p,q,W}(\WP_{p;\GVW},\WP_q))\cap\cV(\Path_{p,q,W}(\WP_p,\WP_{q;\GVW})).
\end{align}
By the path structure of $\bP_{p,q,W}$ and \eqref{eqn:lem6.40(1)-2prep0}, the choice of the pair $\WP_{p;\GVW},\WP_{q;\GVW}\in\cV(\bP_{p,q,W})$ is unique. Moreover, one can immediately see that
\begin{align}\label{eqn:lem6.40(1)-subset2}
\Enrich_W^V(\GVW)\ints\bP_{p,q,W}\subset\Path_{p,q,W}(\WP_{p;\GVW},\WP_{q;\GVW}).
\end{align}
We first show that
\begin{align}\label{eqn:lem6.40(1)-2prep1}
\WP_{p;\GVW}\neq\WP_{q;\GVW}.
\end{align}
Indeed, by \eqref{eqn:lem6.40(1)-2prep0}, it suffices to show that $\Enrich_W^V(\GVW)\ints\bP_{p,q,W}$ contains more than 1 vertex. Since $\Enrich_W^V(\GVW)\ints\bP_{p,q,W}\neq \emptyset$, by Definition \ref{enrich}, there exists some $\GW\in\subord_W^V(\GVW)$ such that $\GW\cap\bP_{p,q,W}\neq \emptyset$. By the fourth assertion in Corollary \ref{remaining properties of sep graph} and Definition \ref{enrich}, we have $|\cV(\Enrich_W^V(\GVW)\ints\bP_{p,q,W})|\geq |\cV(\GW\cap\bP_{p,q,W})|= 2$. This proves \eqref{eqn:lem6.40(1)-2prep1}.

Now, we claim that
\begin{align}\label{eqn:lem6.40(1)-2prep2}
\WP_{p;\GVW},\WP_{q;\GVW}\in\res_W^V(\cV(\GVW))=\res_W^V(\{\VQ_0,...,\VQ_k\}).
\end{align}
\textbf{Proof of \eqref{eqn:lem6.40(1)-2prep2}}:
Since the proof of $\WP_{q;\GVW}\in\res_W^V(\cV(\GVW))$ is identical to the proof of $\WP_{p;\GVW}\in\res_W^V(\cV(\GVW))$, we only prove $\WP_{p;\GVW}\in\res_W^V(\cV(\GVW))$.

If $\WP_{p;\GVW}\not\in \res_W^V(\cV(\GVW))$, by \eqref{enrich simplicial structure}, there exist $s,t\in\{0,...,k\}$ with $s<t$ and some suitable choice of distinct $p_{st},q_{st}\in W$ such that
\begin{align}\label{eqn:lem6.40(1)-2prep3}
\WP_{p;\GVW}\in\cV(\Path_{p_{st},q_{st},W}(\res_W^V(\VQ_s),\res_W^V(\VQ_t)))\setminus\{\res_W^V(\VQ_s),\res_W^V(\VQ_t)\}.
\end{align}
Apply Definition \ref{bdry} and Lemma \ref{unique shortest paths} to the above, we have $\WP_{p;\GVW}\not\in\del\cA_\Sep(W)$. In particular, $\WP_{p;\GVW}\not\in\{\WP_p,\WP_q\}$ due to Lemma \ref{bdry=vertex}.

Let $\WP_{p,1},\WP_{p,2},\WP_{p,3}\in\cV(\bG_W)$ be distinct elements such that the following holds: (See Figure \ref{lem6.40(1)}.)
\begin{figure}[h]
	\centering
	\includegraphics[scale=1.1]{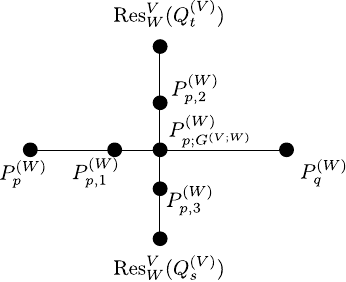}
	\caption{ \label{lem6.40(1)}}
\end{figure}
\begin{enumerate}
\item[(a)] $\WP_{p,1}\in\cV(\Path(\WP_p,\WP_{p;\GVW}))\setminus\{\WP_{p;\GVW}\}$ and $\WP_{p,1}\mbox{---}\WP_{p;\GVW}\in\cE(\bG_W)$. (This is guaranteed by the fact that $\WP_{p;\GVW}\not\in\{\WP_p,\WP_q\}$ and the fact that $\WP_{p;\GVW}\in\cV(\Enrich_W^V(\GVW)\ints\bP_{p,q,W})$ introduced right before \eqref{eqn:lem6.40(1)-2prep0}.)
\item[(b)] $\WP_{p,2},\WP_{p,3}\in\cV(\Path_{p_{st},q_{st},W}(\res_W^V(\VQ_s),\res_W^V(\VQ_t)))\setminus\{\WP_{p;\GVW}\}$. (This is guaranteed by \eqref{eqn:lem6.40(1)-2prep3}.)
\item[(c)] $\WP_{p,2}\not\in\cV(\bP_{p,q,W})$ and $\WP_{p,2}\mbox{---}\WP_{p;\GVW},\WP_{p,3}\mbox{---}\WP_{p;\GVW} \in\cE(\bG_W)$. (This is guaranteed by \eqref{enrich simplicial structure}, \eqref{eqn:lem6.40(1)-2prep0} and \eqref{eqn:lem6.40(1)-2prep3}.)
\end{enumerate}
For any $i\in\{1,2,3\}$, let $\GW_i$ be the unique MCS in $\bG_{W}$ containing the edge $\WP_{p,i}\mbox{---}\WP_{p;\GVW}$. (Uniqueness is guaranteed by the first assertion in Proposition \ref{key prop of ASep graph}.) Observe the following:
\begin{itemize}
\item If $\WP_{p,1}\mbox{---}\WP_{p,i}$ for some $i\in\{2,3\}$, then $\GW_1=\GW_i$ is the unique MCS in $\bG_W$ containing $\WP_{p,1}, \WP_{p,i},\WP_{p;\GVW}$ (due to Item (a) and Item (c) in the above, and the first assertion in Proposition \ref{key prop of ASep graph}). In this case, Item (b) and Definition \ref{subordinate} imply that $\GW_1=\GW_i\in\subord_W^V(\GVW)$. A direct consequence of this is that $\WP_{p,1}\in\cV(\Enrich_W^V(\GVW)\cap\bP_{p,q,W})$. This contradicts \eqref{eqn:lem6.40(1)-2prep0} and Item (a) in the above. Therefore $\WP_{p,1}$ and $\WP_{p,i}$ are not connected by an edge in $\bG_W$ for any $i\in\{2,3\}$. In particular, $\GW_1\neq \GW_i$ for any $i\in\{2,3\}$.
\item By Item (b) and Item (c) in the above, and the first assertion in Corollary \ref{remaining properties of sep graph}, $\WP_{p,2}$ and $\WP_{p,3}$ are not connected by an edge in $\bG_W$. In particular, $\GW_2\neq \GW_3$.
\end{itemize}
Since $\WP_{p,\GVW}\in\cV(\GW_1)\cap\cV(\GW_2)\cap\cV(\GW_3)$, the above two bullet points contradict the second assertion in Proposition \ref{key prop of ASep graph}. Therefore, $\WP_{p;\GVW}\in\res_W^V(\cV(\GVW))$. Recall the discussion right after the claim \eqref{eqn:lem6.40(1)-2prep2}. This finishes the proof of \eqref{eqn:lem6.40(1)-2prep2}.

It follows from \eqref{enrich simplicial structure}, \eqref{eqn:lem6.40(1)-subset2} and \eqref{eqn:lem6.40(1)-2prep2} that
$$\Enrich_W^V(\GVW)\ints\bP_{p,q,W}=\Path_{p,q,W}(\res_W^V(\VQ_{j}),\res_W^V(\VQ_{l})),$$
for some $j,l\in\{0,...,k\}$. Moreover, $\res_W^V(\VQ_{j})\neq \res_W^V(\VQ_{l})$ due to \eqref{eqn:lem6.40(1)-2prep1}. One can easily see that $\VQ_j,\VQ_l\in\cV(\bP_{p,q,V})$ following the definition of $\res_W^V(\cdot)$ in Definition \ref{subsec most annoying}. By the seventh assertion in Lemma \ref{properties of W-face} applied to $\GVW$ and the fact that $\res_W^V(\VQ_{j})\neq \res_W^V(\VQ_{l})$, we have
\begin{align*}
|\cV(\GVW\cap\bP_{p,q,V})\cap(\res_W^V)^{-1}(\res_W^V(\VQ_{j}))|\\
=&|\cV(\GVW\cap\bP_{p,q,V})\cap(\res_W^V)^{-1}(\res_W^V(\VQ_{l}))|=1.
\end{align*}
This proves the second half of the first assertion.

\item[(2).] By Example \ref{W trivial case}, this assertion holds when $|W|=2$ and $|\cF(W)|=1$. Therefore we assume that $|W|>2$ or $|\cF(W)|>1$.

Since $\subord_W^V(\GVW)\neq \emptyset$, we have $\Enrich_W^V(\GVW)\neq \emptyset$. Let $G$ be an arbitrary MCS of $\Enrich_W^V(\GVW)$, then $G$ is a complete subgraph of $\bG_W$. Moreover, by the first assertion of Lemma \ref{easy properties of enrich}, $\Enrich_W^V(\GVW)$ is a connected graph. Hence $G$ contains at least an edge. Let $\GW$ be a MCS of $\bG_W$ such that $G\subset \GW$. Then by Definition \ref{subordinate}, we have $\GW\in\subord_W^V(\GVW)$, which implies that $\GW\subset \Enrich_W^V(\GVW)$. By maximality of $G$, we have $G=\GW$. This proves that $\MCS(\Enrich_W^V(\GVW))\subset\MCS(\bG_W)$.

By the second assertion in Proposition \ref{key prop of ASep graph}, every vertex of $\Enrich_W^V(\GVW)$ is contained in at most 2 MCS of $\bG_W$. Therefore every vertex of $\Enrich_W^V(\GVW)$ is contained in at most 2 MCS of $\Enrich_W^V(\GVW)$.

By Definition \ref{bdry}, Lemma \ref{unique shortest paths} and the first assertion in Lemma \ref{easy properties of enrich}, we have
$$\del\Enrich_W^V(\GVW)\subset\res_W^V(\cV(\GVW)).$$
On the other hand, if $\WP\not\in \del\Enrich_W^V(\GVW)$, there exist distinct $\WP_1,\WP_2\in\cV(\Enrich_W^V(\GVW)$ such that $\WP\not\in\{\WP_1,\WP_2\}$ and $\WP$ is between $\WP_1$ and $\WP_2$ in the sense of Definition \ref{in between for ASep}. Therefore, there exist distinct $p, q\in W$ such that $\WP,\WP_1,\WP_2\in\cV(\bP_{p,q,W})$. By the first assertion in Lemma \ref{easy properties of enrich},
\begin{align}\label{eqn:lem6.40(2)-prep1}
\WP,\WP_1,\WP_2\in\cV(\Path_{p,q,W}(\res_W^V(\VQ_1),\res_W^V(\VQ_2)))
\end{align}
for some distinct $\VQ_1,\VQ_2\in\cV(\GVW).$

If in addition, $\WP\in \res_W^V(\cV(\GVW))$, we assume that $\WP=\res_W^V(\VQ_3)$ for some $\VQ_3\in\cV(\GVW)$ and $\VQ_j=\typemark{F_j}{I_j}{V}$, $j=1,2,3$. Notice that $\VQ_j\in\cV(\bP_{p,q,V})\cap\cV(\GVW)$ for any $j=1,2,3$. (This is due to Definition \ref{res of types} and the fact that $\WP,\WP_1,\WP_2\in\cV(\Path_{p,q,W}(\res_W^V(\VQ_1),\res_W^V(\VQ_2)))\subset \bP_{p,q,W}$.) By the seventh assertion in Lemma \ref{properties of W-face} applied to $\GVW$, we have $\VQ_3\in\{\VQ_1,\VQ_2\}$ and hence
\begin{align}\label{eqn:lem6.40(2)-prep2}
\WP\in\{\res_W^V(\VQ_1),\res_W^V(\VQ_2)\}.
\end{align}
On the other hand, recall that $\WP\not\in\{ \WP_1,\WP_2\}$, and that $\WP$ is between $\WP_1$ and $\WP_2$. By Lemma \ref{unique shortest paths}, \eqref{eqn:lem6.40(2)-prep1} and the path structure of $\Path_{p,q,W}(\res_W^V(\VQ_1),\res_W^V(\VQ_2))$, we have
\begin{align*}
\WP\in&\cV(\Path_{p,q,W}(\WP_1,\WP_2))\setminus\{\WP_1,\WP_2\}\\
\subset& \cV(\Path_{p,q,W}(\res_W^V(\VQ_1),\res_W^V(\VQ_2)))\setminus\{\res_W^V(\VQ_1),\res_W^V(\VQ_2)\}.
\end{align*}
This contradicts \eqref{eqn:lem6.40(2)-prep2}. Hence $\WP\not\in\res_W^V(\cV(\GVW))$. This proves that $\del\Enrich_W^V(\GVW)\supset\res_W^V(\cV(\GVW))$. Hence
\begin{align}\label{eqn:bdry of enrich}
\del\Enrich_W^V(\GVW)=\res_W^V(\cV(\GVW)).
\end{align}
Suppose $\WP\not\in\del\Enrich_W^V(\GVW)=\res_W^V(\cV(\GVW))$, then by the first assertion in Lemma \ref{easy properties of enrich}, there exist distinct $p, q\in W$, distinct $\VQ_1,\VQ_2\in\cV(\GVW)$ and distinct $\WP_1,\WP_2\in\cV(\Path_{p,q,W}(\res_W^V(\VQ_1),\res_W^V(\VQ_2)))$ such that
$$\WP\mbox{---}\WP_1,\WP\mbox{---}\WP_2\in\cE(\Path_{p,q,W}(\res_W^V(\VQ_1),\res_W^V(\VQ_2))).$$
By the first assertion of Corollary \ref{remaining properties of sep graph}, $\WP_1$ and $\WP_2$ are not connected by a single edge in $\bG_W$. Let $\GW_1,\GW_2\in\MCS(\bG_W)$ such that $\WP\mbox{---}\WP_j\in\cE(\GW_j)$, $j=1,2$. Then $\GW_1\neq \GW_2$. By Definition \ref{subordinate} and Definition \ref{enrich},  we know that $\GW_1,\GW_2\subset \Enrich_W^V(\GVW)$ and hence $\GW_1,\GW_2\in\MCS(\Enrich_W^V(\GVW))$ are distinct. Therefore, $\WP$ is contained in two MCS of $\Enrich_W^V(\GVW)$.

On the other hand, suppose $\WP$ is contained in two different MCS $\GW_1,\GW_2$ of $\Enrich_W^V(\GVW)$, choose $\WP_j\in\cV(\GW_j)\setminus\{\WP\}$. Earlier in the proof of this assertion, we showed that $\GW_1$ and $\GW_2$ are MCS of $\bG_W$. Then by the first assertion of Proposition \ref{key prop of ASep graph}, $\WP_1$ and $\WP_2$ are not equal and are not connected by a single edge in $\bG_W$. By the second assertion in Lemma \ref{reinterpretation of edges}, $\WP$ is between $\WP_1$ and $\WP_2$. Then by Definition \ref{bdry} and \eqref{eqn:bdry of enrich}, we have $\WP\not\in\del\Enrich_W^V(\GVW)=\res_W^V(\cV(\GVW))$. Hence $\WP$ is contained in a unique MCS of $\Enrich_W^V(\GVW)$ if and only if $\WP\in\res_W^V(\cV(\GVW))$.

\item[(3).] By Example \ref{W trivial case}, this assertion holds when $|W|=2$ and $|\cF(W)|=1$. Therefore we assume that $|W|>2$ or $|\cF(W)|>1$.

We first notice that for any $\GVW\in\MCS(\bG_{V;W})$, $\subord_W^V(\GVW)=\emptyset$ if and only if $\Enrich_W^V(\GVW)=\emptyset$. (See Definition \ref{subordinate} and Definition \ref{enrich}.) By Definition \ref{subordinate} and the second assertion in Lemma \ref{easy properties of enrich}, for any distinct $\GVW_1,\GVW_2\in\bG_{V;W}$, we have
\begin{align}\label{enrich similar to MCS}
\MCS(\Enrich_W^V(\GVW_1))\ints\MCS(\Enrich_W^V(\GVW_2))=\emptyset
\end{align}
and
$$\bigunion_{\GVW\in\MCS(\bG_{V;W})}\MCS(\Enrich_W^V(\GVW))=\MCS(\bG_W).$$
Therefore it follows from Proposition \ref{key prop of ASep graph} that every vertex in $\bG_W$ is contained in at most 2 $(W;V)$-enriched MCS of $\bG_{V;W}$ and every edge in $\bG_{W}$ is contained in a unique $(W;V)$-enriched MCS of $\bG_{V;W}$.
\begin{itemize}
\item If $\WP\not\in\res_W^V(\cV(\bG_{V;W}))$, then by the second assertion of Lemma \ref{easy properties of enrich}, for any $\GVW\in\MCS(\bG_{V;W})$ such that $\WP\in\cV(\Enrich_W^V(\GVW))$, $\WP$ is contained in exactly 2 MCS in $\Enrich_W^V(\GVW)$. By the first assertion in Proposition \ref{key prop of ASep graph}, $\WP$ is contained in at most 2 MCS of $\bG_W$. Hence by \eqref{enrich similar to MCS}, there exist a unique $\GVW$ such that $\WP\in\cV(\Enrich_W^V(\GVW))$.
\item As is mentioned in the beginning of the proof of this assertion, when $|W|=2$ and $|\cF(W)|=1$, this assertion holds from the discussion in Example \ref{W trivial case}.
\item Assume that $|W|>2$ or $|\cF(W)|>1$. If $\WP\in\res_W^V(\cV(\bG_{V;W}))\setminus\del\cA_\Sep(W)=\res_W^V(\cV(\bG_{V;W}))\setminus\{\typemark{F_x}{\{x\}}{W}|x\in W\}$, then by the second assertion of Corollary \ref{remaining properties of sep graph}, $\WP$ is contained in 2 distinct MCS $\GW_1,\GW_2$ of $\bG_W$. Let $\GVW_j$ be the unique MCS in $\bG_{V;W}$ such that $\GW_j\in\subord_W^V(\GVW_j)$, $1\leq j\leq 2$. (Uniqueness here is guaranteed by Definition \ref{subordinate}.) Since $\WP\in\res_W^V(\cV(\bG_{V;W}))$, by the second assertion of Lemma \ref{easy properties of enrich}, $\GVW_1\neq\GVW_2$. Therefore $\WP$ is contained in exactly 2 $(W;V)$-enriched MCS of $\bG_{V;W}$ due to \eqref{enrich similar to MCS}.
\item Assume that $|W|>2$ or $|\cF(W)|>1$. If $\WP\in\del\cA_\Sep(W)=\{\typemark{F_x}{\{x\}}{W}|x\in W\}$, then by the second assertion of Corollary \ref{remaining properties of sep graph}, $\WP$ is contained in a unique MCS of $\bG_W$. Hence $\WP$ is contained in a unique $(W;V)$-enriched MCS of $\bG_{V;W}$ due to \eqref{enrich similar to MCS}.\qedhere
\end{itemize}
\end{enumerate}
\end{proof}
\begin{rmk}
A direct corollary is that for any $\GVW\in\MCS(\bG_{V;W})$, we have
\begin{align}\label{eqn:MCS=subord}
\subord^V_W(\GVW)=\MCS(\Enrich^V_W(\GVW)).
\end{align}
This clearly holds when $\subord^V_W(\GVW)=\emptyset$. Assume that $\subord^V_W(\GVW)\neq\emptyset$. The fact that $\subord^V_W(\GVW)\subset\MCS(\Enrich^V_W(\GVW))$ follows directly from Definition \ref{enrich}. The fact that $\subord^V_W(\GVW)\supset\MCS(\Enrich^V_W(\GVW))$ follows from Definition \ref{subordinate} and the second assertion in Lemma \ref{easy properties of enrich}.
\end{rmk}

\begin{lemma}\label{combinatorics behind d2=0}
Let $U\subset W\subset V\subset \Gamma x_0$ be finite subsets such that $|U|\geq 2$ and $|W|\geq 3$. For any $\GV\in\MCS(\bG_V)$ such that $\GV\ints\bG_{V;W}\neq \emptyset$ (and hence $\GV\ints\bG_{V;W}\in\MCS(\bG_{V;W})$ due to the {fourth} assertion in Lemma \ref{properties of W-face}), if there exist some $\GW\in\subord_W^V(\GV\ints\bG_{V;W})$ with $\GW\ints\bG_{W;U}\neq\emptyset$, then $\GV\ints\bG_{V;U}\neq \emptyset$ (and hence $\GV\ints\bG_{V;U}\in\MCS(\bG_{V;U})$ due to the {fourth} assertion in Lemma \ref{properties of W-face}).

{Moreover,} if {$\GV\ints\bG_{V;U}\neq \emptyset$}, then we have
\begin{align}\label{important-1}\subord_U^V(\GV\ints\bG_{V;U})=\bigsqcup_{
\substack{
\GW:\GW\in\subord_W^V(\GV\ints\bG_{V;W})\\
{\GW\ints\bG_{W;U}\neq\emptyset}
}
}\subord_U^W(\GW\ints\bG_{W;U}).
\end{align}
As a corollary of this,
$$\Enrich_U^V(\GV\ints\bG_{V;U})=\bigcup_{
\substack{
\GW:\GW\in\subord_W^V(\GV\ints\bG_{V;W})\\
{\GW\ints\bG_{W;U}\neq\emptyset}
}
}\Enrich_U^W(\GW\ints\bG_{W;U}).$$
\end{lemma}
\begin{proof}
{Let $\GW\in\subord_W^V(\GV\ints\bG_{V;W})$ with {$\GW\ints\bG_{W;U}\neq\emptyset$}, by the fourth assertion in Lemma \ref{properties of W-face}, $\GW\ints\bG_{W;U}\in\MCS(\bG_{W;U})$. For any $\GU\in\subord_U^W(\GW\ints\bG_{W;U})$, there exist edges $\UO_1\mbox{---}\UO_2\in\cE(\GU)$ and $\WP_1\mbox{---}\WP_2\in\cE(\GW\ints\bG_{W;U})$ such that
$$\UO_1\mbox{---}\UO_2\in\cE(\Path_{p,q,U}(\res_U^W(\WP_1),\res_U^W(\WP_2)))$$
for some distinct $p, q\in U$. It follows from Definition \ref{enrich} and Definition \ref{res of types} that
$$\WP_1,\WP_2\in\cV(\bP_{p,q,W}\cap\GW)\subset \cV(\bP_{p,q,W}\cap\Enrich_W^V(\GV\cap\bG_{V;W})).$$
By the first assertion in Lemma \ref{easy properties of enrich}, there exist distinct $\VQ_1,\VQ_2\in\cV(\bP_{p,q,V})\ints\cV(\GV\ints\bG_{V;W})$ such that
$$\Enrich_W^V(\GVW)\ints\bP_{p,q,W}=\Path_{p,q,W}(\res_W^V(\VQ_1),\res_W^V(\VQ_2)).$$
Since $p,q\in U$, we have $\VQ_1,\VQ_2\in\cV(\GV\cap\bP_{p,q,V})\subset\cV(\GV\ints\bG_{V;U})\neq\emptyset$, which proves that $\GV\ints\bG_{V;U}\neq\emptyset$.}

{For the remaining parts of this lemma,} we first note that by the fourth assertion in Lemma \ref{properties of W-face}, both sides of both equations in the statement of this lemma are well-defined. Also, the disjoint union sign on the right hand side of \eqref{important-1} follows directly from Definition \ref{subordinate} and the third assertion in Lemma \ref{properties of W-face}.

For any $\GU\in\subord_U^V(\GV\ints\bG_{V;U})$, by Definition \ref{subordinate}, there exist some $\GWU\in\MCS(\bG_{W;U})$ such that $\GU\in\subord_U^W(\GWU)$. Therefore, there exist distinct $\UO_1,\UO_2\in\cV(\GU)$ and distinct $p,q\in U$ such that
$$\UO_1\mbox{---}\UO_2\in\cE(\GU)\subset\cE(\Enrich_U^W(\GWU)\ints\bP_{p,q,U})\neq\emptyset.$$
By the first assertion in Lemma \ref{easy properties of enrich}, there exist distinct $\WP_1,\WP_2\in\cV(\bP_{p,q,W})\ints\cV(\GWU)$ such that
$$\Enrich_U^W(\GWU)\ints\bP_{p,q,U}=\Path_{p,q,U}(\res_U^W(\WP_1),\res_U^W(\WP_2)).$$
Hence
\begin{align}\label{UW-1}
\UO_1\mbox{---}\UO_2\in\cE(\Path_{p,q,U}(\res_U^W(\WP_1),\res_U^W(\WP_2))).
\end{align}
By the third assertion in Lemma \ref{properties of W-face}, there exist a unique $\GW\in\MCS(\bG_W)$ such that $\GW\ints\bG_{W;U}=\GWU$. Again, Definition \ref{subordinate} implies that there exist some $\GVW\in\MCS(\bG_{V;W})$ such that $\GW\in\subord_W^V(\GVW)$. Therefore
$$\WP_1\mbox{---}\WP_2\in\cE(\GW)\subset\cE(\Enrich_W^V(\GVW)\ints\bP_{p,q,W})\neq\emptyset.$$
By the first assertion in Lemma \ref{easy properties of enrich}, there exist distinct $\VQ_1,\VQ_2\in\cV(\bP_{p,q,V})\ints\cV(\GVW)$ such that
$$\Enrich_W^V(\GVW)\ints\bP_{p,q,W}=\Path_{p,q,W}(\res_W^V(\VQ_1),\res_W^V(\VQ_2)).$$
Hence
\begin{align}\label{WV-1}
\WP_1\mbox{---}\WP_2\in\cE(\Path_{p,q,W}(\res_W^V(\VQ_1),\res_W^V(\VQ_2)))
\end{align}
By Lemma \ref{unique shortest paths}, $\WP_1,\WP_2$ are between $\res_W^V(\VQ_1)$ and $\res_W^V(\VQ_2)$ in the sense of Definition \ref{in between for ASep}.

By the first assertion in Lemma \ref{properties of W-face}, there exist a unique $\widetilde\GV\in\MCS(\bG_V)$ such that $\widetilde\GV\ints\bG_{V;W}=\GVW$. Notice that $p,q\in U$, by \eqref{WV-1} and the fact that $\bG_{V;U}\subset\bG_{V;W}$, we have $\VQ_1,\VQ_2\in\cV(\widetilde\GV\ints\bG_{V;U})=\cV(\GVW\ints\bG_{V;U})$. Hence by the fourth assertion in Lemma \ref{properties of W-face}, $\GVW\ints\bG_{V;U}\in\MCS(\bG_{V;U})$. Since $\res_U^W\circ\res_W^V=\res_U^V$ whenever either side is well-defined, it follows from the third assertion of Lemma \ref{res of vertices} that $\res_U^W(\WP_1),\res_U^W(\WP_2)$ are between $\res_U^V(\VQ_1)$ and $\res_U^V(\VQ_2)$ in the sense of Definition \ref{in between for ASep}. By the path structure of $\bP_{p,q,U}$ and Lemma \ref{unique shortest paths},
$$\Path_{p,q,U}(\res_U^W(\WP_1),\res_U^W(\WP_2))\subset\Path_{p,q,U}(\res_U^V(\VQ_1),\res_U^V(\VQ_2)).$$
Hence \eqref{UW-1} implies that
$$\UO_1\mbox{---}\UO_2\in\cE(\Path_{p,q,U}(\res_U^V(\VQ_1),\res_U^V(\VQ_2))).$$
This shows that $\GU\in\subord_U^V(\GVW\ints\bG_{V;U})$. By the eighth assertion in Lemma \ref{properties of W-face}, we have
$$\widetilde\GV\ints\bG_{V;U}=\GVW\ints\bG_{V;U}=\GV\ints\bG_{V;U}\in\MCS(\bG_{V;U}).$$
Hence, the {third} assertion in Lemma \ref{properties of W-face} implies that $\GV=\widetilde\GV$. In particular, {by the fourth assertion in Lemma \ref{properties of W-face}, $\GVW=\GV\ints\bG_{V;W}\neq \emptyset$ implies that $\GVW\in\MCS(\bG_{V;W})$. Therefore by the first assertion in Lemma \ref{properties of W-face}, $\GVW$ contains at least 1 edge.} By the assumptions on $\GW$ and $\GU$, we have $\GW\in\subord_W^V(\GV\ints\bG_{V;W})$ and $\GU\in\subord_U^W(\GW\ints\bG_{W;U})$. This proves that
$$\subord_U^V(\GV\ints\bG_{V;U})\subset\bigsqcup_{
\substack{
\GW\in\subord_W^V(\GV\ints\bG_{V;W})\\
{\GW\ints\bG_{W;U}\neq\emptyset}
}
}\subord_U^W(\GW\ints\bG_{W;U}).$$
On the other hand, {similar to the beginning of the proof,} for any $\GW\in\subord_W^V(\GV\ints\bG_{V;W})$ with {$\GW\ints\bG_{W;U}\neq\emptyset$}, by the fourth assertion in Lemma \ref{properties of W-face}, $\GW\ints\bG_{W;U}\in\MCS(\bG_W)$. For any $\GU\in\subord_U^W(\GW\ints\bG_{W;U})$, there exist edges $\UO_1\mbox{---}\UO_2\in\cE(\GU)$ and $\WP_1\mbox{---}\WP_2\in\cE(\GW\ints\bG_{W;U})$ such that
$$\UO_1\mbox{---}\UO_2\in\cE(\Path_{p,q,U}(\res_U^W(\WP_1),\res_U^W(\WP_2)))$$
for some distinct $p,q\in U$. By the first assertion in Lemma \ref{easy properties of enrich}, there exist distinct $\VQ_1,\VQ_2\in\cV(\bP_{p,q,V})\ints\cV(\GV\ints\bG_{V;W})$ such that
$$\Enrich_W^V(\GVW)\ints\bP_{p,q,W}=\Path_{p,q,W}(\res_W^V(\VQ_1),\res_W^V(\VQ_2)).$$
By Lemma \ref{unique shortest paths}, $\WP_1,\WP_2$ are between $\res_W^V(\VQ_1)$ and $\res_W^V(\VQ_2)$ in the sense of Definition \ref{in between for ASep}. Since $\res_U^W\circ\res_W^V=\res_U^V$ whenever either side is well-defined, it follows from the third assertion of Lemma \ref{res of vertices} that $\res_U^W(\WP_1),\res_U^W(\WP_2)$ are between $\res_U^V(\VQ_1)$ and $\res_U^V(\VQ_2)$ in the sense of Definition \ref{in between for ASep}. By the path structure of $\bP_{p,q,U}$ and Lemma \ref{unique shortest paths},
$$\UO_1\mbox{---}\UO_2\in\cE(\Path_{p,q,U}(\res_U^W(\WP_1),\res_U^W(\WP_2)))\subset\cE(\Path_{p,q,U}(\res_U^V(\VQ_1),\res_U^V(\VQ_2))),$$
which proves that $\GU\in\subord_U^V(\GV\ints\bG_{V;U})$. This proves that
$$\subord_U^V(\GV\ints\bG_{V;U})\supset\bigsqcup_{
\substack{
\GW:\GW\in\subord_W^V(\GV\ints\bG_{V;W})\\
{\GW\ints\bG_{W;U}\neq\emptyset}
}
}\subord_U^W(\GW\ints\bG_{W;U}).$$
Therefore \eqref{important-1} holds. Take the union of all elements in both the left hand side and the right hand side of \eqref{important-1}, we have
$$\Enrich_U^V(\GV\ints\bG_{V;U})=\bigcup_{
\substack{
\GW:\GW\in\subord_W^V(\GV\ints\bG_{V;W})\\
{\GW\ints\bG_{W;U}\neq\emptyset}
}
}\Enrich_U^W(\GW\ints\bG_{W;U}).\qedhere$$
\end{proof}


\section{Homological arguments}\label{sec hom arg}
\subsection{Notations and goals}
\begin{notation}\label{not:basic hom}
Let $C^X:=(C_\bullet(X;\RR),\del^X)$ be the augmented singular chain complex on $X$ with coefficients in $\RR$. To be specific, $C^X$ consists of $\RR[\Gamma]$-modules $C_i(X;\RR)$ with $-2\leq i<\infty$ and $\RR[\Gamma]$-module homomorphisms $\del^X_i:C_i(X;\RR)\to C_{i-1}(X;\RR)$ such that
\begin{itemize}
\item $C_{-2}(X;\RR)=0$ and $C_{-1}(X;\RR)=\RR$. The action of $\Gamma$ on $C_{-1}(X;\RR)=\RR$ is trivial.
\item $\del^X_{-1}=0$ and $\del^X_0(\sum_{p\in X}\alpha_p\cdot p)=\sum_{p\in X}\alpha_p$ for any $\sum_{p\in X}\alpha_p\cdot p\in C_0(X;\RR)$ with $\alpha_p\in\RR$. One can check easily that $\del^X_{0}$ is a homomorphism between $\RR[\Gamma]$-modules.
\item The definitions of $C_i(X;\RR)$ and $\del^X_{i+1}$ are introduced in Definition \ref{singular chains in X} for any $i\geq 0$.
\end{itemize}

Let $Y$ be the geometric realization of the homogeneous bar-construction for $\Gamma$. Namely, $Y$ is a simplicial complex whose $n$-simplices are identified with the ordered $(n+1)$-tuples of elements in $\Gamma$. We will not distinguish between ordered $(n+1)$-tuples of elements in $\Gamma$ and $n$-simplices in $Y$ when there is no ambiguity. In particular, $Y$ admits a natural $\Gamma$-action and $\Gamma\backslash Y$ is a $K(\Gamma,1)$-space. (See \cite[Example 1B.7]{Hatcher02} for the detailed definition of $Y$.)

Let $C^Y:=(C_\bullet(Y;\RR),\del^Y)$ be the augmented cellular chain complex on $Y$. To be specific, $C^Y$ consists of $\RR[\Gamma]$-modules $C_i(Y;\RR)$ with $-2\leq i<\infty$ and $\RR[\Gamma]$-module homomorphisms $\del^Y_i:C_i(Y;\RR)\to C_{i-1}(Y;\RR)$ such that
\begin{itemize}
\item $C_{-2}(Y;\RR)=0$ and $C_{-1}(Y;\RR)=\RR$. The action of $\Gamma$ on $C_{-1}(Y;\RR)=\RR$ is trivial.
\item $C_i(Y;\RR)$ is the $\RR$-vector space with basis $\cB_i=\{(\gamma_0,...,\gamma_{i}):\gamma_j\in\Gamma, 0\leq j\leq i\}$. The $\Gamma$-action is given by $\gamma(\gamma_0,...,\gamma_{i})=(\gamma\gamma_0,...,\gamma\gamma_{i})$ for any $\gamma\in\Gamma$. In other words, $C_i(Y;\RR)$ is the $\RR$-vector space generated by all $i$-dimensional simplices in $Y$.
\item $\del^Y_{-1}=0$, $\del^Y_0(\sum_{\gamma\in\Gamma}\alpha_\gamma\gamma)=\sum_{\gamma\in\Gamma}\alpha_\gamma$ for any $\sum_{\gamma\in\Gamma}\alpha_\gamma\gamma\in C_0(Y;\RR)$, and for any $i\geq 1$ and $\sum_{(\gamma_0,...,\gamma_i)\in\cB_i}\alpha_{(\gamma_0,...,\gamma_i)}(\gamma_0,...,\gamma_i)\in C_i(Y;\RR)$,
$$\del^Y_i\left(\sum_{(\gamma_0,...,\gamma_i)\in\cB_i}\alpha_{(\gamma_0,...,\gamma_i)}(\gamma_0,...,\gamma_i)\right)=\sum_{(\gamma_0,...,\gamma_i)\in\cB_i}\sum_{j=0}^i(-1)^{j}\alpha_{(\gamma_0,...,\gamma_i)}(\gamma_0,...,\widehat{\gamma_j},...,\gamma_i).$$
\end{itemize}
\end{notation}
We recall a well-known lemma from homological algebra.
\begin{lemma}[{\cite[Lemma I.7.4]{Brown82}}]\label{key hom lemma}
Let $(C_\bullet,\del)$, $(C'_\bullet,\del')$ be chain complexes of $\cR$-modules and $r$ be an integer, where $\cR$ is a ring. Let $(f_i:C_i\to C'_i)_{i\leq r}$ be a family of maps such that $\del'_i\circ f_i=f_{i-1}\circ\del_i$ for $i\leq r$. If $C_i$ is projective for $i>r$ and $H_i(C'_\bullet)=0$ for all $i\geq r$, then $(f_i)_{i\leq r}$ extends to a chain map $f_\bullet:C_\bullet\to C'_\bullet$ and $f_\bullet$ is unique up to chain homotopy.
\end{lemma}
\begin{rmk}
Since $X$ and $Y$ are contractible (i.e. $M=\Gamma\backslash X$ and $\Gamma\backslash Y$ are all $K(\Gamma,1)$ spaces), we have $H_i(C^X_\bullet)=H_i(C^Y_\bullet)=0$ for any $i\geq -1$. Notice that for any $i\geq 0$, $C_i(X;\RR)$ and $C_i(Y;\RR)$ are free $\RR[\Gamma]$-modules, hence they are projective $\RR[\Gamma]$-modules. By Lemma \ref{key hom lemma}, for any chain homomorphisms $\psi_\bullet:C^X_\bullet\to C^Y_\bullet$ and $\phi_\bullet:C^Y_\bullet\to C^X_\bullet$ with $\psi_{-1}=\phi_{-1}=\id:\RR\to\RR$, $\phi_\bullet\circ\psi_\bullet:C^X_\bullet\to C^X_\bullet$ is chain homotopic to the identity map on $C^X_\bullet$.
\end{rmk}

The goal of this section is to construct chain homomorphisms $\psi_\bullet:C^X_\bullet\to C^Y_\bullet$ and $\phi_\bullet:C^Y_\bullet\to C^{X,}_\bullet$ with $\psi_{-1}=\phi_{-1}=\id:\RR\to\RR$ so that it is convenient for us to show positivity of the simplicial volume of $M$. The composition $\phi_\bullet\circ\psi_\bullet$ can be thought of as a ``straightening'' of simplices. Vaguely speaking, one way to prove positivity of simplicial volume is to show that every ``straightened simplex'' has uniformly bounded ``volume'' with respect to some generator of $H^n(X;\RR)\isom H^n_{\mathrm{dR}}(X;\RR)$. In our constructions, we fix an $n$-form $\omega$ supported in an annulus region $A_{r,R}(\mathrm{Stab}_F(\Gamma)\backslash F)=A_{r,R}(N)$ for some suitable choice of $0<r<R\ll 1$ depending on $\epsilon_0$ and $\rho_0$. (Details regarding $\omega$ can be found in Section \ref{last sec}.) We denote also by $\omega$ its lift to the universal cover $X$. Then in $X$, $\omega$ is supported in $\sqcup_{\hF\in\Gamma F}A_{r,R}(\hF)$. By {the first assertion in Lemma \ref{properties of Theta sep} and Lemma \ref{bar uniform volume bound}}, for any {barycentric simplex with vertices in $V\subset \Gamma x_0$, its ``volume'' with respect to $\omega$ is bounded above by some constant depending only $|\cF(V)|$ and $\omega$.} Therefore, we want to construct $\psi_\bullet$ and $\phi_\bullet$ such that every ``straightened'' $n$-simplex is a weighted sum of ``special'' {barycentric} simplices satisfying the following:
\begin{itemize}
\item The total weight is uniformly bounded from above
\item For any ``special'' {barycentric} simplex, let $V$ be its set of vertices. Then $V\subset \Gamma x_0$ and $|\cF(V)|$ is uniformly bounded from above.
\end{itemize}

\subsection{Constructions of chain maps $\psi_\bullet$ and $\phi_\bullet$}\label{subsect:phi and psi}
\begin{notation}
The notion of ordered $(k+1)$-tuples in $\Gamma$ can be naturally extended to a $(k+1)$-real multilinear map from $C_0(Y;\RR)^{k+1}$ to $C_k(Y;\RR)$. Namely, for any $k\geq 0$ and any $\sum_{\gamma_i\in\Gamma}\alpha_{i,\gamma_i}\gamma_i\in C_0(Y;\RR)$ with $0\leq i\leq k$ and $\alpha_{i,\gamma_i}\in\RR$, we write
$$\left(\sum_{\gamma_0\in\Gamma}\alpha_{0,\gamma_0}\gamma_0,...,\sum_{\gamma_k\in\Gamma}\alpha_{k,\gamma_k}\gamma_i\right):=\sum_{\gamma_0,...\gamma_k:\gamma_0,...,\gamma_k\in\Gamma}\left(\prod_{i=0}^k\alpha_{i,\gamma_i}\right)\cdot\left(\gamma_0,...,\gamma_k\right)\in C_k(Y;\RR).$$
\end{notation}
\textbf{Construction of $\psi_\bullet$}: The construction of $\psi_\bullet:C^X_\bullet\to C^Y_\bullet$ can be rather flexible and much simpler than the construction of $\phi_\bullet:C^Y_\bullet\to C^X_\bullet$:

For any $x\in X$, we define $\psi_{-1}:\RR\to\RR$ to be the identity map and the $\RR[\Gamma]$-linear map $\psi_0:C_0(X;\RR)\to C_0(Y;\RR)=\RR[\Gamma]$ such that
$$\psi_0(x)=\frac{1}{|B_{\mathrm{diam}(M)+1}(x)\ints\Gamma x_0|}\sum_{\substack{\gamma\in\Gamma\\ \gamma x_0\in B_{\mathrm{diam}(M)+1}(x)\ints\Gamma x_0}}\gamma.$$
Then $\del^Y_0(\psi_0(x))=1=\psi_{-1}(\del^X_0(x))$. For any $k\geq 1$ and any $\sigma\in \cS_k(X)$, we denote by $x_{j;k}=(\delta_{0j},\delta_{1j},...,\delta_{kj})\in\Delta^k_{\RR^{k+1}}$ for any $0\leq j\leq k$, and define $\RR[\Gamma]$-linear maps $\psi_k:C_k(X;\RR)\to C_k(Y;\RR)$ such that
\begin{align}\label{psi}
\psi_k(\sigma)=\left(\psi_0(\sigma(x_{0;k})),...,\psi_0(\sigma(x_{k;k}))\right).
\end{align}
One can immediately see that for any $\sigma\in\cS_k(X)$, $\psi_k(\sigma)$ is a convex combination of $(k+1)$-tuples in $\Gamma$. In particular, for any $k\geq 0$ and any $a\in C_k(X;\RR)$, we have 
\begin{align}\label{eqn:psi norm control}
|\psi_k(a)|_{l^1}\leq |a|_{l^1}.
\end{align} 
Notice that for any $k\geq 1$,
\begin{align*}
\del_k^Y(\psi_k(\sigma))=&\del_k^Y\left(\psi_0(\sigma(x_{0;k})),...,\psi_0(\sigma(x_{k;k}))\right) \\
=&\sum_{j=0}^k(-1)^j\left(\psi_0(\sigma(x_{0;k})),...,\psi_0(\sigma(x_{j-1;k})),\psi_0(\sigma(x_{j+1;k})),...,\psi_0(\sigma(x_{k;k}))\right) \\
=&\psi_{k-1}(\del^X_k(\sigma)).
\end{align*}
This verifies that $\psi_\bullet:C^X\to C^Y$ is an $\RR[\Gamma]$-chain homomorphism such that $\psi_{-1}=\id$.

\textbf{Rough ideas for $\phi_\bullet:C^Y_\bullet\to C^X_\bullet$}: The rough idea behind the construction of $\phi_\bullet:C^Y_\bullet\to C^X_\bullet$ is the following: for any $(\gamma_0,...,\gamma_k)\in\Gamma$, we want to construct $\phi_k(\gamma_0,...,\gamma_k)$ as a ``straightened'' simplex with the set of vertices equal to $V:=\{\gamma_0x_0,...,\gamma_kx_0\}$ and edges of the form $\phi_1(\gamma_i,\gamma_j)=\beta[\gamma_ix_0,\gamma_jx_0]$ introduced in Section \ref{sec bicombing}, $0\leq i<j\leq k$. If we imagine that $\beta[x,y]$ is the {oriented geodesic segment from $x$ to $y$} and there is a {``geodesic simplex''} with edges exactly equal to $\beta[\gamma_ix_0,\gamma_jx_0]$, $0\leq i<j\leq k$. We can then apply the theory of $\Theta$-separation in Section \ref{sec combinatorics} to ``cut'' this ``geodesic simplex'' into ``pieces'' via elements in $\cF(V)$. (See Figure \ref{def7.4}.) ``Pieces'' will be later renamed as ``chambers'', which have a one-to-one correspondence to the elements in $\MCS(\bG_V)$. We will finally construct $\phi_k(\gamma_0,...,\gamma_k)$ by first performing ``filling constructions'' within each ``chamber'' using ``smaller simplices'' and then putting all chambers together to form the ``straightened'' simplex. Since $\phi_\bullet$ is a chain homomorphism, all the above constructions need to be created in a consistent way. This is based on the most technical results in Subsection \ref{subsec most annoying}.

The ``filling constructions'' mentioned above rely on several notions.
\begin{figure}[h]
	\centering
	\includegraphics[scale=1]{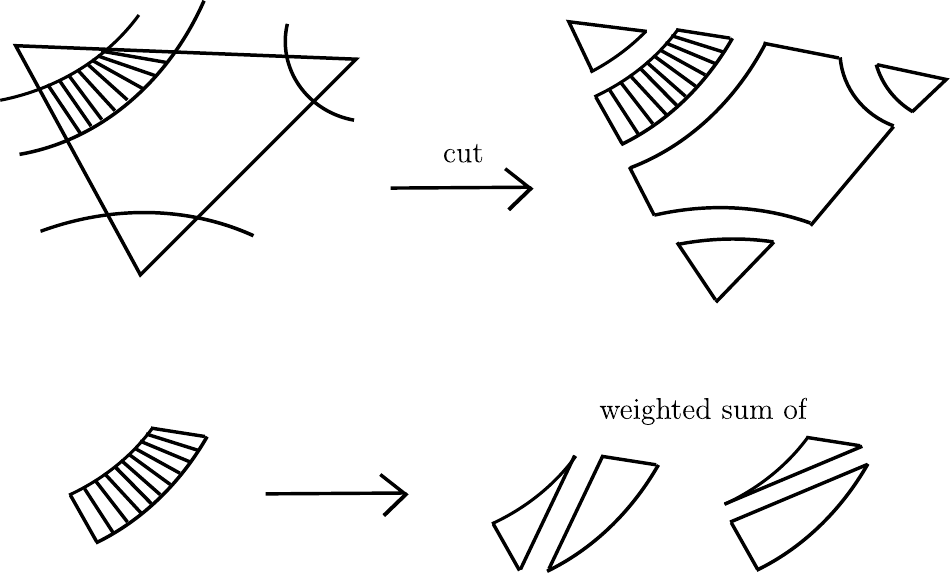}
	\caption{\label{def7.4}}
\end{figure}

\begin{definition}[Barycentric cone, barycentric chains]\label{geo cone}
{For any $p\in X$ and any $\sigma=\Db^{k}(p_0,...,p_k)$, we define the \emph{barycentric cone with apex $p$ and base $\sigma$} as $\Cone(p,\sigma):=\Db^{k+1}(p,p_0,...,p_k)$.

Let $C_{k;\mathrm{bar}}(X;\RR)$ be the collection of all $\RR$-linear combinations of $k$-dimensional barycentric simplices. Any element in $C_{k;\mathrm{bar}}(X;\RR)$ is called a \emph{$k$-dimensional barycentric chain.} We also let $C_{k,x_0;\mathrm{bar}}(X;\RR):=C_{k;\mathrm{bar}}(X;\RR)\cap C_{k,x_0}(X;\RR)$. Then the above definition naturally extends to a $\RR$-bilinear map $\Cone:C_0(X;\RR)\times C_{k;\mathrm{bar}}(X;\RR)\to C_{k+1;\mathrm{bar}}(X;\RR)$. This map naturally restricts to a $\RR$-bilinear map $\Cone:C_{0,x_0}(X;\RR)\times C_{k,x_0;\mathrm{bar}}(X;\RR)\to C_{k+1,x_0;\mathrm{bar}}(X;\RR)$.}
\begin{rmk}
It is straightforward from the definition of {barycentric} cones that
$$\del^X_{k+1}\Cone(p,a)=a-\Cone(p,\del^X_ka),~\forall p\in X, a\in C_{k,\mathrm{bar}}(X;\RR).$$
In particular, if $a$ is a closed $k$-dimensional {barycentric} chain (i.e. $\del^X_ka=0$), then $\del^X_{k+1}\Cone(p,a)=a$.
\end{rmk}
\end{definition}
\begin{definition}[Vertex support]\label{vert supp}
For any $k\geq 0$ and $a=\sum_{\sigma\in\cS_{k}}\alpha_\sigma\sigma\in C_{k}(X;\RR)$, we define its \emph{vertex support} as
$$\vertsupp(a)=\{p\in X|\exists \sigma\in\cS_{k} \mathrm{~s.t.~}p\mathrm{~is~a~vertex~of~}\sigma\mathrm{~and~}\alpha_\sigma\neq 0\}.$$

In particular, if $a\in C_{k,x_0}(X;\RR)$, $\vertsupp(a)\subset \Gamma x_0$. In this case, we also denote by
\begin{align}\label{neighboring flats}
\cF_a=\{\hF\in\Gamma F|\hF=F_x\mathrm{~for~some~}x\in\vertsupp(a)\}.\end{align}
\end{definition}

\begin{definition}[Cone average]\label{FBCone}
For any $k\geq 0$ and any {barycentric chain} $a\in C_{k,x_0;\mathrm{bar}}(X;\RR)$, {the \emph{cone average} of $a$ is defined as the following: If $a=0$, then $\FBCone(a):=0$. If $a\neq 0$, then
$$\FBCone(a)=\frac{1}{|\vertsupp(a)|}\sum_{z\in\vertsupp(a)}\Cone(z,a)\in C_{k+1,x_0;\mathrm{bar}}(X;\RR).$$}
\end{definition}

\begin{rmk}
Here are some straightforward properties of $\Cone$ and $\FBCone$.
\begin{enumerate}
\item[(1).] For any $k\geq 0$, $\Cone(p,a)=-\Cone(p,-a)$ for any $p\in X$ and $a\in C_{k;\mathrm{bar}}(X;\RR)$.
\item[(2).] For any $k\geq 0$, $|\Cone(p,a)|_{l^1}\leq |a|_{l^1}$ for any $p\in X$ and $a\in C_{k;\mathrm{bar}}(X;\RR)$.
\item[(3).] For any $k\geq 0$, ${\FBCone}(a)=-{\FBCone}(-a)$ for any $a\in C_{k;\mathrm{bar}}(X;\RR)$.
\item[(4).] For any $k\geq 0$, $|{\FBCone}(a)|_{l^1}\leq |a|_{l^1}$ for any $a\in C_{k;\mathrm{bar}}(X;\RR)$.
\item[(5).] For any $k\geq 0$, $\cF_{{\FBCone}(a)}{\subset}\cF_{a}$ for any $a\in C_{k,x_0;\mathrm{bar}}(X;\RR)$, where $\cF_a$ is defined in \eqref{neighboring flats}.
\item[(6).] For any $k\geq 0$ and any $a\in C_{k;\mathrm{bar}}(X;\RR)$, if $\del^X_ka=0$, then $\del^X_{k+1}{\FBCone}(a)=a$.
\item[(7).] For any $\gamma\in\Gamma$, any $k\geq 0$ and any $a\in C_{k,x_0;\mathrm{bar}}(X;\RR)$, we have $\gamma\vertsupp(a)=\vertsupp(\gamma a)$, $\gamma\cF_{a}=\cF_{\gamma a}$ and $\gamma{\FBCone}(a)={\FBCone}(\gamma a)$.
\end{enumerate}
\end{rmk}

Now we can start constructing $\phi_\bullet$. We define $\RR$-linear maps $\phi_{-1}=\id:\RR\to\RR$ and $\phi_0(\gamma)=\gamma x_0$ for any $\gamma\in\Gamma$. One can check that both $\phi_{-1}$ and $\phi_0$ are $\Gamma$-equivariant. Moreover, $\del^X_0(\phi_0(\gamma))=1=\phi_{-1}(\del^Y_0(\gamma))$.

\textbf{Construction of $\phi_1$}: For any $(\gamma_0,\gamma_1)\in C_1(Y;\RR)$, we define the $\RR$-linear map $\phi_1$ such that $\phi_1(\gamma_0,\gamma_1)=0$ if $\gamma_0=\gamma_1$ and $\phi_1(\gamma_0,\gamma_1)=\beta[\gamma_0x_0,\gamma_1x_0]$. $\Gamma$-equivariance of $\phi_1$ follows from the fact that $\beta[\cdot,\cdot]$ is a homological bicombing in the sense of Definition \ref{Homological bicombing}. Moreover, the same fact implies that $\del^X_1(\phi_1(\gamma_0,\gamma_1))=\gamma_1x_0-\gamma_0x_0=\phi_0(\del^Y_1(\gamma_0,\gamma_1))$.

For simplicity, we let $p_0=\gamma_0x_0$ and $p_1=\gamma_1x_0$ for any $\gamma_0,\gamma_1\in\Gamma$. Since $\phi$ is anti-symmetric in the sense that $\phi_1(\gamma_1,\gamma_0)=-\phi_1(\gamma_0,\gamma_1)$, we only consider the case when $p_0\neq p_1$. Let $V=\{p_0,p_1\}$ and $\Delta=(\gamma_0,\gamma_1)\in\cB_1\subset C_1(Y;\RR)$. (See Notation \ref{not:basic hom}.)

Although the constructions at this level do not consume a lot of effort, we can already establish the relations between decompositions of the homological bicombing in \eqref{beta' detailed full} and MCS of $\bG_V$, which will play a central role in the constructions of $\phi_k$ when $k\geq 2$.

By Lemma \ref{properties of Theta}, \hyperlink{Theta-3}{property ($\Theta$3) of $\Theta(\cdot,\cdot)$}, we can assume that $\Theta(F_{p_0}, F_{p_1})=\{F_0:=F_{p_0},F_1,...,F_m:=F_{p_1}\}$ for some $m\geq 0$ and distinct $F_1,...,F_{m-1}\in\Gamma F\setminus\{F_{p_0},F_{p_1}\}$ with $\Theta(F_0,F_j)=\{F_0,...,F_j\}$ for any $0\leq j\leq m$. We first introduce the notion of \emph{chambers} and \emph{floors} for the given $V$ and $\Delta$.

\textbf{When $m=0$}: In this case, $|V|=2$ and $|\cF(V)|=1$. Therefore $\bG_V=p_0\mbox{---}p_1$ and hence $\MCS(\bG_V)=\{\bG_V\}$. We define the \emph{(1-dimensional) chamber for $\Delta$ subject to $\bG_V$} as
\begin{align}\label{1D chamber special}
\Chamber_\Delta(\bG_V)=\beta[p_0,p_1]=\frac{1}{2}\left(\beta'[p_0,p_1]-\beta'[p_1,p_0]\right)\in C_{1,x_0;\mathrm{bar}}(X;\RR).
\end{align}
The \emph{(0-dimensional) floor for $\Delta$ subject to $\bG_V$ at $p_0$} is defined as the $p_0$ part in $\del^X_1\Chamber_\Delta(\bG_V)$. Namely,
\begin{align}\label{0D floor special-1}
\Floor_{\Delta,\bG_V}(p_0)=\frac{1}{2}\left(-\cI'_{F_{p_0},p_0}[p_0,p_1]-\cI'_{F_{p_0},p_0}[p_1,p_0]\right)=-p_0\in C_{0,x_0;\mathrm{bar}}(X;\RR).
\end{align}
The \emph{(0-dimensional) floor for $\Delta$ subject to $\bG_V$ at $p_1$} is defined as the $p_1$ part in $\del^X_1\Chamber_\Delta(\bG_V)$. Namely,
\begin{align}\label{0D floor special-2}
\Floor_{\Delta,\bG_V}(p_1)=\frac{1}{2}\left(\cI'_{F_{p_1},p_1}[p_0,p_1]+\cI'_{F_{p_1},p_1}[p_1,p_0]\right)=p_1\in C_{0,x_0;\mathrm{bar}}(X;\RR).
\end{align}
See \eqref{easy del beta' seg} for the definitions of $\cI'_{F_{p_j},p_j}[p_0,p_1]$, $j=1,2$.

\textbf{When $m\geq 1$}: In this case, $|\cF(V)|>1$. Therefore
$$\bG_V=\typemark{F_0}{\{p_0\}}{V}\mbox{---}\typemark{F_1}{\{p_0\}}{V}\mbox{---}\cdots\mbox{---}\typemark{F_m}{\{p_0\}}{V}.$$
We denote by $G_j=\typemark{F_j}{\{p_0\}}{V}\mbox{---}\typemark{F_{j+1}}{\{p_0\}}{V}$, $0\leq j\leq m-1$. Then $\MCS(\bG_V)=\{G_0,...,G_{m-1}\}$. Recall that in \eqref{beta' detailed full} and \eqref{beta} we have $\beta[p_0,p_1]=(\beta'[p_0,p_1]-\beta'[p_1,p_0])/2$,
$$\beta'[p_0,p_1]=\beta'_{F_0F_1}[p_0,p_1]+...+\beta'_{F_{m-1}F_{m}}[p_0,p_1]\mathrm{~and~}\beta'[p_1,p_0]=\beta'_{F_{m}F_{m-1}}[p_1,p_0]+...+\beta'_{F_1F_0}[p_1,p_0].$$
For any $0\leq j\leq m-1$, we define the \emph{(1-dimensional) chamber for $\Delta$ subject to $G_j$} as
\begin{align}\label{1D chamber}
\Chamber_\Delta(G_j)=\frac{1}{2}\left(\beta'_{F_jF_{j+1}}[p_0,p_1]-\beta'_{F_{j+1}F_j}[p_1,p_0]\right)\in C_{1,x_0;\mathrm{bar}}(X;\RR).
\end{align}
The \emph{(0-dimensional) floor for $\Delta$ subject to $G_j$ at $\typemark{F_j}{\{p_0\}}{V}$} is defined as the $\cI_{F_j}'$ part of $\del_1^X\Chamber_\Delta(G_j)$. (See \eqref{beta' seg endpts} and \eqref{del beta' seg}.) Namely,
\begin{align}\label{0D floor-1}
\Floor_{\Delta,G_j}(\typemark{F_j}{\{p_0\}}{V})=\frac{1}{2}\left(-\cI'_{F_{j}}[p_0,p_1]-\cI'_{F_{j}}[p_1,p_0]\right)\in C_{0,x_0;\mathrm{bar}}(X;\RR).
\end{align}
The \emph{(0-dimensional) floor for $\Delta$ subject to $G_j$ at $\typemark{F_{j+1}}{\{p_0\}}{V}$} is defined as the $\cI_{F_{j+1}}'$ part of $\del_1^X\Chamber_\Delta(G_j)$. (See \eqref{beta' seg endpts} and \eqref{del beta' seg}.) Namely,
\begin{align}\label{0D floor-2}
\Floor_{\Delta,G_j}(\typemark{F_{j+1}}{\{p_0\}}{V})=\frac{1}{2}\left(\cI'_{F_{j+1}}[p_0,p_1]+\cI'_{F_{j+1}}[p_1,p_0]\right)\in C_{0,x_0;\mathrm{bar}}(X;\RR).
\end{align}
Recall that $\Delta=(\gamma_0,\gamma_1)$ and $(p_0,p_1)=(\gamma_0x_0,\gamma_1x_0)$, one can see from \eqref{0D floor-1} and \eqref{0D floor-2} that for any $G\in\MCS(\bG_V)$ and any $Q=\typemark{\hF}{\{p_0\}}{V}\in\cV(G)$, the signs in the definition of $\Floor_{\Delta, G}(Q)$ are negative if $\hF$ is the closer to $F_{p_0}$; the signs in the definition of $\Floor_{\Delta, G}(Q)$ are positive if $\hF$ is the closer to $F_{p_1}$.

For any $m\geq 0$, we introduce two more notions related to the ``chambers'' and ``floors'' defined above. Let
\begin{align}\label{0D bchamber}
\BChamber_\Delta(G)=\del^X_1\Chamber_\Delta(G)\in C_{0,x_0;\mathrm{bar}}(X;\RR),~\forall G\in\MCS(\bG_V)
\end{align}
and
\begin{align}\label{-1D bfloor}
\BFloor_{\Delta,G}(Q)=\del_0^X\Floor_{\Delta,G}(Q)\in C_{-1}(X;\RR)\isom\RR,~\forall G\in\MCS(\bG_V)~\mathrm{and}~Q\in\cV(G).
\end{align}
Then we have the following observations.
\begin{proposition}[Properties of ``Floors'' and ``Chambers'' in their lowest possible dimensions]\label{baby constructions}
Let $\Delta=(\gamma_0,\gamma_1)$ for some $\gamma_0\neq\gamma_1\in\Gamma$ and $V=\{p_0:=\gamma_0x_0,p_1:=\gamma_1x_0\}$. Then the following holds.
\begin{enumerate}
\item[(1).] $\phi_1(\Delta)=\sum_{G\in\MCS(\bG_V)}\Chamber_\Delta(G)$.
\item[(2).] For any $G\in \MCS(\bG_V)$ and any $Q\in\cV(G)$, $\BFloor_{\Delta,G}(Q)\in\{1,-1\}$. Moreover, for any $Q\in\cV(\bG_V)$, if there exist $G\neq G'\in\MCS(\bG_V)$ such that $Q\in\cV(G\ints G')$, then
$$\BFloor_{\Delta, G}(Q)+\BFloor_{\Delta,G'}(Q)=0\mathrm{~and~}\Floor_{\Delta, G}(Q)+\Floor_{\Delta,G'}(Q)=0.$$
\item[(3).] For any $G\in \MCS(\bG_V)$ and any $Q\in\cV(G)$, let
\begin{align}\label{eqn:flat part}
F_Q=\hF\text{ when }|\cF(V)|=\{\hF\}\text{ or }Q=\typemark{\hF}{I}{V}.
\end{align}
Then we have
$$\vertsupp(\Floor_{\Delta, G}(Q))\subset\{x\in\Gamma x_0|F_x=F_Q\}.$$
\item[(4).] For any $G\in \MCS(\bG_V)$ and any $Q\in\cV(G)$, $\BFloor_{\Delta, G}(Q)$ and $\Floor_{\Delta, G}(Q)$ are alternating with respect to $\Delta$. Namely, for any permutation $\tau:\{0,1\}\to\{0,1\}$ and $\Delta'=(\gamma_{\tau(0)},\gamma_{\tau(1)})$, we have
$$\BFloor_{\Delta, G}(Q)=(-1)^{\mathrm{sgn}(\tau)}\BFloor_{\Delta', G}(Q)\mathrm{~and~}\Floor_{\Delta, G}(Q)=(-1)^{\mathrm{sgn}(\tau)}\Floor_{\Delta', G}(Q).$$
\item[(5).] For any $G\in \MCS(\bG_V)$, $\BChamber_{\Delta}(G)$ and $\Chamber_{\Delta}(G)$ are alternating with respect to $\Delta$. Namely, for any permutation $\tau:\{0,1\}\to\{0,1\}$ and $\Delta'=(\gamma_{\tau(0)},\gamma_{\tau(1)})$, we have
$$\BChamber_{\Delta}(G)=(-1)^{\mathrm{sgn}(\tau)}\BChamber_{\Delta'}(G)\mathrm{~and~}\Chamber_{\Delta}(G)=(-1)^{\mathrm{sgn}(\tau)}\Chamber_{\Delta'}(G).$$
\item[(6).] For any $G\in \MCS(\bG_V)$, we define $\cF_V(G):=\left\{\hF\in\Gamma F\left|
\hF\in\Theta(F_{Q_1},F_{Q_2}),Q_1,Q_2\in\cV(G).
\right.\right\}$. (See \eqref{eqn:flat part}.) Then
$$\vertsupp(\BChamber_\Delta(G)),\vertsupp(\Chamber_\Delta(G))\subset\{x\in\Gamma x_0|F_x\in\cF_V(G)\}.$$
Moreover, one can choose $\cC_7(1):=2$ such that $|\cF_V(G)|\leq \cC_7(1)$ for any $|V|=2$ and any $G\in\MCS(\bG_V)$.
\item[(7).] For any $\gamma\in\Gamma$, any $G\in\MCS(\bG_V)$ and any $Q\in \cV(G)$, the following holds.
\begin{itemize}
\item $\gamma\BFloor_{\Delta,G}(Q)=\BFloor_{\gamma\Delta,\gamma G}(\gamma Q)$.
\item $\gamma\Floor_{\Delta,G}(Q)=\Floor_{\gamma\Delta,\gamma G}(\gamma Q)$.
\item $\gamma\BChamber_{\Delta}(G)=\BChamber_{\gamma\Delta}(\gamma G)$.
\item $\gamma\Chamber_{\Delta}(G)=\Chamber_{\gamma\Delta}(\gamma G)$.
\end{itemize}
\item[(8).] For any $G\in\MCS(\bG_V)$, we have
$$\BChamber_\Delta(G)=\sum_{Q\in\cV(G)}\Floor_{\Delta,G}(Q).$$
\item[(9).] For any $G\in \MCS(\bG_V)$ and any $Q\in\cV(G)$, $|\Floor_{\Delta,G}(Q)|_{l^1}=1$. Hence $|\BChamber_\Delta(G)|_{l^1}\leq 2$ and $|\Chamber_\Delta(G)|_{l^1}\leq 1$.
\end{enumerate}
\end{proposition}
\begin{proof}
This proposition is simply a summary of the results in Section \ref{sec bicombing} but rewritten partially in the languages of Section \ref{sec combinatorics}.
\begin{enumerate}
\item[(1).] Since $\phi_1(\gamma_0,\gamma_1)=\beta[\gamma_0x_0,\gamma_1x_0]=\beta[p_0,p_1]=(\beta'[p_0,p_1]-\beta'[p_1,p_0])/2$, this follows immediately from \eqref{beta' detailed full}, \eqref{1D chamber special} and \eqref{1D chamber}.
\item[(2).] This follows directly from \eqref{0D floor special-1}, \eqref{0D floor special-2}, \eqref{0D floor-1}, \eqref{0D floor-2} and the discussions near \eqref{beta' seg endpts norm}.
\item[(3).] This follows directly from \eqref{0D floor special-1}, \eqref{0D floor special-2}, \eqref{0D floor-1}, \eqref{0D floor-2}, \eqref{beta' seg endpts} and the third assertion in Proposition \ref{prop of f}.
\item[(4).] This follows directly from \eqref{0D floor special-1}, \eqref{0D floor special-2}, \eqref{0D floor-1}, \eqref{0D floor-2} and \eqref{-1D bfloor}.
\item[(5).] This follows directly from \eqref{1D chamber special}, \eqref{1D chamber} and \eqref{0D bchamber}.
\item[(6).] This follows directly from \eqref{1D chamber special}, \eqref{1D chamber}, \eqref{0D bchamber}, \eqref{l1 norms of beta' seg} and Definition \ref{l1-seminorm}.
\item[(7).] This follows directly from equations \eqref{1D chamber special}-\eqref{-1D bfloor}, the remark after Definition \ref{graph of sep}, and the second assertion in Proposition \ref{prop of f} applied to \eqref{beta' seg endpts} and \eqref{beta' seg}.
\item[(8).] When $|\cF(V)|=1$, this assertion follows directly from \eqref{1D chamber special}, \eqref{0D floor special-1}, \eqref{0D floor special-2} and \eqref{0D bchamber}. When $|\cF(V)|\geq 2$, this assertion follows directly from \eqref{1D chamber}, \eqref{0D floor-1}, \eqref{0D floor-2}, \eqref{0D bchamber} and \eqref{del beta' seg}.
\item[(9).] The fact that $|\Floor_{\Delta,G}(Q)|_{l^1}=1$ follows directly from \eqref{0D floor special-1}, \eqref{0D floor special-2}, \eqref{0D floor-1}, \eqref{0D floor-2} and \eqref{beta' seg endpts norm}. The fact that  $|\BChamber_\Delta(G)|_{l^1}\leq 2$ then follows from the eighth assertion in Proposition \ref{baby constructions}. Finally, the fact that $|\Chamber_\Delta(G)|_{l^1}\leq 1$ follows directly from \eqref{1D chamber special}, \eqref{1D chamber} and \eqref{l1 norms of beta' seg}. \qedhere
\end{enumerate}
\end{proof}

\textbf{\hypertarget{GQj}{Construction of $\phi_k$ when $k\geq 2$}}: Let $\Delta=(\gamma_0,...,\gamma_k)\in C_k(Y;\RR)$. For simplicity, we write $p_j=\gamma_jx_0$ for any $0\leq j\leq k$. Denoted by {$V(\Delta)=\{p_0,...,p_k\}\subset\Gamma x_0$.} If $\gamma_i=\gamma_j$ for some $0\leq i<j\leq k$, then we simply define $\phi_k(\Delta)=0$. (See also \eqref{phi full construction}.) From now on, unless otherwise mentioned, we assume that $\gamma_0,...,\gamma_k$ are distinct, which implies that $p_0,...,p_k$ are also distinct.

We denote by $\Delta_j=(\gamma_0,...,\widehat{\gamma_j},...,\gamma_k)\in C_{k-1}(Y;\RR)$ and {hence $V(\Delta_j)=V\setminus\{\gamma_jx_0\}=V\setminus\{p_j\}$} for any $0\leq j\leq k$. {In particular, for any permutation $\tau:\{0,...,k\}\to\{0,...,k\}$ and $\Delta'=(\gamma_{\tau(0)},...,\gamma_{\tau(k)})$, we have $V(\Delta')=V(\Delta)$ and $V(\Delta'_j)=V\setminus\{p_{\tau(j)}\}=V(\Delta_{\tau(j)})$.}


For any $0\leq j\leq k$ and $G\in\MCS(\bG_{{V(\Delta)}})$ such that $G\ints\bG_{{V(\Delta);V(\Delta_j)}}\neq\emptyset$, we consider the following 2 cases:

\textbf{\hypertarget{GQj def}{Case 1}, {the definition of $G_{Q;j;\Delta}$}}: If ${\gamma_jx_0=p_j}\not\in\Sing(G)$ {and $Q\in\cV(G)\ints\cV(\bG_{V(\Delta);V(\Delta_j)})$}, since we assumed that $G\ints\bG_{{V(\Delta);V(\Delta_j)}}\neq\emptyset$, by Definition \ref{singular} and the seventh assertion in Lemma \ref{properties of W-face}, there exist distinct $x,y\in {V(\Delta_j)}$ such that
$$\left|\res_{{V(\Delta_j)}}^{{V(\Delta)}}(\cV(G\ints\bG_{{V(\Delta);V(\Delta_j)}}))\ints\cV(\bP_{x,y,{V(\Delta_j)}})\right|=2.$$
Hence by Definition \ref{subordinate}, $\subord_{{V(\Delta_j)}}^{{V(\Delta)}}(G\ints\bG_{{V(\Delta_j)}})\neq\emptyset$. Therefore, the first assertion in Lemma \ref{easy properties of enrich} implies that $\Enrich_{{V(\Delta_j)}}^{{V(\Delta)}}(G\ints\bG_{{V(\Delta_j)}})$ is a connected subgraph which contains at least one edge. By the second assertion in Corollary \ref{remaining properties of sep graph}, Definition \ref{subordinate} and the second assertion in Lemma \ref{easy properties of enrich}, there exist a unique MCS in $\bG_{{V(\Delta_j)}}$, denoted as $G_{Q;j;{\Delta}}$, such that $G_{Q;j;{\Delta}}\in\subord_{{V(\Delta_j)}}^{{V(\Delta)}}(G\ints\bG_{{V(\Delta_j)}})$ and $\res_{{V(\Delta_j)}}^{{V(\Delta)}}(Q)\in\cV(G_{Q;j;{\Delta}})$. By the second remark after Definition \ref{res of types} and the remark after Definition \ref{subordinate}, one can easily check that for any $\gamma\in\Gamma$, we have $\gamma \cdot (G_{Q;j;\Delta})=(\gamma G)_{\gamma Q;j;\gamma\Delta}$.

\textbf{Case 2, {the definition of $G_{Q;j;\Delta}^{\mathrm{op}}$}}: If ${\gamma_jx_0=p_j}\in\Sing(G)$ and $Q\in\cV(G){\ints\cV(\bG_{V(\Delta);V(\Delta_j)})}\setminus\del\cA_\Sep(V)$, by Lemma \ref{singular case}, $\res_{{V(\Delta_j)}}^{{V(\Delta)}}(\cV(G\ints\bG_{{V(\Delta_j)}}))$ contains a single element. Since we assume that $Q\not\in\del\cA_\Sep(V)$, by the second assertion in Corollary \ref{remaining properties of sep graph}, there exists a unique $G'\in\MCS(\bG_V)$ such that $G'\neq G$ and $Q\in\cV(G')\ints\cV(G)$. Notice that for any $Q'\in\cV(G)$, $F_{Q'}=F_{Q}$ due to Lemma \ref{singular case}. (See \eqref{eqn:flat part}.) By the first assertion in Lemma \ref{reinterpretation of edges}, for any distinct $Q',Q''\in\cV(G')$, we have $F_{Q'}\neq F_{Q''}$. In particular, ${p_j}\not\in\Sing(G')$ due to Definition \ref{singular}. Therefore we can define $G^{\mathrm{op}}_{Q;j;{\Delta}}:=G'_{Q;j;{\Delta}}$. In other words, $G^{\mathrm{op}}_{Q;j;{\Delta}}$ is the unique MCS in $\bG_{{V(\Delta_j)}}$ such that $G^{\mathrm{op}}_{Q;j;{\Delta}}\in\subord_{{V(\Delta_j)}}^{{V(\Delta)}}(G'\ints\bG_{{V(\Delta_j)}})$ and $\res_{{V(\Delta_j)}}^{{V(\Delta)}}(Q)\in\cV(G^{\mathrm{op}}_{Q;j;{\Delta}})$. By the second remark after Definition \ref{res of types} and the remark after Definition \ref{subordinate}, one can easily check that for any $\gamma\in\Gamma$, we have $\gamma \cdot (G^{\mathrm{op}}_{Q;j;\Delta})=(\gamma G)^{\mathrm{op}}_{\gamma Q;j;\gamma\Delta}$.

{\textbf{Permutations of $\gamma_0,...,\gamma_k$ and their relations with $G_{Q;j;\Delta}$ and $G_{Q;j;\Delta}^{\mathrm{op}}$}: For any permutation $\tau:\{0,...,k\}\to\{0,...,k\}$ with $\Delta'=(\gamma_{\tau(0)},...,\gamma_{\tau(k)})$, and any $G\in\MCS(\bG_{V(\Delta)})=\MCS(\bG_{V(\Delta')})$, if $\gamma_{\tau(j)}x_0\not\in\Sing(G)$ and $G\ints\bG_{V(\Delta);V(\Delta'_j)}= G\ints\bG_{V(\Delta);V(\Delta_{\tau(j)})}\neq\emptyset$, then for any $Q\in\cV(G)\ints\cV(\bG_{V(\Delta);V(\Delta'_j)})$, $G_{Q;j;\Delta'}=G_{Q;\tau(j);\Delta}$. Similarly, if $\gamma_{\tau(j)}x_0\in\Sing(G)$ and $G\ints\bG_{V(\Delta);V(\Delta'_j)}= G\ints\bG_{V(\Delta);V(\Delta_{\tau(j)})}\neq\emptyset$, then for any $Q\in\cV(G)\ints\cV(\bG_{V(\Delta);V(\Delta'_j)})$, $G_{Q;j;\Delta'}^{\mathrm{op}}=G_{Q;\tau(j);\Delta}^{\mathrm{op}}.$}

We first construct $\Chamber$, $\BChamber$, $\Floor$ and $\BFloor$ in an inductive way when $k\geq 2$: For any $G\in\MCS(\bG_{{V(\Delta)}})$ and any $Q\in\cV(G)$, we let $\BFloor_{\Delta,G}(Q)=0$ if $Q\in\del\cA_\Sep({{V(\Delta)}})$. Otherwise, we let
\begin{align}\label{bfloor}
\begin{split}
\BFloor_{\Delta,G}(Q):=&\sum_{\substack{j:0\leq j\leq k\\\cV(\bG_{{V(\Delta);V(\Delta_j)}})\ni Q\\{\gamma_jx_0}\not\in\Sing(G)}}(-1)^j\Floor_{\Delta_j, G_{Q;j;{\Delta}}}(\res_{{V(\Delta_j)}}^{{V(\Delta)}}(Q)) \\
&-\sum_{\substack{j:0\leq j\leq k\\ \cV(\bG_{{V(\Delta);V(\Delta_j)}})\ni Q \\{\gamma_jx_0}\in\Sing(G)}}(-1)^j\Floor_{\Delta_j,G_{Q;j;{\Delta}}^{\mathrm{op}}}(\res_{{V(\Delta_j)}}^{{V(\Delta)}}(Q))\in C_{k-2,x_0;\mathrm{bar}}(X,\RR).
\end{split}
\end{align}
(See Figure \ref{fig:floor}.) Then we define
\begin{align}\label{floor}
\Floor_{\Delta, G}(Q):={\FBCone}(\BFloor_{\Delta,G}(Q))\in C_{k-1,x_0;\mathrm{bar}}(X,\RR).
\end{align}
In particular, if $Q\in\del\cA_\Sep({{V(\Delta)}})$, by {Definition \ref{FBCone}} and \eqref{bfloor}, we have $\Floor_{\Delta, G}(Q)=0$.
\begin{figure}[h]
	\centering
	\includegraphics[scale=1]{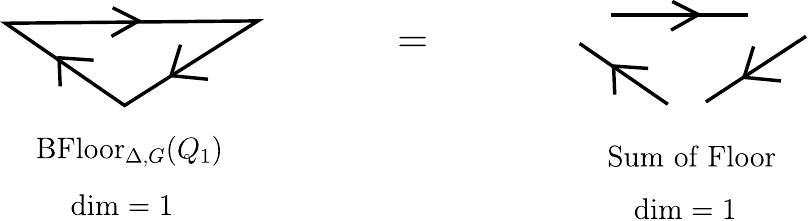}
	\caption{\label{fig:floor}}
\end{figure}

For any $G\in\MCS(\bG_{{V(\Delta)}})$, we define $\BChamber_\Delta(G)\in C_{k-1,x_0;\mathrm{bar}}(X,\RR)$ such that
\begin{align}\label{BChamber}
\begin{split}
&\BChamber_\Delta(G) \\
:=&(-1)^{k-1}\sum_{Q:Q\in\cV(G)}\Floor_{\Delta,G}(Q)+\sum_{\substack{j, G_j:0\leq j\leq k\\ G\ints\bG_{{V(\Delta);V(\Delta_j)}}\neq\emptyset \\ G_j\in\subord_{{V(\Delta_j)}}^{{V(\Delta)}}(G\ints\bG_{{V(\Delta);V(\Delta_j)}})}}(-1)^j\Chamber_{\Delta_j}(G_j)
\end{split}
\end{align}
and correspondingly
\begin{align}\label{Chamber}
\Chamber_\Delta(G):={\FBCone}(\BChamber_\Delta(G))\in C_{k,x_0;\mathrm{bar}}(X,\RR),
\end{align}
where we made use of the fact that $G\ints\bG_{{V(\Delta);V(\Delta_j)}}\neq\emptyset$ if and only if $G\ints\bG_{{V(\Delta);V(\Delta_j)}}\in\MCS(\bG_{{V(\Delta);V(\Delta_j)}})$ from the fourth assertion in Lemma \ref{properties of W-face}. (See Figure \ref{fig:chamber}.)
\begin{figure}[h]
	\centering
	\includegraphics[scale=0.9]{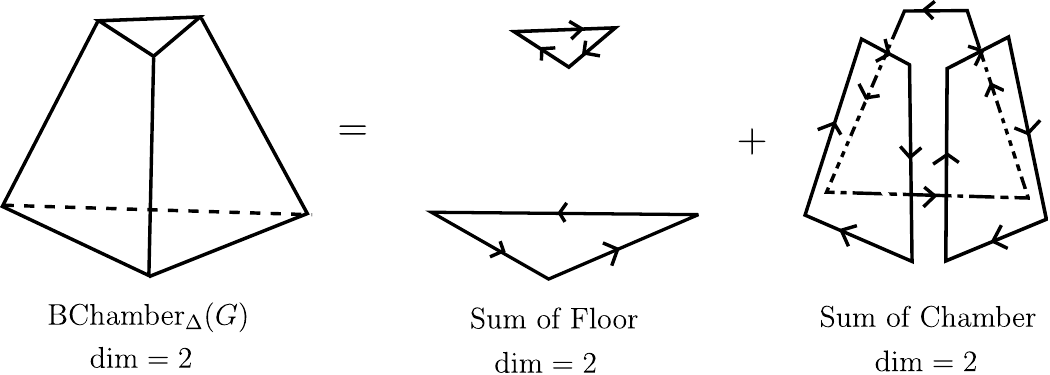}
	\caption{\label{fig:chamber}}
\end{figure}

Similar to the first assertion in Proposition \ref{baby constructions}, for any $\Delta=(\gamma_0,...,\gamma_k)$ (with $\gamma_0,...,\gamma_k$ not necessarily distinct), we let
\begin{align}\label{phi full construction}
\phi_k(\Delta):=\begin{cases}
\displaystyle \sum_{G\in\MCS(\bG_V)}\Chamber_\Delta(G), ~&\text{if }\gamma_0,...,\gamma_k\text{ are distinct,}\\
\displaystyle 0, &\text{otherwise}.
\end{cases}
\end{align}
The fact that $\phi_\bullet$ constructed above is a chain map and other crucial properties of $\phi_\bullet$ can be summarized in the following mega-proposition.

\begin{proposition}[Properties of $\Chamber$, $\BChamber$, $\Floor$ and $\BFloor$ when $k\geq 2$]\label{full constructions}
{Fix any $\Delta=(\gamma_0,...,\gamma_k)$ with distinct $\gamma_0,...,\gamma_k\in\Gamma$ and $p_j:=\gamma_jx_0$, for simplicity we write $V:=V(\Delta)$ and $V_j:=V(\Delta_j)$. For any $G\in\MCS(\bG_V)$ and any $Q\in\cV(G)$, we also write for simplicity that $G_{Q;j}:=G_{Q;j;\Delta}$ and $G_{Q;j}^{\mathrm{op}}:=G_{Q;j;\Delta}^{\mathrm{op}}$ whenever they are well defined.} {(See the \hyperlink{GQj}{definitions} of $G_{Q;j;\Delta}$ and $G_{Q;j;\Delta}^{\mathrm{op}}$.)} Then the following holds when $k\geq 2$.
\begin{enumerate}
\item[(1).] {For any distinct $G, G'\in\MCS(\bG_V)$ such that $\cV(G)\ints\cV(G')\neq \emptyset$, let $Q$ be the unique element in $\cV(G)\ints\cV(G')$ (guaranteed by Proposition \ref{key prop of ASep graph}.) Then we have
$$\BFloor_{\Delta, G}(Q)+\BFloor_{\Delta,G'}(Q)=0\mathrm{~and~hence~}\Floor_{\Delta, G}(Q)+\Floor_{\Delta,G'}(Q)=0.$$
In view of the second assertion in Corollary \ref{remaining properties of sep graph} and the fact that $\BFloor_{\Delta,G}(Q)=0$ whenever $Q\in\del\cA_\Sep(V)$ (introduced right before \eqref{bfloor}), this is equivalent to the following statement: for any $Q\in\cV(\bG_V)$, we have
$$\sum_{\substack{G:G\in\MCS(\bG_V)\\\cV(G)\ni Q}}\BFloor_{\Delta, G}(Q)=0\mathrm{~and~}\sum_{\substack{G:G\in\MCS(\bG_V)\\\cV(G)\ni Q}}\Floor_{\Delta, G}(Q)=0.$$
(Compare this with the second assertion in Proposition \ref{baby constructions}.)}
\item[(2).] For any $G\in\MCS(\bG_V)$ and any $Q\in\cV(G)$, $\BFloor_{\Delta,G}(Q)$ is a closed singular $(k-2)$-chain. As a direct corollary of this and the sixth remark after Definition \ref{FBCone}, we have $\del^X_{k-1}\Floor_{\Delta,G}(Q)=\BFloor_{\Delta,G}(Q)$. (Compare this with \eqref{-1D bfloor}.)
\item[(3).] For any $G\in\MCS(\bG_V)$ and any $Q\in\cV(G)$, if we assume in addition that $\Sing(G)=\emptyset$ and $Q=\typemark{F_{p}}{\{p\}}{V}$ for some $p\in V$, then
\begin{align}\label{(3,2)-1}
\sum_{\substack{j:0\leq j\leq k\\ \cV(\bG_{V;V_j})\ni Q}}(-1)^j\Floor_{\Delta_j,G_{Q;j}}(\res_{V_j}^V(Q))=0.
\end{align}
Moreover, for any $G\in\MCS(\bG_V)$, we have
\begin{align}\label{(3,2)-2}
\begin{split}
&\sum_{Q':Q'\in\cV(G)\ints\del\cA_\Sep(V)}\sum_{\substack{j:0\leq j\leq k\\ \cV(\bG_{V;V_j})\ni Q'\\ {p_j}\not\in\Sing(G)}}(-1)^j\Floor_{\Delta_j,G_{Q';j}}(\res_{V_j}^V(Q'))\\
=&\begin{cases}
\displaystyle (-1)^{j_0}\sum_{\substack{j:0\leq j\leq 2\\ j\neq j_0}}\Floor_{\Delta_{j_0},\bG_{V_{j_0}}}(p_j),~&\begin{aligned}
&\mathrm{if~}|V|=3\Leftrightarrow k=2,\Sing(G)\neq\emptyset\mathrm{~and~}\\
&|\cV(G)\ints\del\cA_\Sep(V)|=2,
\end{aligned}\\
\displaystyle 0,~&\mathrm{otherwise.}
\end{cases}
\end{split}
\end{align}
In the first case of \eqref{(3,2)-2}, $j_0\in\{0,1,2\}$ is the unique element such that $\typemark{F_{j_0}}{\{p_{j_0}\}}{V}\not\in\cV(G)$. (Compare this with \eqref{bfloor} when $Q\not\in\del\cA_\Sep(V)$.)
\item[(4).] For any $G\in\MCS(\bG_V)$ and any $Q\in\cV(G)$, we have
$$\vertsupp(\BFloor_{\Delta,G}(Q))\subset\{x\in\Gamma x_0|F_x=F_Q\}.$$
(See \eqref{eqn:flat part} for the definition of $F_Q$.) As a consequence of this,
$$\vertsupp(\Floor_{\Delta,G}(Q))\subset\{x\in\Gamma x_0|F_x=F_Q\}.$$
(Compare this with the third assertion in Proposition \ref{baby constructions}.)
\item[(5).] For any $G\in\MCS(\bG_V)$ and any $Q\in\cV(G)$, $\BFloor_{\Delta,G}(Q)$ is alternating with respect to $\Delta$. Namely, for any permutation $\tau:\{0,...,k\}\to\{0,...,k\}$ and $\Delta'=(\gamma_{\tau(0)},...,\gamma_{\tau(k)})$, we have
$$\BFloor_{\Delta,G}(Q)=(-1)^{\mathrm{sgn}(\tau)}\BFloor_{\Delta',G}(Q).$$
As a direct consequence of this and the third remark after Definition \ref{FBCone}, we have
$$\Floor_{\Delta,G}(Q)=(-1)^{\mathrm{sgn}(\tau)}\Floor_{\Delta',G}(Q);$$
(Compare this with the fourth assertion in Proposition \ref{baby constructions}.)
\item[(6).] For any $G\in\MCS(\bG_V)$, $\BChamber_\Delta(G)$ is a closed singular $(k-1)$-chain. As a direct corollary of this and the sixth remark after Definition \ref{FBCone}, we have $\del^X_{k}\Chamber_\Delta(G)=\BChamber_\Delta(G)$. Moreover, by the first assertion in Proposition \ref{baby constructions} and \eqref{phi full construction}, we have
\begin{align}\label{(6,k)}
\begin{split}
\del^X_{k}\phi_k(\Delta)=&\sum_{G:G\in\MCS(\bG_V)}
\BChamber_\Delta(G)\\
=&\sum_{j=0}^k(-1)^j\sum_{G_j:G_j\in\MCS(\bG_{V_j})}\Chamber_{\Delta_j}(G_j)=\sum_{j=0}^k(-1)^j\phi_{k-1}(\Delta_j)=\phi_{k-1}(\del^Y_{k}(\Delta)).
\end{split}
\end{align}
(Compare this with \eqref{0D bchamber} and the eighth assertion in Proposition \ref{baby constructions}.)
\item[(7).] There exists a constant $\cC_6(k):=\cC_6(k,\epsilon_0,\rho_0)$ such that
$$\sum_{\substack{Q,G:Q\in\cV(G)\\G\in\MCS(\bG_V)}}|\BFloor_{\Delta,G}(Q)|_{l^1}\leq\cC_6(k)\mathrm{~and~}\sum_{\substack{Q,G:Q\in\cV(G)\\G\in\MCS(\bG_V)}}|\Floor_{\Delta,G}(Q)|_{l^1}\leq \cC_6(k).$$
\item[(8).] For any $G\in\MCS(\bG_V)$, $\BChamber_\Delta(G)$ is alternating with respect to $\Delta$. Namely, for any permutation $\tau:\{0,...,k\}\to\{0,...,k\}$ and $\Delta'=(\gamma_{\tau(0)},...,\gamma_{\tau(k)})$, we have
$$\BChamber_\Delta(G)=(-1)^{\mathrm{sgn}(\tau)}\BChamber_{\Delta'}(G).$$
As a direct consequence of this and the third remark after Definition \ref{FBCone}, we have
$$\Chamber_\Delta(G)=(-1)^{\mathrm{sgn}(\tau)}\Chamber_{\Delta'}(G).$$
(Compare this with the fifth assertion in Proposition \ref{baby constructions}.)
\item[(9).] For any $G\in\MCS(\bG_V)$, let $\cF_V(G)$ be the same as in the sixth assertion in Proposition \ref{baby constructions}. Then
$$\vertsupp(\BChamber_\Delta(G)),\vertsupp(\Chamber_\Delta(G))\subset\{x\in\Gamma x_0|F_x\in\cF_V(G)\}.$$
Moreover, one can choose $\cC_7(k)>0$ only depending on $k$ such that $|\cF_V(G)|\leq \cC_7(k)$. (Compare this with the sixth assertion in Proposition \ref{baby constructions}.)
\item[(10).] There exist a constant $\cC_8(k):=\cC_8(k,\epsilon_0,\rho_0)$ such that
$$\sum_{G:G\in\MCS(\bG_V)}|\BChamber_\Delta(G)|_{l^1},\sum_{G:G\in\MCS(\bG_V)}|\Chamber_\Delta(G)|_{l^1}\leq\cC_8(k).$$
\item[(11).] For any $\gamma\in\Gamma$, any $G\in\MCS(\bG_V)$ and any $Q\in\cV(G)$, the following holds.
\begin{itemize}
\item $\gamma\BFloor_{\Delta,G}(Q)=\BFloor_{\gamma\Delta,\gamma G}(\gamma Q)$.
\item $\gamma\Floor_{\Delta,G}(Q)=\Floor_{\gamma\Delta,\gamma G}(\gamma Q)$.
\item $\gamma\BChamber_{\Delta}(G)=\BChamber_{\gamma\Delta}(\gamma G)$.
\item $\gamma\Chamber_{\Delta}(G)=\Chamber_{\gamma\Delta}(\gamma G)$.
\end{itemize}
(Compare this with the seventh assertion in Proposition \ref{baby constructions}.)
\end{enumerate}
\end{proposition}
We leave the inductive proof of Proposition \ref{full constructions} to Subsections \ref{subsec 2D case} and \ref{subsec kD case}. We finish this subsection by stating a corollary of Proposition \ref{full constructions} regarding the properties of $\phi_\bullet$. But before that, we first introduce the notion of special geodesic simplices.
\begin{definition}[Special {barycentric} simplices]\label{special bar simplices}
For any $m\geq 3$, $k\geq 0$ and $p_0,...,p_k\in X$, we say that $\Db^k(p_0,...,p_k)$ is a \emph{special {barycentric} simplex with depth at most $m$} if the following holds:
\begin{enumerate}
\item[(1)] $p_0,...,p_k\in \Gamma x_0$.
\item[(2).] $|\cF(\{p_0,...,p_k\})|\leq m$.
\end{enumerate}
\end{definition}
\begin{corollary}\label{properties of straightening}
$\phi_\bullet:C^Y_\bullet\to C^X_\bullet$ is a chain homomorphism of $\RR[\Gamma]$-modules. Moreover, it satisfies the following properties.
\begin{enumerate}
\item[(1).] For any $k\geq 1$, $\phi_k$ is alternating. Namely, for any $\gamma_0,...,\gamma_k\in\Gamma$ and any permutation $\tau:\{0,...,k\}\to\{0,...,k\}$, we have $\phi_k(\gamma_0,..,\gamma_k)=(-1)^{\mathrm{sgn}(\tau)}\phi_k(\gamma_{\tau(0)}..,\gamma_{\tau(k)})$.
\item[(2).] For any $k\geq 2$ and $\gamma_0,...,\gamma_k\in\Gamma$, $|\phi_k(\gamma_0,...,\gamma_k)|_{l^1}\leq \cC_8(k)$ (introduced in the tenth assertion in Proposition \ref{full constructions}). Moreover, there exist a constant $\cC_9(k)$ which only depends on $k$ such that $\phi_k(\gamma_0,...,\gamma_k)$ is a linear combination of special {barycentric} simplices of depth at most $\cC_9(k)$ in the sense of Definition \ref{special bar simplices}.
\end{enumerate}
\end{corollary}
\begin{proof}
If $\gamma_0,...,\gamma_k\in\Gamma$ are not pairwise distinct, by \eqref{phi full construction}, the fifth assertion in Proposition \ref{baby constructions} and the eighth assertion in Proposition \ref{full constructions}, one can easily check that $0=\phi_{k-1}(\del_k^Y(\gamma_0,...,\gamma_k))=\del_{k}^X(\phi_k(\gamma_0,...,\gamma_k))$. Hence, the fact that $\phi_\cdot$ is a chain homomorphism of $\RR[\Gamma]$-modules follows directly from the construction of $\phi_1$, \eqref{phi full construction}, the sixth and the eleventh assertions in Proposition \ref{full constructions}. The first property follows directly from the construction of $\phi_1$, \eqref{phi full construction} and the eighth assertion in Proposition \ref{full constructions}. It remains to verify the second property. Since $\phi_k(\gamma_0,...,\gamma_k)=0$ whenever $\gamma_i=\gamma_j$ for some $0\leq i<j\leq k$, we assume WLOG that $\gamma_0,...,\gamma_k$ are distinct.

We first proof a lemma. Its proof relies on the application of a ``coarse triangular inequality'': Lemma \ref{properties of Theta}, \hyperlink{Theta-4}{property ($\Theta$4) of $\Theta(\cdot,\cdot)$}.
\begin{lemma}\label{coarse trig ineq}
For any $k\geq 2$ and any $F_1,...,F_k\in \Gamma F$, let $\cF:=\union_{1\leq i<j\leq k}\Theta(F_i,F_j)$. If there exists some $m\geq1  $ such that $|\Theta(F_i,F_j)|\leq m$ for any $1\leq i<j\leq k$, then for any $F',F''\in\cF$, $\Theta(F',F'')\leq 3m+6.$
\end{lemma}
\begin{proof}[Proof of Lemma \ref{coarse trig ineq}]
Assume that $F'\in\Theta(F_{i'},F_{j'})$ and $F''\in\Theta(F_{i''},F_{j''})$ for some $1\leq i',i'',j',j''\leq k$. Then by Lemma \ref{properties of Theta}, \hyperlink{Theta-3}{property ($\Theta$3) of $\Theta(\cdot,\cdot)$}, $|\Theta(F',F_{i'})|, |\Theta(F'',F_{i''})|\leq m$. Apply Lemma \ref{properties of Theta}, \hyperlink{Theta-4}{property ($\Theta$4) of $\Theta(\cdot,\cdot)$}, we have
$$
\left.
\begin{aligned}
\left.
\begin{aligned}
&|\Theta(F',F_{i'})|\leq m \\
&|\Theta(F_{i'},F_{i''})|\leq m
\end{aligned}\right\}\implies &|\Theta(F',F_{i''})|\leq 2m+3\\
&|\Theta(F'',F_{i''})|\leq m
\end{aligned}
\right\}\implies \Theta(F',F'')\leq 3m+6. \qedhere
$$
\end{proof}
Back to the proof of the second property in Corollary \ref{properties of straightening}. $|\phi_k(\gamma_0,...,\gamma_k)|_{l^1}\leq \cC_8(k)$ follows directly from the tenth assertion in Proposition \ref{full constructions} and \eqref{phi full construction}. Thanks to \eqref{phi full construction}, it suffices to show that for any $\Delta:=(\gamma_0,...,\gamma_k)$ with $V:=\{p_0:=\gamma_0x_0,...,p_k:=\gamma_kx_0\}$, and any $G\in\MCS(\bG_V)$, $\Chamber_\Delta(G)$ is a linear combination of special {barycentric} simplices of depth at most $\cC_9(k)$ in the sense of Definition \ref{special bar simplices}.

Let $\cV(G)=\{Q_1,...,Q_l\}$ with $Q_1,...,Q_l$ distinct. By the ninth assertion in Lemma \ref{full constructions},
we have $|\Theta(F_{Q_i},F_{Q_j})|\leq \cC_7(k)$ for any $1\leq i,j\leq k.$
Choose $\cC_9(k)=3\cC_7(k)+6$ and the rest of the corollary follows directly from the ninth assertion in Proposition \ref{full constructions} and Lemma \ref{coarse trig ineq}.
\end{proof}

\subsection{Proof of Proposition \ref{full constructions} when $k=2$}\label{subsec 2D case}
Before we present the proof of Proposition \ref{full constructions}, we first explain the structure of the proof: We define a lexicographic ordering on $\{(k,j):k\geq 2, 1\leq j\leq {11}\}$. For any $(k_1,j_1),(k_2,j_2)\in\{(k,j):k\geq 2, 1\leq j\leq {11}\}$, we say that $(k_1,j_1)<(k_2,j_2)$ if $k_1< k_2$ or $k_1=k_2$ and simultaneouly $j_1< j_2$. \textbf{We will always prove the $j_1$-th assertion in Proposition \ref{full constructions} for the case $k=k_1$ before the proof of the $j_2$-th assertion in Proposition \ref{full constructions} for the case $k=k_2$ whenever $(k_1,j_1)<(k_2,j_2)$. The proof of the $j_2$-th assertion in Proposition \ref{full constructions} for the case $k=k_2$ may rely on the $j_1$-th assertion in Proposition \ref{full constructions} for the case $k=k_1$ for any $(k_1,j_1)<(k_2,j_2)$.}

The proof of some assertions in the case of $k=2$ works for general $k$. We will indicate them in the proof in order to avoid repetitions. As mentioned above, the inductive proof for the $j$-th assertion in Proposition \ref{full constructions} in the case $k$ assumes that for any $(k',j')<(k,j)$, the $j'$-th assertion in Proposition \ref{full constructions} in the case $k'$ is proved.

\begin{proof}[Proof of Proposition \ref{full constructions} when $k=2$]
\begin{enumerate}
\item[(1).] {\textbf{The proof of this assertion works for any $k\geq 2$, assuming that for any $(k',j')<(k,1)$, the $j'$-th assertion in Proposition \ref{full constructions} for the case $k'$ holds. (The arguments in the proof of this assertion still work under weaker assumptions.)}

First, we can assume WLOG that $\Sing(G')=\emptyset$. This is because by Lemma \ref{singular case}, if both $\Sing(G)$ and $\Sing(G')$ are non-empty, then for any $Q',Q''\in\cV(G)\union\cV(G')$, $F_{Q'}=F_{Q''}=F_Q$. (See \eqref{eqn:flat part} for the definition of $F_Q$.) By the first assertion in Lemma \ref{reinterpretation of edges}, we have $G=G'$, contradicting the assumption that $G\neq G'$.

If ${p_j}\in V\setminus\Sing(G)$ and $\cV(\bG_{V;V_j})\ni Q$, then by Definition \ref{subordinate}, we have $G_{Q;j}\neq G'_{Q;j}$ (due to the fact that $G_{Q;j}$ and $G'_{Q;j}$ are subordinate to different MCS in $\bG_{V;V_j}$) and $\res_{V_j}^V(Q)\in\cV(G_{Q;j})\ints\cV(G'_{Q;j})$. (See \hyperlink{GQj def}{the definition of $G_{Q;j}$}.) Therefore by the second assertion in Proposition \ref{baby constructions} or the first assertion in Proposition \ref{full constructions} for the case $k-1$, we have
\begin{align}\label{(1,k-1)-1}
\Floor_{\Delta_j,G_{Q;j}}(\res_{V_j}^V(Q))+\Floor_{\Delta_j,G'_{Q;j}}(\res_{V_j}^V(Q))=0.
\end{align}

If ${p_j}\in\Sing(G)$ and $\cV(\bG_{V;V_j})\ni Q$, by the \hyperlink{GQj def}{definition of $G_{Q;j}$ and $G_{Q;j}^{\mathrm{op}}$}, we have $G_{Q;j}^{\mathrm{op}}=G'_{Q;j}$ and hence
\begin{align}\label{(1,k-1)-2}
\Floor_{\Delta_j,G^{\mathrm{op}}_{Q;j}}(\res_{V_j}^V(Q))=\Floor_{\Delta_j,G'_{Q;j}}(\res_{V_j}^V(Q)).
\end{align}
Apply \eqref{(1,k-1)-1} and \eqref{(1,k-1)-2} to \eqref{bfloor}, we have
\begin{align*}
&\BFloor_{\Delta,G}(Q)\\
=&\sum_{\substack{j:0\leq j\leq k\\\cV(\bG_{V;V_j})\ni Q\\{p_j}\not\in\Sing(G)}}(-1)^j\Floor_{\Delta_j, G_{Q;j}}(\res_{V_j}^V(Q))-\sum_{\substack{j:0\leq j\leq k\\ \cV(\bG_{V;V_j})\ni Q \\{p_j}\in\Sing(G)}}(-1)^j\Floor_{\Delta_j,G_{Q;j}^{\mathrm{op}}}(\res_{V_j}^V(Q)) \\
=&-\sum_{\substack{j:0\leq j\leq k\\\cV(\bG_{V;V_j})\ni Q\\{p_j}\not\in\Sing(G)}}(-1)^j\Floor_{\Delta_j, G'_{Q;j}}(\res_{V_j}^V(Q))-\sum_{\substack{j:0\leq j\leq k\\ \cV(\bG_{V;V_j})\ni Q \\{p_j}\in\Sing(G)}}(-1)^j\Floor_{\Delta_j,G'_{Q;j}}(\res_{V_j}^V(Q)) \\
=&-\sum_{\substack{j:0\leq j\leq k\\\cV(\bG_{V;V_j})\ni Q}}(-1)^j\Floor_{\Delta_j, G'_{Q;j}}(\res_{V_j}^V(Q))=-\BFloor_{\Delta,G'}(Q).
\end{align*}
The fact that $\Floor_{\Delta,G}(Q)=-\Floor_{\Delta,G'}(Q)$ follows directly from \eqref{floor} and the third remark after {Definition \ref{FBCone}}.}
\item[(2).] {By the definition of $\BFloor$ in \eqref{bfloor}, it suffices to verify the case when $Q\not\in\del\cA_\Sep(V)$. We split the proof of this assertion into 2 major cases, with various minor cases. The discussions in Example \ref{ex:triangle} are heavily used here. (See the items under \textbf{Useful claims}.)

\textbf{Case 1:} Assume that for any distinct $\typemark{F_1}{I_1}{V},\typemark{F_2}{I_2}{V}\in\cV(G)$, we have $F_1\neq F_2$. In this case, by Lemma \ref{singular case}, we have $\Sing(G)=\emptyset$.

We need to consider 6 minor cases. \textbf{Case 1a}, \textbf{Case 1c} and \textbf{Case 1e} are similar; \textbf{Case 1b}, \textbf{Case 1d} and \textbf{Case 1f} are similar.
\begin{figure}[h]
	\centering
	\includegraphics[scale=0.6]{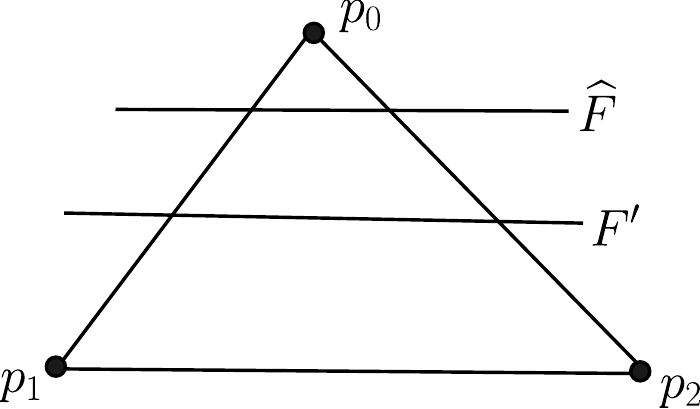}
	\caption{\label{fig:(2,2)-1}}
\end{figure}

\textbf{Case 1a}: Assume that $Q=\typemark{\hF}{\{p_0\}}{V}$ and there exists some $\typemark{F'}{I'}{V}\in\cV(G)$ such that $\hF\in\Theta(F_{p_0},F')\setminus\{F'\}$. (See Figure \ref{fig:(2,2)-1}.) By the fact that $Q\not\in\del\cA_\Sep(V)$, we have $\hF\neq F_{p_0}$ and hence $|\cF(V_1)|,|\cF(V_2)|\geq 2$. By the discussions in Example \ref{ex:triangle} (item (1) in \textbf{Useful claims}), for any $\typemark{F''}{I''}{V}\in\cV(G)$, we have $\hF\in\Theta(F_{p_0},F'')$. In view of the comments after \eqref{0D floor-2}, the following holds:


\begin{itemize}
\item $\del_0^X\Floor_{\Delta_1,G_{Q;1}}((\hF,\{\{p_0\},\{p_2\}\}))=-1$. This follows from \eqref{beta' seg endpts norm}, \eqref{0D floor-1} and the fact that $|\cF(V_1)|\geq 2$.
\item $\del_0^X\Floor_{\Delta_2,G_{Q;2}}((\hF,\{\{p_0\},\{p_1\}\})) =-1$. This follows from \eqref{beta' seg endpts norm}, \eqref{0D floor-1} and the fact that $|\cF(V_2)|\geq 2$.
\end{itemize}
By \eqref{bfloor} and the above, we have
\begin{align*}
\BFloor_{\Delta, G}(Q)=&\sum_{\substack{j:0\leq j\leq 2\\ \cV(\bG_{V;V_j})\ni Q}}(-1)^j\Floor_{\Delta_j,G_{Q;j}}(\res_{V_j}^V(Q)) \\
=&-\Floor_{\Delta_1,G_{Q;1}}((\hF,\{\{p_0\},\{p_2\}\}))+\Floor_{\Delta_2,G_{Q;2}}((\hF,\{\{p_0\},\{p_1\}\}))
\end{align*}
and therefore
\begin{align}\label{1a}
\del^X_0(\BFloor_{\Delta, G}(Q))=-(-1)+(-1)=0.
\end{align}
\begin{figure}[h]
	\centering
	\includegraphics[scale=0.6]{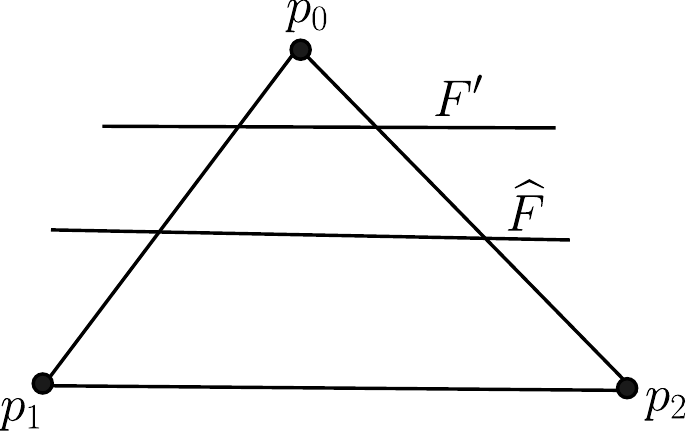}
	\caption{\label{fig:(2,2)-2}}
\end{figure}

\textbf{Case 1b}: Assume that $Q=\typemark{\hF}{\{p_0\}}{V}$ and there exists some $\typemark{F'}{I'}{V}\in\cV(G)$ such that $F'\in\Theta(F_{p_0},\hF)\setminus\{\hF\}$. (See Figure \ref{fig:(2,2)-2}.) By the fact that $Q\not\in\del\cA_\Sep(V)$ and the discussions in Example \ref{ex:triangle} (item (1) in \textbf{Useful claims}), there exists some $G'\in\MCS(\bG_V)\setminus\{G\}$ such that $Q\in\cV(G')$. Moreover, $G'$ and $Q$ satisfy \textbf{Case 1a}. Hence by \eqref{1a} and the first assertion in Proposition \ref{full constructions} when $k=2$, we have
\begin{align}\label{1b}
\del^X_0(\BFloor_{\Delta, G}(Q))=-\del^X_0(\BFloor_{\Delta,G'}(Q))=0.
\end{align}
\begin{figure}[h]
	\centering
	\includegraphics[scale=0.6]{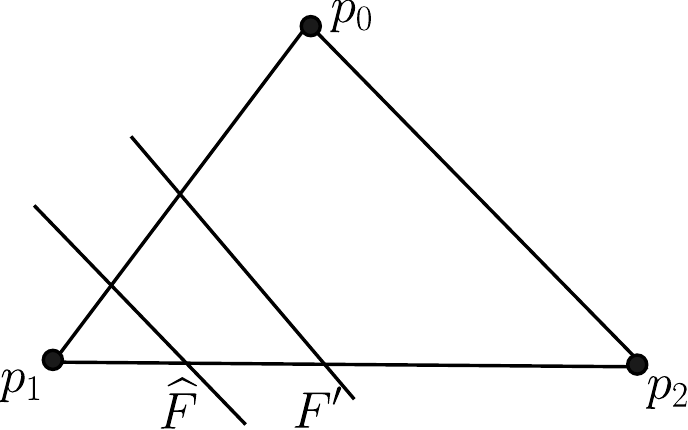}
	\caption{\label{fig:(2,2)-3}}
\end{figure}

\textbf{Case 1c}: Assume that $Q=\typemark{\hF}{\{p_1\}}{V}$ and there exists some $\typemark{F'}{I'}{V}\in\cV(G)$ such that $\hF\in\Theta(F_{p_1},F')\setminus\{F'\}$. (See Figure \ref{fig:(2,2)-3}.) By the fact that $Q\not\in\del\cA_\Sep(V)$, we have $\hF\neq F_{p_1}$ and hence $|\cF(V_0)|,|\cF(V_2)|\geq 2$. By the discussions in Example \ref{ex:triangle} (item (1) in \textbf{Useful claims}), for any $\typemark{F''}{I''}{V}\in\cV(G)$, we have $\hF\in\Theta(F_{p_1},F'')$. In view of the comments after \eqref{0D floor-2}, the following holds:
\begin{itemize}
\item $\del_0^X\Floor_{\Delta_0,G_{Q;0}}((\hF,\{\{p_1\},\{p_2\}\}))=-1$. This follows from \eqref{beta' seg endpts norm}, \eqref{0D floor-1} and the fact that $|\cF(V_0)|\geq 2$.
\item $\del_0^X\Floor_{\Delta_2,G_{Q;2}}((\hF,\{\{p_1\},\{p_0\}\})) =1$. This follows from \eqref{beta' seg endpts norm}, \eqref{0D floor-2} and the fact that $|\cF(V_2)|\geq 2$.
\end{itemize}
By \eqref{bfloor} and the above, we have
\begin{align*}
\BFloor_{\Delta, G}(Q)=&\sum_{\substack{j:0\leq j\leq 2\\ \cV(\bG_{V;V_j})\ni Q}}(-1)^j\Floor_{\Delta_j,G_{Q;j}}(\res_{V_j}^V(Q)) \\
=&\Floor_{\Delta_0,G_{Q;0}}((\hF,\{\{p_1\},\{p_2\}\}))+\Floor_{\Delta_2,G_{Q;2}}((\hF,\{\{p_1\},\{p_0\}\}))
\end{align*}
and therefore
\begin{align}\label{1c}
\del^X_0(\BFloor_{\Delta, G}(Q))=(-1)+1=0.
\end{align}
\begin{figure}[h]
	\centering
	\includegraphics[scale=0.6]{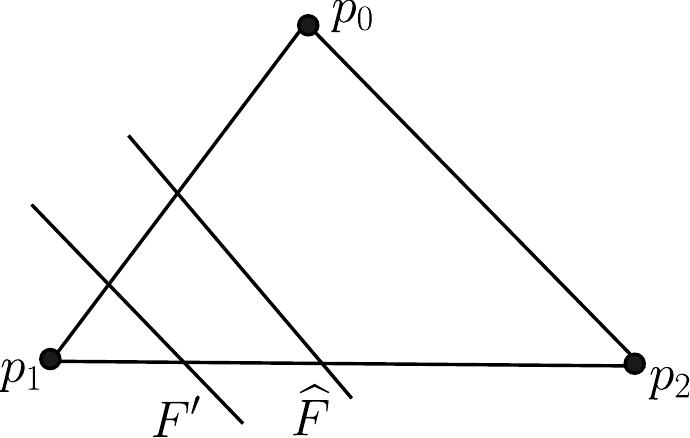}
	\caption{\label{fig:(2,2)-4}}
\end{figure}

\textbf{Case 1d}: Assume that $Q=\typemark{\hF}{\{p_1\}}{V}$ and there exists some $\typemark{F'}{I'}{V}\in\cV(G)$ such that $F'\in\Theta(F_{p_1},\hF)\setminus\{\hF\}$. (See Figure \ref{fig:(2,2)-4}.) By the fact that $Q\not\in\del\cA_\Sep(V)$ and the discussions in Example \ref{ex:triangle}  (item (1) in \textbf{Useful claims}), there exists some $G'\in\MCS(\bG_V)\setminus\{G\}$ such that $Q\in\cV(G')$. Moreover, $G'$ and $Q$ satisfy \textbf{Case 1c}. Hence by \eqref{1c} and the first assertion in Proposition \ref{full constructions} when $k=2$, we have
\begin{align}\label{1d}
\del^X_0(\BFloor_{\Delta, G}(Q))=-\del^X_0(\BFloor_{\Delta, G'}(Q))=0.
\end{align}
\begin{figure}[h]
	\centering
	\includegraphics[scale=0.6]{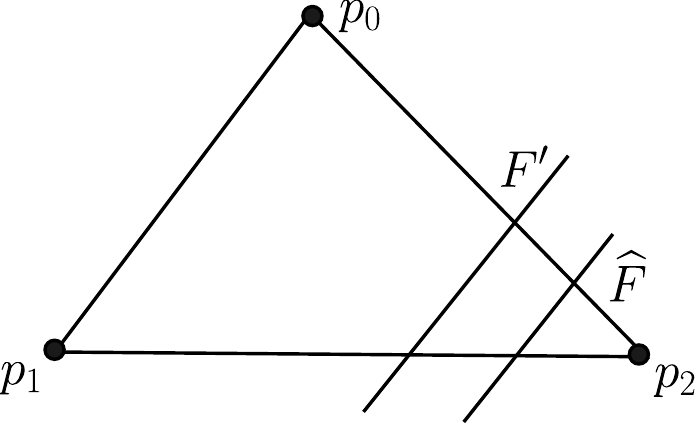}
	\caption{\label{fig:(2,2)-5}}
\end{figure}

\textbf{Case 1e}: Assume that $Q=\typemark{\hF}{\{p_2\}}{V}$ and there exists some $\typemark{F'}{I'}{V}\in\cV(G)$ such that $\hF\in\Theta(F_{p_2},F')\setminus\{F'\}$. (See Figure \ref{fig:(2,2)-5}.) By the fact that $Q\not\in\del\cA_\Sep(V)$, we have $\hF\neq F_{p_2}$ and hence $|\cF(V_0)|,|\cF(V_1)|\geq 2$. By the discussions in Example \ref{ex:triangle} (item (1) in \textbf{Useful claims}), for any $\typemark{F''}{I''}{V}\in\cV(G)$, we have $\hF\in\Theta(F_{p_2},F'')$. In view of the comments after \eqref{0D floor-2}, the following holds:
\begin{itemize}
\item $\del_0^X\Floor_{\Delta_0,G_{Q;0}}((\hF,\{\{p_2\},\{p_1\}\}))=1$. This follows from \eqref{beta' seg endpts norm}, \eqref{0D floor-2} and the fact that $|\cF(V_0)|\geq 2$.
\item $\del_0^X\Floor_{\Delta_1,G_{Q;1}}((\hF,\{\{p_2\},\{p_0\}\})) =1$. This follows from \eqref{beta' seg endpts norm}, \eqref{0D floor-2} and the fact that $|\cF(V_1)|\geq 2$.
\end{itemize}
By \eqref{bfloor} and the above, we have
\begin{align*}
\BFloor_{\Delta, G}(Q)=&\sum_{\substack{j:0\leq j\leq 2\\ \cV(\bG_{V;V_j})\ni Q}}(-1)^j\Floor_{\Delta_j,G_{Q;j}}(\res_{V_j}^V(Q)) \\
=&\Floor_{\Delta_0,G_{Q;0}}((\hF,\{\{p_2\},\{p_1\}\}))-\Floor_{\Delta_1,G_{Q;1}}((\hF,\{\{p_2\},\{p_0\}\}))
\end{align*}
and therefore
\begin{align}\label{1e}
\del^X_0(\BFloor_{\Delta, G}(Q))=1-1=0.
\end{align}
\begin{figure}[h]
	\centering
	\includegraphics[scale=0.6]{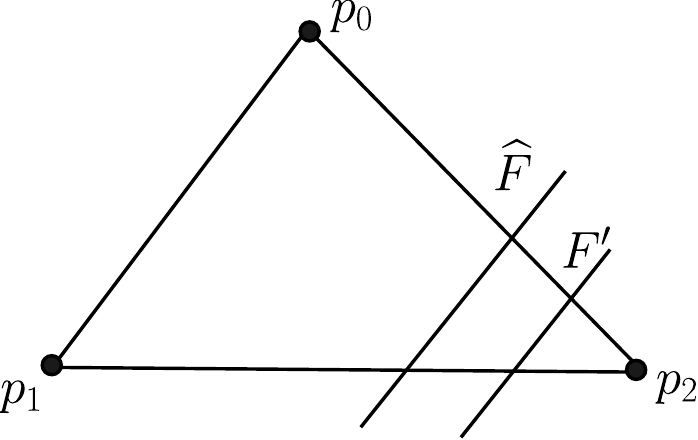}
	\caption{\label{fig:(2,2)-6}}
\end{figure}

\textbf{Case 1f}: Assume that $Q=\typemark{\hF}{\{p_2\}}{V}$ and there exists some $\typemark{F'}{I'}{V}\in\cV(G)$ such that $F'\in\Theta(F_{p_2},\hF)\setminus\{\hF\}$. (See Figure \ref{fig:(2,2)-6}.) By the fact that $Q\not\in\del\cA_\Sep(V)$ and the discussions in Example \ref{ex:triangle} (item (1) in \textbf{Useful claims}), there exists some $G'\in\MCS(\bG_V)\setminus\{G\}$ such that $Q\in\cV(G')$. Moreover, $G'$ and $Q$ satisfy \textbf{Case 1e}. Hence by \eqref{1e} and the first assertion in Proposition \ref{full constructions} when $k=2$, we have
\begin{align}\label{1f}
\del^X_0(\BFloor_{\Delta, G}(Q))=-\del^X_0(\BFloor_{\Delta, G'}(Q))=0.
\end{align}
This finishes the discussion in \textbf{Case 1}.
\begin{figure}[h]
	\centering
	\includegraphics[scale=0.6]{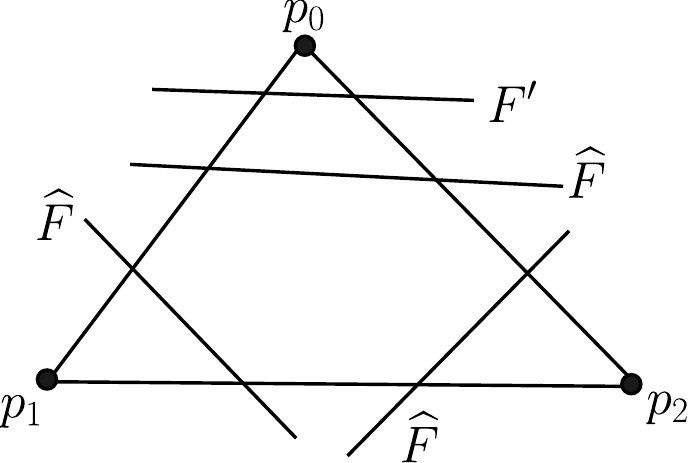}
	\caption{\label{fig:(2,2)-7}}
\end{figure}

\textbf{Case 2}: If there exist distinct $\typemark{F_1}{I_1}{V},\typemark{F_2}{I_2}{V}\in\cV(G)$ such that $F_1=F_2$, then
$$\Sep_V(F_1)=\PSep_V(F_1)=\{\stype{\{p_0\}}{V},\stype{\{p_1\}}{V},\stype{\{p_2\}}{V}\}.$$
Hence if we assume that $Q=\typemark{\hF}{I}{V}$, then $\hF=F_1=F_2$ and
$$\cV(G)=\{\typemark{\hF}{\{p_0\}}{V},\typemark{\hF}{\{p_1\}}{V},\typemark{\hF}{\{p_2\}}{V}\}.$$
By the fact that $Q\not\in\del\cA_\Sep(V)$ and the discussions in Example \ref{ex:triangle} (item (1) in \textbf{Useful claims}), there exists some $G'\in\MCS(\bG_V)\setminus\{G\}$ such that $Q\in\cV(G')$. Moreover, $G'$ and $Q$ satisfies \textbf{Case 1}. Therefore by equations \eqref{1a}-\eqref{1f} in \textbf{Case 1} and the first assertion in Proposition \ref{full constructions} when $k= 2$, we have
$$\del^X_0\BFloor_{\Delta,G}(Q)=-\del^X_0\BFloor_{\Delta,G'}(Q)=0.$$
This completes the proof that $\BFloor_{\Delta,G}(Q)$ is a closed singular $0$-chain.}

\item[(3).] {Let $Q=\typemark{F_{p_l}}{\{p_l\}}{V}$ for some $l\in\{0,1,2\}$ and $\Sing(G)= \emptyset$, if there exist some $\typemark{\hF}{I}{V}\in\cV(G)\setminus\{Q\}$ such that $\hF=F_{p_l}$, then $\stype{\{p_0\}}{V},\stype{\{p_1\}}{V},\stype{\{p_2\}}{V}\in\Sep_V(F_{p_l})$. Hence
$$\cV(G)=\left\{\typemark{F_{p_l}}{\{p_0\}}{V},\typemark{F_{p_l}}{\{p_1\}}{V},\typemark{F_{p_l}}{\{p_2\}}{V}\right\}.$$
Since $\Sing(G)=\emptyset$, by Definition \ref{res of types} and Definition \ref{singular}, we have $F_{p_1}=F_{p_2}=F_{p_0}=F_{p_l}$ and $G=\bG_V$. Hence $G_{Q;j}=\bG_{V_j}$ whenever $j\neq l$. By \eqref{0D floor special-1} and \eqref{0D floor special-2}, we have
\begin{align*}
&\sum_{\substack{j:0\leq j\leq 2\\ \cV(\bG_{V;V_j})\ni Q}}(-1)^j\Floor_{\Delta,G_{Q;j}}(\res_{V_j}^V(Q))\\
=&\begin{cases}
\displaystyle -\Floor_{\Delta_1,\bG_{V_1}}(p_0)+\Floor_{\Delta_2,\bG_{V_2}}(p_0)=-(-p_0)+(-p_0)=0,~&l=0, \\
\displaystyle \Floor_{\Delta_0,\bG_{V_0}}(p_1)+\Floor_{\Delta_2,\bG_{V_2}}(p_1)=-p_1+p_1=0,~&l=1,\\
\displaystyle \Floor_{\Delta_0,\bG_{V_0}}(p_2)-\Floor_{\Delta_1,\bG_{V_1}}(p_2)=p_2-p_2=0,~&l=2.
\end{cases}
\end{align*}

Now we assume that for any $\typemark{\hF}{I}{V}\in\cV(G)\setminus\{Q\}$, $\hF\neq F_{p_l}$. We discuss the cases $l=0,1,2$ separately.

\textbf{Case 1}: $l=0$. Then $F_{p_1},F_{p_2}\in\Gamma F\setminus\{ F_{p_0}\}$. Similar to \textbf{Case 1a} in the proof of the second assertion in Proposition \ref{full constructions} when $k=2$, \eqref{0D floor-1} and \eqref{beta' seg endpts} imply that
\begin{align*}
&\sum_{\substack{j:0\leq j\leq 2\\ \cV(\bG_{V;V_j})\ni Q}}(-1)^j\Floor_{\Delta_j,G_{Q;j}}(\res_{V_j}^V(Q)) \\
=&-\Floor_{\Delta_1,G_{Q;1}}((F_{p_0},\{\{p_0\},\{p_2\}\}))+\Floor_{\Delta_2,G_{Q;2}}((F_{p_0},\{\{p_0\},\{p_1\}\})) \\
=&-\frac{1}{2}\left(-\cI'_{F_{p_0}}[p_0,p_2]-\cI'_{F_{p_0}}[p_2,p_0]\right)+\frac{1}{2}\left(-\cI'_{F_{p_0}}[p_0,p_1]-\cI'_{F_{p_0}}[p_1,p_0]\right) \\
=&-\frac{1}{2}(-p_0-p_0)+\frac{1}{2}(-p_0-p_0)=0.
\end{align*}

\textbf{Case 2}: $l=1$. Then $F_{p_0},F_{p_2}\in\Gamma F\setminus\{ F_{p_1}\}$. Similar to \textbf{Case 1c} in the proof of the second assertion in Proposition \ref{full constructions} when $k=2$, \eqref{0D floor-1}, \eqref{0D floor-2} and \eqref{beta' seg endpts} imply that
\begin{align*}
&\sum_{\substack{j:0\leq j\leq 2\\ \cV(\bG_{V;V_j})\ni Q}}(-1)^j\Floor_{\Delta_j,G_{Q;j}}(\res_{V_j}^V(Q)) \\
=&\Floor_{\Delta_0,G_{Q;0}}((F_{p_1},\{\{p_1\},\{p_2\}\}))+\Floor_{\Delta_2,G_{Q;2}}((F_{p_1},\{\{p_1\},\{p_0\}\})) \\
=&\frac{1}{2}\left(-\cI'_{F_{p_1}}[p_1,p_2]-\cI'_{F_{p_1}}[p_2,p_1]\right)+\frac{1}{2}\left(\cI'_{F_{p_1}}[p_0,p_1]+\cI'_{F_{p_1}}[p_1,p_0]\right) \\
=&\frac{1}{2}(-p_1-p_1)+\frac{1}{2}(p_1+p_1)=0.
\end{align*}

\textbf{Case 3}: $l=2$. Then $F_{p_0},F_{p_1}\in\Gamma F\setminus\{ F_{p_2}\}$. Similar to \textbf{Case 1e} in the proof of the second assertion in Proposition \ref{full constructions} when $k=2$, \eqref{0D floor-2} and \eqref{beta' seg endpts} imply that
\begin{align*}
&\sum_{\substack{j:0\leq j\leq 2\\ \cV(\bG_{V;V_j})\ni Q}}(-1)^j\Floor_{\Delta_j,G_{Q;j}}(\res_{V_j}^V(Q)) \\
=&\Floor_{\Delta_0,G_{Q;0}}((F_{p_2},\{\{p_2\},\{p_1\}\}))-\Floor_{\Delta_1,G_{Q;1}}((F_{p_2},\{\{p_2\},\{p_0\}\})) \\
=&\frac{1}{2}\left(\cI'_{F_{p_2}}[p_1,p_2]+\cI'_{F_{p_2}}[p_2,p_1]\right)-\frac{1}{2}\left(\cI'_{F_{p_2}}[p_0,p_2]+\cI'_{F_{p_2}}[p_2,p_0]\right) \\
=&\frac{1}{2}(p_2+p_2)-\frac{1}{2}(p_2+p_2)=0.
\end{align*}

This finishes the proof of \eqref{(3,2)-1} when $k=2$. Thanks to \eqref{(3,2)-1} when $k=2$, it remains to prove \eqref{(3,2)-2} when $k=2$ for $G\in\MCS(\bG_V)$ such that $\Sing(G)\neq\emptyset$ and $\cV(G)\ints\del\cA_\Sep(V)\neq \emptyset$.
By the discussions in Example \ref{ex:triangle} (item (2) in \textbf{Useful claims}), there exists some $\hF\in\{F_{p_0},F_{p_1},F_{p_2}\}$ such that
$$\cV(G)=\left\{\typemark{\hF}{\{p_0\}}{V},\typemark{\hF}{\{p_1\}}{V},\typemark{\hF}{\{p_2\}}{V}\right\}.$$
Moreover,  $|\{p\in V|F_p=\hF\}|=|\cV(G)\cap\del\cA_\Sep(V)|\in\{1,2\}$. We discuss these two cases separately.

\textbf{Case 1}: If $|\{p\in V|F_p=\hF\}|=|\cV(G)\cap\del\cA_\Sep(V)|=1$, then
$\Sing(G)=V=\{p_0,p_1,p_2\}$. As a direct corollary,
\begin{align}\label{(3,2)-2-1}
\sum_{Q':Q'\in\cV(G)\ints\del\cA_\Sep(V)}\sum_{\substack{j:0\leq j\leq 2\\ \cV(\bG_{V;V_j})\ni Q'\\ {p_j}\not\in\Sing(G)}}(-1)^j\Floor_{\Delta_j,G_{Q';j}}(\res_{V_j}^V(Q'))=0.
\end{align}

\textbf{Case 2}: If $|\{p\in V|F_p=\hF\}|=|\cV(G)\cap\del\cA_\Sep(V)|=2$, then
$\Sing(G)=V\setminus\{p_{j_0}\}$, where $j_0$ is the unique element in $\{0,1,2\}$ which is not contained in $|\{p\in V|F_p=\hF\}|$. Moreover, $|\cF(V_{j_0})|=1$ and $|V_{j_0}|=2$. Let $Q_j:=\typemark{F_{p_j}}{\{p_j\}}{V}$ for any $j\in\{0,1,2\}$. By Definition \ref{res of types}, the \hyperlink{GQj}{definition} of $G_{Q;j}$ and the above discussions, we have
\begin{align}\label{(3,2)-2-2}
&\sum_{Q':Q'\in\cV(G)\ints\del\cA_\Sep(V)}\sum_{\substack{j:0\leq j\leq 2\\ \cV(\bG_{V;V_j})\ni Q'\\{ p_j}\not\in\Sing(G)}}(-1)^j\Floor_{\Delta_j,G_{Q';j}}(\res_{V_j}^V(Q')) \nonumber\\
=&\sum_{\substack{Q':Q'\in\cV(G)\ints\del\cA_\Sep(V)\\Q'\in\cV(\bG_{V;V_{j_0}})}}(-1)^{j_0}\Floor_{\Delta_{j_0},G_{Q';j_0}}(\res_{V_{j_0}}^V(Q')) \nonumber\\
=&\sum_{\substack{j:0\leq j\leq 2\\ j\neq j_0}}(-1)^{j_0}\Floor_{\Delta_{j_0},G_{Q_j;j_0}}(p_i)=(-1)^{j_0}\sum_{\substack{0\leq j\leq 2\\ j\neq j_0}}\Floor_{\Delta_{j_0},\bG_{V_{j_0}}}(p_j).
\end{align}
Hence \eqref{(3,2)-2} follows from \eqref{(3,2)-1}, \eqref{(3,2)-2-1} and \eqref{(3,2)-2-2}.}

\item[(4).] \textbf{The proof of this assertion works for any $k\geq 2$, assuming that for any $(k',j')<(k,4)$, the $j'$-th assertion in Proposition \ref{full constructions} for the case $k'$ holds (although the arguments still work under weaker assumptions).}

In the case $k$, by \eqref{bfloor}, the third assertion in Proposition \ref{baby constructions} or the fourth assertion in Proposition \ref{full constructions} for the case $k-1$, we have
$$\vertsupp(\BFloor_{\Delta,G}(Q))\subset\{x\in\Gamma x_0|F_x=F_Q\}.$$
By the fifth remark after {Definition \ref{FBCone}},
$$\vertsupp(\Floor_{\Delta,G}(Q))\subset\{x\in\Gamma x_0|F_x=F_Q\}.$$
\item[(5).] \textbf{The proof of this assertion works for any $k\geq 2$, assuming that for any $(k',j')<(k,5)$, the $j'$-th assertion in Proposition \ref{full constructions} for the case $k'$ holds (although the arguments still work under weaker assumptions).}

It suffices to prove this assertion for any permutation $\tau$ of $\{0,...,k\}$ satisfying the following: there exist $0\leq j_1<j_2\leq k$ such that
$$\tau(i)=\begin{cases}
\displaystyle i,~&\mathrm{if}~ i\neq j_1,j_2, \\
\displaystyle j_2,~&\mathrm{if}~i=j_1,\\
\displaystyle j_1,~&\mathrm{if}~i=j_2.
\end{cases}$$
In particular, $(-1)^{\mathrm{sgn}(\tau)}=-1$.

By \eqref{bfloor} and the {first} assertion in Proposition \ref{full constructions} in the case of $k$, we can assume WLOG that $\Sing(G)=\emptyset$. (This is because for those $G$ such that $\Sing(G)\neq\emptyset$, if $Q\in\del\cA_\Sep(V)$, then nothing needs to be proved. If $Q\not\in\del\cA_\Sep(V)$, by the assertion in Corollary \ref{remaining properties of sep graph}, one can choose $G'\in \MCS(\bG_V)$ such that $G\neq G'$ and $Q\in\cV(G)\ints\cV(G')$. By Lemma \ref{singular case} and the first assertion in Lemma \ref{reinterpretation of edges}, $\Sing(G')=\emptyset$. Hence, by the first assertion in Proposition \ref{full constructions} in the case $k$, one only needs to prove $\BFloor_{\Delta,G'}(Q)=-\BFloor_{\Delta',G'}(Q)$ and $\Floor_{\Delta,G'}(Q)=-\Floor_{\Delta',G'}(Q)$ instead.)

{Recall the \hyperlink{GQj}{notations} at the beginning of the construction of $\phi_k$, $k\geq 2$, we have $V(\Delta')=V$, }
$${V(\Delta'_i)=\begin{cases}
\displaystyle V_i,~&\mathrm{if}~ i\not\in\{ j_1,j_2\}, \\
\displaystyle V_{j_2},~&\mathrm{if}~i=j_1,\\
\displaystyle V_{j_1},~&\mathrm{if}~i=j_2
\end{cases}\mathrm{~and}~
G_{Q;i;\Delta'}=
\begin{cases}
\displaystyle G_{Q;i},~&\mathrm{if}~ i\not\in\{ j_1,j_2\}, \\
\displaystyle G_{Q;j_2},~&\mathrm{if}~i=j_1,\\
\displaystyle G_{Q;j_1},~&\mathrm{if}~i=j_2.
\end{cases}}$$
Therefore by the fourth assertion in Proposition \ref{baby constructions} or the fifth assertion in Proposition \ref{full constructions} in the case $k-1$, we have
\begin{align}\label{(5,k-1)}
\Floor_{\Delta'_i,G_{Q;i;{\Delta'}}}(\res_{{V(\Delta'_i)}}^{{V(\Delta')}}(Q))=\begin{cases}
\displaystyle -\Floor_{\Delta_i,G_{Q;i}}(\res_{V_i}^V(Q)), ~&\mathrm{if}~i\not\in\{ j_1,j_2\}, \\
\displaystyle (-1)^{j_2-j_1-1}\Floor_{\Delta_{j_2},G_{Q;j_2}}(\res_{V_{j_2}}^V(Q)), ~&\mathrm{if}~i=j_1,\\
\displaystyle (-1)^{j_1-j_2-1}\Floor_{\Delta_{j_1},G_{Q;j_1}}(\res_{V_{j_1}}^V(Q)), ~&\mathrm{if}~i=j_2.
\end{cases}
\end{align}
Apply \eqref{(5,k-1)} to \eqref{bfloor}, we have $\BFloor_{\Delta,G}(Q)=-\BFloor_{\Delta',G}(Q)$ in the case $k$. The fact that $\Floor_{\Delta,G}(Q)=-\Floor_{\Delta',G}(Q)$ in the case $k$ follows directly from the third remark after {Definition \ref{FBCone}}.

\item[(6).] We first notice that by \eqref{BChamber}, the second assertion in Corollary \ref{remaining properties of sep graph} and the {first} assertion in Proposition \ref{full constructions} in the case $k$,
\begin{align}\label{(6,k)-1g}
\begin{split}
\sum_{G:G\in\MCS(\bG_V)}
\BChamber_\Delta(G)
=&-(-1)^k\sum_{G:G\in\MCS(\bG_V)}\sum_{Q:Q\in\cV(G)}\Floor_{\Delta,G}(Q) \\
&+\sum_{G:G\in\MCS(\bG_V)}\sum_{\substack{j,G_j:0\leq j\leq k\\ G\ints\bG_{V;V_j}\neq\emptyset \\ G_j\in\subord_{V_j}^V(G\ints\bG_{V;V_j})}}(-1)^j\Chamber_{\Delta_j}(G_j) \\
=&-(-1)^k\sum_{Q:Q\in\cV(\bG_V)}\sum_{\substack{G:G\in\MCS(\bG_V)\\\cV(G)\ni Q}}\Floor_{\Delta,G}(Q) \\
&+\sum_{j=0}^k(-1)^j\sum_{\substack{G:G\in\MCS(\bG_V)\\G\ints\bG_{V;V_j}\neq\emptyset }}\sum_{{G_j:  G_j\in\subord_{V_j}^V(G\ints\bG_{V;V_j})}}\Chamber_{\Delta_j}(G_j) \\
=&\sum_{j=0}^k(-1)^j\sum_{\substack{G:G\in\MCS(\bG_V)\\G\ints\bG_{V;V_j}\neq\emptyset }}\sum_{{G_j:  G_j\in\subord_{V_j}^V(G\ints\bG_{V;V_j})}}\Chamber_{\Delta_j}(G_j).
\end{split}
\end{align}
It follows from Definition \ref{subordinate}, the third and the fourth assertions in Lemma \ref{properties of W-face} that for any $0\leq j\leq k$,
\begin{align}\label{(6,k)-2g}
&\sum_{\substack{G:G\in\MCS(\bG_V)\\G\ints\bG_{V;V_j}\neq\emptyset }}\sum_{{G_j:  G_j\in\subord_{V_j}^V(G\ints\bG_{V;V_j})}}\Chamber_{\Delta_j}(G_j)\nonumber\\
=&\sum_{G_j:G_j\in\MCS(\bG_{V_j})}\sum_{\substack{G:G\in\MCS(\bG_V)\\G\cap\bG_{V;V_j}\neq \emptyset\\\subord_{V_j}^V(G\ints\bG_{V;V_j})\ni G_j}}\Chamber_{\Delta_j}(G_j)=\sum_{G_j:G_j\in\MCS(\bG_{V_j})}\Chamber_{\Delta_j}(G_j)
\end{align}
Combining \eqref{(6,k)-1g} and \eqref{(6,k)-2g}, we obtain
\begin{align}\label{(6,k)-3g}
\sum_{G:G\in\MCS(\bG_V)}
\BChamber_\Delta(G)
=\sum_{j=0}^k(-1)^j\sum_{G_j:G_j\in\MCS(\bG_{V_j})}\Chamber_{\Delta_j}(G_j).
\end{align}
The right hand side of \eqref{(6,k)-3g} equals $\phi_{k-1}(\del_k^Y(\Delta))$ by the first assertion in Proposition \ref{baby constructions} or \eqref{phi full construction}. Therefore \eqref{(6,k)} follows from \eqref{(6,k)-3g} and the sixth remark after {Definition \ref{FBCone}} if we can prove that $\BChamber_\Delta(G)$ is closed for any $G\in\MCS(\bG_V)$.

When $k=2$, \eqref{BChamber} (used in the first equality), the {second} assertion in Proposition \ref{full constructions} when $k=2$ (used in the second equality), \eqref{0D bchamber} (used in the second equality), the eighth assertion in Proposition \ref{baby constructions} (used in the third equality) and \eqref{bfloor} (used in the third equality) imply that
\begin{align}\label{bchamber closed}
&\del^X_1\BChamber_\Delta(G) \nonumber\\
=&-\sum_{Q:Q\in\cV(G)}\del^X_1\Floor_{\Delta,G}(Q)+\sum_{\substack{j,G_j:0\leq j\leq 2\\ G\ints\bG_{V;V_j}\neq\emptyset \\ G_j\in\subord_{V_j}^V(G\ints\bG_{V;V_j})}}(-1)^j\del^X_1\Chamber_{\Delta_j}(G_j) \nonumber\\
=&-\sum_{Q:Q\in\cV(G)}\BFloor_{\Delta,G}(Q)+\sum_{\substack{j,G_j:0\leq j\leq 2\\ G\ints\bG_{V;V_j}\neq\emptyset \\ G_j\in\subord_{V_j}^V(G\ints\bG_{V;V_j})}}(-1)^j\BChamber_{\Delta_j}(G_j) \nonumber\\
=&-\sum_{\substack{Q:Q\in\cV(G)\\ Q\not\in\del\cA_\Sep(V)}}\sum_{\substack{j:0\leq j\leq 2\\\cV(\bG_{V;V_j})\ni Q\\{p_j}\not\in\Sing(G)}}(-1)^j\Floor_{\Delta_j, G_{Q;j}}(\res_{V_j}^V(Q)) \nonumber\\
&+\sum_{\substack{Q:Q\in\cV(G)\\ Q\not\in\del\cA_\Sep(V)}}\sum_{\substack{j:0\leq j\leq 2\\ \cV(\bG_{V;V_j})\ni Q \\{p_j}\in\Sing(G)}}(-1)^j\Floor_{\Delta_j,G_{Q;j}^{\mathrm{op}}}(\res_{V_j}^V(Q)) \nonumber\\
&+\sum_{\substack{j,G_j:0\leq j\leq 2\\ G\ints\bG_{V;V_j}\neq\emptyset \\ G_j\in\subord_{V_j}^V(G\ints\bG_{V;V_j})}}\sum_{P_j:P_j\in\cV(G_j)}(-1)^j\Floor_{\Delta_j,G_j}(P_j) \nonumber\\
\begin{split}
=&-\sum_{\substack{Q:Q\in\cV(G)\\ Q\not\in\del\cA_\Sep(V)}}\sum_{\substack{j:0\leq j\leq 2\\\cV(\bG_{V;V_j})\ni Q\\{p_j}\not\in\Sing(G)}}(-1)^j\Floor_{\Delta_j, G_{Q;j}}(\res_{V_j}^V(Q)) \\
&+\sum_{\substack{Q:Q\in\cV(G)\\ Q\not\in\del\cA_\Sep(V)}}\sum_{\substack{j:0\leq j\leq 2\\ \cV(\bG_{V;V_j})\ni Q \\{p_j}\in\Sing(G)}}(-1)^j\Floor_{\Delta_j,G_{Q;j}^{\mathrm{op}}}(\res_{V_j}^V(Q)) \\
&+\sum_{\substack{j,G_j:0\leq j\leq 2\\ G\ints\bG_{V;V_j}\neq\emptyset \\ G_j\in\subord_{V_j}^V(G\ints\bG_{V;V_j})\\{p_j}\not\in\Sing(G)}}\sum_{P_j:P_j\in\cV(G_j)}(-1)^j\Floor_{\Delta_j,G_j}(P_j) \\
&+\sum_{\substack{j,G_j:0\leq j\leq 2\\ G\ints\bG_{V;V_j}\neq\emptyset \\ G_j\in\subord_{V_j}^V(G\ints\bG_{V;V_j})\\{p_j}\in\Sing(G)}}\sum_{P_j:P_j\in\cV(G_j)}(-1)^j\Floor_{\Delta_j,G_j}(P_j).
\end{split}
\end{align}

When ${p_j}\in\Sing(G)$, by Lemma \ref{singular case} and Definition \ref{subordinate}, we have $\subord_{V_j}^V(G\ints\bG_{V;V_j})=\emptyset$. Hence
\begin{align}\label{bchamber closed-1}
\sum_{\substack{j,G_j:0\leq j\leq 2\\ G\ints\bG_{V;V_j}\neq\emptyset \\ G_j\in\subord_{V_j}^V(G\ints\bG_{V;V_j})\\{p_j}\in\Sing(G)}}\sum_{P_j:P_j\in\cV(G_j)}(-1)^j\Floor_{\Delta_j,G_j}(P_j)   =0.
\end{align}

When ${p_j}\not\in\Sing(G)$, by the discussions in Example \ref{ex:triangle} (item (3) in \textbf{Useful claims}),
the restriction map
$$\res_{V_j}^V|_{\cV(G\ints\bG_{V;V_j})}:\cV(G\ints\bG_{V;V_j})\to\cV(\bG_{V_j})$$
is injective.

As a corollary of \eqref{eqn:MCS=subord} (used in the first equality and the fifth equalty), the second assertion in Proposition \ref{baby constructions} when $k=2$ (used in the fourth equality, the inner-most summation in the second big summation is zero), the previous paragraph (used in the fifth equality), the second assertion in Lemma \ref{easy properties of enrich} (used in the fourth and the sixth equality) and the \hyperlink{GQj}{definition} of $G_{Q;j}=G_{Q;j;\Delta}$ (used in the sixth equality; see also the statement of Proposition \ref{full constructions} for the simplified notations), we have
\begin{align}\label{bchamber closed-2-1}
&\sum_{\substack{j,G_j:0\leq j\leq 2\\ G\ints\bG_{V;V_j}\neq\emptyset \\ G_j\in\subord_{V_j}^V(G\ints\bG_{V;V_j})\\{p_j}\not\in\Sing(G)}}\sum_{P_j:P_j\in\cV(G_j)}(-1)^j\Floor_{\Delta_j,G_j}(P_j)\nonumber \\
=&\sum_{\substack{j:0\leq j\leq 2\\ G\ints\bG_{V;V_j}\neq\emptyset \\ {p_j}\not\in\Sing(G)}}\sum_{\substack{G_j,P_j:G_j\in\MCS(\Enrich_{V_j}^V(G\ints\bG_{V;V_j}))\\P_j\in\cV(G_j)}}(-1)^j\Floor_{\Delta_j,G_j}(P_j) \nonumber\\
=&\sum_{\substack{j:0\leq j\leq 2\\ G\ints\bG_{V;V_j}\neq\emptyset \\ {p_j}\not\in\Sing(G)}}(-1)^j\sum_{P_j:P_j\in\cV(\Enrich_{V_j}^V(G\ints\bG_{V;V_j}))}\sum_{\substack{G_j:G_j\in\MCS(\Enrich_{V_j}^V(G\ints\bG_{V;V_j}))\\ \cV(G_j)\ni P_j}}\Floor_{\Delta_j,G_j}(P_j) \nonumber\\
=&\sum_{\substack{j:0\leq j\leq 2\\ G\ints\bG_{V;V_j}\neq\emptyset \\ {p_j}\not\in\Sing(G)}}(-1)^j\sum_{P_j:P_j\in\res_{V_j}^V(\cV(G\ints\bG_{V;V_j}))}\sum_{\substack{G_j:G_j\in\MCS(\Enrich_{V_j}^V(G\ints\bG_{V;V_j}))\\ \cV(G_j)\ni P_j}}\Floor_{\Delta_j,G_j}(P_j) \nonumber\\
&+\sum_{\substack{j:0\leq j\leq 2\\ G\ints\bG_{V;V_j}\neq\emptyset \\ {p_j}\not\in\Sing(G)}}(-1)^j\sum_{\substack{P_j:P_j\in\cV(\Enrich_{V_j}^V(G\ints\bG_{V;V_j}))\\P_j\not\in\res_{V_j}^V(\cV(G\ints\bG_{V;V_j}))}}\sum_{\substack{G_j:G_j\in\MCS(\Enrich_{V_j}^V(G\ints\bG_{V;V_j}))\\ \cV(G_j)\ni P_j}}\Floor_{\Delta_j,G_j}(P_j) \nonumber\\
=&\sum_{\substack{j:0\leq j\leq 2\\ G\ints\bG_{V;V_j}\neq\emptyset \\ {p_j}\not\in\Sing(G)}}(-1)^j\sum_{P_j:P_j\in\res_{V_j}^V(\cV(G\ints\bG_{V;V_j}))}\sum_{\substack{G_j:G_j\in\MCS(\Enrich_{V_j}^V(G\ints\bG_{V;V_j}))\\ \cV(G_j)\ni P_j}}\Floor_{\Delta_j,G_j}(P_j) \nonumber\\
=&\sum_{\substack{j:0\leq j\leq 2\\ G\ints\bG_{V;V_j}\neq\emptyset \\ {p_j}\not\in\Sing(G)}}(-1)^j\sum_{Q:Q\in\cV(G\ints\bG_{V;V_j})}\sum_{\substack{G_j:G_j\in\subord_{V_j}^V(G\ints\bG_{V;V_j})\\ \cV(G_j)\ni \res_{V_j}^V(Q)}}\Floor_{\Delta_j,G_j}(\res_{V_j}^V(Q)) \nonumber\\
=&\sum_{Q:Q\in\cV(G)}\sum_{\substack{j:0\leq j\leq 2\\\cV(\bG_{V;V_j})\ni Q\\{p_j}\not\in\Sing(G)}}(-1)^j\Floor_{\Delta_j, G_{Q;j}}(\res_{V_j}^V(Q)).
\end{align}

Based on \eqref{(3,2)-2} in the {third} assertion in Proposition \ref{full constructions} and the fact that $k=2$ (equivalently $|V|=3$), we discuss the following $2$ cases:

\textbf{Case 1}: If $\Sing(G)\neq \emptyset$ and $|\cV(G)\ints\del\cA_\Sep(V)|=2$, by the discussions in Example \ref{ex:triangle} (item (2) in \textbf{Useful claims}), there exists $\hF\in\{F_{p_0},F_{p_1},F_{p_2}\}$ such that
$$\cV(G)=\{\typemark{\hF}{\{p_0\}}{V}, \typemark{\hF}{\{p_1\}}{V},\typemark{\hF}{\{p_2\}}{V}\}.$$
Moreover, there exists a unique $j_0\in\{0,1,2\}$ such that $Q_{j_0}:=\typemark{F_{p_{j_0}}}{\{p_{j_0}\}}{V}$ is the unique element in $\del\cA_\Sep(V)\setminus\cV(G)$ and $\Sing(G)=V\setminus\{p_{j_0}\}$. In particular, $F_p=\hF$ for any $p\in V\setminus\{p_{j_0}\}$. Then by \eqref{(3,2)-2} in the case of $k=2$, we have
\begin{align}\label{bchamber closed-2-2special}
\sum_{{Q:Q\in\cV(G)\cap\del\cA_\Sep(V)}}\sum_{\substack{j:0\leq j\leq 2\\\cV(\bG_{V;V_j})\ni Q\\{p_j}\not\in\Sing(G)}}(-1)^j\Floor_{\Delta_j, G_{Q;j}}(\res_{V_j}^V(Q))=(-1)^{j_0}\sum_{\substack{0\leq j\leq 2\\ j\neq j_0}}\Floor_{\Delta_{j_0},\bG_{V_{j_0}}}(p_j).
\end{align}
For simplicity, we denote by $Q_l^G:=\typemark{\hF}{\{p_l\}}{V}$ for any $l\in\{0,1,2\}$. Then, $Q_{j_0}^G$ is the unique element in $\cV(G)\setminus\del\cA_\Sep(V)$. Apply \eqref{bchamber closed-1}, \eqref{bchamber closed-2-1} and \eqref{bchamber closed-2-2special} to \eqref{bchamber closed}, we have
\begin{align}\label{bchamber closed prep special-1}
&\del^X_1\BChamber_\Delta(G) \nonumber\\
=&-\sum_{\substack{Q:Q\in\cV(G)\\ Q\not\in\del\cA_\Sep(V)}}\sum_{\substack{j:0\leq j\leq 2\\\cV(\bG_{V;V_j})\ni Q\\{p_j}\not\in\Sing(G)}}(-1)^j\Floor_{\Delta_j, G_{Q;j}}(\res_{V_j}^V(Q)) \nonumber\\
&+\sum_{\substack{Q:Q\in\cV(G)\\ Q\not\in\del\cA_\Sep(V)}}\sum_{\substack{j:0\leq j\leq 2\\ \cV(\bG_{V;V_j})\ni Q \\{p_j}\in\Sing(G)}}(-1)^j\Floor_{\Delta_j,G_{Q;j}^{\mathrm{op}}}(\res_{V_j}^V(Q)) \nonumber\\
&+\sum_{\substack{j,G_j:0\leq j\leq 2\\ G\ints\bG_{V;V_j}\neq\emptyset \\ G_j\in\subord_{V_j}^V(G\ints\bG_{V;V_j})\\{p_j}\not\in\Sing(G)}}\sum_{P_j:P_j\in\cV(G_j)}(-1)^j\Floor_{\Delta_j,G_j}(P_j) \nonumber\\
&+\sum_{\substack{j,G_j:0\leq j\leq 2\\ G\ints\bG_{V;V_j}\neq\emptyset \\ G_j\in\subord_{V_j}^V(G\ints\bG_{V;V_j})\\{p_j}\in\Sing(G)}}\sum_{P_j:P_j\in\cV(G_j)}(-1)^j\Floor_{\Delta_j,G_j}(P_j) \nonumber\\
=&-\sum_{\substack{Q:Q\in\cV(G)\\ Q\not\in\del\cA_\Sep(V)}}\sum_{\substack{j:0\leq j\leq 2\\\cV(\bG_{V;V_j})\ni Q\\{p_j}\not\in\Sing(G)}}(-1)^j\Floor_{\Delta_j, G_{Q;j}}(\res_{V_j}^V(Q)) \nonumber\\
&+\sum_{\substack{Q:Q\in\cV(G)\\ Q\not\in\del\cA_\Sep(V)}}\sum_{\substack{j:0\leq j\leq 2\\ \cV(\bG_{V;V_j})\ni Q \\{p_j}\in\Sing(G)}}(-1)^j\Floor_{\Delta_j,G_{Q;j}^{\mathrm{op}}}(\res_{V_j}^V(Q)) \nonumber\\
&+\sum_{Q:Q\in\cV(G)}\sum_{\substack{j:0\leq j\leq 2\\\cV(\bG_{V;V_j})\ni Q\\{p_j}\not\in\Sing(G)}}(-1)^j\Floor_{\Delta_j, G_{Q;j}}(\res_{V_j}^V(Q))+0  \nonumber\\
=&\sum_{\substack{Q:Q\in\cV(G)\\ Q\not\in\del\cA_\Sep(V)}}\sum_{\substack{j:0\leq j\leq 2\\ \cV(\bG_{V;V_j})\ni Q \\{p_j}\in\Sing(G)}}(-1)^j\Floor_{\Delta_j,G_{Q;j}^{\mathrm{op}}}(\res_{V_j}^V(Q)) \nonumber\\
&+\sum_{{Q:Q\in\cV(G)\cap\del\cA_\Sep(V)}}\sum_{\substack{j:0\leq j\leq 2\\\cV(\bG_{V;V_j})\ni Q\\{p_j}\not\in\Sing(G)}}(-1)^j\Floor_{\Delta_j, G_{Q;j}}(\res_{V_j}^V(Q))\nonumber\\
=&\sum_{\substack{j:0\leq j\leq 2\\ j\neq j_0}}(-1)^j\Floor_{\Delta_j,G_{Q_{j_0}^G;j}^{\mathrm{op}}}(\res_{V_j}^V(Q_{j_0}^G))+(-1)^{j_0}\sum_{\substack{j:0\leq j\leq 2\\ j\neq j_0}}\Floor_{\Delta_{j_0},\bG_{V_{j_0}}}(p_j).
\end{align}
Recall that $F_p=\hF$ for any $p\in V\setminus\{p_{j_0}\}$ and that $F_{p_{j_0}}\neq \hF$. By \eqref{0D floor special-1}, \eqref{0D floor special-2}, \eqref{0D floor-1} and \eqref{0D floor-2}, we can further simplify \eqref{bchamber closed prep special-1} as
\begin{align}\label{bchamber closed prep special-2}
&\sum_{\substack{j:0\leq j\leq 2\\ j\neq j_0}}(-1)^j\Floor_{\Delta_j,G_{Q_{j_0}^G;j}^{\mathrm{op}}}(\res_{V_j}^V(Q_{j_0}^G))+(-1)^{j_0}\sum_{\substack{j:0\leq j\leq 2\\ j\neq j_0}}\Floor_{\Delta_{j_0},\bG_{V_{j_0}}}(p_j)\nonumber\\
=&\begin{cases}
\displaystyle \begin{aligned}
&-\Floor_{\Delta_1,G_{Q^G_0;1}^\mathrm{op}}(\hF,\{\{p_0\},\{p_2\}\})+\Floor_{\Delta_2,G_{Q^G_0;2}^\mathrm{op}}(\hF,\{\{p_0\},\{p_1\}\}) \\
&+\Floor_{\Delta_{0},\bG_{V_{0}}}(p_1)+\Floor_{\Delta_{0},\bG_{V_{0}}}(p_2),
\end{aligned}~&\mathrm{if}~j_0=0,\\
\displaystyle \begin{aligned}
&\Floor_{\Delta_0,G_{Q^G_1;0}^\mathrm{op}}(\hF,\{\{p_1\},\{p_2\}\})+\Floor_{\Delta_2,G_{Q^G_1;2}^\mathrm{op}}(\hF,\{\{p_1\},\{p_0\}\}) \\
&-\Floor_{\Delta_{1},\bG_{V_{1}}}(p_0)-\Floor_{\Delta_{1},\bG_{V_{1}}}(p_2),
\end{aligned}~&\mathrm{if}~j_0=1,\\
\displaystyle \begin{aligned}
&\Floor_{\Delta_0,G_{Q^G_2;0}^\mathrm{op}}(\hF,\{\{p_2\},\{p_1\}\})-\Floor_{\Delta_1,G_{Q^G_2;1}^\mathrm{op}}(\hF,\{\{p_2\},\{p_0\}\}) \\
&+\Floor_{\Delta_{2},\bG_{V_{2}}}(p_0)+\Floor_{\Delta_{2},\bG_{V_{2}}}(p_1),
\end{aligned}~&\mathrm{if}~j_0=2
\end{cases}\nonumber\\
=&\begin{cases}
\displaystyle -p_2+p_1+(-p_1)+p_2=0,~&\mathrm{if}~j_0=0,\\
\displaystyle p_2-p_0-(-p_0)-p_2=0,~&\mathrm{if}~j_0=1,\\
\displaystyle -p_1-(-p_0)+(-p_0)+p_1=0,~&\mathrm{if}~j_0=2.
\end{cases}
\end{align}
Combining \eqref{bchamber closed prep special-1} and \eqref{bchamber closed prep special-2}, we have $\del^X_1\BChamber_\Delta(G)=0$ in \textbf{Case 1}.

\textbf{Case 2}: This is just the complementary case of the previous case. If $\Sing(G)= \emptyset$ or $|\cV(G)\ints\del\cA_\Sep(V)|\neq 2$, then by \eqref{(3,2)-2} in the case of $k=2$, we have
\begin{align}\label{bchamber closed-2-2}
\sum_{Q:Q\in\cV(G)\ints\del\cA_\Sep(V)}\sum_{\substack{j:0\leq j\leq k\\ \cV(\bG_{V;V_j})\ni Q\\ {p_j}\not\in\Sing(G)}}(-1)^j\Floor_{\Delta_j,G_{Q;j}}(\res_{V_j}^V(Q))=0.
\end{align}
Similar to the previous case, we can apply \eqref{bchamber closed-1}, \eqref{bchamber closed-2-1} and \eqref{bchamber closed-2-2} to \eqref{bchamber closed} and obtain
\begin{align}\label{bchamber closed prep-1}
\del^X_1\BChamber_\Delta(G)
=&-\sum_{\substack{Q:Q\in\cV(G)\\ Q\not\in\del\cA_\Sep(V)}}\sum_{\substack{j:0\leq j\leq 2\\\cV(\bG_{V;V_j})\ni Q\\{p_j}\not\in\Sing(G)}}(-1)^j\Floor_{\Delta_j, G_{Q;j}}(\res_{V_j}^V(Q)) \nonumber\\
&+\sum_{\substack{Q:Q\in\cV(G)\\ Q\not\in\del\cA_\Sep(V)}}\sum_{\substack{j:0\leq j\leq 2\\ \cV(\bG_{V;V_j})\ni Q \\{p_j}\in\Sing(G)}}(-1)^j\Floor_{\Delta_j,G_{Q;j}^{\mathrm{op}}}(\res_{V_j}^V(Q)) \nonumber\\
&+\sum_{Q:Q\in\cV(G)}\sum_{\substack{j:0\leq j\leq 2\\\cV(\bG_{V;V_j})\ni Q\\{p_j}\not\in\Sing(G)}}(-1)^j\Floor_{\Delta_j, G_{Q;j}}(\res_{V_j}^V(Q))+0  \nonumber\\
=&\sum_{\substack{Q:Q\in\cV(G)\\ Q\not\in\del\cA_\Sep(V)}}\sum_{\substack{j:0\leq j\leq 2\\ \cV(\bG_{V;V_j})\ni Q \\{p_j}\in\Sing(G)}}(-1)^j\Floor_{\Delta_j,G_{Q;j}^{\mathrm{op}}}(\res_{V_j}^V(Q)) \nonumber\\
&+\sum_{{Q:Q\in\cV(G)\cap\del\cA_\Sep(V)}}\sum_{\substack{j:0\leq j\leq 2\\\cV(\bG_{V;V_j})\ni Q\\{p_j}\not\in\Sing(G)}}(-1)^j\Floor_{\Delta_j, G_{Q;j}}(\res_{V_j}^V(Q))\nonumber\\
=&\sum_{\substack{Q:Q\in\cV(G)\\ Q\not\in\del\cA_\Sep(V)}}\sum_{\substack{j:0\leq j\leq 2\\ \cV(\bG_{V;V_j})\ni Q \\{p_j}\in\Sing(G)}}(-1)^j\Floor_{\Delta_j,G_{Q;j}^{\mathrm{op}}}(\res_{V_j}^V(Q)).
\end{align}

If $\Sing(G)=\emptyset$, \eqref{bchamber closed prep-1} is obviously equal to $0$, which proves that $\del_1^X\BChamber_\Delta(G)=0$. It remains for us to consider the case when $\Sing(G)\neq\emptyset$ and $|\cV(G)\ints\del\cA_\Sep(V)|\neq 2$.

By Lemma \ref{singular case}, for any $j\in\{0,1,2\}$ such that $p_j\in\Sing(G)$, there exist a unique $P_j^G\in\cV(\bG_{V_j})$ such that $\res_{V_j}^V(\cV(G)\cap\cV(\bG_{V;V_j}))=\{P_j^G\}$. In particular, $\cV(G)\ints\cV(\bG_{V;V_j})\neq \emptyset$. Therefore we can rewrite \eqref{bchamber closed prep-1} as
\begin{align}\label{bchamber closed prep-2}
\begin{split}
&\sum_{\substack{Q:Q\in\cV(G)\\ Q\not\in\del\cA_\Sep(V)}}\sum_{\substack{j:0\leq j\leq 2\\ \cV(\bG_{V;V_j})\ni Q \\{p_j}\in\Sing(G)}}(-1)^j\Floor_{\Delta_j,G_{Q;j}^{\mathrm{op}}}(\res_{V_j}^V(Q))\\
=&\sum_{\substack{j:0\leq j\leq 2\\{p_j}\in\Sing(G)}}\sum_{\substack{Q:Q\in\cV(G)\ints\cV(\bG_{V;V_j})\\ Q\not\in\del\cA_\Sep(V)\\ \res_{V_j}^V(Q)=P_j^G }}(-1)^j\Floor_{\Delta_j,G_{Q;j}^{\mathrm{op}}}(P_j^G).
\end{split}
\end{align}

We first consider a special (sub)case:

\textbf{Case 2a:} Suppose in addition that $\cV(G)\ints\del\cA_\Sep(V)\neq \emptyset$, by the discussions in Example \ref{ex:triangle} (item (2) in \textbf{Useful claims}), we have $|\cV(G)\ints\del\cA_\Sep(V)|=1$. Let $\typemark{F_{p_l}}{\{p_l\}}{V}$ be the unique element in $\cV(G)\ints\del\cA_\Sep(V)$. In particular, $F_{p_0},F_{p_1},F_{p_2}$ are pairwise distinct and
\begin{align}\label{bchamber closed prep-3a: vertices of G}
\cV(G)=\{\typemark{F_{p_l}}{\{p_0\}}{V},\typemark{F_{p_l}}{\{p_1\}}{V},\typemark{F_{p_l}}{\{p_2\}}{V}\}.
\end{align}
Let $p_3:=p_0$ and $p_4:=p_1$. Then we have
\begin{align}\label{eqn:case2a sing set}
\Sing(G)=V=\{p_0,p_1,p_2\} \mathrm{~and~} P_j^G=(F_{p_l},\{\{p_{j+1}\},\{p_{j+2}\}\})\mathrm{~for~any~} j\in\{0,1,2\}.
\end{align}
Moreover, $P_l^G\in\res_{V_l}^V(\cV(\bG_{V;V_l}))\setminus\del\cA_\Sep(V_l)$. By the third assertion in Lemma \ref{easy properties of enrich} and the third assertion in Lemma \ref{properties of W-face}, $P_l^G$ is contained in exactly 2 $(V_l;V)$-enriched MCS of $\bG_{V;V_l}$, denoted as $\Enrich_{V_l}^{V}(G_1\ints\bG_{V;V_l})$ and $\Enrich_{V_l}^{V}(G_2\ints\bG_{V;V_l})$ for some distinct $G_1,G_2\in\MCS(\bG_{V})$. In particular, by Lemma \ref{reinterpretation of edges},  Lemma \ref{unique shortest paths} and the first assertion in Lemma \ref{easy properties of enrich}, we have $(\res_{V_l}^V)^{-1}(P_l^G)\ints\cV(G_t)\neq \emptyset$ for any $t\in\{1,2\}$. Since by \eqref{bchamber closed prep-3a: vertices of G}, $\subord_{V_l}^V(G\ints\bG_{V;V_l})=\emptyset$ and $\Enrich_{V_l}^V(G\ints\bG_{V;V_l})=\emptyset$. Hence $G\not\in\{ G_1,G_2\}$. (See Figure \ref{fig:(2,6)-1}.) Recall that $P_l^G\not\in\del\cA_\Sep(V_l)$ and $F_{p_l}\neq F_{p_i}$ for any $i\in\{0,1,2\}\setminus\{l\}$. By \eqref{bchamber closed prep-3a: vertices of G} again,
\begin{figure}[h]
	\centering
	\includegraphics[scale=0.9]{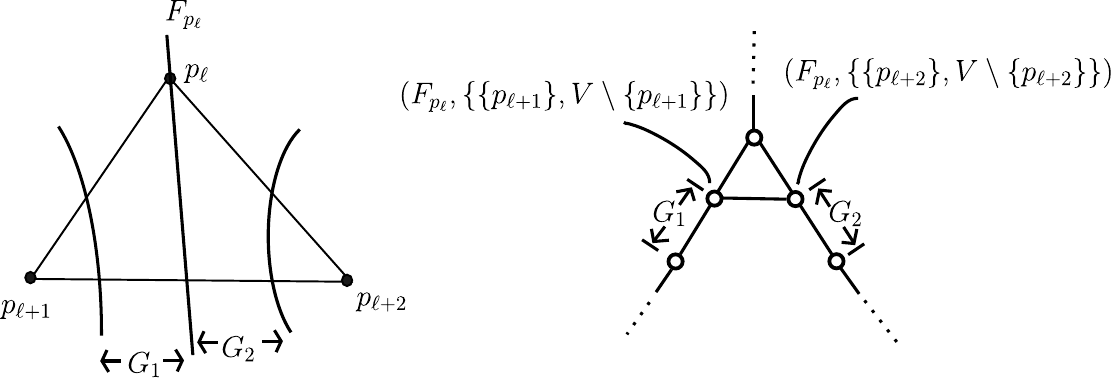}
	\caption{\label{fig:(2,6)-1}}
\end{figure}
\begin{align}\label{bchamber closed prep-3a: vertices of l-face}
\begin{split}
(\res_{V_l}^V)^{-1}(P_l^G)=&\{\typemark{F_{p_l}}{\{p_{l+1}\}}{V},\typemark{F_{p_l}}{\{p_{l+2}\}}{V}\}\\
=&\cV(G)\ints\cV(\bG_{V;V_l})=\cV(G)\setminus\del\cA_\Sep(V).
\end{split}
\end{align}
Therefore $\cV(G)\ints\cV(G_1)\neq\emptyset$ and $\cV(G)\ints\cV(G_2)\neq \emptyset$. By the second assertion in Proposition \ref{key prop of ASep graph}, we can assume WLOG that
$$\cV(G)\ints\cV(G_1)=\{Q_{l+1}:=\typemark{F_{p_l}}{\{p_{l+1}\}}{V}\}$$
and
$$\cV(G)\ints\cV(G_2)=\{Q_{l+2}:=\typemark{F_{p_l}}{\{p_{l+2}\}}{V}\}.$$
By the first assertion in Lemma \ref{reinterpretation of edges}, we have $\mathrm{Sing}(G_t)=\emptyset$ for any $t\in\{1,2\}$. Hence, we have
$$G_{Q_{l+1};l}^{\mathrm{op}}=(G_1)_{Q_{l+1};l}\mathrm{~and~}G_{Q_{l+2};l}^{\mathrm{op}}=(G_2)_{Q_{l+2};l}.$$
(See the notations in the statement of Proposition \ref{full constructions} and the \hyperlink{GQj}{definitions} of $G_{Q;j}$ and $G_{Q;j}^{\mathrm{op}}$.) For simplicity, we let $Q_4:=Q_1$ and $Q_5:=Q_2$. Recall that $F_{p_0}, F_{p_1}, F_{p_2}$ are distinct elements. Hence by \eqref{0D floor-1} (used in the seventh equality), \eqref{0D floor-2} (used in the seventh equality), \eqref{bchamber closed prep-3a: vertices of G} (used in the fifth equality), \eqref{eqn:case2a sing set} (used in the first equality and the sixth equality), \eqref{bchamber closed prep-3a: vertices of l-face} (used in the second equality), the second assertion in Proposition \ref{baby constructions} (used in the fourth equality) and the above (used in the third equality), we have
\begin{align}\label{bchamber closed prep-3a}
&\sum_{\substack{j:0\leq j\leq 2\\{p_j}\in\Sing(G)}}\sum_{\substack{Q:Q\in\cV(G)\ints\cV(\bG_{V;V_j})\\ Q\not\in\del\cA_\Sep(V)\\ \res_{V_j}^V(Q)=P_j^G }}(-1)^j\Floor_{\Delta_j,G_{Q;j}^{\mathrm{op}}}(P_j^G)\nonumber\\
=&\sum_{\substack{j:0\leq j\leq 2\\j\neq l}}\sum_{\substack{Q:Q\in\cV(G)\ints\cV(\bG_{V;V_j})\\ Q\not\in\del\cA_\Sep(V)\\ \res_{V_j}^V(Q)=P_j^G }}(-1)^j\Floor_{\Delta_j,G_{Q;j}^{\mathrm{op}}}(P_j^G)\nonumber\\
&+\sum_{\substack{Q:Q\in\cV(G)\ints\cV(\bG_{V;V_l})\\ Q\not\in\del\cA_\Sep(V)\\ \res_{V_l}^V(Q)=P_l^G }}(-1)^l\Floor_{\Delta_l,G_{Q;l}^{\mathrm{op}}}(P_l^G)  \nonumber\\
=&\sum_{\substack{j:0\leq j\leq 2\\j\neq l}}\sum_{\substack{Q:Q\in\cV(G)\ints\cV(\bG_{V;V_j})\\ Q\not\in\del\cA_\Sep(V)\\ \res_{V_j}^V(Q)=P_j^G }}(-1)^j\Floor_{\Delta_j,G_{Q;j}^{\mathrm{op}}}(P_j^G)\nonumber\\
&+(-1)^l(\Floor_{\Delta_l,G_{Q_{l+1};l}^{\mathrm{op}}}(P_l^G) +\Floor_{\Delta_l,G_{Q_{l+2};l}^{\mathrm{op}}}(P_l^G))\nonumber \\
=&\sum_{\substack{j:0\leq j\leq 2\\j\neq l}}\sum_{\substack{Q:Q\in\cV(G)\ints\cV(\bG_{V;V_j})\\ Q\not\in\del\cA_\Sep(V)\\ \res_{V_j}^V(Q)=P_j^G }}(-1)^j\Floor_{\Delta_j,G_{Q;j}^{\mathrm{op}}}(P_j^G)\nonumber\\
&+(-1)^l(\Floor_{\Delta_l,(G_1)_{Q_{l+1};l}}(P_l^G) +\Floor_{\Delta_l,(G_2)_{Q_{l+2};l}}(P_l^G))\nonumber \\
=&\sum_{\substack{0\leq j\leq 2\\j\neq l}}\sum_{\substack{Q\in\cV(G)\ints\cV(\bG_{V;V_j})\\ Q\not\in\del\cA_\Sep(V)\\ \res_{V_j}^V(Q)=P_j^G }}(-1)^j\Floor_{\Delta_j,G_{Q;j}^{\mathrm{op}}}(P_j^G)\nonumber\\
=&\begin{cases}
\displaystyle (-1)^1\Floor_{\Delta_1,G_{Q_2;1}^{\mathrm{op}}}(P_1^G)+(-1)^2\Floor_{\Delta_2,G_{Q_1;2}^{\mathrm{op}}}(P_2^G),~&\mathrm{if}~l=0, \\
\displaystyle (-1)^0\Floor_{\Delta_0,G_{Q_2;0}^{\mathrm{op}}}(P_0^G)+(-1)^2\Floor_{\Delta_2,G_{Q_0;2}^{\mathrm{op}}}(P_2^G),~&\mathrm{if}~l=1,\\
\displaystyle (-1)^0\Floor_{\Delta_0,G_{Q_1;0}^{\mathrm{op}}}(P_0^G)+(-1)^1\Floor_{\Delta_1,G_{Q_0;1}^{\mathrm{op}}}(P_1^G),~&\mathrm{if}~l=2
\end{cases}\nonumber\\
=&\begin{cases}
\displaystyle -\Floor_{\Delta_1,G_{Q_2;1}^{\mathrm{op}}}((F_{p_0},\{\{p_0\},\{p_2\}\}))+\Floor_{\Delta_2,G_{Q_1;2}^{\mathrm{op}}}((F_{p_0},\{\{p_0\},\{p_1\}\})),~&\mathrm{if}~l=0, \\
\displaystyle \Floor_{\Delta_0,G_{Q_2;0}^{\mathrm{op}}}((F_{p_1},\{\{p_1\},\{p_2\}\}))+\Floor_{\Delta_2,G_{Q_0;2}^{\mathrm{op}}}((F_{p_1},\{\{p_1\},\{p_0\}\})),~&\mathrm{if}~l=1,\\
\displaystyle \Floor_{\Delta_0,G_{Q_1;0}^{\mathrm{op}}}((F_{p_2},\{\{p_2\},\{p_1\}\}))-\Floor_{\Delta_1,G_{Q_0;1}^{\mathrm{op}}}((F_{p_2},\{\{p_2\},\{p_0\}\})),~&\mathrm{if}~l=2
\end{cases}\nonumber\\
=&\begin{cases}
\displaystyle -(-p_0)+(-p_0)=0,~&\mathrm{if}~l=0, \\
\displaystyle -p_1+p_1=0,~&\mathrm{if}~l=1,\\
\displaystyle p_2-p_2=0,~&\mathrm{if}~l=2.
\end{cases}
\end{align}
Therefore by \eqref{bchamber closed prep-1}, \eqref{bchamber closed prep-2} and \eqref{bchamber closed prep-3a}, we have $\del^X_1\BChamber_\Delta(G)=0$ in this (sub)case.

Now we prove the remaining (sub)case of \textbf{Case 2}.

\textbf{Case 2b:}
Suppose in addition that $\cV(G)\ints\del\cA_\Sep(V)=\emptyset$, by Lemma \ref{singular case} and the fact that $\Sing(G)\neq \emptyset$, there exist some $\hF\in\cF(V)\setminus\{F_{p_0},F_{p_1},F_{p_2}\}$.
\begin{align}\label{bchamber closed prep-3b: vertices of G}
\cV(G)=\{\typemark{\hF}{\{p_0\}}{V}, \typemark{\hF}{\{p_1\}}{V},\typemark{\hF}{\{p_2\}}{V}\}.
\end{align}
For simplicity, we write $Q_j=\typemark{\hF}{\{p_j\}}{V}$ for any $j\in\{0,1,2\}$. Since $\hF\in\Theta(F_{p_i},F_{p_j})\setminus\{F_{p_i},F_{p_j}\}$ for any $i\neq j\in\{0,1,2\}$, we have $F_{p_0},F_{p_1},F_{p_2}$ are distinct element and therefore $\Sing(G)=V$.

Let $p_3:=p_0$, $p_4:=p_1$, $Q_3:=Q_0$ and $Q_4:=Q_1$. Fix any $j\in\{0,1,2\}$. Then by \eqref{bchamber closed prep-3b: vertices of G} and the fact that $\hF\in\cF(V)\setminus\{F_{p_0},F_{p_1},F_{p_2}\}$, we have $P_j^G=(\hF,\{\{p_{j+1}\},\{p_{j+2}\}\})\in\res_{V_j}^V(\cV(\bG_{V;V_j}))\setminus\del\cA_\Sep(V_j)$ and
\begin{align}\label{bchamber closed prep-3b: vertices of j-face}
\begin{split}
(\res_{V_j}^V)^{-1}(P_j^G)=\{Q_{j+1},Q_{j+2}\}=\cV(G)\ints\cV(\bG_{V;V_j})\subset\cV(\bG_{V})\setminus\del\cA_\Sep(V).
\end{split}
\end{align}
By the third assertion in Lemma \ref{easy properties of enrich} and the third assertion in Lemma \ref{properties of W-face}, $P_j^G$ is contained in exactly $2$ $(V_j;V)$-enriched MCS of $\bG_{V;V_j}$, denoted as $\Enrich_{V_j}^V(G_{1;j}\ints\bG_{V;V_j})$ and $\Enrich_{V_j}^V(G_{2;j}\ints\bG_{V;V_j}))$ for some  distinct $G_{1;j}, G_{2;j}\in\MCS(\bG_{V})$. In particular, by Lemma \ref{reinterpretation of edges},  Lemma \ref{unique shortest paths} and the first assertion in Lemma \ref{easy properties of enrich}, we have $(\res_{V_j}^V)^{-1}(P_j^G)\ints\cV(G_{t;j})\neq \emptyset$ for any $t\in\{1,2\}$. Notice that by \eqref{bchamber closed prep-3b: vertices of G}, $\subord_{V_j}^V(G\ints\bG_{V;V_j})=\emptyset$ and hence $\Enrich_{V_j}^V(G\ints\bG_{V;V_j})=\emptyset$. As a direct corollary of this, $G\not\in\{ G_{1;j},G_{2;j}\}$. (See Figure \ref{fig:(2,6)-2}.) By the second assertion in Proposition \ref{key prop of ASep graph} and \eqref{bchamber closed prep-3b: vertices of j-face}, we can assume WLOG that
\begin{figure}[h]
	\centering
	\includegraphics[scale=0.9]{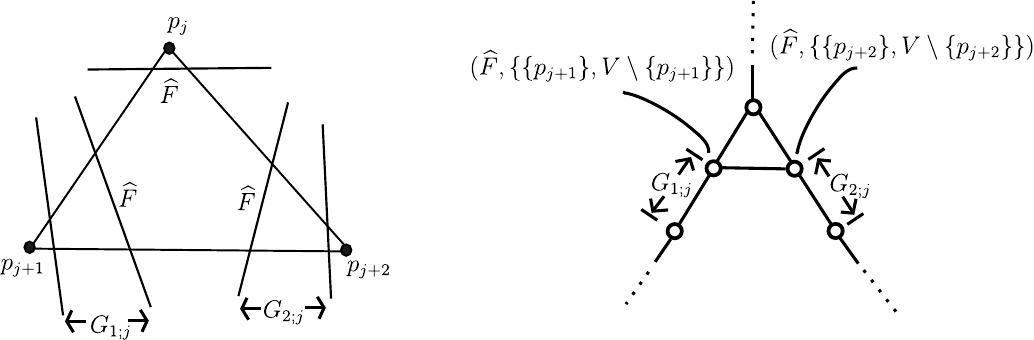}
	\caption{\label{fig:(2,6)-2}}
\end{figure}
$$\cV(G)\ints\cV(G_{1;j})=\{Q_{j+1}\}\mathrm{~and~}\cV(G)\ints\cV(G_{2;j})=\{Q_{j+2}\}.$$
Therefore, we have
$$G_{Q_{j+1};j}^{\mathrm{op}}=(G_{1;j})_{Q_{j+1};j}\mathrm{~and~}G_{Q_{j+2};j}^{\mathrm{op}}=(G_{2;j})_{Q_{j+2};j}.$$
(See the notations in the statement of Proposition \ref{full constructions} and the \hyperlink{GQj}{definitions} of $G_{Q;j}$ and $G_{Q;j}^{\mathrm{op}}$.) Hence by \eqref{bchamber closed prep-3b: vertices of j-face}, the second assertion in Proposition \ref{baby constructions} and the above, we have
\begin{align}\label{bchamber closed prep-3b}
&\sum_{\substack{j:0\leq j\leq 2\\{p_j}\in\Sing(G)}}\sum_{\substack{Q:Q\in\cV(G)\ints\cV(\bG_{V;V_j})\\ Q\not\in\del\cA_\Sep(V)\\ \res_{V_j}^V(Q)=P_j^G }}(-1)^j\Floor_{\Delta_j,G_{Q;j}^{\mathrm{op}}}(P_j^G) \nonumber\\
=&\sum_{j=0}^2(-1)^j(\Floor_{\Delta_j,G_{Q_{j+1};j}^{\mathrm{op}}}(P_j^G) +\Floor_{\Delta_j,G_{Q_{j+2};j}^{\mathrm{op}}}(P_j^G) )\nonumber\\
=&\sum_{j=0}^2(-1)^j(\Floor_{\Delta_j,(G_{1;j})_{Q_{j+1};j}}(P_j^G) +\Floor_{\Delta_j,(G_{2;j})_{Q_{j+2};j}}(P_j^G) )=0.
\end{align}
Therefore by \eqref{bchamber closed prep-1}, \eqref{bchamber closed prep-2} and \eqref{bchamber closed prep-3b}, we have $\del^X_1\BChamber_\Delta(G)=0$ in this (sub)case. This finishes the proof of $\BChamber_\Delta(G)$ for any $G\in\MCS(\bG_V)$ and hence the proof of the sixth assertion in Proposition \ref{full constructions} when $k=2$.

\item[\hypertarget{FP-7}{(7).}] {When $k=2$, we adopt the same notations as in Example \ref{ex:triangle}. Then by the discussions in Example \ref{ex:triangle}, for any $j\in\{0,1,2\}$, any $0\leq t\leq m_{j+1}-1$, any $0\leq s\leq m_{j+2}-1$, any $Q\in\cV(G_{j+1;t})$ and any $Q'\in\cV(G_{j+2;s})$, we have $(G_{j+1;t})_{Q;j}=G^{(j)}_{j+1;t}$ and $(G_{j+2;s})_{Q';j}=G^{(j)}_{j+2;s}$.}

By Proposition \ref{prop of beta'} (used in the eleventh (in)equality), the ninth assertion in Proposition \ref{baby constructions} (used in the third (in)equality), \eqref{beta' seg endpts} (used in the eighth (in)equality), \eqref{bfloor} (used in the second (in)equality), \eqref{0D floor-1} (used in the sixth (in)equality), \eqref{0D floor-2} (used in the sixth (in)equality) and \eqref{eqn:Fj(j+1,j+2)} (used in the tenth (in)equality), we have
\begin{align}\label{(7,2)}
&\sum_{\substack{Q,G:Q\in\cV(G)\\G\in\MCS(\bG_V)}}|\BFloor_{\Delta,G}(Q)|_{l^1} \nonumber\\
=&\sum_{j=0}^2|\BFloor_{\Delta,G_{\mathrm{center}}}(Q_{j;m_j})|_{l^1}+\sum_{j=0}^2\sum_{t=0}^{m_j-1}\left(|\BFloor_{\Delta,G_{j;t}}(Q_{j;t})|_{l^1}+|\BFloor_{\Delta,G_{j;t}}(Q_{j;t+1})|_{l^1}\right)\nonumber\\
\leq&\sum_{\substack{j:0\leq j\leq 2\\Q_{j;m_j}\not\in\del\cA_\Sep(V)}}\sum_{\substack{l:0\leq l\leq 2\\\cV(\bG_{V;V_l})\ni Q_{j;m_j}\\{p_l}\not\in\Sing(G)}}\left|\Floor_{\Delta_l, (G_{\mathrm{center}})_{Q_{j;m_j};l}}(P^{(l)}_{j;m_j})\right|_{l_1}\nonumber\\
&+\sum_{\substack{j:0\leq j\leq 2\\Q_{j;m_j}\not\in\del\cA_\Sep(V)}}\sum_{\substack{l:0\leq l\leq 2\\ \cV(\bG_{V;V_l})\ni Q_{j;m_j} \\{p_l}\in\Sing(G)}}\left|\Floor_{\Delta_l,(G_{\mathrm{center}})_{Q_{j;m_j};l}^{\mathrm{op}}}(P^{(l)}_{j;m_j})\right|_{l^1} \nonumber \\
&+\sum_{j=0}^2\sum_{t=0}^{m_j-1}\left|\sum_{\substack{l:0\leq l\leq 2\\l\neq j}}(-1)^l\Floor_{\Delta_l,(G_{j;t})_{Q_{j;t};l}}(P^{(l)}_{j;t})\right|_{l^1}\nonumber\\
&+\sum_{j=0}^2\sum_{t=0}^{m_j-1}\left|\sum_{\substack{l:0\leq l\leq 2\\l\neq j}}(-1)^l\Floor_{\Delta_l,(G_{j;t})_{Q_{j;t+1};l}}(P^{(l)}_{j;t+1})\right|_{l^1} \nonumber\\
\leq&3\times 3\times 1+3\times 3\times 1+\sum_{j=0}^2\sum_{t=0}^{m_j-1}\left|\sum_{\substack{0\leq l\leq 2\\l\neq j}}(-1)^l\Floor_{\Delta_l,(G_{j;t})_{Q_{j;t};l}}(P^{(l)}_{j;t})\right|_{l^1}\nonumber \\
&+\sum_{j=0}^2\sum_{t=0}^{m_j-1}\left|\sum_{\substack{l:0\leq l\leq 2\\l\neq j}}(-1)^l\Floor_{\Delta_l,(G_{j;t})_{Q_{j;t+1};l}}(P^{(l)}_{j;t+1})\right|_{l^1} \nonumber\\
=&18+\sum_{j=0}^2\sum_{t=0}^{m_j-1}\left|\sum_{\substack{l:0\leq l\leq 2\\l\neq j}}(-1)^l\Floor_{\Delta_l,G_{j;t}^{(l)}}(P^{(l)}_{j;t})\right|_{l^1}+\sum_{j=0}^2\sum_{t=0}^{m_j-1}\left|\sum_{\substack{l:0\leq l\leq 2\\l\neq j}}(-1)^l\Floor_{\Delta_l,G_{j;t}^{(l)}}(P^{(l)}_{j;t+1})\right|_{l^1} \nonumber\\
=&18+\sum_{t=0}^{m_0-1}\left|(-1)^1\Floor_{\Delta_1,G_{0;t}^{(1)}}(P^{(1)}_{0;t})+(-1)^2\Floor_{\Delta_2,G_{0;t}^{(2)}}(P^{(2)}_{0;t})\right|_{l^1} \nonumber\\
&+\sum_{t=0}^{m_1-1}\left|(-1)^0\Floor_{\Delta_0,G_{1;t}^{(0)}}(P^{(0)}_{1;t})+(-1)^2\Floor_{\Delta_2,G_{1;t}^{(2)}}(P^{(2)}_{1;t})\right|_{l^1} \nonumber\\
&+\sum_{t=0}^{m_2-1}\left|(-1)^0\Floor_{\Delta_0,G_{2;t}^{(0)}}(P^{(0)}_{2;t})+(-1)^1\Floor_{\Delta_1,G_{2;t}^{(1)}}(P^{(1)}_{2;t})\right|_{l^1} \nonumber\\
&+\sum_{t=0}^{m_0-1}\left|(-1)^1\Floor_{\Delta_1,G_{0;t}^{(1)}}(P^{(1)}_{0;t+1})+(-1)^2\Floor_{\Delta_2,G_{0;t}^{(2)}}(P^{(2)}_{0;t+1})\right|_{l^1} \nonumber\\
&+\sum_{t=0}^{m_1-1}\left|(-1)^0\Floor_{\Delta_0,G_{1;t}^{(0)}}(P^{(0)}_{1;t+1})+(-1)^2\Floor_{\Delta_2,G_{1;t}^{(2)}}(P^{(2)}_{1;t+1})\right|_{l^1} \nonumber\\
&+\sum_{t=0}^{m_2-1}\left|(-1)^0\Floor_{\Delta_0,G_{2;t}^{(0)}}(P^{(0)}_{2;t+1})+(-1)^1\Floor_{\Delta_1,G_{2;t}^{(1)}}(P^{(1)}_{2;t+1})\right|_{l^1} \nonumber\\
=&18+\frac{1}{2}\sum_{t=0}^{m_0-1}\left|-(-\cI'_{F_{0;t}}[p_0,p_2]-\cI'_{F_{0;t}}[p_2,p_0])+(-\cI'_{F_{0;t}}[p_0,p_1]-\cI'_{F_{0;t}}[p_1,p_0])\right|_{l^1} \nonumber\\
&+\frac{1}{2}\sum_{t=0}^{m_1-1}\left|(-\cI'_{F_{1;t}}[p_1,p_2]-\cI'_{F_{1;t}}[p_2,p_1])+(\cI'_{F_{1;t}}[p_0,p_1]+\cI'_{F_{1;t}}[p_1,p_0])\right|_{l^1} \nonumber\\
&+\frac{1}{2}\sum_{t=0}^{m_2-1}\left|(\cI'_{F_{2;t}}[p_1,p_2]+\cI'_{F_{2;t}}[p_2,p_1])-(\cI'_{F_{2;t}}[p_0,p_2]+\cI'_{F_{2;t}}[p_2,p_0])\right|_{l^1} \nonumber\\
&+\frac{1}{2}\sum_{t=0}^{m_0-1}\left|-(\cI'_{F_{0;t+1}}[p_0,p_2]+\cI'_{F_{0;t+1}}[p_2,p_0])+(\cI'_{F_{0;t+1}}[p_0,p_1]+\cI'_{F_{0;t+1}}[p_1,p_0])\right|_{l^1} \nonumber\\
&+\frac{1}{2}\sum_{t=0}^{m_1-1}\left|(\cI'_{F_{1;t+1}}[p_1,p_2]+\cI'_{F_{1;t+1}}[p_2,p_1])+(-\cI'_{F_{1;t+1}}[p_0,p_1]-\cI'_{F_{1;t+1}}[p_1,p_0])\right|_{l^1} \nonumber\\
&+\frac{1}{2}\sum_{t=0}^{m_2-1}\left|(-\cI'_{F_{2;t+1}}[p_1,p_2]-\cI'_{F_{2;t+1}}[p_2,p_1])-(-\cI'_{F_{2;t+1}}[p_0,p_2]-\cI'_{F_{2;t+1}}[p_2,p_0])\right|_{l^1} \nonumber\\
\leq&18+\frac{1}{2}\sum_{t=0}^{m_0-1}\left(\left|\cI'_{F_{0;t}}[p_0,p_2]-\cI'_{F_{0;t}}[p_0,p_1]\right|_{l^1}+\left|\cI'_{F_{0;t}}[p_2,p_0]-\cI'_{F_{0;t}}[p_1,p_0]\right|_{l^1}\right) \nonumber\\
&+\frac{1}{2}\sum_{t=0}^{m_1-1}\left(\left|\cI'_{F_{1;t}}[p_1,p_0]-\cI'_{F_{1;t}}[p_1,p_2]\right|_{l^1}+\left|\cI'_{F_{1;t}}[p_0,p_1]-\cI'_{F_{1;t}}[p_2,p_1]\right|_{l^1}\right) \nonumber\\
&+\frac{1}{2}\sum_{t=0}^{m_2-1}\left(\left|\cI'_{F_{2;t}}[p_2,p_0]-\cI'_{F_{2;t}}[p_2,p_1]\right|_{l^1}+\left|\cI'_{F_{2;t}}[p_0,p_2]-\cI'_{F_{2;t}}[p_1,p_2]\right|_{l^1}\right) \nonumber\\
&+\frac{1}{2}\sum_{t=0}^{m_0-1}\left(\left|\cI'_{F_{0;t+1}}[p_0,p_2]-\cI'_{F_{0;t+1}}[p_0,p_1]\right|_{l^1}+\left|\cI'_{F_{0;t+1}}[p_2,p_0]-\cI'_{F_{0;t+1}}[p_1,p_0]\right|_{l^1}\right) \nonumber\\
&+\frac{1}{2}\sum_{t=0}^{m_1-1}\left(\left|\cI'_{F_{1;t+1}}[p_1,p_0]-\cI'_{F_{1;t+1}}[p_1,p_2]\right|_{l^1}+\left|\cI'_{F_{1;t+1}}[p_0,p_1]-\cI'_{F_{1;t+1}}[p_2,p_1]\right|_{l^1}\right) \nonumber\\
&+\frac{1}{2}\sum_{t=0}^{m_2-1}\left(\left|\cI'_{F_{2;t+1}}[p_2,p_0]-\cI'_{F_{2;t+1}}[p_2,p_1]\right|_{l^1}+\left|\cI'_{F_{2;t+1}}[p_0,p_2]-\cI'_{F_{2;t+1}}[p_1,p_2]\right|_{l^1}\right) \nonumber\\
=&18+\frac{1}{2}\sum_{t=1}^{m_0-1}\left(\left|\cI'_{F_{0;t}}[p_0,p_2]-\cI'_{F_{0;t}}[p_0,p_1]\right|_{l^1}+\left|\cI'_{F_{0;t}}[p_2,p_0]-\cI'_{F_{0;t}}[p_1,p_0]\right|_{l^1}\right) \nonumber\\
&+\frac{1}{2}\sum_{t=1}^{m_1-1}\left(\left|\cI'_{F_{1;t}}[p_1,p_0]-\cI'_{F_{1;t}}[p_1,p_2]\right|_{l^1}+\left|\cI'_{F_{1;t}}[p_0,p_1]-\cI'_{F_{1;t}}[p_2,p_1]\right|_{l^1}\right) \nonumber\\
&+\frac{1}{2}\sum_{t=1}^{m_2-1}\left(\left|\cI'_{F_{2;t}}[p_2,p_0]-\cI'_{F_{2;t}}[p_2,p_1]\right|_{l^1}+\left|\cI'_{F_{2;t}}[p_0,p_2]-\cI'_{F_{2;t}}[p_1,p_2]\right|_{l^1}\right) \nonumber\\
&+\frac{1}{2}\sum_{t=1}^{m_0}\left(\left|\cI'_{F_{0;t}}[p_0,p_2]-\cI'_{F_{0;t}}[p_0,p_1]\right|_{l^1}+\left|\cI'_{F_{0;t}}[p_2,p_0]-\cI'_{F_{0;t}}[p_1,p_0]\right|_{l^1}\right) \nonumber\\
&+\frac{1}{2}\sum_{t=1}^{m_1}\left(\left|\cI'_{F_{1;t}}[p_1,p_0]-\cI'_{F_{1;t}}[p_1,p_2]\right|_{l^1}+\left|\cI'_{F_{1;t}}[p_0,p_1]-\cI'_{F_{1;t}}[p_2,p_1]\right|_{l^1}\right) \nonumber\\
&+\frac{1}{2}\sum_{t=1}^{m_2}\left(\left|\cI'_{F_{2;t}}[p_2,p_0]-\cI'_{F_{2;t}}[p_2,p_1]\right|_{l^1}+\left|\cI'_{F_{2;t}}[p_0,p_2]-\cI'_{F_{2;t}}[p_1,p_2]\right|_{l^1}\right) \nonumber\\
\leq&18+\sum_{t=1}^{m_0}\left(\left|\cI'_{F_{0;t}}[p_0,p_2]-\cI'_{F_{0;t}}[p_0,p_1]\right|_{l^1}+\left|\cI'_{F_{0;t}}[p_2,p_0]-\cI'_{F_{0;t}}[p_1,p_0]\right|_{l^1}\right) \nonumber\\
&+\sum_{t=1}^{m_1}\left(\left|\cI'_{F_{1;t}}[p_1,p_0]-\cI'_{F_{1;t}}[p_1,p_2]\right|_{l^1}+\left|\cI'_{F_{1;t}}[p_0,p_1]-\cI'_{F_{1;t}}[p_2,p_1]\right|_{l^1}\right) \nonumber\\
&+\sum_{t=1}^{m_2}\left(\left|\cI'_{F_{2;t}}[p_2,p_0]-\cI'_{F_{2;t}}[p_2,p_1]\right|_{l^1}+\left|\cI'_{F_{2;t}}[p_0,p_2]-\cI'_{F_{2;t}}[p_1,p_2]\right|_{l^1}\right) \nonumber\\
=&18+\sum_{\hF\in\cF_{p_0}(p_1,p_2)\setminus\{F_{p_0}\}}\left(\left|\cI'_{\hF}[p_0,p_2]-\cI'_{\hF}[p_0,p_1]\right|_{l^1}+\left|\cI'_{\hF}[p_2,p_0]-\cI'_{\hF}[p_1,p_0]\right|_{l^1}\right) \nonumber\\
&+\sum_{\hF\in\cF_{p_1}(p_0,p_2)\setminus\{F_{p_1}\}}\left(\left|\cI'_{\hF}[p_1,p_0]-\cI'_{\hF}[p_1,p_2]\right|_{l^1}+\left|\cI'_{\hF}[p_0,p_1]-\cI'_{\hF}[p_2,p_1]\right|_{l^1}\right) \nonumber\\
&+\sum_{\hF\in\cF_{p_2}(p_0,p_1)\setminus\{F_{p_2}\}}\left(\left|\cI'_{\hF}[p_2,p_0]-\cI'_{\hF}[p_2,p_1]\right|_{l^1}+\left|\cI'_{\hF}[p_0,p_2]-\cI'_{\hF}[p_1,p_2]\right|_{l^1}\right) \nonumber\\
\leq&18+6\cC_4.
\end{align}
Choose $\cC_6(2)=18+6\cC_4$ and the assertion follows from \eqref{floor}, \eqref{(7,2)} and the fourth remark after {Definition \ref{FBCone}}.

\item[(8).] \textbf{The proof of this assertion works for any $k\geq 2$, assuming that for any $(k',j')<(k,8)$, the $j'$-th assertion in Proposition \ref{full constructions} for the case $k'$ holds (although the arguments still work under weaker assumptions).}

The proof of this assertion is very similar to the proof of the fifth assertion in Proposition \ref{full constructions}. By the definition of $\Delta'$, we can easily see that $V(\Delta')=V$. It suffices to prove this assertion for any permutation $\tau$ of $\{0,...,k\}$ satisfying the following: there exist $0\leq j_1<j_2\leq k$ such that
$$\tau(i)=\begin{cases}
\displaystyle i,~&\mathrm{if}~ i\not\in\{ j_1,j_2\}, \\
\displaystyle j_2,~&\mathrm{if}~i=j_1,\\
\displaystyle j_1,~&\mathrm{if}~i=j_2
\end{cases}\mathrm{~and~hence~}
V(\Delta'_i)=\begin{cases}
\displaystyle V_i,~&\mathrm{if}~ i\not\in\{ j_1,j_2\}, \\
\displaystyle V_{j_2},~&\mathrm{if}~i=j_1,\\
\displaystyle V_{j_1},~&\mathrm{if}~i=j_2.
\end{cases}$$
In particular, $(-1)^{\mathrm{sgn}(\tau)}=-1$. By the fifth assertion in Proposition \ref{baby constructions} or the eighth assertion in Proposition \ref{full constructions} in the case of $k-1$, we have
\begin{align}\label{(8,k-1)}
&\sum_{\substack{G_{\tau(i)}:G\ints \bG_{V(\Delta');V(\Delta'_{\tau(i)})}\neq\emptyset\\G_{\tau(i)}\in\subord_{V(\Delta'_{\tau(i)})}^{V(\Delta')}(G\ints\bG_{V(\Delta');V(\Delta'_{\tau(i)})})}}\Chamber_{\Delta'_i}(G_{\tau(i)}) \nonumber\\
=&\begin{cases}
\displaystyle -\sum_{\substack{G_i:G\ints \bG_{V;V_{i}}\neq\emptyset\\G_i\in\subord_{V_i}^V(G\ints\bG_{V;V_i})}}\Chamber_{\Delta_i}(G_i), ~&\mathrm{if}~i\neq j_1,j_2, \\
\displaystyle (-1)^{j_2-j_1-1}\sum_{\substack{G_{j_2}:G\ints \bG_{V;V_{j_2}}\neq\emptyset\\G_{j_2}\in\subord_{V_{j_2}}^V(G\ints\bG_{V;V_{j_2}})}}\Chamber_{\Delta_{j_2}}(G_{j_2}), ~&\mathrm{if}~i=j_1,\\
\displaystyle (-1)^{j_1-j_2-1}\sum_{\substack{G_{j_1}:G\ints \bG_{V;V_{j_1}}\neq\emptyset\\G_{j_1}\in\subord_{V_{j_1}}^V(G\ints\bG_{V;V_{j_1}})}}\Chamber_{\Delta_{j_1}}(G_{j_1}), ~&\mathrm{if}~i=j_2.
\end{cases}
\end{align}
Apply the fifth assertion in Proposition \ref{full constructions} in the case of $k$ and \eqref{(8,k-1)} to \eqref{BChamber}, we have $\BChamber_{\Delta'}(G)=-\BChamber_{\Delta}(G)$. The rest of this assertion follows directly from the above and the third remark after {Definition \ref{FBCone}}.
\item[(9).] \textbf{The proof of this assertion works for any $k\geq 2$, assuming that for any $(k',j')<(k,9)$, the $j'$-th assertion in Proposition \ref{full constructions} for the case $k'$ holds (although the arguments still work under weaker assumptions).}

By Lemma \ref{unique shortest paths} and Definition \ref{subordinate}, for any $0\leq j\leq k$, any $G\in\MCS(\bG_V)$ with $G\ints\bG_{V;V_j}\neq \emptyset$ and any $G_j\in\subord_{V_j}^V(G\ints\bG_{V;V_j})$, we have $\cF_{V_j}(G_j)\subset\cF_V(G)$. Hence by the sixth assertion in Proposition \ref{baby constructions} or the ninth assertion in Proposition \ref{full constructions} in the case of $k-1$, we have
\begin{align}\label{(k,9)-1}
\vertsupp(\Chamber_{\Delta_j}(G_j))\subset\{x\in\Gamma x_0|F_x\in\cF_{V_j}(G_j)\}\subset\{x\in\Gamma x_0|F_x\in\cF_V(G)\}.
\end{align}
By the fourth assertion in Proposition \ref{full constructions} in the case of $k$, for any $Q\in\cV(G)$, we have
\begin{align}\label{(k,9)-2}
\vertsupp(\Floor_{\Delta,G}(Q))\subset\{x\in\Gamma x_0|F_x=F_Q\}\subset\{x\in\Gamma x_0|F_x\in\cF_V(G)\}.
\end{align}
Therefore, the fact that
$$\vertsupp(\BChamber_\Delta(G))\subset\{x\in\Gamma x_0|F_x\in\cF_V(G)\}$$
follows directly from \eqref{BChamber}, \eqref{(k,9)-1} and \eqref{(k,9)-2}. The fact that
$$\vertsupp(\Chamber_\Delta(G))\subset\{x\in\Gamma x_0|F_x\in\cF_V(G)\}$$
follows directly from \eqref{Chamber} and the fifth remark after Definition \ref{FBCone}.

Notice that for any distinct $Q, Q'\in\cV(G)$, by the first assertion in Lemma \ref{reinterpretation of edges}, we have $\Theta(F_Q,F_{Q'})\ints\cA_0(V)\subset\{F_Q,F_{Q'}\}$. Hence $\cF_V(G)\ints\cA_0(V)\subset\{F_Q|Q\in\cV(G)\}$. Notice that $\cF_V(G)\subset\cF(V)$, by the fifth assertion in Corollary \ref{remaining properties of sep graph} and Lemma \ref{almost all good}, we can choose
$$\cC_7(k):=k+1+3\cdot\comb{k+1}{3}\cdot\comb{k+1}{2}$$
and hence
$$|\cF_V(G)|\leq |V|+|\cF(V)\setminus\cA_0(V)|\leq \cC_7(k).$$

\item[(10).] Based on \eqref{BChamber} and the seventh assertion in Proposition \ref{full constructions} when $k=2$, we have
\begin{align}\label{(10,2)-1}
&\sum_{G:G\in\MCS(\bG_V)}|\BChamber_\Delta(G)|_{l^1}\nonumber\\
\leq &\sum_{\substack{G,Q:G\in\MCS(\bG_V)\\Q\in\cV(G)}}|\Floor_{\Delta,G}(Q)|_{l^1}+\sum_{G:G\in\MCS(\bG_V)}\left|\sum_{\substack{j,G_j:0\leq j\leq k\\ G\ints\bG_{V;V_j}\neq\emptyset \\ G_j\in\subord_{V_j}^V(G\ints\bG_{V;V_j})}}(-1)^j\Chamber_{\Delta_j}(G_j)\right|_{l^1}\nonumber\\
\leq &\cC_6(2)+\sum_{G:G\in\MCS(\bG_V)}\left|\sum_{\substack{j,G_j:0\leq j\leq 2\\ G\ints\bG_{V;V_j}\neq\emptyset \\ G_j\in\subord_{V_j}^V(G\ints\bG_{V;V_j})}}(-1)^j\Chamber_{\Delta_j}(G_j)\right|_{l^1}.
\end{align}
By the discussions in Example \ref{ex:triangle}, the following holds:
\begin{enumerate}
\item For any distinct $G,G'\in\MCS(\bG_V)$, we have $|\cF_V(G)\ints\cF_V(G')|\leq 1$.
\item By \eqref{1D chamber special}, \eqref{1D chamber}, \eqref{beta' seg} and \eqref{l1 norms of beta' seg}, for any $j\in\{0,1,2\}$, any $G\in\MCS(\bG_V)$ such that $G\cap\bG_{V;V_j}\neq \emptyset$, and any $G_j\in\subord^V_{V_j}(G\cap\bG_{V;V_j})$, we have $\Chamber_{\Delta_j}(G_{j})$ is a linear combination of geodesic bicombing of the form $\ovec{[x,y]}$ with $x,y\in\Gamma x_0$ and $F_x,F_y\in\cF_{V_j}(G_j)\subset\cF_{V}(G)$. Moreover, if $G\neq G_{\mathrm{center}}$, then $F_x\neq F_y$ in the above discussion.
\end{enumerate}
It follows from the remark after Definition \ref{dfn:geo.bicomb} (used in the first equality), the third assertion in Lemma \ref{properties of W-face} (used in the third equality), Definition \ref{subordinate} (used in the third equality), the first assertion in Proposition \ref{baby constructions} (used in the fourth equality), Corollary \ref{bicombing finished} (used in the last inequality) and the above discussions (used in the first equality) that
\begin{align}\label{(10,2)-2}
&\sum_{G:G\in\MCS(\bG_V)}\left|\sum_{\substack{j,G_j:0\leq j\leq 2\\ G\ints\bG_{V;V_j}\neq\emptyset \\ G_j\in\subord_{V_j}^V(G\ints\bG_{V;V_j})}}(-1)^j\Chamber_{\Delta_j}(G_j)\right|_{l^1}\nonumber\\
=&\left|\sum_{G:G\in\MCS(\bG_V)}\sum_{\substack{j,G_j:0\leq j\leq 2\\ G\ints\bG_{V;V_j}\neq\emptyset \\ G_j\in\subord_{V_j}^V(G\ints\bG_{V;V_j})}}(-1)^j\Chamber_{\Delta_j}(G_j)\right|_{l^1} \nonumber\\
=&\left|\sum_{j=0}^2\sum_{G_j:G_j\in\MCS(\bG_{V_j})}\sum_{\substack{G:G\in\MCS(\bG_{V})\\ G\ints\bG_{V;V_j}\neq\emptyset \\ \subord_{V_j}^V(G\ints\bG_{V;V_j})\ni G_j}}(-1)^j\Chamber_{\Delta_j}(G_j)\right|_{l^1} \nonumber\\
=&\left|\sum_{j=0}^2(-1)^j\sum_{G_j:G_j\in\MCS(\bG_{V_j})}\Chamber_{\Delta_j}(G_j)\right|_{l^1} \nonumber\\
=&\left|\sum_{j=0}^2(-1)^j\phi_1(\Delta_j)\right|_{l^1}\nonumber\\
=&\left|\beta[p_1,p_2]-\beta[p_0,p_2]+\beta[p_0,p_1]\right|_{l^1}=\left|\beta[p_1,p_2]+\beta[p_2,p_0]+\beta[p_0,p_1]\right|_{l^1}\leq \cC_5.
\end{align}
Choose $\cC_8(2)=\cC_6(2)+\cC_5$ and the rest of this assertion follows directly from \eqref{(10,2)-1}, \eqref{(10,2)-2} and the fourth remark after {Definition \ref{FBCone}}.

\item[(11).] \textbf{The proof of this assertion works for any $k\geq 2$, assuming that for any $(k',j')<(k,11)$, the $j'$-th assertion in Proposition \ref{full constructions} for the case $k'$ holds (although the arguments still work under weaker assumptions).}

Recall the definitions in \eqref{bfloor}, \eqref{floor}, \eqref{BChamber} and \eqref{Chamber}. This assertion in the case of $k$ then follows directly from the remark after Definition \ref{res of types}, the \hyperlink{GQj}{definitions} of $G_{Q;j;\Delta}$ and $G^{\mathrm{op}}_{Q;j;\Delta}$, the remark after Definition \ref{singular}, the remark after Definition \ref{subordinate}, the seventh remark after Definition \ref{FBCone} and either the seventh assertion in Proposition \ref{baby constructions} (when $k=2$), or the eleventh assertion in Proposition \ref{full constructions} in the case of $k-1$ (when $k\geq 3$).\qedhere
\end{enumerate}
\end{proof}
\subsection{Proof of Proposition \ref{full constructions} when $k\geq 3$}\label{subsec kD case}
This subsection is the continuation of Subsection \ref{subsec 2D case}. As is mentioned before, some assertions in Proposition \ref{full constructions} when $k\geq 3$ are already proved in Subsection \ref{subsec 2D case}. We will indicate those assertions in the following proof.

For any $1\leq j\leq 12$, we will always prove the $j$-th assertion in Proposition \ref{full constructions} for the case $k$ under the assumption that for any $(k',j')<(k,12)$, the $j'$-th assertion in Proposition \ref{full constructions} for the case $k'$ holds. The inductive proof structure also can be found in Subsection \ref{subsec 2D case}.
\begin{proof}[Proof of Proposition \ref{full constructions} when $k\geq 3$]
\begin{enumerate}
\item[(1).] Already proved in Subsection \ref{subsec 2D case}.
\item[(2).] By \eqref{bfloor}, this assertion is clearly true if $Q\in\del\cA_\Sep(V)$ when $k\geq 2$. When $Q\not\in\del\cA_\Sep(V)$, by the second assertion in Corollary \ref{remaining properties of sep graph}, there exist distinct $G, G'\in\MCS(\bG_V)$ such that $Q\in\cV(G)\ints\cV(G')$. Hence by the first assertion in Proposition \ref{full constructions} for the case of $k$, the second assertion in Proposition \ref{key prop of ASep graph} and the first assertion in Lemma \ref{reinterpretation of edges}, we can assume WLOG that
\begin{align}\label{eqn:(k,2)}
F_{Q'}\neq F_Q\text{ for any }Q'\in\cV(G)\setminus\{Q\}.
\end{align}
(This is because $\BFloor_{\Delta,G}(Q)$ is closed if and only if $\BFloor_{\Delta,G'}(Q)$ is closed due to the first assertion in Proposition \ref{full constructions} for the case of $k$. Moreover, if there exist some $Q''\in\cV(G)\setminus\{Q\}$ such that $F_{Q''}=F_Q$, the second assertion in Proposition \ref{key prop of ASep graph} and the first assertion in Lemma \ref{reinterpretation of edges} imply that for any $Q'\in\cV(G')\setminus\{Q\}$, we have $F_{Q'}\neq F_Q$. Then we can prove that $\BFloor_{\Delta,G'}(Q)$ is closed instead.) In particular, by Lemma \ref{singular case}, $\Sing(G)=\emptyset$.

In addition to the above, we have the following claim:
\begin{center}
\emph{For any $0\leq j\leq k$ such that $Q\in\cV(\bG_{V;V_j})$, we have $\Sing(G_{Q;j})=\emptyset$.}
\end{center}
If not, by Lemma \ref{singular case}, there exist some $P\in\cV(G_{Q;j})\setminus\{\res_{V_j}^V(Q)\}$ such that $F_P=F_Q=F_{\res_{V_j}^V(Q)}$. Let $Q=\typemark{F_P}{I}{V}$ and $P=\typemark{F_P}{J}{V_j}$ for some $\emptyset\neq I\subsetneq V$ and $\emptyset\neq J\subsetneq V_j$ such that $I\ints V_j\neq\emptyset$ and $V_j\setminus I\neq \emptyset$. The assumption that $P\in\cV(G_{Q;j})\setminus\{\res_{V_j}^V(Q)\}$ implies that $\rstype{I}{V_j}\neq\stype{J}{V_j}$.  Since $P\in\cA_\PSep(V_j)$ and $Q\in\cV(\bG_{V;V_j})\setminus\del\cA_\Sep(V)\subset\AnPSep(V)$, by Lemma \ref{prim decomp} and the first assertion in Lemma \ref{res of vertices}, we can assume WLOG that $J=H^{V_j}_{F_P}(q)$, $I=H^{V}_{F_P}(q')$ and $J\ints I=\emptyset$ for some $q,q'\in  V_j$. (Here we used the fact that $I\ints V_j=H^{V_j}_{F_P}(q')$ from the first assertion in Lemma \ref{res of vertices} and applied Lemma \ref{prim decomp} to the fact that $P\neq \res_{V_j}^V(Q)$.) Let $I'=H^V_{F_P}(q)$. Recall that $Q\not\in\del\cA_\Sep(V)\implies F_P=F_Q\in\cA_0(V)$, we have $Q':=\typemark{F_P}{I'}{V}\in\AnPSep(V)\subset\cV(\bG_V)$. By the first assertion in Lemma \ref{res of vertices} and Lemma \ref{prim decomp} again, we have $I'\ints V_j=J$ and $I'\ints I=\emptyset$. In particular, $F_Q=F_{Q'}$, $P=\res_{V_j}^V(Q')\neq\res_{V_j}^V(Q)$ and hence $Q'\neq Q$.

Let $G'$ is the unique element in $\MCS(\bG_V)\setminus\{G\}$ such that $Q\in\cV(G')$ (guaranteed by the second assertion in Corollary \ref{remaining properties of sep graph} and the assumption that $Q\not\in\del\cA_\Sep(V)$). By the first assertion in Proposition \ref{key prop of ASep graph}, \eqref{eqn:(k,2)} and the first assertion in Lemma \ref{reinterpretation of edges}, we have $Q'\in\cV(G')$. In particular, $G\cap\bG_{V;V_j}$ and $G'\cap\bG_{V;V_j}$ are distinct MCS in $\bG_{V;V_j}$ due to the fourth assertion in Lemma \ref{properties of W-face}. By Definition \ref{subordinate} and the \hyperlink{GQj}{definitions} of $G_{Q;j}$, this implies that $G_{Q;j}\in\subord_{V_j}^V(G'\cap\bG_{V;V_j})\cap\subord_{V_j}^V(G\cap\bG_{V;V_j})=\emptyset$, which is impossible. (See also the notations in the statement in Proposition \ref{full constructions}.) This proves the claim that $\Sing(G_{Q;j})=\emptyset$ for any $0\leq j\leq k$ such that $Q\in\cV(\bG_{V;V_j})$.

For any distinct $i,j\in\{0,...,k\}$, we denote by $V_{ij}=V_{ji}=V\setminus\{p_i,p_j\}$. Suppose $Q\in\cV(\bG_{V;V_{ij}})\subset\cV(\bG_{V;V_{j}})$, then by the fact that $\Sing(G_{Q;j})=\Sing(G)=\emptyset$, the \hyperlink{GQj}{definition} of $G_{Q;j}$ and Lemma \ref{combinatorics behind d2=0}, we have
$$\left(G_{Q;j}\right)_{\res_{V_j}^V(Q);i;\Delta_j}\in\subord_{V_{ij}}^{V_j}(G_{Q;j}\cap\bG_{V_j;V_{ij}})\subset\subord_{V_{ij}}^V(G\cap\bG_{V;V_{ij}})\neq\emptyset.$$
In particular, $\res_{V_{ij}}^V(Q)=\res_{V_{ij}}^{V_j}(\res_{V_j}^V(Q))$ is a vertex of $\left(G_{Q;j}\right)_{\res_{V_j}^V(Q);i;\Delta_j}$. On the other hand, by \eqref{eqn:MCS=subord} and the second assertion in Lemma \ref{easy properties of enrich}, there exists a unique MCS in $\bG_{V_{ij}}$, denoted as $G_{Q;ij}$, such that $G_{Q;ij}\in\subord_{V_{ij}}^V(G\cap\bG_{V;V_{ij}})$ and $\res_{V_{ij}}^V(Q)\in\cV(G_{Q;ij})$. Therefore
\begin{align}\label{GQij}
G_{Q;ij}=\left(G_{Q;j}\right)_{\res_{V_j}^V(Q);i;\Delta_j}=\left(G_{Q;i}\right)_{\res_{V_i}^V(Q);j;\Delta_i}.
\end{align}
Hence, by the second assertion in Proposition \ref{full constructions} for the case $k-1$ (used in the second equality), \eqref{bfloor} (used in the first and the third equalities) and \eqref{GQij} (used in the fourth equality), we have
\begin{align}\label{(2,k)}
&\del_{k-2}^X\BFloor_{\Delta,G}(Q) \nonumber\\
=&\sum_{\substack{j:0\leq j\leq k\\\cV(\bG_{V;V_j})\ni Q}}(-1)^j\del_{k-2}^X\Floor_{\Delta_j, G_{Q;j}}(\res_{V_j}^{V}(Q)) \nonumber\\
=&\sum_{\substack{j:0\leq j\leq k\\\cV(\bG_{V;V_j})\ni Q}}(-1)^j\BFloor_{\Delta_j, G_{Q;j}}(\res_{V_j}^{V}(Q)) \nonumber\\
=&\sum_{\substack{j:0\leq j\leq k\\\cV(\bG_{V;V_j})\ni Q}}(-1)^j\sum_{\substack{i:0\leq i<j\\ \cV(\bG_{V_j,V_{ij}})\ni\res_{V_{ij}}^{V_j}(Q)}}(-1)^i\Floor_{\Delta_{ij}, (G_{Q;j})_{\res_{V_j}^{V}(Q);i;\Delta_j}}(\res_{V_{ij}}^{V_j}(\res_{V_j}^{V}(Q))) \nonumber\\
&+\sum_{\substack{j:0\leq j\leq k\\\cV(\bG_{V;V_j})\ni Q}}(-1)^j\sum_{\substack{i:j<i\leq k\\ \cV(\bG_{V_j,V_{ij}})\ni\res_{V_{ij}}^{V_j}(Q)}}(-1)^{i-1}\Floor_{\Delta_{ij}, (G_{Q;j})_{\res_{V_j}^{V}(Q);i;\Delta_j}}(\res_{V_{ij}}^{V_j}(\res_{V_j}^{V}(Q))) \nonumber\\
=&\sum_{\substack{i,j:0\leq i<j\leq k\\\cV(\bG_{V;V_j})\ni Q\\\cV(\bG_{V_j,V_{ij}})\ni\res_{V_{j}}^{V}(Q)}}(-1)^{i+j}\Floor_{\Delta_{ij}, G_{Q;ij}}(\res_{V_{ij}}^{V}(Q)) \nonumber\\
&+\sum_{\substack{i,j:0\leq j<i\leq k\\\cV(\bG_{V;V_j})\ni Q\\\cV(\bG_{V_j,V_{ij}})\ni\res_{V_{j}}^{V}(Q)}}(-1)^{i+j-1}\Floor_{\Delta_{ij}, G_{Q;ij}}(\res_{V_{ij}}^{V}(Q)) \nonumber\\
=&\sum_{\substack{i,j:0\leq i<j\leq k\\\cV(\bG_{V;V_{ij}})\ni Q}}(-1)^{i+j}\Floor_{\Delta_{ij}, G_{Q;ij}}(\res_{V_{ij}}^{V}(Q)) \nonumber\\
&+\sum_{\substack{i,j:0\leq j<i\leq k\\\cV(\bG_{V;V_{ij}})\ni Q}}(-1)^{i+j-1}\Floor_{\Delta_{ij}, G_{Q;ij}}(\res_{V_{ij}}^{V}(Q)) =0.
\end{align}
The fact that $\del^X_{k-1}\Floor_{\Delta,G}(Q)=\BFloor_{\Delta,G}(Q)$ follows directly from \eqref{(2,k)} and the sixth remark after {Definition \ref{FBCone}}.

\item[(3).] For any $Q\in\del\cA_\Sep(V)$ and any $0\leq j\leq k$, if $Q\in\cV(\bG_{V;V_j})$, then $\res_{V_j}^V(Q)\in\del\cA_\Sep(V_j)$ (due to the second assertion in Lemma \ref{res of vertices}). Hence by \eqref{bfloor}, {Definition \ref{FBCone}} and the assumption that $k\geq 3$, for any $G_j\in\MCS(\bG_{V_j})$ with $\res_{V_j}^V(Q)\in\cV(G_j)$, $\Floor_{\Delta_j,G_j}(\res_{V_j}^V(Q))=0$. This assertion is a direct corollary of this fact when $k\geq 3$.
\item[(4).] Already proved in Subsection \ref{subsec 2D case}.
\item[(5).] Already proved in Subsection \ref{subsec 2D case}.

\item[(6).] By \eqref{bfloor} (used in the third equality), \eqref{BChamber} (used in the first and the third equalities), the second assertion in Proposition \ref{full constructions} for the case $k$ (used in the second equality), and the sixth assertion in Proposition \ref{full constructions} for the case $k-1$ (used in the second equality), we take the boundary of \eqref{BChamber} and obtain
\begin{align}\label{(6,k)'}
&\del_{k-1}^X\BChamber_\Delta(G)\nonumber\\
=&(-1)^{k-1}\sum_{Q:Q\in\cV(G)}\del^X_{k-1}\Floor_{\Delta,G}(Q)+\sum_{\substack{j,G_j:0\leq j\leq k\\G\ints G_{V;V_j}\neq \emptyset\\ G_j\in\subord_{V_j}^V(G\ints\bG_{V;V_j})}}(-1)^j\del^X_{k-1}\Chamber_{\Delta_j}(G_j) \nonumber\\
=&(-1)^{k-1}\sum_{Q:Q\in\cV(G)}\BFloor_{\Delta,G}(Q)+\sum_{\substack{j,G_j:0\leq j\leq k\\G\ints G_{V;V_j}\neq \emptyset\\ G_j\in\subord_{V_j}^V(G\ints\bG_{V;V_j})}}(-1)^j\BChamber_{\Delta_j}(G_j) \nonumber\\
=&(-1)^{k-1}\sum_{\substack{Q:Q\in\cV(G)\\Q\not\in\del\cA_\Sep(V)}}\sum_{\substack{j:0\leq j\leq k\\ \cV(\bG_{V;V_j})\ni Q\\p_j\not\in\Sing(G)}}(-1)^j\Floor_{\Delta_j,G_{Q;j}}(\res_{V_j}^V(Q)) \\
&-(-1)^{k-1}\sum_{\substack{Q:Q\in\cV(G)\\Q\not\in\del\cA_\Sep(V)}}\sum_{\substack{j:0\leq j\leq k\\ \cV(\bG_{V;V_j})\ni Q\\p_j\in\Sing(G)}}(-1)^j\Floor_{\Delta_j,G_{Q;j}^{\mathrm{op}}}(\res_{V_j}^V(Q)) \nonumber\\
+&\sum_{\substack{j,G_j:0\leq j\leq k\\G\ints \bG_{V;V_j}\neq \emptyset\\ G_j\in\subord_{V_j}^V(G\ints\bG_{V;V_j})}}(-1)^{j+k-2}\sum_{P_j:P_j\in\cV(G_j)}\Floor_{\Delta_j,G_j}(P_j)\nonumber \\
+&\sum_{\substack{j,G_j:0\leq j\leq k\\G\ints \bG_{V;V_j}\neq \emptyset\\ G_j\in\subord_{V_j}^V(G\ints\bG_{V;V_j})}}(-1)^j\sum_{\substack{i,G_{ij}:0\leq i<j\\G_j\ints\bG_{V_j;V_{ij}}\neq \emptyset \\ G_{ij}\in\subord_{V_{ij}}^{V_j}(G_j\ints\bG_{V_j;V_{ij}})}}(-1)^i\Chamber_{\Delta_{ij}}(G_{ij})\nonumber\\
+&\sum_{\substack{j,G_j:0\leq j\leq k\\G\ints \bG_{V;V_j}\neq \emptyset\\ G_j\in\subord_{V_j}^V(G\ints\bG_{V;V_j})}}(-1)^j\sum_{\substack{i,G_{ij}:j<i\leq k\\G_j\ints\bG_{V_j;V_{ij}}\neq \emptyset \\ G_{ij}\in\subord_{V_{ij}}^{V_j}(G_j\ints\bG_{V_j;V_{ij}})}}(-1)^{i-1}\Chamber_{\Delta_{ij}}(G_{ij})\nonumber.
\end{align}
Our goal is to prove that the second sum, the sum of the first and the third sums and the sum of the last two sums in \eqref{(6,k)'} all vanish. We first consider the last two sums. By Lemma \ref{combinatorics behind d2=0}, we have
\begin{align}\label{(6,k)-1}
&\sum_{\substack{j,G_j:0\leq j\leq k\\G\ints \bG_{V;V_j}\neq \emptyset\\ G_j\in\subord_{V_j}^V(G\ints\bG_{V;V_j})}}(-1)^j\sum_{\substack{i,G_{ij}:0\leq i<j\\G_j\ints\bG_{V_j;V_{ij}}\neq \emptyset \\ G_{ij}\in\subord_{V_{ij}}^{V_j}(G_j\ints\bG_{V_j;V_{ij}})}}(-1)^i\Chamber_{\Delta_{ij}}(G_{ij})\nonumber\\
+&\sum_{\substack{j,G_j:0\leq j\leq k\\G\ints \bG_{V;V_j}\neq \emptyset\\ G_j\in\subord_{V_j}^V(G\ints\bG_{V;V_j})}}(-1)^j\sum_{\substack{i,G_{ij}:j<i\leq k\\G_j\ints\bG_{V_j;V_{ij}}\neq \emptyset \\ G_{ij}\in\subord_{V_{ij}}^{V_j}(G_j\ints\bG_{V_j;V_{ij}})}}(-1)^{i-1}\Chamber_{\Delta_{ij}}(G_{ij})\nonumber\\
=&\sum_{\substack{i,j:0\leq i<j\leq k\\G\ints \bG_{V;V_{ij}}\neq\emptyset}}\sum_{\substack{G_j:G_j\ints \bG_{V_j;V_{ij}}\neq\emptyset\\G_j\in\subord_{V_j}^V(G\ints\bG_{V;V_j})}}\sum_{\substack{G_{ij}: G_{ij}\in\subord_{V_{ij}}^{V_j}(G_j\ints\bG_{V_j;V_{ij}})}}(-1)^{i+j}\Chamber_{\Delta_{ij}}(G_{ij})\nonumber\\
+&\sum_{\substack{i,j:0\leq j<i\leq k\\G\ints \bG_{V;V_{ij}}\neq\emptyset}}\sum_{\substack{G_j:G_j\ints \bG_{V_j;V_{ij}}\neq\emptyset\\G_j\in\subord_{V_j}^V(G\ints\bG_{V;V_j})}}\sum_{\substack{G_{ij}: G_{ij}\in\subord_{V_{ij}}^{V_j}(G_j\ints\bG_{V_j;V_{ij}})}}(-1)^{i+j-1}\Chamber_{\Delta_{ij}}(G_{ij})\nonumber\\
\begin{split}
=&\sum_{\substack{i,j,G_{ij}:0\leq i<j\leq k\\G\ints \bG_{V;V_{ij}}\neq\emptyset\\G_{ij}\in\subord_{V_{ij}}^V(G\ints\bG_{V;V_{ij}})}}(-1)^{i+j}\Chamber_{\Delta_{ij}}(G_{ij})\\
+&\sum_{\substack{i,j,G_{ij}:0\leq j<i\leq k\\G\ints \bG_{V;V_{ij}}\neq\emptyset\\G_{ij}\in\subord_{V_{ij}}^V(G\ints\bG_{V;V_{ij}})}}(-1)^{i+j-1}\Chamber_{\Delta_{ij}}(G_{ij})=0
\end{split}
\end{align}

We now consider the sum of the first and the third sums in \eqref{(6,k)'}. Notice that by Lemma \ref{singular case} and Definition \ref{subordinate}, for any $0\leq j\leq k$ such that $p_j\in\Sing(G)$, $\subord_{V_j}^V(G\ints\bG_{V;V_j})=\emptyset$. Hence
\begin{align}\label{(6,k)-2-1}
\begin{split}
&\sum_{\substack{j,G_j:0\leq j\leq k\\G\ints \bG_{V;V_j}\neq \emptyset\\ G_j\in\subord_{V_j}^V(G\ints\bG_{V;V_j})}}(-1)^{j+k-2}\sum_{P_j:P_j\in\cV(G_j)}\Floor_{\Delta_j,G_j}(P_j) \\
=&\sum_{\substack{j:0\leq j\leq k\\G\ints \bG_{V;V_j}\neq \emptyset\\p_j\not\in\Sing(G)}} \sum_{\substack{G_j,P_j:G_j\in\subord_{V_j}^V(G\ints\bG_{V;V_j})\\P_j\in\cV(G_j)}}(-1)^{j+k-2}\Floor_{\Delta_j,G_j}(P_j).
\end{split}
\end{align}
By Definition \ref{singular}, the fourth and the seventh assertions in Lemma \ref{properties of W-face}, for any $0\leq j\leq k$ such that $p_j\not\in\Sing(G)$ and $G\ints\bG_{V;V_j}\neq \emptyset$, the map $\res_{V_j}^V$ restricted to $\cV(G)\ints\cV(\bG_{V;V_j})$ is injective. To be specific, the fact that $G\ints\bG_{V;V_j}\neq \emptyset$ implies $G\cap\bG_{V;V_j}\in\MCS(\bG_{V;V_j})$ (due to the fourth assertion in Lemma \ref{properties of W-face}). If $\res_{V_j}^V$ restricted to $\cV(G)\ints\cV(\bG_{V;V_j})$ is not injective, then by the seventh assertion in Lemma \ref{properties of W-face}, there exist distinct $q_1,q_2\in V_j$ such that $|\res_{V_j}^V(\cV(G\cap\bG_{V;V_j})\cap\cV(\bP_{q_1,q_2,V_j}))|=1$. By Definition \ref{singular}, this implies that $p_j\in\Sing(G)$, which leads to a contradiction.

By the \hyperlink{GQj}{definition} of $G_{Q;j}$ (used in the fifth equality. See also the notation convention in the statement of Proposition \ref{full constructions}), \eqref{(6,k)-2-1} (used in the first equality), Definition \ref{subordinate} (used in the second equality), Definition \ref{enrich} (used in the second equality), \eqref{eqn:MCS=subord} (used in the third equality), the second assertion in Lemma \ref{easy properties of enrich} (used in the third equality), the first assertion in Proposition \ref{full constructions} for the case $k$ (used in the third equality), the third assertion in Proposition \ref{full constructions} for the case $k$ (used in the seventh equality) and the above discussion (used in the fourth equality), we can conclude that
\begin{align}\label{(6,k)-2}
&(-1)^{k-1}\sum_{\substack{Q:Q\in\cV(G)\\Q\not\in\del\cA_\Sep(V)}}\sum_{\substack{j:0\leq j\leq k\\ \cV(\bG_{V;V_j})\ni Q\\p_j\not\in\Sing(G)}}(-1)^j\Floor_{\Delta_j,G_{Q;j}}(\res_{V_j}^V(Q)) \nonumber\\
+&\sum_{\substack{j,G_j:0\leq j\leq k\\G\ints \bG_{V;V_j}\neq \emptyset\\ G_j\in\subord_{V_j}^V(G\ints\bG_{V;V_j})}}(-1)^{j+k-2}\sum_{P_j:P_j\in\cV(G_j)}\Floor_{\Delta_j,G_j}(P_j) \nonumber\\
=&\sum_{\substack{j:0\leq j\leq k\\p_j\not\in\Sing(G)}}\sum_{\substack{Q:Q\in\cV(G)\ints\cV(\bG_{V;V_j})\\Q\not\in\del\cA_\Sep(V)}}(-1)^{j+k-1}\Floor_{\Delta_j,G_{Q;j}}(\res_{V_j}^V(Q))\nonumber\\
&+\sum_{\substack{j:0\leq j\leq k\\G\ints \bG_{V;V_j}\neq \emptyset\\p_j\not\in\Sing(G)}} \sum_{\substack{G_j,P_j:G_j\in\subord_{V_j}^V(G\ints\bG_{V;V_j})\\P_j\in\cV(G_j)}}(-1)^{j+k-2}\Floor_{\Delta_j,G_j}(P_j)\nonumber \\
=&\sum_{\substack{j:0\leq j\leq k\\p_j\not\in\Sing(G)}}\sum_{\substack{Q:Q\in\cV(G)\ints\cV(\bG_{V;V_j})\\Q\not\in\del\cA_\Sep(V)}}(-1)^{j+k-1}\Floor_{\Delta_j,G_{Q;j}}(\res_{V_j}^V(Q))\nonumber\\
&+\sum_{\substack{j:0\leq j\leq k\\p_j\not\in\Sing(G)}} \sum_{\substack{P_j:P_j\in\res_{V_j}^V(\cV(G)\ints\cV(\bG_{V;V_j}))}}\sum_{\substack{G_j:G_j\in\subord_{V_j}^V(G\ints\bG_{V;V_j})\\\cV(G_j)\ni P_j}}(-1)^{j+k-2}\Floor_{\Delta_j,G_j}(P_j)\nonumber \\
&+\sum_{\substack{j:0\leq j\leq k\\G\cap\bG_{V;V_j}\neq\emptyset\\p_j\not\in\Sing(G)}} \sum_{\substack{P_j:P_j\in\cV(\Enrich^V_{V_j}(G\ints\bG_{V;V_j}))\\P_j\not\in\res_{V_j}^V(\cV(G)\ints\cV(\bG_{V;V_j}))}}\sum_{\substack{G_j:G_j\in\subord_{V_j}^V(G\ints\bG_{V;V_j})\\\cV(G_j)\ni P_j}}(-1)^{j+k-2}\Floor_{\Delta_j,G_j}(P_j)\nonumber \\
=&\sum_{\substack{j:0\leq j\leq k\\p_j\not\in\Sing(G)}}\sum_{\substack{Q:Q\in\cV(G)\ints\cV(\bG_{V;V_j})\\Q\not\in\del\cA_\Sep(V)}}(-1)^{j+k-1}\Floor_{\Delta_j,G_{Q;j}}(\res_{V_j}^V(Q))\nonumber\\
&+\sum_{\substack{j:0\leq j\leq k\\p_j\not\in\Sing(G)}} \sum_{\substack{P_j:P_j\in\res_{V_j}^V(\cV(G)\ints\cV(\bG_{V;V_j}))}}\sum_{\substack{G_j:G_j\in\subord_{V_j}^V(G\ints\bG_{V;V_j})\\\cV(G_j)\ni P_j}}(-1)^{j+k-2}\Floor_{\Delta_j,G_j}(P_j)\nonumber \\
=&\sum_{\substack{j:0\leq j\leq k\\p_j\not\in\Sing(G)}}\sum_{\substack{Q:Q\in\cV(G)\ints\cV(\bG_{V;V_j})\\Q\not\in\del\cA_\Sep(V)}}(-1)^{j+k-1}\Floor_{\Delta_j,G_{Q;j}}(\res_{V_j}^V(Q))\nonumber\\
&+\sum_{\substack{j:0\leq j\leq k\\p_j\not\in\Sing(G)}} \sum_{\substack{Q:Q\in\cV(G)\ints\cV(\bG_{V;V_j})}}\sum_{\substack{G_j:G_j\in\subord_{V_j}^V(G\ints\bG_{V;V_j})\\\cV(G_j)\ni \res^V_{V_j}(Q)}}(-1)^{j+k-2}\Floor_{\Delta_j,G_j}(\res_{V_j}^V(Q))\nonumber \\
=&\sum_{\substack{j:0\leq j\leq k\\p_j\not\in\Sing(G)}}\sum_{\substack{Q:Q\in\cV(G)\ints\cV(\bG_{V;V_j})\\Q\not\in\del\cA_\Sep(V)}}(-1)^{j+k-1}\Floor_{\Delta_j,G_{Q;j}}(\res_{V_j}^V(Q))\nonumber\\
&+\sum_{\substack{j:0\leq j\leq k\\p_j\not\in\Sing(G)}} \sum_{\substack{Q:Q\in\cV(G)\ints\cV(\bG_{V;V_j})}}(-1)^{j+k-2}\Floor_{\Delta_j,G_{Q;j}}(\res_{V_j}^V(Q))\nonumber \\
=&\sum_{\substack{j:0\leq j\leq k\\p_j\not\in\Sing(G)}}\sum_{\substack{Q:Q\in\cV(G)\ints\cV(\bG_{V;V_j})\\Q\in\del\cA_\Sep(V)}}(-1)^{j+k-2}\Floor_{\Delta_j,G_{Q;j}}(\res_{V_j}^V(Q))=0.
\end{align}
It remains for us to consider the second sum in \eqref{(6,k)'}. For any $0\leq j\leq k$ such that $p_j\in\Sing(G)$, by Lemma \ref{singular case}, $\res_{V_j}^V(\cV(G\ints\bG_{V;V_j}))$ contains only one element, which we denote by $P^G_j$. Since $\res_{V_j}^V(\del\cA_\Sep(V)\ints\cV(\bG_{V;V_j}))\subset\del\cA_\Sep(V_j)$ (due to the second assertion in Lemma \ref{res of vertices}), by \eqref{floor} (used in the second equality) and the above (used in the second and the third equalities), we have
\begin{align}\label{(6,k)-3-1}
&(-1)^{k-1}\sum_{\substack{Q:Q\in\cV(G)\\Q\not\in\del\cA_\Sep(V)}}\sum_{\substack{j:0\leq j\leq k\\ \cV(\bG_{V;V_j})\ni Q\\p_j\in\Sing(G)}}(-1)^j\Floor_{\Delta_j,G_{Q;j}^{\mathrm{op}}}(\res_{V_j}^V(Q))\nonumber\\
=&(-1)^{k-1}\sum_{\substack{j:0\leq j\leq k\\p_j\in\Sing(G)}}\sum_{\substack{Q:Q\in\cV(G)\ints\cV(\bG_{V;V_j})\\Q\not\in\del\cA_\Sep(V)}}(-1)^j\Floor_{\Delta_j,G_{Q;j}^{\mathrm{op}}}(\res_{V_j}^V(Q))\nonumber\\
=&(-1)^{k-1}\sum_{\substack{j:0\leq j\leq k\\p_j\in\Sing(G)}}\sum_{\substack{Q:Q\in\cV(G)\ints\cV(\bG_{V;V_j})\\\res_{V_j}^V(Q)\not\in\del\cA_\Sep(V_j)}}(-1)^j\Floor_{\Delta_j,G_{Q;j}^{\mathrm{op}}}(\res_{V_j}^V(Q))\nonumber\\
=&(-1)^{k-1}\sum_{\substack{j:0\leq j\leq k\\p_j\in\Sing(G)\\P_j^G\not\in\del\cA_\Sep(V_j)}}\sum_{\substack{Q:Q\in\cV(G)\ints\cV(\bG_{V;V_j})\\\res^V_{V_j}(Q)=P_j^G}}(-1)^j\Floor_{\Delta_j,G_{Q;j}^{\mathrm{op}}}(P_j^G).
\end{align}
By Lemma \ref{singular case} again, for any $j\in\Sing(G)$, $(\res_{V_j}^V)^{-1}(P_j^G)=\{Q_j^{(1)},Q_j^{(2)}\}$ for some $Q_j^{(1)},Q_j^{(2)}\in\cV(G)$. In particular, $F_P=F_{Q_j^{(1)}}=F_{Q_j^{(2)}}$. Moreover, if we assume in addition that $P_j^G\not\in\del\cA_\Sep(V_j)$, the second assertion in Lemma \ref{res of vertices} implies that $Q_j^{(1)},Q_j^{(2)}\not\in\del\cA_\Sep(V)$. By the second assertion of Corollary \ref{remaining properties of sep graph}, we can find $G_{1;j},G_{2;j}\in\MCS(\bG_{V})\setminus\{G\}$ such that $Q_j^{(t)}\in\cV(G_{t;j})$ for any $t\in\{1,2\}$. By the first assertion of Lemma \ref{reinterpretation of edges}, the first assertion in Proposition \ref{key prop of ASep graph}  and Lemma \ref{singular case}, we have $G_{1;j}\neq G_{2;j}$ and $F_{Q'}\neq F_P$ for any $Q'\in\cV(G_{1;j})\union\cV(G_{2;j})\setminus\{Q_j^{(1)},Q_j^{(2)}\}$. In particular, Lemma \ref{singular case} applied to $G_{1;j},G_{2;j}$ implies that $\Sing(G_{1;j})=\Sing(G_{2;j})=\emptyset$.
Hence by the \hyperlink{GQj}{definitions} of $G_{Q,j}$, we have $(G_{t;j})_{Q_j^{(t)};j}\in\subord_{V_j}^V(G_{t;j}\ints\bG_{V;V_j})$ is well-defined for any $t\in\{1,2\}$. By Definition \ref{subordinate}, $(G_{1;j})_{Q_j^{(1)};j}\neq (G_{2;j})_{Q_j^{(2)};j}$. Therefore by the second assertion Proposition \ref{key prop of ASep graph}, $(G_{1;j})_{Q_j^{(1)};j},(G_{2;j})_{Q_j^{(2)};j}$ are the only MCS in $\bG_{V_j}$ which contains $P_j^G$ as a vertex.

As a corollary of the first assertion in Proposition \ref{full constructions} for the case $k$ (used in the third equality), the \hyperlink{GQj}{definitions} of $G_{Q,j}$ and $G_{Q,j}^{\mathrm{op}}$ (used in the second equality), and the above (used in the first equality),
\begin{align}\label{(6,k)-3-2}
&(-1)^{k-1}\sum_{\substack{j:0\leq j\leq k\\p_j\in\Sing(G)\\P_j^G\not\in\del\cA_\Sep(V_j)}}\sum_{\substack{Q:Q\in\cV(G)\ints\cV(\bG_{V;V_j})\\\res^V_{V_j}(Q)=P_j^G}}(-1)^j\Floor_{\Delta_j,G_{Q;j}^{\mathrm{op}}}(P_j^G)\nonumber\\
=&(-1)^{k-1}\sum_{\substack{j:0\leq j\leq k\\p_j\in\Sing(G)\\P_j^G\not\in\del\cA_\Sep(V_j)}}(-1)^j\left[\Floor_{\Delta_j,G_{Q_j^{(1)};j}^{\mathrm{op}}}(P_j^G)+\Floor_{\Delta_j,G_{Q_j^{(2)};j}^{\mathrm{op}}}(P_j^G)\right]\nonumber\\
=&(-1)^{k-1}\sum_{\substack{j:0\leq j\leq k\\p_j\in\Sing(G)\\P_j^G\not\in\del\cA_\Sep(V_j)}}(-1)^{j}\left[\Floor_{\Delta_j,(G_{1;j})_{Q_j^{(1)};j}}(P_j^G)+\Floor_{\Delta_j,(G_{2;j})_{Q_j^{(2)};j}}(P_j^G)\right]\nonumber\\
=&(-1)^{k-1}\sum_{\substack{j:0\leq j\leq k\\p_j\in\Sing(G)\\P_j^G\not\in\del\cA_\Sep(V_j)}}(-1)^{j}\cdot0=0.
\end{align}
As a summary, we can apply \eqref{(6,k)-1}, \eqref{(6,k)-2}, \eqref{(6,k)-3-1} and \eqref{(6,k)-3-2} to \eqref{(6,k)'} and obtain $\del^X_{k-1}\BChamber_\Delta(G)=0$. The proof of the rest of the assertion follows from the same arguments as in \eqref{(6,k)-1g}, \eqref{(6,k)-2g} and \eqref{(6,k)-3g}.

\item[(7).] By \eqref{bfloor}, we have
\begin{align}\label{(7,k)-g}
&\sum_{\substack{Q,G:Q\in\cV(G)\\G\in\MCS(\bG_V)}}|\BFloor_{\Delta,G}(Q)|_{l^1} \nonumber\\
\leq&\sum_{\substack{Q,G:Q\in\cV(G)\\G\in\MCS(\bG_V)}}\sum_{\substack{j:0\leq j\leq k\\\cV(\bG_{V;V_j})\ni Q\\p_j\not\in\Sing(G)}}|\Floor_{\Delta_j,G_{Q;j}}(\res_{V_j}^V(Q))|_{l^1}\nonumber\\
&+\sum_{\substack{Q,G:Q\in\cV(G)\\G\in\MCS(\bG_V)}}\sum_{\substack{j:0\leq j\leq k\\\cV(\bG_{V;V_j})\ni Q\\p_j\in\Sing(G)}}|\Floor_{\Delta_j,G_{Q;j}^{\mathrm{op}}}(\res_{V_j}^V(Q))|_{l^1}\nonumber\\
=&\sum_{j=0}^k\sum_{Q:Q\in\cV(\bG_{V;V_j})}\sum_{\substack{G:G\in\MCS(\bG_V)\\\cV(G)\ni Q\\\Sing(G)\not\ni p_j}}|\Floor_{\Delta_j,G_{Q;j}}(\res_{V_j}^V(Q))|_{l^1}\nonumber\\
&+\sum_{j=0}^k\sum_{Q:Q\in\cV(\bG_{V;V_j})}\sum_{\substack{G:G\in\MCS(\bG_V)\\\cV(G)\ni Q\\\Sing(G)\ni p_j}}|\Floor_{\Delta_j,G_{Q;j}^{\mathrm{op}}}(\res_{V_j}^V(Q))|_{l^1}\nonumber\\
=&\sum_{j=0}^k\sum_{P:P\in\cV(\bG_{V_j})}\sum_{Q:Q\in(\res_{V_j}^V)^{-1}(P)}\sum_{\substack{G:G\in\MCS(\bG_V)\\\cV(G)\ni Q\\\Sing(G)\not\ni p_j}}|\Floor_{\Delta_j,G_{Q;j}}(P)|_{l^1}\nonumber\\
&+\sum_{j=0}^k\sum_{P:P\in\cV(\bG_{V_j})}\sum_{Q:Q\in(\res_{V_j}^V)^{-1}(P)}\sum_{\substack{G:G\in\MCS(\bG_V)\\\cV(G)\ni Q\\\Sing(G)\ni p_j}}|\Floor_{\Delta_j,G_{Q;j}^{\mathrm{op}}}(P)|_{l^1}
\end{align}
For any $0\leq j\leq k$, any $P\in\cV(\bG_{V_j})$ and any distinct $G_1, G_2\in\MCS(\bG_V)$ such that $\Sing(G_t)\not\ni p_j$ and $(\res_{V_j}^V)^{-1}(P)\ints\cV(G_t)\neq \emptyset$ for any $t\in\{1,2\}$, by Definition \ref{subordinate} and the \hyperlink{GQj}{definition} of $G_{Q,j}$, we have $(G_1)_{Q;j}\neq (G_2)_{Q;j}$. Since $|V\setminus V_j|=1$, $|(\res_{V_j}^V)^{-1}(P)|\leq 2$. Therefore for any $0\leq j\leq k$ and any $P\in\cV(\bG_{V_j})$, we have
\begin{align}\label{(7,k)-g1}
\begin{split}
&\sum_{Q:Q\in(\res_{V_j}^V)^{-1}(P)}\sum_{\substack{G:G\in\MCS(\bG_V)\\\cV(G)\ni Q\\\Sing(G)\not\ni p_j}}|\Floor_{\Delta_j,G_{Q;j}}(P)|_{l^1}\\
\leq&\sum_{Q:Q\in(\res_{V_j}^V)^{-1}(P)}\sum_{\substack{G_j:G_j\in\MCS(\bG_{V_j})\\\cV(G_j)\ni P}}|\Floor_{\Delta_j,G_{j}}(P)|_{l^1}\leq 2\sum_{\substack{G_j:G_j\in\MCS(\bG_{V_j})\\\cV(G_j)\ni P}}|\Floor_{\Delta_j,G_{j}}(P)|_{l^1}.
\end{split}
\end{align}
For any $0\leq j\leq k$ and any $P\in\cV(\bG_{V_j})$, by the first assertion in Lemma \ref{reinterpretation of edges} and Lemma \ref{singular case}, there exist at most one MCS in $\bG_{V}$ satisfying the following properties:
\begin{itemize}
\item $p_j$ is one of its singular points in the sense of Definition \ref{singular}.
\item The vertex set of this MCS contains some element in $(\res_{V_j}^V)^{-1}(P)$.
\end{itemize}
Notice that $|(\res_{V_j}^V)^{-1}(P)|\leq 2$, we have
\begin{align}\label{(7,k)-g2}
\begin{split}
&\sum_{Q:Q\in(\res_{V_j}^V)^{-1}(P)}\sum_{\substack{G:G\in\MCS(\bG_V)\\\cV(G)\ni Q\\\Sing(G)\ni p_j}}|\Floor_{\Delta_j,G_{Q;j}^{\mathrm{op}}}(P)|_{l^1}\\
\leq&\sum_{Q:Q\in(\res_{V_j}^V)^{-1}(P)}\sum_{\substack{G_j:G_j\in\MCS(\bG_{V_j})\\\cV(G_j)\ni P}}|\Floor_{\Delta_j,G_{j}}(P)|_{l^1}\leq 2\sum_{\substack{G_j:G_j\in\MCS(\bG_{V_j})\\\cV(G_j)\ni P}}|\Floor_{\Delta_j,G_{j}}(P)|_{l^1}.
\end{split}
\end{align}
Apply \eqref{(7,k)-g1}, \eqref{(7,k)-g2} and the seventh assertion in Proposition \ref{full constructions} for the case $k-1$ to \eqref{(7,k)-g}, we obtain
\begin{align}\label{(7,k)-final}
&\sum_{\substack{Q,G:Q\in\cV(G)\\G\in\MCS(\bG_V)}}|\BFloor_{\Delta,G}(Q)|_{l^1} \nonumber\\
\leq&4\sum_{j=0}^k\sum_{P:P\in\cV(\bG_{V_j})}\sum_{\substack{G_j:G_j\in\MCS(\bG_{V_j})\\\cV(G_j)\ni P}}|\Floor_{\Delta_j,G_{j}}(P)|_{l^1}\nonumber\\
=&4\sum_{j=0}^k\sum_{\substack{P,G_j:P\in\cV(G_j)\\G_j\in\MCS(\bG_{V_j})}}|\Floor_{\Delta_j,G_{j}}(P)|_{l^1} \leq 4(k+1)\cC_6(k-1).
\end{align}
Choose $\cC_6(k)=4(k+1)\cC_6(k-1)$ and the rest of the assertion follows directly from \eqref{(7,k)-final}, \eqref{floor} and the fourth remark after {Definition \ref{FBCone}}.
\item[(8).] Already proved in Subsection \ref{subsec 2D case}.
\item[(9).] Already proved in Subsection \ref{subsec 2D case}.
\item[(10).] By \eqref{BChamber} (used in the first inequality), the third assertion in Lemma \ref{properties of W-face} (used in the third (in)equality), Definition \ref{subordinate} (used in the third (in)equality), the seventh assertion in Proposition \ref{full constructions} for the case $k$ (used in the fourth (in)equality) and the tenth assertion in Proposition \ref{full constructions} for the case $k-1$ (used in the fourth (in)equality), we have
\begin{align}\label{(10,k)-final}
&\sum_{G:G\in\MCS(\bG_V)}|\BChamber_\Delta(G)|_{l^1}\nonumber\\
\leq&\sum_{G:G\in\MCS(\bG_V)}\sum_{Q:Q\in\cV(G)}|\Floor_{\Delta,G}(Q)|_{l^1}+\sum_{G:G\in\MCS(\bG_V)}\sum_{\substack{j,G_j:0\leq j\leq k\\G\ints\bG_{V;V_j}\neq \emptyset\\G_j\in\subord_{V_j}^V(G\ints\bG_{V;V_j})}}|\Chamber_{\Delta_j}(G_j)|_{l^1}\nonumber\\
=&\sum_{\substack{Q,G:Q\in\cV(G)\\G\in\MCS(\bG_V)}}|\Floor_{\Delta,G}(Q)|_{l^1}+\sum_{j=0}^k\sum_{\substack{G:G\in\MCS(\bG_V)\\G\ints\bG_{V;V_j}\neq\emptyset}}\sum_{G_j:G_j\in\subord_{V_j}^V(G\ints\bG_{V;V_j})}|\Chamber_{\Delta_j}(G_j)|_{l^1}\nonumber\\
\leq&\sum_{\substack{Q,G:Q\in\cV(G)\\G\in\MCS(\bG_V)}}|\Floor_{\Delta,G}(Q)|_{l^1}+\sum_{j=0}^k\sum_{G_j:G_j\in\MCS(\bG_{V_j})}|\Chamber_{\Delta_j}(G_j)|_{l^1}\nonumber\\
\leq& \cC_6(k)+(k+1)\cC_8(k-1).
\end{align}
Choose $\cC_8(k)=\cC_6(k)+(k+1)\cC_8(k-1)$ and the rest of the assertion follows directly from \eqref{(10,k)-final}, \eqref{Chamber} and the fourth remark after {Definition \ref{FBCone}}.
\item[(11).] Already proved in Subsection \ref{subsec 2D case}.\qedhere
\end{enumerate}
\end{proof}

\section{Proof of the theorems}\label{last sec}
\subsection{Preparations}
By Corollary \ref{properties of straightening}, for any singular $n$-simplex in $X$, $\phi_\bullet\circ\psi_\bullet(\sigma)$ is a weighted sum of special {barycentric} simplices with depth at most $\cC_9(n)$ whose weight is uniformly bounded from above by $\cC_8(n)$. By the first assertion in Lemma \ref{properties of Theta sep} and Definition \ref{special bar simplices}, the image of any special {barycentric} simplex with depth at most $\cC_9(n)$ is $(c_3(\epsilon_4/4)/2)$-close to at most $\cC_9(n)$ elements in $\Gamma F$. {In this subsection, we show that one can choose an $n$-form $\omega$ such that for any special barycentric simplex $\sigma$ with depth at most $m$, $\int_{\sP\circ\sigma}\omega$ can be uniformly bounded from above. Namely, we have the following lemma.}

\begin{lemma}\label{bar uniform volume bound}
Let $J:M\to\RR_{\geq 0}$ be a smooth function such that
$$\mathrm{supp}(J)\subset \{p\in M|r<d_M(p,N)<R\},$$
where $0<r<R<\min\{c_3(\epsilon_4/4)/2,\mathrm{injrad}(M)/3\}$. (See Lemma \ref{bar: away from edge, away from simplex} for the definition of $c_3(\cdot)$.) Consider the $n$-form $\omega_J(p)=J(p)d\vol_M$ on $M$. Then there exists some constant $\cC_1:=\cC_1(J,r,R,m)$ such that for any $n$-dimensional special barycentric simplex $\sigma:=\Db^n(p_0,...,p_n)$ on $X$ with depth at most $m$, we have
$$-\cC_1(J,r,R,m)<\int_{\sP\circ\sigma}\omega_J:=\int_{\Delta^n_{\RR^{n+1}}}(\sP\circ\sigma)^*\omega_J=\int_{\Delta^n_{\RR^{n+1}}}\sigma^*(\sP^*(\omega_J))<\cC_1(J,r,R,m),$$
where $\sP:X\to M$ is the covering map introduced in Section \ref{sect:setting}.
\end{lemma}
\begin{proof}
Let $V:=\{p_0,...,p_n\}$, $\widetilde{J}:=J\circ\sP$ and $\widetilde{\omega}_J:=\sP^*(\omega_J)$. Then we have
$$\mathrm{supp}(\widetilde{J}),\mathrm{supp}(\widetilde{\omega}_J)\subset\bigsqcup_{\widehat{F}\in\Gamma F}A_{r,R}(\widehat F).$$
By the definition of $\Theta(\cdot,\cdot)$, Definition \ref{special bar simplices} and the fact that $X$ is nonpositively curved, for any $i,j\in\{0,...,n\}$ with $i\neq j$, any $p_i'\in B_{\epsilon_4/4}(p_i)$, any $p_j'\in B_{\epsilon_4/4}(p_j)$ and any $\widehat{F}\in\Gamma F\setminus\cF(V)$, we have
$$d([p_i',p_j'],\widehat{F})> d([p_i,p_j],\widehat{F})-\frac{\epsilon_4}{4}\geq \frac{\epsilon_4}{2}-\frac{\epsilon_4}{4}=\frac{\epsilon_4}{4}.$$
By Lemma \ref{bar: away from edge, away from simplex} and the assumptions on $R$, for any $(p_0',...,p_n')\in B_{\epsilon_4/4}(p_0)\times...\times B_{\epsilon_4/4}(p_n)$ and any $\widehat{F}\in\Gamma F\setminus\cF(V)$, we have
\begin{align}\label{eqn:away from irrelatevant flats}
\Db^n(p_0',...,p_n')(\Delta^n_{\RR^{n+1}})\ints A_{r,R}(\widehat F)=\emptyset.
\end{align}
For any $\widehat F\in\Gamma F$, we define an $n$-form $\widetilde{\omega}_{J,F}$ on $X$ such that
$$\widetilde{\omega}_{J;F}(p)=\begin{cases}
\displaystyle \widetilde{\omega}_{J}(p),~&\mathrm{if}~d(p,\widehat F)\leq R,\\
\displaystyle 0,~&\mathrm{else}.
\end{cases}$$
then for any $(p_0',...,p_n')\in B_{\epsilon_4/4}(p_0)\times...\times B_{\epsilon_4/4}(p_n)$, \eqref{eqn:away from irrelatevant flats} implies that
\begin{align}\label{eqn:finite sum volume}
\int_{\Delta^n_{\RR^{n+1}}}(\Db^n(p_0',...,p_n'))^*\widetilde{\omega}_J=\sum_{\widehat F\in\cF(V)}\int_{\Delta^n_{\RR^{n+1}}}(\Db^n(p_0',...,p_n'))^*\widetilde{\omega}_{J,\widehat F}.
\end{align}
Now we claim that for any $(p_0',...,p_n')\in B_{\epsilon_4/4}(p_0)\times...\times B_{\epsilon_4/4}(p_n)$ and any $\widehat F\in\cF(V)$, there exists some $\cC_1':=\cC_1'(J,r,R)>0$ such that
\begin{align}\label{eqn:bar uniform volume bound each piece}
\left|\int_{\Delta^n_{\RR^{n+1}}}(\Db^n(p_0',...,p_n'))^*\widetilde{\omega}_{J,\widehat F}\right|\leq \cC_1'(J,r,R).
\end{align}
Suppose \eqref{eqn:bar uniform volume bound each piece} holds, then we choose $\cC_1=m\cC_1'$ and the lemma follow immediately from \eqref{eqn:finite sum volume} and \eqref{eqn:bar uniform volume bound each piece} (in the case when $(p_0',...,p_n')=(p_0,...,p_n)$). Therefore it remains for us to verify the claim \eqref{eqn:bar uniform volume bound each piece}.

By the second remark after Definition \ref{bar simplex} and Lemma \ref{Lip bar simplex}, it suffices to verify \eqref{eqn:bar uniform volume bound each piece} for $(p_0',...,p_n')$ on a dense subset of $B_{\epsilon_4/4}(p_0)\times...\times B_{\epsilon_4/4}(p_n)$. By Corollary \ref{bar almost injectivity}, the subset of points $(p_0',...,p_n')\in B_{\epsilon_4/4}(p_0)\times...\times B_{\epsilon_4/4}(p_n)$ satisfying the conditions in the statement of Corollary \ref{bar almost injectivity} is dense. Therefore, WLOG we assume that there exists an open and dense subset $\cU_{(p_0',...,p_n')}\subset X$ whose complement has $0$ measure with respect to $d\vol_X$, such that for any $q\in \Db^n(p_0',...,p_n')(\Delta^n_{\RR^{n+1}})\cap(\cU_{(p_0',...,p_{n}')})$, there exists a unique $(a_0,...,a_n)\in \Delta^n_{\RR^{n+1}}$ satisfying
$$\Db^n(p_0',...,p_n')(a_0,...,a_n)=q.$$
Let
$$\cU_{\mathrm{reg}}:=\{(a_0,...,a_n)\in\Delta^n_{\RR^{n+1}}|\mathrm{Jac}|_{(a_0,...,a_n)}(\Db^n(p_0',...,p_n'))\neq 0\}.$$
Then the restriction of $\Db^n(p_0',...,p_n')$ onto $\cU_{\mathrm{reg}}\ints (\Db^n(p_0',...,p_n'))^{-1}(\cU_{(p_0',...,p_n')})$ is a diffeomorphism onto its image. Moreover, the subset $\cU_{\mathrm{reg}}\ints (\Delta^n_{\RR^{n+1}}\setminus(\Db^n(p_0',...,p_n'))^{-1}(\cU_{(p_0',...,p_n')}))$ has $0$ Lebesgue measure. Since $(\Db^n(p_0',...,p_n'))^*\widetilde{\omega}_{J,\widehat F}$ is absolutely continuous with respect to the Lebesgue measure on $\Delta^n_{\RR^{n+1}}$, by Lemma \ref{bar bded ints away from flats} we have
\begin{align*}
\left|\int_{\Delta^n_{\RR^{n+1}}}(\Db^n(p_0',...,p_n'))^*\widetilde{\omega}_{J,\widehat F}\right|=&\left|\int_{\cU_{\mathrm{reg}}}(\Db^n(p_0',...,p_n'))^*\widetilde{\omega}_{J,\widehat F}\right|\\
=&\left|\int_{\cU_{\mathrm{reg}}\ints (\Db^n(p_0',...,p_n'))^{-1}(\cU_{(p_0',...,p_n')})}(\Db^n(p_0',...,p_n'))^*\widetilde{\omega}_{J,\widehat F}\right|\\
=&\left|\int_{\Db^n(p_0',...,p_n')(\cU_{\mathrm{reg}})\ints \cU_{(p_0',...,p_n')}}\widetilde{\omega}_{J,\widehat F}\right|\\
\leq&\left| \int_{A_{r,R}(\widehat F)\ints \Db^n(p_0',...,p_n')(\Delta^n_{\RR^{n+1}})}\widetilde{\omega}_{J,\widehat F}\right|\\
\leq &\left(\max_{p\in M}|J(p)|\right)\cdot\vol_X(A_{r,R}(\widehat F)\ints \Db^n(p_0',...,p_n')(\Delta^n_{\RR^{n+1}}))\\
\leq &\left(\max_{p\in M}|J(p)|\right)\cdot\sum_{j=0}^n\vol_X(B_{2R+R_1(r/2)}(\mathrm{Proj}_{\widehat F}(p_j')))\\
\leq &(n+1)\left(\max_{p\in M}|J(p)|\right)\cdot\vol_X(B_{2R+R_1(r/2)+\mathrm{diam}(M)+1}(x_0)),
\end{align*}
where $x_0$ is a fixed point in $X$ introduced in Notations \ref{Fx def}. Choose $\cC_1':=(n+1)\left(\max_{p\in M}|J(p)|\right)\cdot\vol_X(B_{2R+R_1(r/2)+\mathrm{diam}(M)+1}(x_0))$ and then \eqref{eqn:bar uniform volume bound each piece} follows. This completes the proof of Lemma \ref{bar uniform volume bound}.
\end{proof}

\subsection{Proof of Theorem \ref{thm:main}}
\begin{notation}
For any $n$-form $\omega$ on $M$ (or $X$) and any smooth $n$-dimensional chain $a=\sum_{\sigma}\alpha_\sigma\sigma$ with $\alpha_\sigma\in\RR$, (that is to say, the sum is over all smooth $n$-dimensional singular simplex $\sigma$ in the sense of \cite[page 416]{Lee03}), we let
$$\int_a\omega:=\sum_{\sigma}\alpha_\sigma\int_\sigma\omega=\sum_{\sigma}\alpha_\sigma\int_{\Delta^n_{\RR^{n+1}}}\sigma^*\omega.$$
\end{notation}
\begin{proof}[Proof of Theorem \ref{thm:main}]
By Lemma \ref{key hom lemma}, the constructions of $\phi_\bullet$ and $\psi_\bullet$, and Corollary \ref{properties of straightening}, we know that the $\RR[\Gamma]$-chain map $\phi_\bullet\circ\psi_{\bullet}$ is chain homotopic to the identity map on $C^X_\bullet$. Moreover, by \eqref{eqn:psi norm control} and Corollary \ref{properties of straightening}, for any singular simplex $\sigma\in C_k(X,\RR)$, $\phi_\bullet\circ\psi_{\bullet}(\sigma)$ is a linear combination of special barycentric simplices of depth at most $\cC_9(k)$ in the sense of Definition \ref{special bar simplices} satisfying $|\phi_k\circ\psi_{k}(\sigma)|_{l^1}\leq \cC_8(k)$.

Let $C^M:=(C_\bullet(M;\RR),\del^M)$ be the augmented chain complex of singular simplices on $M$ with coefficients in $\RR$. The covering map $\sP:X\to M$ naturally induces surjective $\RR$-linear maps $\sP_\bullet:C_\bullet(X;\RR)\to C_\bullet(M;\RR)$. (Namely, $\sP_{k}=\id$ for any $k<0$ and $\sP_{k}(\sigma)=\sP\circ\sigma$ for any $k\geq 0$ and any $k$-dimensional singular simplex $\sigma$.) Notice that for any $i,j\geq -2$ and any $\RR[\Gamma]$-homomorphism $h_X:C_i(X;\RR)\to C_j(X,\RR)$, there exists a unique $\RR$-linear map $h_M:C_i(M;\RR)\to C_j(M,\RR)$ such that $h_M\circ\sP_i=\sP_j\circ h_X$. By the fact that $\del^M_k\circ \sP_k=\sP_{k-1}\circ\del^X_k$ for any $k\geq -1$, there exists a chain map $\mathrm{st}_\bullet:C_\bullet(M;\RR)\to C_\bullet(M;\RR)$ such that $\mathrm{st}_\bullet\circ\sP_\bullet=\sP_\bullet\circ (\phi_\bullet\circ\psi_{\bullet})$. Moreover, by Lemma \ref{smoothness of bar simplex} and the uniqueness in the above discussions, $\mathrm{st}_\bullet$ is $\RR$-chain homotopic to the identity map with the image of $\mathrm{st}_n$ lying in the collections of smooth singular $n$-chains on $M$ in the sense of \cite[page 461]{Lee03}. In particular, for any closed $n$-chain $a\in C_n(M;\RR)$, the induced homology classes $[\mathrm{st}_n(a)]$ and $[a]$ in $H_n(M;\RR)$ are equal.

Fix some constants $0<r<R<\min\{c_3(\epsilon_4/4)/2,\mathrm{injrad}(M)/3\}$, a smooth function $J:M\to\RR_{\geq 0}$ such that $\mathrm{supp}(J)\subset \{p\in M|r<d_M(p,N)<R\}$ and its corresponding $n$-form $\omega_J(p)=J(p)d\vol_M$ on $M$.  For any singular simplex $\sigma\in C_n(M;\RR)$, we can find some $\widetilde{\sigma}\in C_n(X;\RR)$ such that $\sP_n(\widetilde\sigma)=\sigma$. Then by Lemma \ref{bar uniform volume bound},
\begin{align}\label{eqn:vol est of st simplex}
\int_{\mathrm{st}_n(\sigma)}\omega_J=\int_{\mathrm{st}_n(\sP_n(\widetilde\sigma))}\omega_J=\int_{\sP_n\circ\phi_n\circ\psi_{n}(\widetilde\sigma)}\omega_J<\cC_8(n)\cC_1(K,r,R,\cC_9(n)).
\end{align}
Hence for any closed $n$-chain $a$ such that $\int_a d\vol=1$, \eqref{eqn:vol est of st simplex} and the de Rham theorem \cite[Theorem 16.12]{Lee03} imply that
$$1=\int_ad\vol_M=\frac{\int_a\omega_J}{\int_M\omega_J}=\frac{\int_{\mathrm{st}_n(a)}\omega_J}{\int_M\omega_J}\leq \frac{\cC_8(n)\cC_1(K,r,R,\cC_9(n))}{\int_M\omega_J}\cdot|a|_{l^1}.$$
Taking infimum over all such $a$, we have
$$\|M\|\geq \frac{\int_M\omega_J}{\cC_8(n)\cC_1(K,r,R,\cC_9(n))}>0.$$
This completes the proof of Theorem \ref{thm:main}.
\end{proof}

\subsection{Proof of Corollary \ref{cor:main-1} and Corollary \ref{cor:main-2}}
\begin{proof}[Proof of Corollary \ref{cor:main-1}]
Let $X$ be the universal cover of $M$ and $F$ be a codimension one flat in $X$. Let $\sP:X\to M$ be the covering map. By the main theorem in \cite{Bangert91}, we assume WLOG that $\sP(F)$ is closed (and possibly with self-intersections). 

When $M$ has non-trivial Euclidean de Rham factors,  by \cite{Eberlein82} and \cite[page 8]{Gromov82}, we have $\|M\|=0$. 

Assume that $M$ does not have non-trivial Euclidean de Rham factors, by Theorem \ref{thm:main}, it suffices to show that $\sP(F)$ is a flat, isolated, closed codimension one submanifold in $M$. (See Section \ref{sect:setting} for the definition of isolatedness.) If not, by the fact that $F$ is flat, there exists a geodesic $c:\RR\to X$ such that $c(\RR)\not\subset F$ and that $d(c(t),F)$ is uniformally bounded from above. By \cite[1.A]{Bangert91}, $X$ is isometric to the product of $\RR$ and a Hadamard manifold with dimension $\dim(X)-1$. (This is due to the analytic assumption and the Sandwich Lemma in nonpositive curvature. See the discussion in \cite[1.A]{Bangert91} for more details.) This contradicts the assumption that $M$ does not have non-trivial Euclidean de Rham factors, finishing the proof.
\end{proof}
\begin{proof}[Proof of Corollary \ref{cor:main-2}]
By Corollary \ref{cor:main-1}, it remains to prove the case when there exists a maximal higher rank submanifold of type (3a). Let $F$ maximal higher rank submanifold of type (3a), which projects to a closed totally geodesic submanifold $N$ in $M$. Observe that if a $2$-flat intersects $F$ transversally, then this $2$-flat is not a maximal higher rank submanifold. (This is because the set $P$ of all points which lie on parallels of their intersection properly contains this $2$-flat. Due to the analytic assumptions, $P$ is a higher rank submanifold and has dimension at least $3$. Hence the $2$-flat is not maximal.) Hence by \cite[Lemma 3.3]{Schroeder89} and the above discussions, $F$ does not intersect any other maximal higher rank submanifold. This implies that $N$ is isolated. By Theorem \ref{thm:main}, we have $\|M\|>0$. 
\end{proof}

\def\cprime{$'$}
\providecommand{\bysame}{\leavevmode\hbox to3em{\hrulefill}\thinspace}
\providecommand{\MR}{\relax\ifhmode\unskip\space\fi MR }
\providecommand{\MRhref}[2]{%
	\href{http://www.ams.org/mathscinet-getitem?mr=#1}{#2}
}
\providecommand{\href}[2]{#2}

\end{document}